\documentclass[3p,10pt]{elsarticle}

\usepackage{amssymb,amsthm}
\usepackage{latexsym,amsmath,subfigure,color,siunitx,arydshln,multirow,xcolor,tikz}
\usepackage[export]{adjustbox}

\journal{}

\newcommand{\eps}{\varepsilon}
\newcommand{\epsb}{\varepsilon_\mathrm{b}}
\newcommand{\mub}{\mu_\mathrm{b}}
\newcommand{\sigmab}{\sigma_\mathrm{b}}
\newcommand{\set}[1]{\left\{#1\right\}}

\newcommand{\E}{\mathrm{E}}

\newcommand{\ma}{\mathbf{a}}
\newcommand{\mb}{\mathbf{b}}
\newcommand{\md}{\mathbf{d}}
\newcommand{\mf}{\mathbf{f}}
\newcommand{\mg}{\mathbf{g}}
\newcommand{\mr}{\mathbf{r}}
\newcommand{\mA}{\mathbf{A}}
\newcommand{\mB}{\mathbf{B}}

\newcommand{\mE}{\mathbf{E}}
\newcommand{\mH}{\mathbf{H}}
\newcommand{\mU}{\mathbf{U}}
\newcommand{\mV}{\mathbf{V}}
\newcommand{\vt}{\boldsymbol{\theta}}
\newcommand{\vv}{\boldsymbol{\vartheta}}

\DeclareMathOperator*{\inc}{inc}
\DeclareMathOperator*{\scat}{scat}
\DeclareMathOperator*{\tot}{tot}
\DeclareMathOperator*{\noise}{noise}

\DeclareMathOperator*{\tx}{TX}
\DeclareMathOperator*{\rx}{RX}
\DeclareMathOperator*{\NF}{NF}

\theoremstyle{plain}
\newtheorem{theorem}{Theorem}[section]
\newtheorem{corollary}[theorem]{Corollary}

\theoremstyle{remark}
\newtheorem{remark}{Remark}[section]
\newtheorem{example}{Example}[section]

\begin{document}

\begin{frontmatter}



\title{Application and analysis of MUSIC algorithm for anomaly detection in microwave imaging without a switching device}

\author{Won-Kwang Park}
\ead{parkwk@kookmin.ac.kr}
\address{Department of Information Security, Cryptology, and Mathematics, Kookmin University, Seoul, 02707, Korea.}

\begin{abstract}
Although the MUltiple SIgnal Classification (MUSIC) algorithm has demonstrated suitability as a microwave imaging technique for detecting anomalies, there is a fundamental limit that it requires a switching device to be used which permits an antenna to transmit and receive signals simultaneously. In this paper, we design a MUSIC-type imaging function using scattering parameter data to find small anomaly and explore its mathematical structure. Considering the investigated structure, we confirm that the imaging performance is highly dependent on the antenna configurations and suggest an arrangement of antennas to enhance imaging performance. Simulation results with synthetic data are displayed to support theoretical result.
\end{abstract}

\begin{keyword}
Microwave imaging \sep MUltiple SIgnal Classification (MUSIC) \sep Scattering parameter data \sep Simulation results \sep Switching device


\end{keyword}

\end{frontmatter}




\section{Introduction}\label{sec:1}
Retrieving of unknown targets from scattered field or scattering parameter data in multi-static measurement environment has been of great interest for many years in both inverse scattering problem and microwave imaging due to its various potential applications such as non-destructive evaluation \cite{CMPD,FMGD,TLRD}, biomedical imaging \cite{A2,CJGT,ML2}, radar imaging \cite{JBRBTFC,LSL,SSA}. To this ends, various reconstruction algorithms have been proposed and they are classified as iterative (or quantitative) and non-iterative (or qualitative) methods. We refer to remarkable researches \cite{AGJKLSW,BCS,C7,CK,DL,SH,Z2} for the description of various iterative and non-iterative techniques in both inverse scattering problems and microwave imaging.

MUltiple SIgnal Classification (MUSIC) algorithm is a well-known,  promising qualitative method in both inverse scattering problem and microwave imaging. From the pioneering research by Devaney \cite{D}, MUSIC was successfully applied for identifying point-like scatterers in multi-static measurement data. After the successful application, it has been applied various inverse scattering problems such as identification of two- and three-dimensional small targets \cite{AILP,ZC}, fast imaging of crack-like defects \cite{P-MUSIC1,P-MUSIC8}, detection of internal corrosion \cite{AKKLV}, imaging of buried targets \cite{AIL1,G2}, limited-aperture and limited-view problems \cite{AJP,P-MUSIC7} as well as microwave imaging such as damage assessment for intricate structures \cite{BYG}, through-the-wall radar imaging \cite{CPPPDT}, time-reversal imaging \cite{MGS}, radar imaging \cite{OBP}, anomaly detection from scattering matrix \cite{P-MUSIC6}, breast cancer detection \cite{RSCGBA}, and synthetic aperture radar (SAR) imaging \cite{ZZK}.

Let us emphasize that for a successful application of MUSIC in microwave imaging, the generated scattering matrix whose elements are scattering parameters must be symmetric. If the matrix is non-symmetric, it is impossible to apply the MUSIC directly. Due to this reason, most of researches has been performed with the strong assumption that the antenna can transmit and receive signals, i.e., switching devices can be used. However, the issue of microwave imaging without a switching device may also exist as a result of limited system configurations.

In this paper, we consider the application of MUSIC algorithm for identifying small anomaly from measured scattering parameter data at a fixed frequency when the switching device is not available. To this end, we assume that the scattering matrix is non-symmetric, design a new imaging function of MUSIC based on the structure of left- and right-singular vectors associated with nonzero singular values of the scattering matrix, and show that it can be represented by an infinite series of Bessel functions and antenna arrangements. This is based on the application of the Born approximation and integral equation formula for the scattering parameter in the existence of a small anomaly. Based on the theoretical result, we confirm that the imaging performance of proposed MUSIC is significantly dependent on the total number and location of transmitting and receiving antennas. To support the theoretical result, various simulation results with synthetic and experimental data are exhibited. It is worth to mention that if the total number of transmitting and receiving antennas is sufficiently large, it is possible to obtain a good result via MUSIC. However, in real-world microwave imaging applications, increasing total number of antennas somehow difficult. Instead of increasing total number of antennas, we investigate a method of arranging the antennas for obtaining a better result based on the elimination of unnecessary extra terms  that disturb the detection in the imaging function.

The rest of this paper is organized as follows. In Section \ref{sec:2}, we briefly survey the basic concept of scattering parameters and MUSIC algorithm without switching devices. In Section \ref{sec:3}, we prove that the imaging function can be expressed as an infinite series of Bessel functions and the antenna arrangement. Sections \ref{sec:4} and \ref{sec:5} present a set of simulation results with various antenna arrangements from synthetic and experimental data, respectively. Section \ref{sec:6} provides short conclusion and perspective of this study.

\section{Scattering Parameter and MUSIC Algorithm}\label{sec:2}
This section provides an overview of the scattering parameters and introduces MUSIC-type imaging function. Throughout this paper, we denote $\Sigma$ be a circular shaped anomaly to be imaged with radius $\alpha$ and center $\mr_\star$, and it is surrounded by a set of transmit and receive antennas $\mathcal{D}_{s}$ with location $\md_s$, $s=1,2,\cdots,S$. We let $\mathcal{B}_m$ denotes a transmitter antenna located at $\mb_m$, $m=1,2,\cdots,M$, that illuminates the $\Sigma$ with the microwave signal at a time and $\mathcal{A}_n$ denotes a receiver antenna located at $\ma_n$, $n=1,2,\cdots,N$, that collects the scattered signal by $\Sigma$. Let us denote $\mA=\set{\mathcal{A}_n:n=1,2,\cdots,N}$ and $\mB=\set{\mathcal{B}_m:m=1,2,\cdots,M}$ be the set of antennas for receiving and transmitting signals, respectively. Since there is no switching derivation through which the antenna transmits and receives signals, we assume that $\mA\cap\mB=\varnothing$ and $\mA\cup\mB\subseteq\set{\mathcal{D}_s:s=1,2,\cdots,S}$.

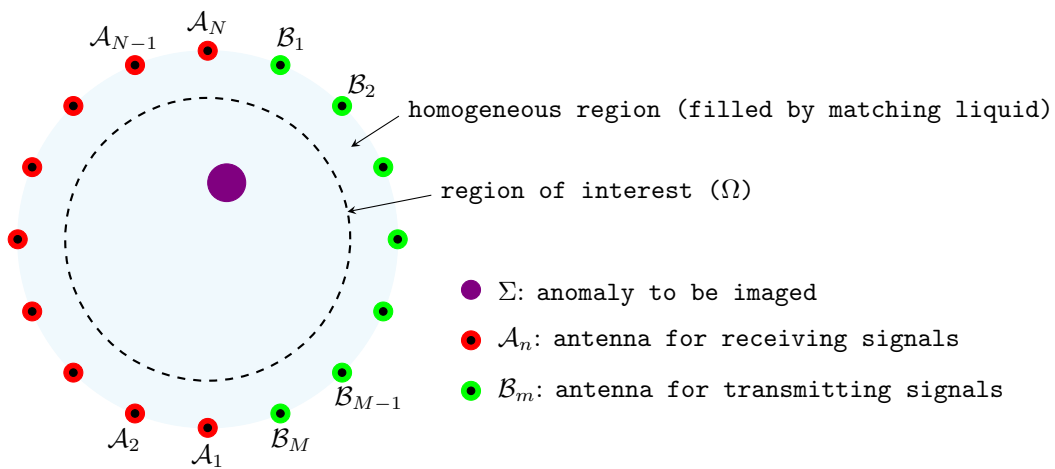
\begin{figure}[h]
\begin{center}
\begin{tikzpicture}[scale=2.5]
\draw[cyan!5,fill=cyan!5] (0,0) circle (1cm);
\draw[stealth-] ({0.9*cos(33)},{0.9*sin(33)}) -- ({1.2*cos(33)},{1.2*sin(33)});
\node[right] at ({1.2*cos(33)},{1.25*sin(33)}) {\texttt{homogeneous region (filled by matching liquid)}};

\draw[dashed,thick] (0,0) circle (0.75cm);
\draw[stealth-] ({0.75*cos(11.25)},{0.75*sin(11.25)}) -- ({1.2*cos(11.25)},{1.2*sin(11.25)});
\node[right] at ({1.2*cos(11.25)},{1.3*sin(11.25)}) {\texttt{region of interest ($\Omega$)}};

\foreach \alpha in {90,112.5,...,270}
{\draw[red,fill=red] ({cos(\alpha)},{sin(\alpha)}) circle (0.05cm);
\draw[black,fill=black] ({cos(\alpha)},{sin(\alpha)}) circle (0.02cm);}
\node at ({1.15*cos(270)},{1.15*sin(270)}) {$\mathcal{A}_{1}$};
\node at ({1.15*cos(247.5)},{1.15*sin(247.5)}) {$\mathcal{A}_{2}$};
\node at ({1.15*cos(112.5)},{1.15*sin(112.5)}) {$\mathcal{A}_{N-1}$};
\node at ({1.15*cos(90)},{1.15*sin(90)}) {$\mathcal{A}_{N}$};
\draw[red,fill=red] ({1.5*cos(337.5)},{1.4*sin(337.5)}) circle (0.05cm);
\draw[black,fill=black] ({1.5*cos(337.5)},{1.4*sin(337.5)}) circle (0.02cm);
\node[right] at ({1.6*cos(337.5)},{1.4*sin(337.5)}) {$\mathcal{A}_n$: \texttt{antenna for receiving signals}};

\foreach \beta in {-67.5,-45,...,67.5}
{\draw[green,fill=green] ({cos(\beta)},{sin(\beta)}) circle (0.05cm);
\draw[black,fill=black] ({cos(\beta)},{sin(\beta)}) circle (0.02cm);}
\node at ({1.15*cos(67.5)},{1.15*sin(67.5)}) {$\mathcal{B}_{1}$};
\node at ({1.15*cos(45)},{1.15*sin(45)}) {$\mathcal{B}_{2}$};
\node at ({1.2*cos(315)},{1.2*sin(315)}) {$\mathcal{B}_{M-1}$};
\node at ({1.15*cos(292.5)},{1.15*sin(292.5)}) {$\mathcal{B}_{M}$};
\draw[green,fill=green] ({1.5*cos(337.5)},{2.1*sin(337.5)}) circle (0.05cm);
\draw[black,fill=black] ({1.5*cos(337.5)},{2.1*sin(337.5)}) circle (0.02cm);
\node[right] at ({1.6*cos(337.5)},{2.1*sin(337.5)}) {$\mathcal{B}_m$: \texttt{antenna for transmitting signals}};

\draw[violet,fill=violet] (0.1,0.3) circle (0.1cm);

\draw[violet,fill=violet] ({1.5*cos(337.5)},{0.7*sin(337.5)}) circle (0.05cm);
\node[right] at ({1.6*cos(337.5)},{0.75*sin(337.5)}) {$\Sigma$: \texttt{anomaly to be imaged}};

\end{tikzpicture}
\caption{\label{Description}Illustration of configuration setup.}
\end{center}
\end{figure}

Throughout this paper, we assume that the microwave machine is filled by matching liquid, whose values of permittivity and conductivity at given angular frequency $\omega=2\pi f$ are $\epsb$ and $\sigmab$, respectively. Analogously, we set the values of permittivity and conductivity of $\Sigma$ at $\omega$ are $\eps_\star$ and $\sigma_\star$, respectively. Using this format, we introduce the piecewise constant permittivity and conductivity as
\[\eps(\mr)=\left\{\begin{array}{rcl}
\eps_\star & \text{for} & \mr\in\Sigma\\
\epsb & \text{for} & \mr\notin\Sigma
\end{array}
\right.
\quad\text{and}\quad
\sigma(\mr)=\left\{\begin{array}{rcl}
\sigma_\star & \text{for} & \mr\in\Sigma\\
\sigmab & \text{for} & \mr\notin\Sigma,
\end{array}
\right.\]
respectively. With this, we define a background wavenumber $k$ that satisfies $k=\omega\sqrt{\mub(\epsb+i\sigmab/\omega)}$, where $\mub$ represents magnetic permeability $\mub=4\pi\times\SI{e-7}{\henry/\meter}$. Let $S(n,m)$ be the $S-$parameter (or, scattering parameter), which is the measurement data through the microwave machine that is defined as the following ratio
\[S(n,m)=\frac{\mathrm{V}_n^-}{\mathrm{V}_m^+},\]
where
$\mathrm{V}_m^+$ and $\mathrm{V}_n^-$ denote the input voltage (or incident wave) at the $\mathcal{B}_m$ and the output voltage (or reflected wave) at the $\mathcal{A}_n$, respectively. We let $S_{\tot}(n,m)$ and $S_{\inc}(n,m)$ be measured scattering parameters in the presence and absence of $\Sigma$, respectively. Then, the scattered field $S-$parameter $\mathrm{S}_{\scat}(n,m)$  can be represented as follows (see \cite{HSM2} for instance):
\begin{equation}\label{Sparameter}
\mathrm{S}_{\scat}(n,m)=S_{\tot}(n,m)-S_{\inc}(n,m)=\frac{ik^2}{4\omega\mu_{\mathrm{b}}}\int_{\Sigma}\left(\frac{\eps(\mr)-\epsb}{\epsb}+i\frac{\sigma(\mr)-\sigmab}{\omega\epsb}\right)\mE_{\inc}(\mr,\mb_m)\cdot\mE_{\tot}(\ma_n,\mr)d\mr,
\end{equation}
where $\mE_{\inc}(\mr,\mb_m)$ denotes the incident electric field due to the point current density $\mathbf{J}$ at $\mb_m$ that satisfies
\[\nabla\times\mE_{\inc}(\mr,\mb_m)=-i\omega\mub\mH_{\inc}(\mr,\mb_m)\quad\text{and}\quad\nabla\times\mH_{\inc}(\mr,\mb_m)=(\sigmab+i\omega\epsb)\mE_{\inc}(\mr,\mb_m)\]
and $\mE_{\tot}(\ma_n,\mr)$ is corresponding total electric field in the presence of $\Sigma$ that satisfies
\[\nabla\times\mE_{\tot}(\ma_n,\mr)=-i\omega\mub\mH_{\tot}(\ma_n,\mr)\quad\text{and}\quad\nabla\times\mH_{\tot}(\ma_n,\mr)=(\sigma(\mr)+i\omega\eps(\mr))\mE_{\tot}(\ma_n,\mr)\]
with the transmission condition on the boundary of $\Sigma$. Here, we assume a time-harmonic dependence $e^{-i\omega t}$ and $\mH$ denotes the magnetic field.

In order to design MUSIC-type imaging algorithm, we consider the following scattering matrix $\mathbb{K}$, whose elements are $\mathrm{S}_{\scat}(n,m)$;
\begin{equation}\label{ScatteringMatrix}
\mathbb{K}=\begin{bmatrix}
S_{\scat}(1,1) & S_{\scat}(1,2) & \cdots & S_{\scat}(1,M) \\
S_{\scat}(2,1) & S_{\scat}(2,2) & \cdots & S_{\scat}(2,M) \\
\vdots & \vdots & \ddots & \vdots \\
S_{\scat}(N,1) & S_{\scat}(N,2) & \cdots & S_{\scat}(N,M)
\end{bmatrix}.
\end{equation}
It is worth emphasizing that the matrix $\mathbb{K}$ of \eqref{ScatteringMatrix} is neither symmetric nor Hermitian (because $N\ne M$). Due to this reason, the typical MUSIC-type imaging algorithm introduced in \cite{AILP,C,D,P-MUSIC6} cannot be applied for identifying $\mr_\star\in\Sigma$. For designing an alternative imaging function of MUSIC, we consider the structure of $\mathbb{K}$. Notice that similar to the recent works \cite{P-MUSIC6,HSM2,P-SUB11,P-SUB16,SSKLJ}, the height of microwave machine can be said to be long enough. Hence, based on the mathematical treatment of the scattering of time-harmonic electromagnetic waves from thin infinitely long cylindrical obstacles, only the $z-$component of the incident and total fields can be used. This means that the considered problem becomes a two-dimensional one and correspondingly, we can obtain a cross-section image of $\Sigma$. Then, by letting $\E_{\inc}$ and $\E_{\tot}$ respectively be the $z-$component of the  $\mE_{\inc}$ and $\mE_{\tot}$, $\mathrm{S}_{\scat}(n,m)$ of \eqref{Sparameter} becomes
\begin{equation}\label{Sparameter1}
\mathrm{S}_{\scat}(n,m)=\frac{ik^2}{4\omega\mub}\int_{\Sigma}\mathcal{O}(\mr)\E_{\inc}(\mr,\mb_m)\E_{\tot}(\ma_n,\mr)d\mr,\quad\text{where}\quad \mathcal{O}(\mr)=\frac{\eps(\mr)-\epsb}{\epsb}+i\frac{\sigma(\mr)-\sigmab}{\omega\epsb}.
\end{equation}
Now, assume that $\alpha$ satisfies
\begin{equation}\label{SmallAnomaly}
4\alpha\left(\sqrt{\frac{\eps_\star}{\epsb}}-1\right)<\lambda,
\end{equation}
where $\lambda$ denotes the wavelength. Then, based on \cite{SKL}, it is possible to apply the Born approximation to \eqref{Sparameter1}. With this, based on the reciprocity property of incident field, (\ref{Sparameter1}) can be approximated as
\begin{equation}\label{ScatteringParameter}
S_{\scat}(n,m)\approx\frac{ik^2}{4\omega\mub}\int_{\Sigma}\mathcal{O}(\mr_\star)\E_{\inc}(\mr,\mb_m)\E_{\inc}(\ma_n,\mr)d\mr\approx\frac{ik^2\alpha^2\pi}{4\omega\mub}\mathcal{O}(\mr_\star)\E_{\inc}(\ma_n,\mr_\star)\E_{\inc}(\mb_m,\mr_\star)
\end{equation}
and correspondingly, $\mathbb{K}$ can be approximated as
\begin{align}
\begin{aligned}\label{StructureMatrix}
\mathbb{K}&\approx\frac{ik^2\alpha^2\pi}{4\omega\mub}\mathcal{O}(\mr_\star)\begin{bmatrix}
\E_{\inc}(\ma_1,\mr_\star)\E_{\inc}(\mb_1,\mr_\star) & \E_{\inc}(\ma_1,\mr_\star)\E_{\inc}(\mb_2,\mr_\star) & \cdots & \E_{\inc}(\ma_1,\mr_\star)\E_{\inc}(\mb_M,\mr_\star) \\
\E_{\inc}(\ma_2,\mr_\star)\E_{\inc}(\mb_1,\mr_\star) & \E_{\inc}(\ma_2,\mr_\star)\E_{\inc}(\mb_2,\mr_\star) & \cdots & \E_{\inc}(\ma_2,\mr_\star)\E_{\inc}(\mb_M,\mr_\star) \\
\vdots & \vdots & \ddots & \vdots \\
\E_{\inc}(\ma_N,\mr_\star)\E_{\inc}(\mb_1,\mr_\star) & \E_{\inc}(\ma_N,\mr_\star)E_{\inc}(\mb_2,\mr_\star) & \cdots & \E_{\inc}(\ma_N,\mr_\star)\E_{\inc}(\mb_M,\mr_\star)
\end{bmatrix}\\
&=\frac{ik^2\alpha^2\pi}{4\omega\mub}\mathcal{O}(\mr_\star)
\begin{bmatrix}
\E_{\inc}(\ma_1,\mr_\star)\\
\E_{\inc}(\ma_2,\mr_\star)\\
\vdots\\
\E_{\inc}(\ma_N,\mr_\star)
\end{bmatrix}
\Big[\E_{\inc}(\mb_1,\mr_\star)\quad\E_{\inc}(\mb_2,\mr_\star)\quad\cdots\quad\E_{\inc}(\mb_M,\mr_\star)
\Big].
\end{aligned}
\end{align}
Since there exists a small anomaly, only the first singular value of $\mathbb{K}$ is nonzero so that the singular value decomposition (SVD) of $\mathbb{K}$ can be written as
\begin{equation}\label{SVD}
\mathbb{K}=\mathbb{USV}^*\approx\tau_1\mU_1\mV_1^*,
\end{equation}
where $\tau_1$ is the first singular value of $\mathbb{K}$, $\mU_1$ and $\mV_1$ denote the left- and right-singular vector of $\mathbb{K}$, respectively. Since $\set{\mU_1}$ and $\set{\mV_1}$ respectively span the signal subspaces of $\mathbb{KK}^*$ and $\mathbb{K}^*\mathbb{K}$, we can define orthonormal projection operators onto the different noise subspaces such that
\[\mathbb{F}_{\noise}=\mathbb{I}_N-\mU_1\mU_1^*\quad\text{and}\quad\mathbb{G}_{\noise}=\mathbb{I}_M-\mV_1\mV_1^*,\]
where $\mathbb{I}_M$ denotes the identity matrix of size $M$.

Now, based on the structures \eqref{StructureMatrix} and \eqref{SVD}, we define two unit vectors for each $\mr\in\Omega$ such that
\begin{equation}\label{TestVector}
\mf(\mr)=\left(\sum_{m=1}^M|\E(\mb_m,\mr)|^2\right)^{-1/2}
\begin{bmatrix}
\E(\ma_1,\mr)\\
\E(\ma_2,\mr)\\
\vdots\\
\E(\ma_N,\mr)
\end{bmatrix}
\quad\text{and}\quad
\mg(\mr)=\left(\sum_{n=1}^N|\overline{\E(\mb_m,\mr)}|^2\right)^{-1/2}
\begin{bmatrix}
\overline{\E(\mb_1,\mr)}\\
\overline{\E(\mb_2,\mr)}\\
\vdots\\\overline{\E(\mb_M,\mr)}
\end{bmatrix}.
\end{equation}
Then, based on \cite{P-MUSIC6,C,AK2}, it can be concluded that
\[\mf(\mr)\in\text{Range}(\mathbb{KK}^*)\quad\text{and}\quad\mg(\mr)\in\text{Range}(\mathbb{K^*}\mathbb{K})\quad\text{if and only if}\quad \mr=\mr_\star\in\Sigma.\]
This implies $|\mathbb{F}_{\noise}(\mf(\mr_\star))|=0$ and $|\mathbb{G}_{\noise}(\mg(\mr_\star))|=0$, and correspondingly, the following imaging function of MUSIC can be introduced: for $\mr\in\Omega$,
\[\mathfrak{F}(\mr)=\frac12\left(\frac{1}{|\mathbb{F}_{\noise}(\mf(\mr))|}+\frac{1}{|\mathbb{G}_{\noise}(\mg(\mr))|}\right).\]
Then, the map of $\mathfrak{F}(\mr)$ is expected to contain a peak of large magnitudes (theoretically, $+\infty$) when $\mr=\mr_\star\in\Sigma$. Thus, it will be possible to identify $\mr_\star$ via designed MUSIC. However, the imaging result is highly dependent on the arrangement of $\mathcal{A}_n$ and $\mathcal{B}_m$ but theoretical reason of this phenomena cannot be explained satisfactorily yet.

\section{Theoretical Result and Some Properties of Imaging Function}\label{sec:3}
To explain the feasibility of $\mathfrak{F}(\mr)$ and dependence of the antenna setting, we prove that the imaging function is expressed by an infinite series of Bessel functions and antenna arrangement.

\begin{theorem}[Structure of the Imaging Function for Single Anomaly]\label{StructureImaging}\label{Theorem}
Let $\vt_n=\ma_n/|\ma_n|=\ma_n/R=(\cos\theta_n,\sin\theta_n)$, $\vv_m=\mb_m/|\mb_m|=\mb_m/R=(\cos\vartheta_m,\sin\vartheta_m)$, and $\mr-\mr_\star=|\mr-\mr_\star|(\cos\phi,\sin\phi)$. If $M,N>1$ and $|\ma_n-\mr|,|\mb_m-\mr|\gg1/4|k|$ for $n=1,2,\cdots,N$ and $m=1,2,\cdots,M$, then $\mathfrak{F}(\mr)$ can be represented as follows: for $\mr\in\Omega$,
\begin{equation}\label{StructureImagingFunction}
\mathfrak{F}(\mr)\approx\frac{1}{2}\Big(\mathfrak{F}_{\rx}(\mr)+\mathfrak{F}_{\tx}(\mr)\Big),
\end{equation}
where
\[\mathfrak{F}_{\rx}(\mr)=\left(1-\bigg|J_0(k|\mr-\mr_\star|)+\frac{\mathcal{E}_{\rx}(\mr,\ma_n)}{N}\bigg|^2\right)^{-1/2}\quad\text{and}\quad\mathfrak{F}_{\tx}(\mr)=\left(1-\bigg|J_0(k|\mr-\mr_\star|)+\frac{\mathcal{E}_{\tx}(\mr,\mb_m)}{M}\bigg|^2\right)^{-1/2}.\]
Here, $J_p$ denotes the Bessel function of integer order $p$ of the first kind and
\begin{align*}
\mathcal{E}_{\rx}(\mr,\ma_n)&=\sum_{m=1}^{M}\sum_{p=-\infty,p\ne0}^{\infty}i^pJ_p(k|\mr-\mr_\star|)e^{ip(\theta_n-\phi)}\\
\mathcal{E}_{\tx}(\mr,\mb_m)&=\sum_{n=1}^{N}\sum_{q=-\infty,q\ne0}^{\infty}i^qJ_q(k|\mr-\mr_\star|)e^{iq(\vartheta_m-\phi)}.
\end{align*}
\end{theorem}
\begin{proof}
Since $|\ma_n-\mr|,|\mb_m-\mr|\gg1/4|k|$, the following asymptotic forms hold
\begin{align}
\begin{aligned}\label{AsymptoticField}
\E(\ma_n,\mr)&=\frac{i}{4}H_0^{(2)}(k|\ma_n-\mr|)\approx\frac{(-1+i)e^{-ik|\ma_n|}}{4\sqrt{k\pi|\ma_n|}}e^{ik\vt_n\cdot\mr}=\frac{(-1+i)e^{-ikR}}{4\sqrt{k\pi R}}e^{ik\vt_n\cdot\mr},\\
\E(\mb_m,\mr)&=\frac{i}{4}H_0^{(2)}(k|\mb_m-\mr|)\approx\frac{(-1+i)e^{-ik|\mb_m|}}{4\sqrt{k\pi|\mb_m|}}e^{ik\vv_m\cdot\mr}=\frac{(-1+i)e^{-ikR}}{4\sqrt{k\pi R}}e^{ik\vv_m\cdot\mr}.
\end{aligned}
\end{align}
Thus, $\mf(\mr)$ of \eqref{TestVector} and $\mathbb{K}$ of \eqref{ScatteringMatrix} can be written as
\[\mf(\mr)=\frac{1}{\sqrt{M}}
\begin{bmatrix}
e^{ik\vt_1\cdot\mr}\\
e^{ik\vt_2\cdot\mr}\\
\vdots\\
e^{ik\vt_M\cdot\mr}
\end{bmatrix}\quad\text{and}\quad\mathbb{K}\approx\frac{\alpha^2ke^{-2ikR}}{32R\omega\mub}\mathcal{O}(\mr_\star)\begin{bmatrix}
e^{ik(\vt_1+\vv_1)\cdot\mr_\star} & e^{ik(\vt_1+\vv_2)\cdot\mr_\star} & \cdots & e^{ik(\vt_1+\vv_M)\cdot\mr_\star} \\
e^{ik(\vt_2+\vv_1)\cdot\mr_\star} & e^{ik(\vt_2+\vv_2)\cdot\mr_\star} & \cdots & e^{ik(\vt_2+\vv_M)\cdot\mr_\star} \\
\vdots & \vdots & \ddots & \vdots \\
e^{ik(\vt_N+\vv_1)\cdot\mr_\star} & e^{ik(\vt_N+\vv_2)\cdot\mr_\star} & \cdots & e^{ik(\vt_N+\vv_M)\cdot\mr_\star}
\end{bmatrix},\]
respectively. Since, $|\tau_1|^2\mU_1\mU_1^*=\mathbb{KK}^*$, we can evaluate
\[\mU_1\mU_1^*=\frac{1}{|\tau_1|^2}\mathbb{KK}^*=\left|\frac{\alpha^2\mathcal{O}(\mr_\star)k}{32R\omega\mub\tau_1}\right|^2M\begin{bmatrix}
e^{ik(\vt_1-\vt_1)\cdot\mr_\star} & e^{ik(\vt_1-\vt_2)\cdot\mr_\star} & \cdots & e^{ik(\vt_1-\vt_N)\cdot\mr_\star} \\
e^{ik(\vt_2-\vt_1)\cdot\mr_\star} & e^{ik(\vt_2-\vt_2)\cdot\mr_\star} & \cdots & e^{ik(\vt_2-\vt_N)\cdot\mr_\star} \\
\vdots & \vdots & \ddots & \vdots \\
e^{ik(\vt_N-\vt_1)\cdot\mr_\star} & e^{ik(\vt_N-\vt_2)\cdot\mr_\star} & \cdots & e^{ik(\vt_N-\vt_N)\cdot\mr_\star}
\end{bmatrix}\]
and correspondingly, by letting $C=\left|\frac{\alpha^2\mathcal{O}(\mr_\star)k}{32R\omega\mub\tau_1}\right|^2M\in\mathbb{R}$,
\begin{align*}
\mathbb{F}_{\noise}(\mf(\mr))&=(\mathbb{I}_N-\mU_1\mU_1^*)\mf(\mr)\\
&=\frac{1}{\sqrt{N}}
\begin{bmatrix}
e^{ik\vt_1\cdot\mr}\\
e^{ik\vt_2\cdot\mr}\\
\vdots\\
e^{ik\vt_N\cdot\mr}
\end{bmatrix}-\frac{C}{\sqrt{N}}\begin{bmatrix}
e^{ik(\vt_1-\vt_1)\cdot\mr_\star} & e^{ik(\vt_1-\vt_2)\cdot\mr_\star} & \cdots & e^{ik(\vt_1-\vt_N)\cdot\mr_\star} \\
e^{ik(\vt_2-\vt_1)\cdot\mr_\star} & e^{ik(\vt_2-\vt_2)\cdot\mr_\star} & \cdots & e^{ik(\vt_2-\vt_N)\cdot\mr_\star} \\
\vdots & \vdots & \ddots & \vdots \\
e^{ik(\vt_N-\vt_1)\cdot\mr_\star} & e^{ik(\vt_N-\vt_2)\cdot\mr_\star} & \cdots & e^{ik(\vt_N-\vt_N)\cdot\mr_\star}
\end{bmatrix}
\begin{bmatrix}
e^{ik\vt_1\cdot\mr}\\
e^{ik\vt_2\cdot\mr}\\
\vdots\\
e^{ik\vt_N\cdot\mr}
\end{bmatrix}\\
&=\frac{1}{\sqrt{N}}
\begin{bmatrix}
\displaystyle e^{ik\vt_1\cdot\mr}-Ce^{ik\vt_1\cdot\mr_\star}\sum_{n=1}^{N}e^{ik\vt_n(\mr-\mr_\star)}\\
\displaystyle e^{ik\vt_2\cdot\mr}-Ce^{ik\vt_2\cdot\mr_\star}\sum_{n=1}^{N}e^{ik\vt_n(\mr-\mr_\star)}\\
\vdots\\
\displaystyle e^{ik\vt_N\cdot\mr}-Ce^{ik\vt_N\cdot\mr_\star}\sum_{n=1}^{N}e^{ik\vt_n(\mr-\mr_\star)}
\end{bmatrix}.
\end{align*}
With this, we can write
\begin{align}
\begin{aligned}\label{FactorF}
|\mathbb{F}_{\noise}(\mf(\mr))|^2&=\mathbb{F}_{\noise}(\mf(\mr))\cdot\overline{\mathbb{F}_{\noise}(\mf(\mr))}\\
&=\frac{1}{N}
\begin{bmatrix}
\displaystyle e^{ik\vt_1\cdot\mr}-Ce^{ik\vt_1\cdot\mr_\star}\sum_{n=1}^{N}e^{ik\vt_n(\mr-\mr_\star)}\\
\displaystyle e^{ik\vt_2\cdot\mr}-Ce^{ik\vt_2\cdot\mr_\star}\sum_{n=1}^{N}e^{ik\vt_n(\mr-\mr_\star)}\\
\vdots\\
\displaystyle e^{ik\vt_N\cdot\mr}-Ce^{ik\vt_N\cdot\mr_\star}\sum_{n=1}^{N}e^{ik\vt_n(\mr-\mr_\star)}
\end{bmatrix}\cdot
\begin{bmatrix}
\displaystyle e^{-ik\vt_1\cdot\mr}-Ce^{-ik\vt_1\cdot\mr_\star}\sum_{n=1}^{N}e^{-ik\vt_n(\mr-\mr_\star)}\\
\displaystyle e^{-ik\vt_2\cdot\mr}-Ce^{-ik\vt_2\cdot\mr_\star}\sum_{n=1}^{N}e^{-ik\vt_n(\mr-\mr_\star)}\\
\vdots\\
\displaystyle e^{-ik\vt_N\cdot\mr}-Ce^{-ik\vt_N\cdot\mr_\star}\sum_{n=1}^{N}e^{-ik\vt_n(\mr-\mr_\star)}
\end{bmatrix}\\
&=\frac{1}{N}\sum_{n=1}^{N}\Big(1-\Phi(\mr)-\overline{\Phi(\mr)}+\Psi(\mr)\overline{\Psi(\mr)}\Big),
\end{aligned}
\end{align}
where
\[\Phi(\mr)=Ce^{ik\vt_n\cdot(\mr-\mr_\star)}\sum_{n'=1}^{N}e^{-ik\vt_{n'}\cdot(\mr-\mr_\star)}\quad\text{and}\quad\Psi(\mr)=-Ce^{ik\vt_n\cdot\mr_\star}\sum_{n'=1}^{N}e^{ik\vt_{n'}\cdot(\mr-\mr_\star)}.\]

Since $\vt_n\cdot(\mr-\mr_\star)=|\mr-\mr_\star|\cos(\theta_m-\phi)$ and the following Jacobi-Anger expansion formula holds uniformly
\[e^{ix\cos\theta}=J_0(x)+\sum_{p=-\infty,p\ne0}^{\infty}i^pJ_p(x)e^{ip\theta},\]
we have
\begin{align*}
\frac{1}{N}\sum_{n=1}^{N}\Phi(\mr)&=\frac{C}{N}\left(\sum_{n=1}^{N}e^{ik\vt_n\cdot(\mr-\mr_\star)}\right)\left(\sum_{n=1}^{N}e^{-ik\vt_n\cdot(\mr-\mr_\star)}\right)=\frac{C}{N}\left|\sum_{n=1}^{N}e^{ik|\mr-\mr_\star|\cos(\theta_n-\phi)}\right|^2\\
&=\frac{C}{N}\left|\sum_{n=1}^{N}\bigg(J_0(k|\mr-\mr_\star|)+\sum_{p=-\infty,p\ne0}^{\infty}i^pJ_p(k|\mr-\mr_\star|)e^{ip(\theta_n-\phi)}\bigg)\right|^2\\
&=CN\left|J_0(k|\mr-\mr_\star|)+\frac{1}{N}\sum_{n=1}^{N}\sum_{p=-\infty,p\ne0}^{\infty}i^pJ_p(k|\mr-\mr_\star|)e^{ip(\theta_n-\phi)}\right|^2
\end{align*}
and consequently
\begin{equation}\label{TermPhi}
\frac{1}{N}\sum_{n=1}^{N}\Big(\Phi(\mr)+\overline{\Phi(\mr)}\Big)=2CN\left|J_0(k|\mr-\mr_\star|)+\frac{1}{N}\sum_{n=1}^{N}\sum_{p=-\infty,p\ne0}^{\infty}i^pJ_p(k|\mr-\mr_\star|)e^{ip(\theta_n-\phi)}\right|^2.
\end{equation}
Moreover, since
\begin{align*}
\Psi(\mr)\overline{\Psi(\mr)}&=\left(-Ce^{ik\vt_n\cdot\mr_\star}\sum_{n'=1}^{N}e^{ik\vt_{n'}\cdot(\mr-\mr_\star)}\right)\left(-Ce^{-ik\vt_n\cdot\mr_\star}\sum_{n'=1}^{N}e^{-ik\vt_{n'}\cdot(\mr-\mr_\star)}\right)\\
&=C^2\left|\sum_{n'=1}^{N}e^{ik|\mr-\mr_\star|\cos(\theta_{n'}-\phi)}\right|^2=C^2N^2\left|J_0(k|\mr-\mr_\star|)+\frac{1}{N}\sum_{n'=1}^{N}\sum_{p=-\infty,p\ne0}^{\infty}i^pJ_p(k|\mr-\mr_\star|)e^{ip(\theta_{n'}-\phi)}\right|^2,
\end{align*}
we can obtain
\begin{align}
\begin{aligned}\label{TermPsi}
\frac{1}{N}\sum_{n=1}^{N}\Psi(\mr)\overline{\Psi(\mr)}&=\frac{C^2N^2}{N}\sum_{n=1}^{N}\left|J_0(k|\mr-\mr_\star|)+\frac{1}{N}\sum_{n'=1}^{N}\sum_{p=-\infty,p\ne0}^{\infty}i^pJ_p(k|\mr-\mr_\star|)e^{ip(\theta_{n'}-\phi)}\right|^2\\
&=C^2N^2\left|J_0(k|\mr-\mr_\star|)+\frac{1}{N}\sum_{n=1}^{N}\sum_{p=-\infty,p\ne0}^{\infty}i^pJ_p(k|\mr-\mr_\star|)e^{ip(\theta_n-\phi)}\right|^2.
\end{aligned}
\end{align}
Hence, by combining \eqref{FactorF}, \eqref{TermPhi}, and \eqref{TermPsi}, we have
\[|\mathbb{F}_{\noise}(\mf(\mr))|^2=1-(2CN-C^2N^2)\left|J_0(k|\mr-\mr_\star|)+\frac{1}{N}\sum_{n=1}^{N}\sum_{p=-\infty,p\ne0}^{\infty}i^pJ_p(k|\mr-\mr_\star|)e^{ip(\theta_n-\phi)}\right|^2.\]
Since $|\mathbb{F}_{\noise}(\mf(\mr_\star))|=0$, $J_0(0)=1$, and $J_p(0)=0$ for any nonzero integer $p$,
\[|\mathbb{F}_{\noise}(\mf(\mr_\star))|=1-2CN-C^2N^2=(1-CN)^2=0\quad\text{implies}\quad CN=1.\]
Therefore,
\begin{equation}\label{ProjectionF}
|\mathbb{F}_{\noise}(\mf(\mr))|^2=1-\left|J_0(k|\mr-\mr_\star|)+\frac{1}{N}\sum_{n=1}^{N}\sum_{p=-\infty,p\ne0}^{\infty}i^pJ_p(k|\mr-\mr_\star|)e^{ip(\theta_n-\phi)}\right|^2.
\end{equation}

Next, since $|\tau_1|^2\mV_1\mV_1^*=\mathbb{K^*K}$, we can evaluate
\[\mV_1\mV_1^*=\frac{1}{|\tau_1|^2}\mathbb{K^*K}=\left|\frac{\alpha^2\mathcal{O}(\mr_\star)k}{32R\omega\mub\tau_1}\right|^2N\begin{bmatrix}
e^{-ik(\vv_1-\vv_1)\cdot\mr_\star} & e^{-ik(\vv_1-\vv_2)\cdot\mr_\star} & \cdots & e^{-ik(\vv_1-\vv_M)\cdot\mr_\star} \\
e^{-ik(\vv_2-\vv_1)\cdot\mr_\star} & e^{-ik(\vv_2-\vv_2)\cdot\mr_\star} & \cdots & e^{-ik(\vv_2-\vv_M)\cdot\mr_\star} \\
\vdots & \vdots & \ddots & \vdots \\
e^{-ik(\vv_M-\vv_1)\cdot\mr_\star} & e^{-ik(\vv_M-\vv_2)\cdot\mr_\star} & \cdots & e^{-ik(\vv_M-\vv_M)\cdot\mr_\star}
\end{bmatrix}.\]
Then, throughout a similar process, we can also obtain
\begin{equation}\label{ProjectionG}
|\mathbb{G}_{\noise}(\mf(\mr))|^2=1-\left|J_0(k|\mr-\mr_\star|)+\frac{1}{M}\sum_{m=1}^{M}\sum_{q=-\infty,q\ne0}^{\infty}(-i)^qJ_q(k|\mr-\mr_\star|)e^{-iq(\vartheta_m-\phi)}\right|^2.
\end{equation}
Since $|z|^2=|\overline{z}|^2$ for any $z\in\mathbb{C}$, based on the \eqref{ProjectionF} and \eqref{ProjectionG}, we can obtain \eqref{StructureImagingFunction}.
\end{proof}

Based on the explored structure of the designed imaging function \eqref{StructureImagingFunction}, we can examine some properties.

\begin{remark}[Decomposition of the Imaging Function]\label{Remark1}
The imaging function is decomposed by two factors; the factor $J_0(k|\mr-\mr_\star|)$, which is independent on the antenna arrangement, contributes to the identification because $J_0(0)=1$, and the factors $\mathcal{E}_{\rx}(\mr,\ma_n)$ and $\mathcal{E}_{\tx}(\mr,\mb_m)$ that are significantly dependent on the antenna arrangement disturb the identification by generating several artifacts because $J_p(0)=0$ for any nonzero integer $p$.
\end{remark}

\begin{remark}[Applicability and Limitation]\label{Remark2}
Based on the Remark \ref{Remark1}, the location $\mr_\star$ can be identified through the map of $\mathfrak{F}(\mr)$ if $M$ (total number of transmitting antennas) and $N$ (total number of receiving antennas) are sufficiently large. However, if $M$ or $N$ is not large enough, it will be very difficult to recognize the location $\mr_\star$ because $\mathfrak{F}(\mr)$ will be dominated by the factors $\mathcal{E}_{\rx}(\mr,\ma_n)$ or $\mathcal{E}_{\tx}(\mr,\mb_m)$. Let us emphasize that if one of $M$ or $N$ is large enough then the location $\mr_\star$ can be identified through the map of $\mathfrak{F}_{\rx}(\mr)$ or $\mathfrak{F}_{\tx}(\mr)$ because the terms $\mathcal{E}_{\rx}(\mr,\ma_n)/N$ or $\mathcal{E}_{\tx}(\mr,\mb_m)/M$ will be dominated by the factor $J_0(k|\mr-\mr_\star|)$.
\end{remark}

\begin{remark}[Optimal Antenna Arrangement]\label{Remark3}
Based on the \eqref{StructureImagingFunction}, it will be possible to obtain a good result if $\mathcal{E}_{\rx}(\mr,\ma_n)=0$ and $\mathcal{E}_{\tx}(\mr,\mb_m)=0$ simultaneously. On e possible method is to increase $M$ and $N$ but this is difficult due to the simulation configuration. Note that based on the asymptotic form of the Bessel function
\[J_p(k|\mr-\mr_\star|)\approx\sqrt{\frac{2}{k\pi|\mr-\mr_\star|}}\cos\left(k|\mr-\mr_\star|-\frac{p\pi}{2}-\frac{\pi}{4}\right)\quad\text{implies}\quad\lim_{\omega\to+\infty}J_p(k|\mr-\mr_\star|)=0.\]
Hence, by applying extreme frequency $\omega\longrightarrow+\infty$, we can eliminate $\mathcal{E}_{\rx}(\mr,\ma_n)$ and $\mathcal{E}_{\tx}(\mr,\mb_m)$ and correspondingly it is expected that good imaging results can be retrieved. However, this is an ideal condition. Notice that since
\[\mathcal{E}_{\rx}(\mr,\ma_n)=\sum_{n=1}^{N}\sum_{p=-\infty,p\ne0}^{\infty}i^pJ_p(k|\mr-\mr_\star|)e^{ip(\theta_n-\phi)}=\sum_{p=-\infty,p\ne0}^{\infty}i^pJ_p(k|\mr-\mr_\star|)\sum_{n=1}^{N}e^{ip(\theta_n-\phi)},\]
the factor $\mathcal{E}_{\rx}(\mr,\ma_n)$ can be eliminated if $\sum_{n=1}^{N}e^{ip(\theta_n-\phi)}=0$. Based on recent contribution \cite[Observation 4.4]{P-SUB16}, if $M$ and $N$ are even numbers, antennas are located uniformly on a circular array, and positions of two antennas are symmetric about the origin, then it can be expected that we can obtain a good result whose imaging quality is almost same the one obtained with the switching device. We refer to Example \ref{ex3} for a detailed discussion.
\end{remark}

Based on the above discussions, we can conclude the following important result of unique determination.

\begin{corollary}[Unique Identification]
Assume that total number of antennas for receiving and transmitting signals is sufficiently large or the antenna arrangement satisfies the condition in Remark \ref{Remark3}. Then, the location of anomaly can be identified
uniquely through the map of $\mathfrak{F}(\mr)$.
\end{corollary}

%

\section{Simulation Results With Synthetic Data}\label{sec:4}
Results of numerical simulation are provided to support the Theorem \ref{StructureImaging}. In preparation for the simulation, a fixed frequency $f=\SI{1}{\GHz}$ was applied and $S=16$ antennas $\mathcal{D}_s$ were arranged uniformly such that
\[\md_s=\SI{0.09}{\meter}(\cos\theta_s,\sin\theta_s),\quad\theta_s=\frac{3\pi}{2}-\frac{2(s-1)\pi}{S}.\]
We set $\Omega$ as a circle of radius $\SI{0.08}{\meter}$ centered at the origin with $(\epsb,\sigmab)=(20\eps_0,\SI{0.2}{\siemens/\meter})$ and $\Sigma$ as a disk with center $\mr_\star=(\SI{0.01}{\meter},\SI{0.03}{\meter})$, radius $\alpha=\SI{0.01}{\meter}$, and $(\eps_\star,\sigma_\star)=(55\eps_0,\SI{1.2}{\siemens/\meter})$. Here, $\eps_0=\SI{8.854e-12}{\farad/\meter}$ is the vacuum permittivity. We refer to Figure \ref{IllustrationAnomalies} for illustration. With this setting, every measurement data $S_{\scat}(n,m)$ and incident field data $\E(\md_s,\mr)$ were generated utilizing CST STUDIO SUITE.

To compare the accuracy of the imaging results through the maps of $\mathfrak{F}_{\tx}(\mr)$, $\mathfrak{F}_{\rx}(\mr)$, and $\mathfrak{F}(\mr)$, the Jaccard index which measures the similarity between two finite sets $\mathcal{N}_1$ and $\mathcal{N}_2$ is considered such that (see \cite{Jaccard} for instance)
\[\mathcal{J}[\mathcal{N}_1,\mathcal{N}_2](\SI{}{\percent})=\frac{|\mathcal{N}_1\cap\mathcal{N}_2|}{|\mathcal{N}_1\cup\mathcal{N}_2|}\times\SI{100}{\percent}.\]
In this paper, the Jaccard index is calculated by comparing two sets: for given threshold $\zeta\in[0,1]$,
\[\mathcal{N}_1=\left\{\begin{array}{ccl}
\mathcal{N}_{\NF}(\mr)&\text{if}&\mathcal{N}_{\NF}(\mr)\geq\zeta\\
0&\text{if}&\mathcal{N}_{\NF}(\mr)<\zeta
\end{array}\right.\quad\text{and}\quad
\mathcal{N}_2=\frac{|k(\mr)-k|}{\displaystyle\max_{\mr\in\Omega}|k(\mr)-k|},\]
where $\mathcal{N}_{\NF}(\mr)$ is one of the following normalized functions
\[\mathcal{N}_{\tx}(\mr)=\frac{\mathfrak{F}_{\tx}(\mr)}{\displaystyle\max_{\mr\in\Omega}|\mathfrak{F}_{\tx}(\mr)|},\quad\mathcal{N}_{\rx}(\mr)=\frac{\mathfrak{F}_{\rx}(\mr)}{\displaystyle\max_{\mr\in\Omega}|\mathfrak{F}_{\rx}(\mr)|},\quad\text{and}\quad\mathcal{N}(\mr)=\frac{\mathfrak{F}(\mr)}{\displaystyle\max_{\mr\in\Omega}|\mathfrak{F}(\mr)|}.\]

\begin{figure}[h]
\begin{center}
\begin{tikzpicture}[scale=2.4]
\draw[yellow,dashed] (1.1,1.1) -- (-1.1,1.1) -- (-1.1,-1.1) -- (1.1,-1.1) -- cycle;
\foreach \alpha in {0,22.5,...,337.5}
{\draw[green,fill=green] ({cos(\alpha)},{sin(\alpha)}) circle (0.05cm);
\draw[black,fill=black] ({cos(\alpha)},{sin(\alpha)}) circle (0.02cm);}
\node at (0,0) {\footnotesize$(\epsb,\sigmab)=(20\eps_0,\SI{0.2}{\siemens/\meter})$};
\end{tikzpicture}\hfill
\begin{tikzpicture}[scale=2.4]
\draw[yellow,dashed] (1.1,1.1) -- (-1.1,1.1) -- (-1.1,-1.1) -- (1.1,-1.1) -- cycle;
\foreach \alpha in {0,22.5,...,337.5}
{\draw[green,fill=green] ({cos(\alpha)},{sin(\alpha)}) circle (0.05cm);
\draw[black,fill=black] ({cos(\alpha)},{sin(\alpha)}) circle (0.02cm);}
\draw[red,fill=red] (0.1,0.3) circle (0.1cm);
\node at (0,0.1) {\footnotesize$(\eps_\star,\sigma_\star)=(55\eps_0,\SI{1.2}{\siemens/\meter})$};
\end{tikzpicture}\hfill
\begin{tikzpicture}[scale=2.4]
\draw[yellow,dashed] (1.1,1.1) -- (-1.1,1.1) -- (-1.1,-1.1) -- (1.1,-1.1) -- cycle;
\foreach \alpha in {0,22.5,...,337.5}
{\draw[green,fill=green] ({cos(\alpha)},{sin(\alpha)}) circle (0.05cm);
\draw[black,fill=black] ({cos(\alpha)},{sin(\alpha)}) circle (0.02cm);}
\draw[red,fill=red] (0.1,0.3) circle (0.1cm);
\node at (0,0.1) {\footnotesize$(\eps_1,\sigma_1)=(55\eps_0,\SI{1.2}{\siemens/\meter})$};
\draw[orange,fill=orange] (-0.4,-0.2) circle (0.1cm);
\node at (0,-0.4) {\footnotesize$(\eps_2,\sigma_2)=(45\eps_0,\SI{1.0}{\siemens/\meter})$};
\end{tikzpicture}
\caption{\label{IllustrationAnomalies}Illustration of the background (left), single (center) and multiple small anomalies (right).}
\end{center}
\end{figure}
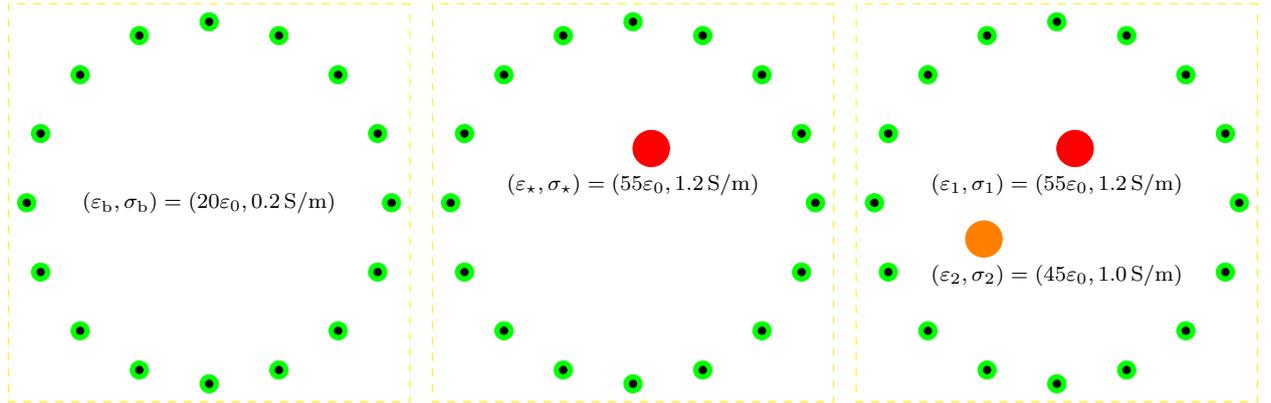



\begin{example}[Small Number of Transmitters]\label{ex1}
Figure \ref{Result1} shows the maps of $\mathfrak{F}_{\tx}(\mr)$, $\mathfrak{F}_{\rx}(\mr)$, $\mathfrak{F}(\mr)$, and Jaccard index with transmitting antenna setting $\mathbf{B}_1=\set{\mathcal{D}_m,m=12,13,14}$ and receiving antenna settings $\mathbf{A}_1=\set{\mathcal{D}_n:n=4,5,6}$, $\mathbf{A}_2=\set{\mathcal{D}_n:n=3,4,5,6,7}$, $\mathbf{A}_3=\set{\mathcal{D}_n:n=2,3,4,5,6,7,8}$, and $\mathbf{A}_4=\set{\mathcal{D}_n:n=1,2,3,4,5,6,7,8,9}$. Based on the result, we can examine that it is very difficult to identify $\mr_\star$ with settings $\mathbf{A}_j\cup\mathbf{B}_1$, $j=1,2,3,4$, due to the blurring effect in the neighborhood of $\mr_\star$. This means that even with a large number of receive antennas, it will be very difficult to obtain good results using a small number of transmit antennas.
\end{example}

\begin{example}[Increasing Total Number of Transmitters and Receivers]\label{ex2}
Figure \ref{Result2} shows the maps of $\mathfrak{F}(\mr)$ with transmitting antenna setting $\mathbf{B}_2=\set{\mathcal{D}_m,m=10,11,12,13,14,15,16}$ and same receiving antenna settings $\mathbf{A}_j$, $j=1,2,3,4$. Opposite to the results with $\mathbf{B}_1$, it is possible to identify $\mr_\star$ with with settings $\mathbf{A}_j\cup\mathbf{B}_2$, $j=2,3,4$ while is still difficult to identify due to the blurring effect in the neighborhood of $\mr_\star$ with the setting $\mathbf{A}_1\cup\mathbf{B}_2$. Therefore, we conclude that increasing the total number of transmitting and receiving antennas will guarantee good results.
\end{example}


\begin{figure}[h]
  \centering
  \includegraphics[width=0.25\textwidth]{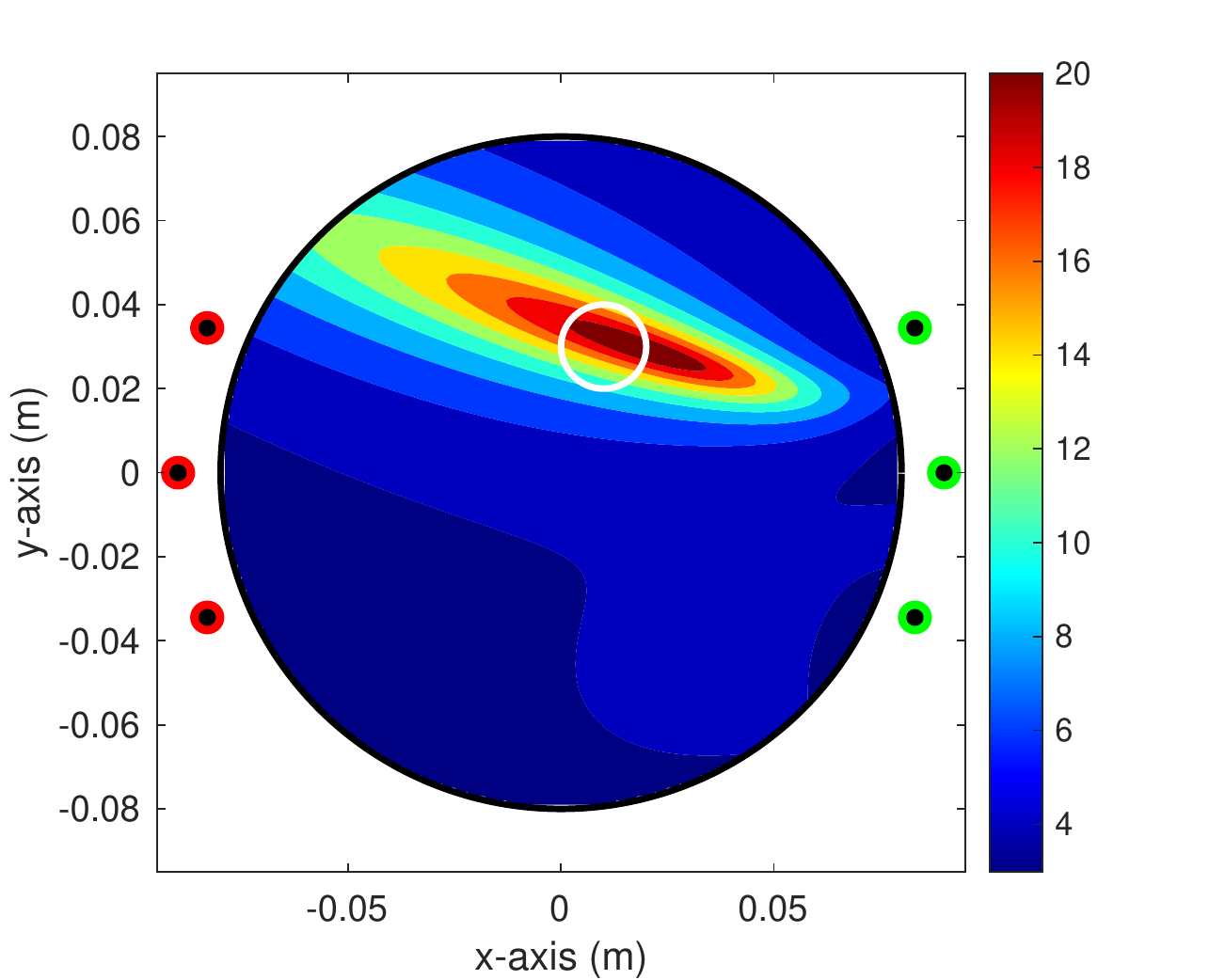}\hfill
  \includegraphics[width=0.25\textwidth]{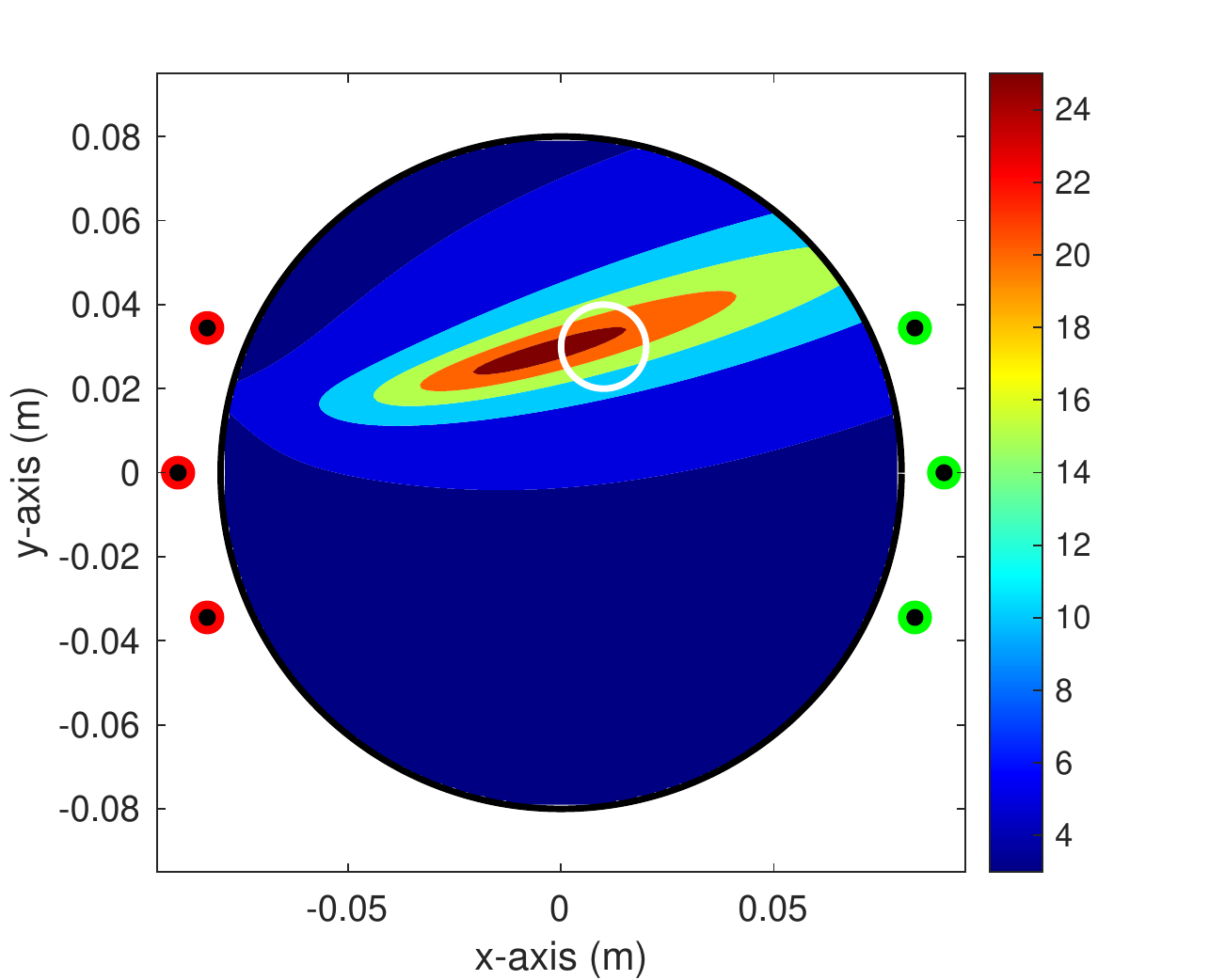}\hfill
  \includegraphics[width=0.25\textwidth]{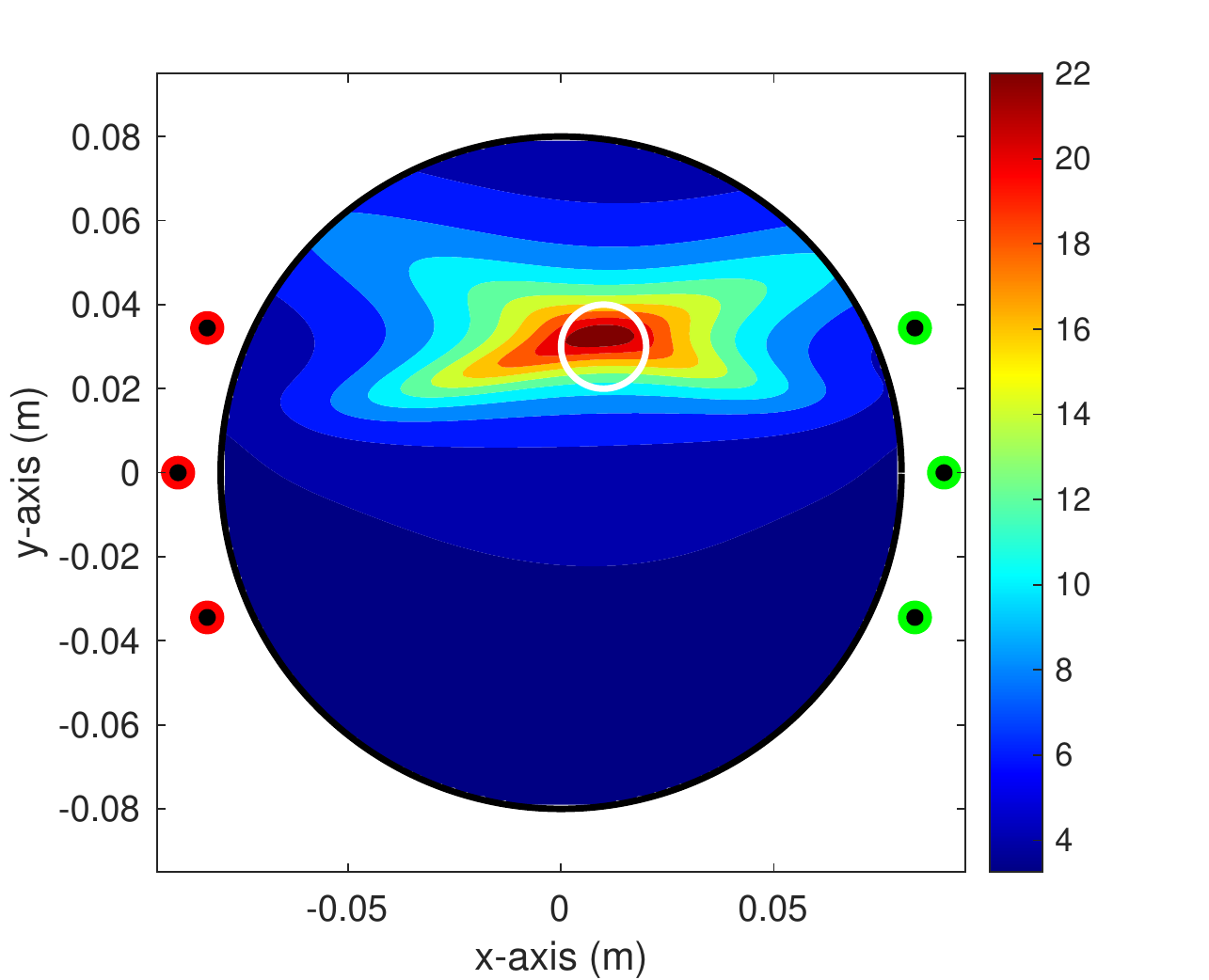}\hfill
  \includegraphics[width=0.25\textwidth]{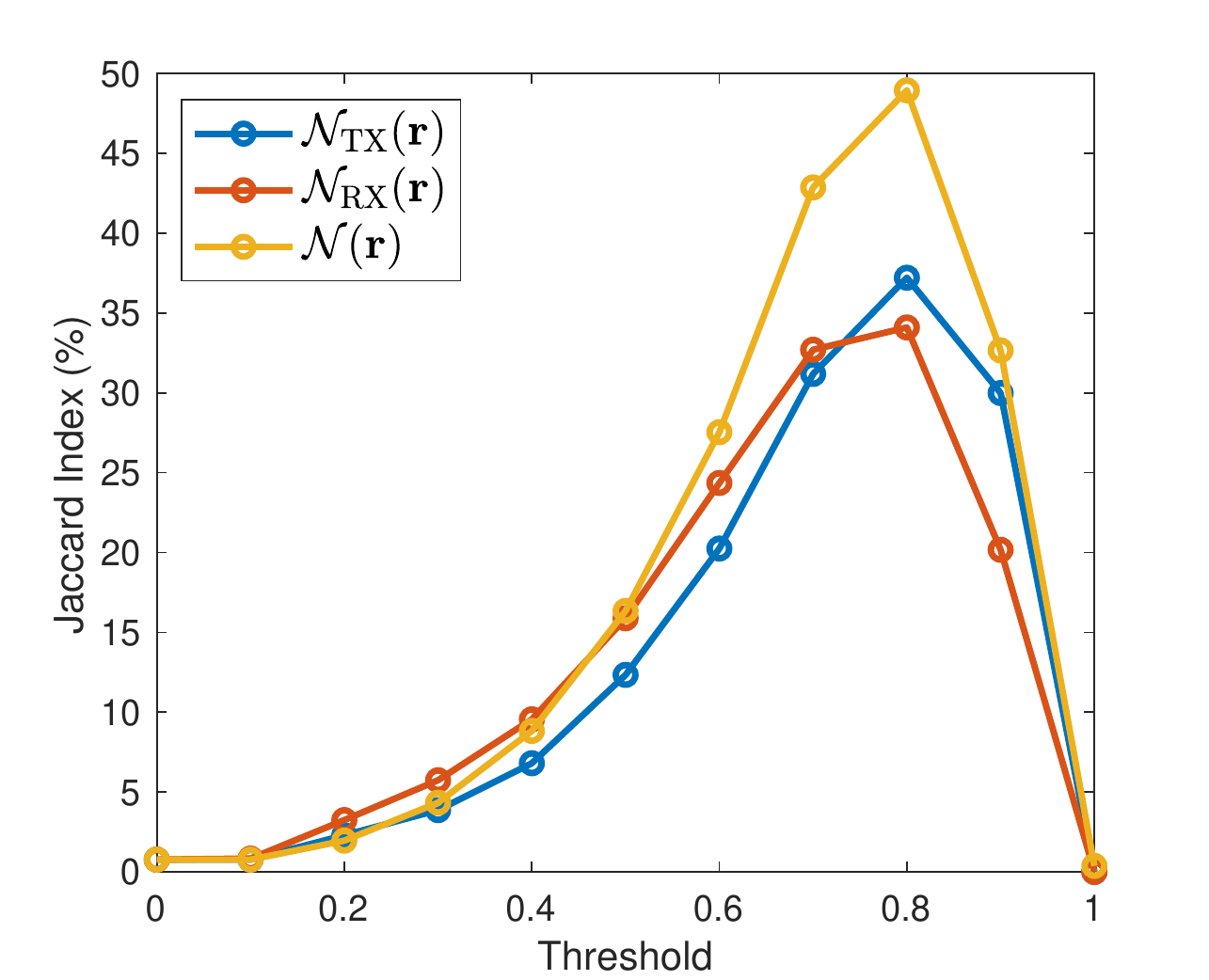}\\
  \includegraphics[width=0.25\textwidth]{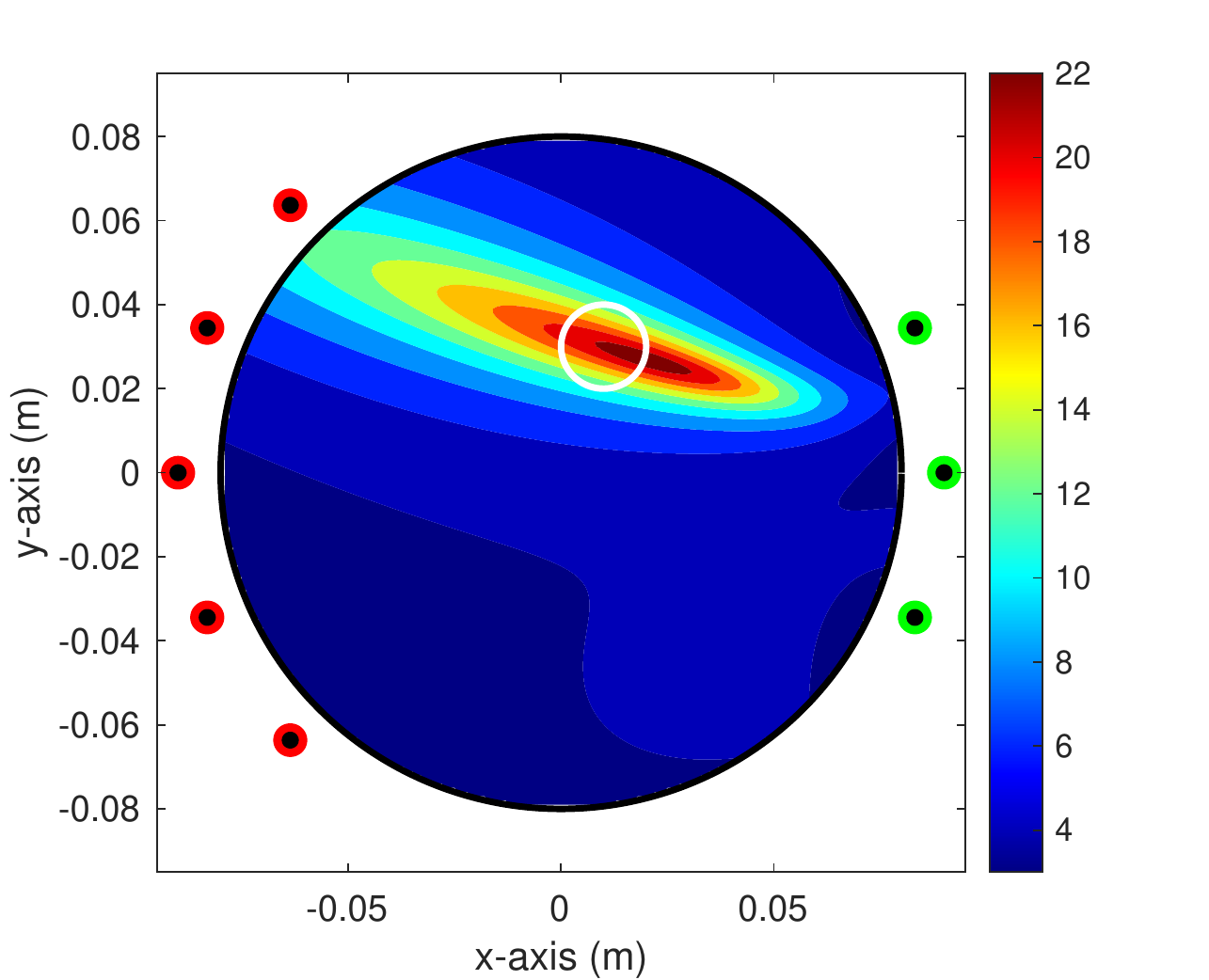}\hfill
  \includegraphics[width=0.25\textwidth]{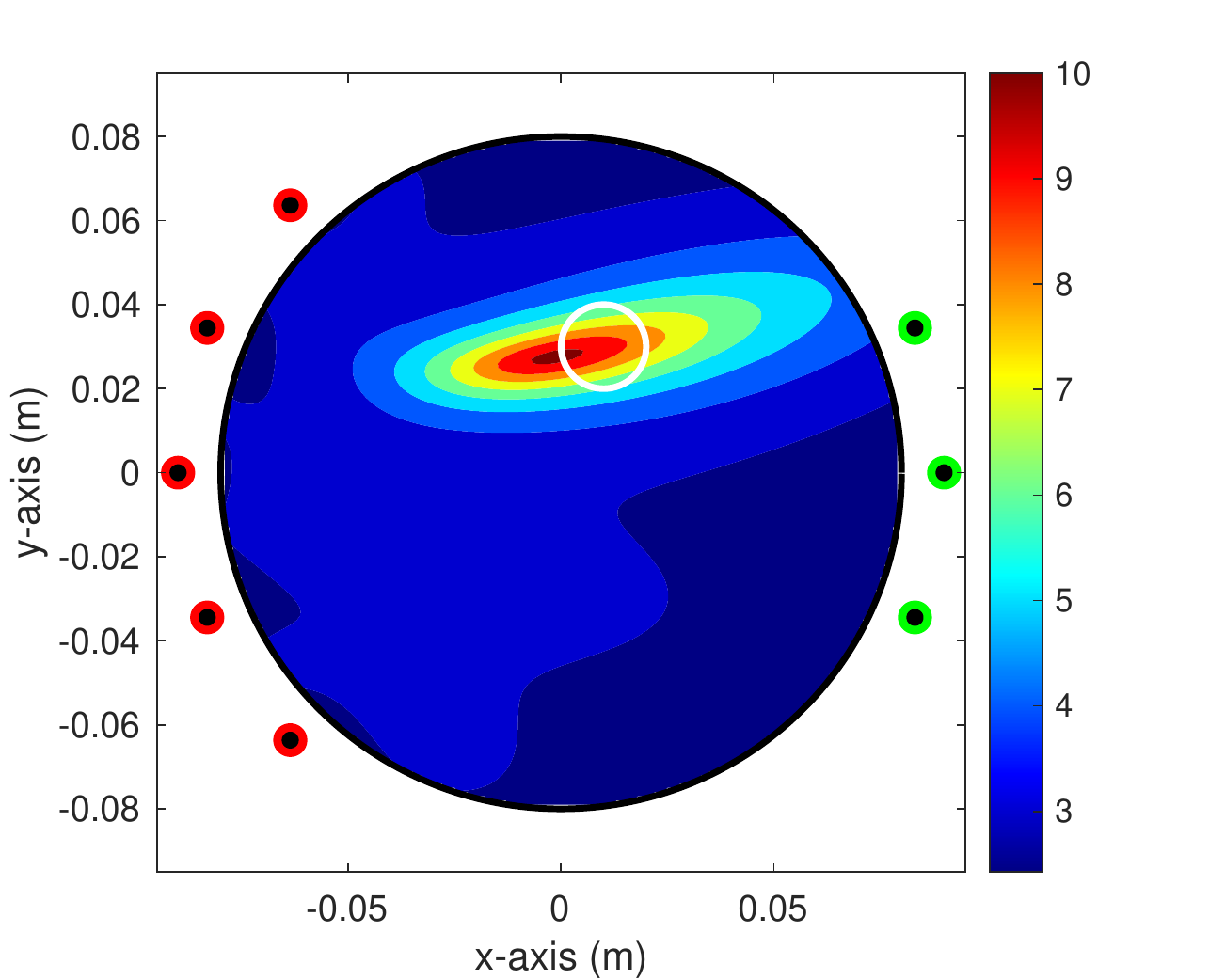}\hfill
  \includegraphics[width=0.25\textwidth]{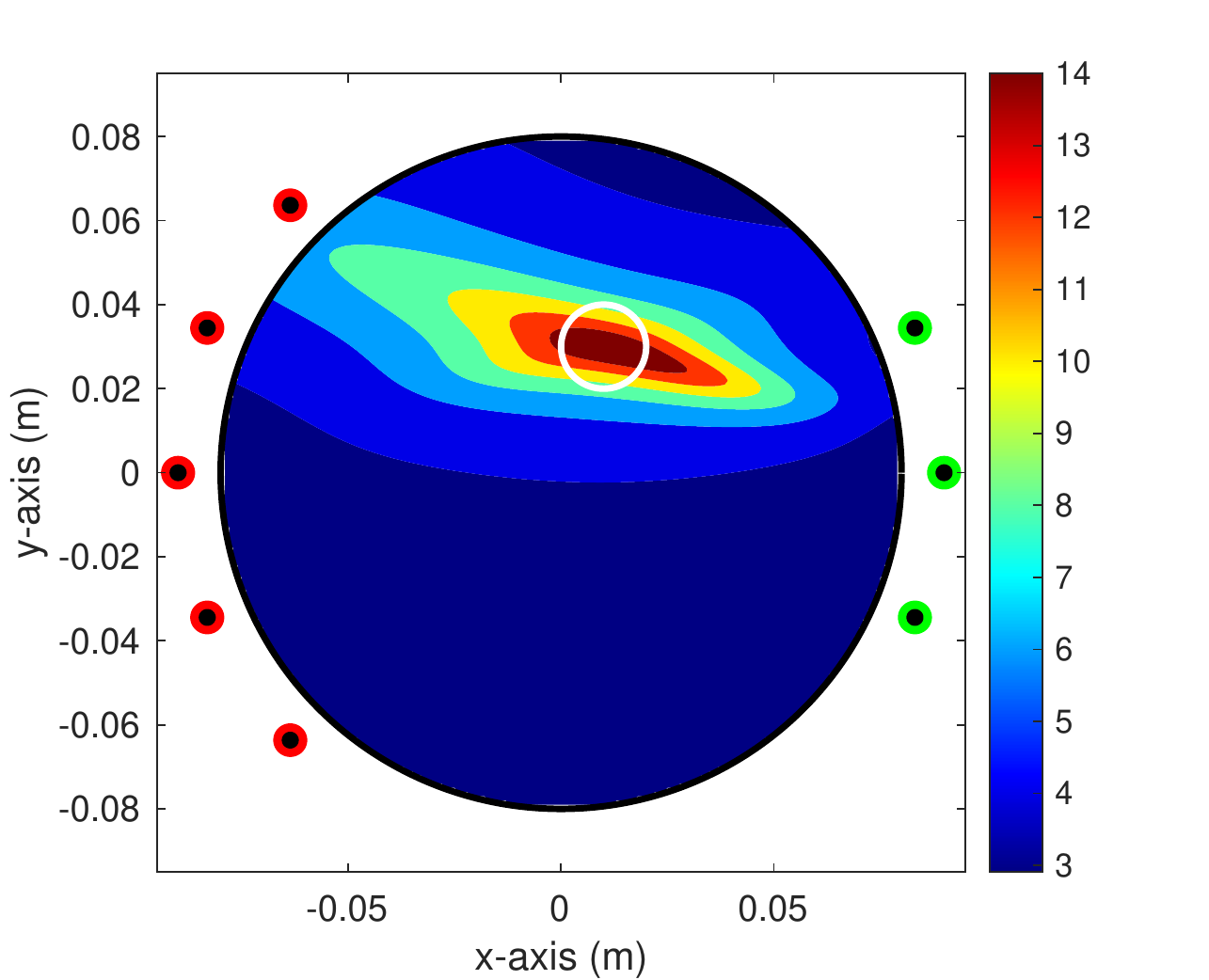}\hfill
  \includegraphics[width=0.25\textwidth]{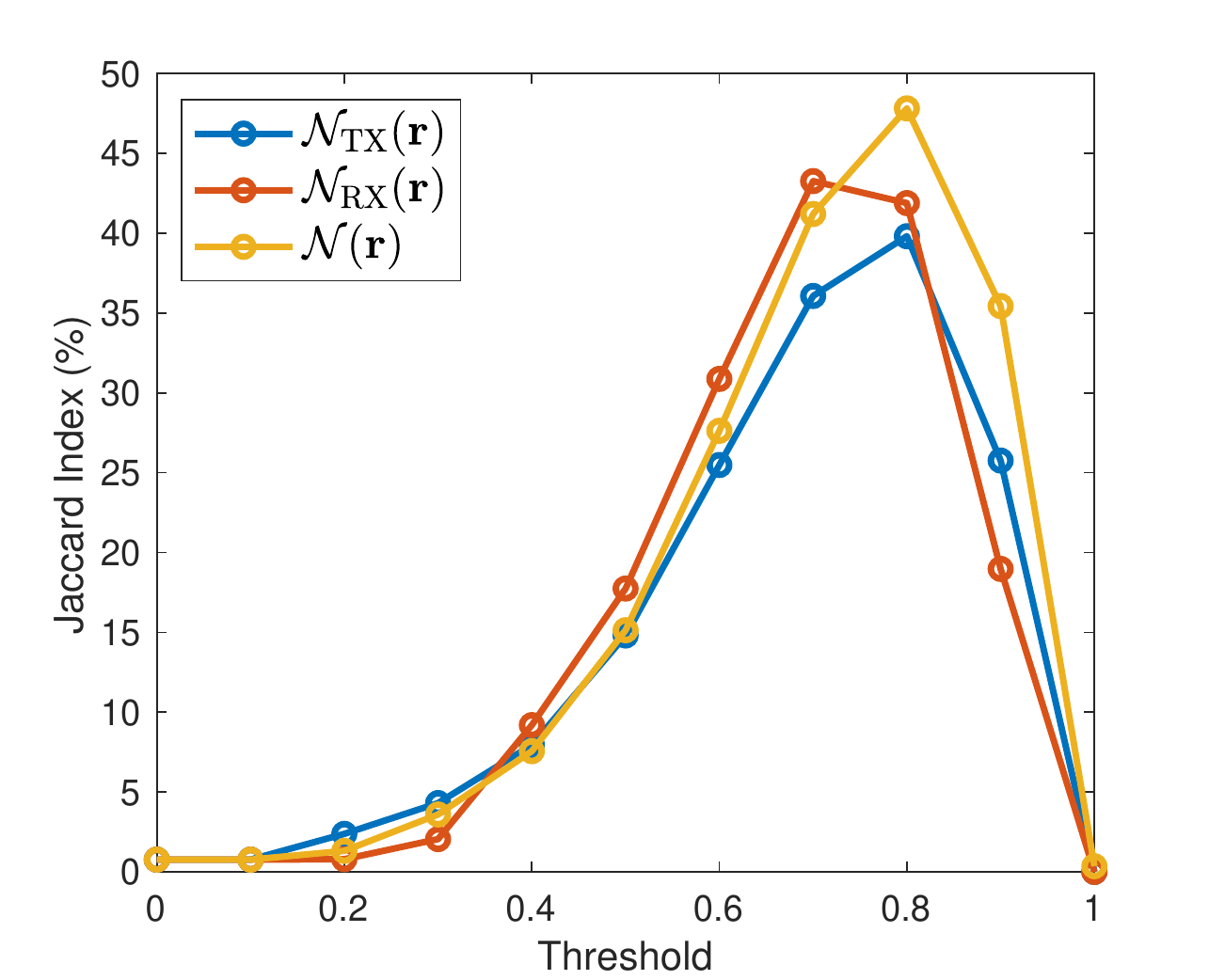}\\
  \includegraphics[width=0.25\textwidth]{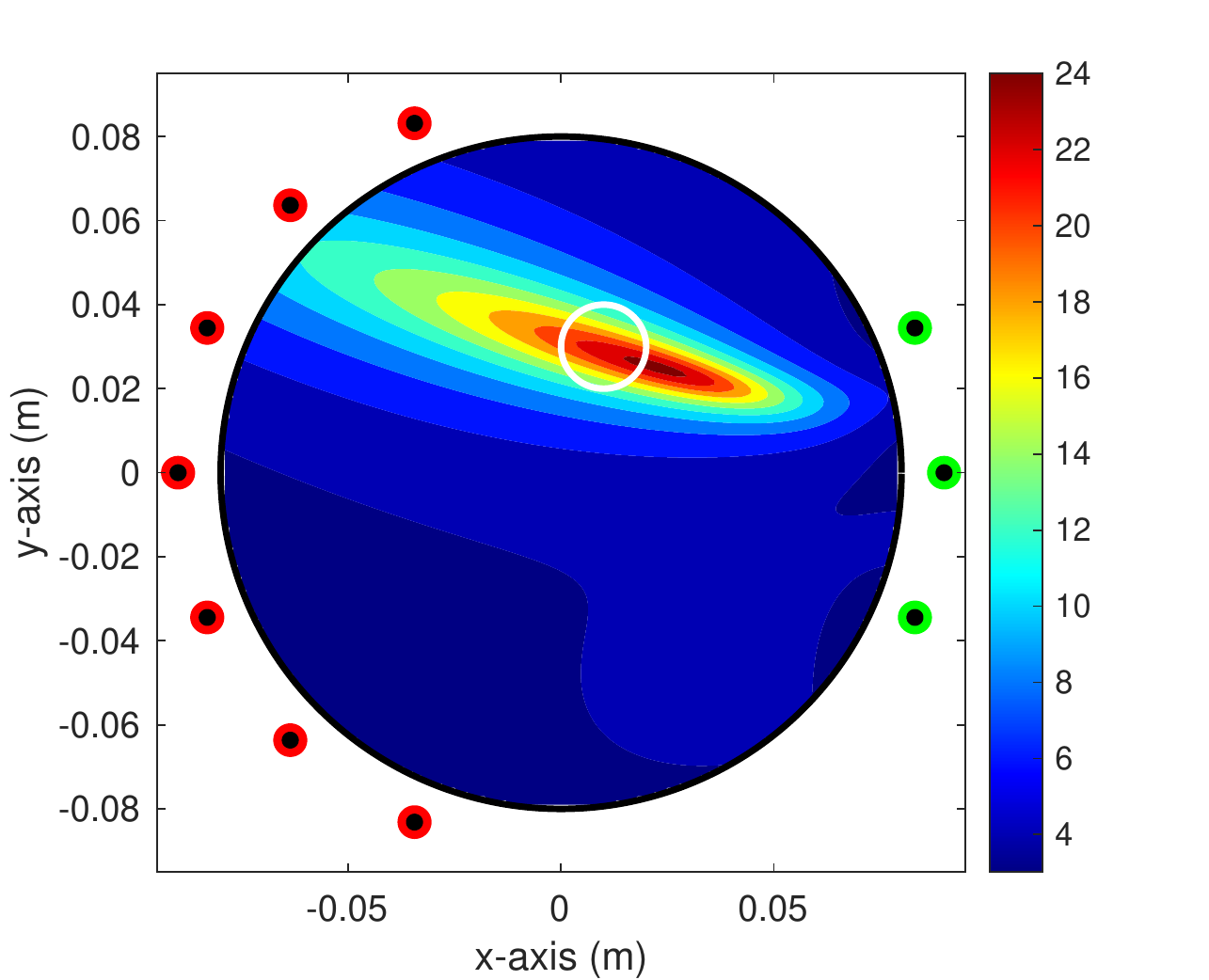}\hfill
  \includegraphics[width=0.25\textwidth]{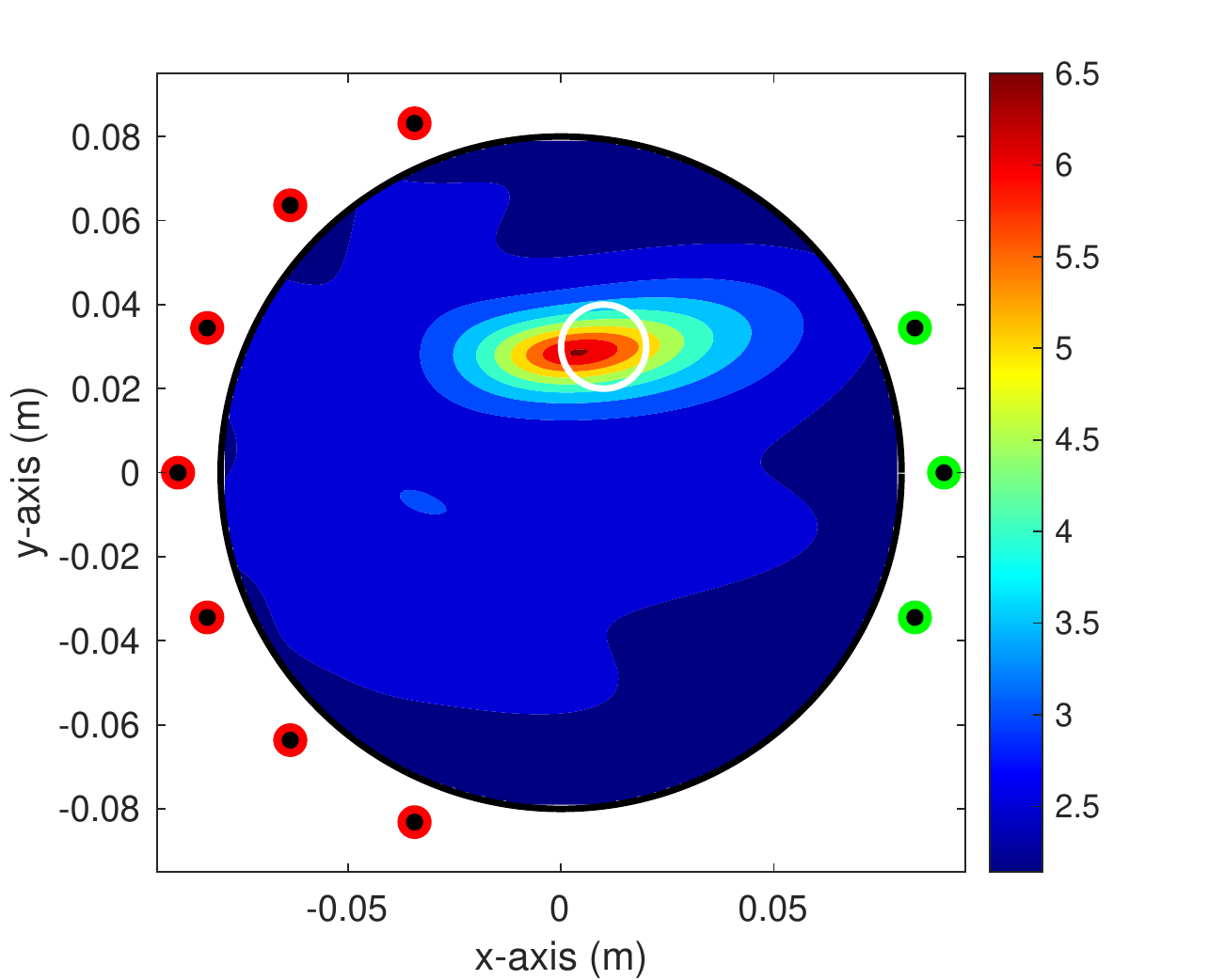}\hfill
  \includegraphics[width=0.25\textwidth]{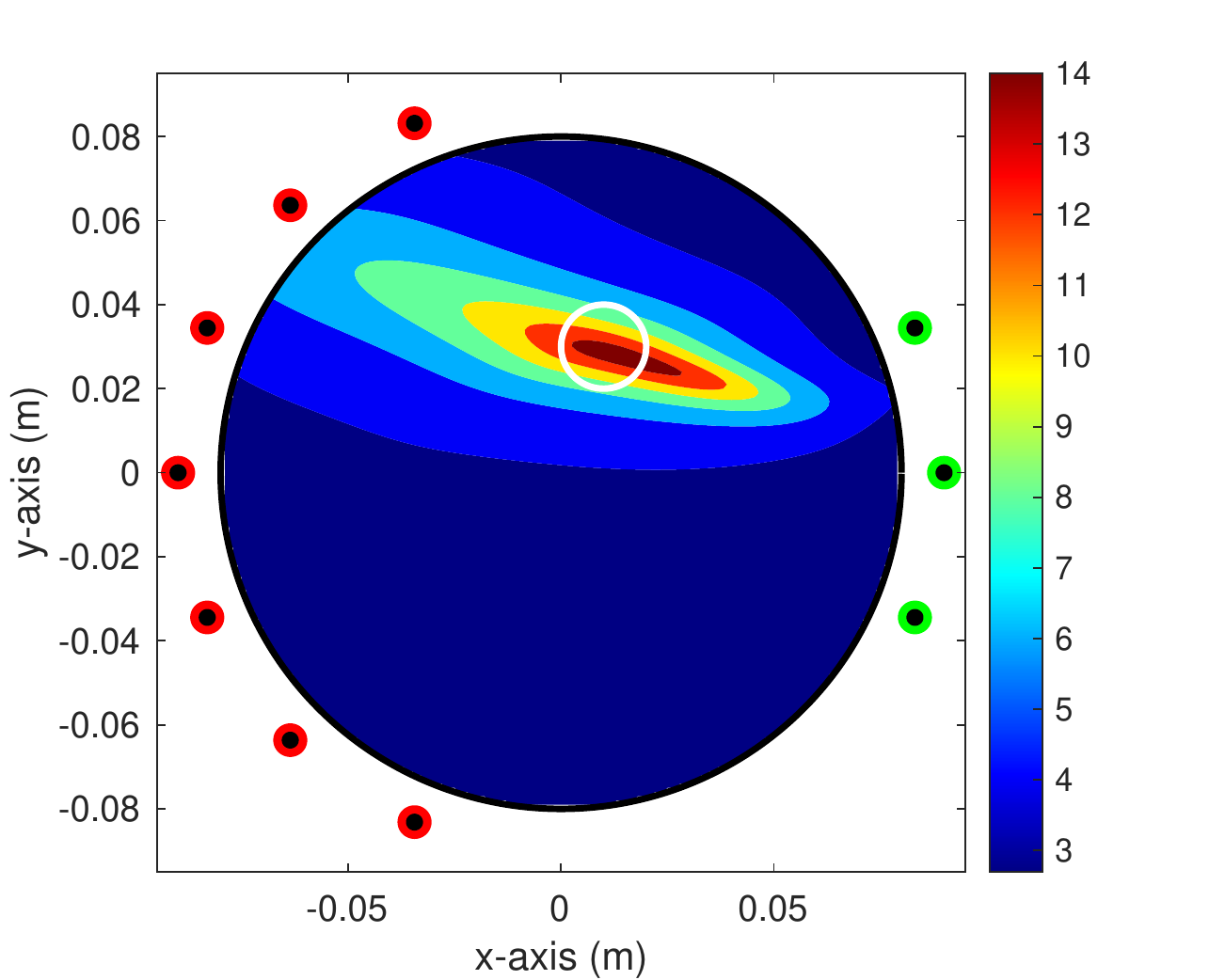}\hfill
  \includegraphics[width=0.25\textwidth]{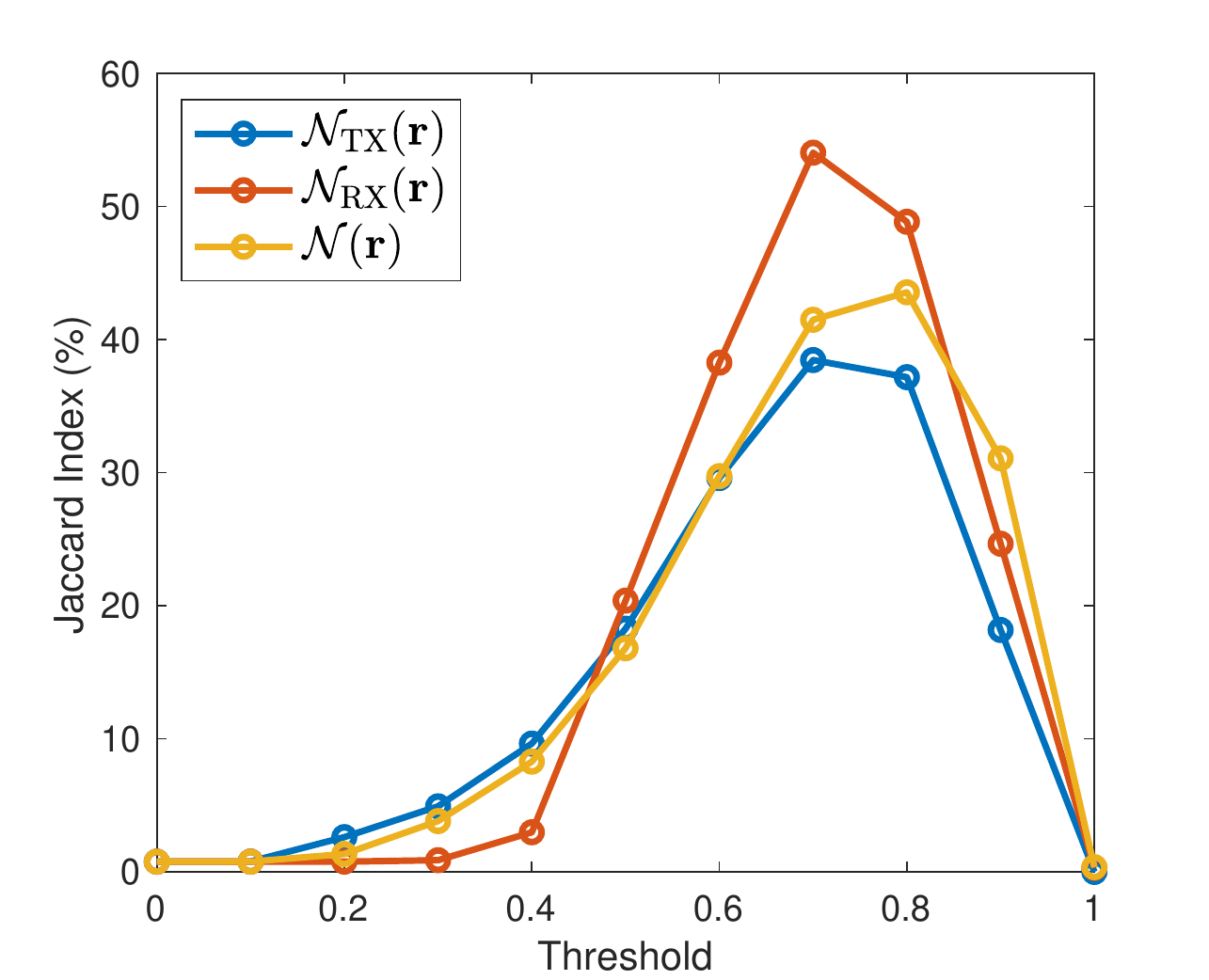}\\
  \includegraphics[width=0.25\textwidth]{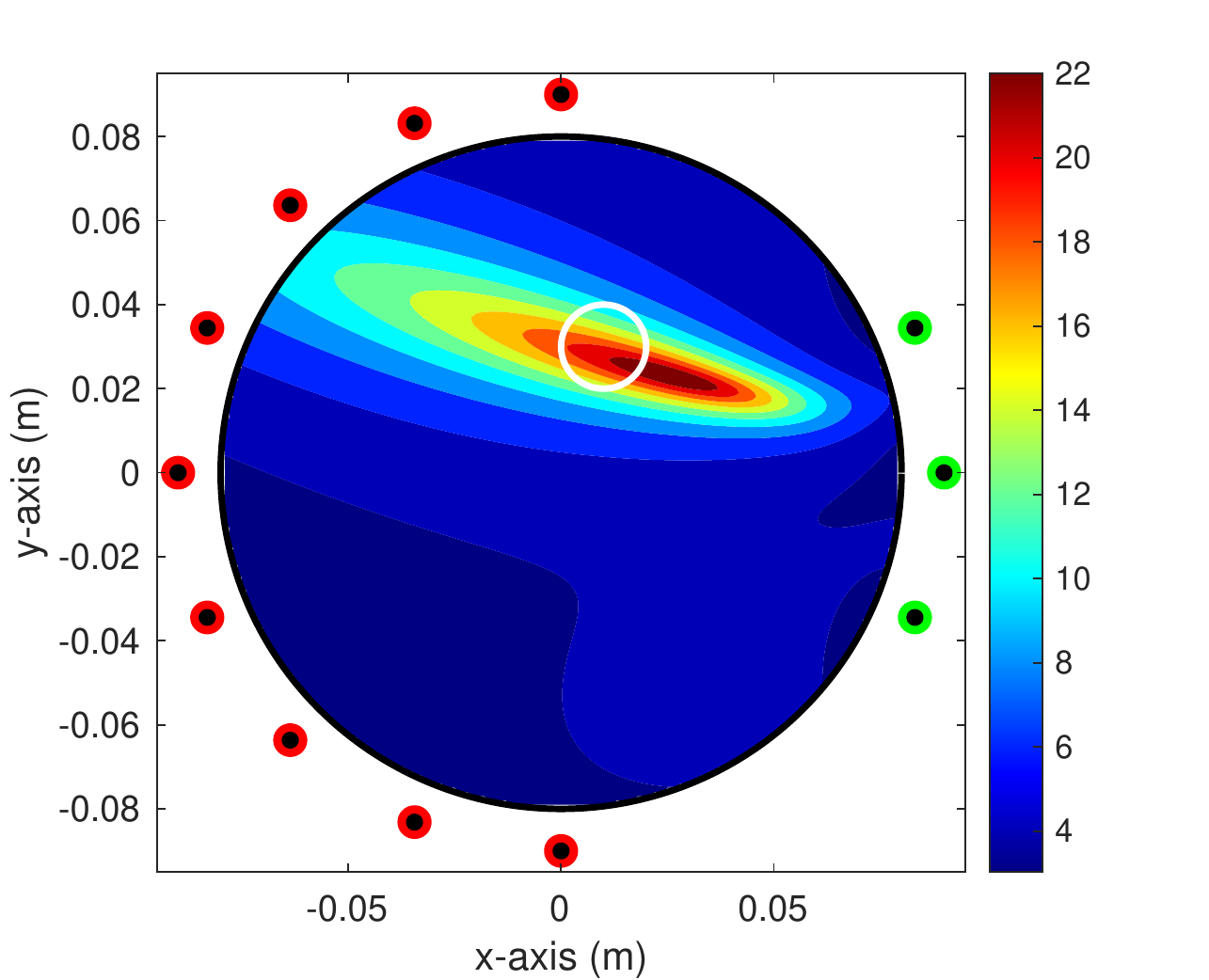}\hfill
  \includegraphics[width=0.25\textwidth]{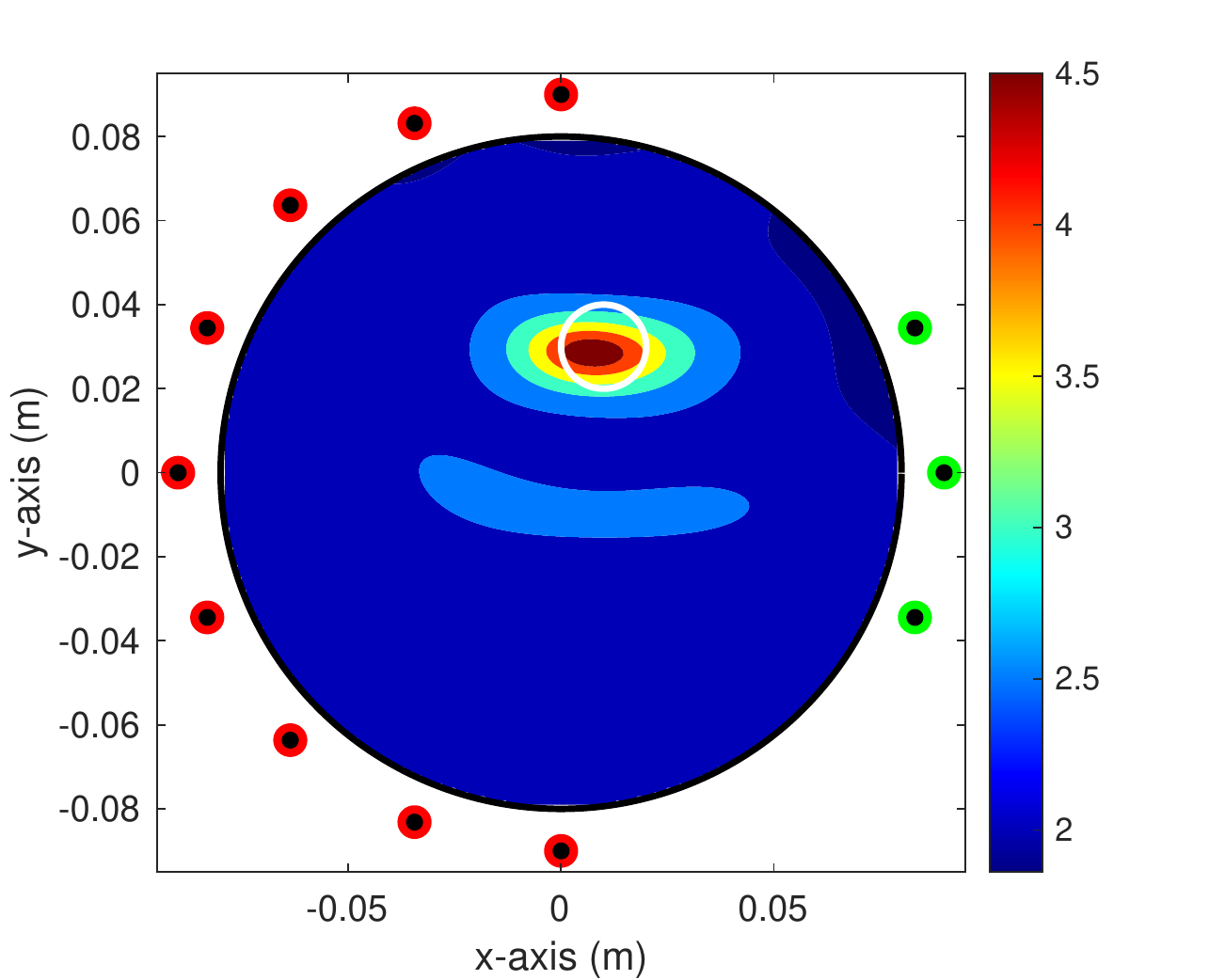}\hfill
  \includegraphics[width=0.25\textwidth]{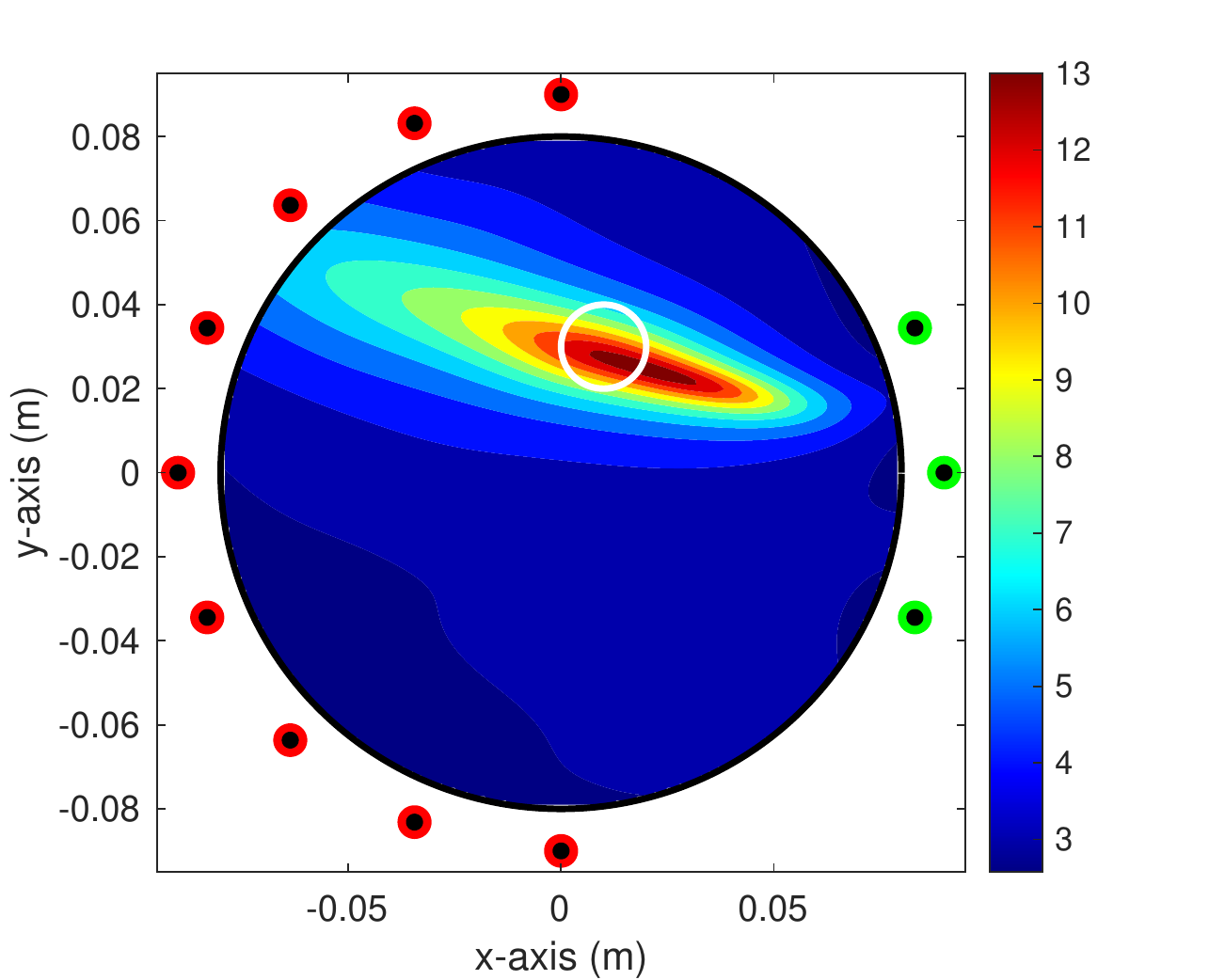}\hfill
  \includegraphics[width=0.25\textwidth]{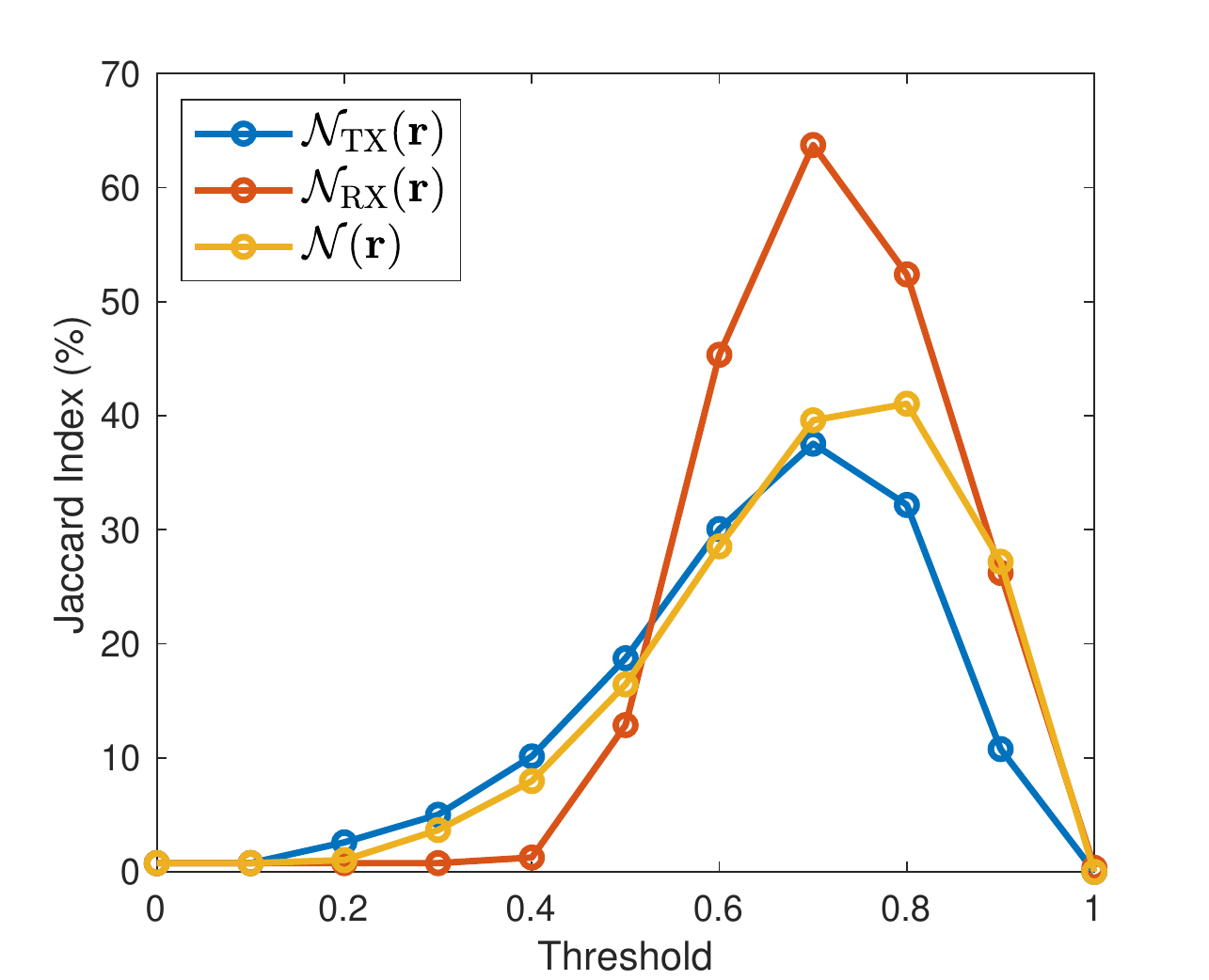}
  \caption{\label{Result1}(Example \ref{ex1}) Maps of $\mathfrak{F}_{\tx}(\mr)$ (first column), $\mathfrak{F}_{\rx}(\mr)$ (second column), $\mathfrak{F}(\mr)$ (third column), and Jaccard index (fourth column). Green and red colored circles describe the location of transmitters and receivers, respectively.}
\end{figure}

\begin{figure}[h]
  \centering
  \includegraphics[width=0.25\textwidth]{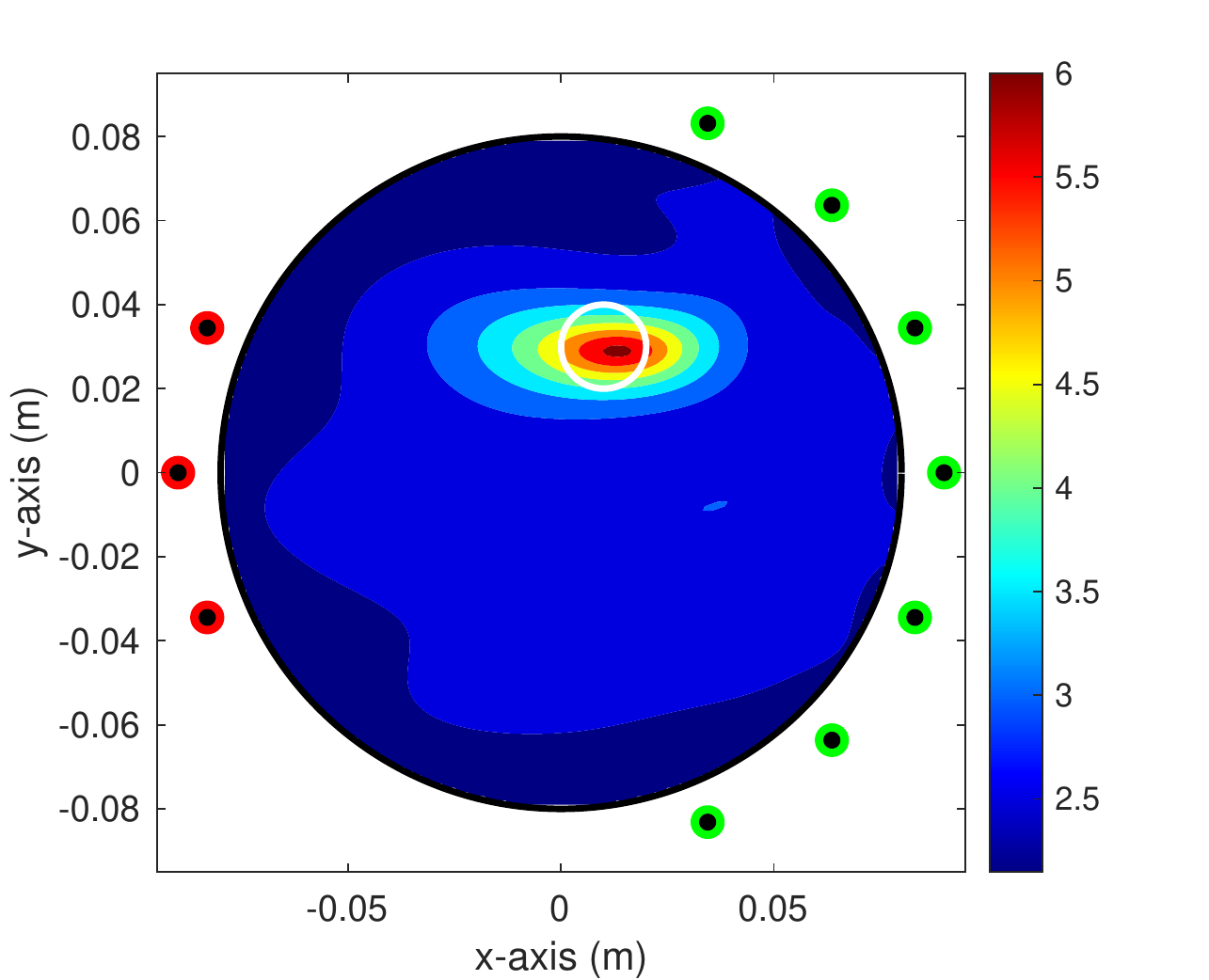}\hfill
  \includegraphics[width=0.25\textwidth]{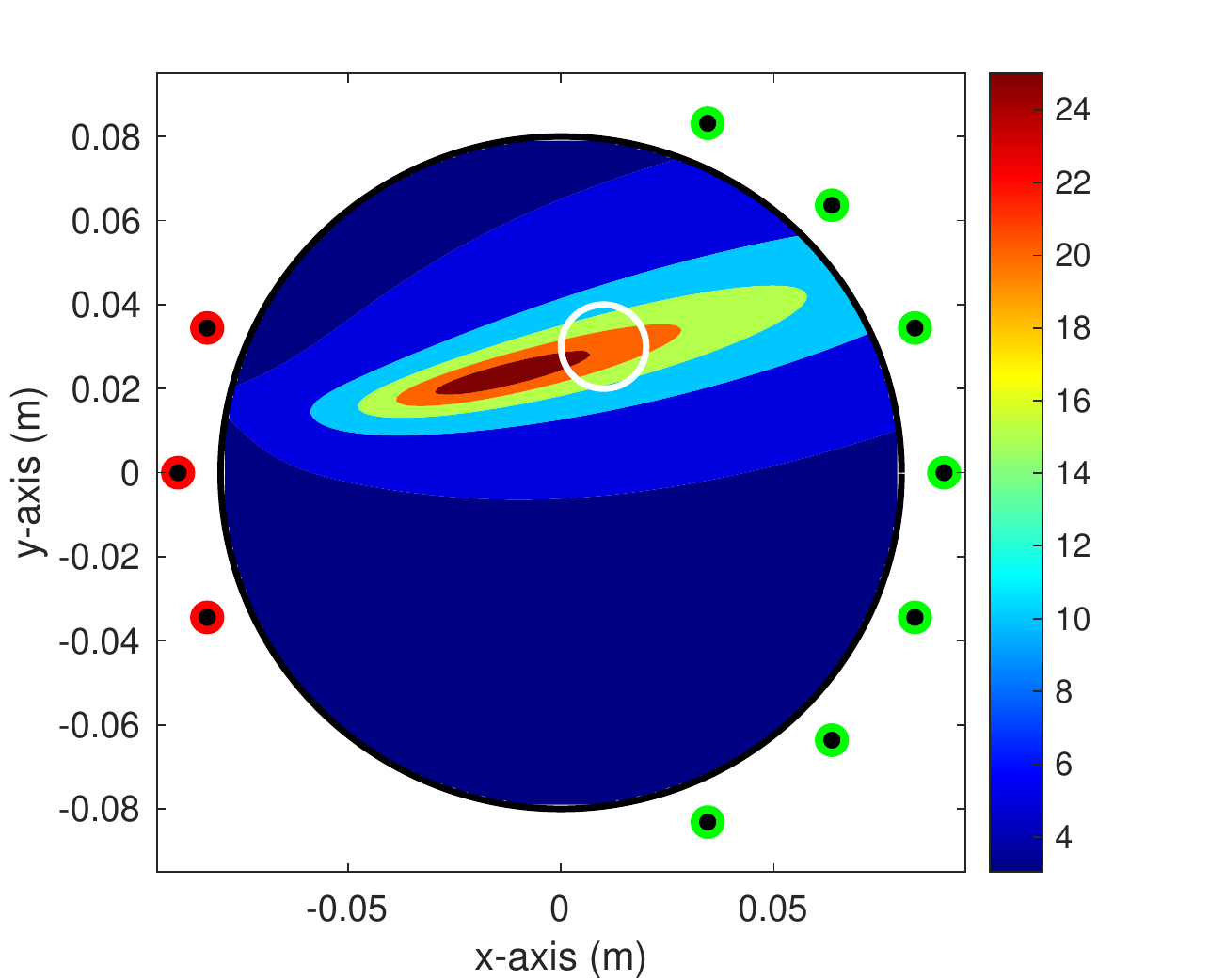}\hfill
  \includegraphics[width=0.25\textwidth]{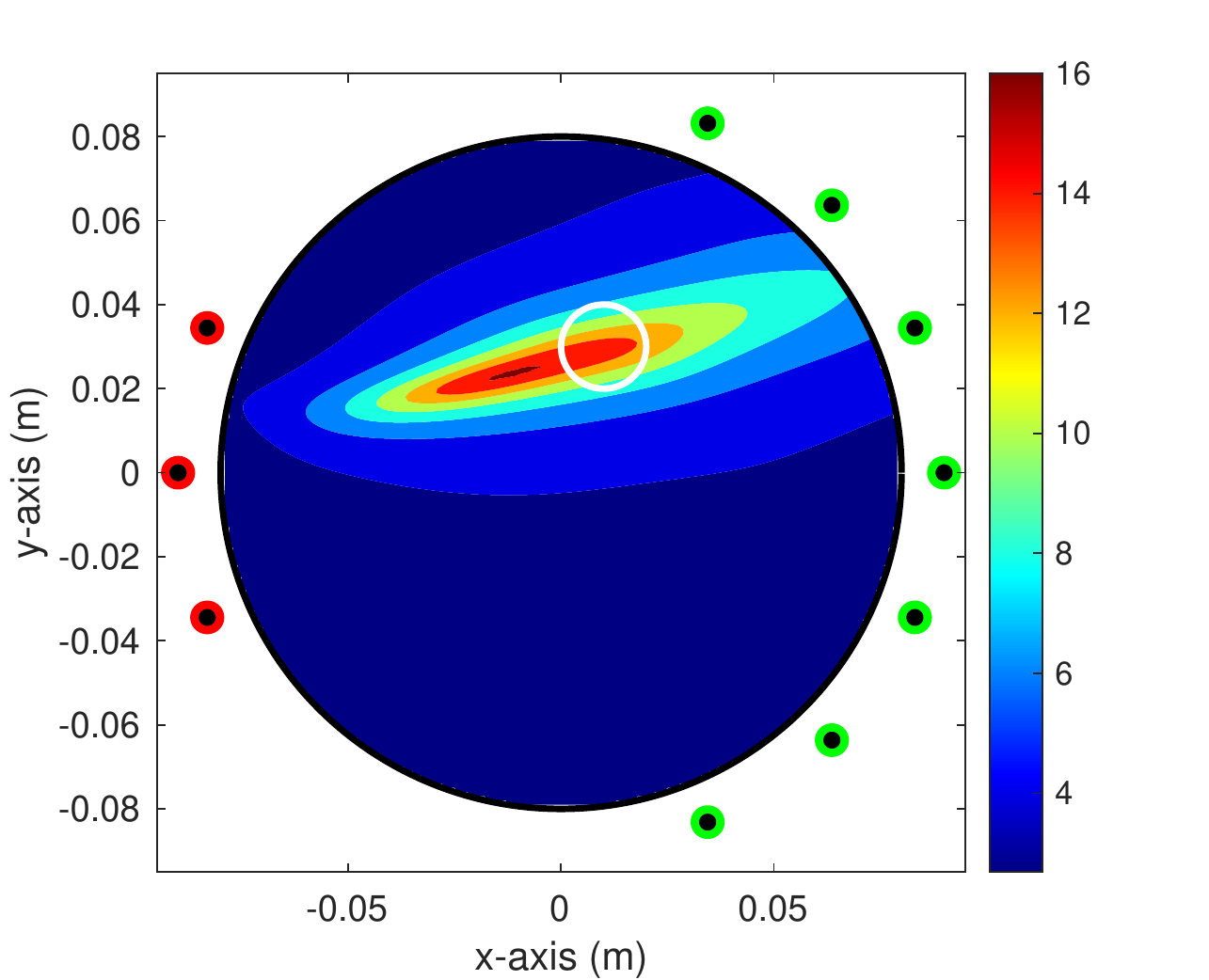}\hfill
  \includegraphics[width=0.25\textwidth]{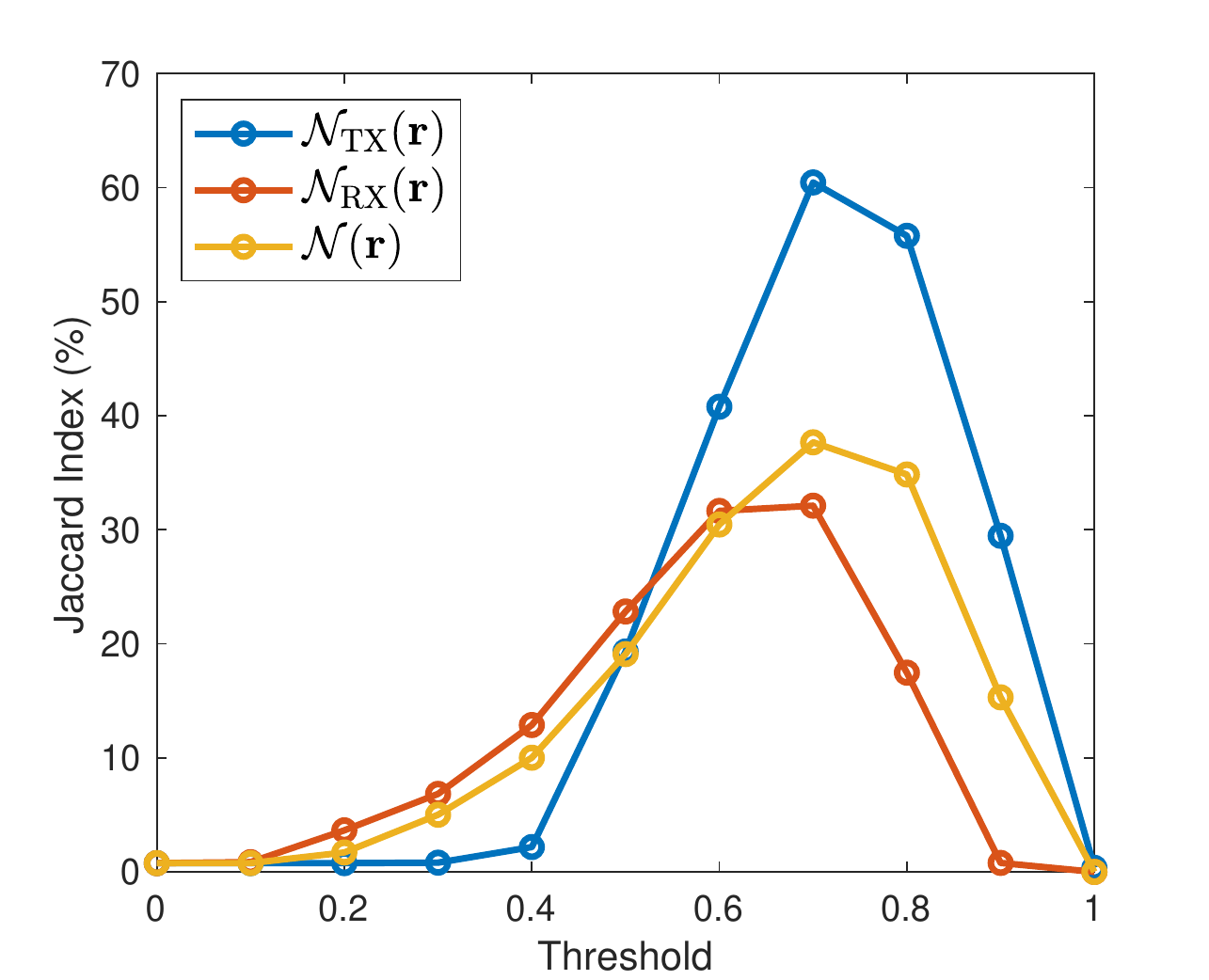}\\
  \includegraphics[width=0.25\textwidth]{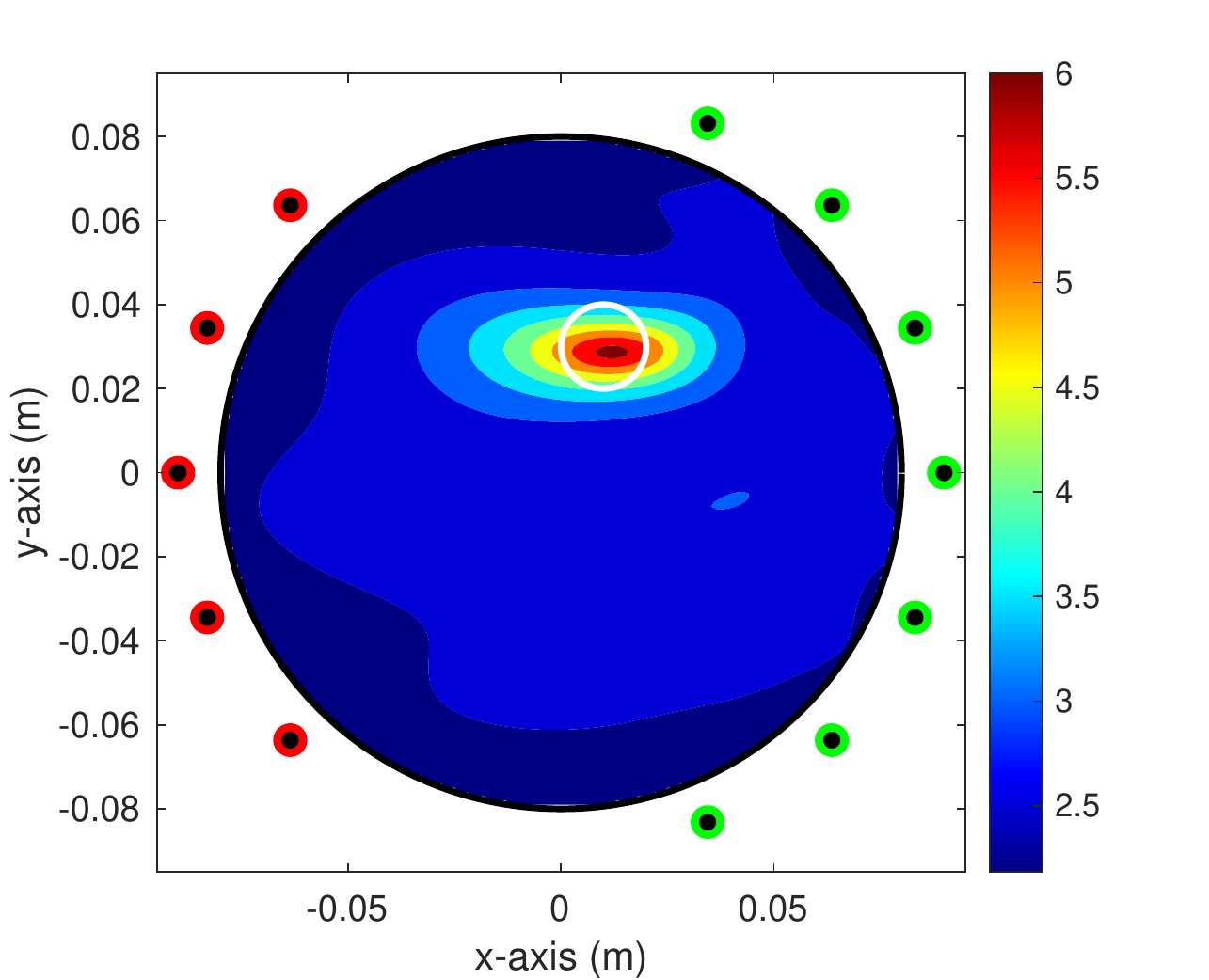}\hfill
  \includegraphics[width=0.25\textwidth]{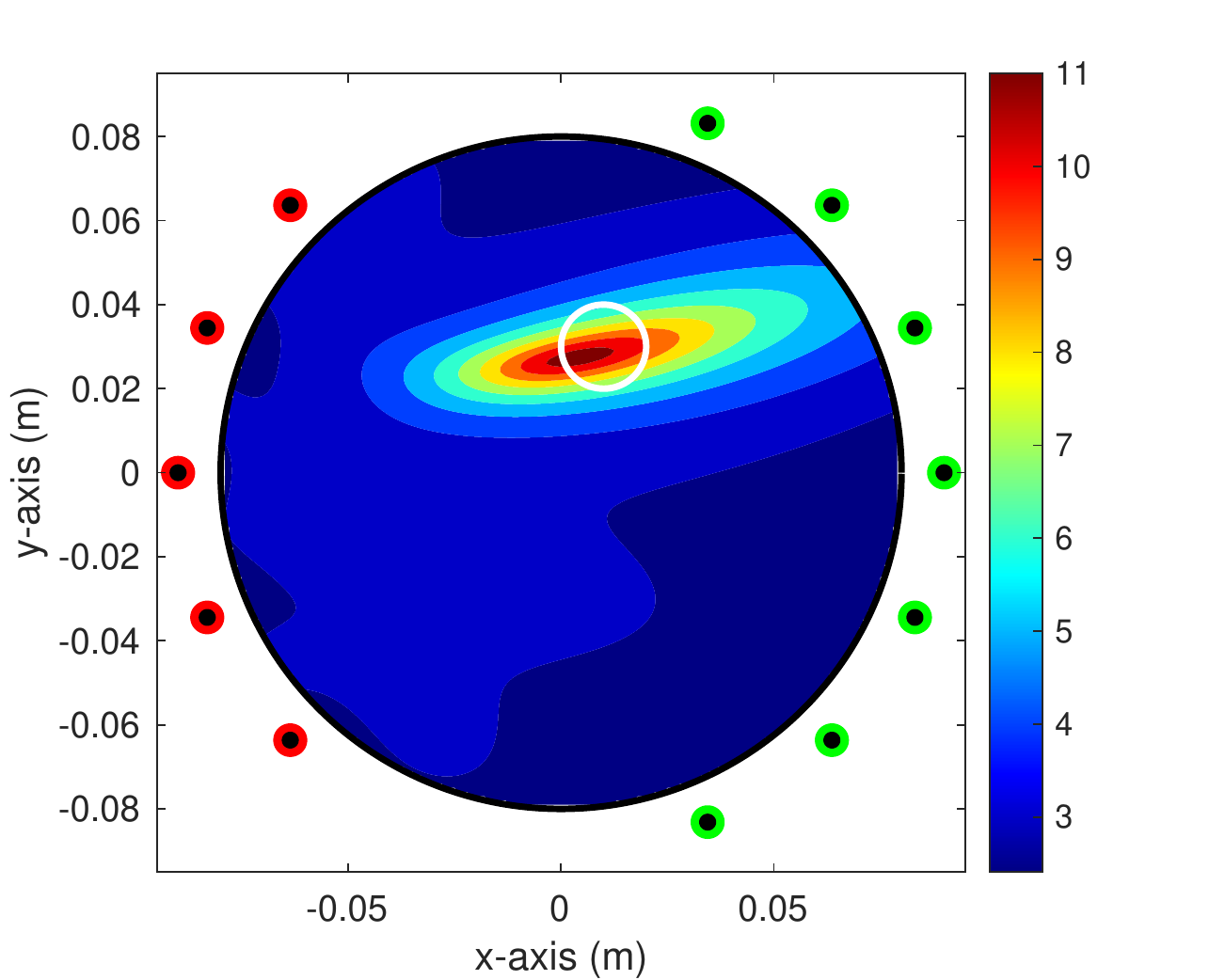}\hfill
  \includegraphics[width=0.25\textwidth]{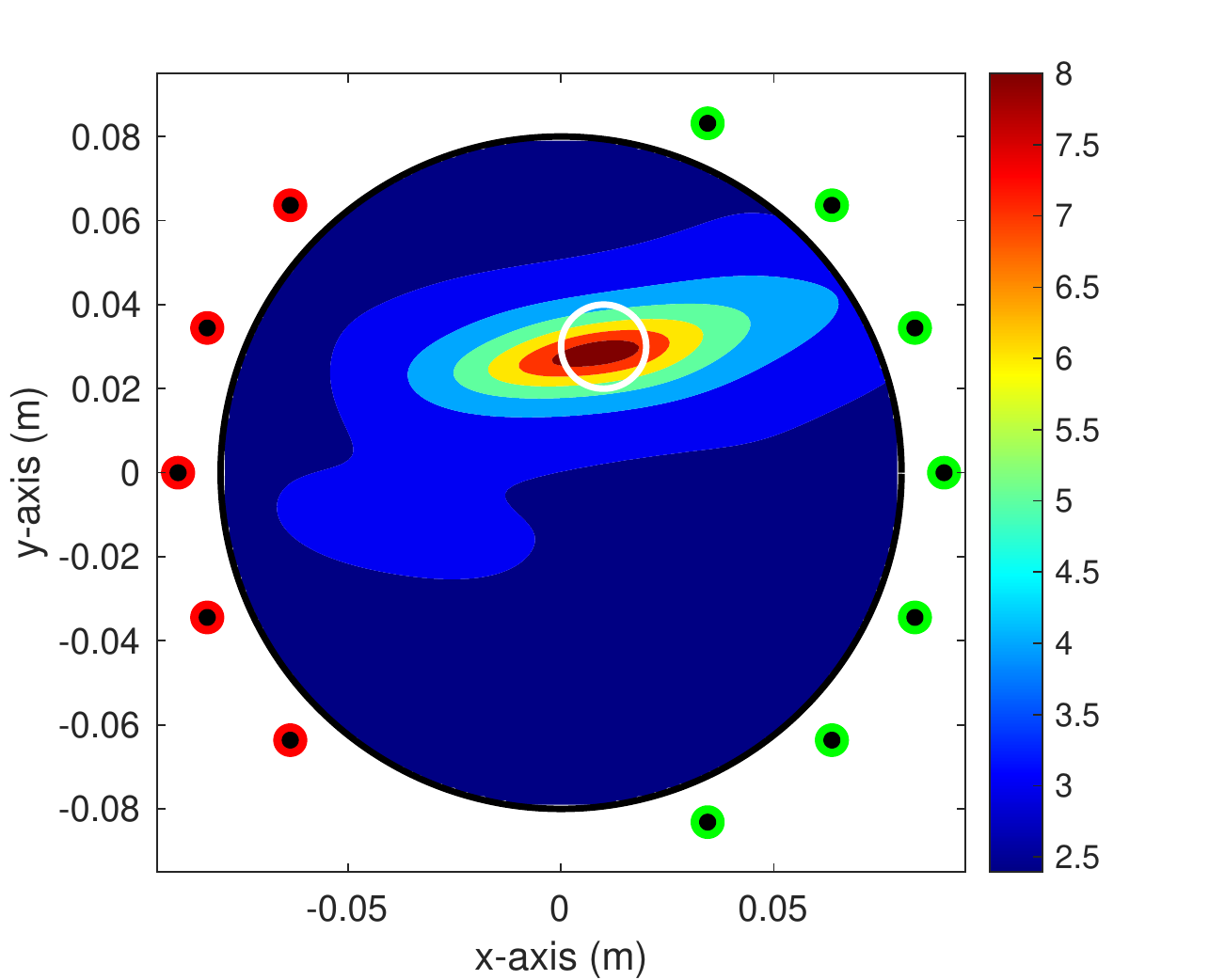}\hfill
  \includegraphics[width=0.25\textwidth]{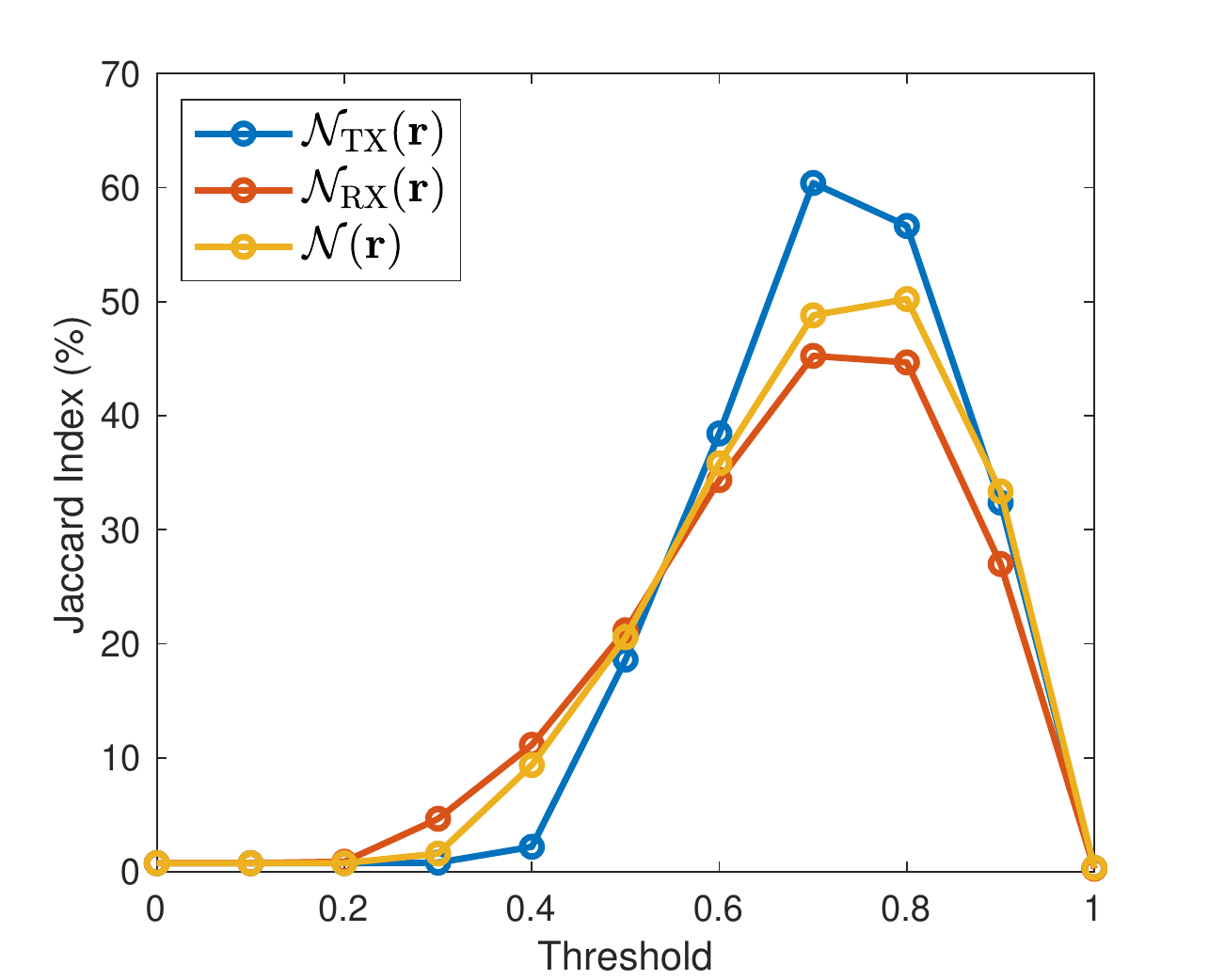}\\
  \includegraphics[width=0.25\textwidth]{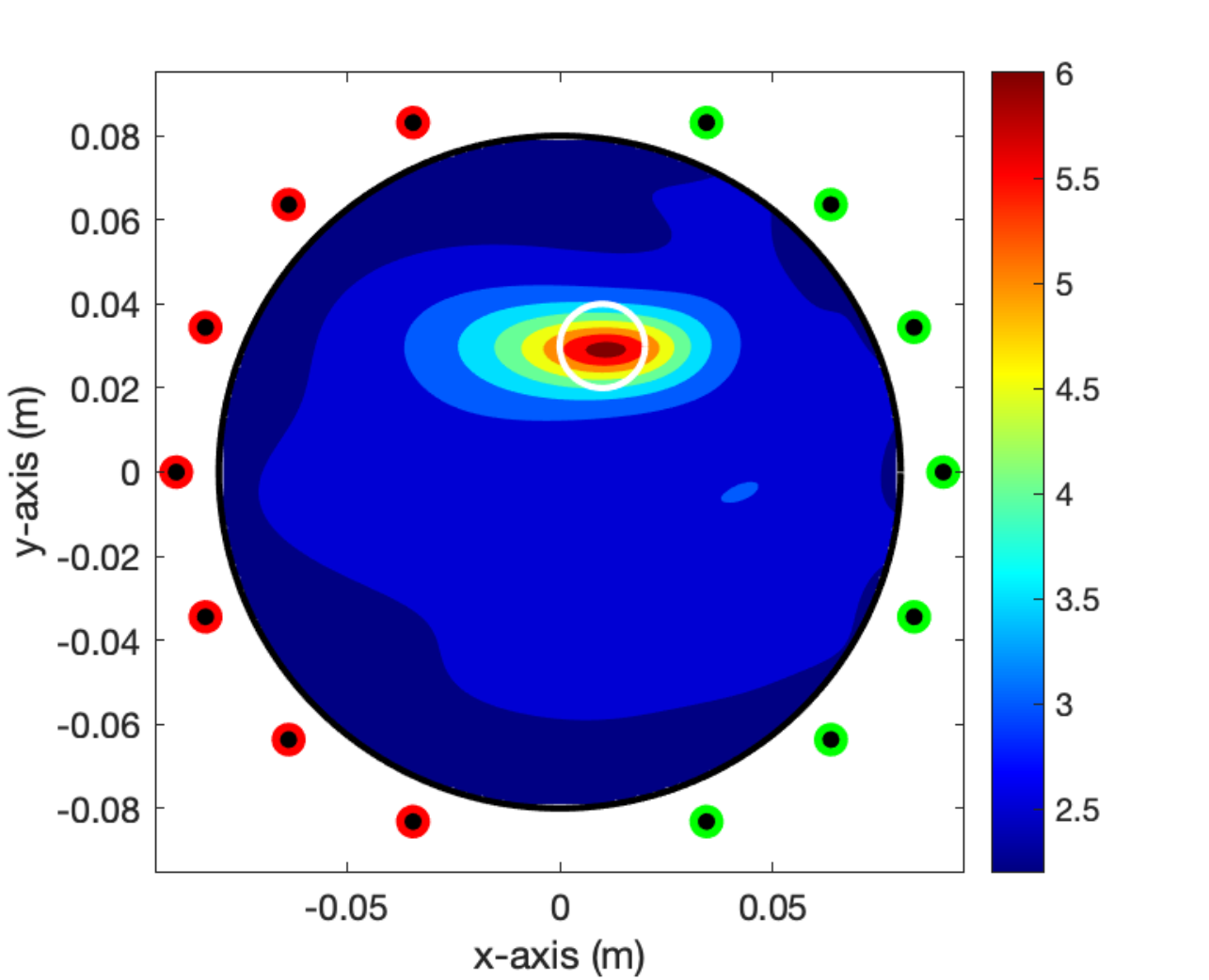}\hfill
  \includegraphics[width=0.25\textwidth]{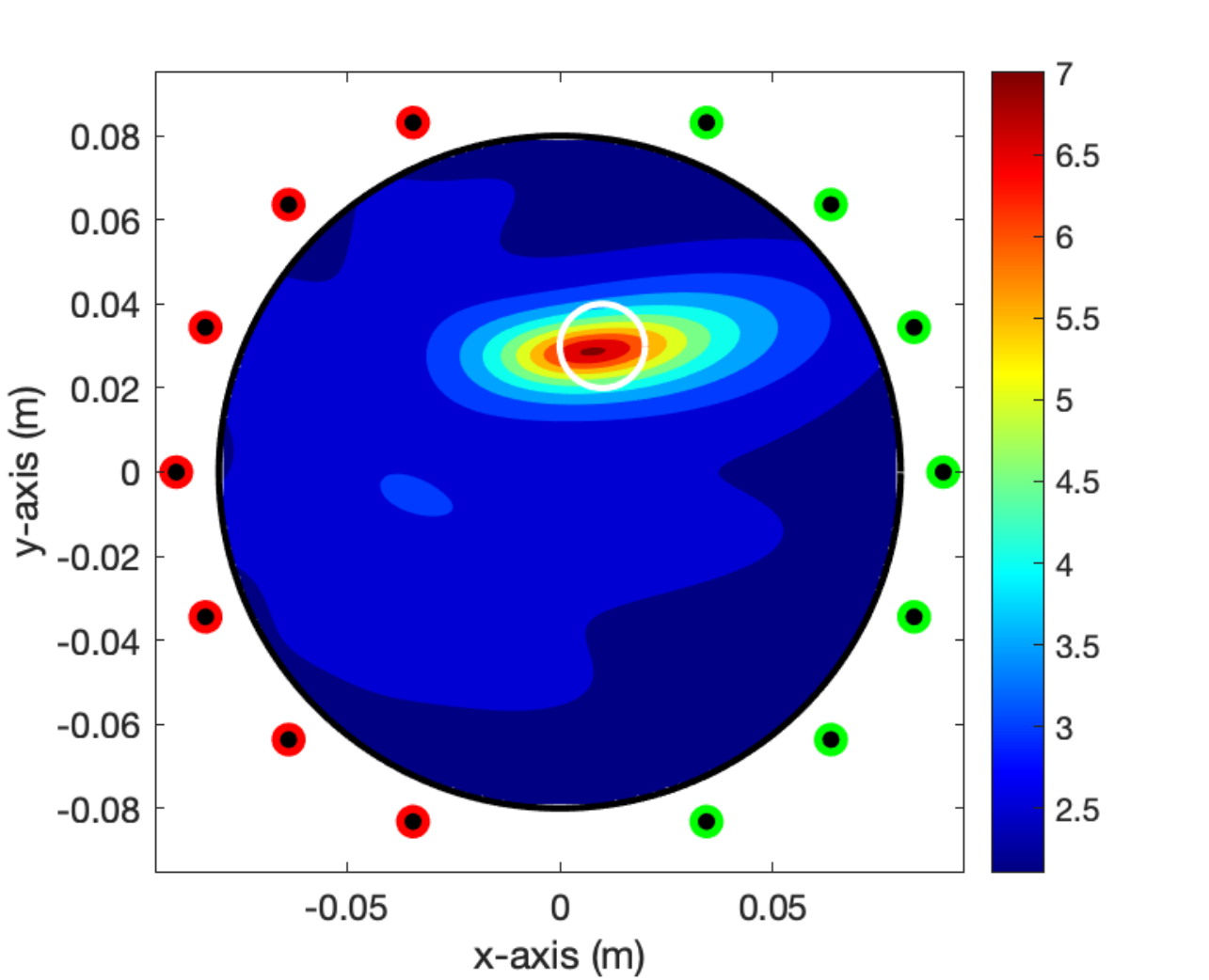}\hfill
  \includegraphics[width=0.25\textwidth]{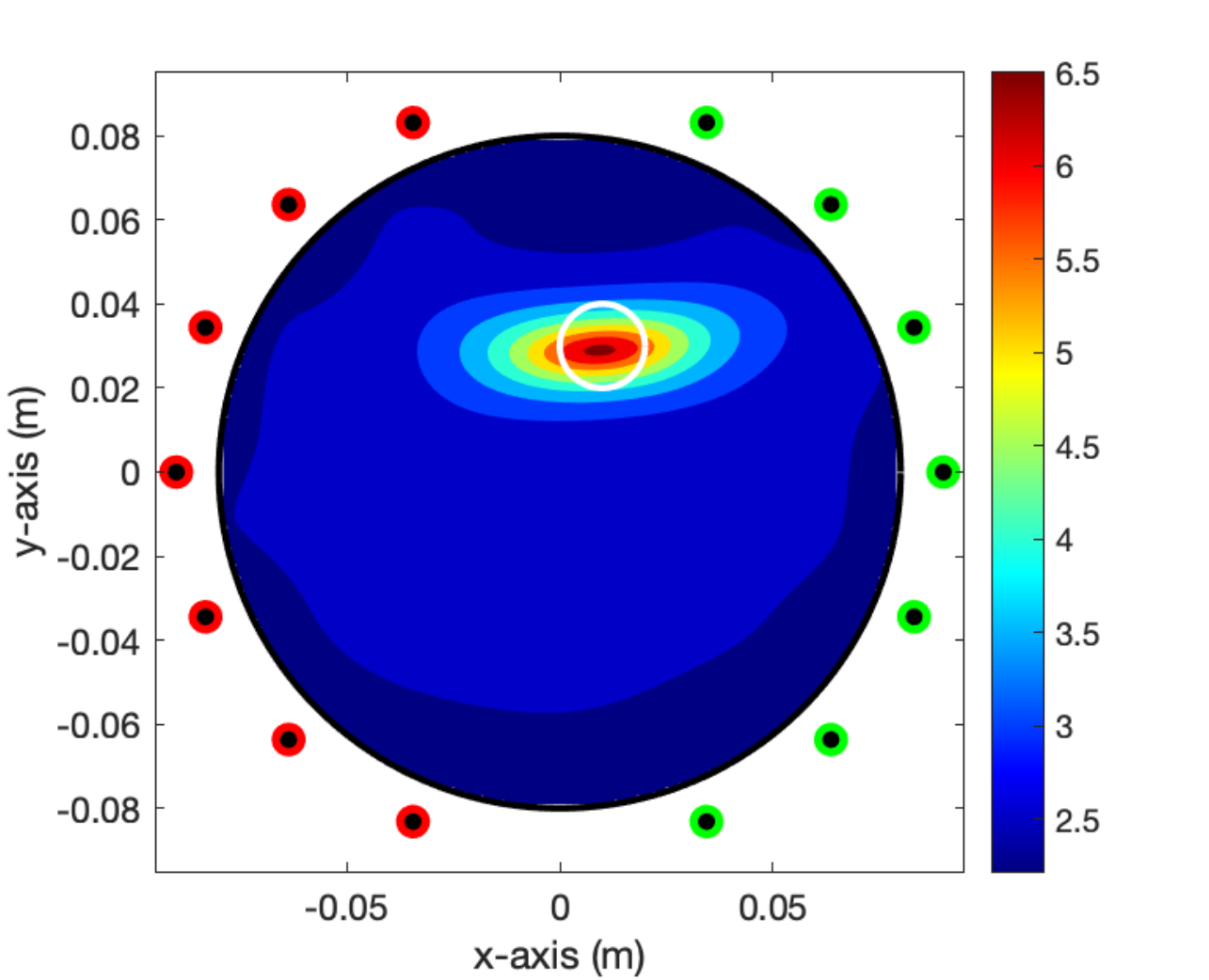}\hfill
  \includegraphics[width=0.25\textwidth]{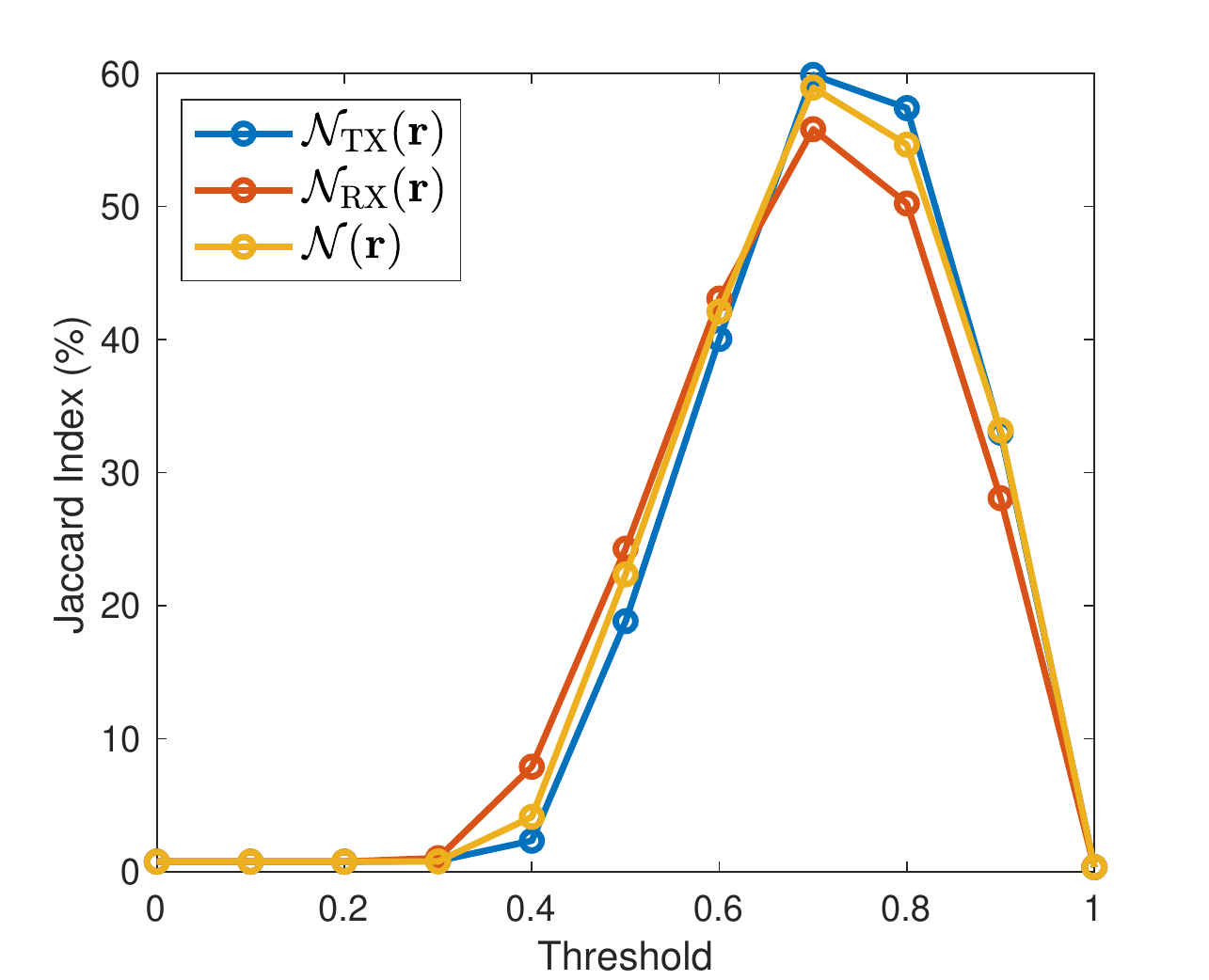}\\
  \includegraphics[width=0.25\textwidth]{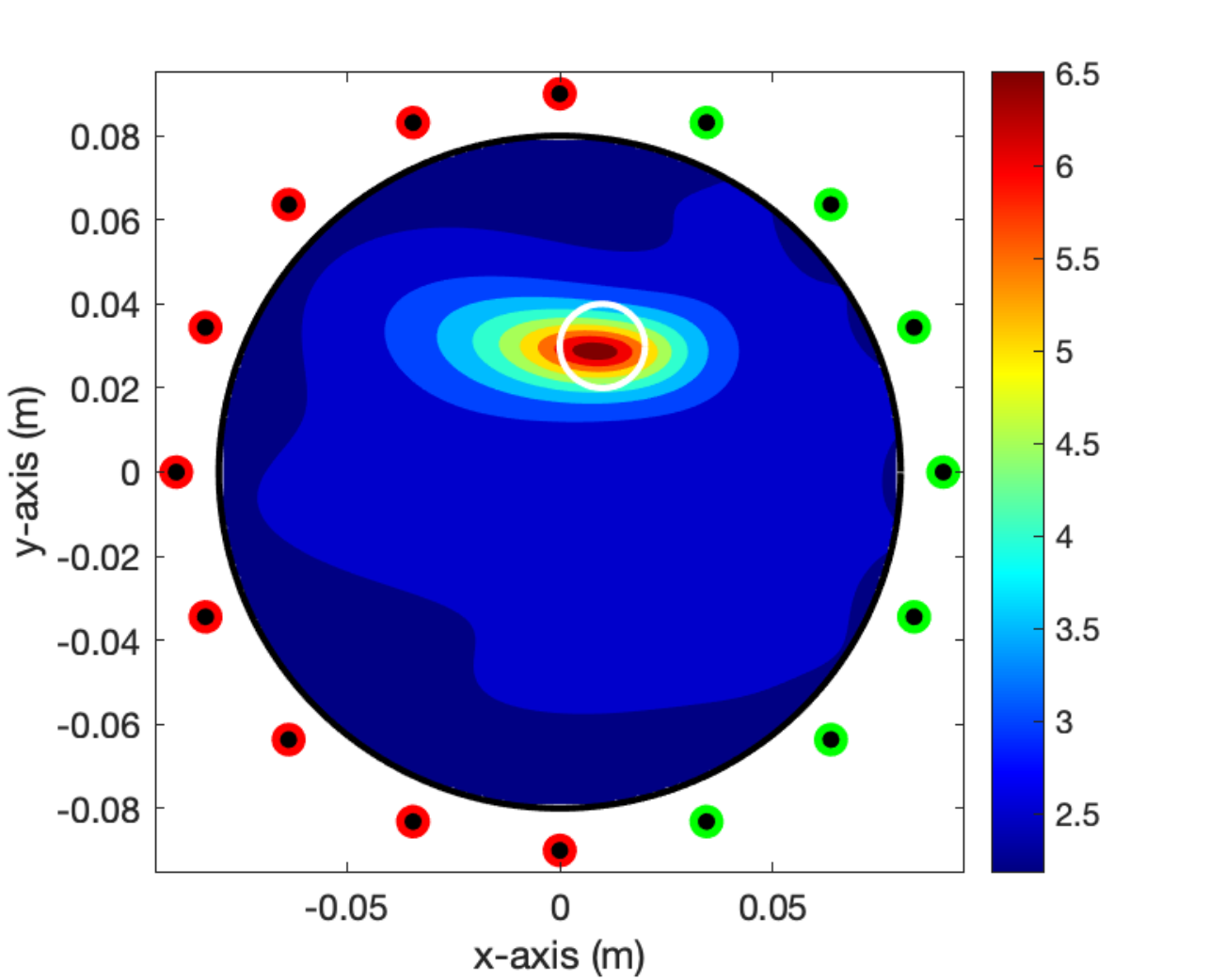}\hfill
  \includegraphics[width=0.25\textwidth]{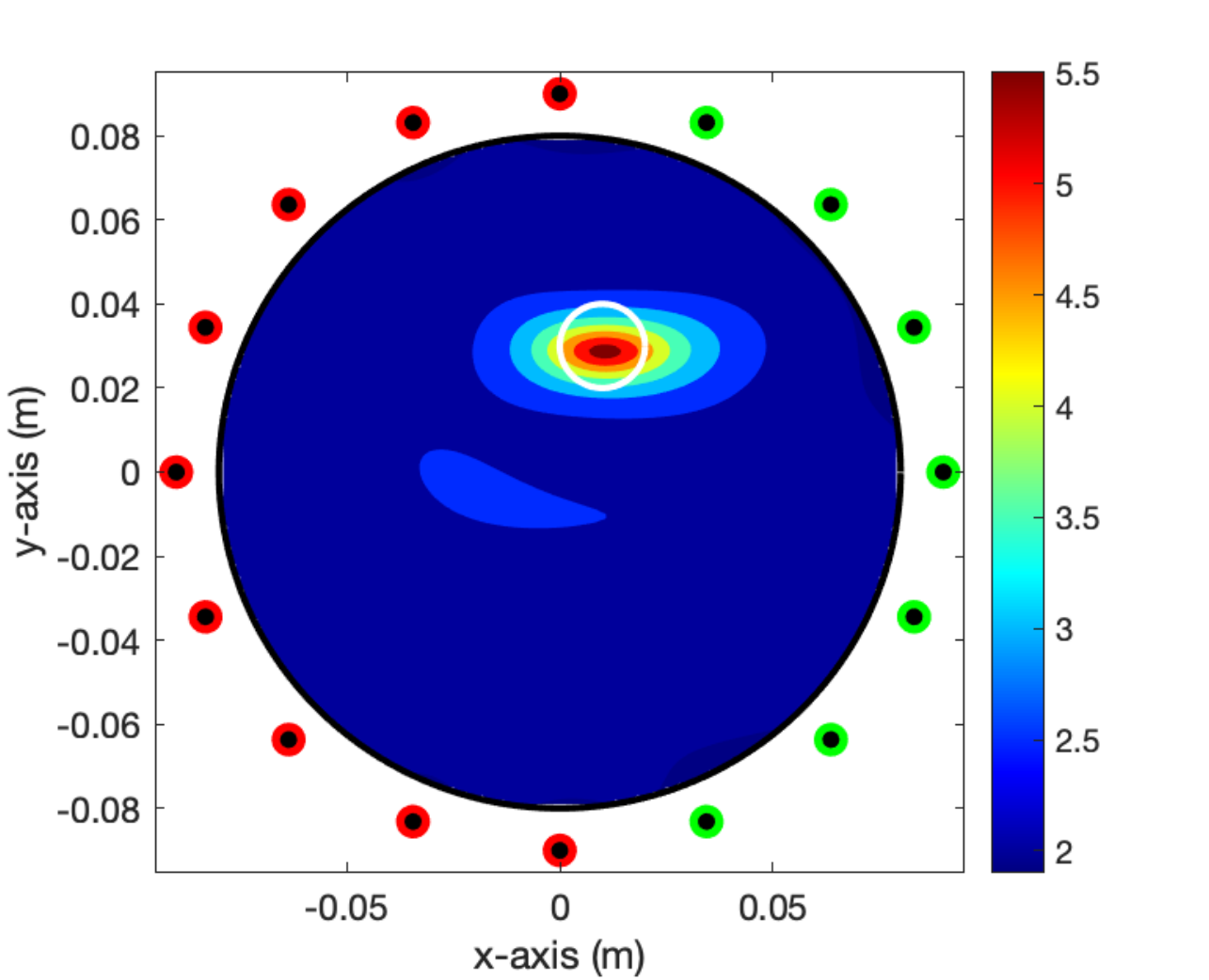}\hfill
  \includegraphics[width=0.25\textwidth]{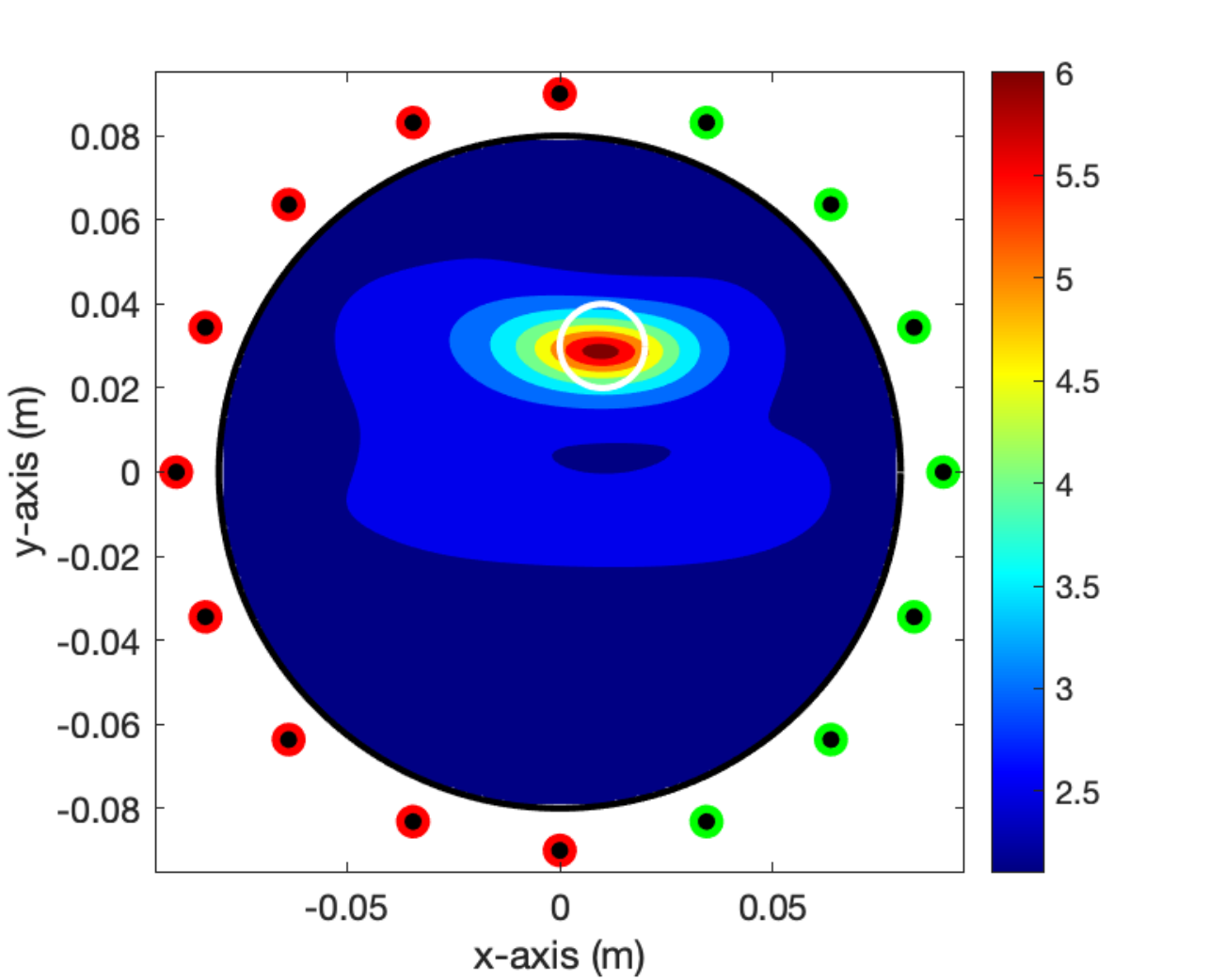}\hfill
  \includegraphics[width=0.25\textwidth]{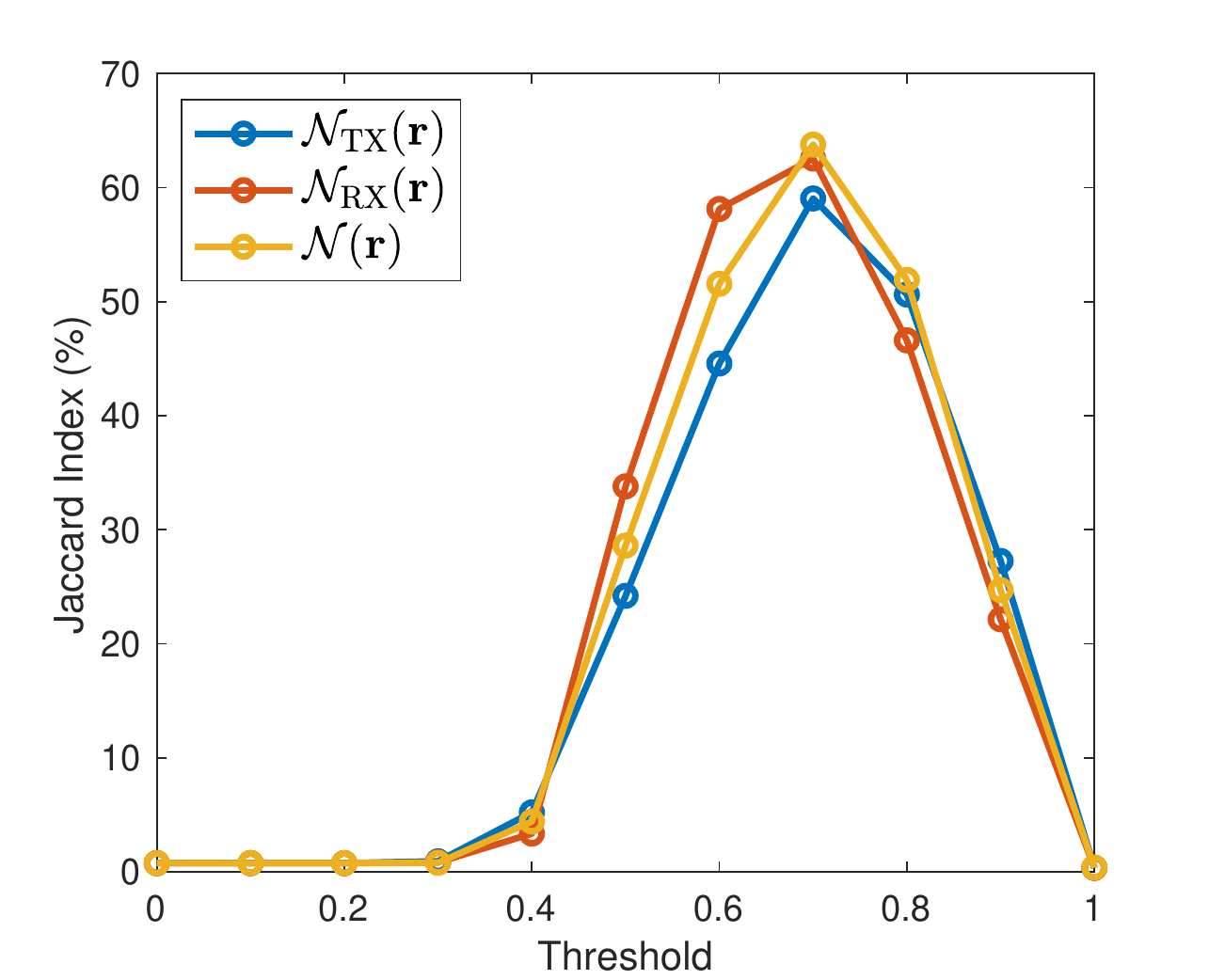}
  \caption{\label{Result2}(Example \ref{ex2}) Maps of $\mathfrak{F}_{\tx}(\mr)$ (first column), $\mathfrak{F}_{\rx}(\mr)$ (second column), $\mathfrak{F}(\mr)$ (third column), and Jaccard index (fourth column). Green and red colored circles describe the location of transmitters and receivers, respectively.}
\end{figure}

\begin{example}[Effects on the Arrangement of Antennas: Single Anomaly]\label{ex3}
As we discussed in the Remark \ref{Remark3}, it will be possible to obtain a good result if $\mathcal{E}_{\rx}(\mr,\ma_n)=0$ and $\mathcal{E}_{\tx}(\mr,\mb_m)=0$ simultaneously. Based on the result in \cite{P-SUB16}, if one adopts the transmitting antenna setting $\mathbf{B}_\star=\set{\mathcal{D}_m,m=1,3,5,7,9,11,13,15}$ and receiving antenna setting $\mathbf{A}_\star=\set{\mathcal{D}_m,m=2,4,6,8,10,12,14,16}$, the factors $\mathcal{E}_{\rx}(\mr,\ma_n)$ and $\mathcal{E}_{\tx}(\mr,\mb_m)$ of \eqref{StructureImagingFunction} can be removed, i.e., this is the best option for achieving good results. To demonstrate this, we consider the maps of $\mathfrak{F}(\mr)$ in Figure \ref{Result3} with transmitting antenna settings $\mathbf{B}_5=\set{\mathcal{D}_m,m=1,5,9,13}$ and $\mathbf{B}_\star$, and receiving antenna settings $\mathbf{A}_5=\set{\mathcal{D}_n:n=2,7,11,15}$, $\mathbf{A}_6=\set{\mathcal{D}_n:n=2,4,6,8,10,12,14}$, $\mathbf{A}_7=\set{\mathcal{D}_n:n=2,3,4,6,7,8,10,11,12,14,15,16}$, and $\mathbf{A}_\star$. Based on the imaging results, the peak of large magnitude appears at $\mr_\star$ however, due to the appearance of some artifacts with large magnitude (setting $\mathbf{A}_5\cup\mathbf{B}_5$) or another peak of large magnitude (settings $\mathbf{A}_6\cup\mathbf{B}_5$ and $\mathbf{A}_7\cup\mathbf{B}_5$), it is very hard to recognize the true location of $\mr_\star$. Fortunately, very good imaging result appeared with setting $\mathbf{A}_\star\cup\mathbf{B}_\star$ and by comparing with the results in Example \ref{ex1}, not only the location but also the shape of the outline $\Sigma$ can be identified clearly and accurately.
\end{example}

\begin{figure}[h]
  \centering
  \includegraphics[width=0.25\textwidth]{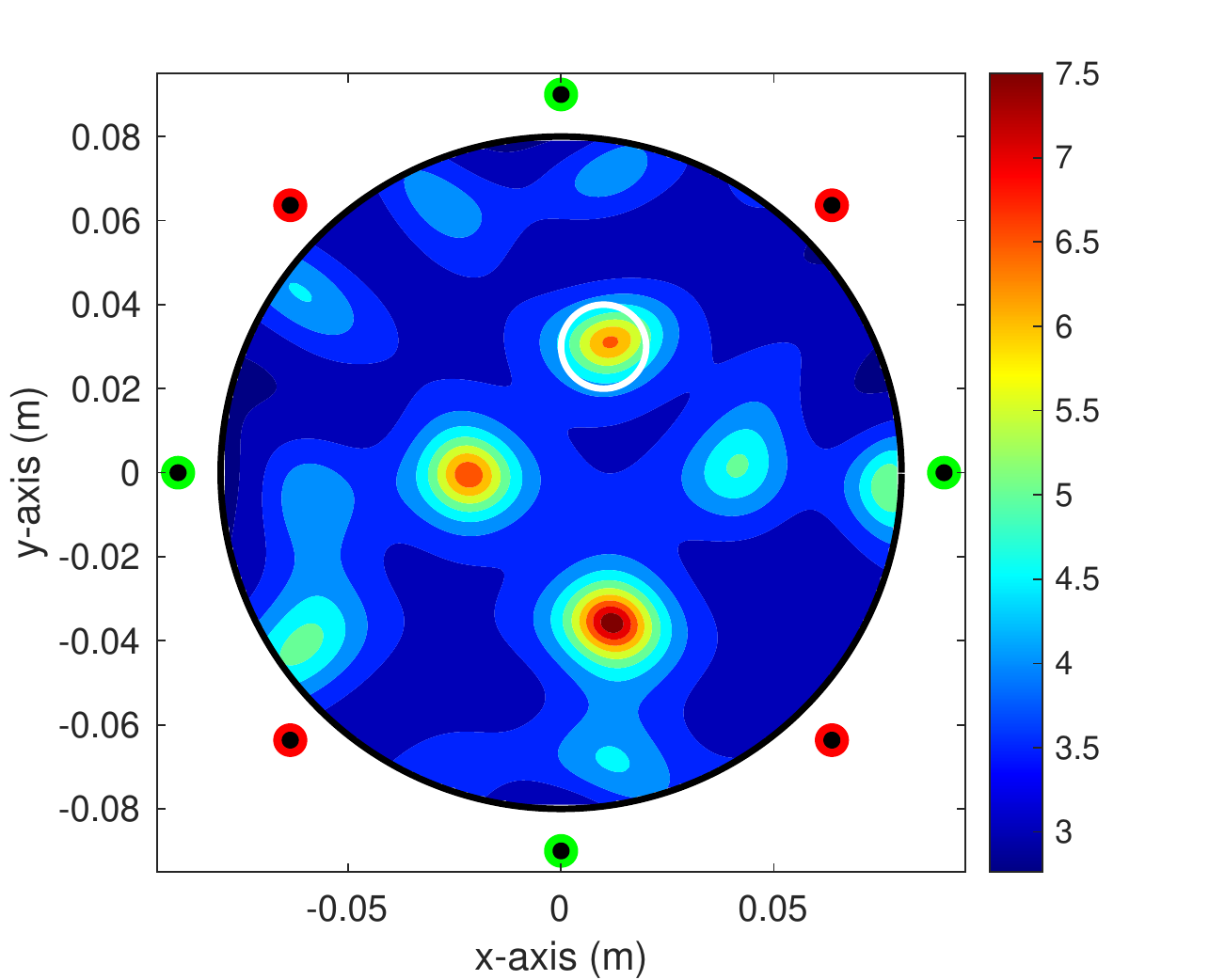}\hfill
  \includegraphics[width=0.25\textwidth]{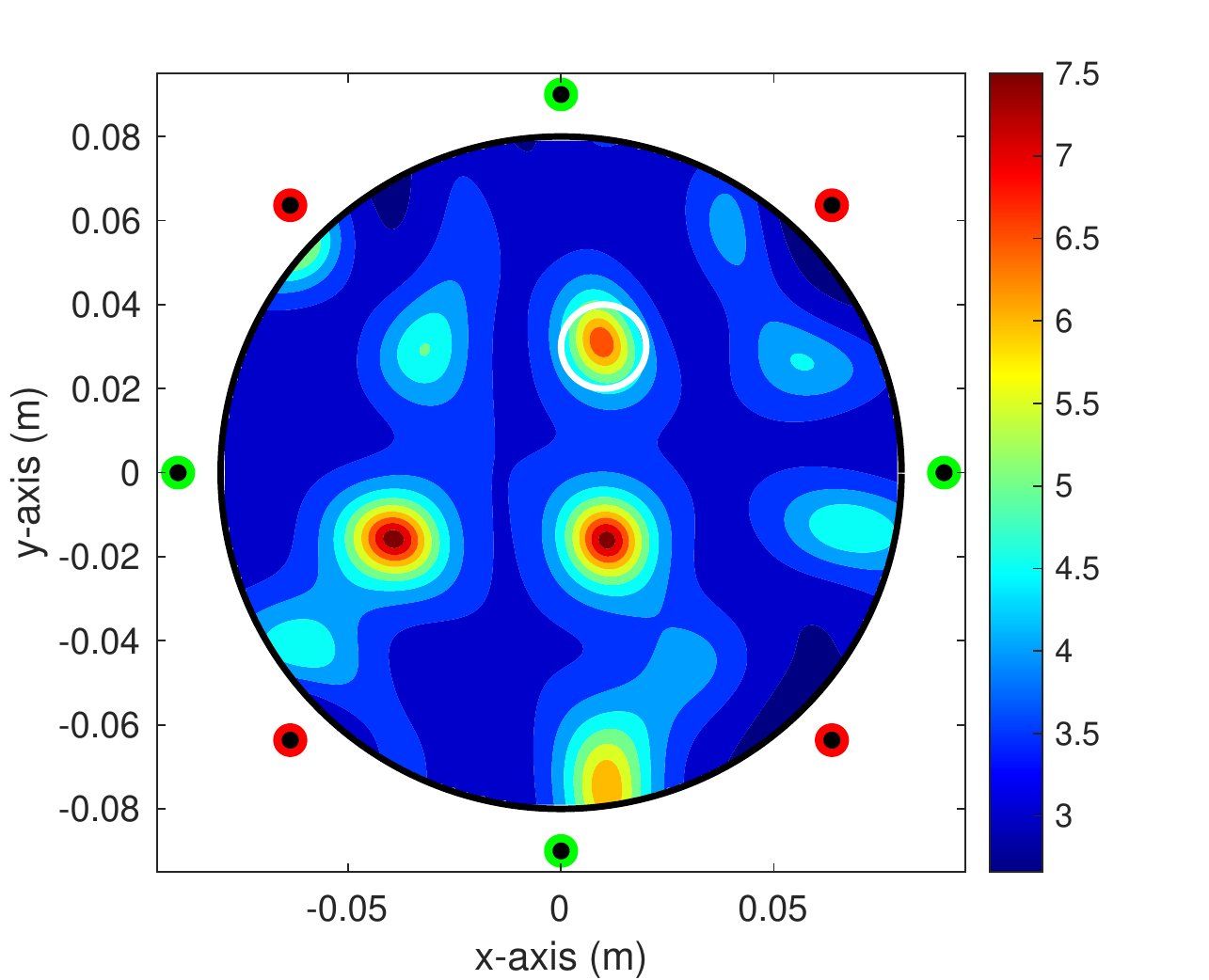}\hfill
  \includegraphics[width=0.25\textwidth]{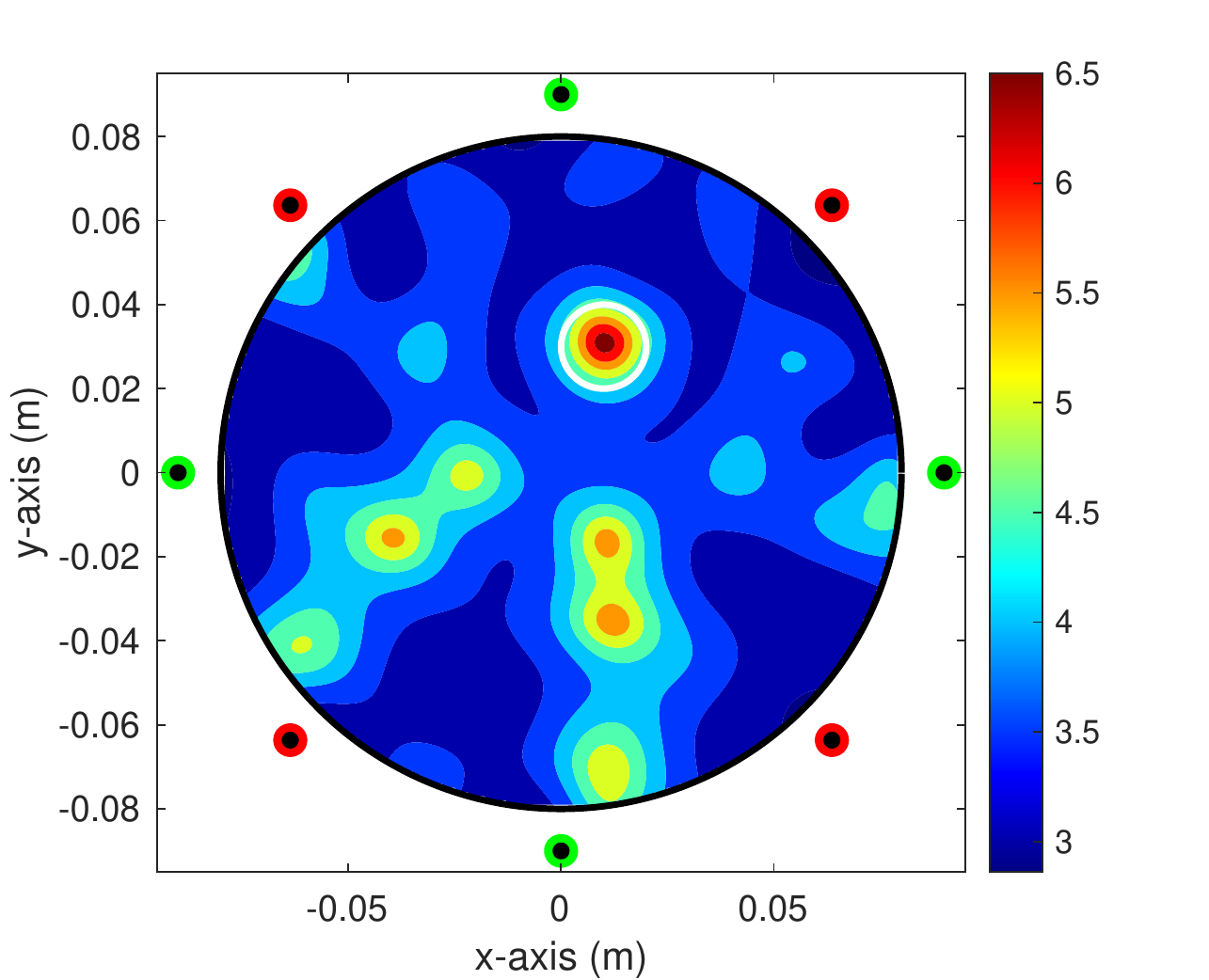}\hfill
  \includegraphics[width=0.25\textwidth]{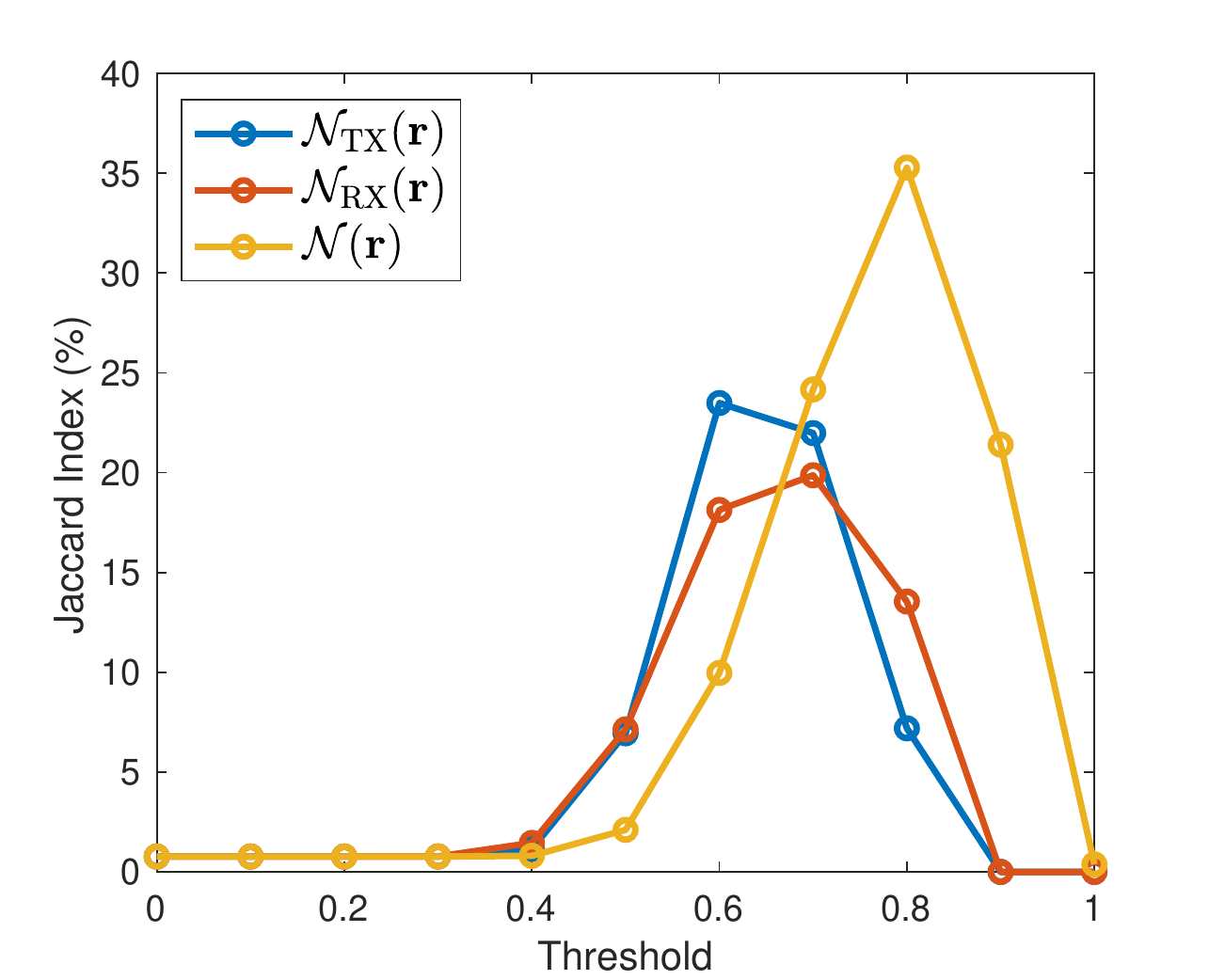}\\
  \includegraphics[width=0.25\textwidth]{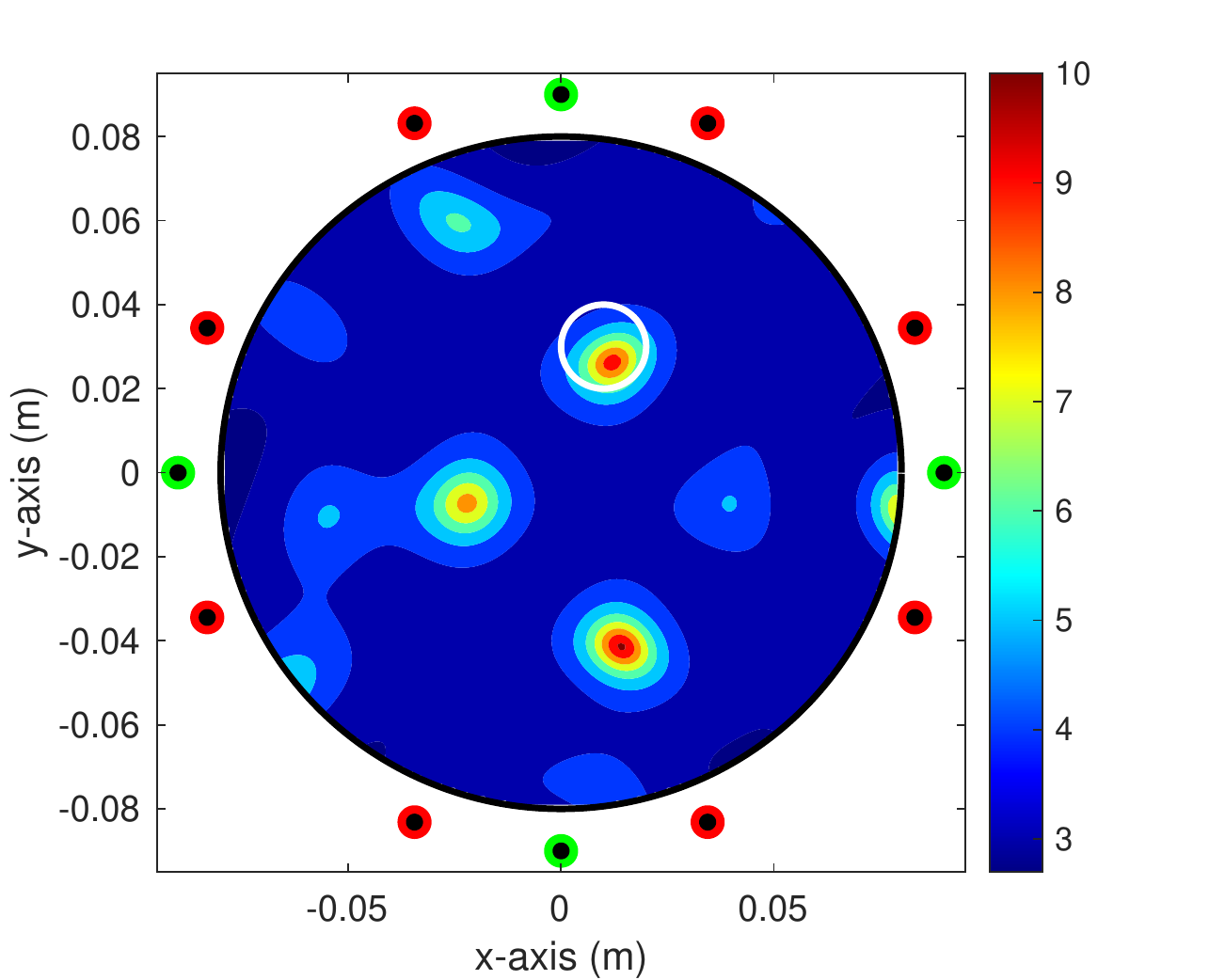}\hfill
  \includegraphics[width=0.25\textwidth]{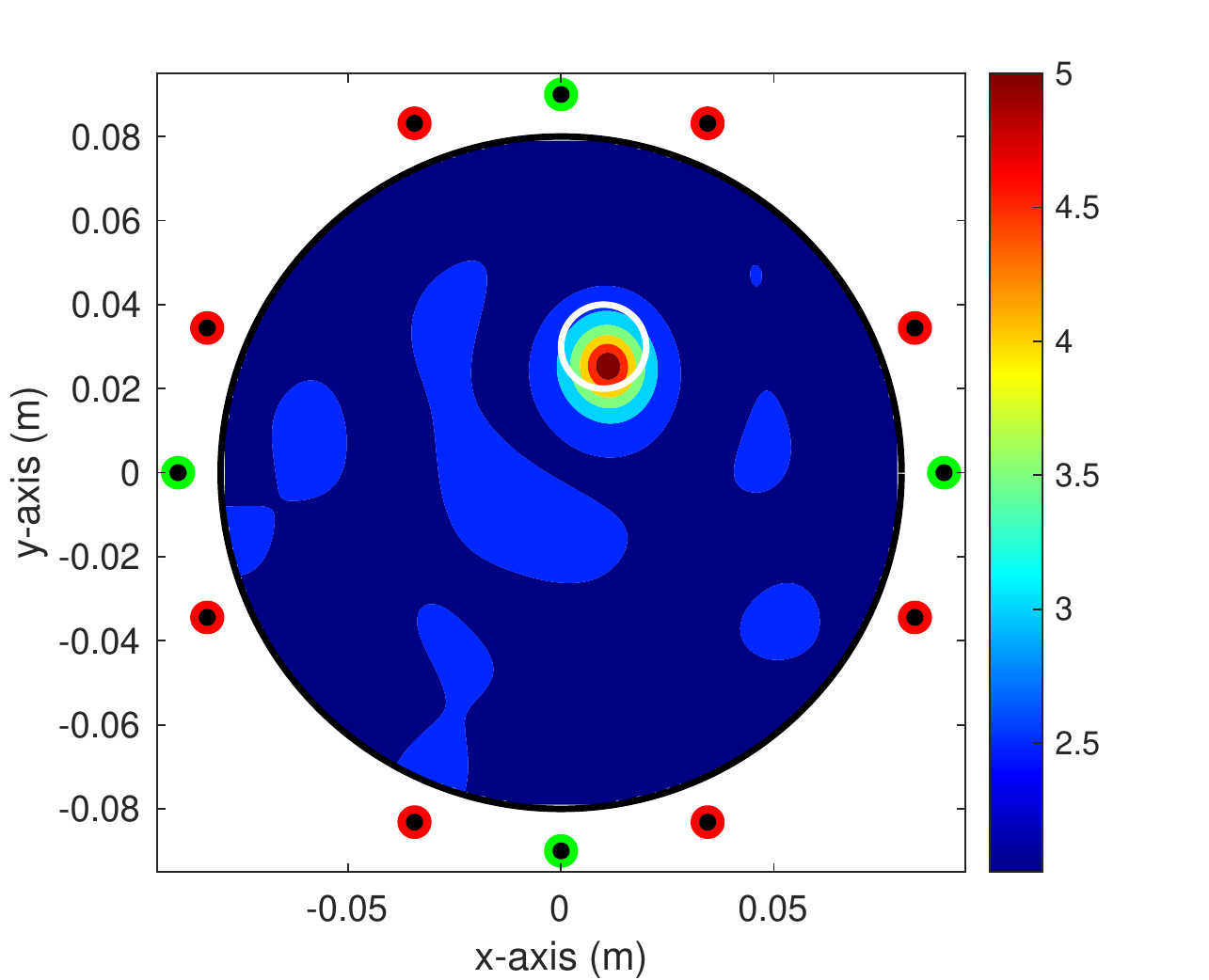}\hfill
  \includegraphics[width=0.25\textwidth]{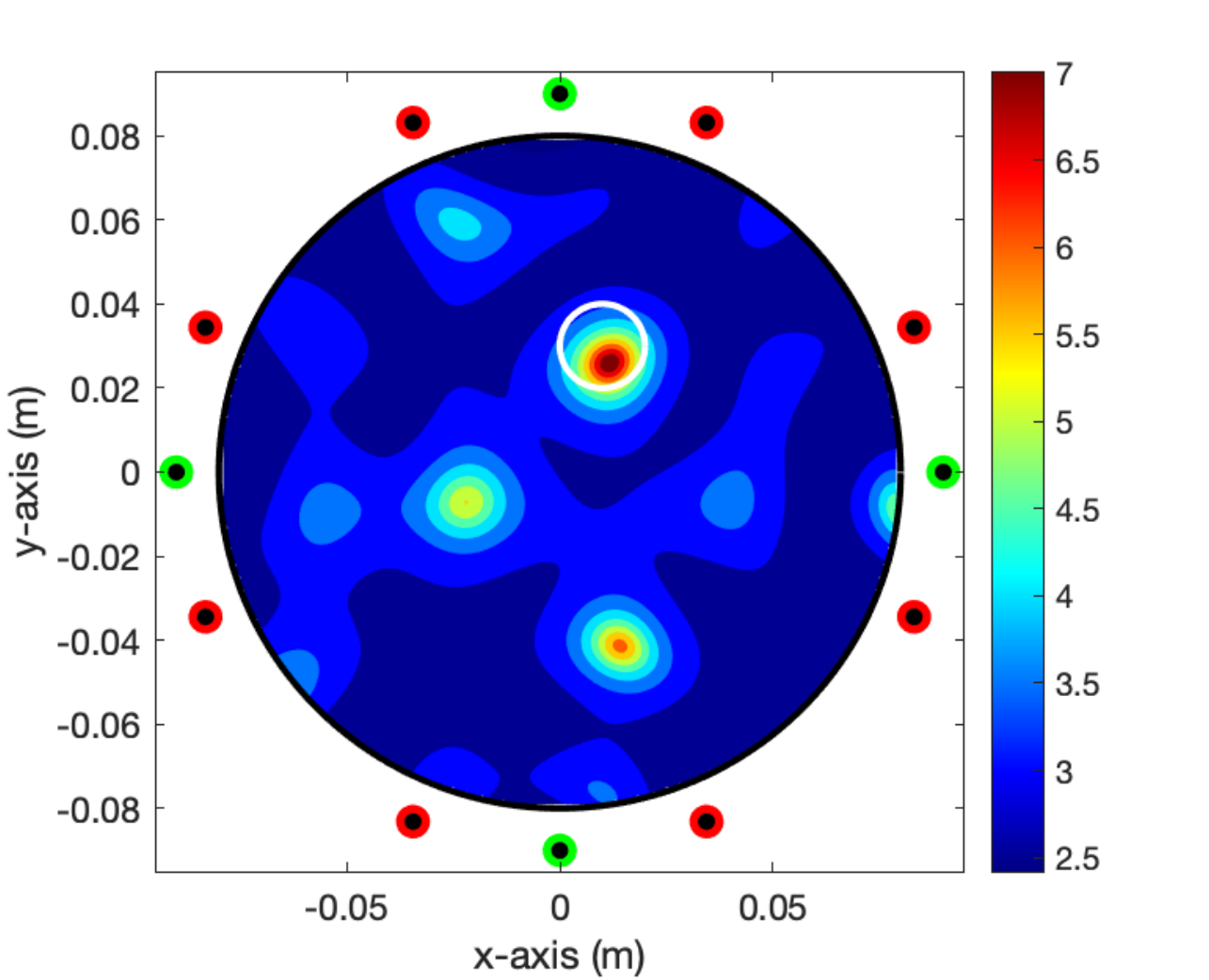}\hfill
  \includegraphics[width=0.25\textwidth]{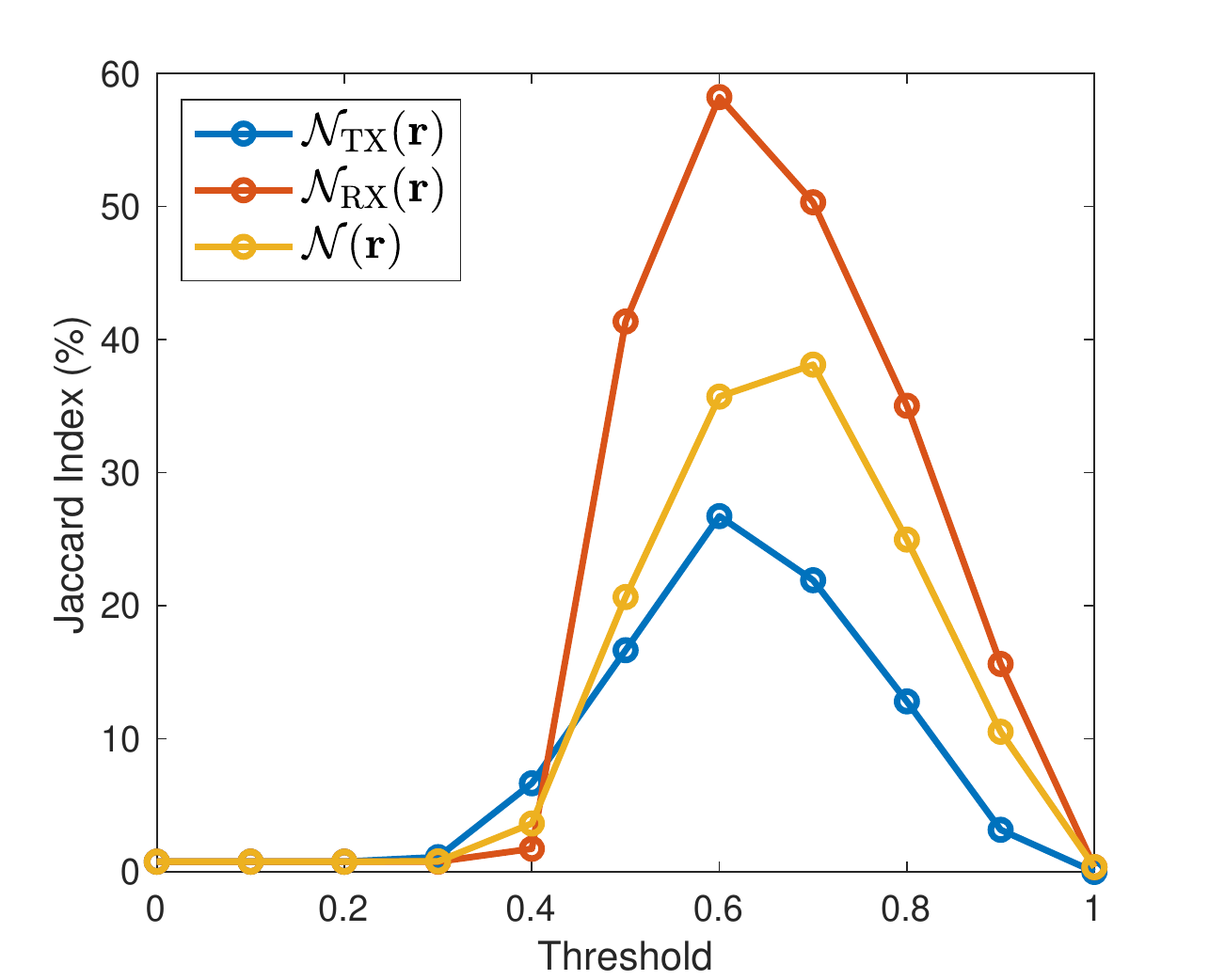}\\
  \includegraphics[width=0.25\textwidth]{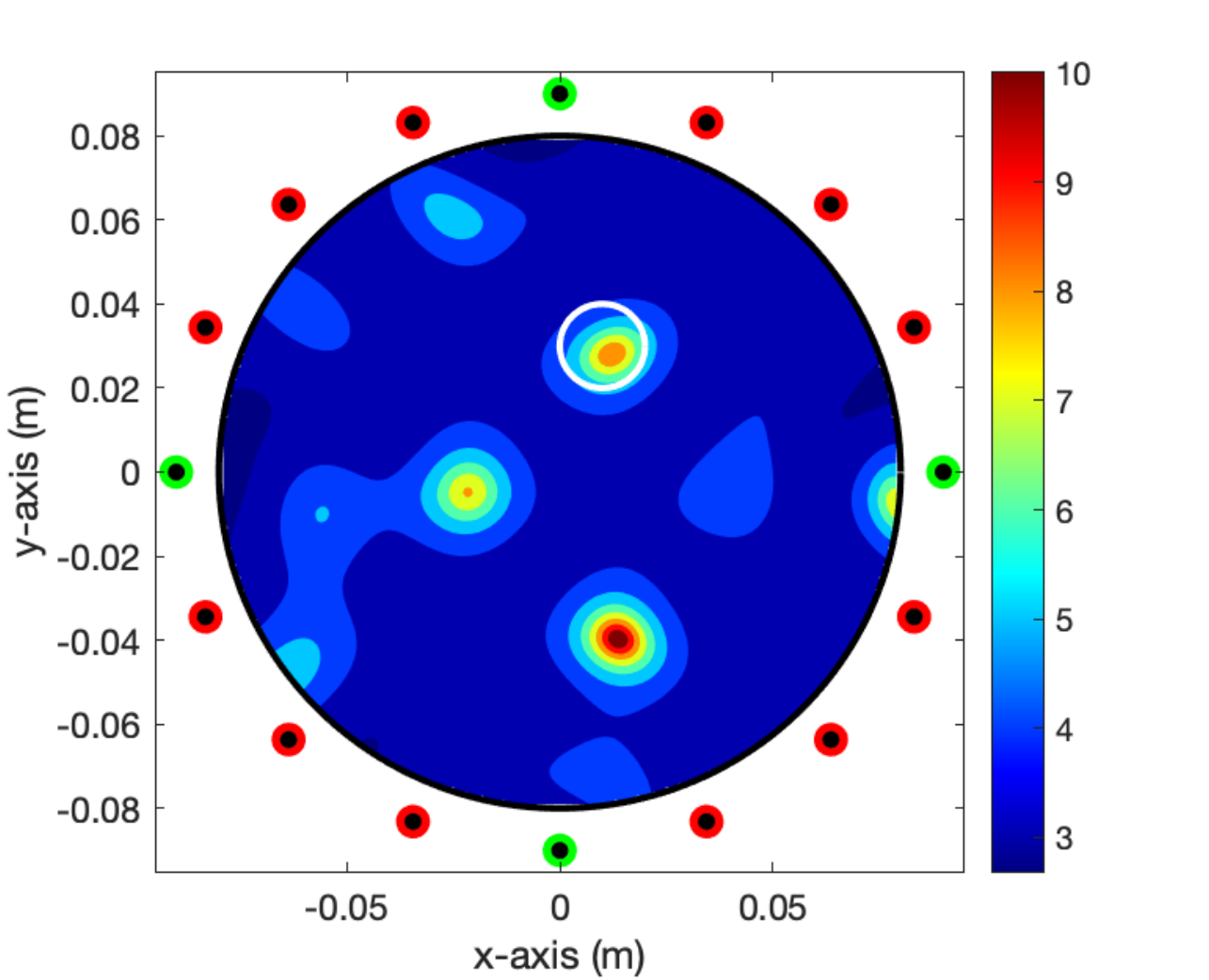}\hfill
  \includegraphics[width=0.25\textwidth]{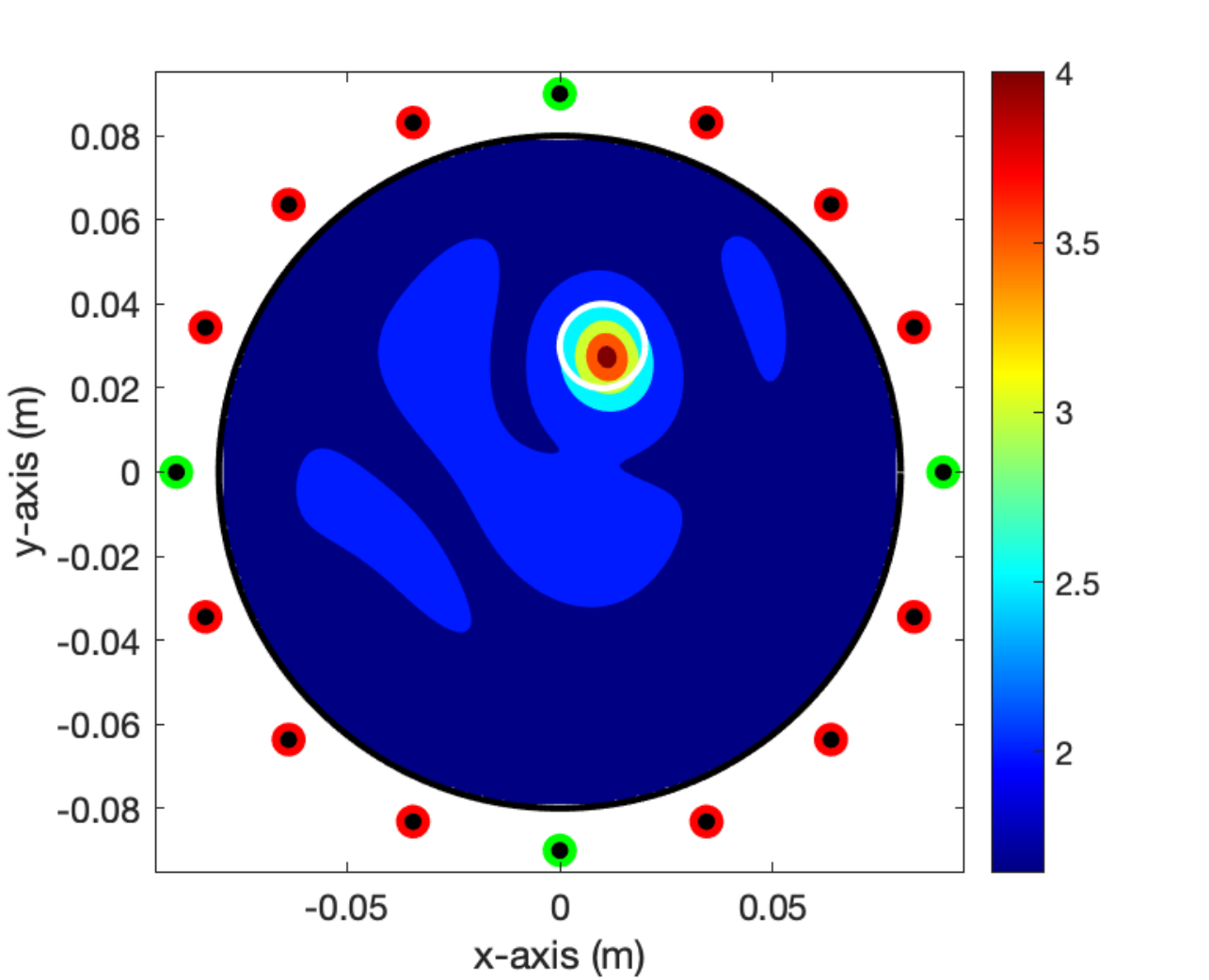}\hfill
  \includegraphics[width=0.25\textwidth]{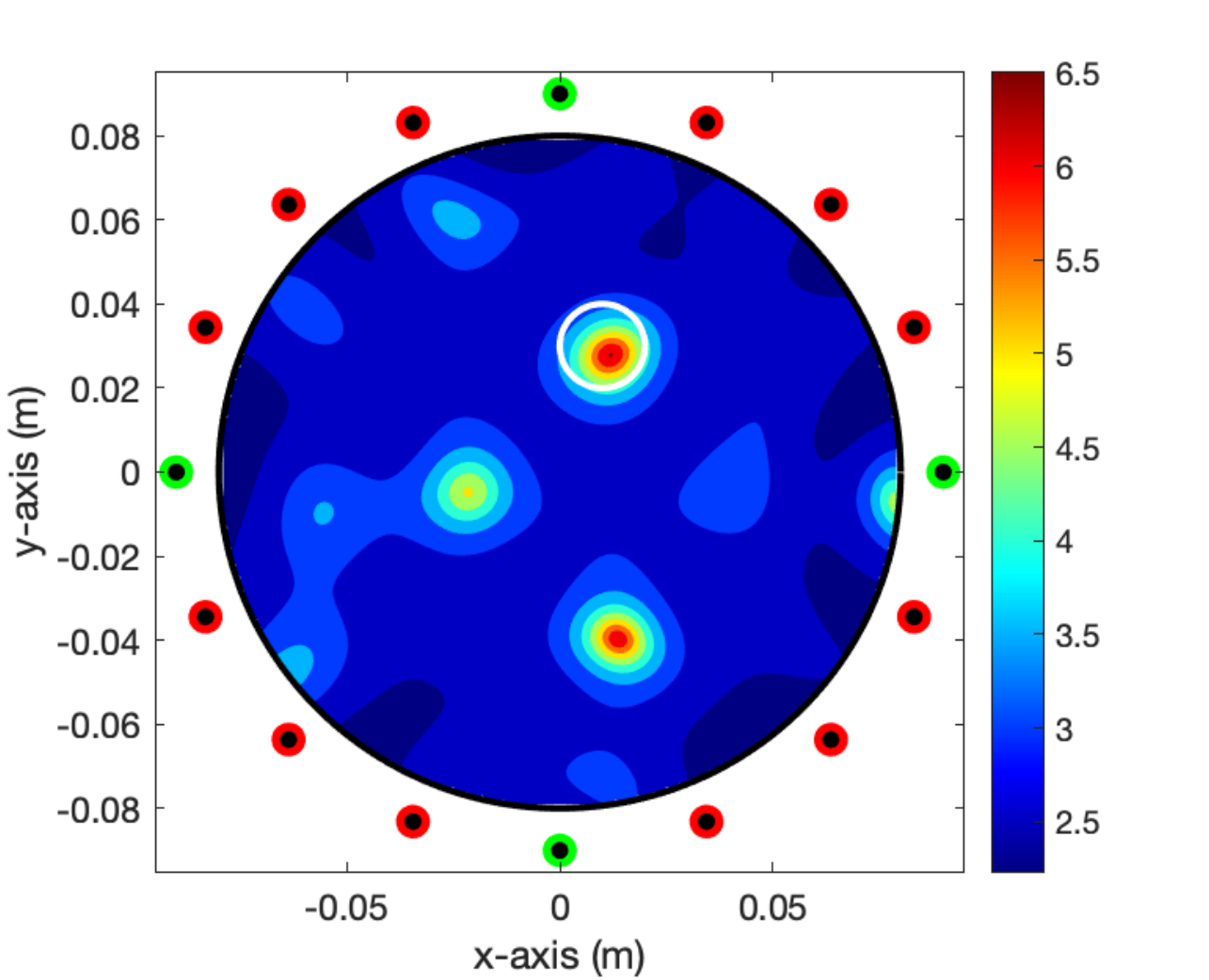}\hfill
  \includegraphics[width=0.25\textwidth]{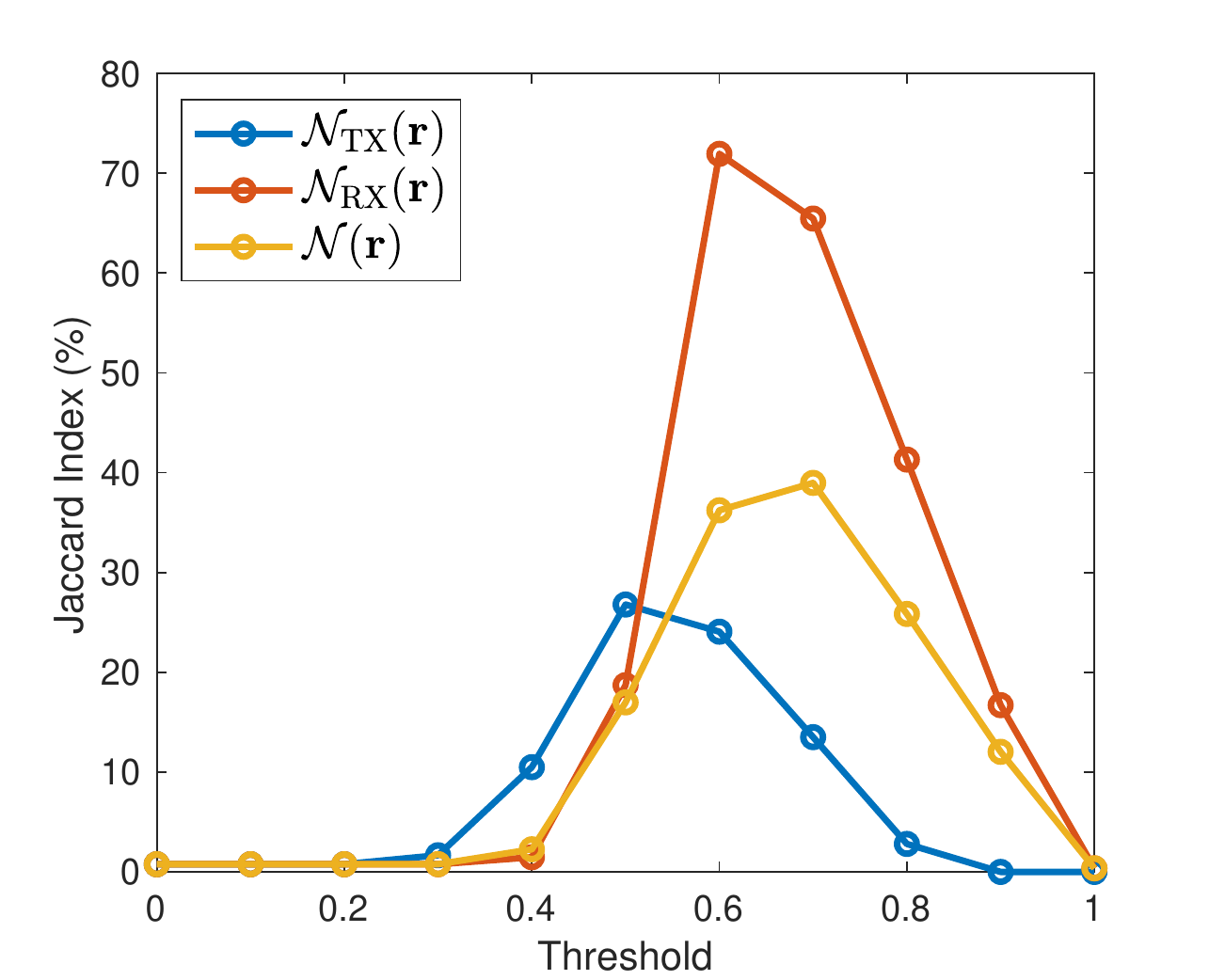}\\
  \includegraphics[width=0.25\textwidth]{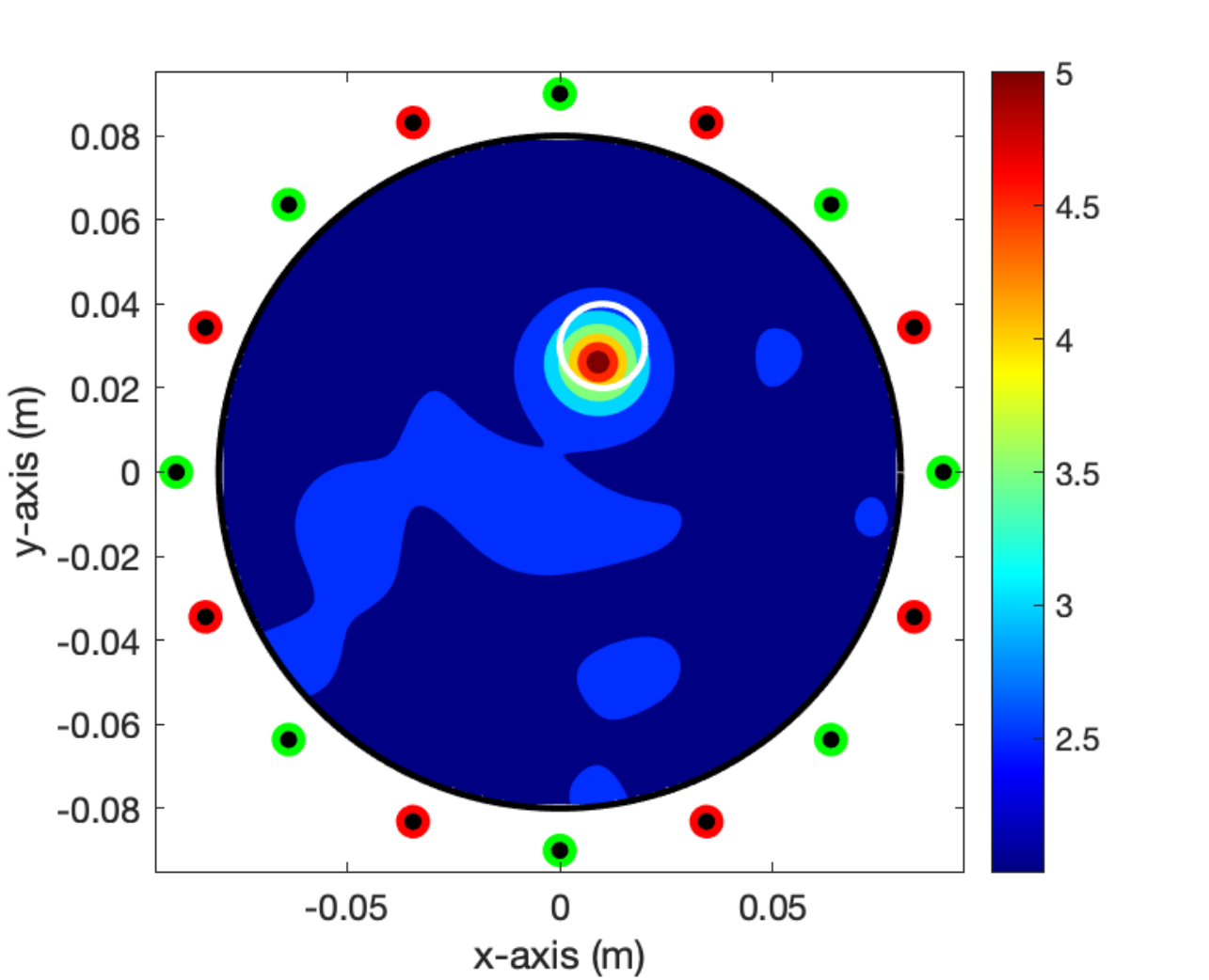}\hfill
  \includegraphics[width=0.25\textwidth]{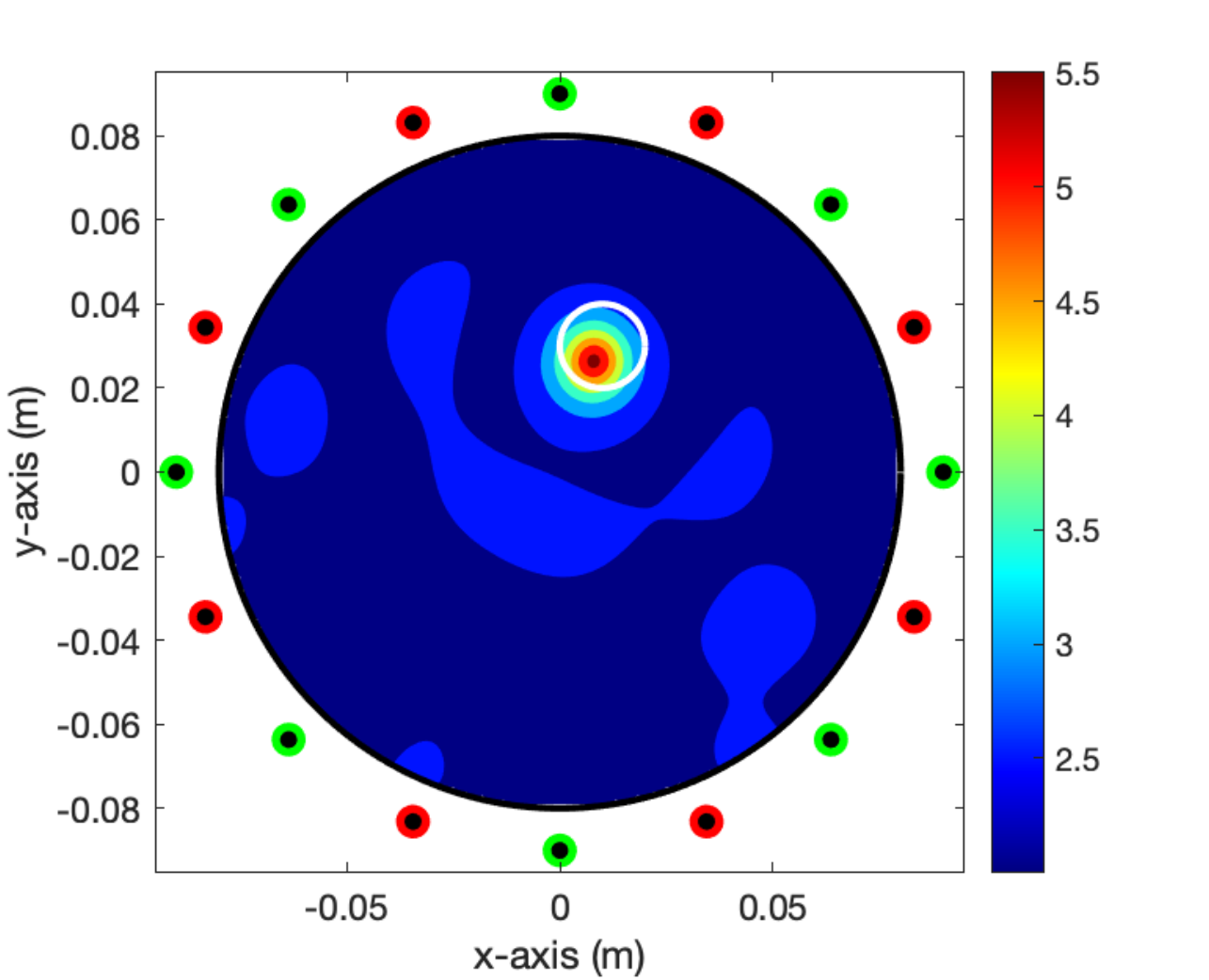}\hfill
  \includegraphics[width=0.25\textwidth]{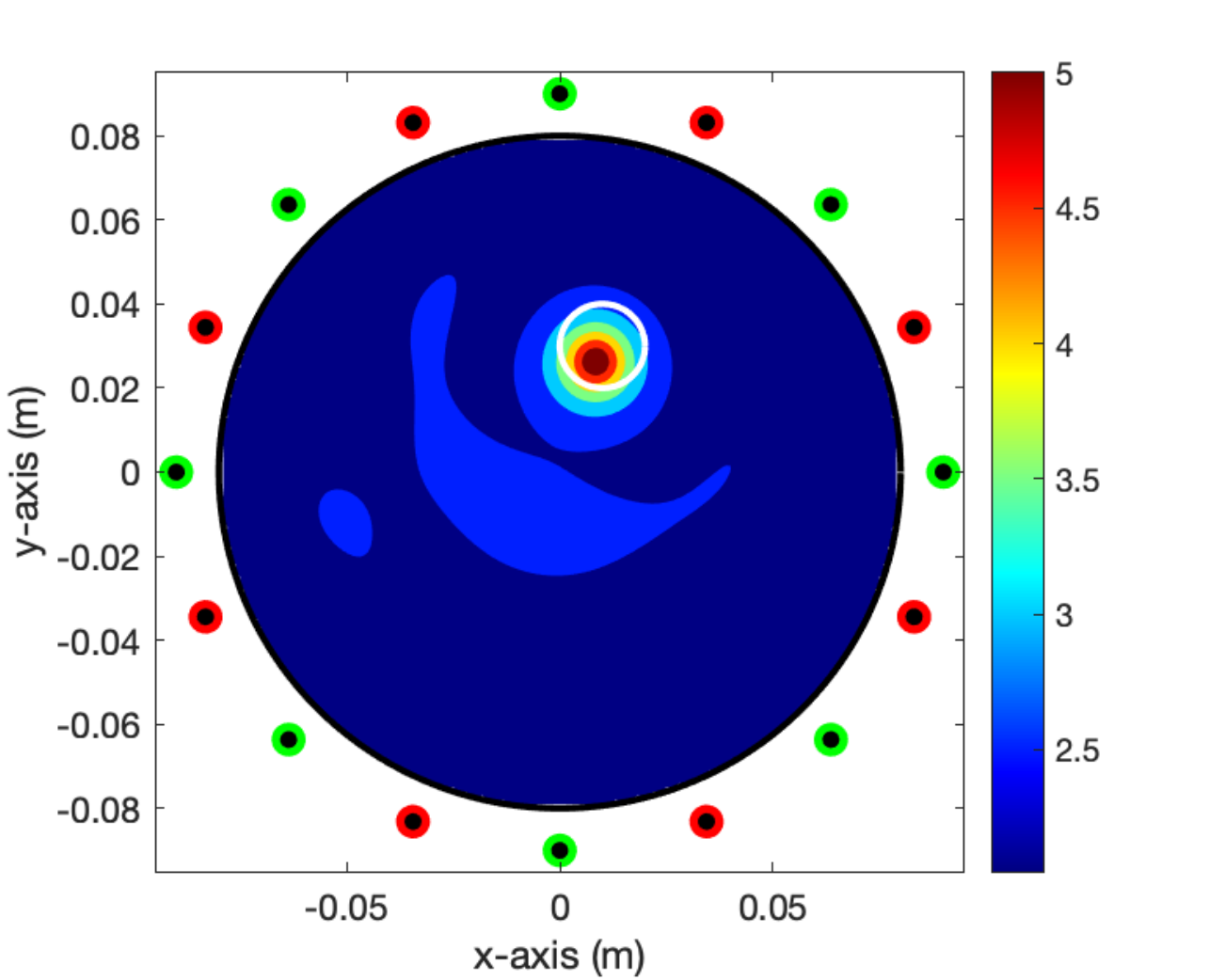}\hfill
  \includegraphics[width=0.25\textwidth]{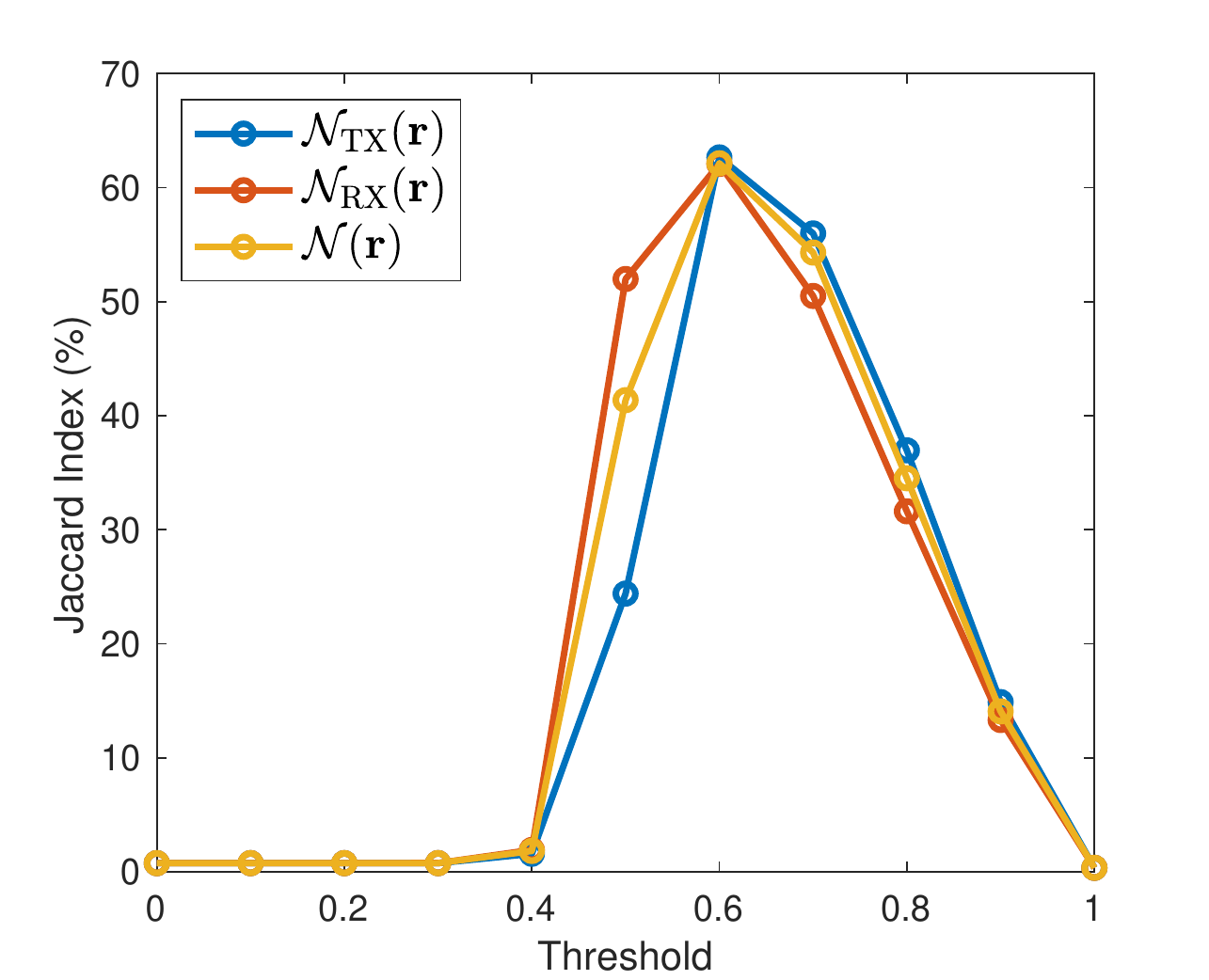}
  \caption{\label{Result3}(Example \ref{ex3}) Maps of $\mathfrak{F}_{\tx}(\mr)$ (first column), $\mathfrak{F}_{\rx}(\mr)$ (second column), $\mathfrak{F}(\mr)$ (third column), and Jaccard index (fourth column). Green and red colored circles describe the location of transmitters and receivers, respectively.}
\end{figure}

\begin{example}[Imaging of Multiple Anomalies]\label{ex4}
Now, let us apply the MUSIC for identifying two different anomalies $\Sigma_1$ and $\Sigma_2$ with locations $\mr_1=(\SI{0.01}{\meter},\SI{0.03}{\meter})$, $\mr_2=(-\SI{0.04}{\meter},-\SI{0.02}{\meter})$, same radii $\alpha_1=\alpha_2=\SI{0.01}{\meter}$, and material properties $(\eps_1,\sigma_1)=(55\eps_0,\SI{1.2}{\siemens/\meter})$, $(\eps_2,\sigma_2)=(45\eps_0,\SI{1.0}{\siemens/\meter})$ (see Figure \ref{IllustrationAnomalies} for an illustration). Figure \ref{Result4} shows the maps of $\mathfrak{F}_{\tx}(\mr)$, $\mathfrak{F}_{\rx}(\mr)$, $\mathfrak{F}(\mr)$, and Jaccard index with transmitting antenna setting $\mathbf{B}_1$ and receiving antenna settings $\mathbf{A}_1$, $\mathbf{A}_2$, $\mathbf{A}_3$, and $\mathbf{A}_4$ introduced in the Example \ref{ex1}. Opposite to the imaging results in Figure \ref{Result1}, it is impossible to recognize the existence of two anomalies with antenna settings $\mathbf{A}_1\cup\mathbf{B}_1$. Moreover, if one increases total number of receiving antennas such as $\mathbf{A}_2$, $\mathbf{A}_3$, and $\mathbf{A}_4$, it is still impossible to recognize the presence of the anomalies through the maps of $\mathfrak{F}_{\tx}(\mr)$ and $\mathfrak{F}(\mr)$. Same as the imaging of single anomaly, despite the appearance of some artifacts, it is possible to identify the existence of two anomalies via the map of $\mathfrak{F}_{\rx}(\mr)$ with $\mathbf{A}_2$, $\mathbf{A}_3$, and $\mathbf{A}_4$. It is interesting to observe that $\mathfrak{F}_{\rx}(\mr_1)>\mathfrak{F}_{\rx}(\mr_2)$ when $\eps_1>\eps_2$ and $\sigma_1>\sigma_2$. Based on this result, we can examine that if an anomaly has a smaller value of permittivity or conductivity (here, $\Sigma_2$) than the other one (here, $\Sigma_1$), this anomaly will have less influence on the scattering matrix and as a consequence, the value of $\mathfrak{F}_{\rx}(\mr_2)$ will be smaller than the one of $\mathfrak{F}_{\rx}(\mr_1)$.
\end{example}

\begin{figure}[h]
  \centering
  \includegraphics[width=0.25\textwidth]{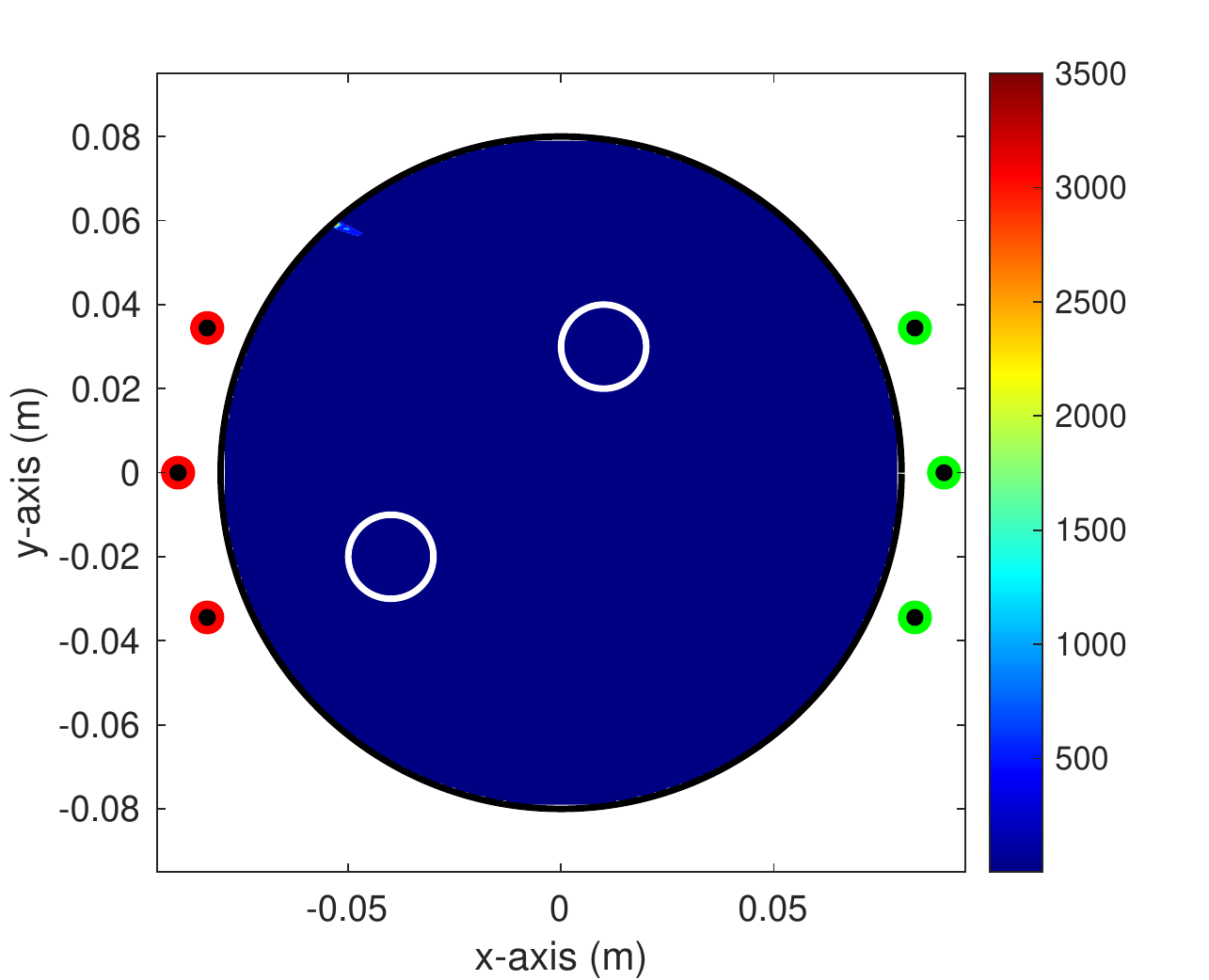}\hfill
  \includegraphics[width=0.25\textwidth]{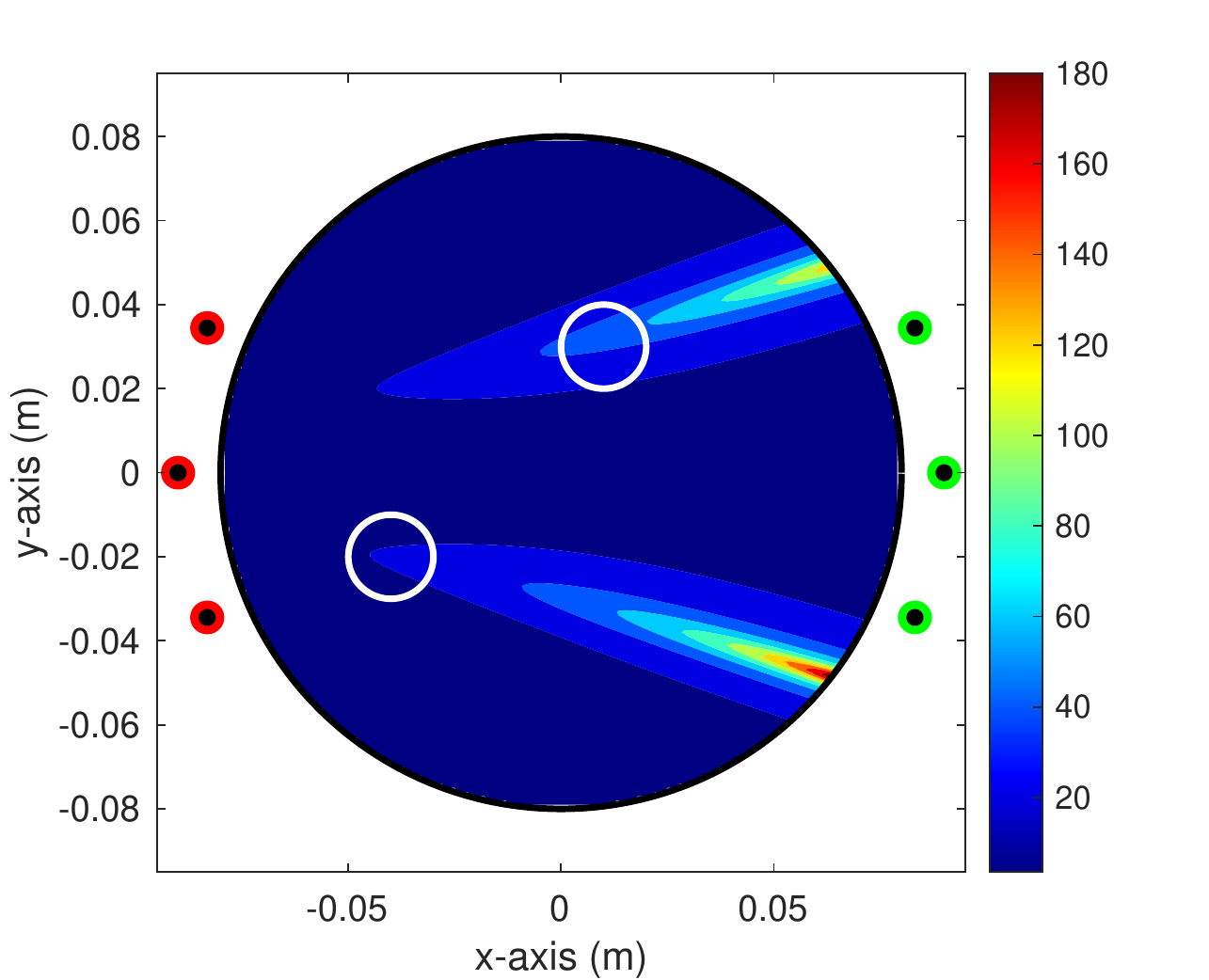}\hfill
  \includegraphics[width=0.25\textwidth]{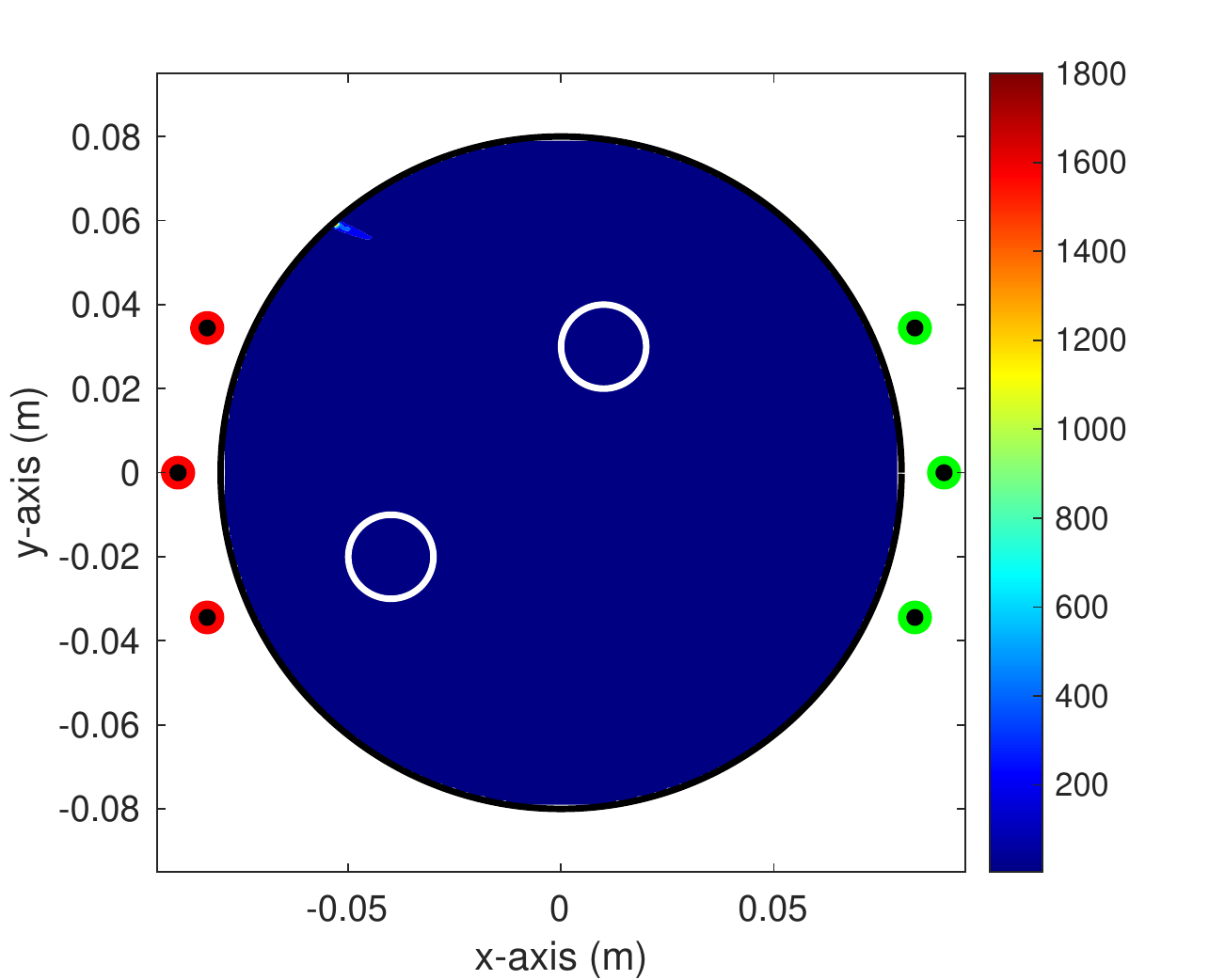}\hfill
  \includegraphics[width=0.25\textwidth]{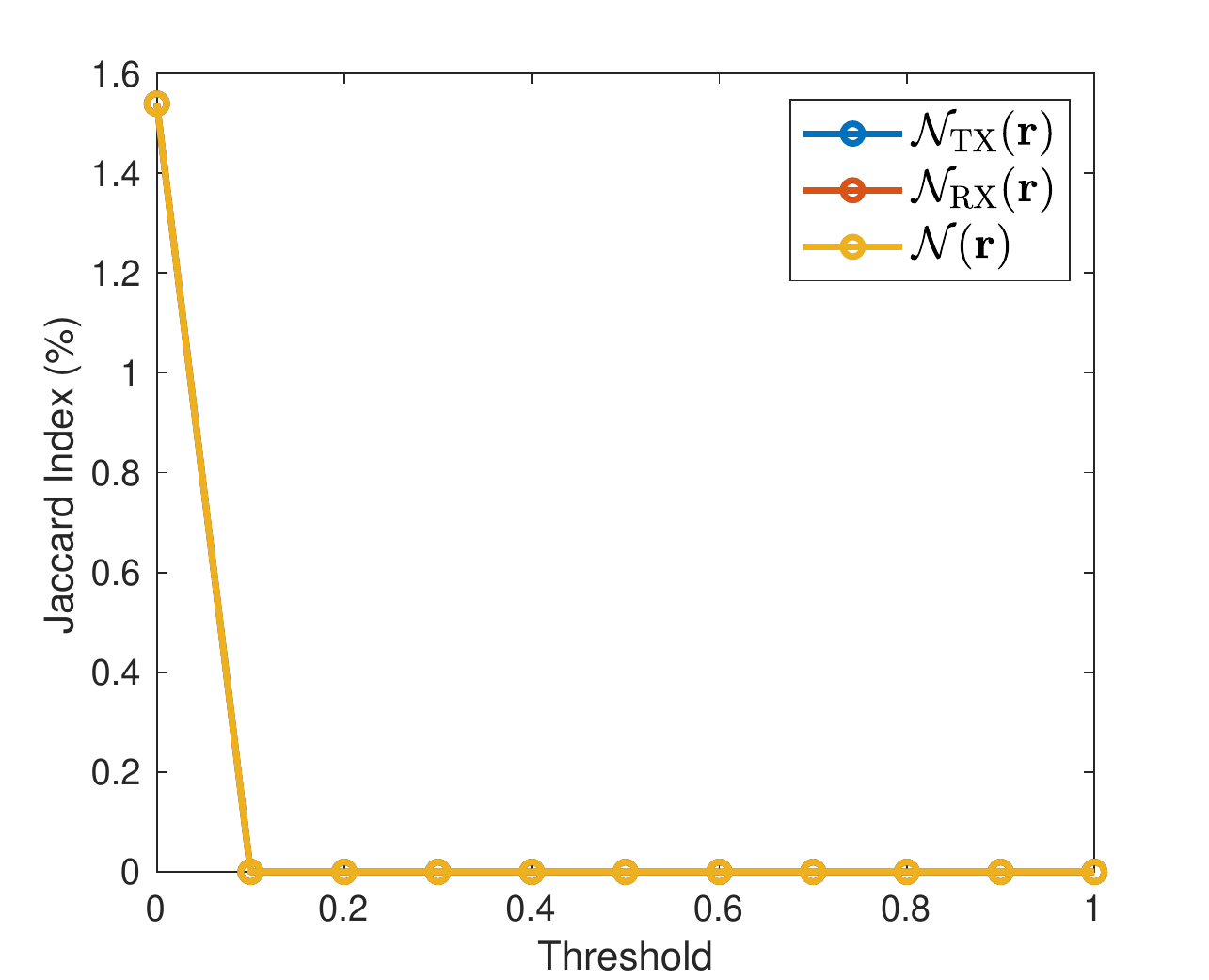}\\
  \includegraphics[width=0.25\textwidth]{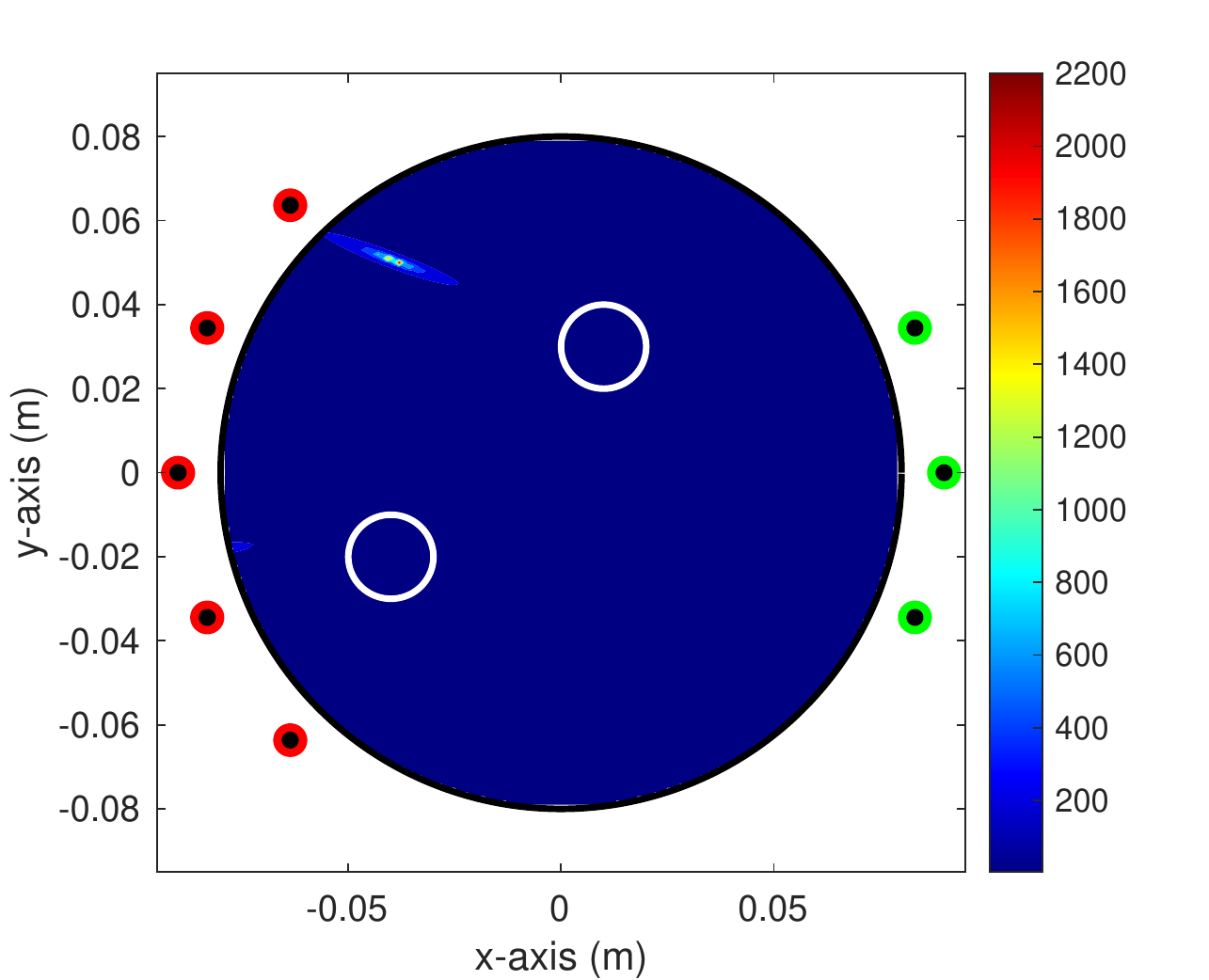}\hfill
  \includegraphics[width=0.25\textwidth]{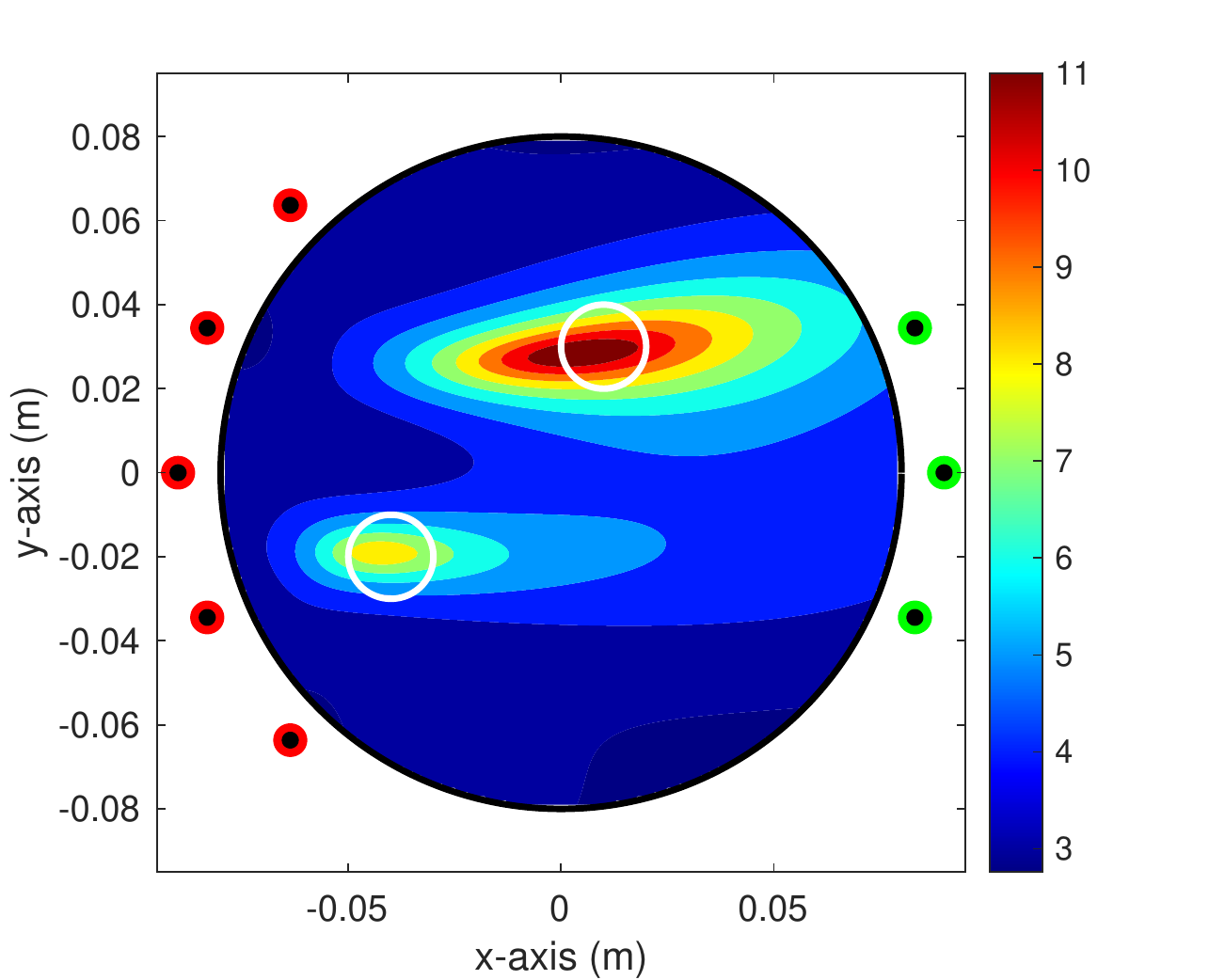}\hfill
  \includegraphics[width=0.25\textwidth]{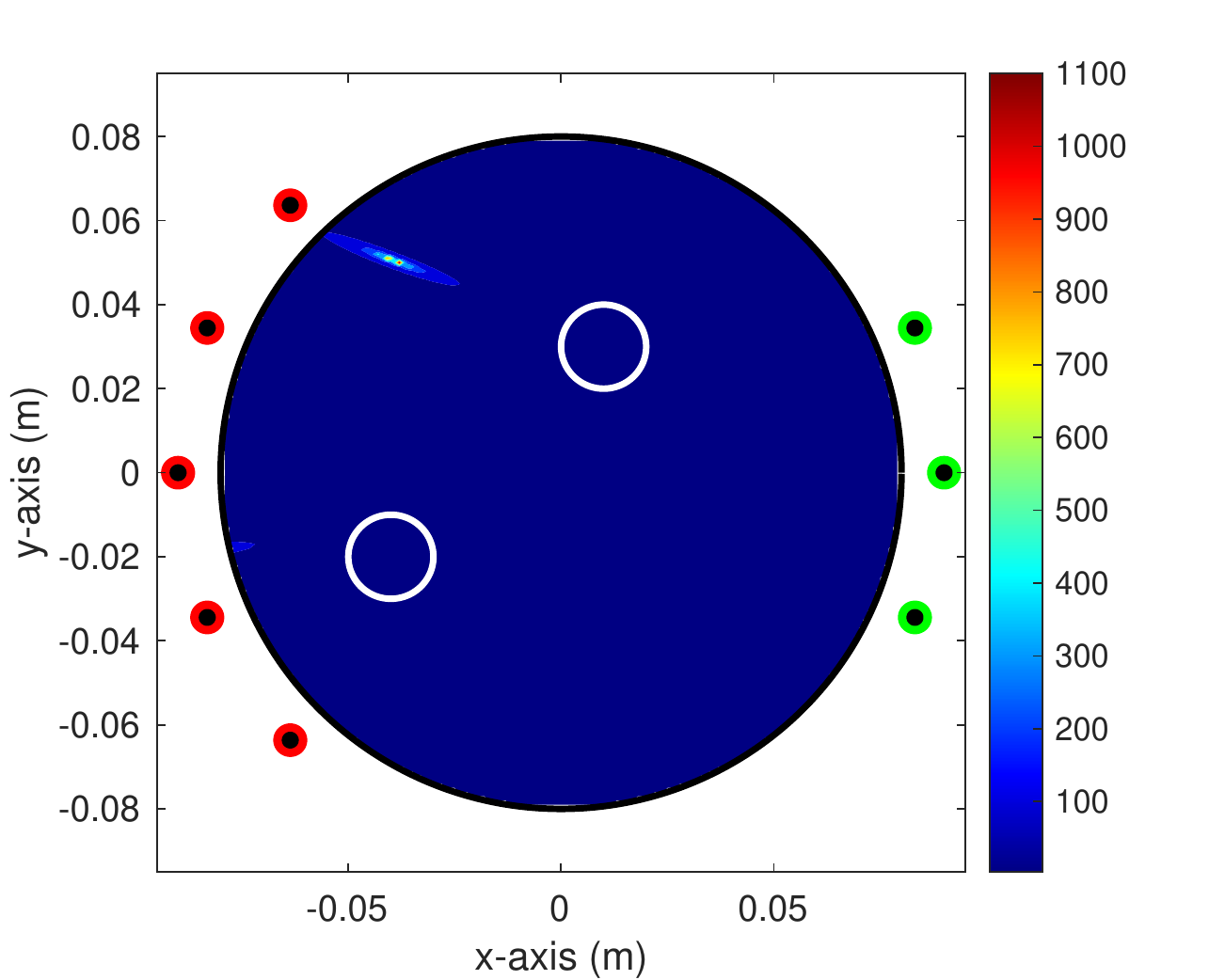}\hfill
  \includegraphics[width=0.25\textwidth]{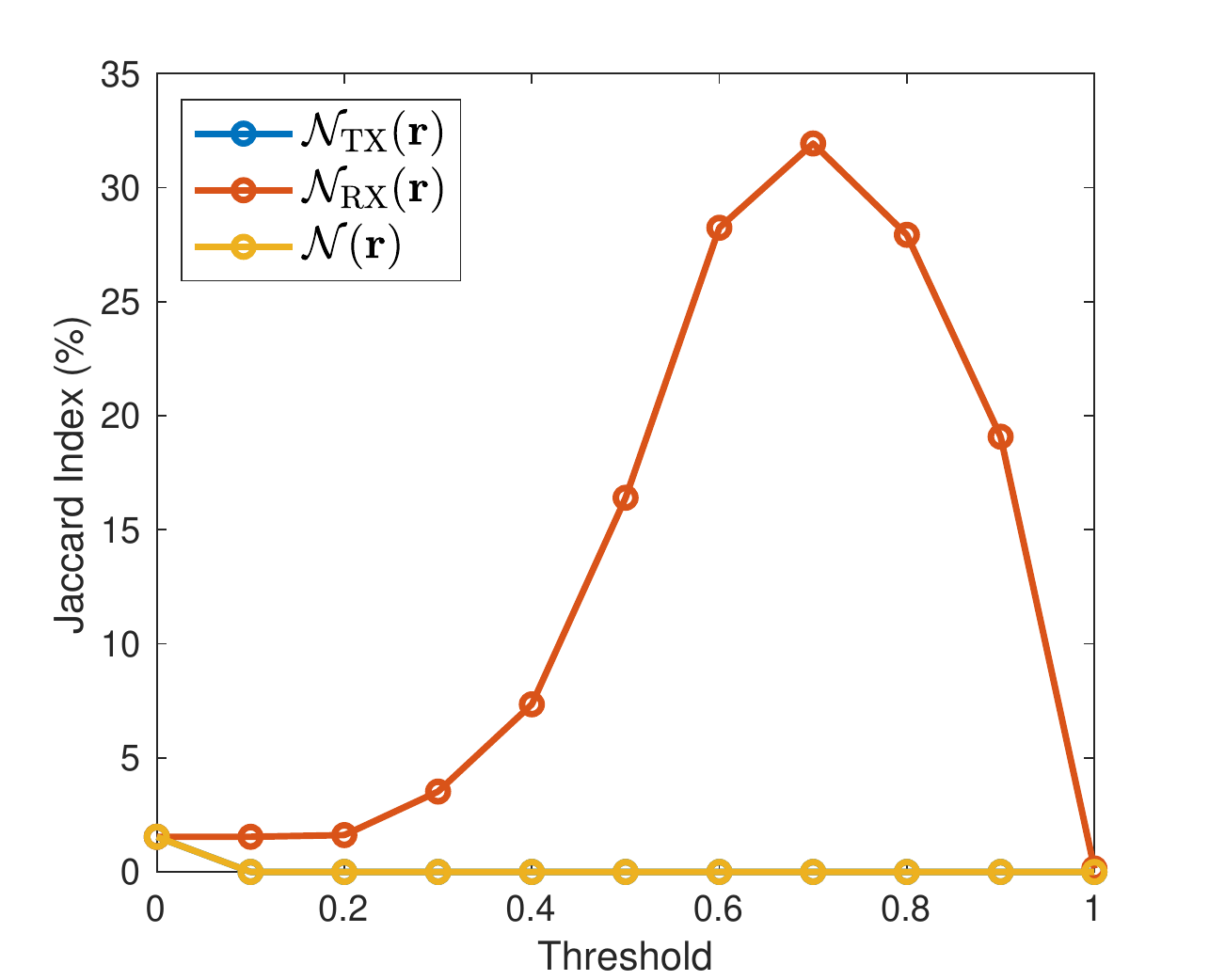}\\
  \includegraphics[width=0.25\textwidth]{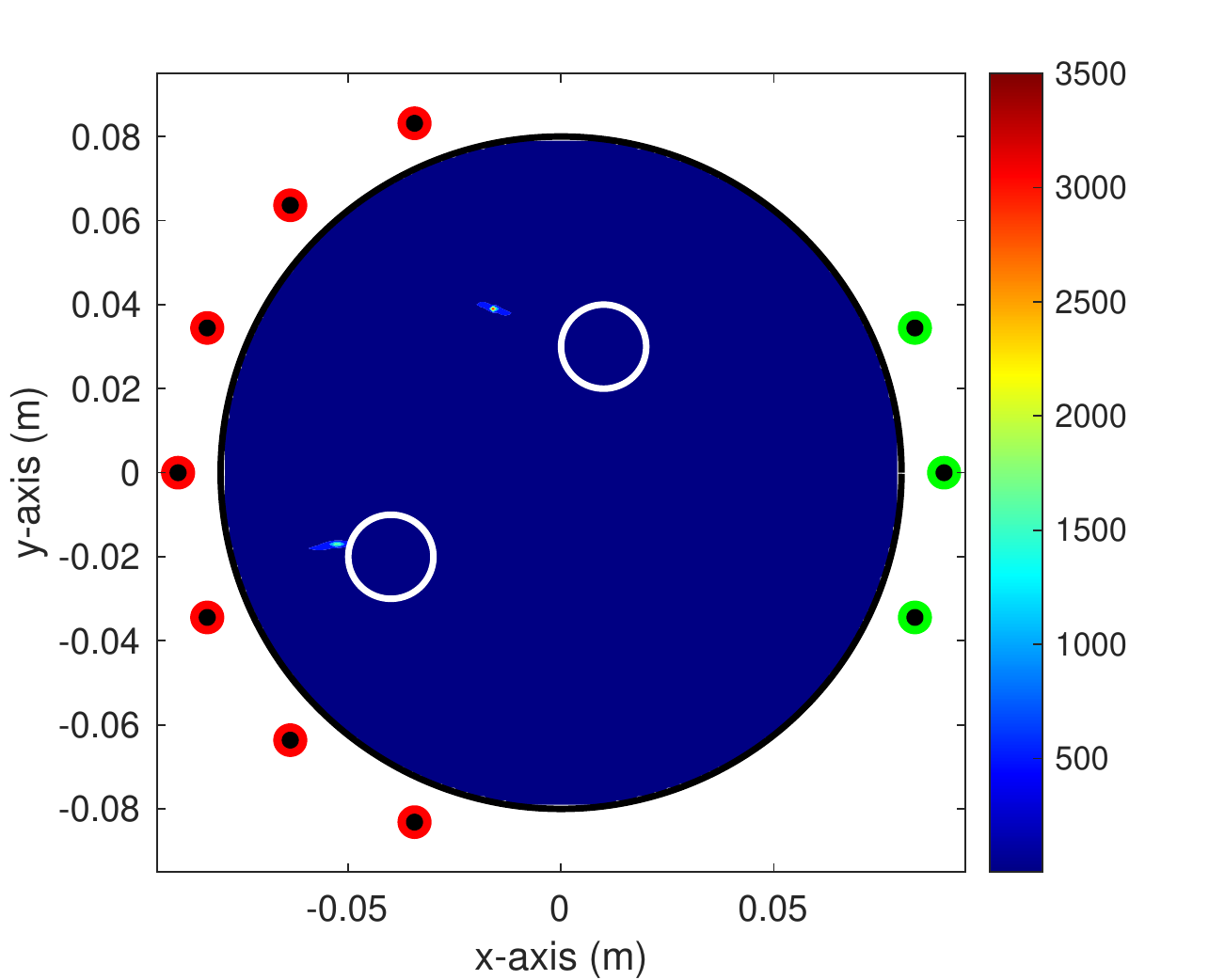}\hfill
  \includegraphics[width=0.25\textwidth]{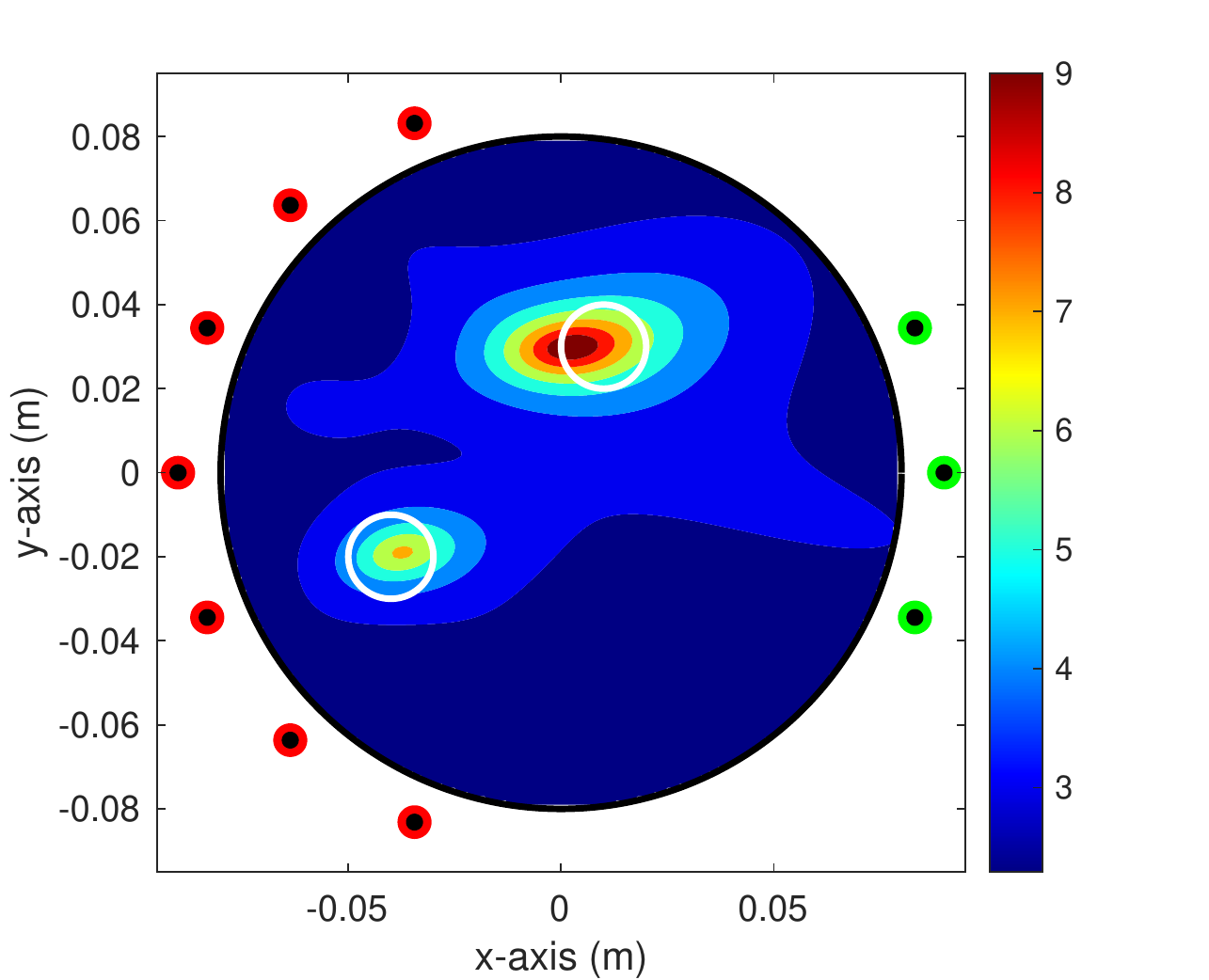}\hfill
  \includegraphics[width=0.25\textwidth]{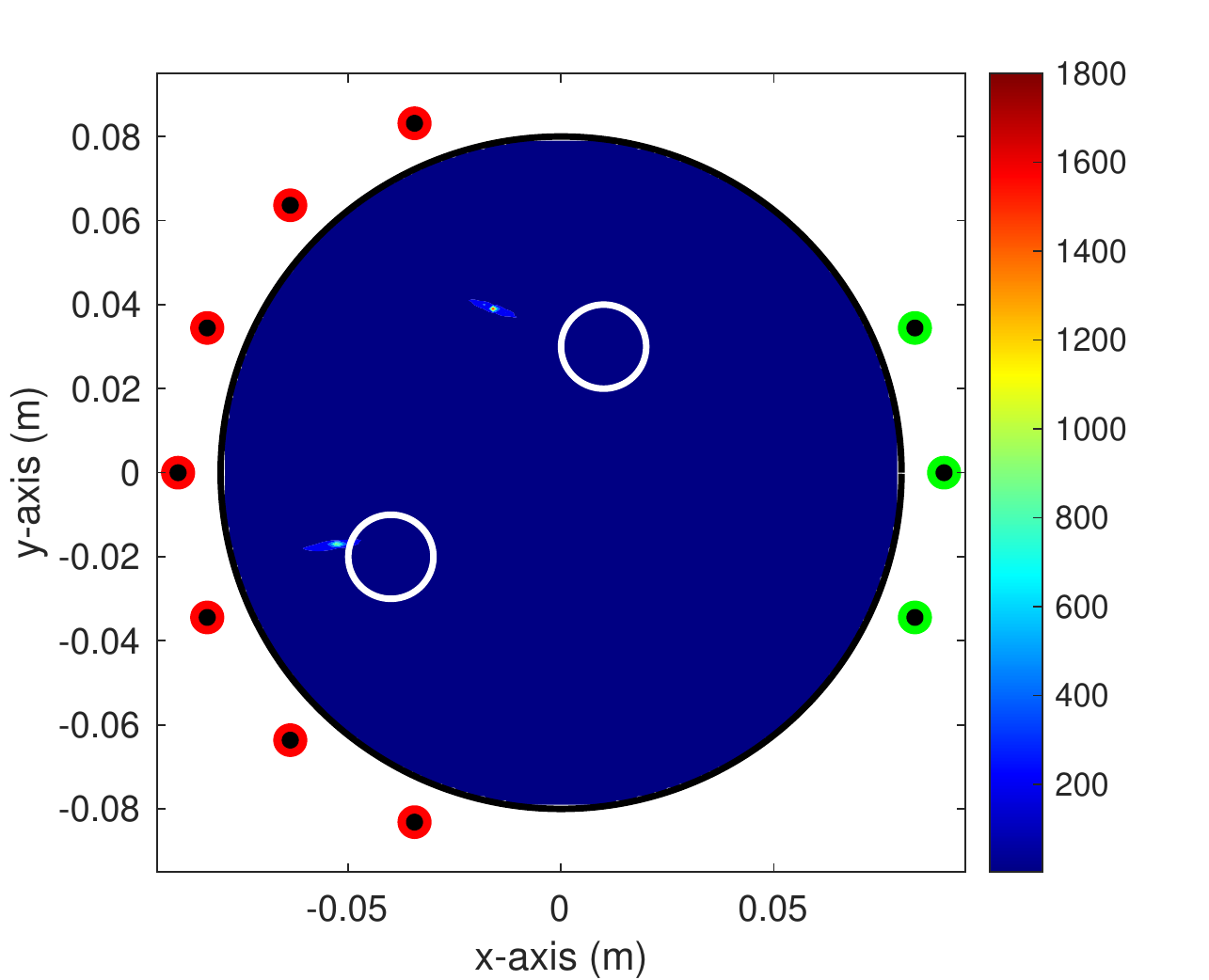}\hfill
  \includegraphics[width=0.25\textwidth]{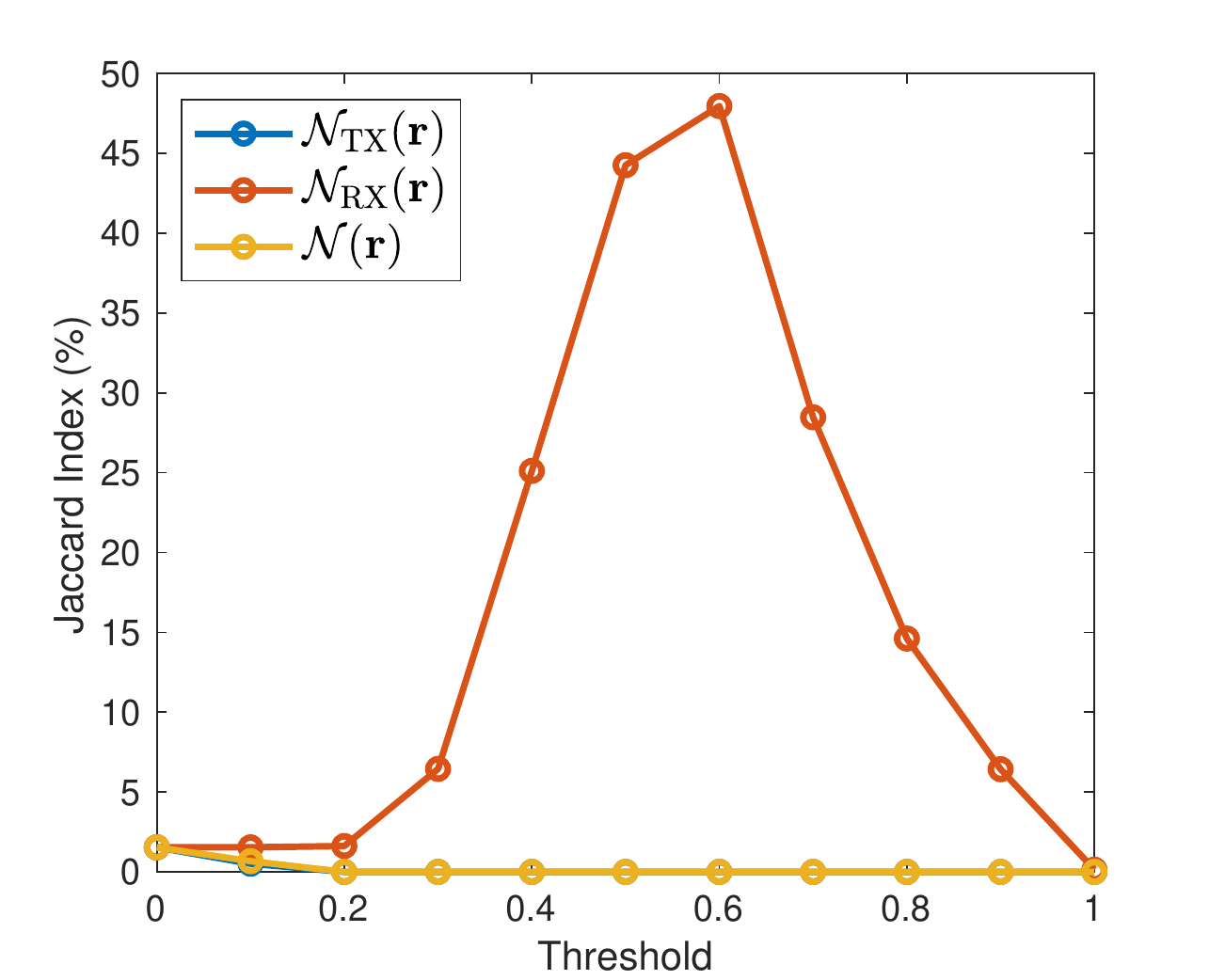}\\
  \includegraphics[width=0.25\textwidth]{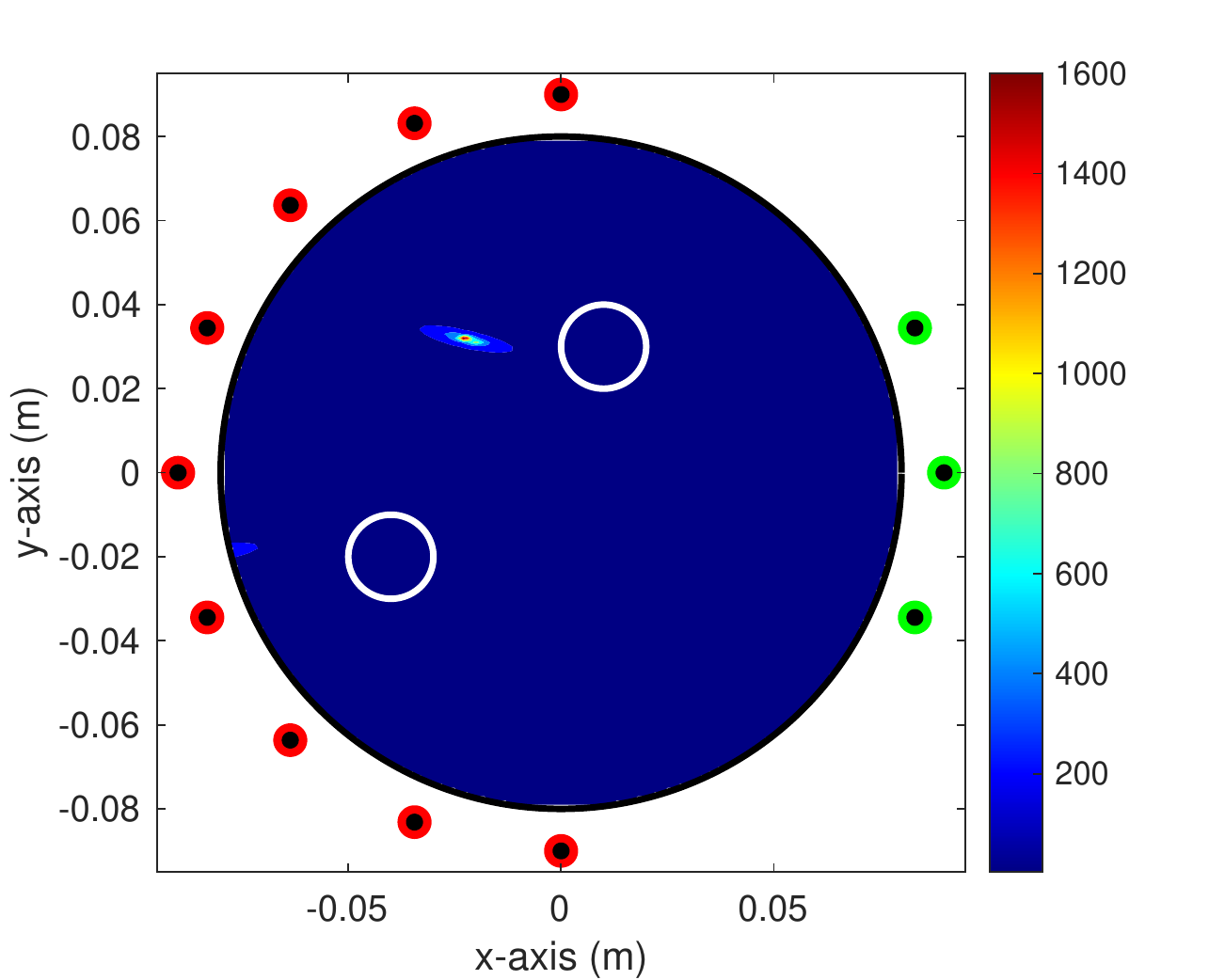}\hfill
  \includegraphics[width=0.25\textwidth]{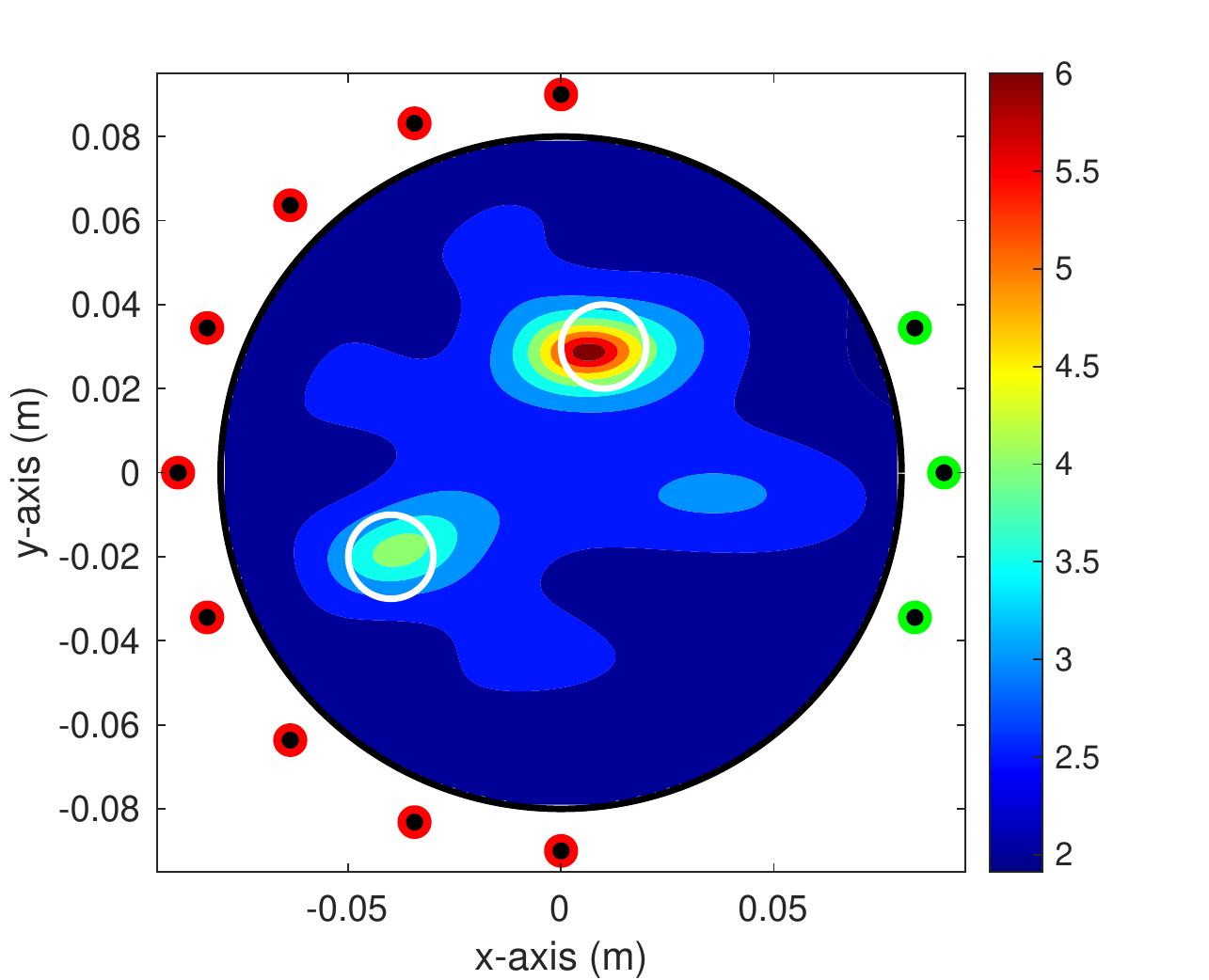}\hfill
  \includegraphics[width=0.25\textwidth]{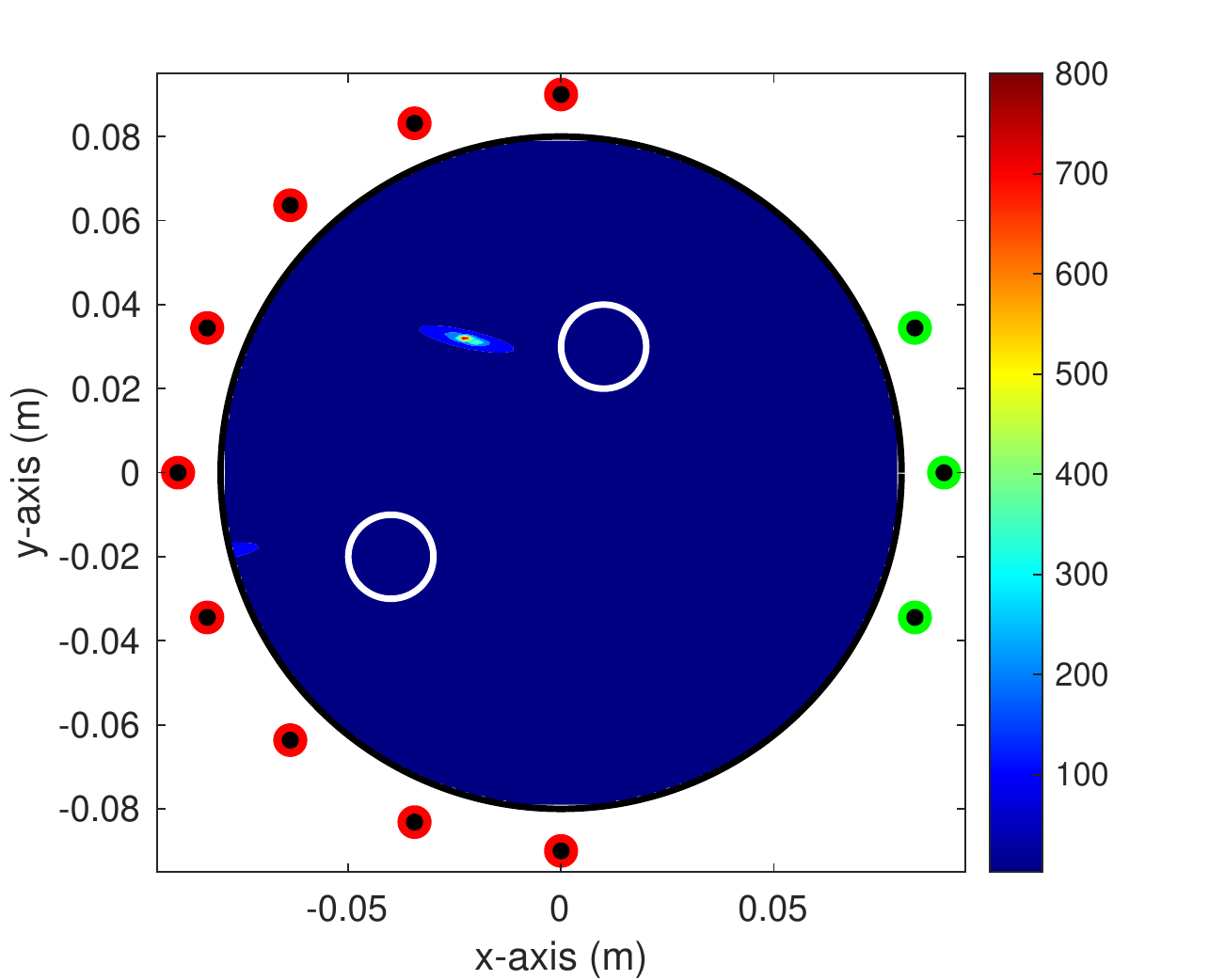}\hfill
  \includegraphics[width=0.25\textwidth]{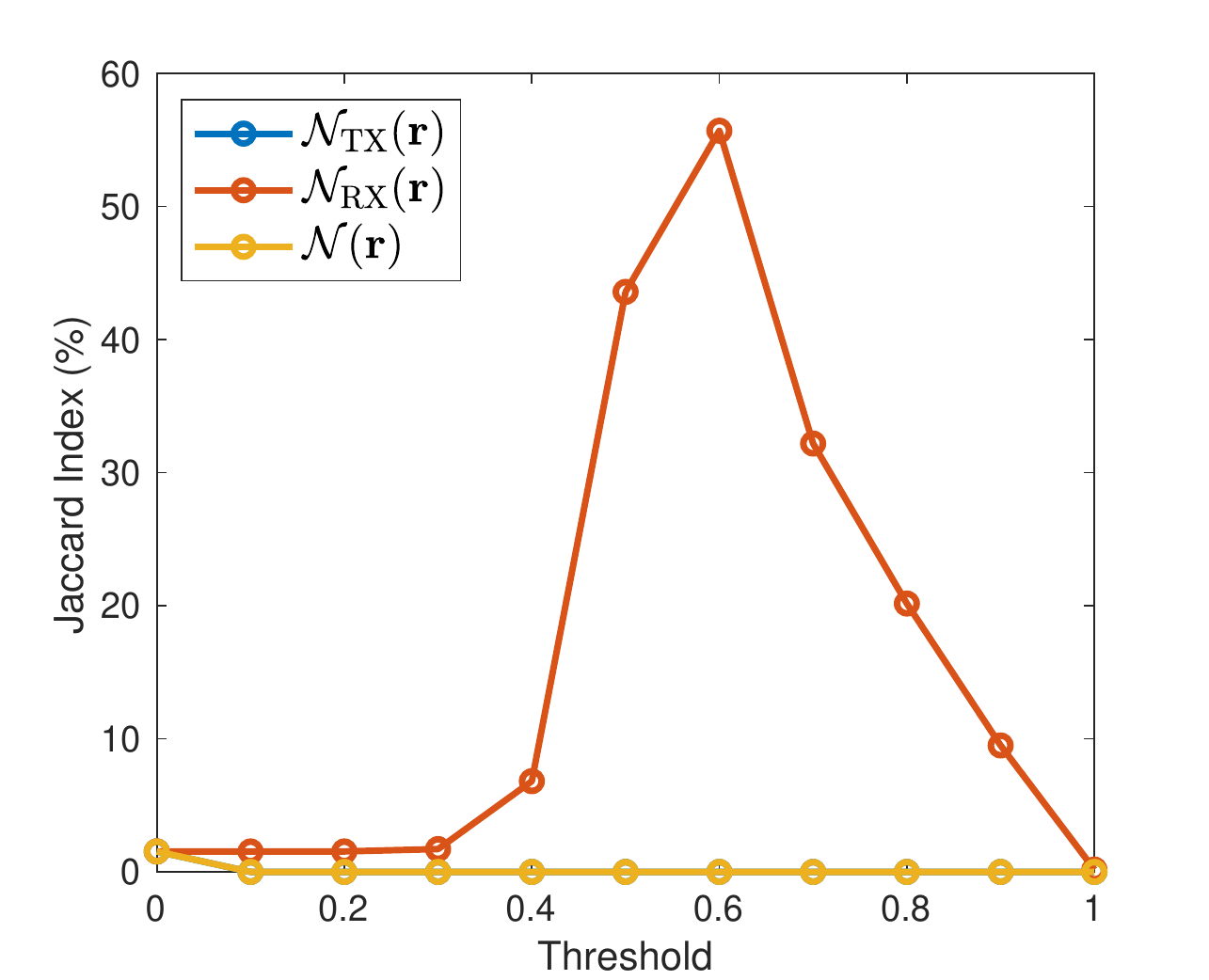}
  \caption{\label{Result4}(Example \ref{ex4}) Maps of $\mathfrak{F}_{\tx}(\mr)$ (first column), $\mathfrak{F}_{\rx}(\mr)$ (second column), $\mathfrak{F}(\mr)$ (third column), and Jaccard index (fourth column). Green and red colored circles describe the location of transmitters and receivers, respectively.}
\end{figure}

\begin{figure}[h]
  \centering
  \includegraphics[width=0.25\textwidth]{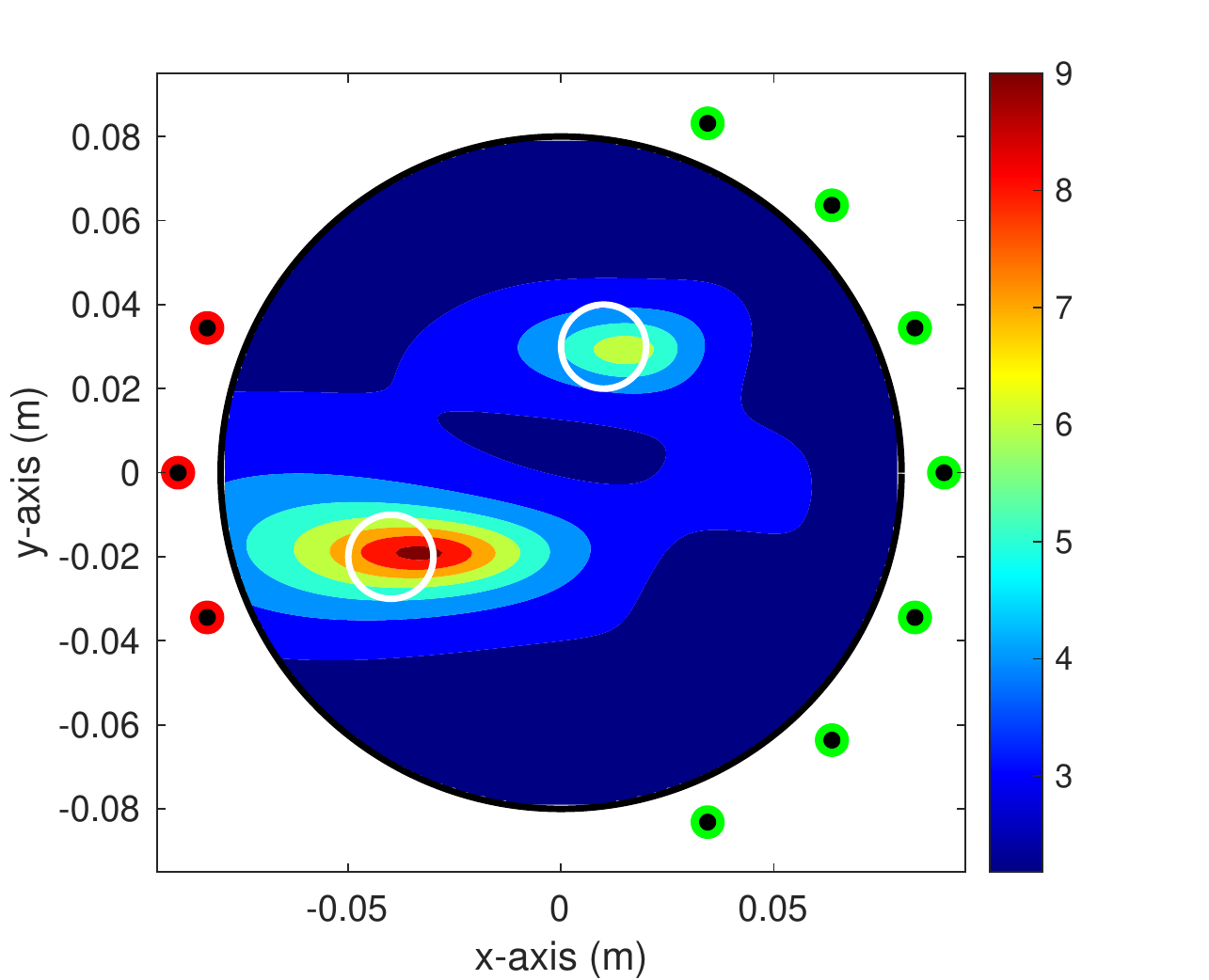}\hfill
  \includegraphics[width=0.25\textwidth]{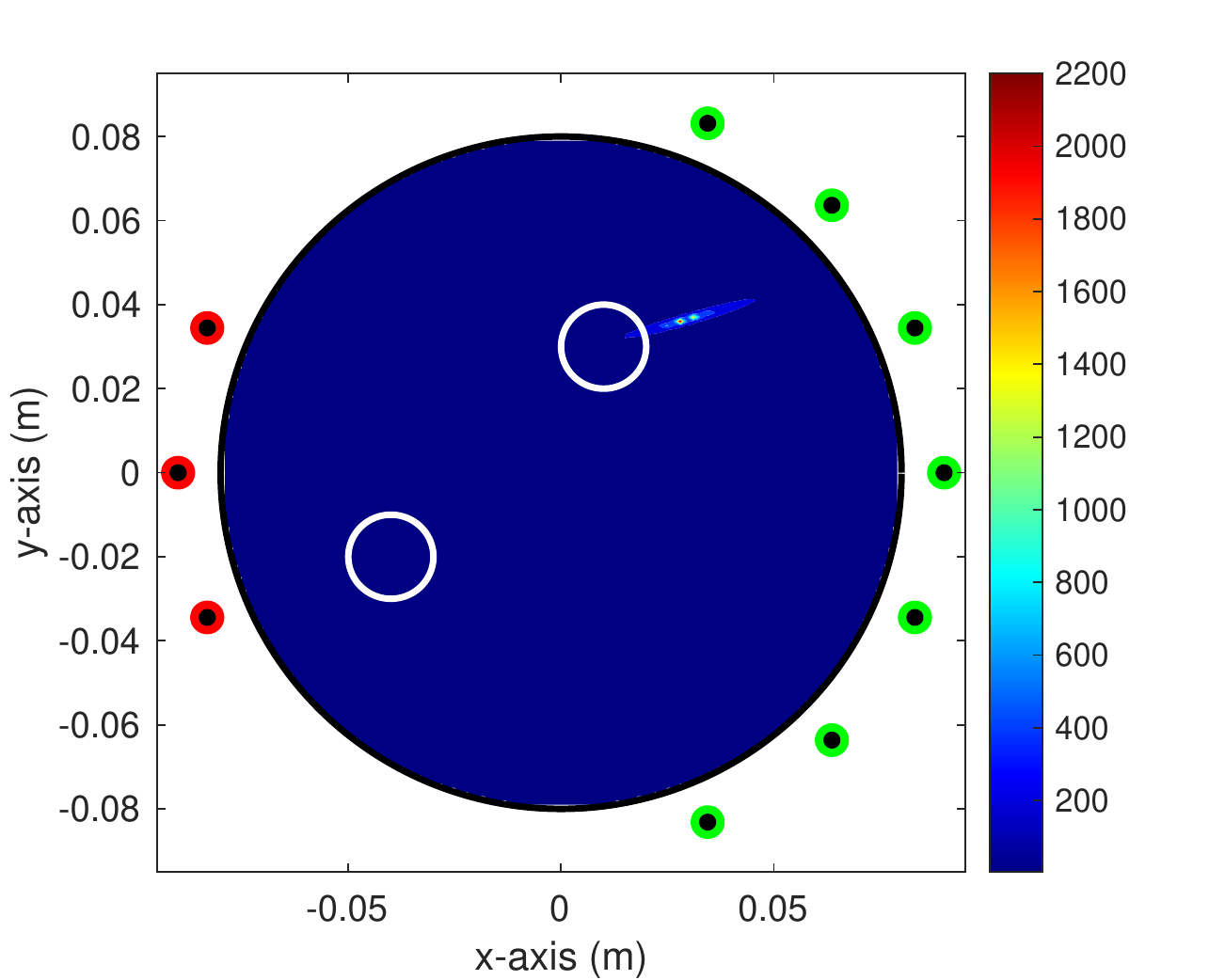}\hfill
  \includegraphics[width=0.25\textwidth]{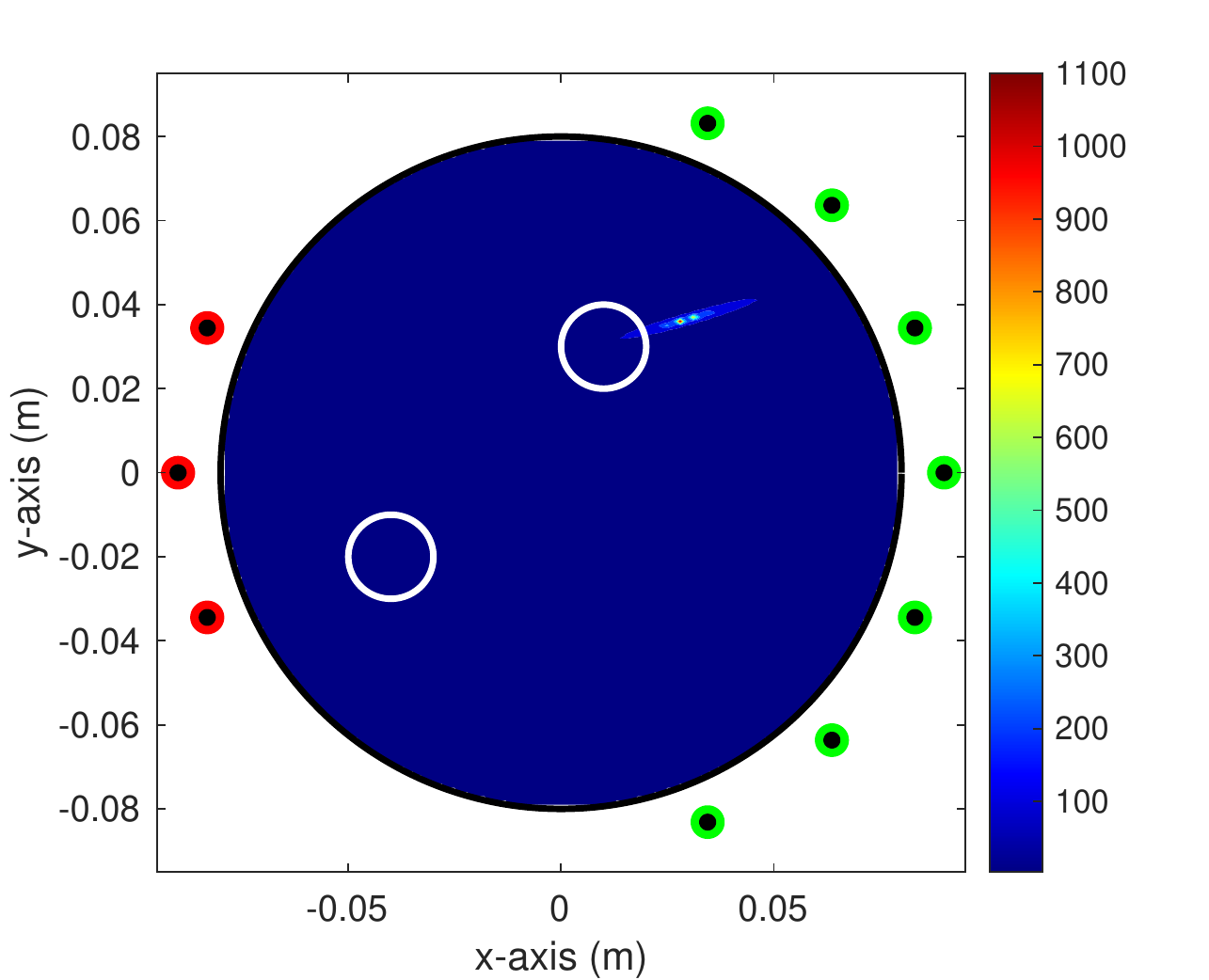}\hfill
  \includegraphics[width=0.25\textwidth]{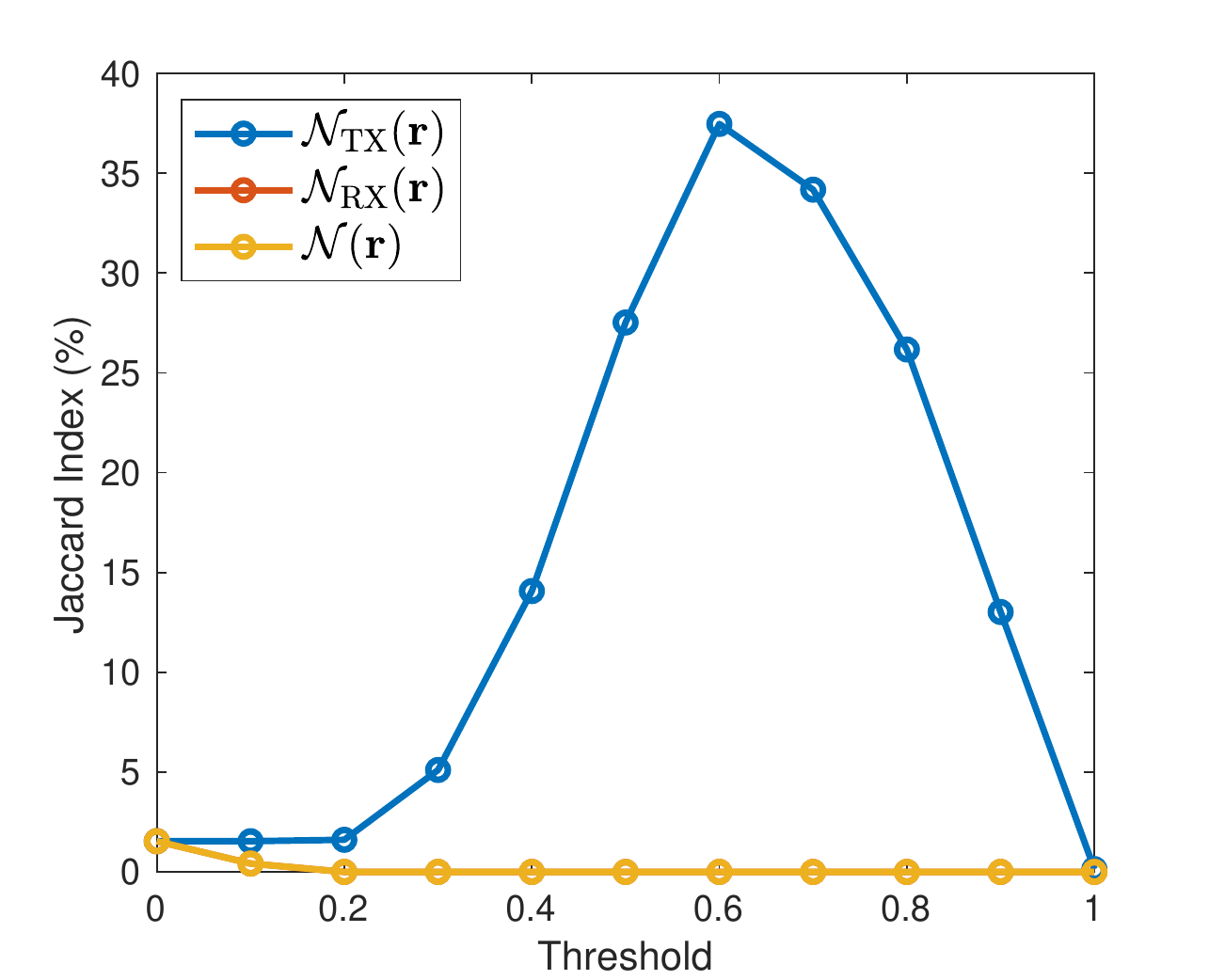}\\
  \includegraphics[width=0.25\textwidth]{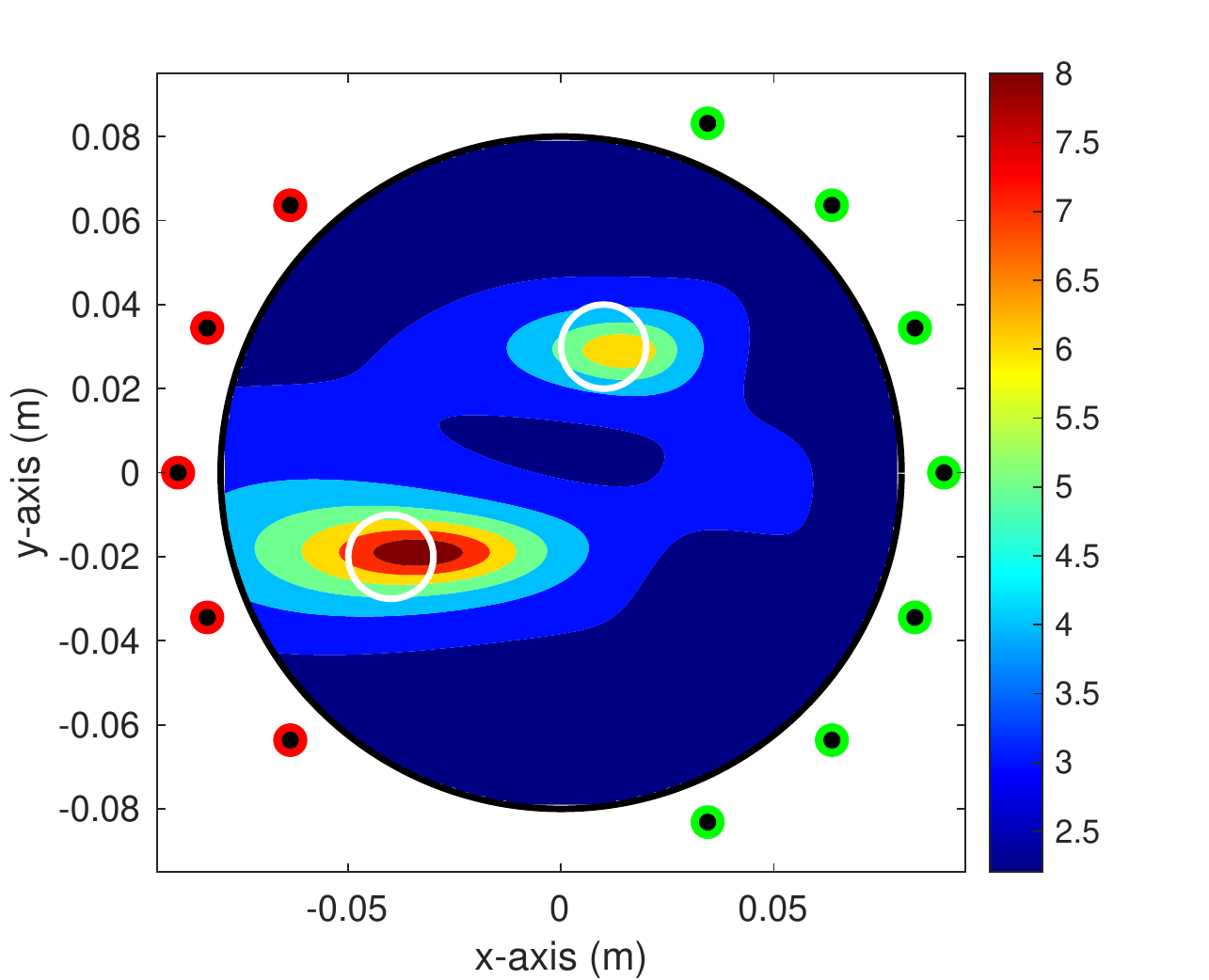}\hfill
  \includegraphics[width=0.25\textwidth]{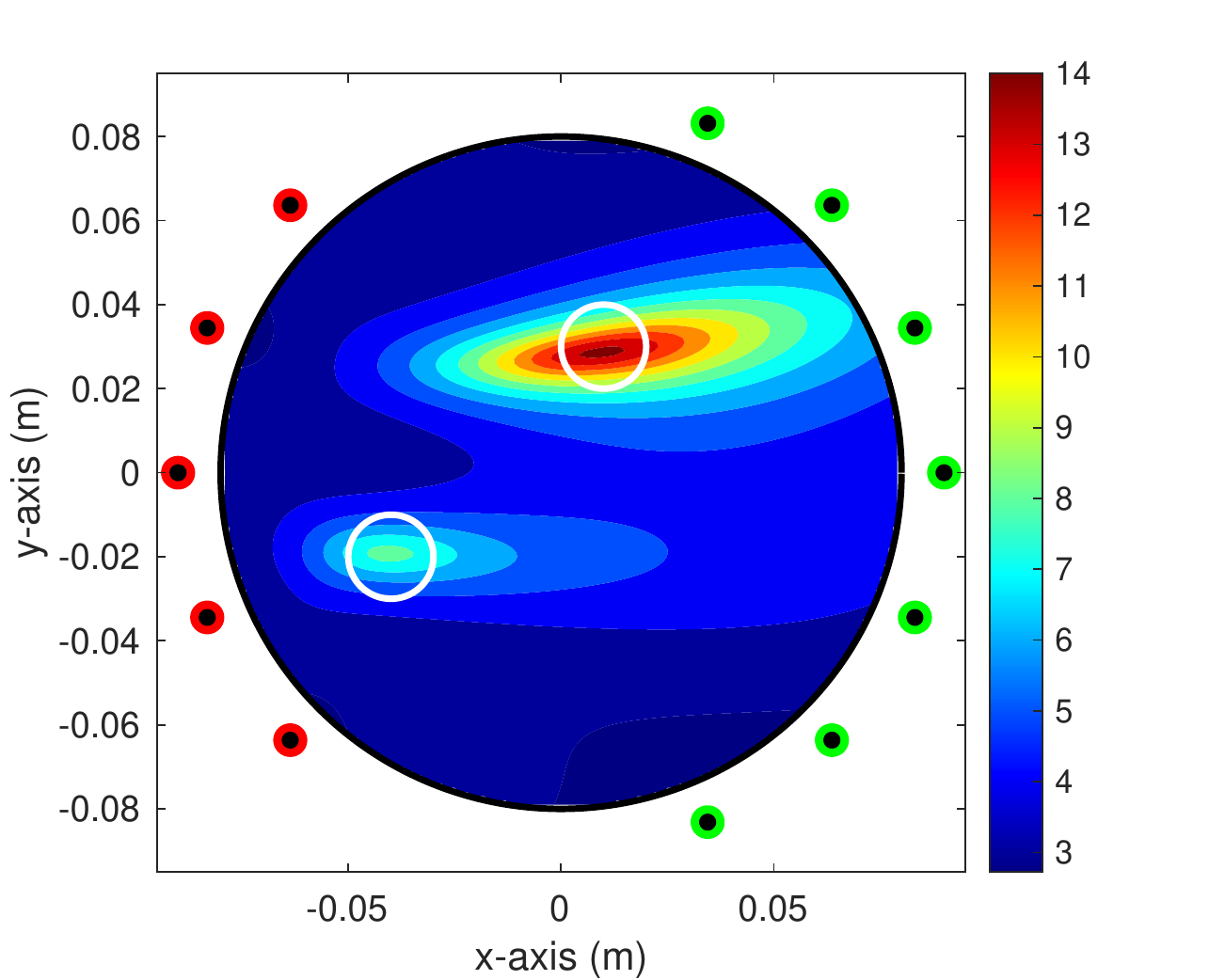}\hfill
  \includegraphics[width=0.25\textwidth]{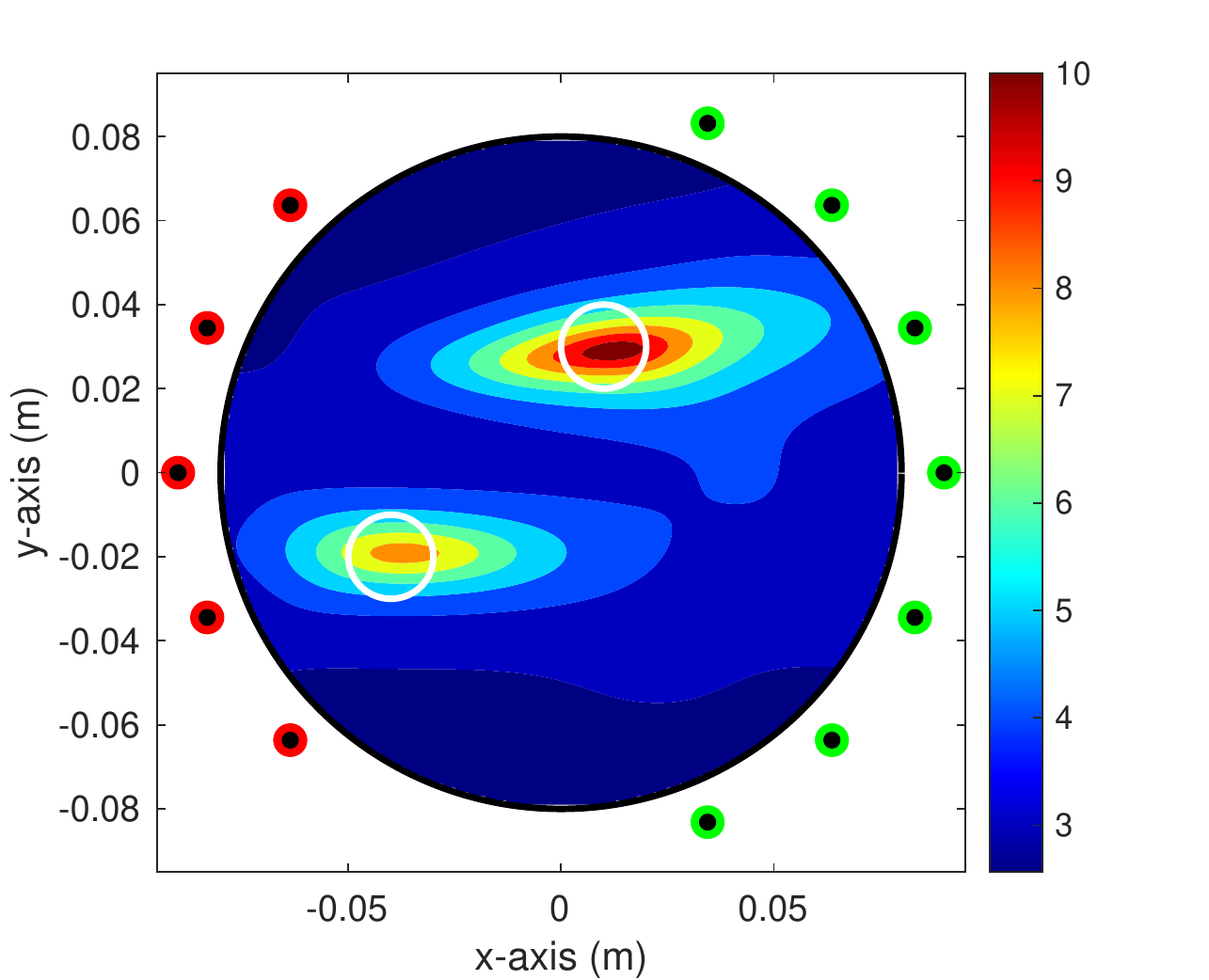}\hfill
  \includegraphics[width=0.25\textwidth]{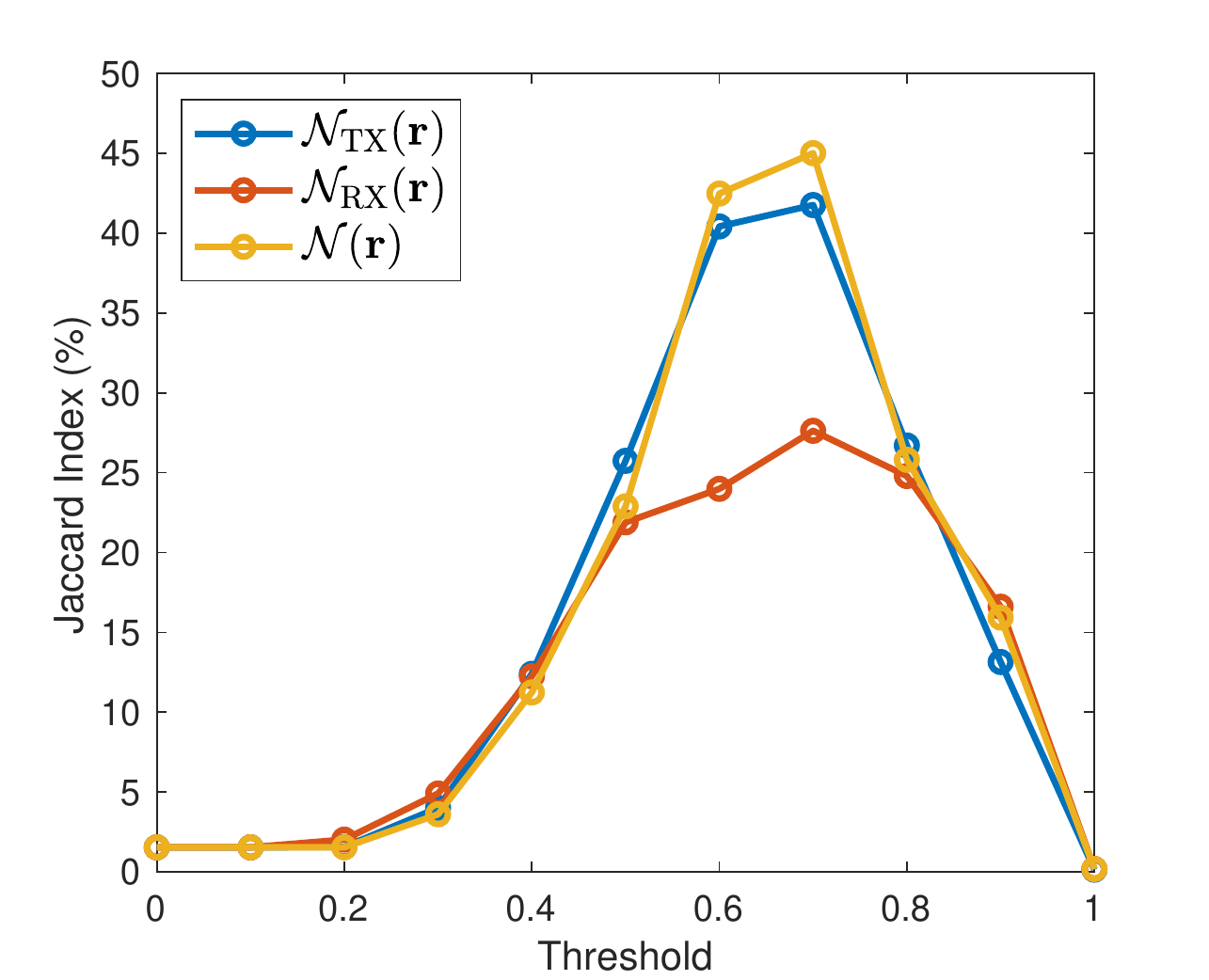}\\
  \includegraphics[width=0.25\textwidth]{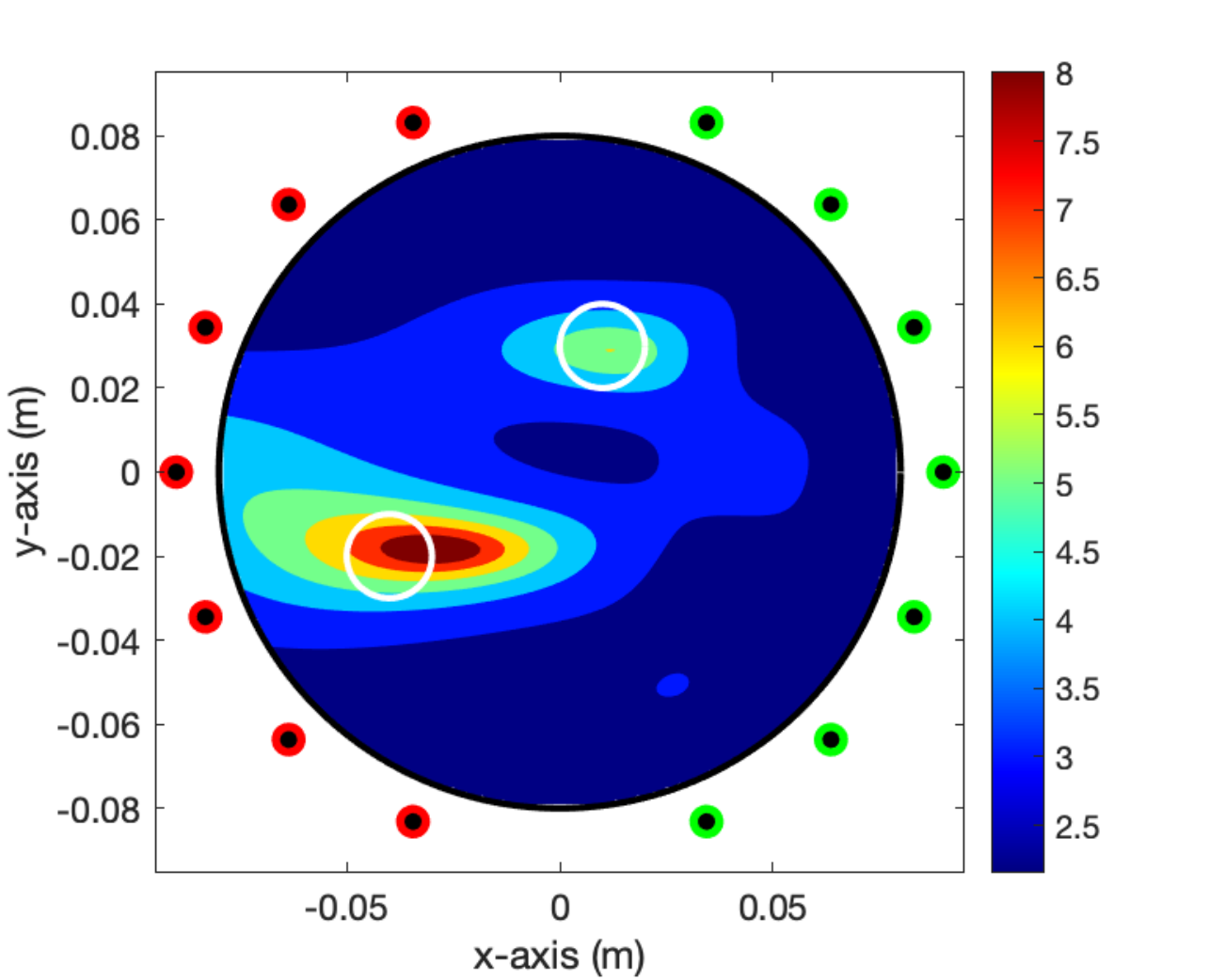}\hfill
  \includegraphics[width=0.25\textwidth]{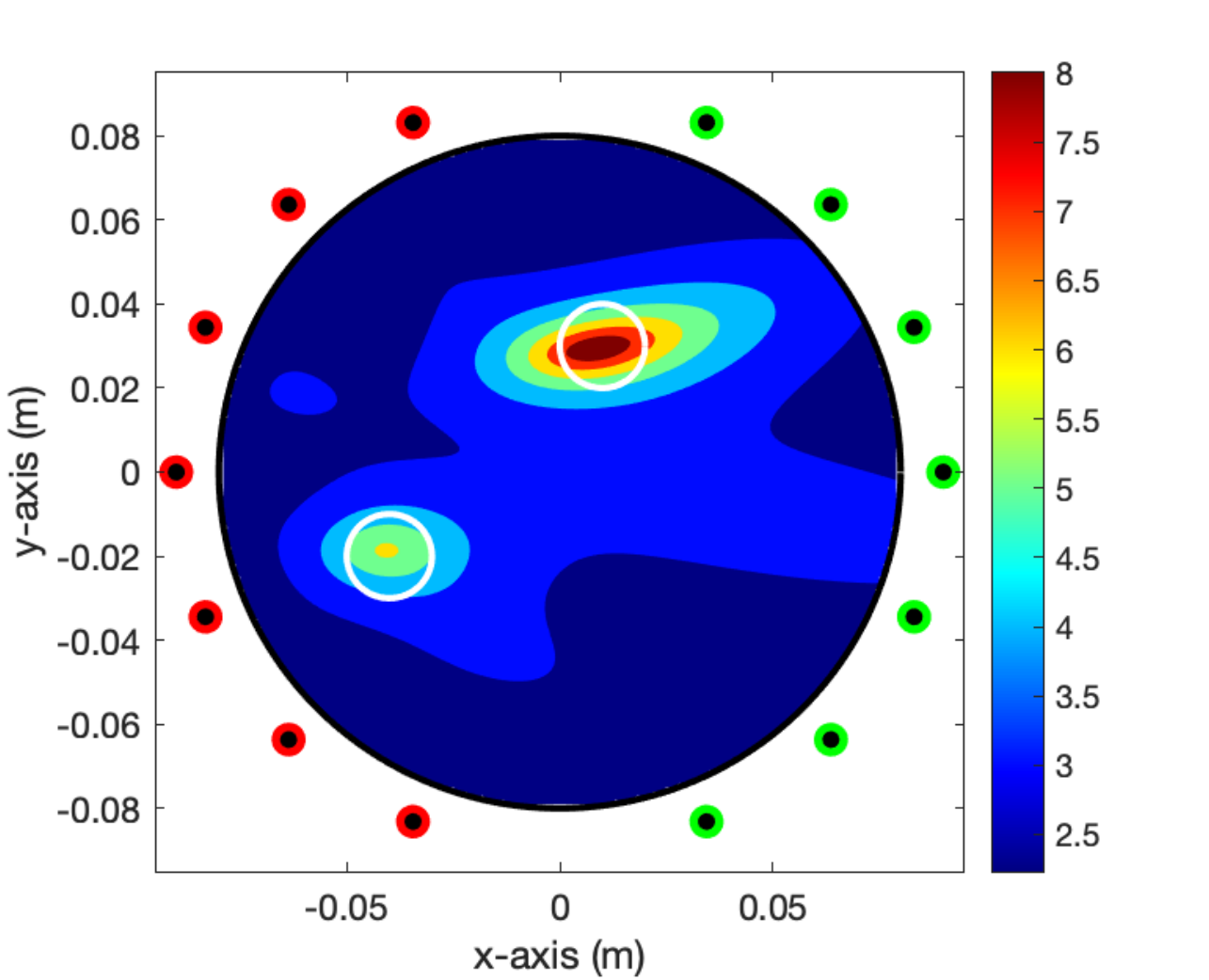}\hfill
  \includegraphics[width=0.25\textwidth]{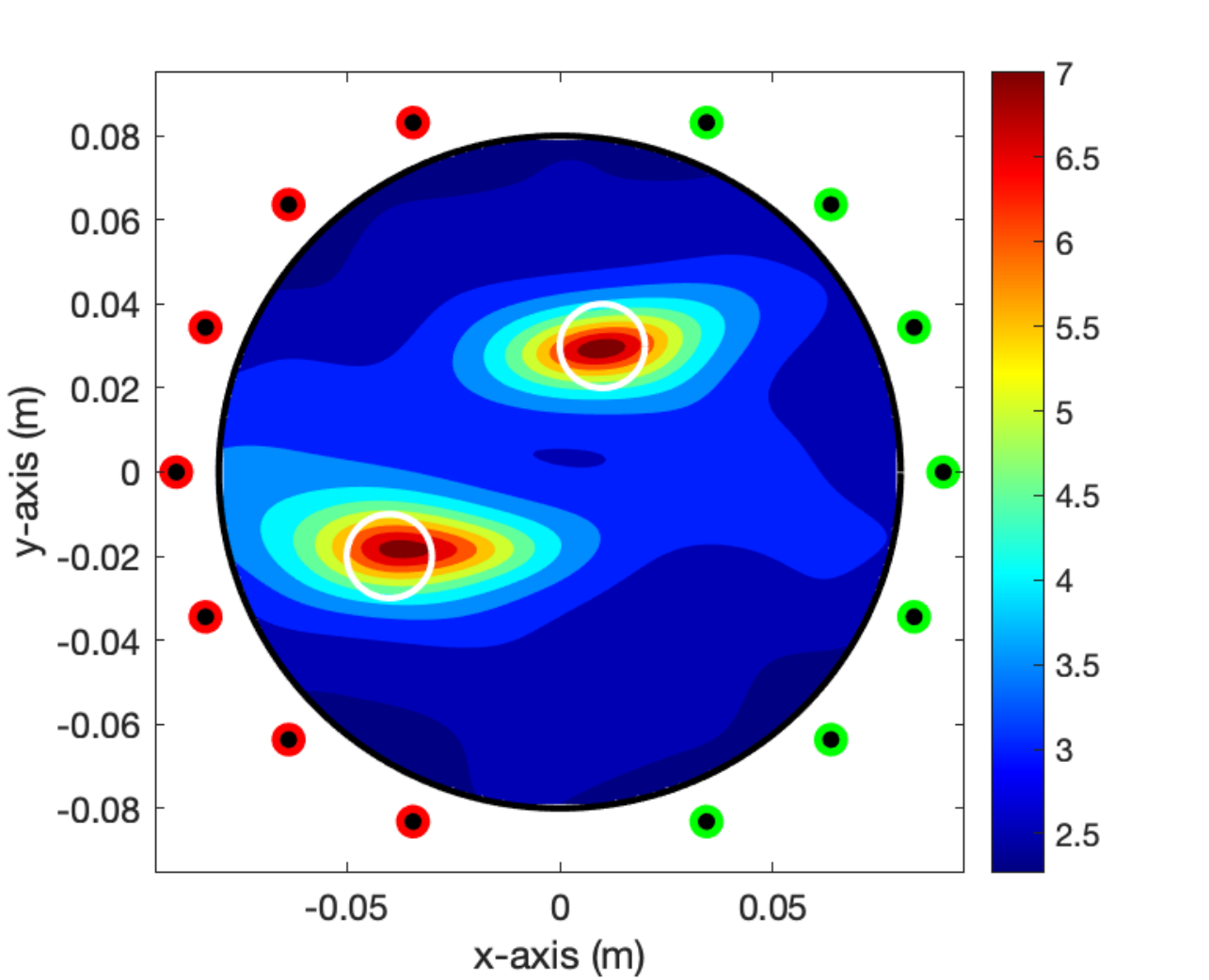}\hfill
  \includegraphics[width=0.25\textwidth]{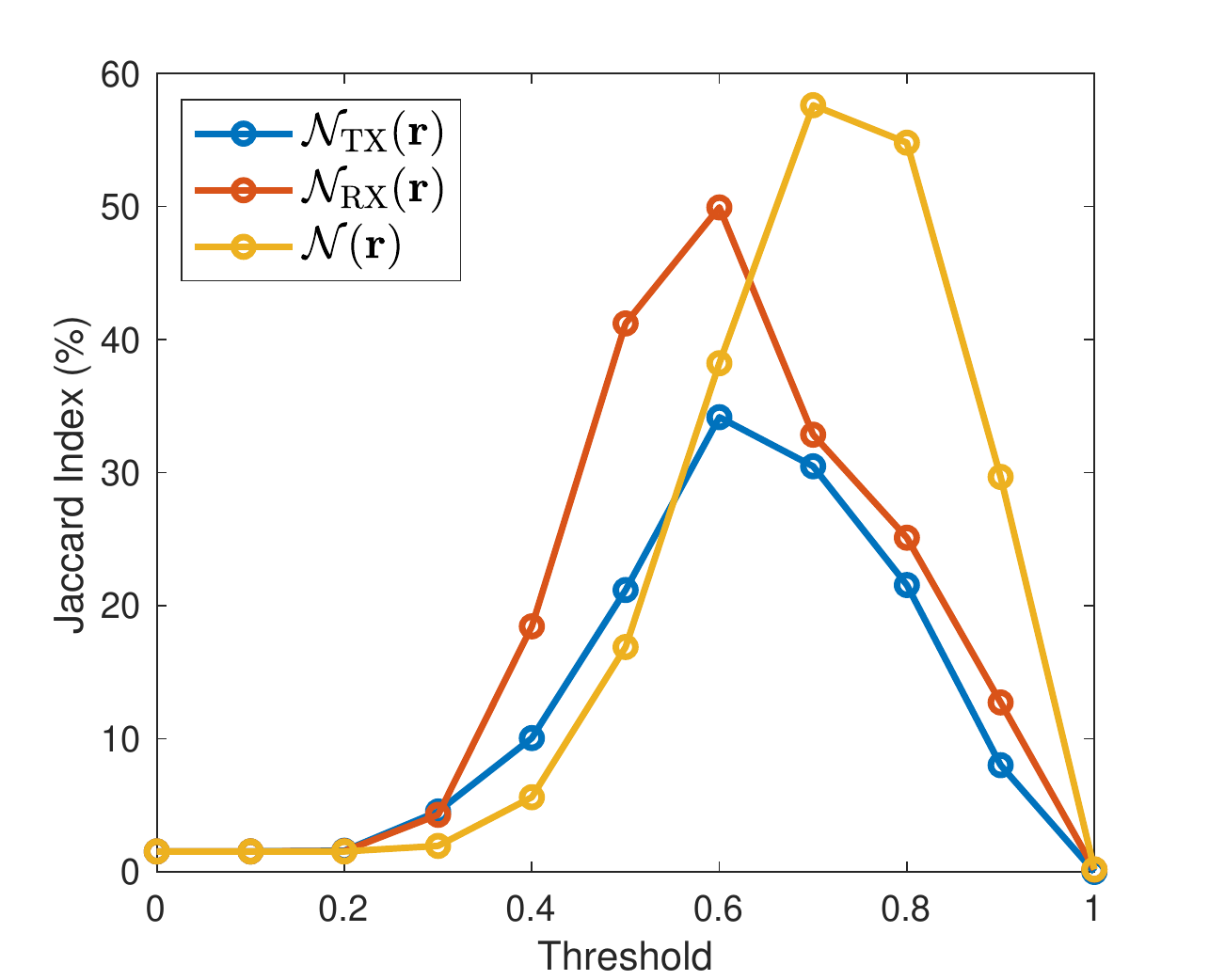}\\
  \includegraphics[width=0.25\textwidth]{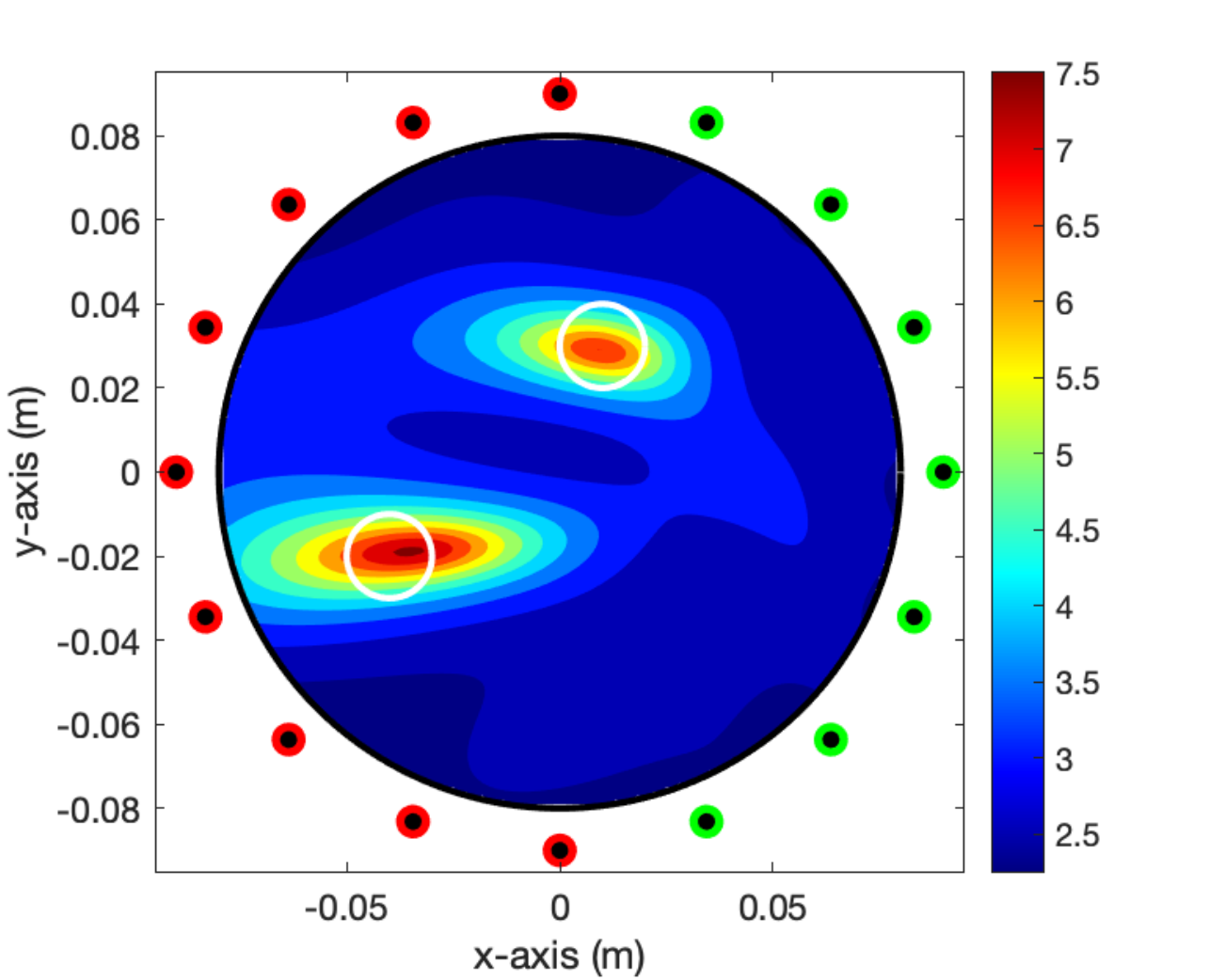}\hfill
  \includegraphics[width=0.25\textwidth]{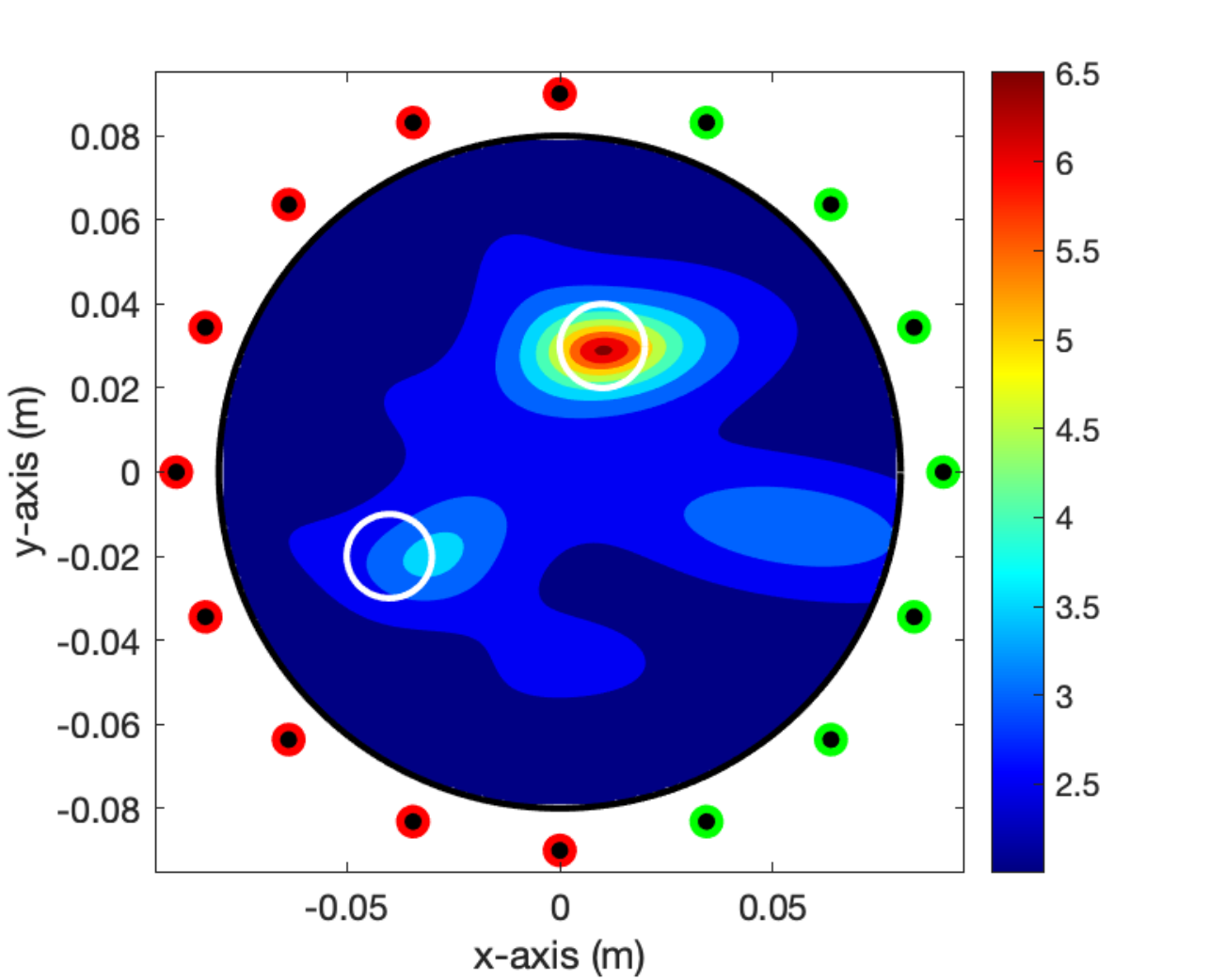}\hfill
  \includegraphics[width=0.25\textwidth]{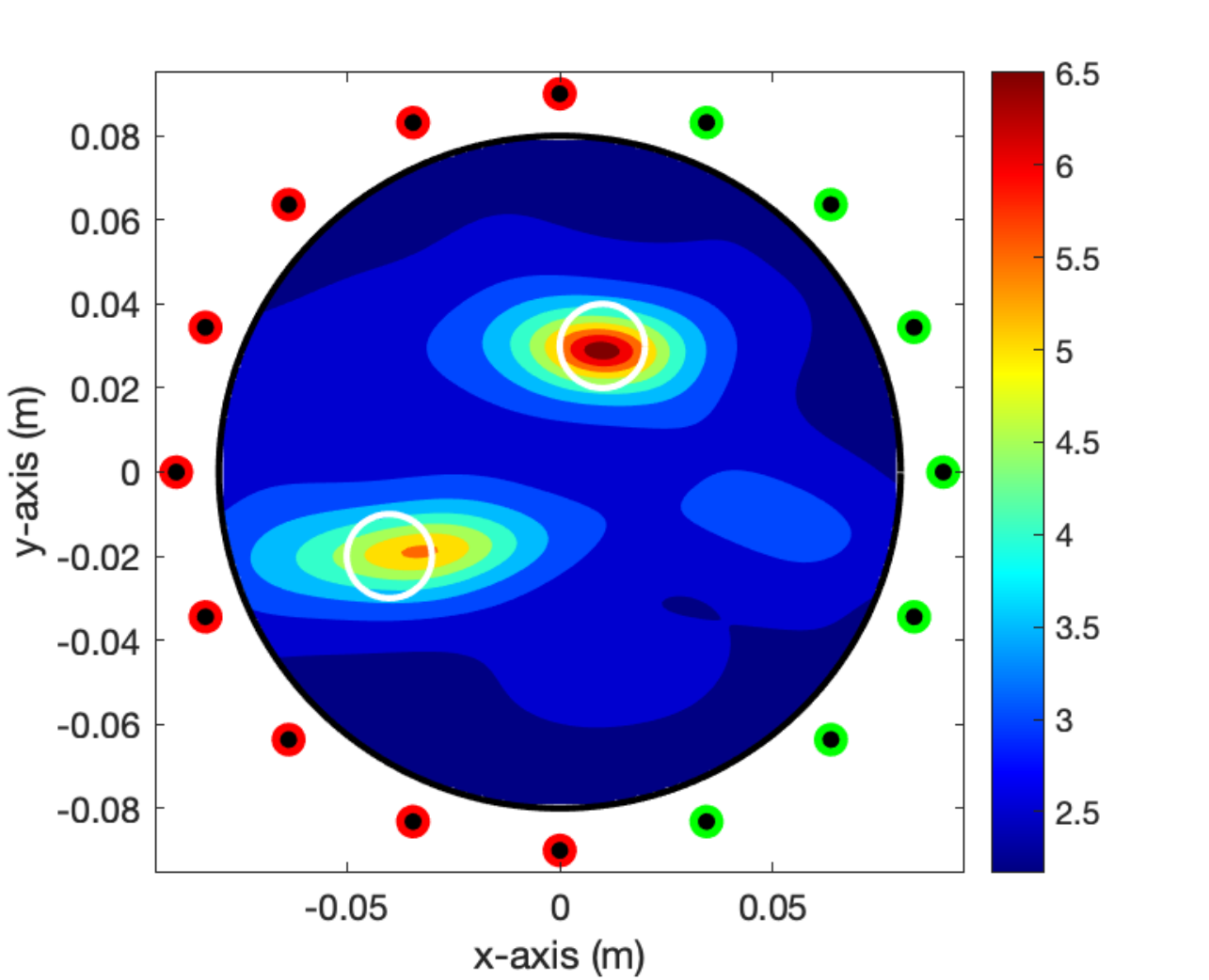}\hfill
  \includegraphics[width=0.25\textwidth]{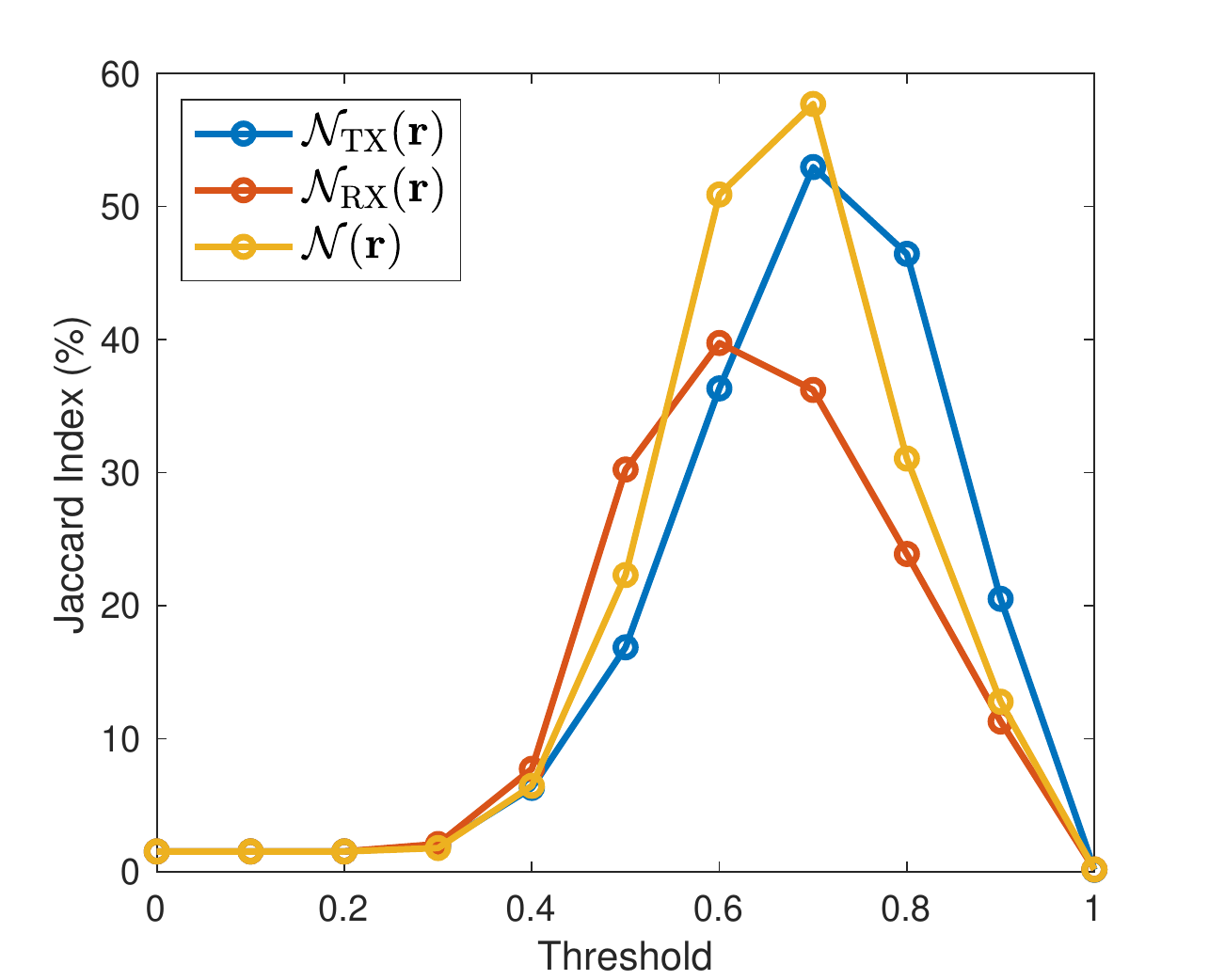}
  \caption{\label{Result5}(Example \ref{ex4}) Maps of $\mathfrak{F}_{\tx}(\mr)$ (first column), $\mathfrak{F}_{\rx}(\mr)$ (second column), $\mathfrak{F}(\mr)$ (third column), and Jaccard index (fourth column). Green and red colored circles describe the location of transmitters and receivers, respectively.}
\end{figure}

\begin{example}[Effects on the Arrangement of Antennas: Multiple Anomalies]\label{ex5}
Here, we consider the imaging results with the antenna settings introduced in the Example \ref{ex3}. Similar to the results in Figure \ref{Result3}, it is very difficult to distinguish the location of anomalies and artifacts with the antenna settings $\mathbf{A}_5\cup\mathbf{B}_5$, $\mathbf{A}_6\cup\mathbf{B}_5$, and $\mathbf{A}_7\cup\mathbf{B}_5$. Moreover, opposite to the imaging of single anomaly, it is very difficult to identify two anomalies due to the presence of unexpected artifacts via the map of $\mathfrak{F}_{\tx}(\mr)$ and extremely small value of $\mathfrak{F}_{\rx}(\mr_2)$. On the other hand, the location of anomalies can be identified through the map of $\mathfrak{F}(\mr)$.
\end{example}

\begin{figure}[h]
  \centering
  \includegraphics[width=0.25\textwidth]{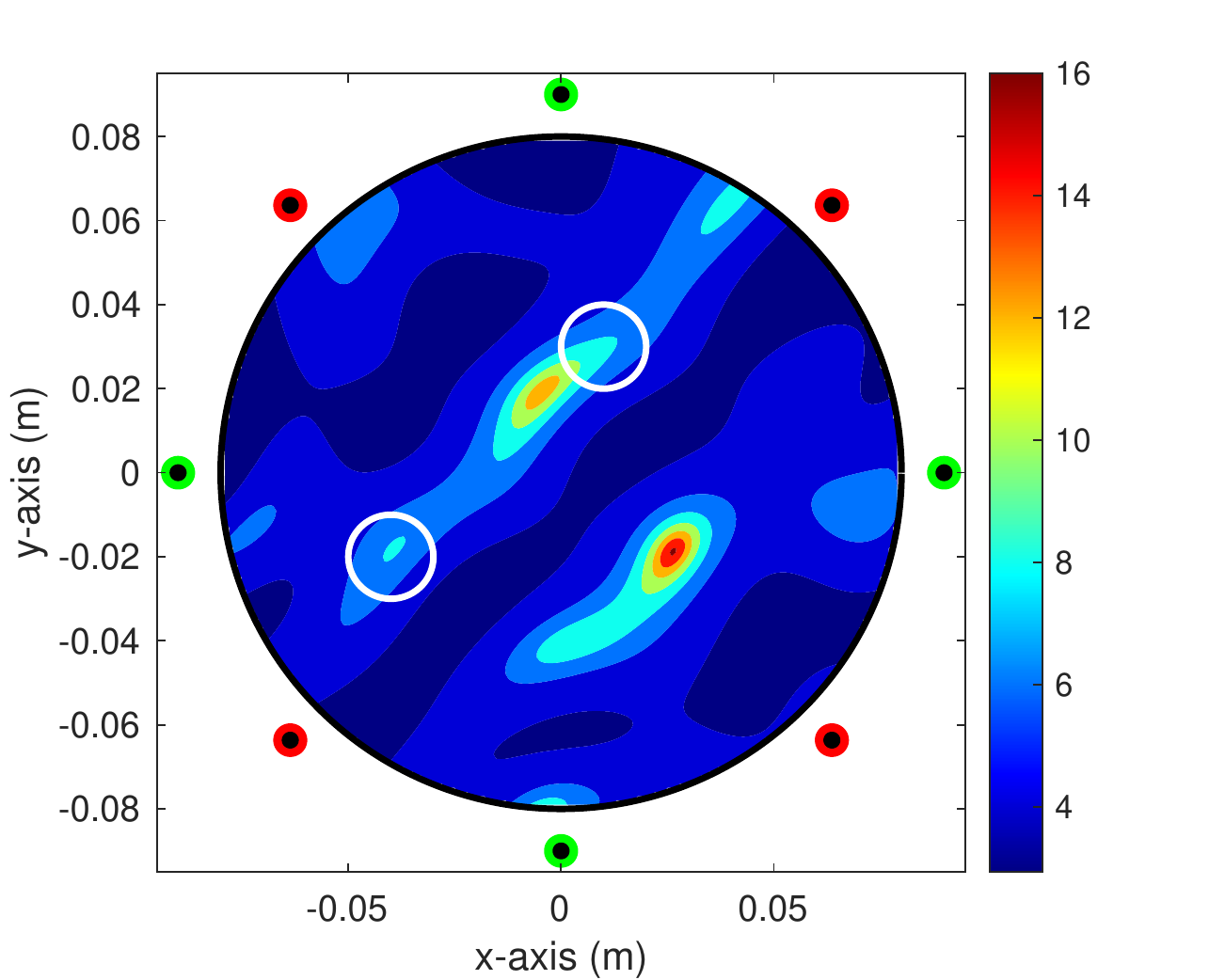}\hfill
  \includegraphics[width=0.25\textwidth]{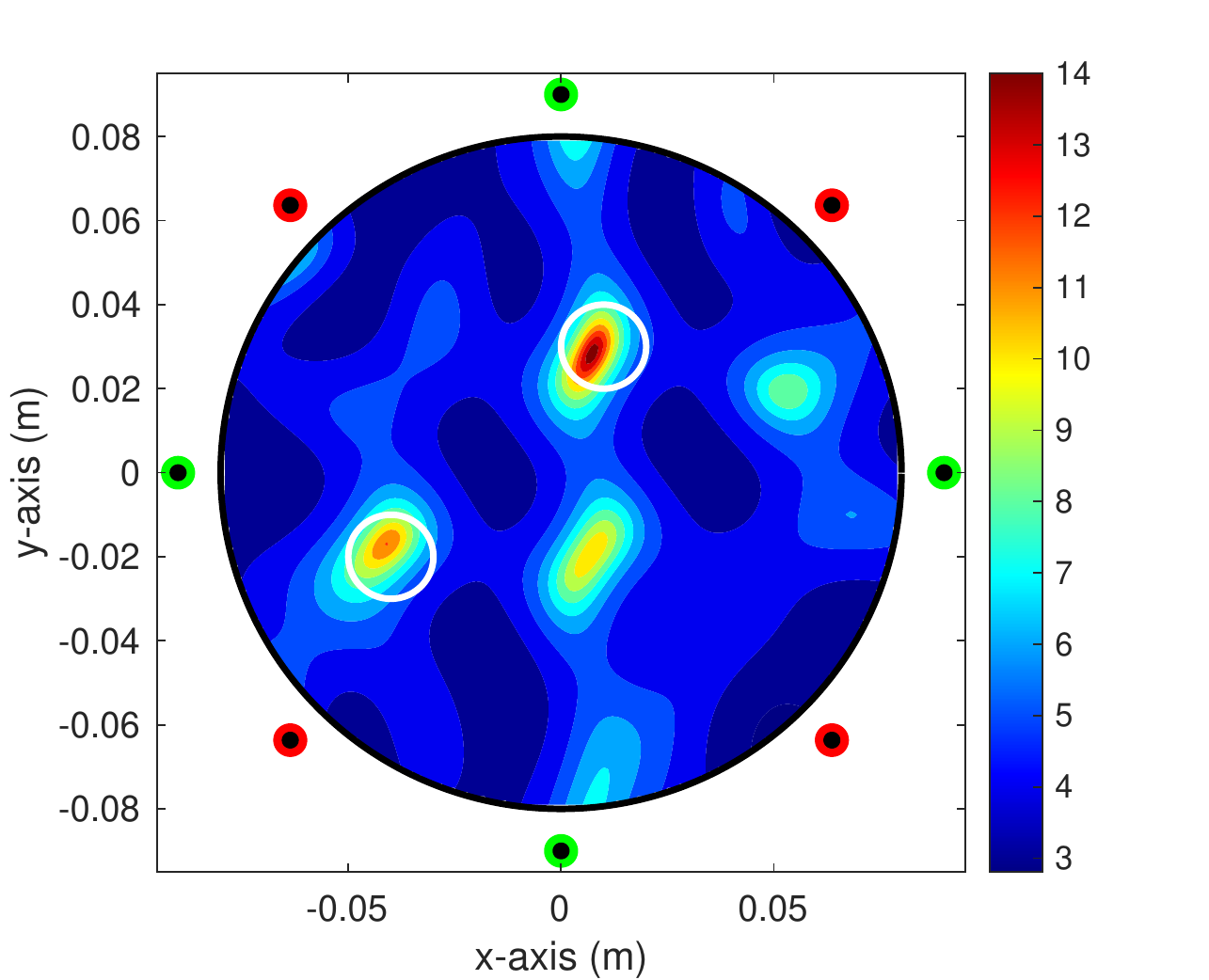}\hfill
  \includegraphics[width=0.25\textwidth]{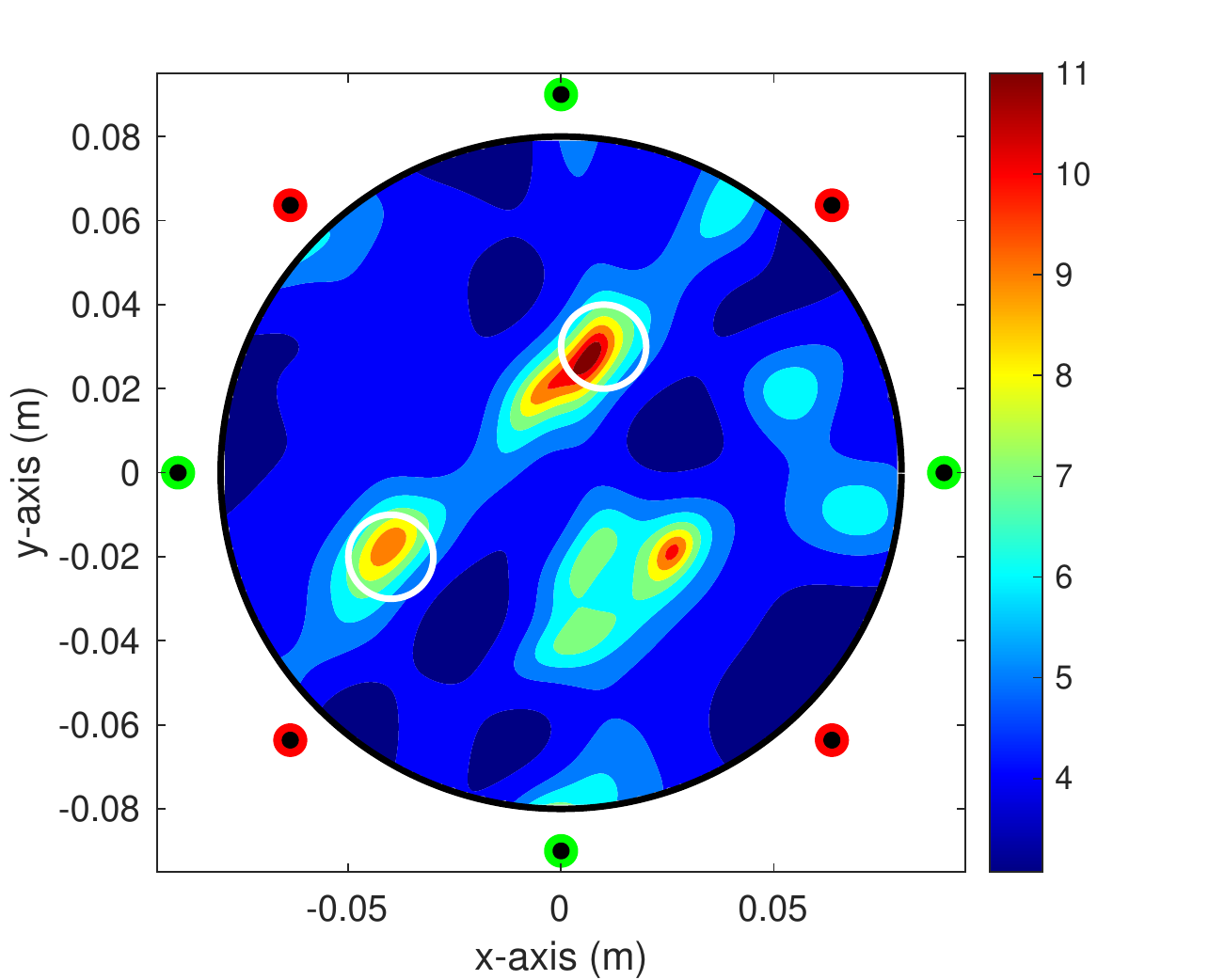}\hfill
  \includegraphics[width=0.25\textwidth]{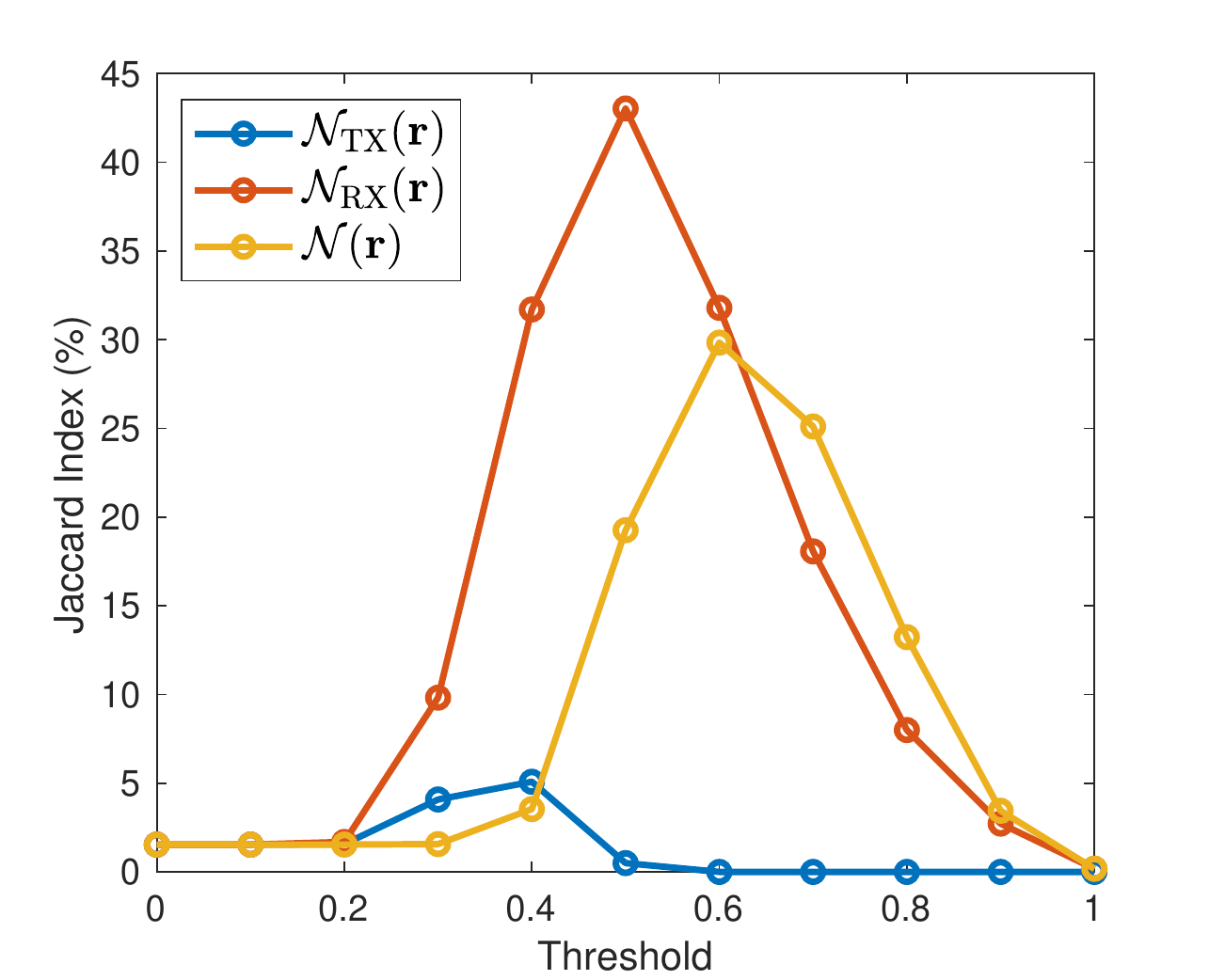}\\
  \includegraphics[width=0.25\textwidth]{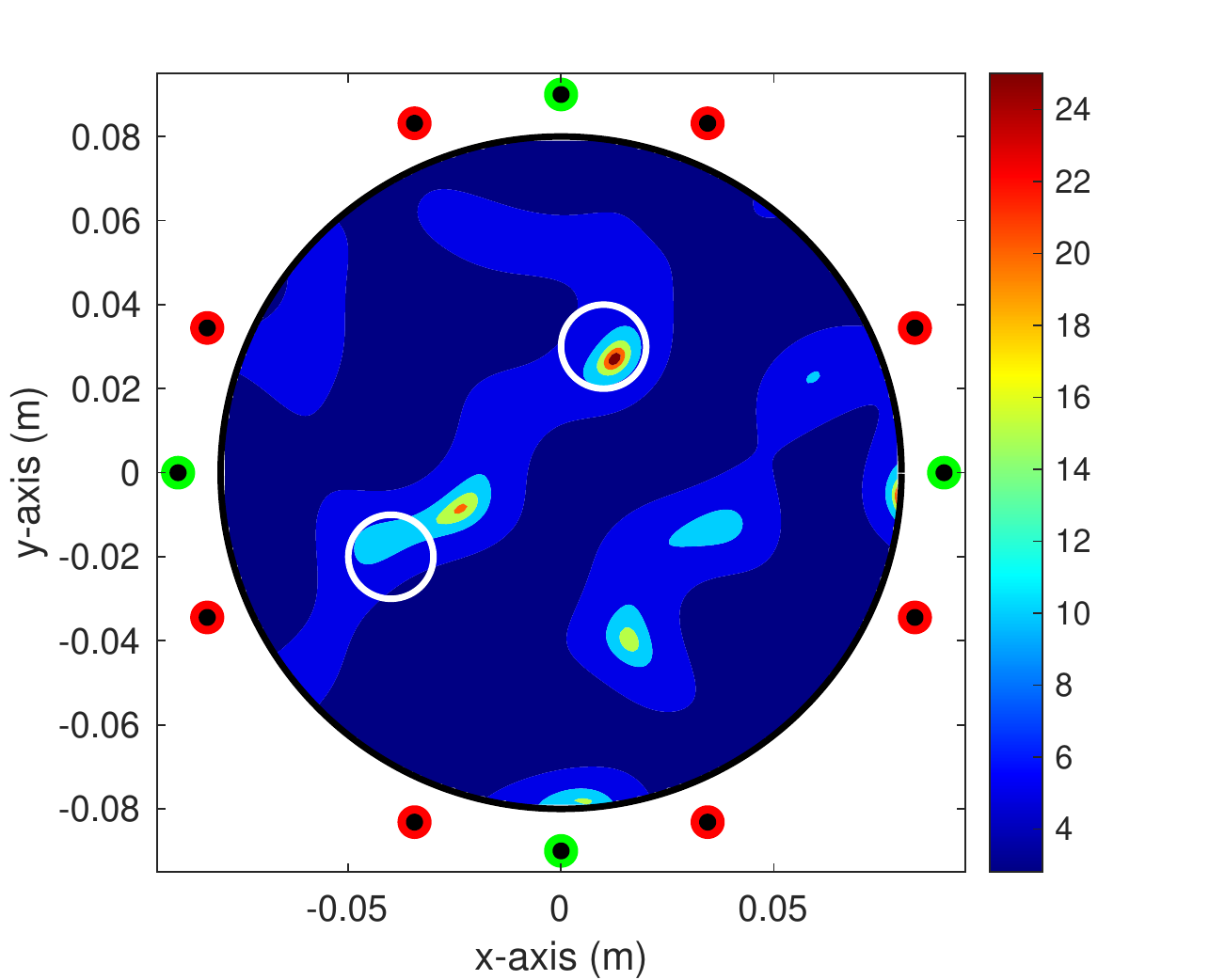}\hfill
  \includegraphics[width=0.25\textwidth]{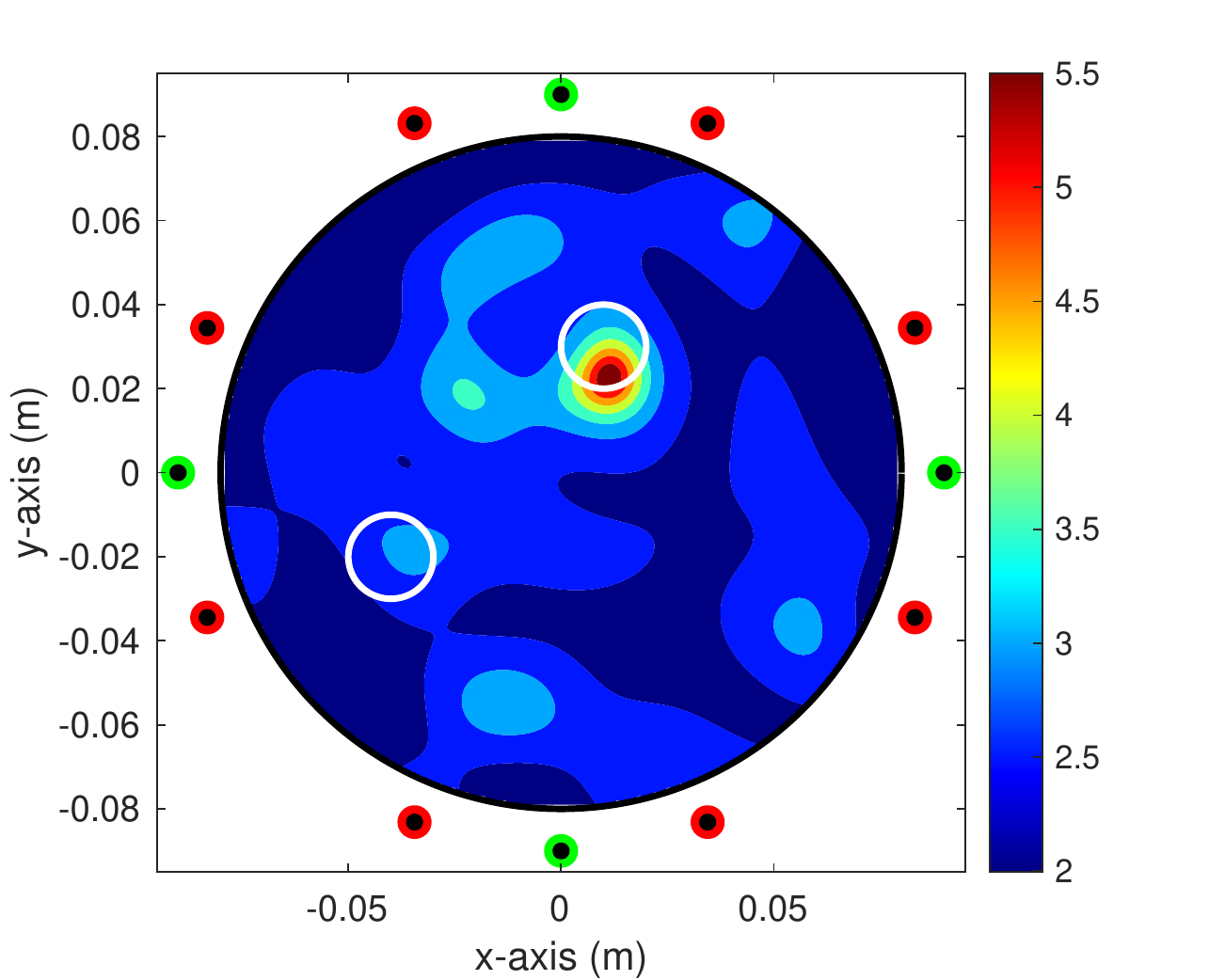}\hfill
  \includegraphics[width=0.25\textwidth]{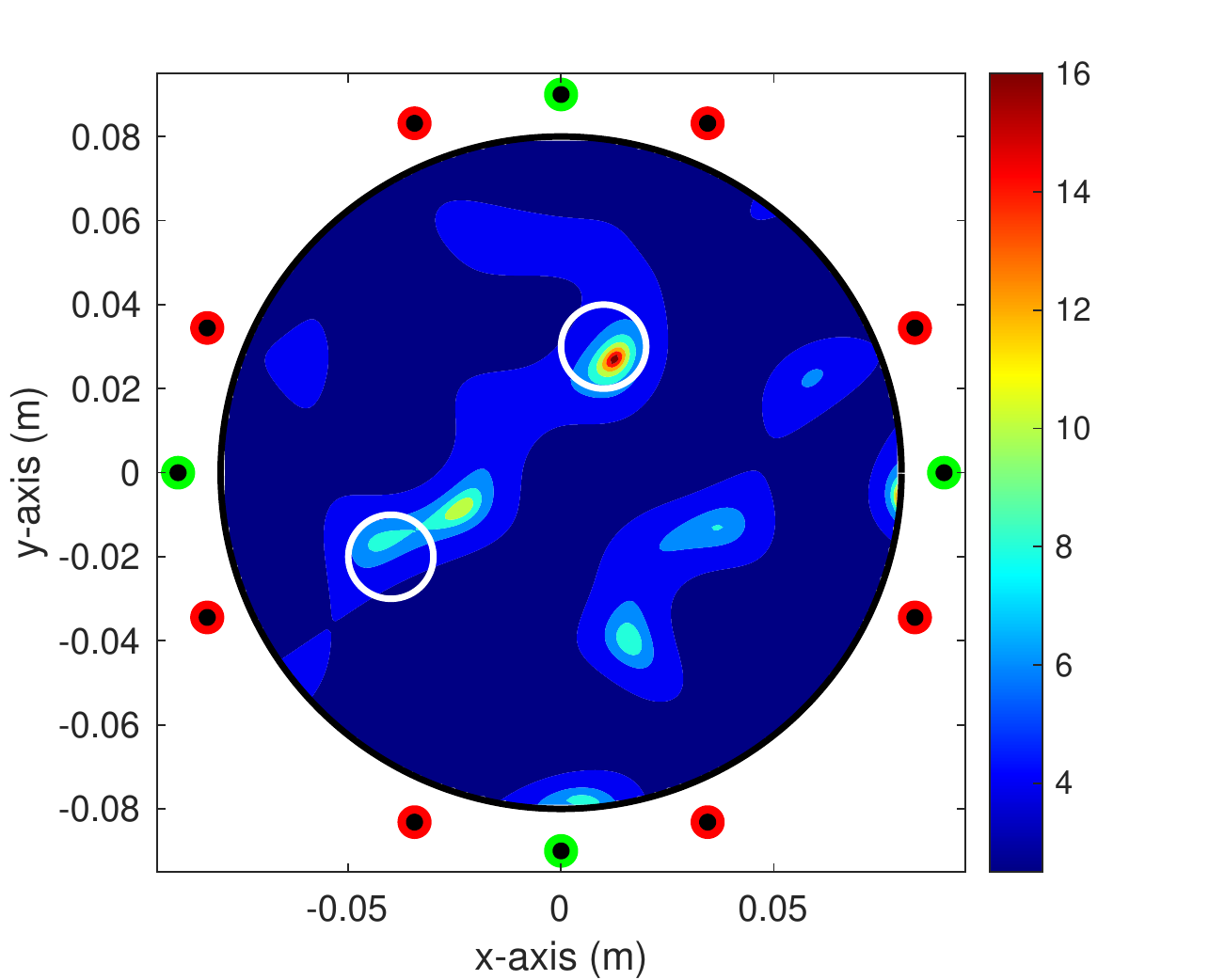}\hfill
  \includegraphics[width=0.25\textwidth]{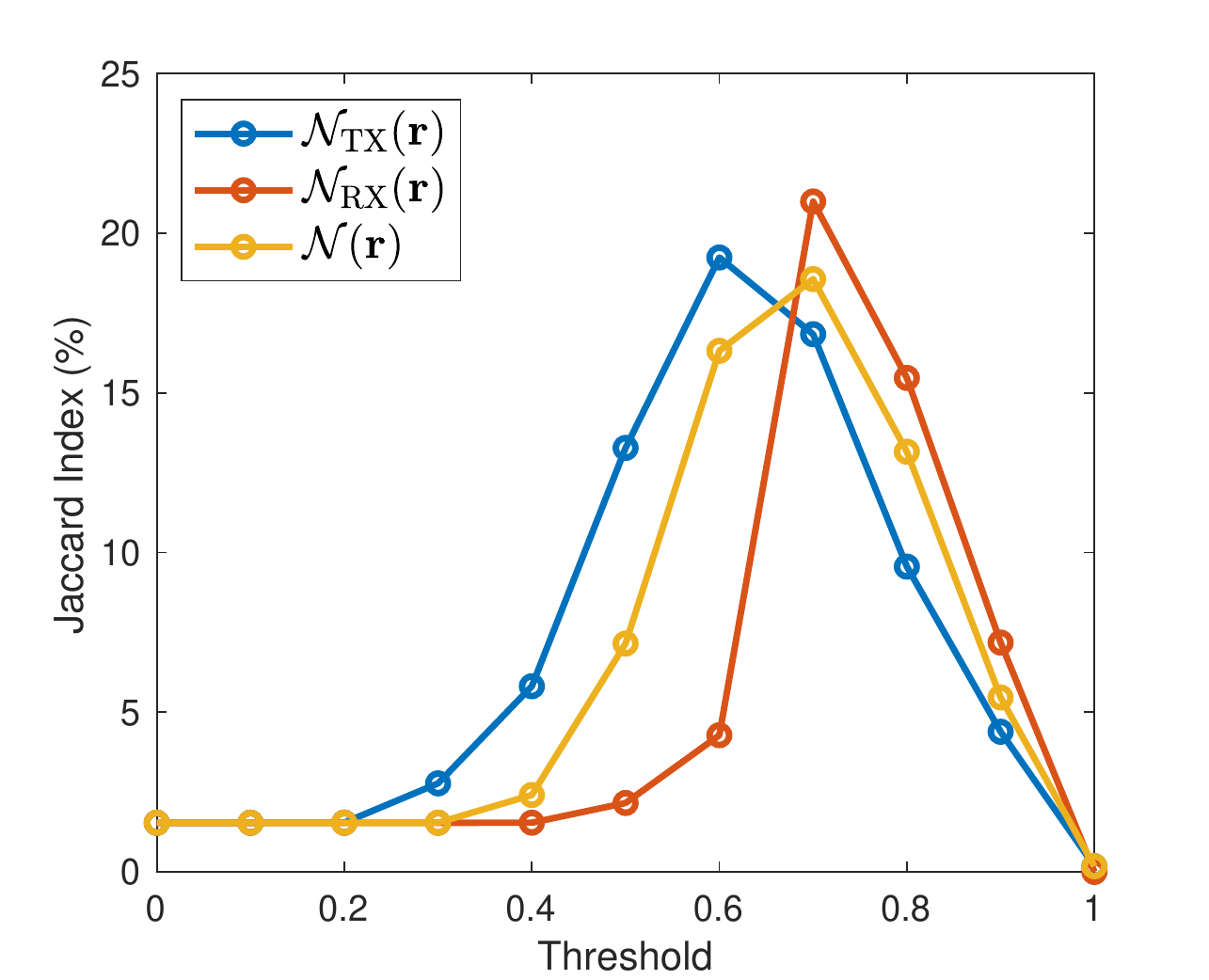}\\
  \includegraphics[width=0.25\textwidth]{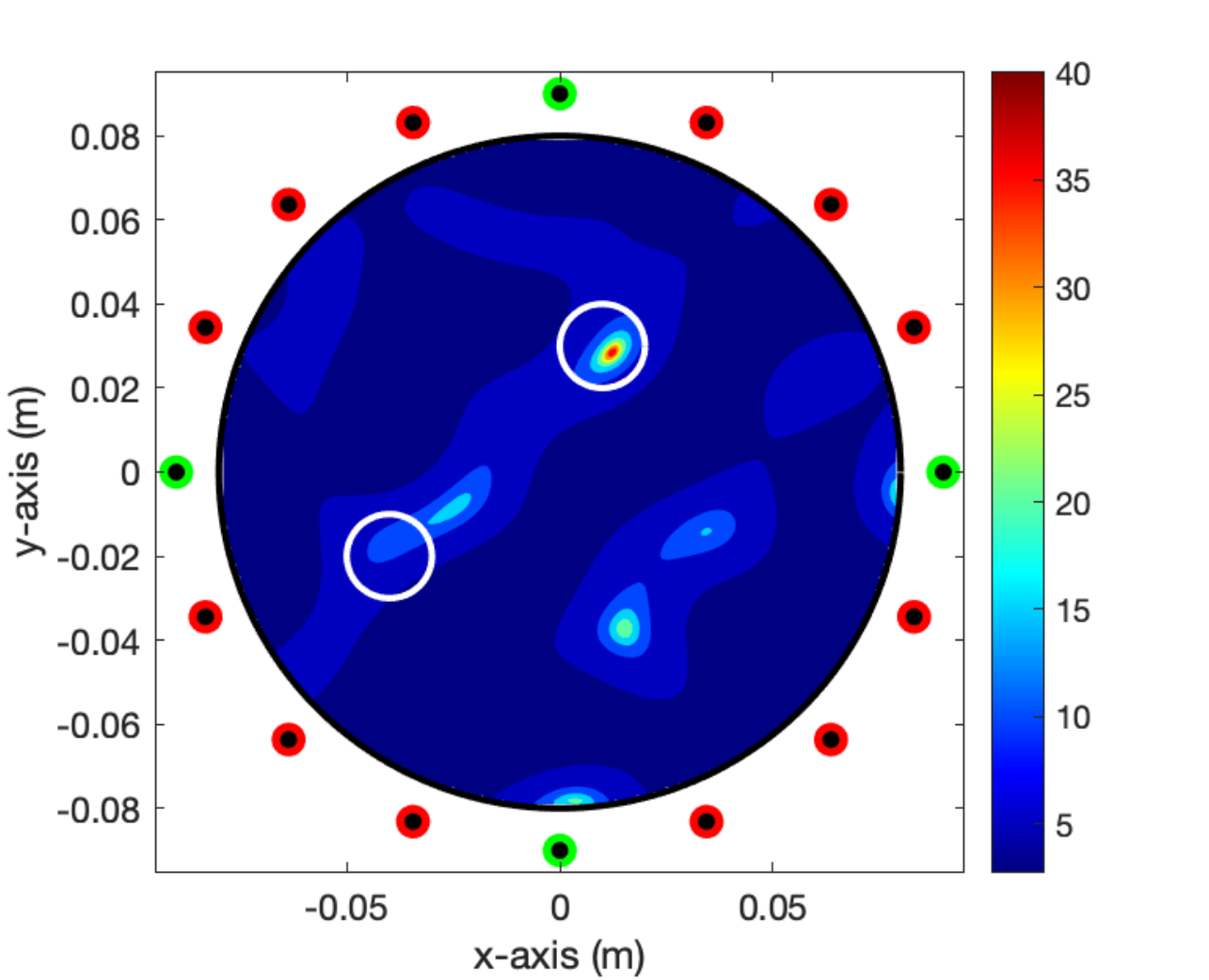}\hfill
  \includegraphics[width=0.25\textwidth]{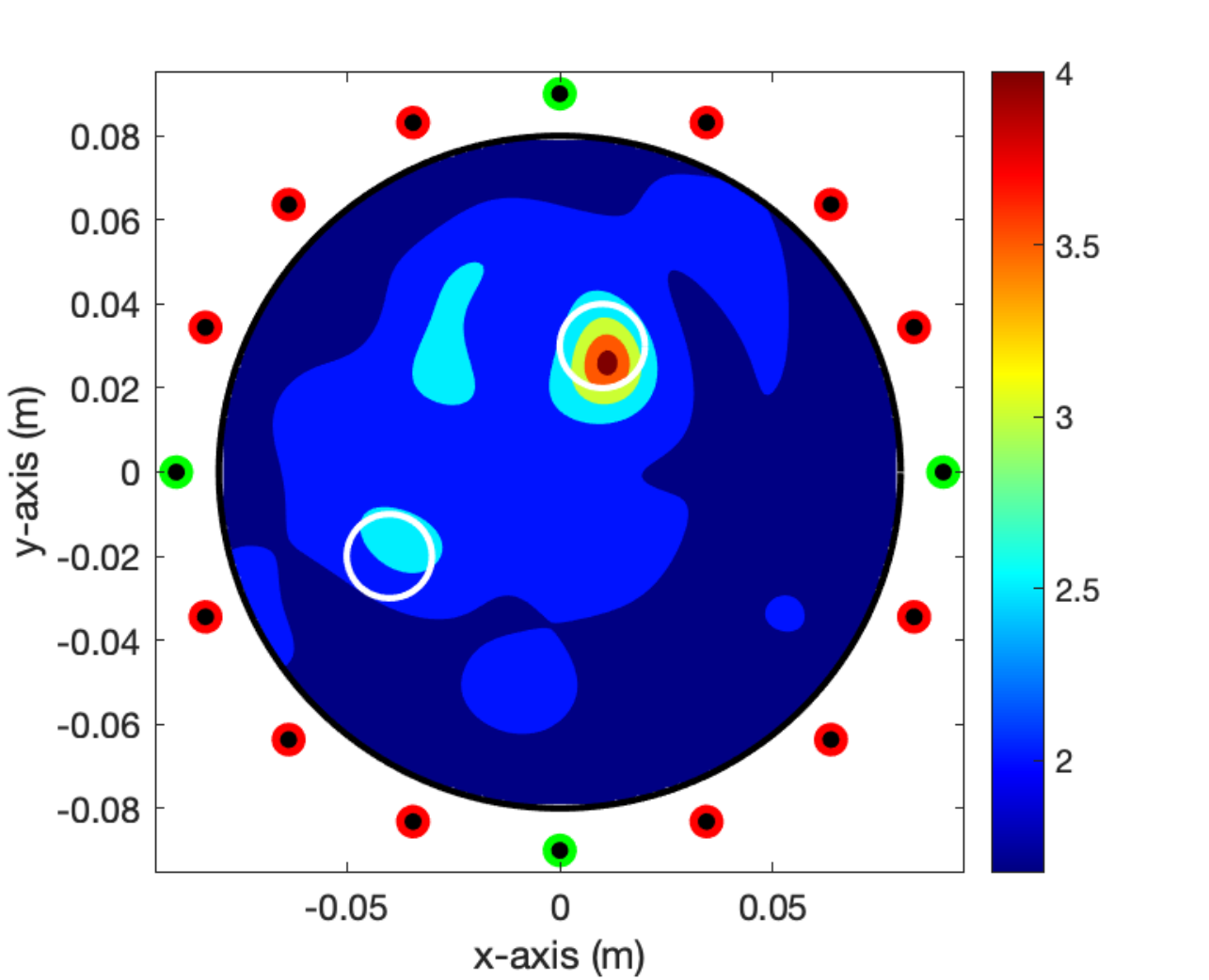}\hfill
  \includegraphics[width=0.25\textwidth]{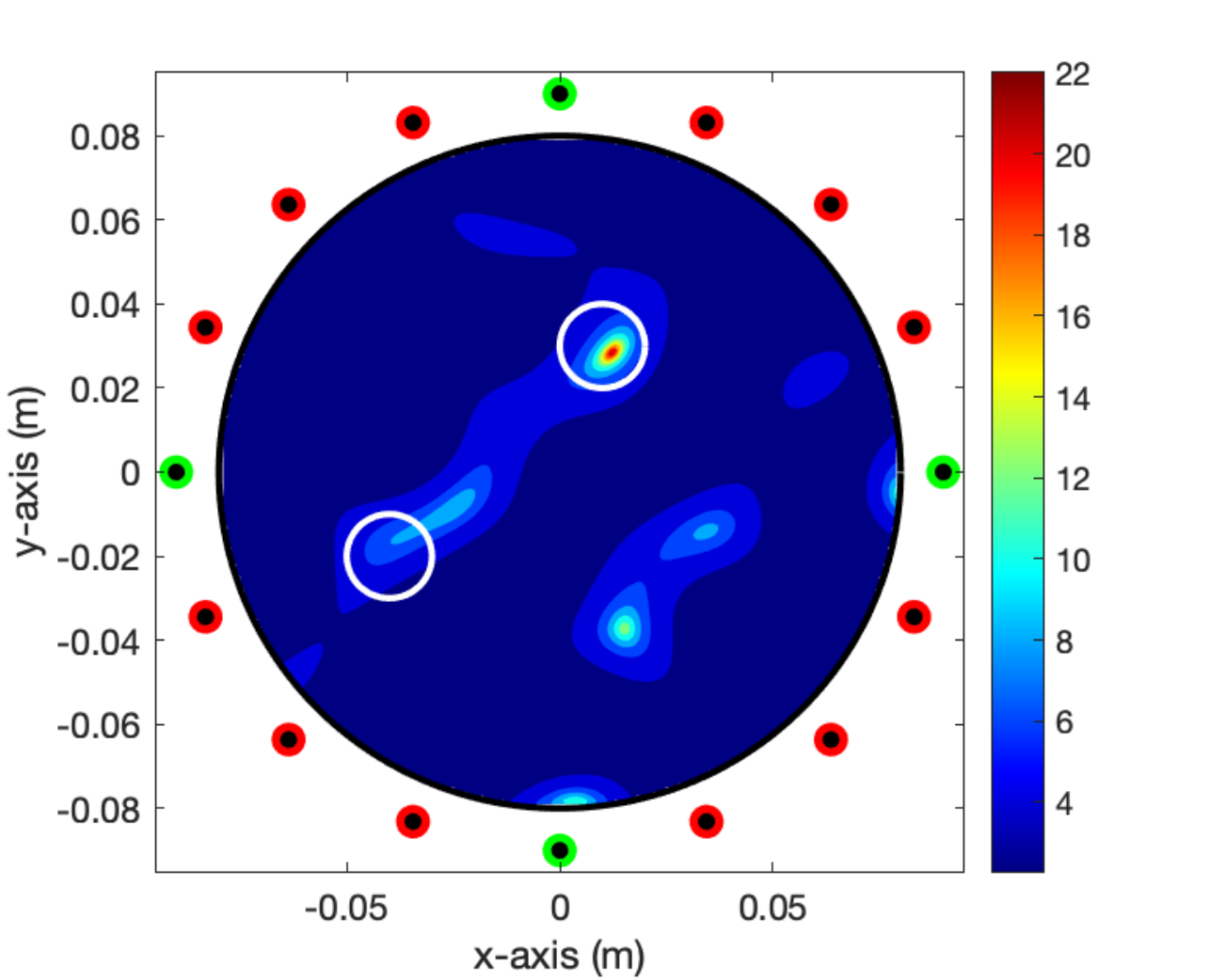}\hfill
  \includegraphics[width=0.25\textwidth]{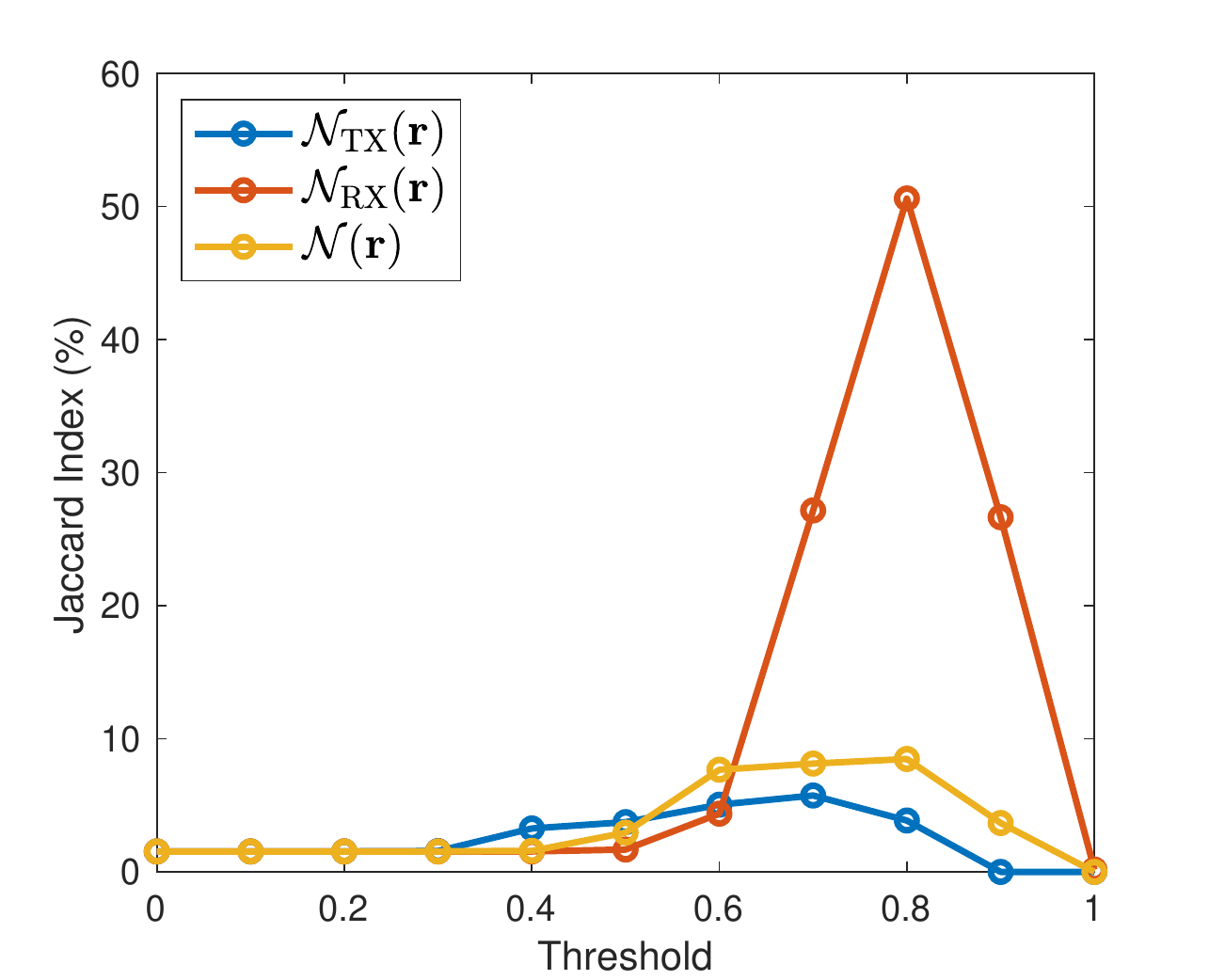}\\
  \includegraphics[width=0.25\textwidth]{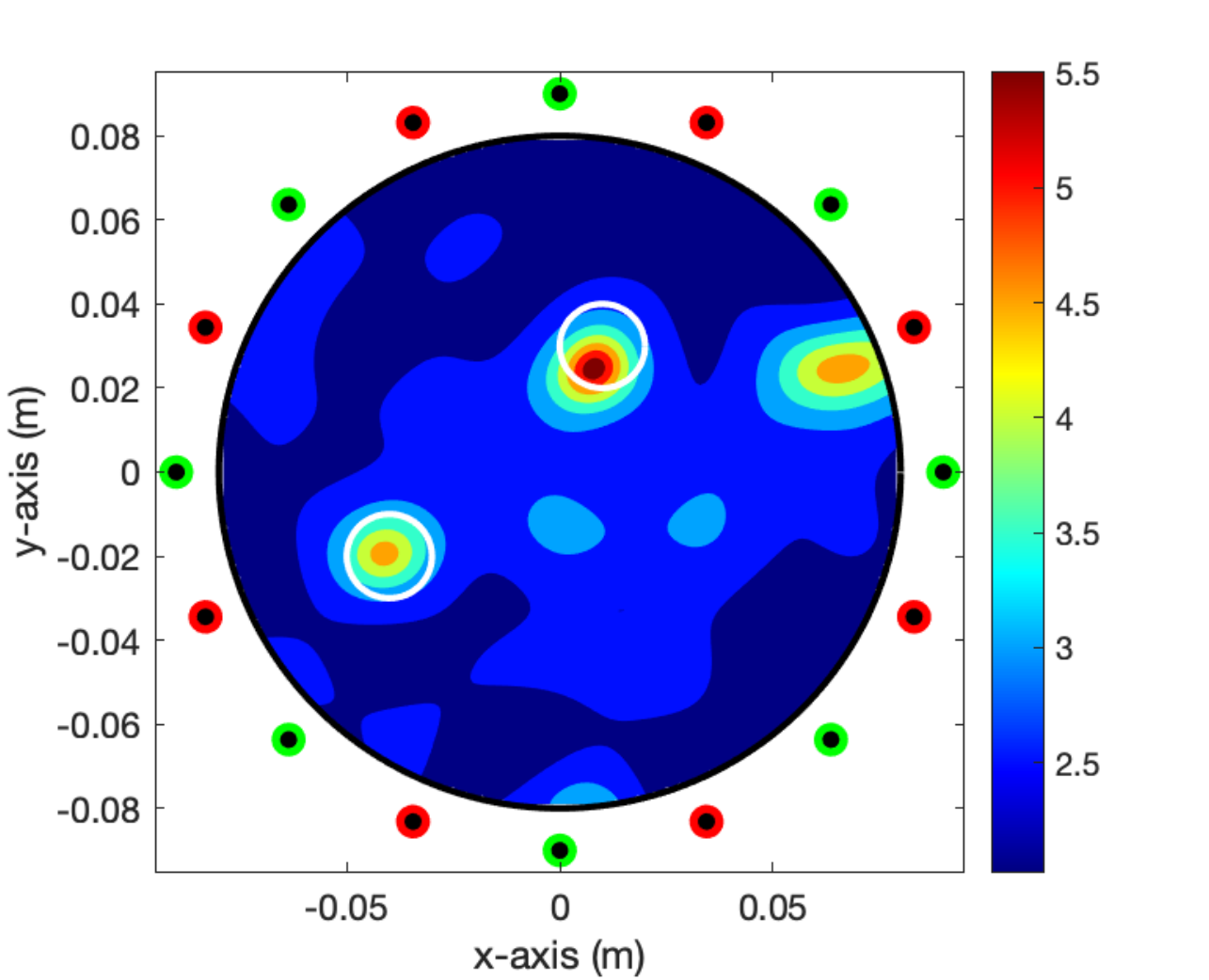}\hfill
  \includegraphics[width=0.25\textwidth]{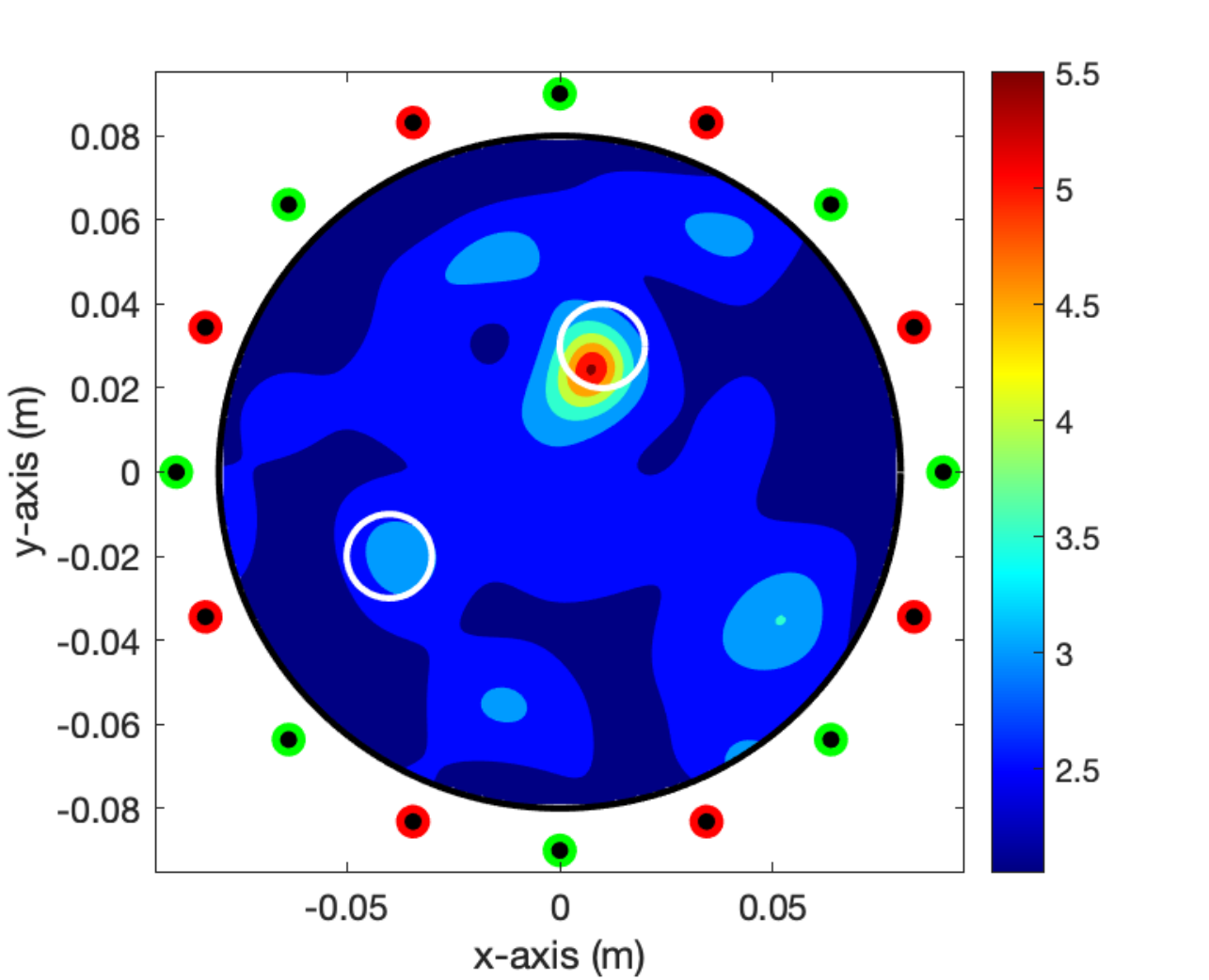}\hfill
  \includegraphics[width=0.25\textwidth]{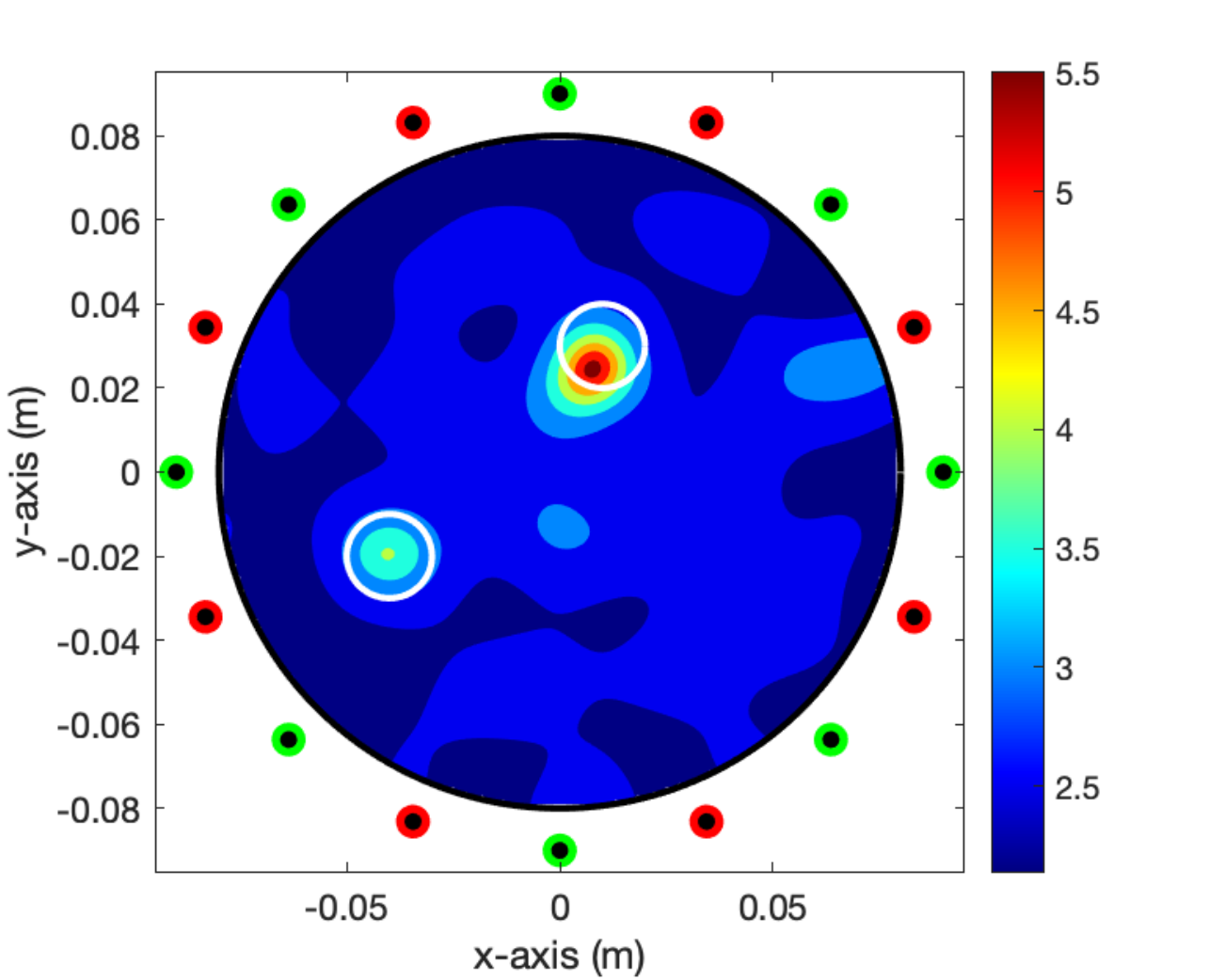}\hfill
  \includegraphics[width=0.25\textwidth]{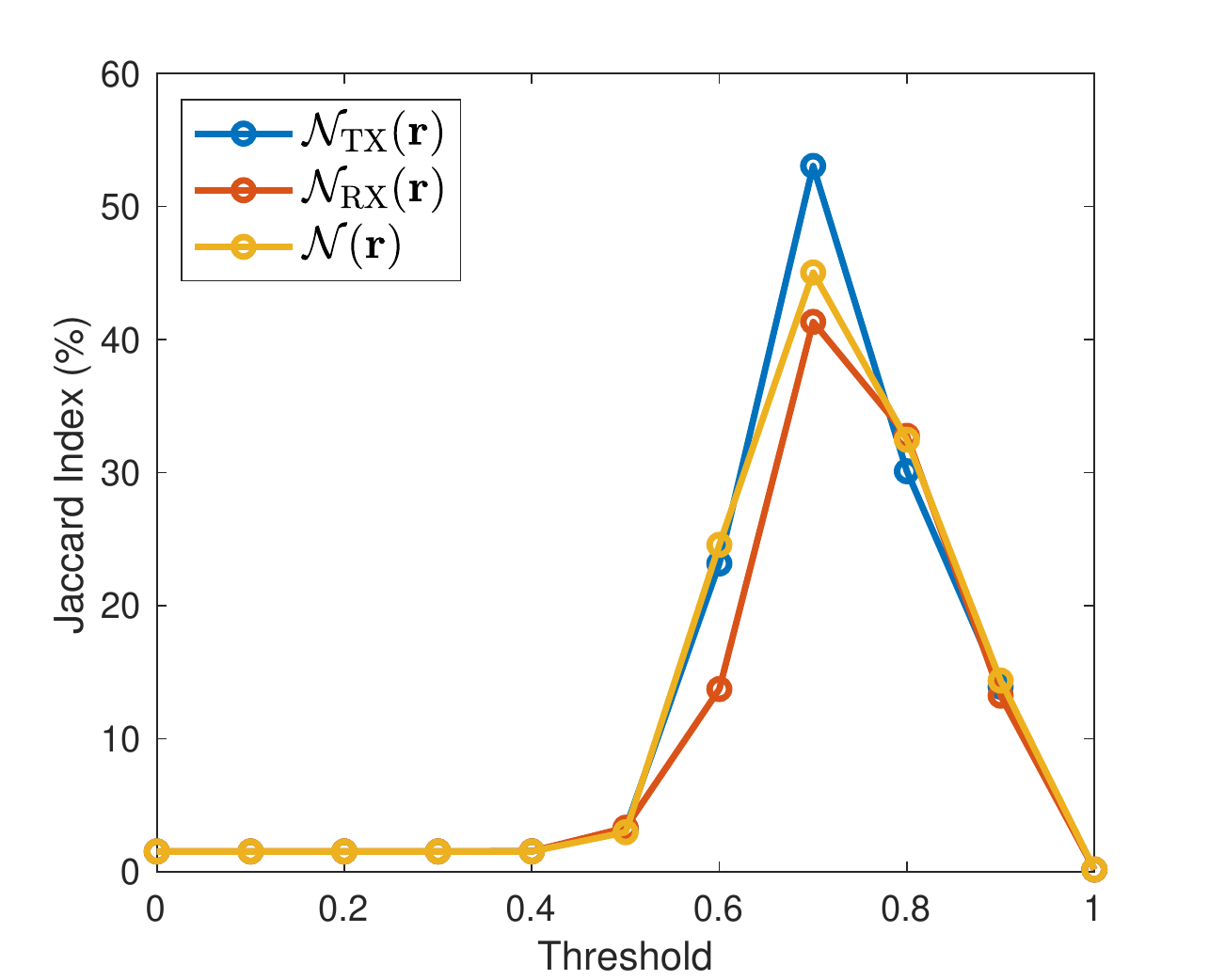}
  \caption{\label{Result6}(Example \ref{ex5}) Maps of $\mathfrak{F}_{\tx}(\mr)$ (first column), $\mathfrak{F}_{\rx}(\mr)$ (second column), $\mathfrak{F}(\mr)$ (third column), and Jaccard index (fourth column). Green and red colored circles describe the location of transmitters and receivers, respectively.}
\end{figure}

\begin{example}[Extremely Small Number of Transmitters]\label{ex6}
Based on the simulation results in Examples \ref{ex1}, \ref{ex2}, \ref{ex4}, and \ref{ex5}, it is expect that the anomalies can be identified successfully by increasing total number of receivers. To this end, let us consider the imaging results with single (from Example \ref{ex1}) and multiple (from Example \ref{ex4}) anomalies. Based on the results in Figure \ref{Result7} with single (say, $\mathcal{D}_1$) and two transmitters (say, $\mathcal{D}_1$ and $\mathcal{D}_9$), as we expected, the location of single anomaly was successfully identified through the map of $\mathfrak{F}_{\rx}(\mr)$ while it is impossible to recognize the anomaly through the maps of $\mathfrak{F}_{\tx}(\mr)$ (first column), $\mathfrak{F}_{\rx}(\mr)$ (second column), $\mathfrak{F}(\mr)$ and $\mathfrak{F}(\mr)$ due to the small number of transmitting antennas.

Unfortunately, based on the Figure \ref{Result8} with two ($\mathcal{D}_1$ and $\mathcal{D}_9$) and three ($\mathcal{D}_1$, $\mathcal{D}_5$, and $\mathcal{D}_9$) transmitters, it is very hard to identify two anomalies due to the appearance of artifact with large magnitude in the neighborhood of $\Sigma_1$ through the map of $\mathfrak{F}_{\rx}(\mr)$. Hence, we can conclude that if the total number of transmitting or receiving antennas is small, it will be very hard to identify the location of anomalies through the MUSIC algorithm.
\end{example}

\begin{figure}[h]
  \centering
  \includegraphics[width=0.25\textwidth]{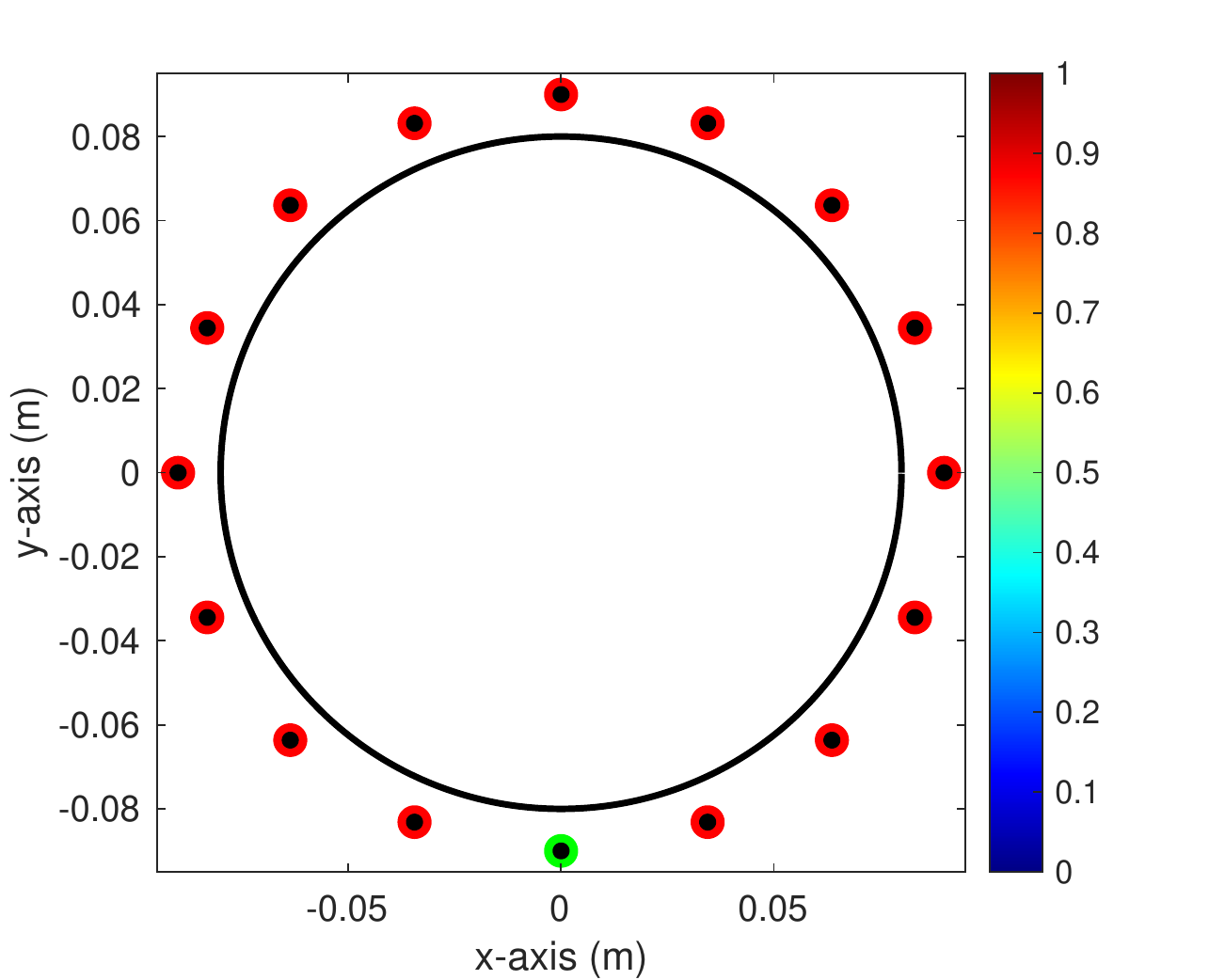}\hfill
  \includegraphics[width=0.25\textwidth]{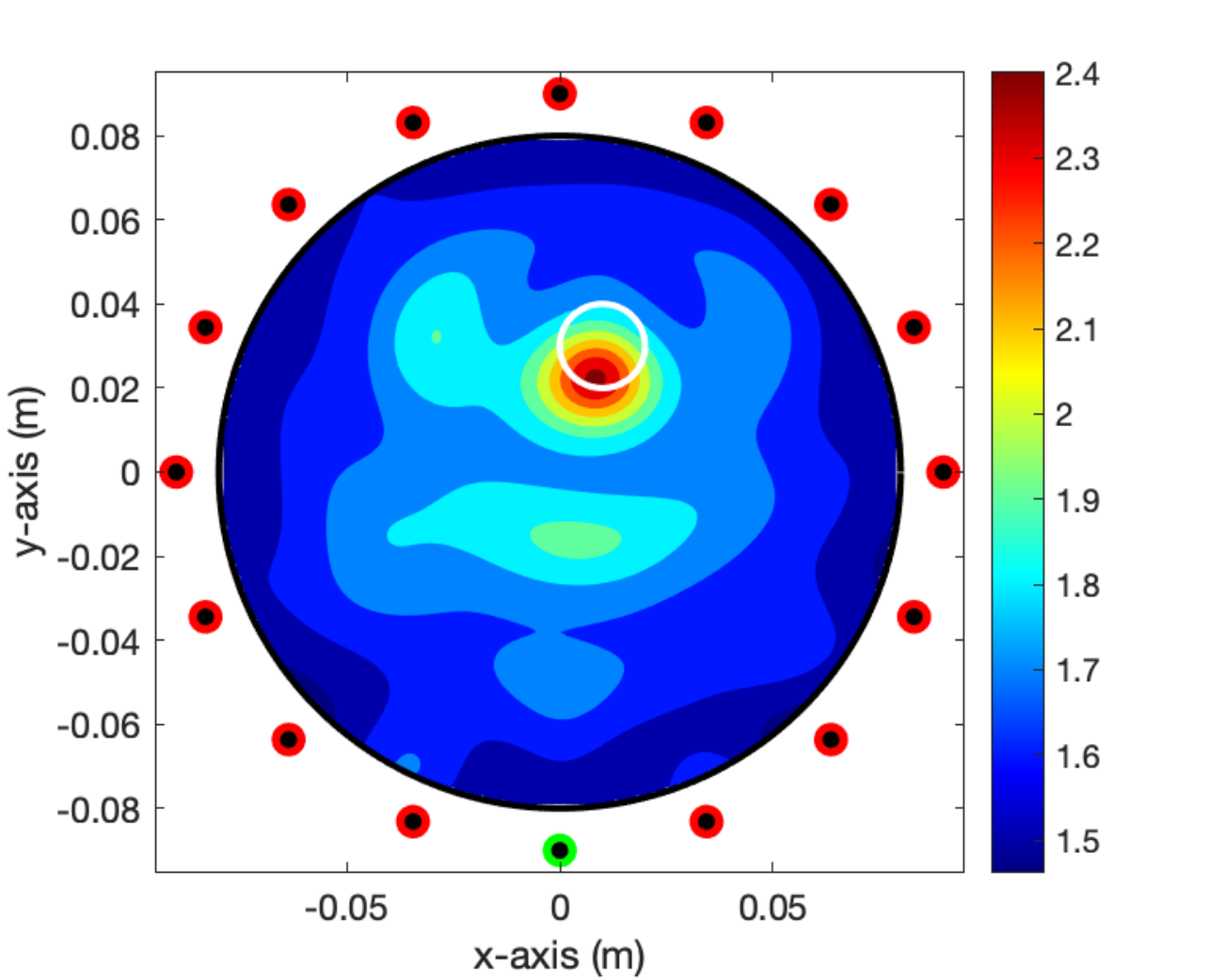}\hfill
  \includegraphics[width=0.25\textwidth]{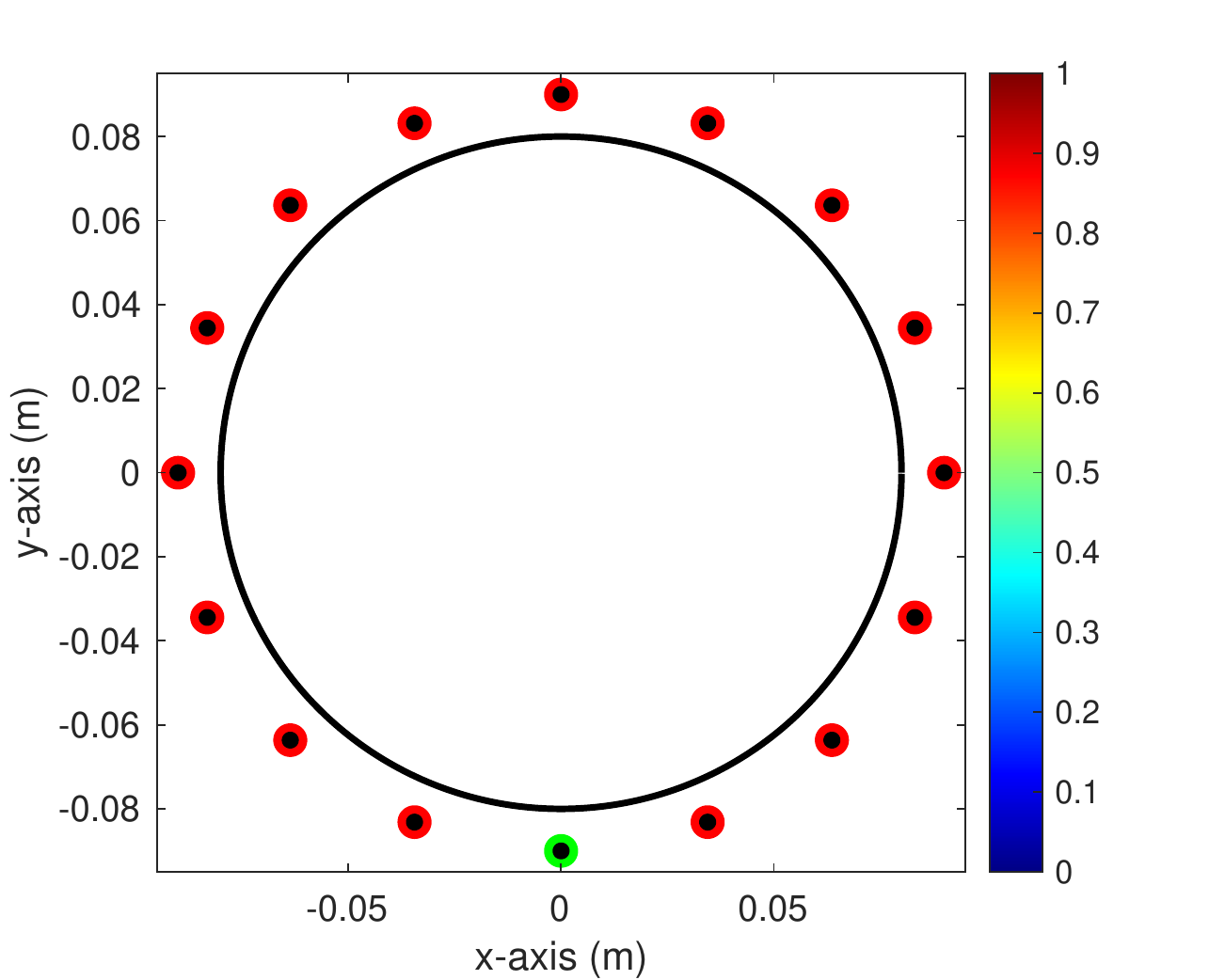}\hfill
  \includegraphics[width=0.25\textwidth]{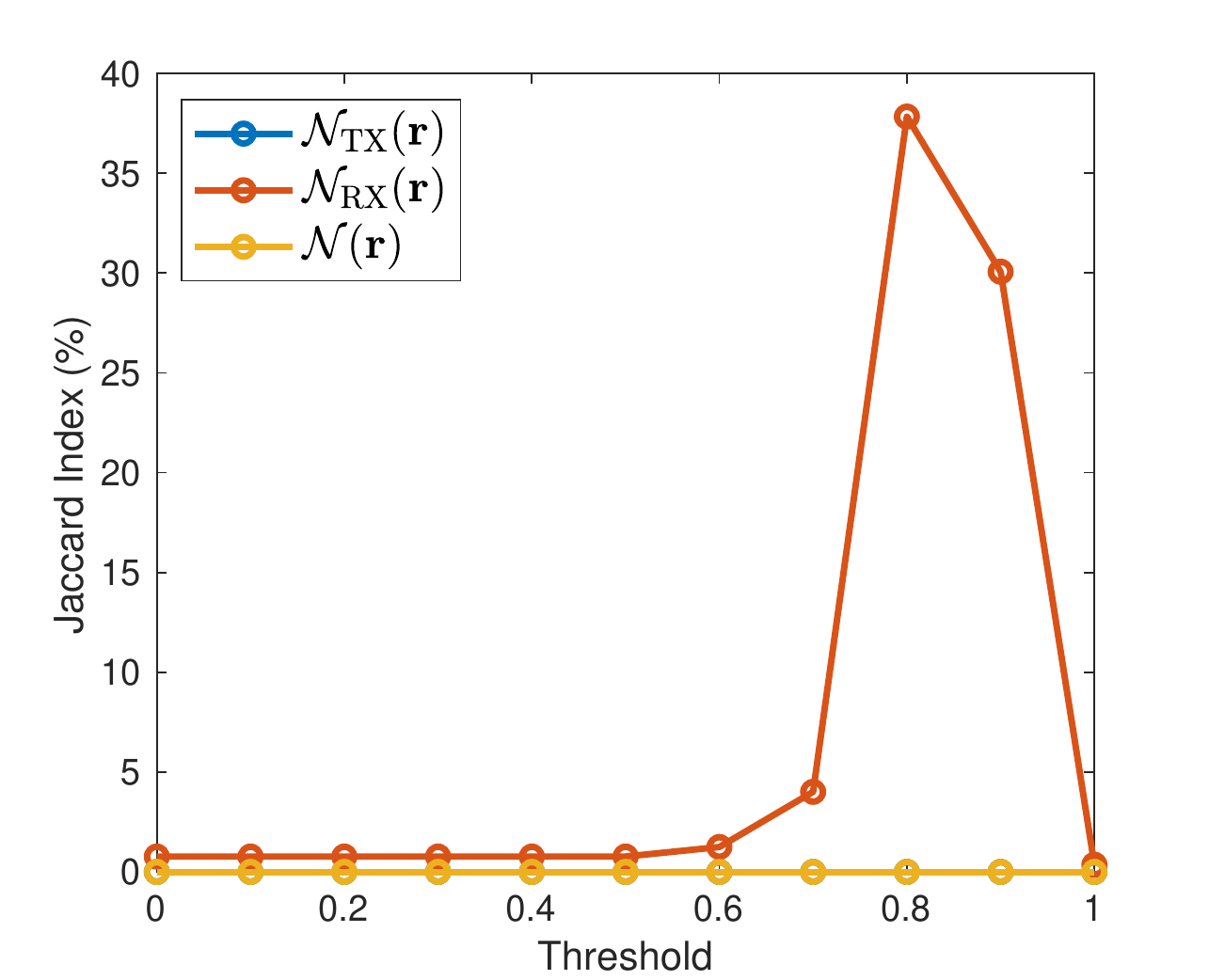}\\
  \includegraphics[width=0.25\textwidth]{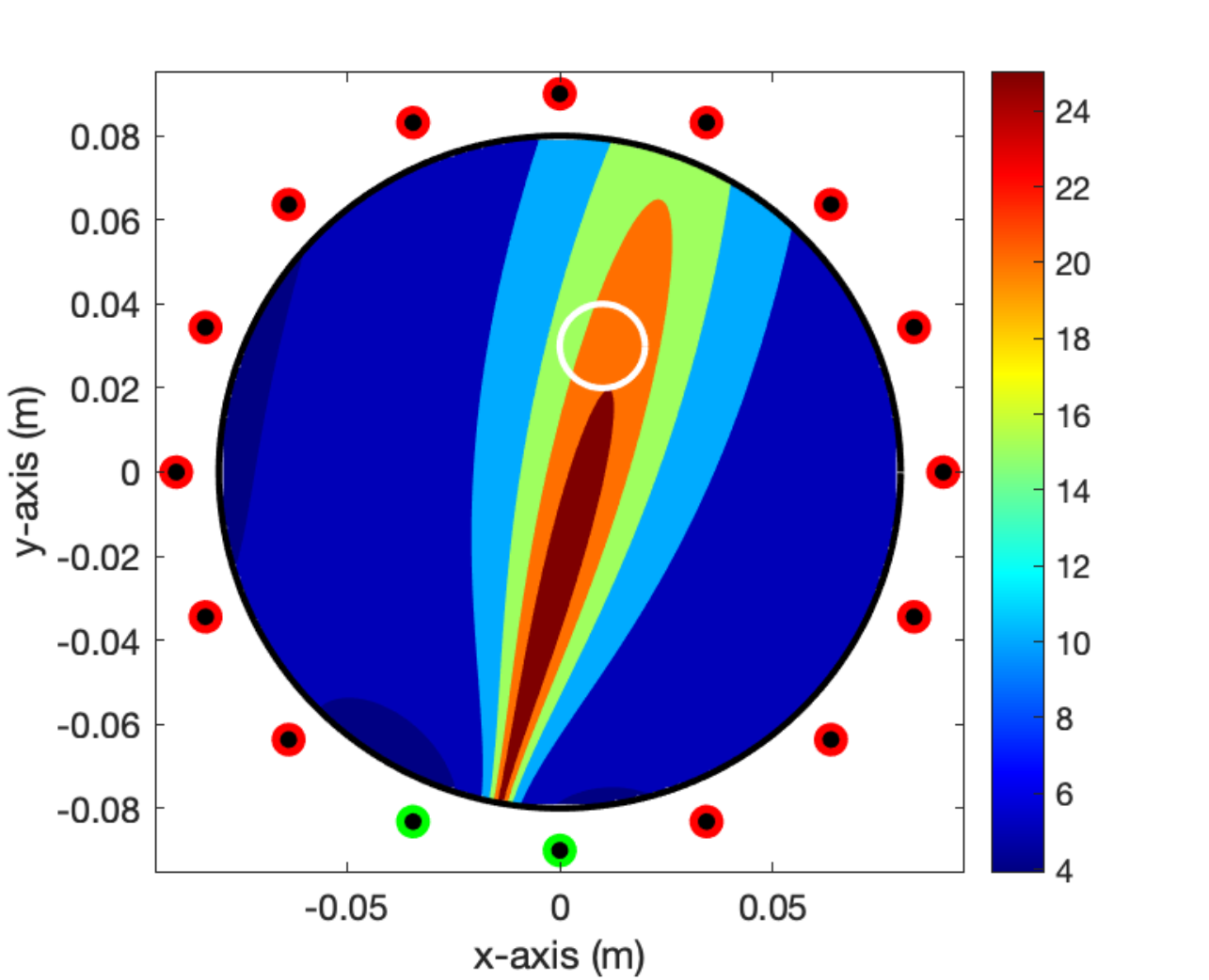}\hfill
  \includegraphics[width=0.25\textwidth]{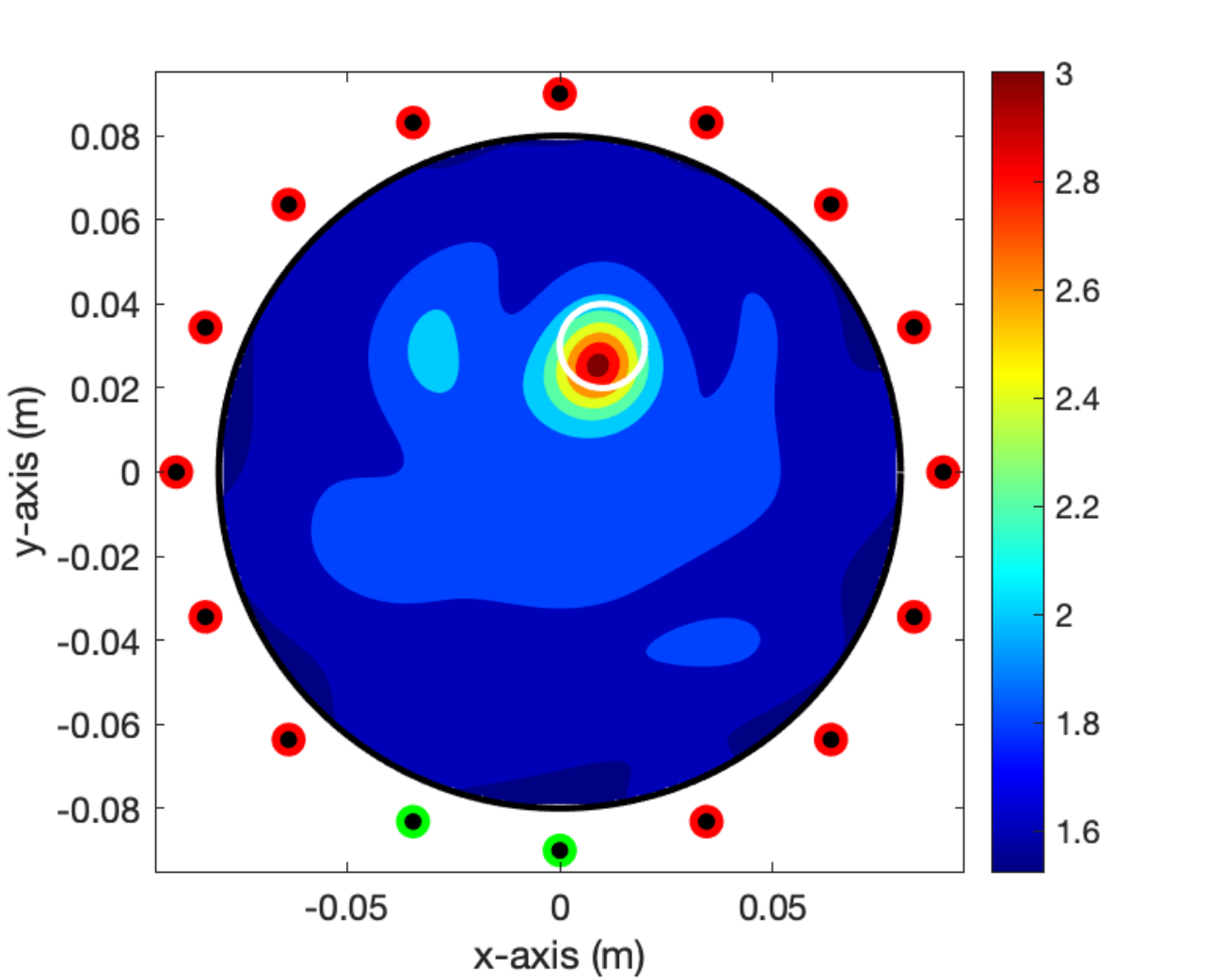}\hfill
  \includegraphics[width=0.25\textwidth]{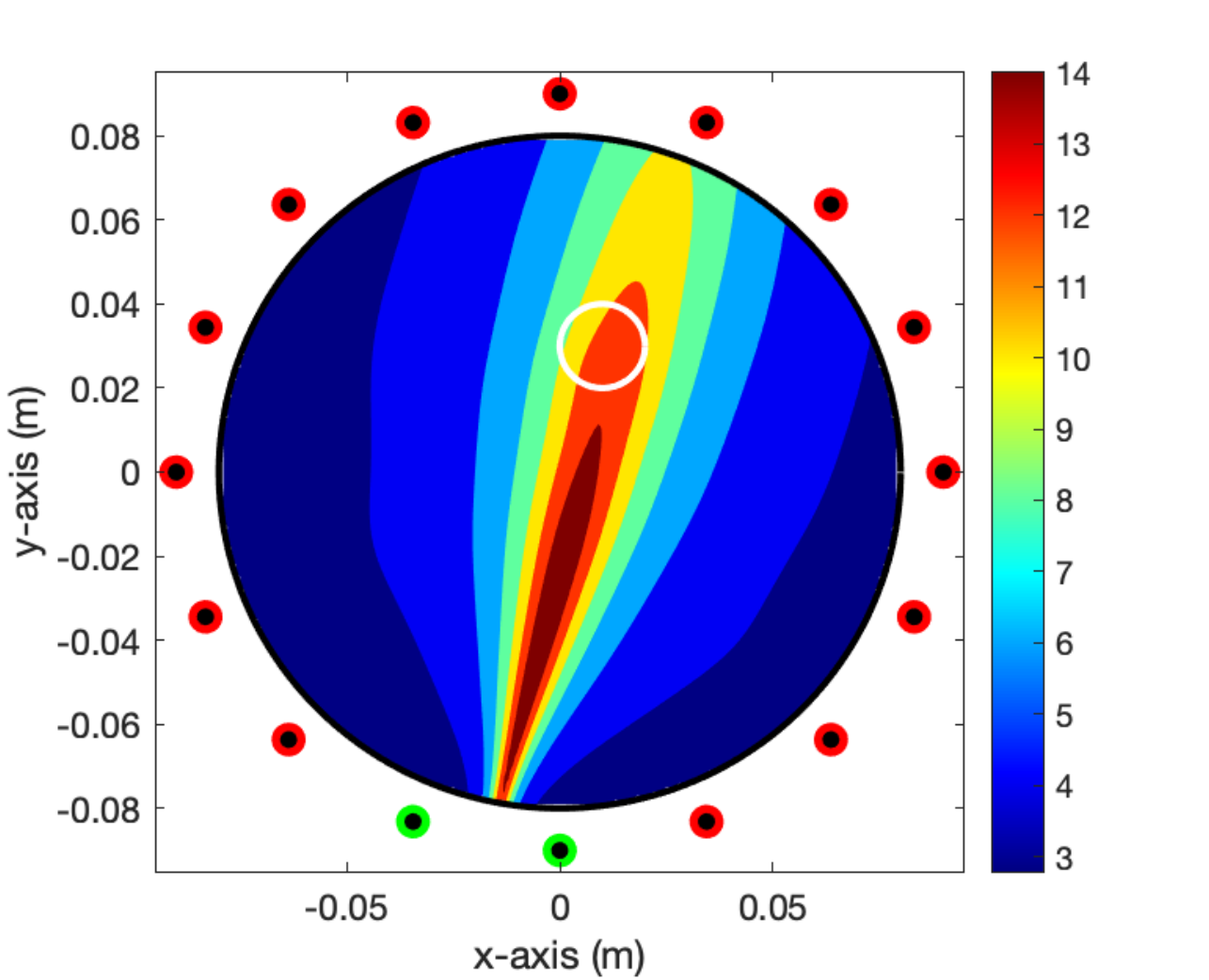}\hfill
  \includegraphics[width=0.25\textwidth]{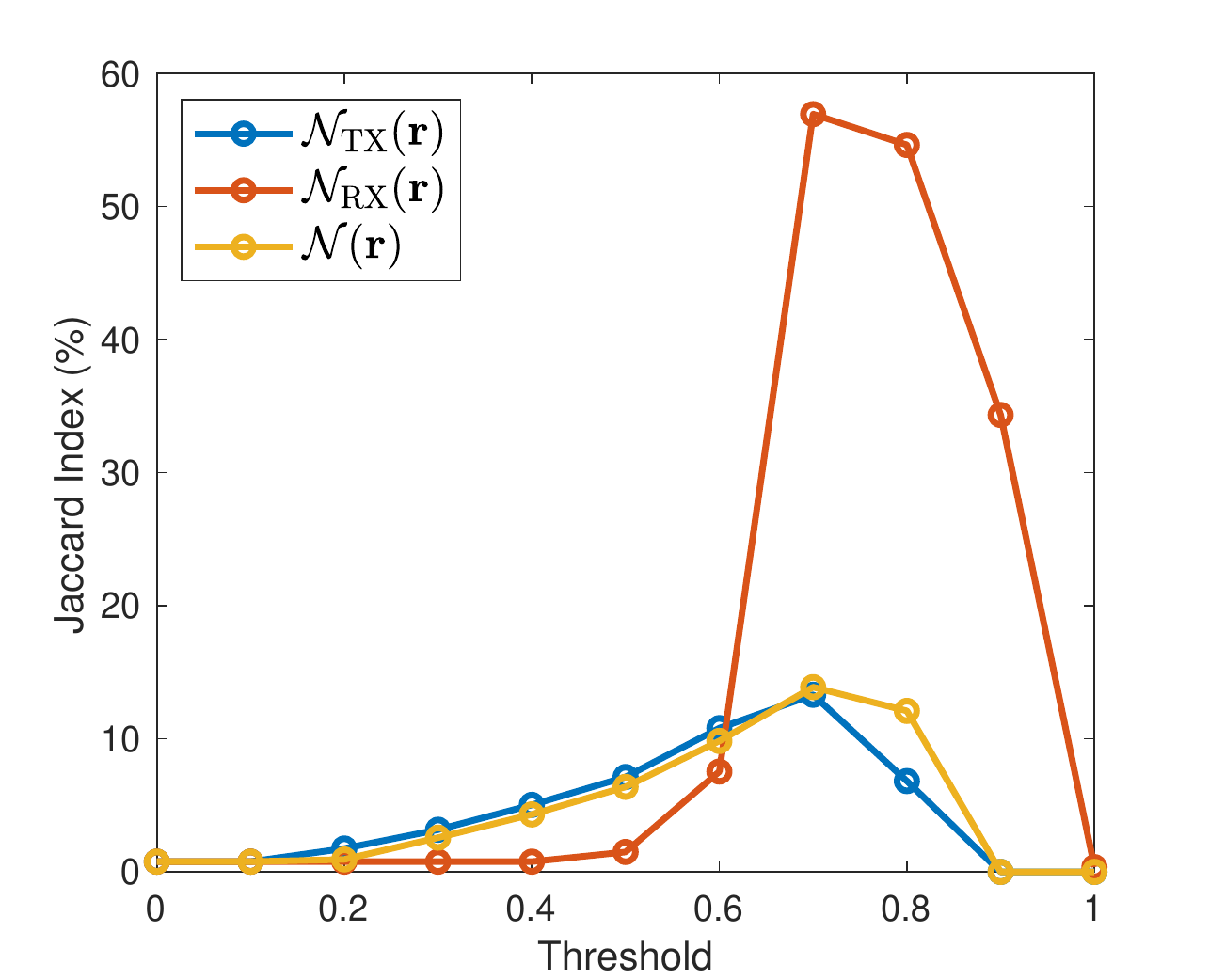}\\
  \includegraphics[width=0.25\textwidth]{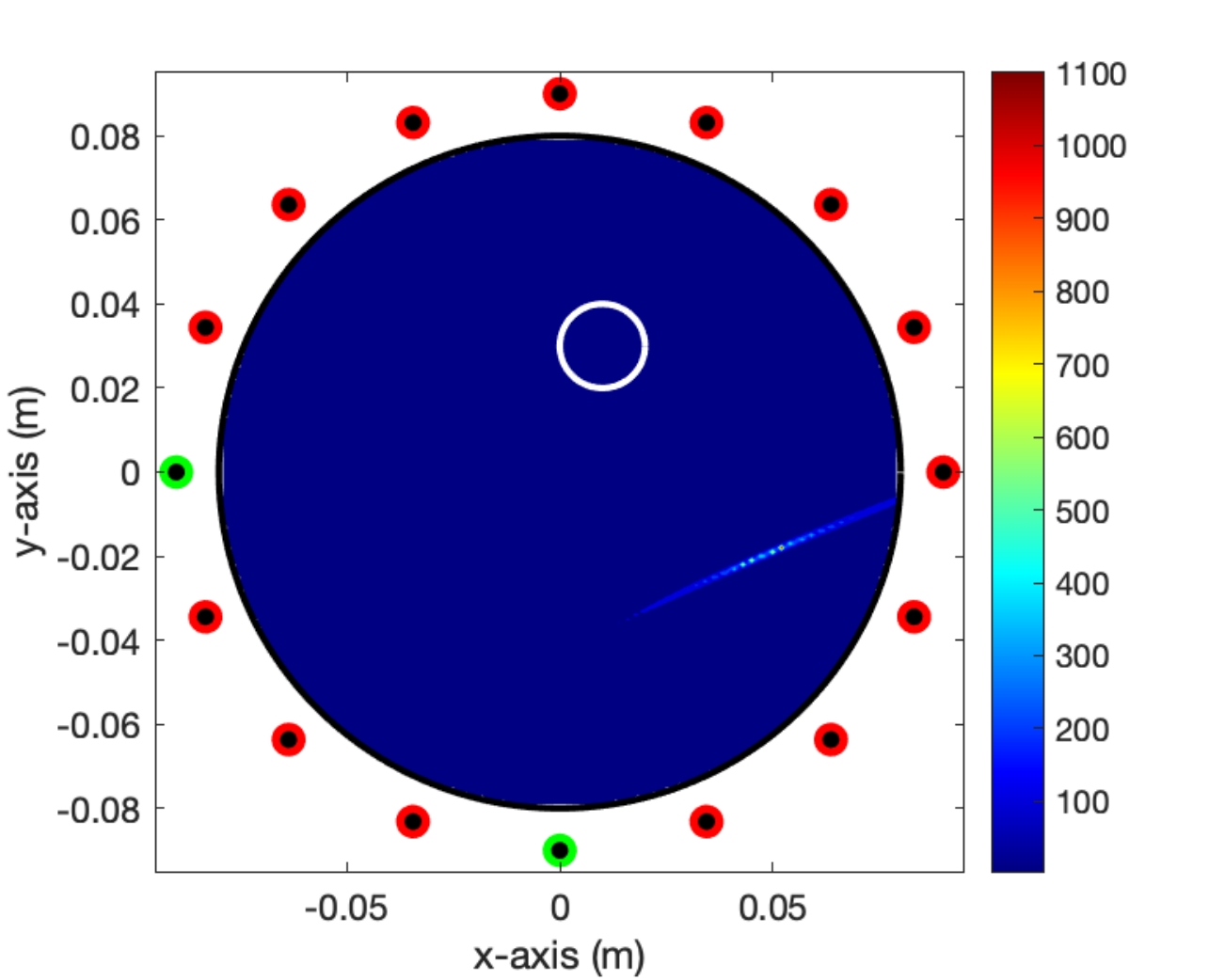}\hfill
  \includegraphics[width=0.25\textwidth]{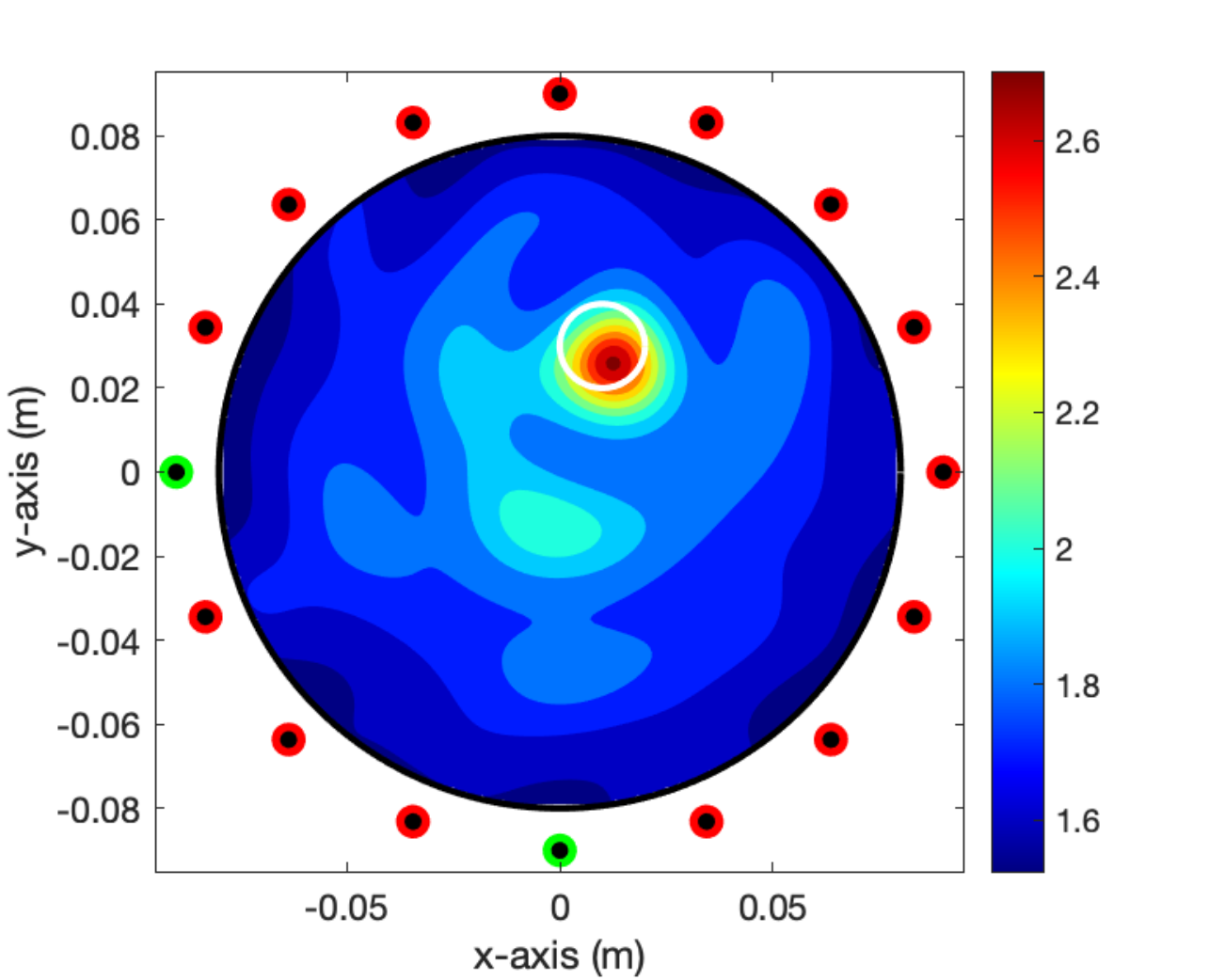}\hfill
  \includegraphics[width=0.25\textwidth]{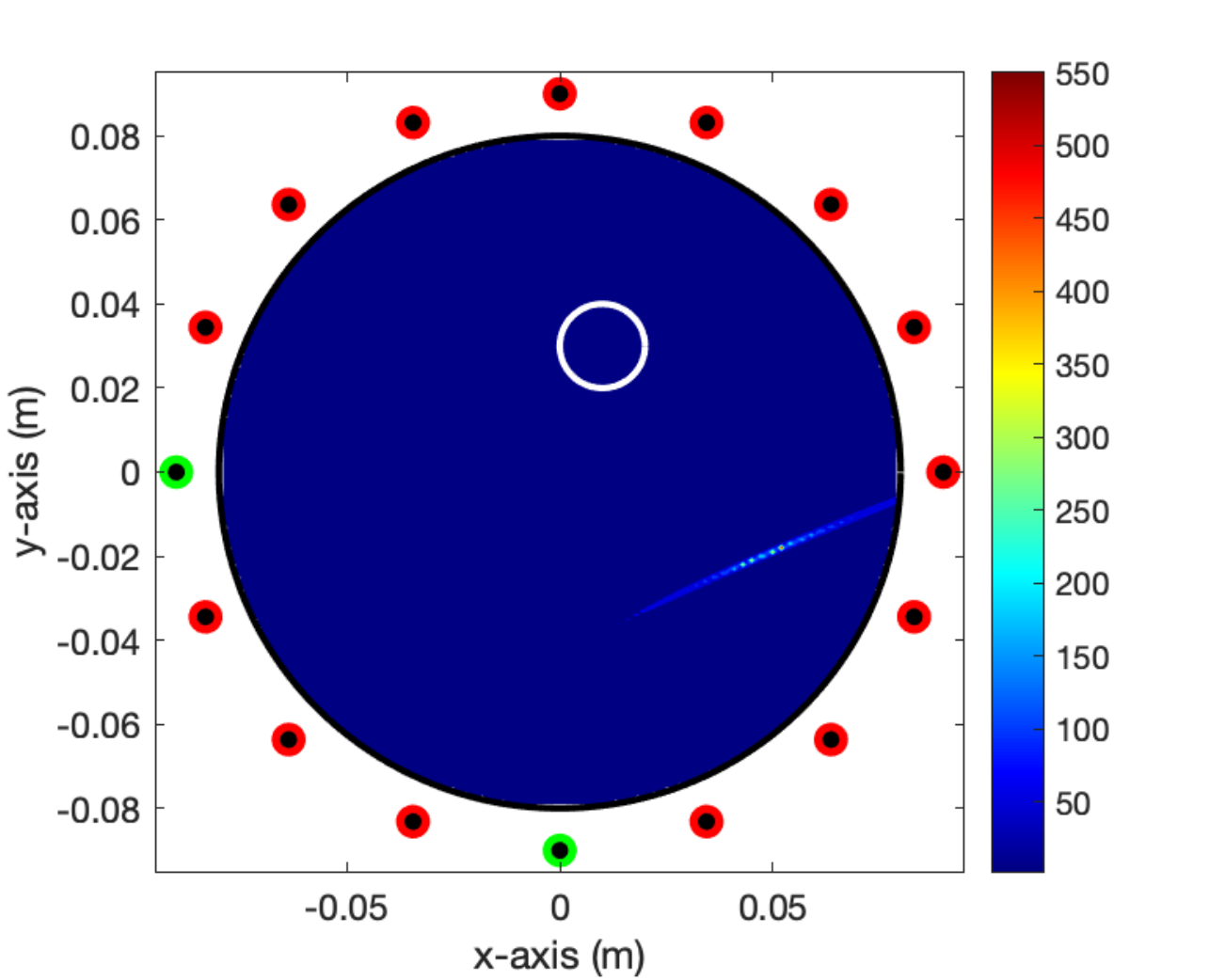}\hfill
  \includegraphics[width=0.25\textwidth]{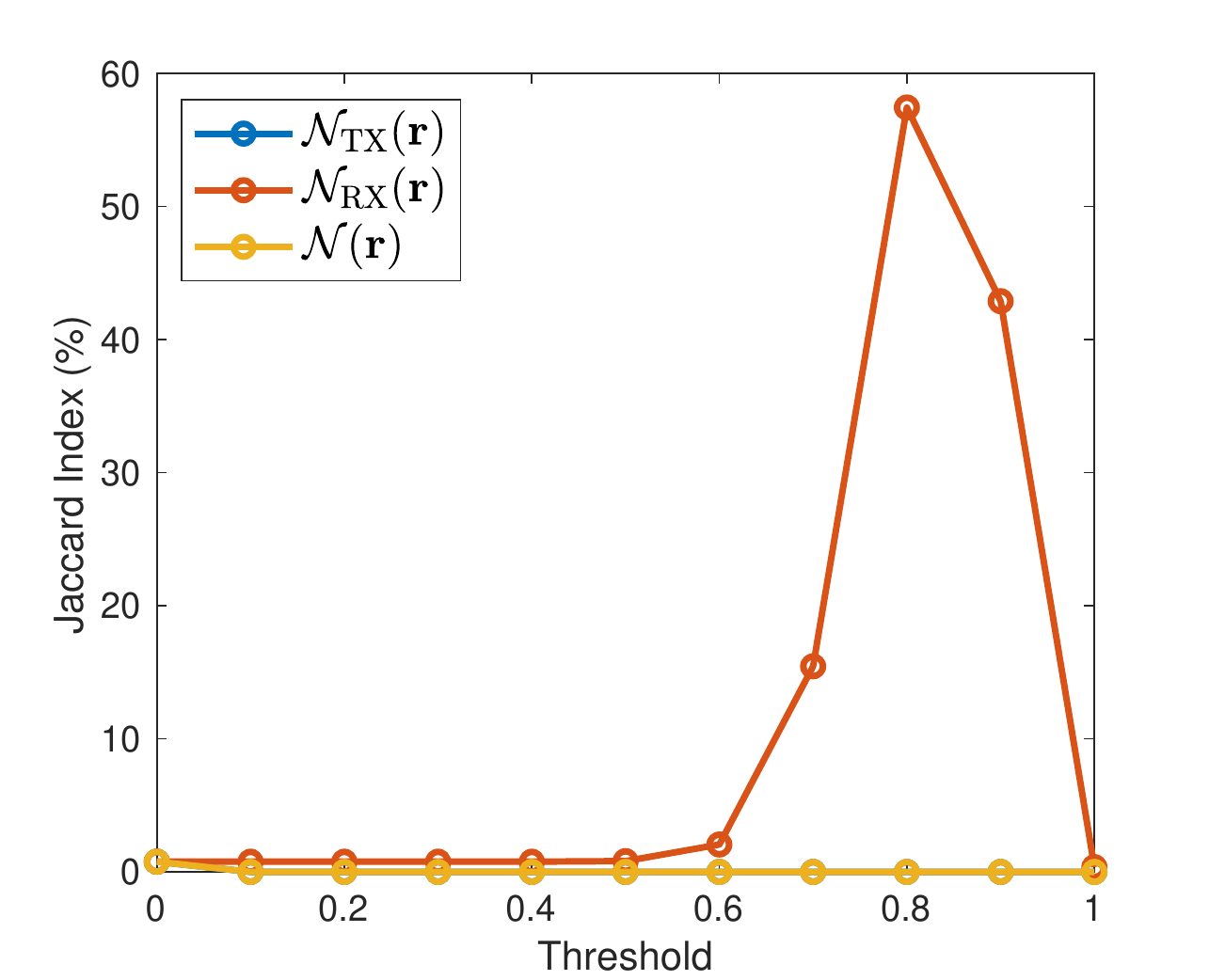}\\
  \includegraphics[width=0.25\textwidth]{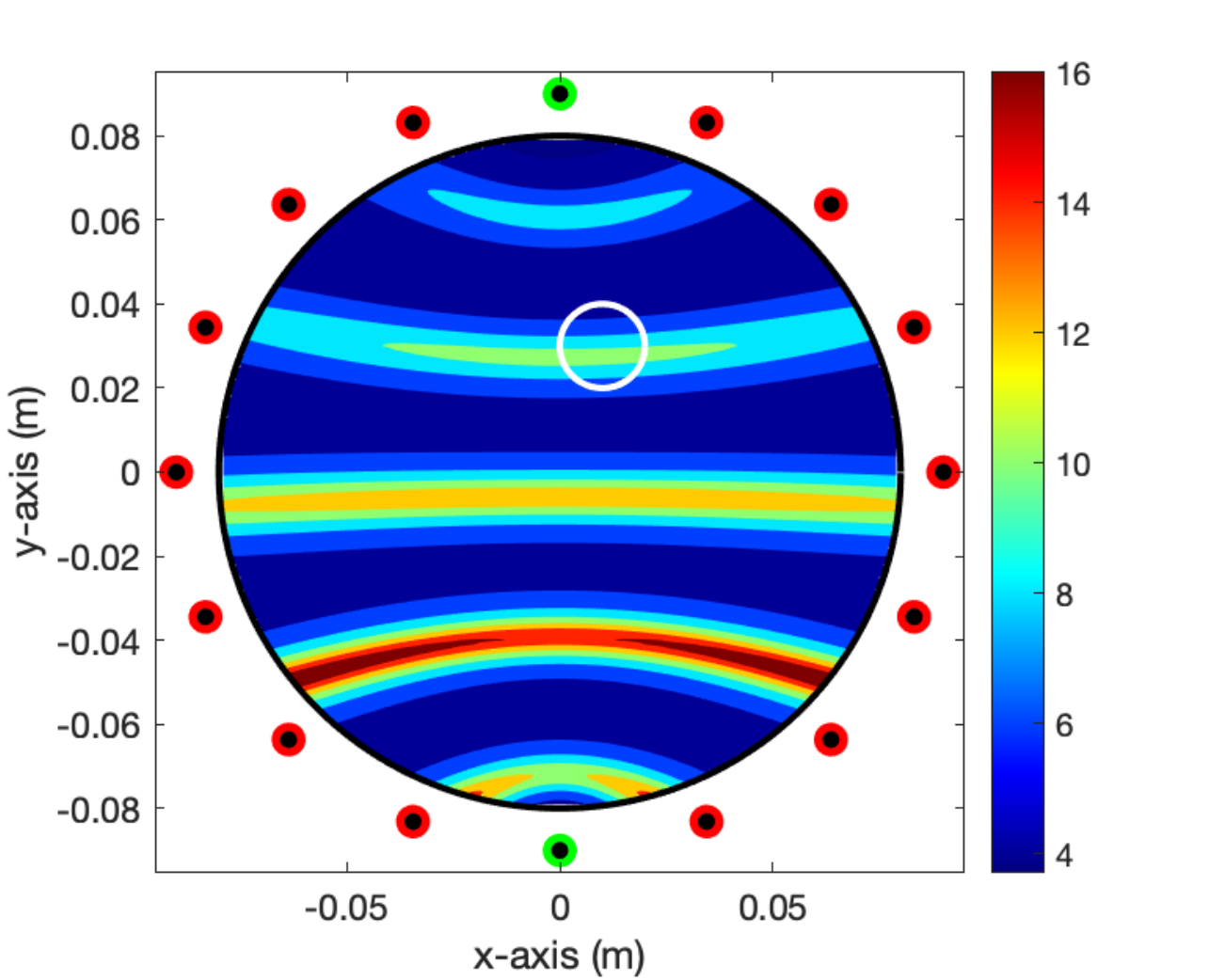}\hfill
  \includegraphics[width=0.25\textwidth]{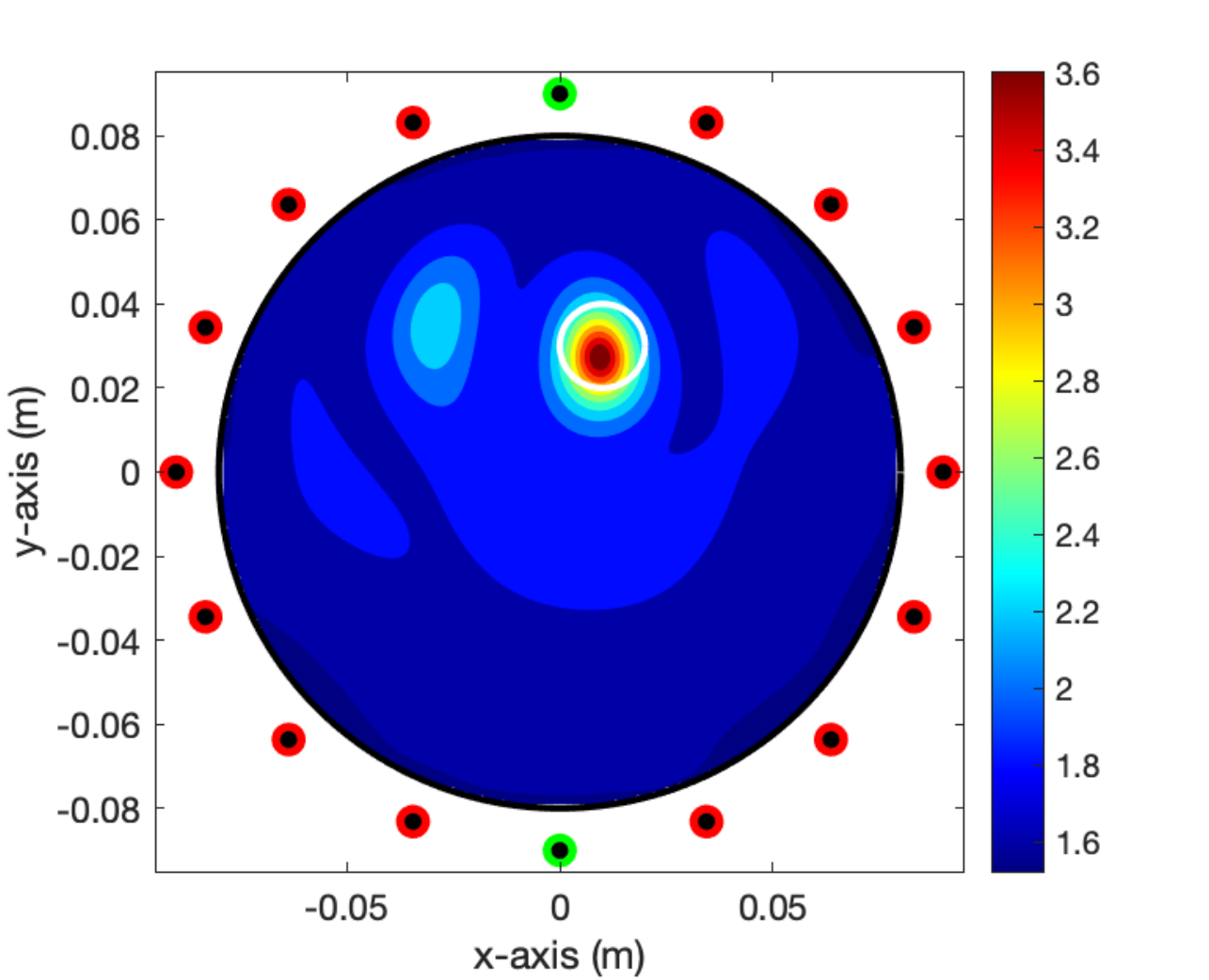}\hfill
  \includegraphics[width=0.25\textwidth]{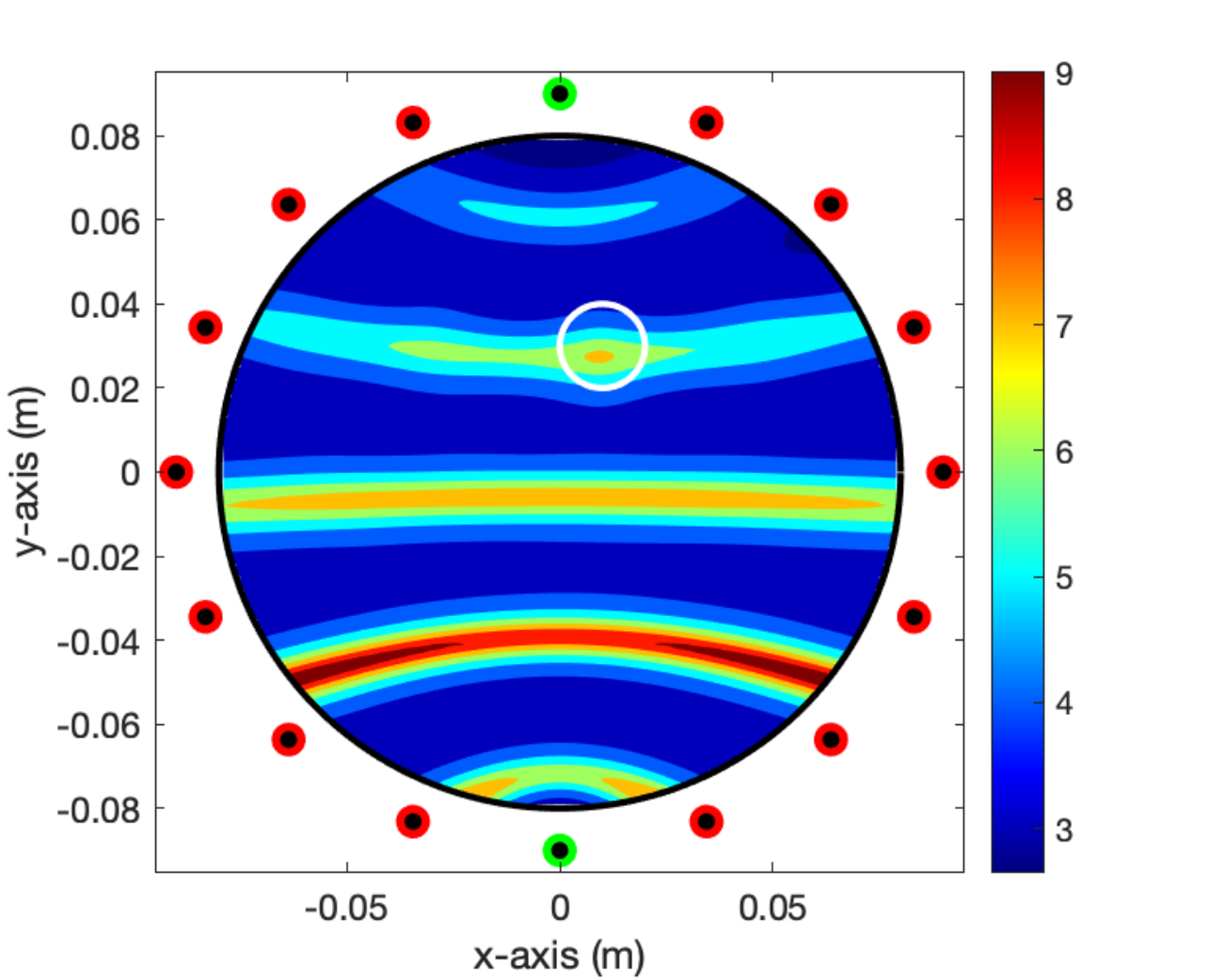}\hfill
  \includegraphics[width=0.25\textwidth]{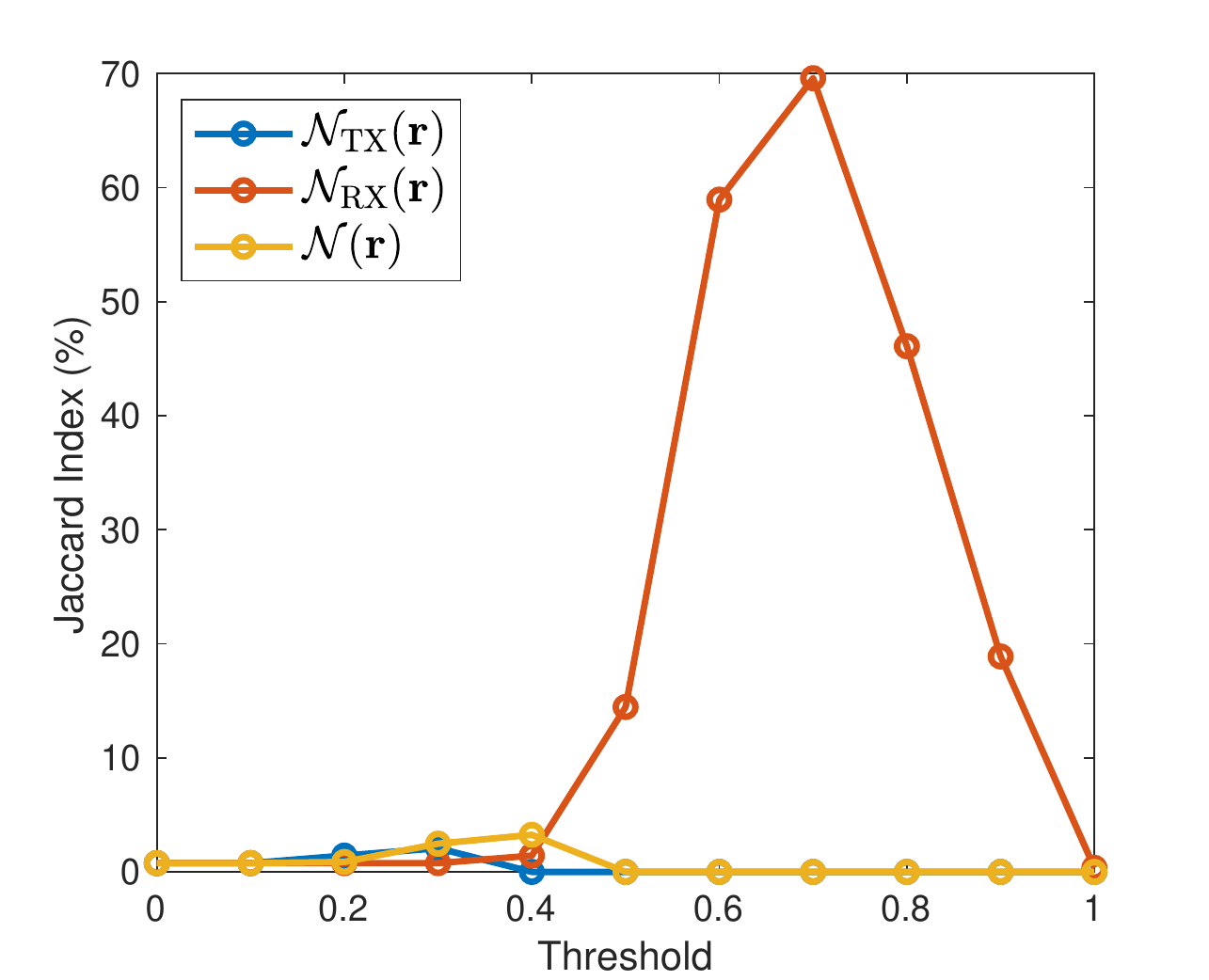}
  \caption{\label{Result7}(Example \ref{ex6}) Maps of $\mathfrak{F}_{\tx}(\mr)$ (first column), $\mathfrak{F}_{\rx}(\mr)$ (second column), $\mathfrak{F}(\mr)$ (third column), and Jaccard index (fourth column). Green and red colored circles describe the location of transmitters and receivers, respectively.}
\end{figure}

\begin{figure}[h]
  \centering
  \includegraphics[width=0.25\textwidth]{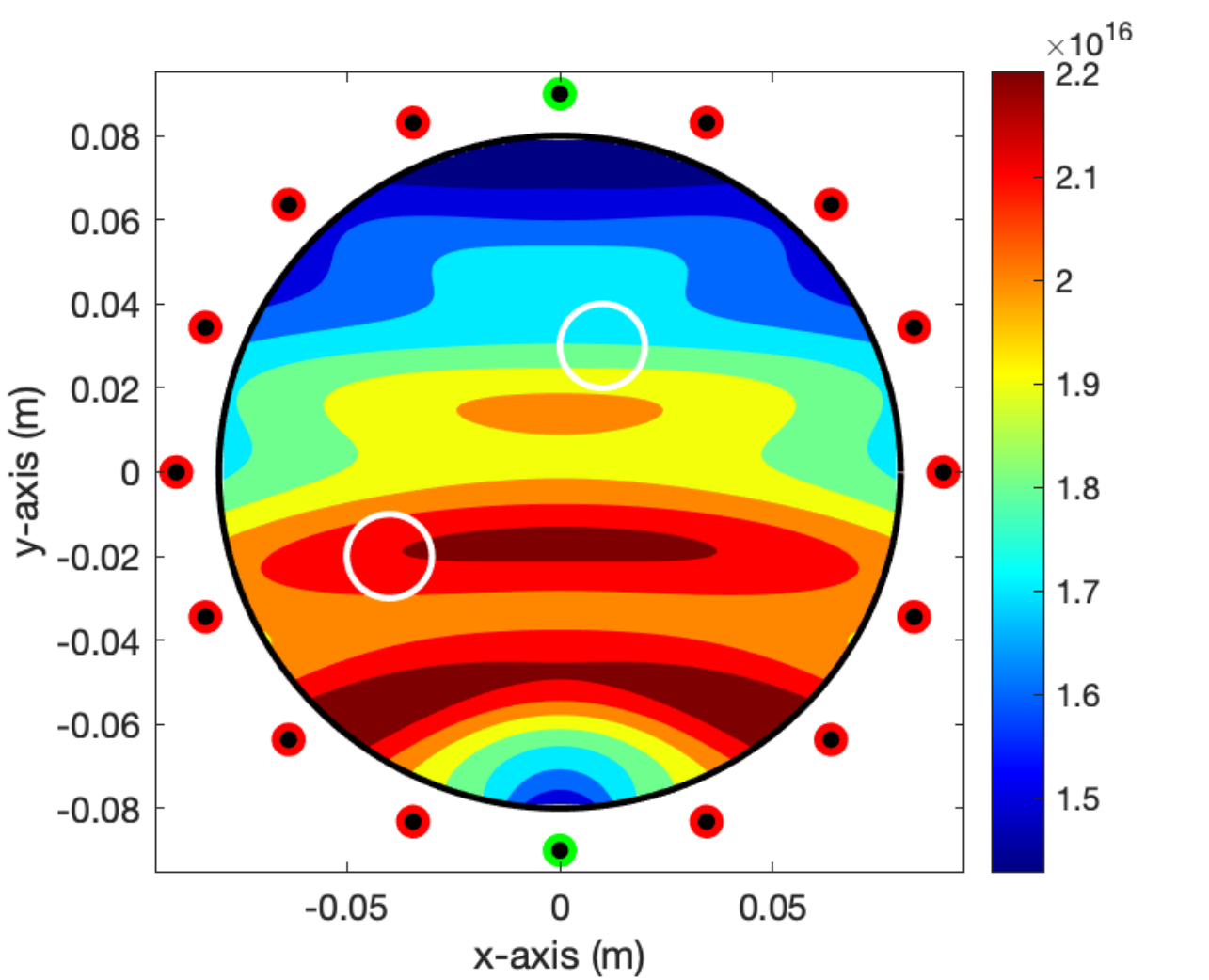}\hfill
  \includegraphics[width=0.25\textwidth]{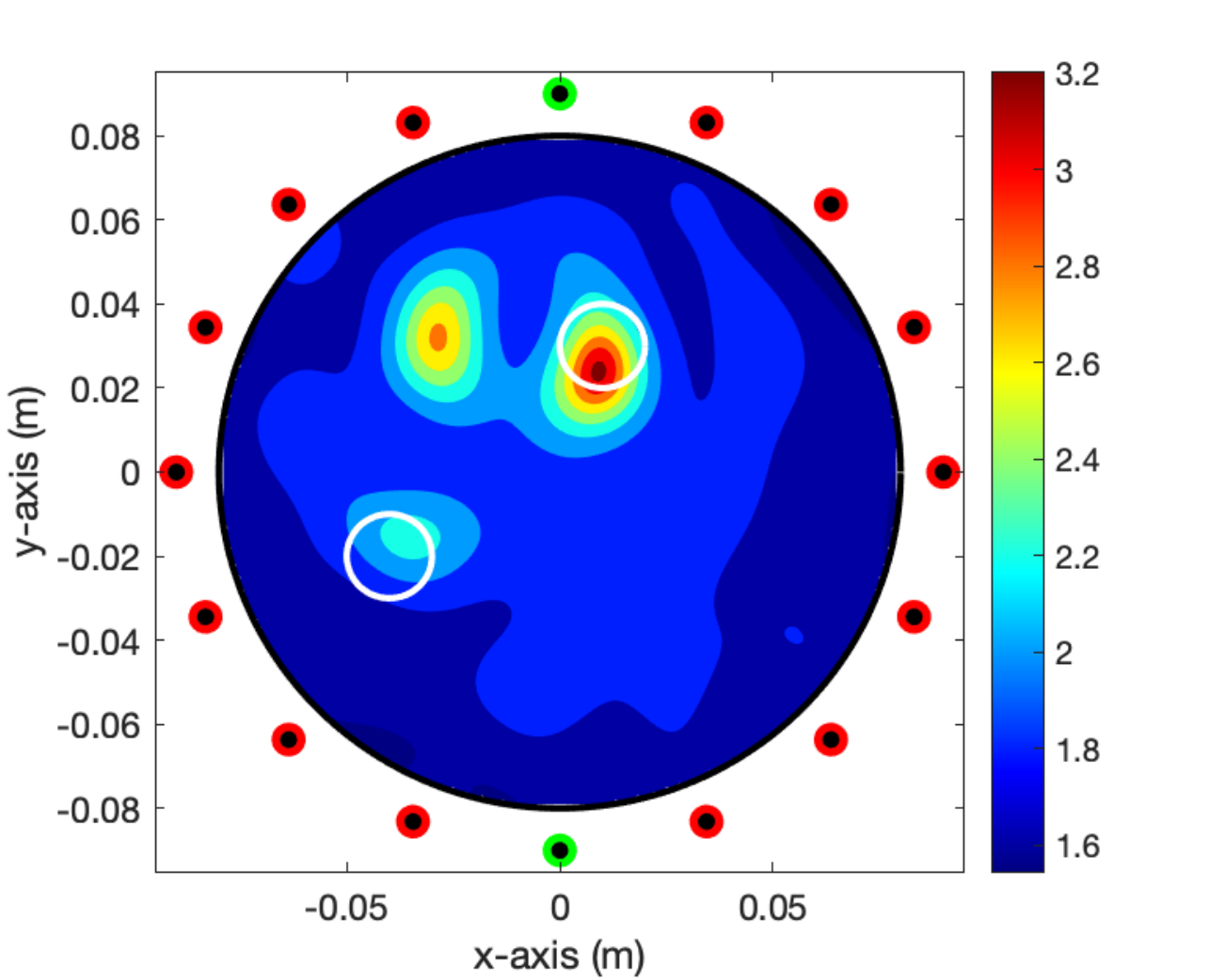}\hfill
  \includegraphics[width=0.25\textwidth]{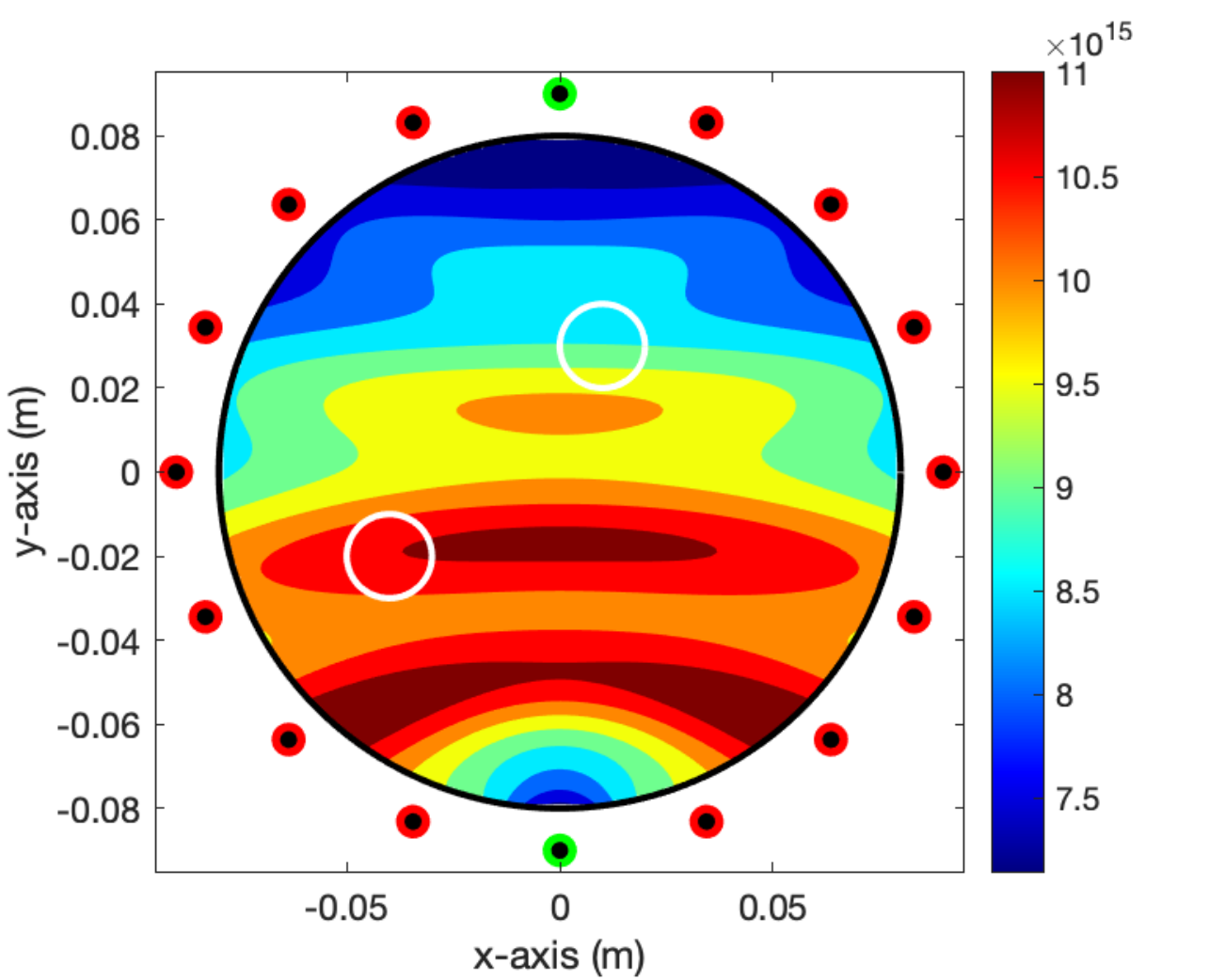}\hfill
  \includegraphics[width=0.25\textwidth]{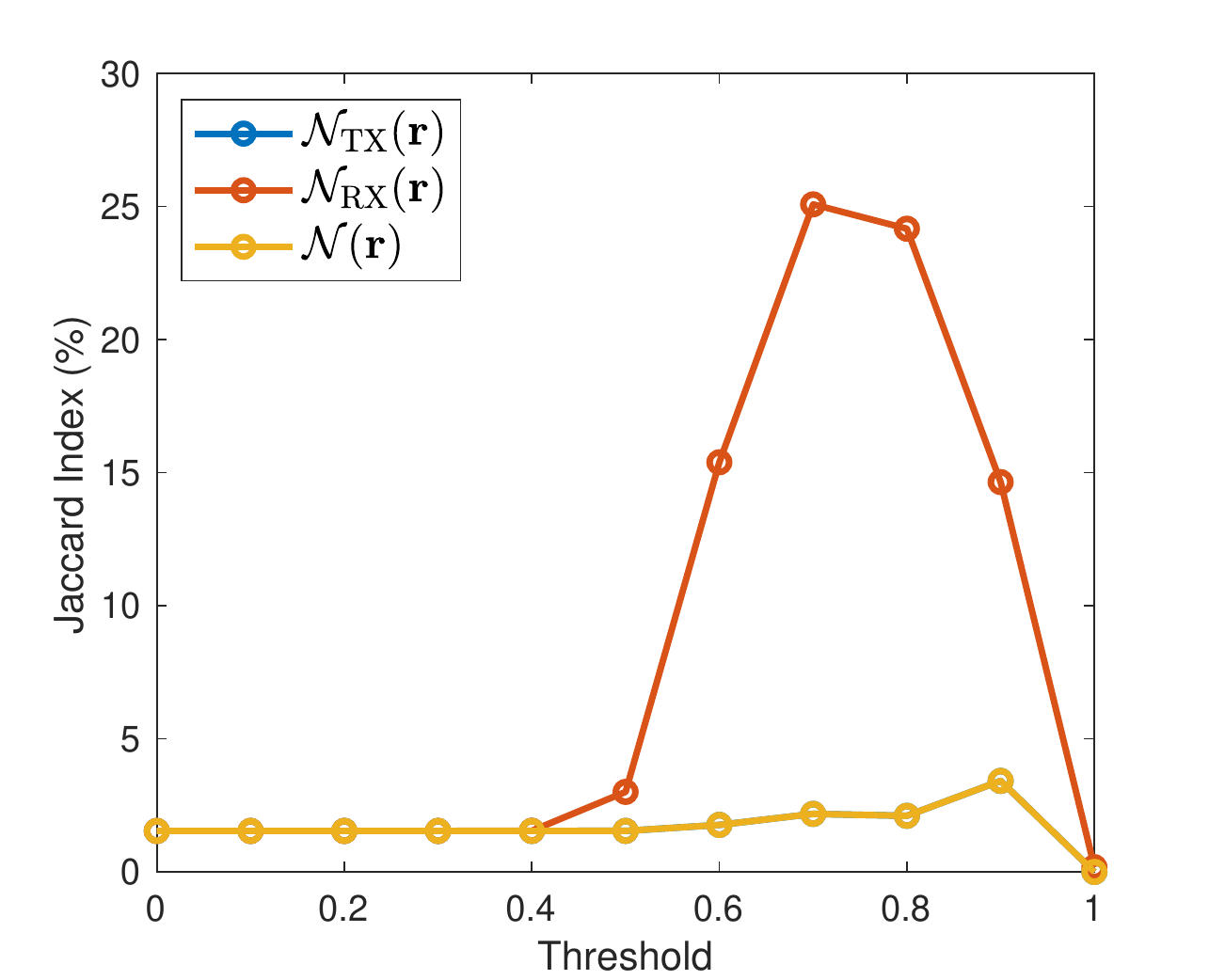}\\
  \includegraphics[width=0.25\textwidth]{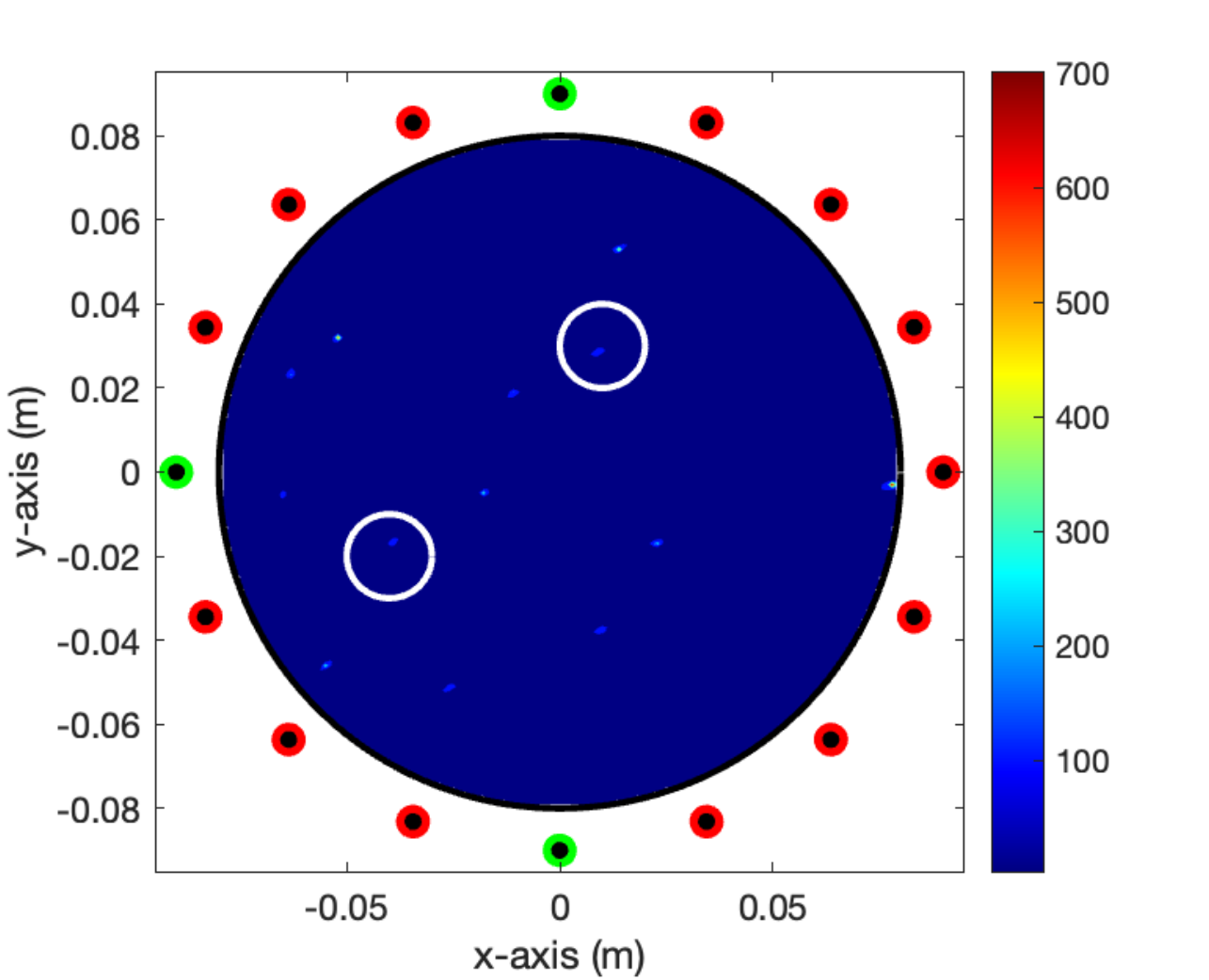}\hfill
  \includegraphics[width=0.25\textwidth]{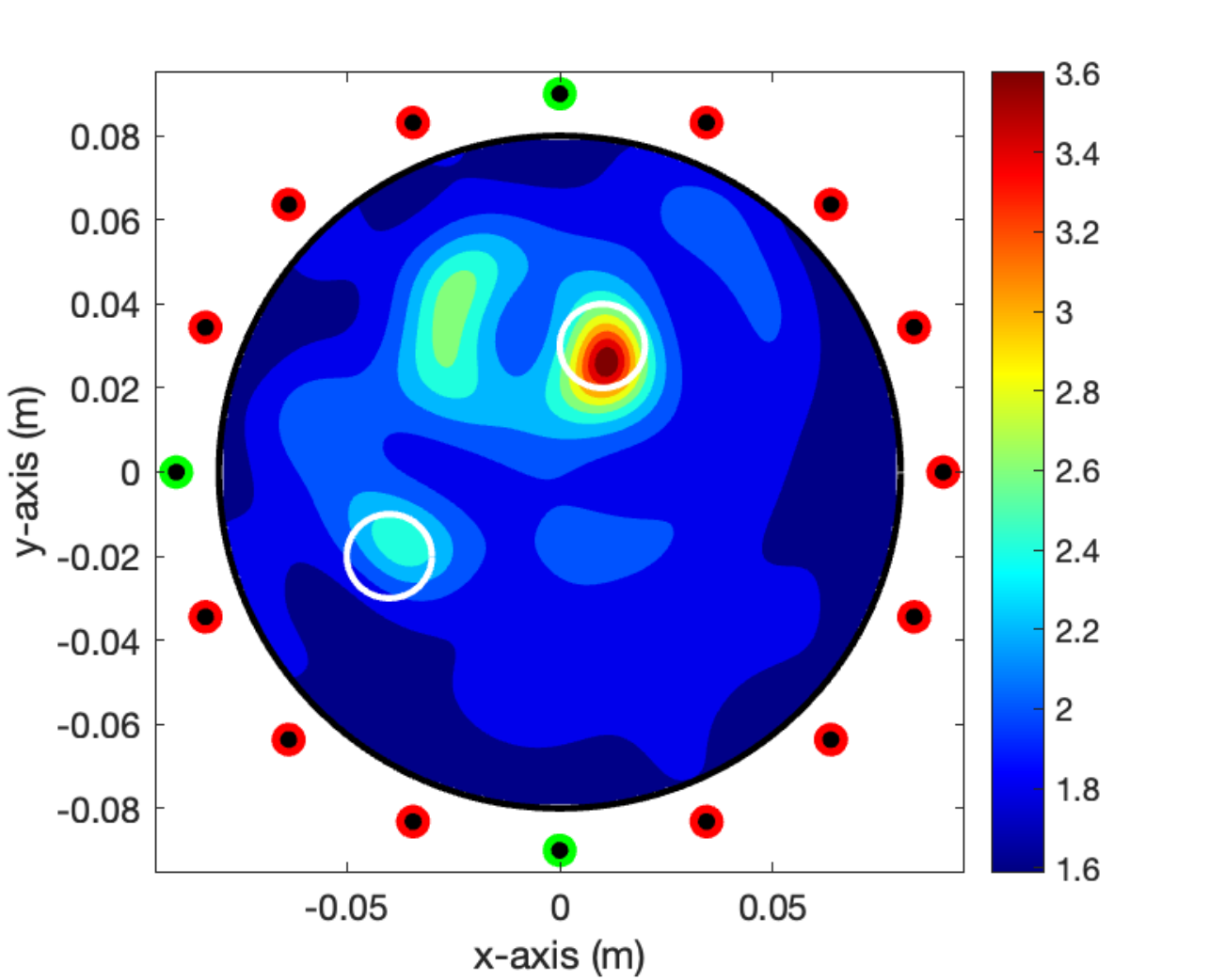}\hfill
  \includegraphics[width=0.25\textwidth]{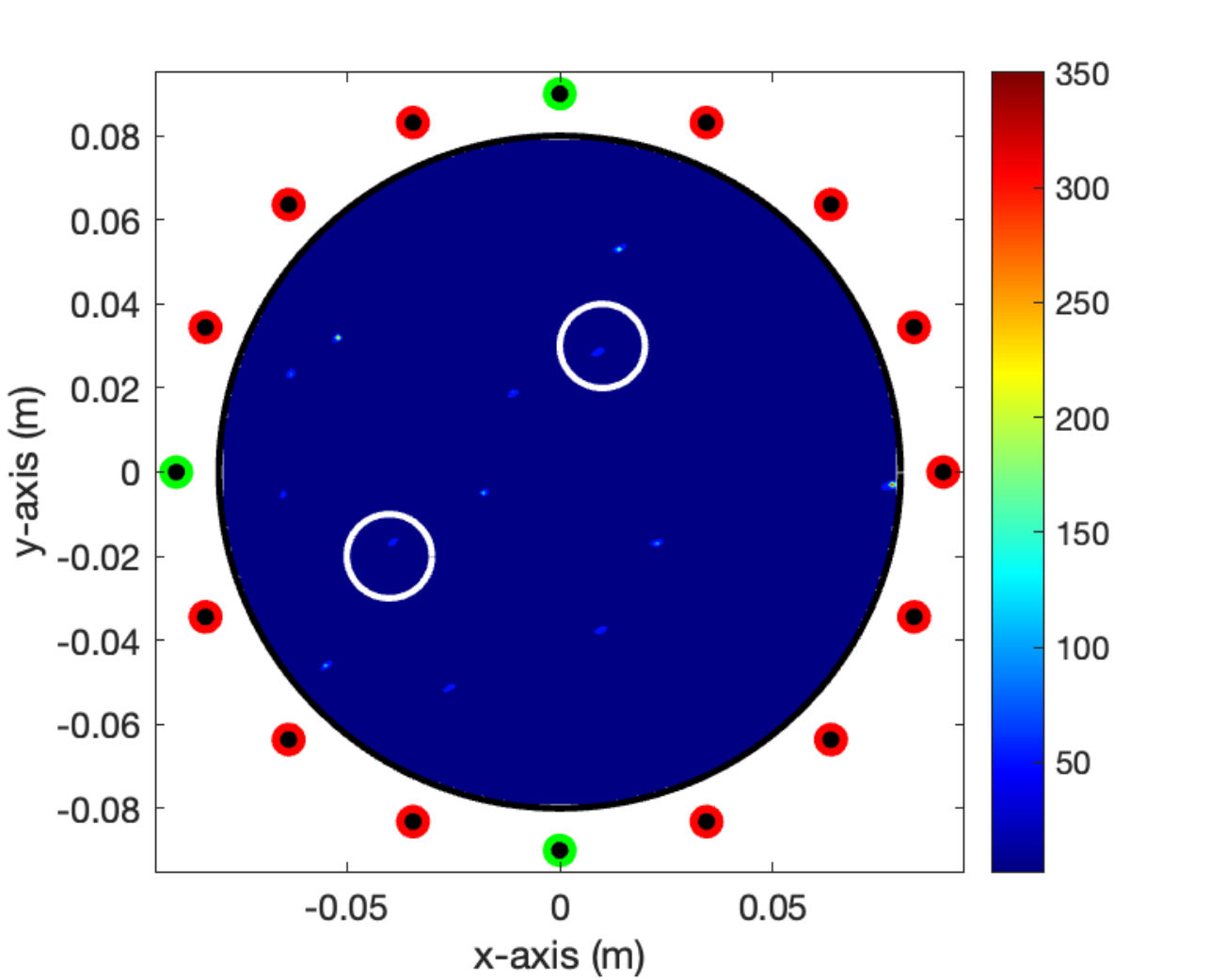}\hfill
  \includegraphics[width=0.25\textwidth]{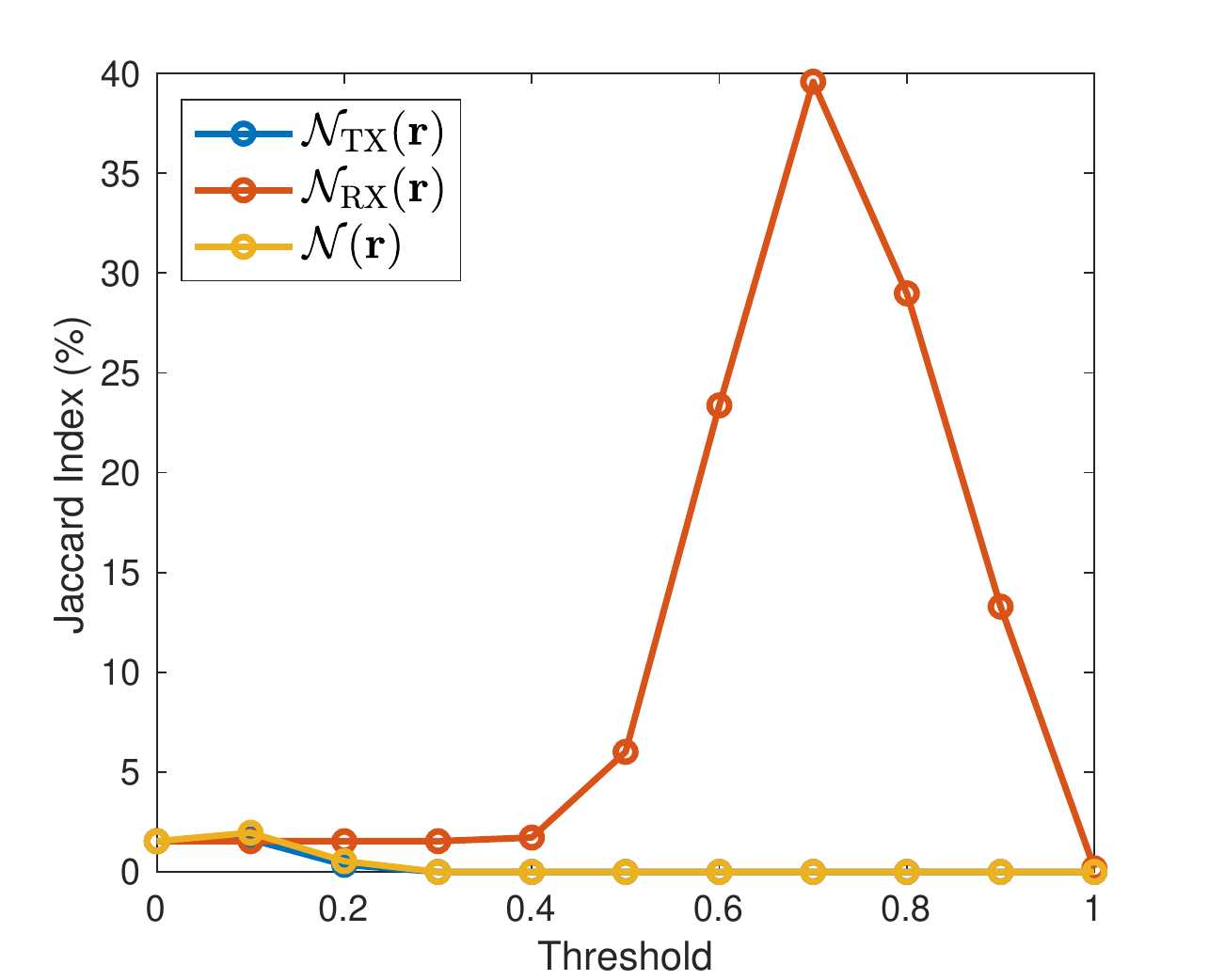}
  \caption{\label{Result8}(Example \ref{ex6}) Maps of $\mathfrak{F}_{\tx}(\mr)$ (first column), $\mathfrak{F}_{\rx}(\mr)$ (second column), $\mathfrak{F}(\mr)$ (third column), and Jaccard index (fourth column). Green and red colored circles describe the location of transmitters and receivers, respectively.}
\end{figure}

\section{Simulation Results With Experimental Data}\label{sec:5}
Following to the recent contribution \cite{P-MUSIC6}, it has been confirmed that the MUSIC algorithm can be applied to the real-world microwave imaging for identifying anomalies. In this section, we perform the imaging procedure with the experimental data to demonstrate the applicability of designed MUSIC. To this end, we used the microwave machine manufactured by the research team of the Radio Environment \& Monitoring research group of the Electronics and Telecommunications Research Institute (ETRI). We refer to \cite{KLKJS} for a detailed description of the machine. For the simulation, the machine was filled with water as a matching liquid with permittivity $\epsb=78\eps_0$ and conductivity $\sigmab=\SI{0.2}{\siemens/\meter}$ at operating frequency $f=\SI{925}{\mega\hertz}$, and we set the ROI as a circle of radius $\SI{0.08}{\meter}$ centered at the origin. We refer to Figure \ref{IllustrationReal} for the illustration.

\begin{figure}[h]
\begin{center}
\includegraphics[width=0.33\textwidth]{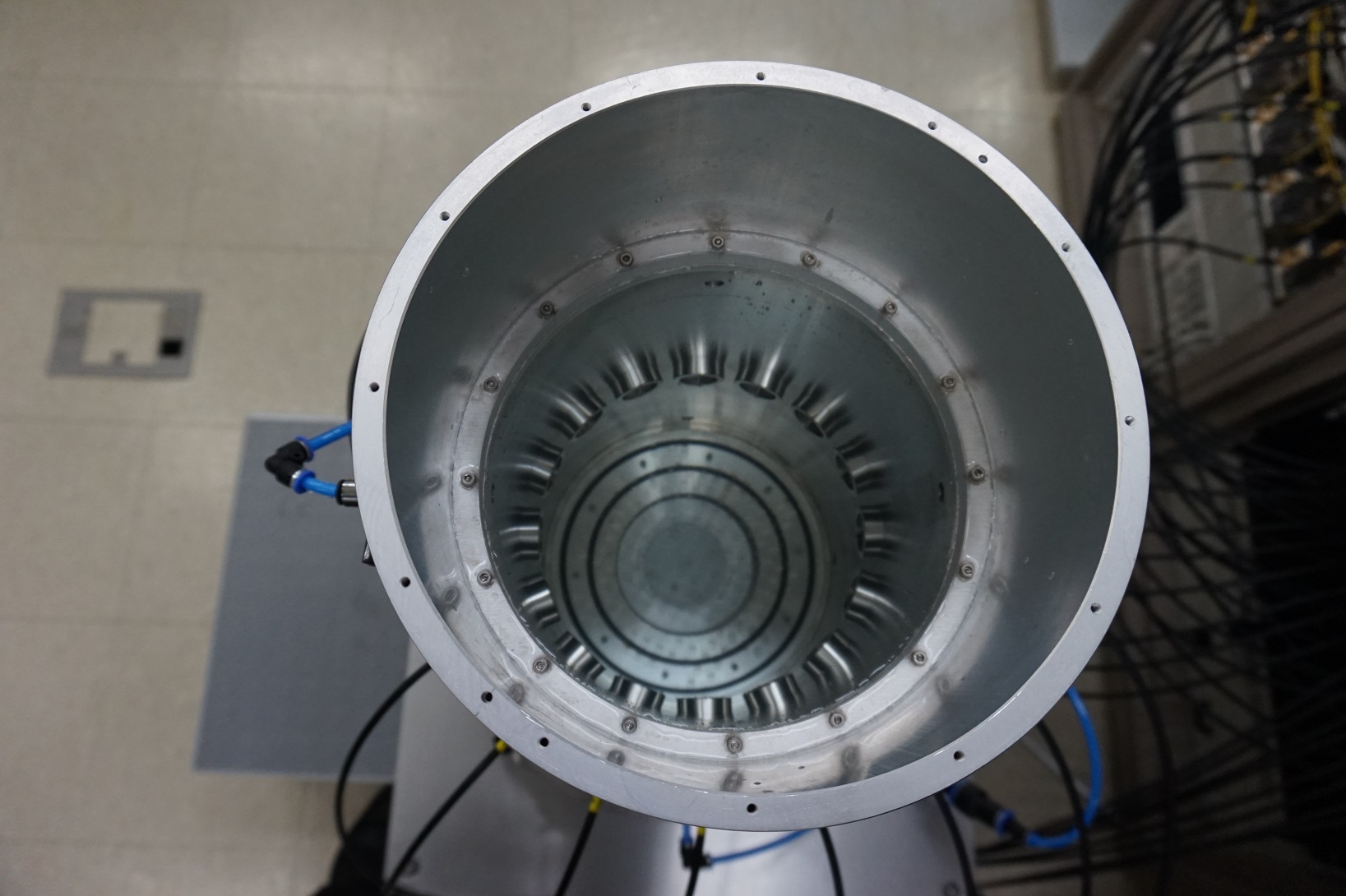}\hfill
\includegraphics[width=0.33\textwidth]{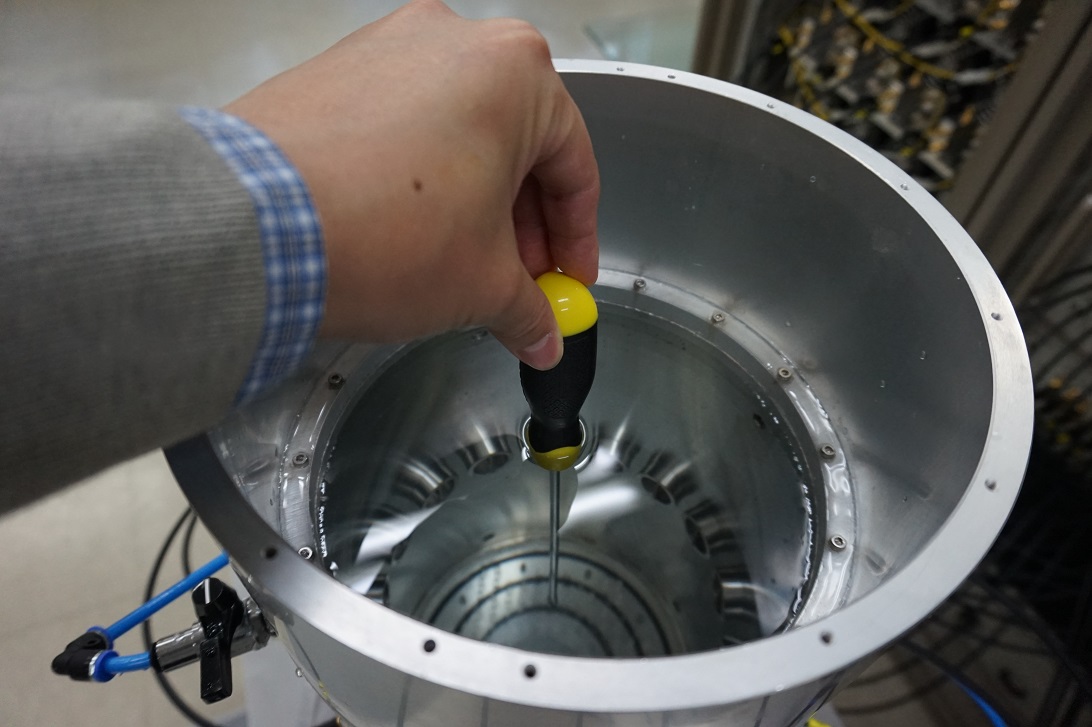}\hfill
\includegraphics[width=0.33\textwidth]{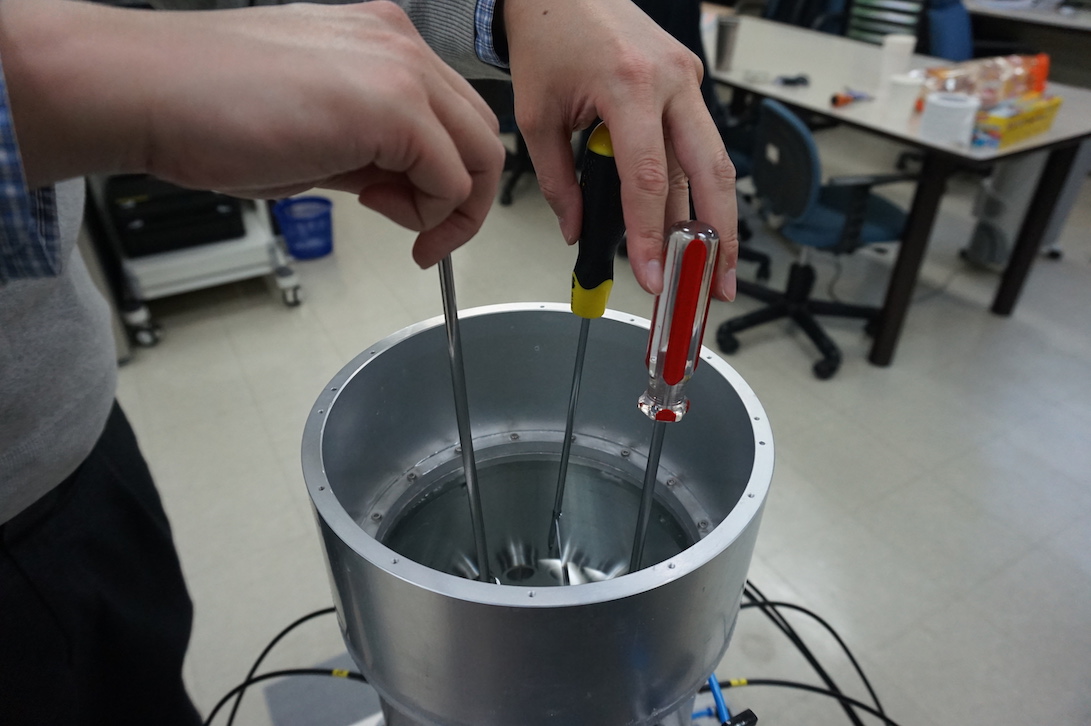}
\caption{\label{IllustrationReal}Microwave machine (left), single (center) and three (right) screw drivers.}
\end{center}
\end{figure}

\begin{example}[Imaging of Screw Driver]\label{exR1}
Here, we consider the imaging of the cross-section of single screw driver. Figure \ref{ResultR1} shows the maps of $\mathfrak{F}_{\tx}(\mr)$, $\mathfrak{F}_{\rx}(\mr)$, $\mathfrak{F}(\mr)$, and Jaccard index with the same antenna settings in Example \ref{ex1}. Based on the result, we can examine that it is very difficult to identify $\mr_\star$ due to the appearance of several artifacts. This result indicates that, same as the Example \ref{ex1}, it will be very difficult to recognize the anomaly using a small number of transmit antennas even with a large number of receive antennas.

Figure \ref{ResultR2} shows the maps of $\mathfrak{F}_{\tx}(\mr)$, $\mathfrak{F}_{\rx}(\mr)$, $\mathfrak{F}(\mr)$, and Jaccard index with the same antenna settings in Example \ref{ex2}. In contrast to the synthetic data experience, it is difficult to distinguish the location of anomaly and artifacts through the maps of $\mathfrak{F}_{\tx}(\mr)$ and $\mathfrak{F}_{\rx}(\mr)$ with settings $\mathbf{A}_j\cup\mathbf{B}_2$, $j=1,2,3,4$. Notice that it is difficult to identify $\mr_\star$ through the map of $\mathfrak{F}(\mr)$ with settings $\mathbf{A}_j\cup\mathbf{B}_2$, $j=1,2$ but the recognization of anomaly becomes possible with settings $\mathbf{A}_j\cup\mathbf{B}_2$, $j=3,4$. Hence, similar to the synthetic data experience, it will be possible to identify small anomaly by increasing the total number of transmitting and receiving antennas.

Figure \ref{ResultR3} shows the maps of $\mathfrak{F}_{\tx}(\mr)$, $\mathfrak{F}_{\rx}(\mr)$, $\mathfrak{F}(\mr)$, and Jaccard index with the same antenna settings in Example \ref{ex3}. Based on the imaging results, we can examine that
\begin{enumerate}
\item[\textcircled{1}] it is difficult to recognize the existence of the anomaly through the map of $\mathfrak{F}_{\tx}(\mr)$ with $\mathbf{A}_j\cup\mathbf{B}_5$, $j=5,6,7$ and $\mathbf{A}_\star\cup\mathbf{B}_\star$,
\item[\textcircled{2}] it is difficult to identify $\mr_\star$ through the map of $\mathfrak{F}_{\tx}(\mr)$ with $\mathbf{A}_j\cup\mathbf{B}_5$, $j=5,6$ due to the appearance of several artifacts but, it is still difficult to distinguish anomaly from artifacts, the peak of large magnitude becomes appear in the map of $\mathfrak{F}_{\rx}(\mr)$ with $\mathbf{A}_7\cup\mathbf{B}_5$ and $\mathbf{A}_\star\cup\mathbf{B}_\star$,
\item[\textcircled{3}] the location of anomaly cannot be retrieved through the map of $\mathfrak{F}(\mr)$ with $\mathbf{A}_j\cup\mathbf{A}_5$, $j=5,6,7$ but one can obtain the best imaging result by selecting the imaging function $\mathfrak{F}(\mr)$ and antenna configuration $\mathbf{A}_\star\cup\mathbf{B}_\star$.
\end{enumerate}
\end{example}

\begin{figure}[h]
  \centering
  \includegraphics[width=0.25\textwidth]{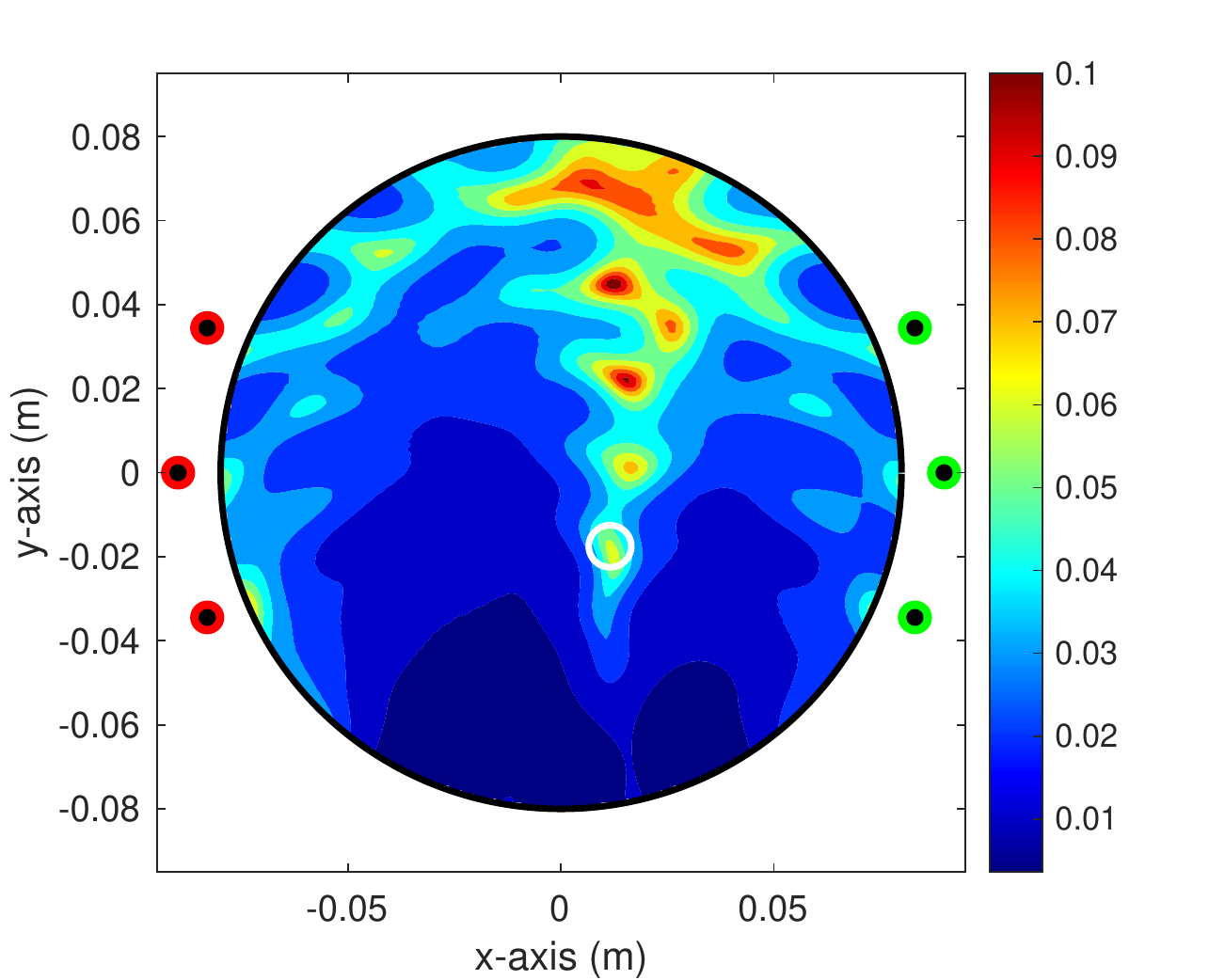}\hfill
  \includegraphics[width=0.25\textwidth]{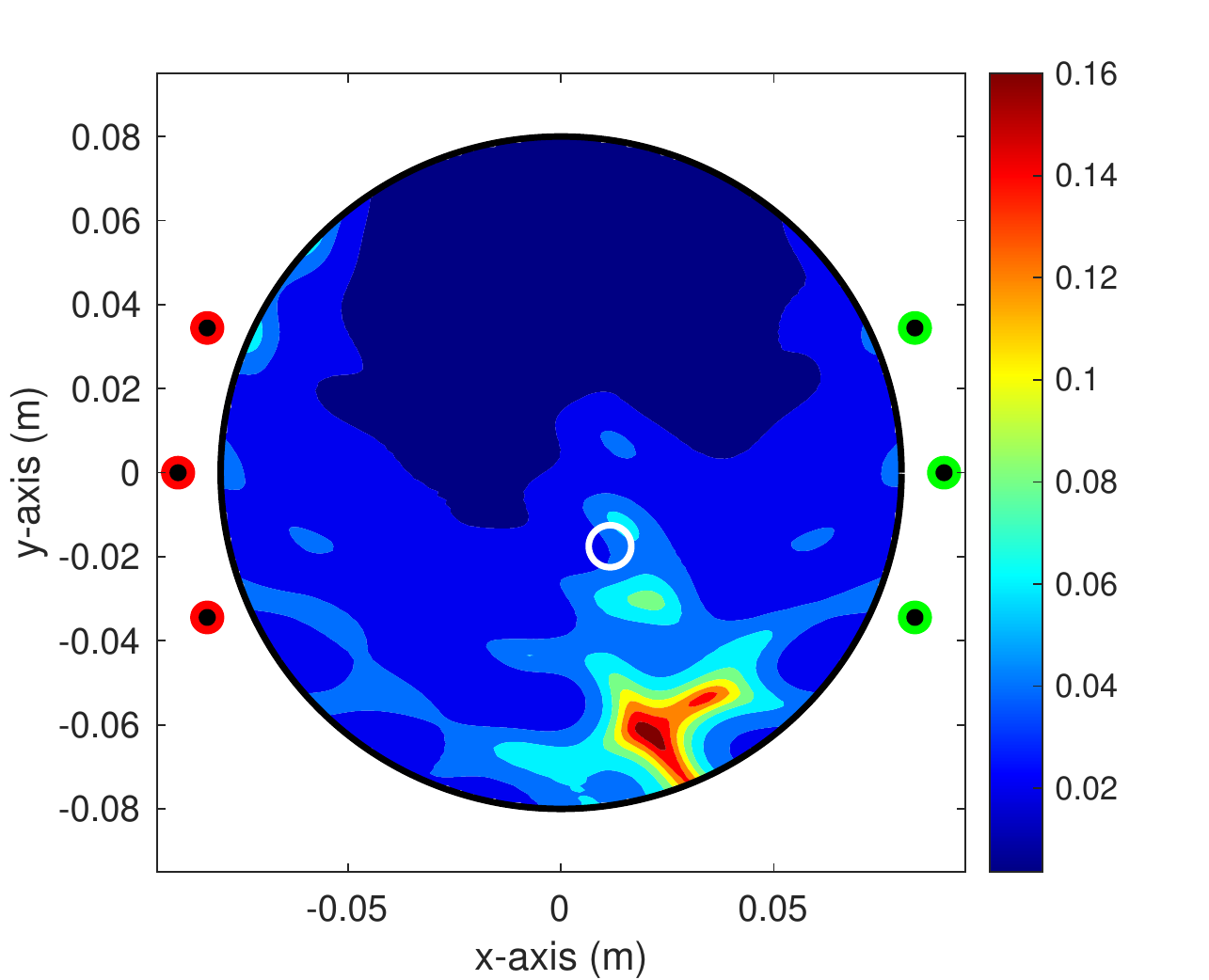}\hfill
  \includegraphics[width=0.25\textwidth]{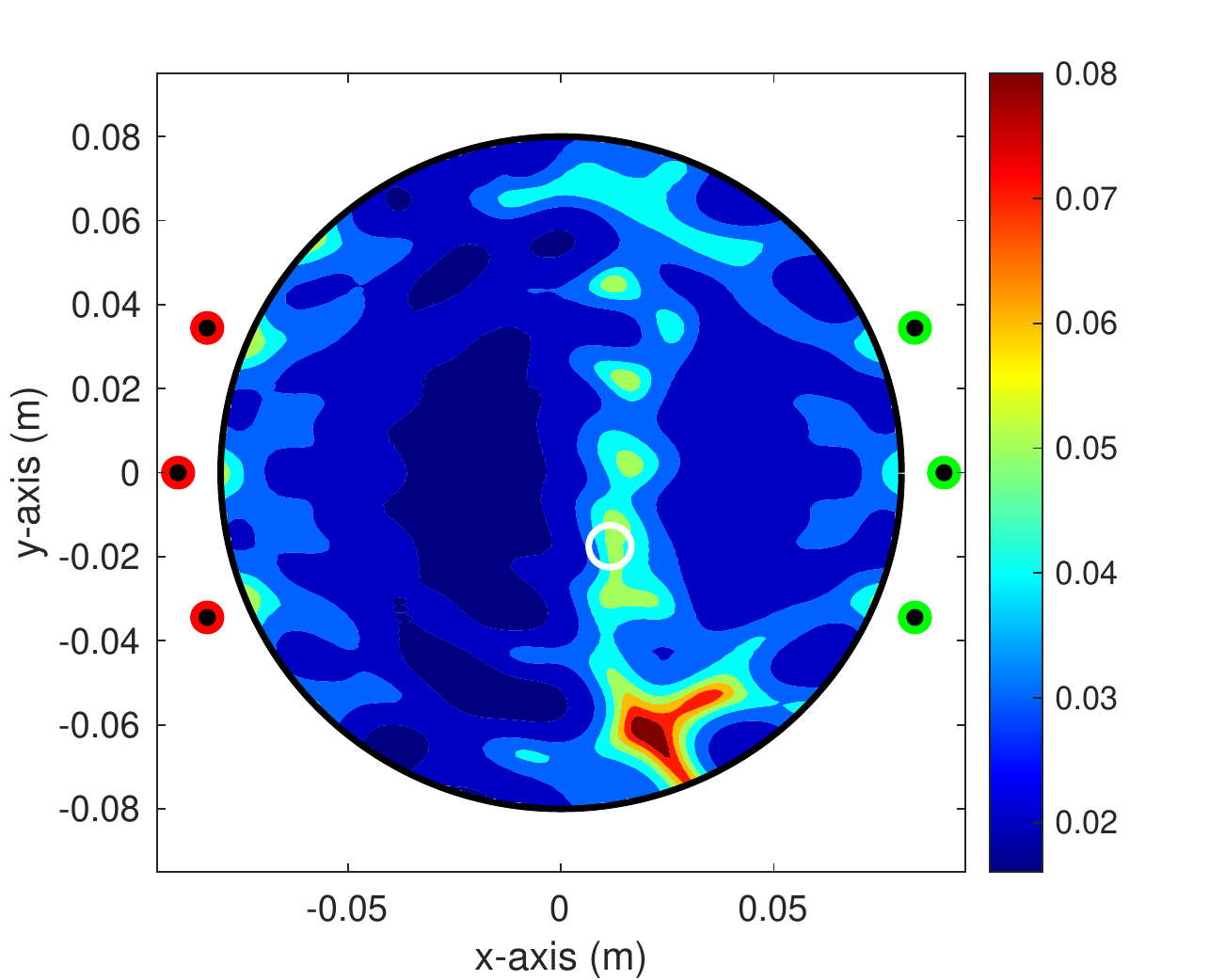}\hfill
  \includegraphics[width=0.25\textwidth]{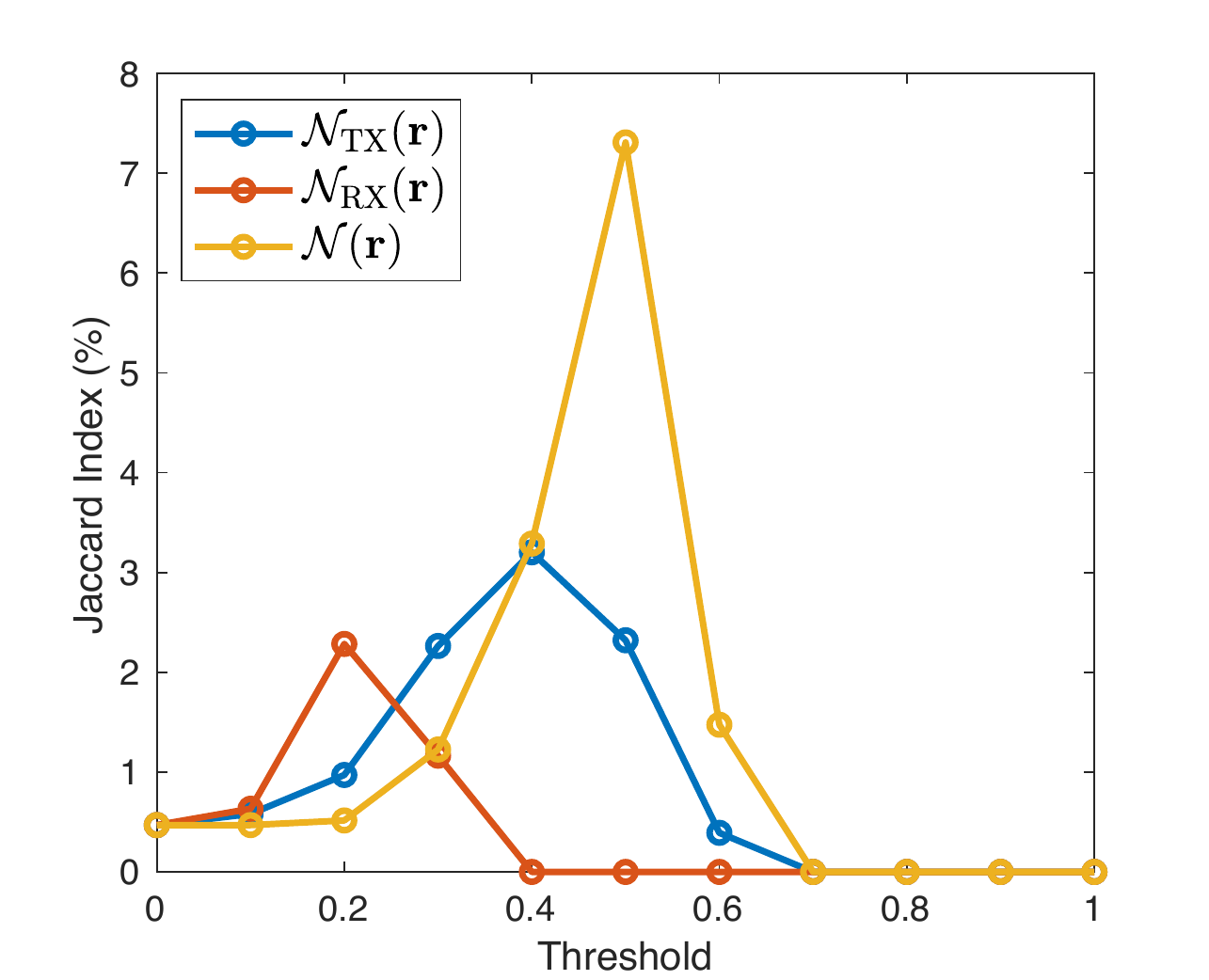}\\
  \includegraphics[width=0.25\textwidth]{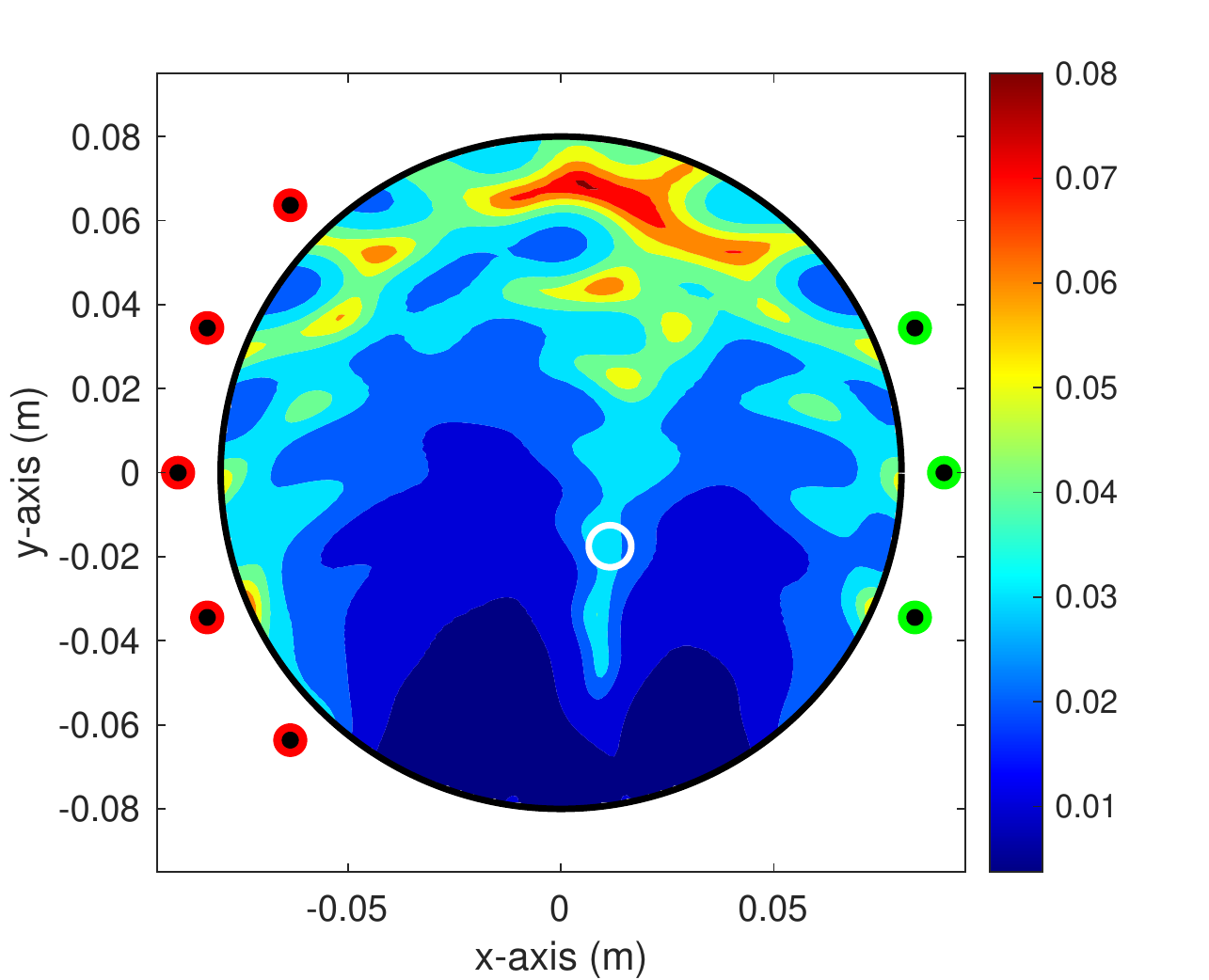}\hfill
  \includegraphics[width=0.25\textwidth]{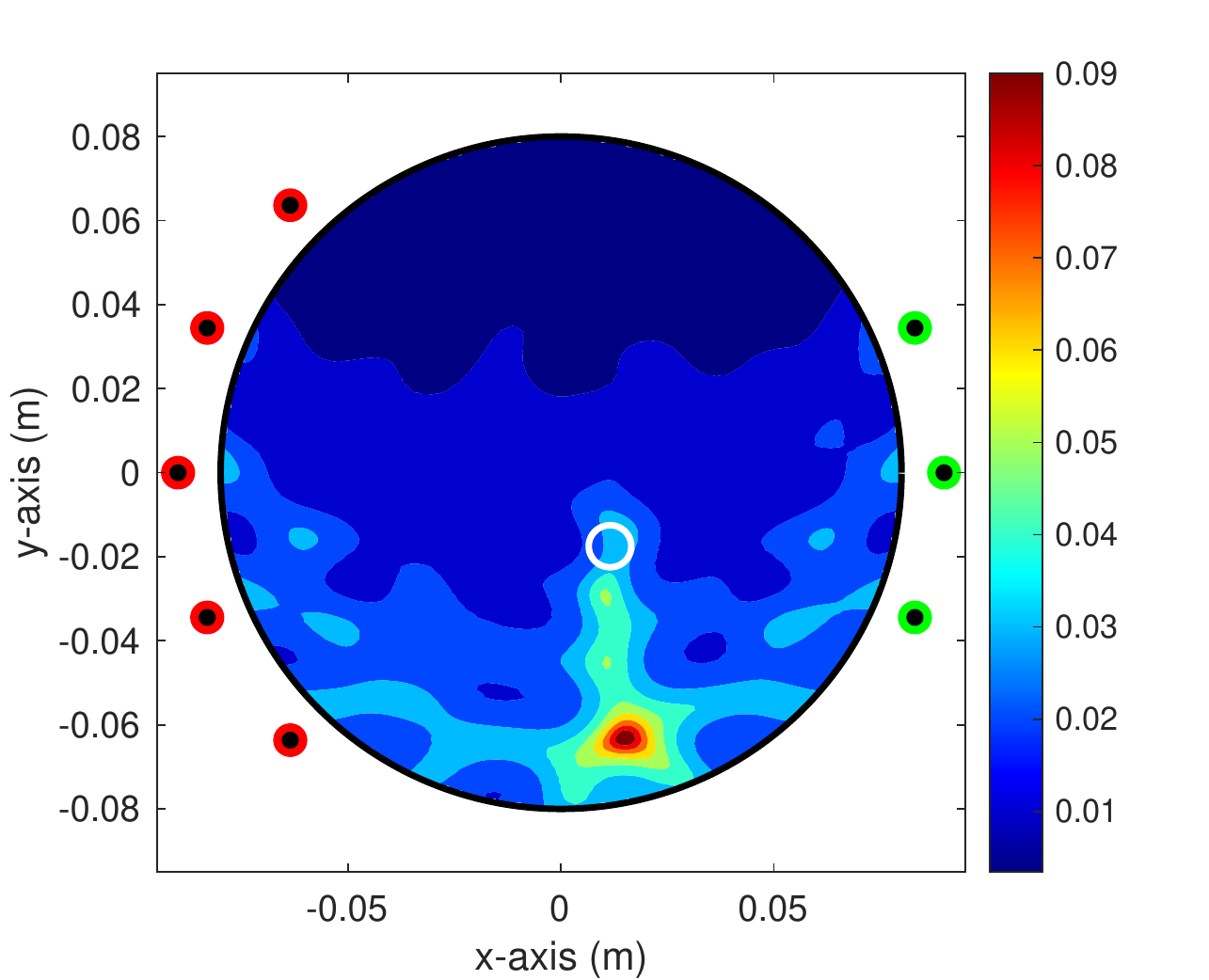}\hfill
  \includegraphics[width=0.25\textwidth]{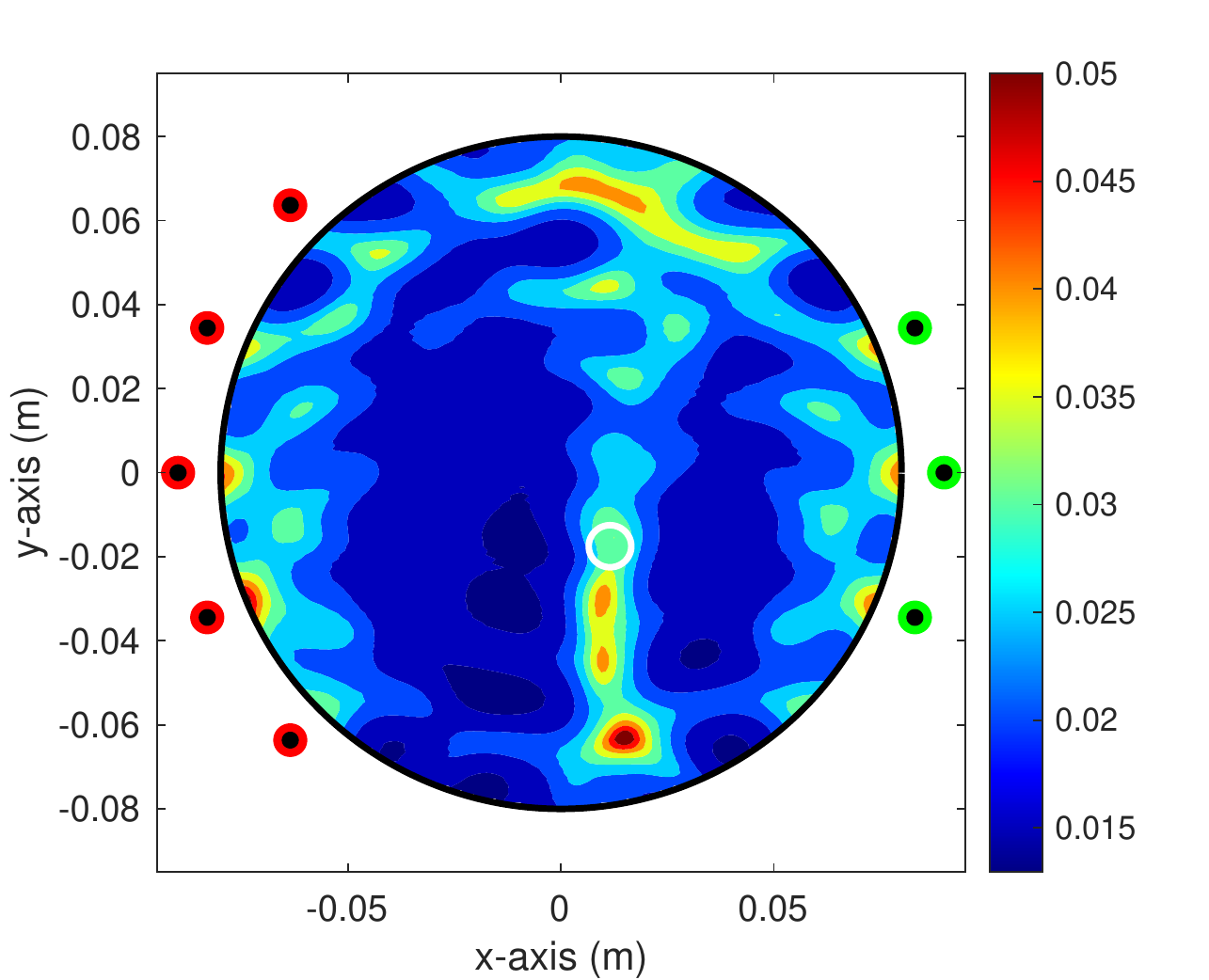}\hfill
  \includegraphics[width=0.25\textwidth]{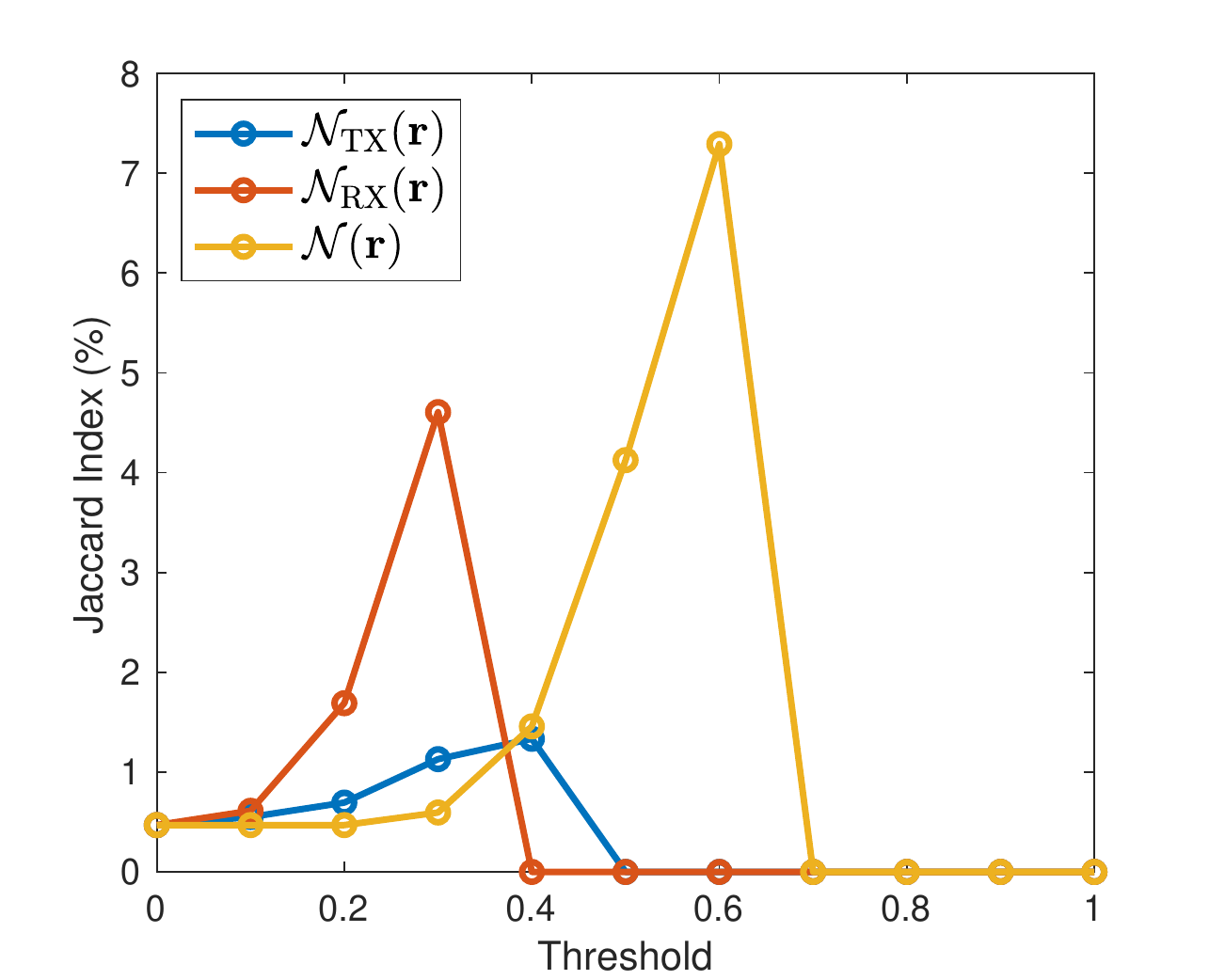}\\
  \includegraphics[width=0.25\textwidth]{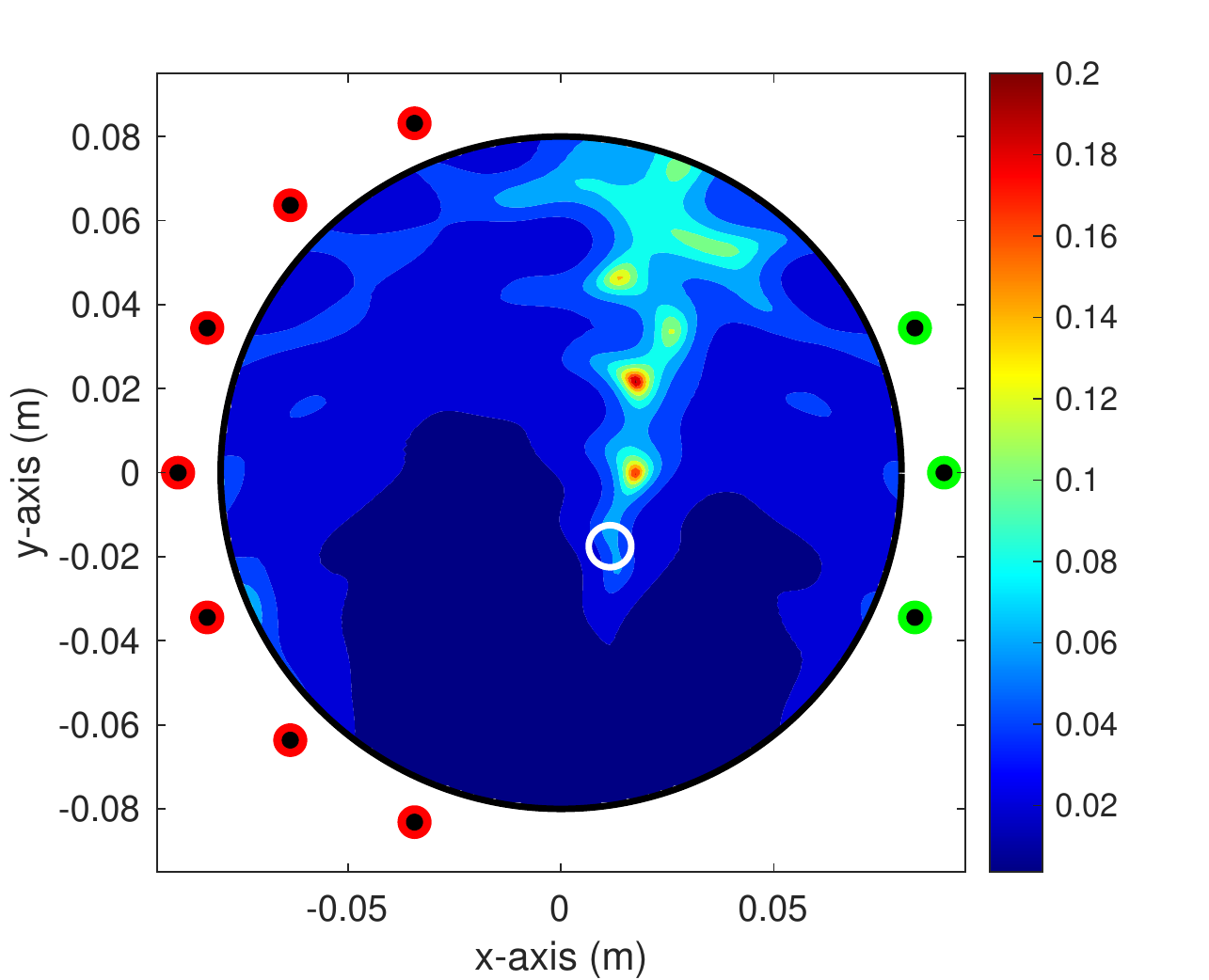}\hfill
  \includegraphics[width=0.25\textwidth]{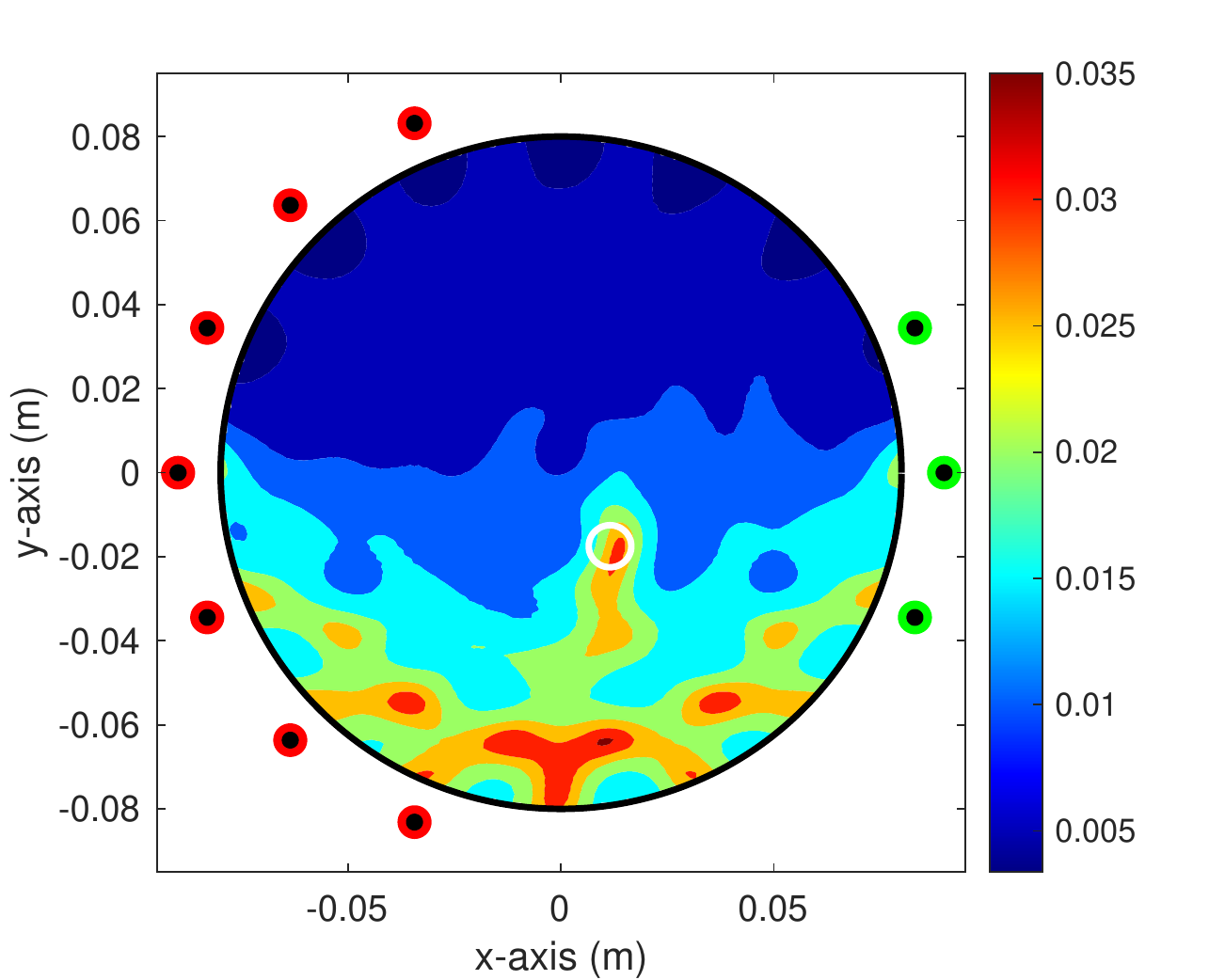}\hfill
  \includegraphics[width=0.25\textwidth]{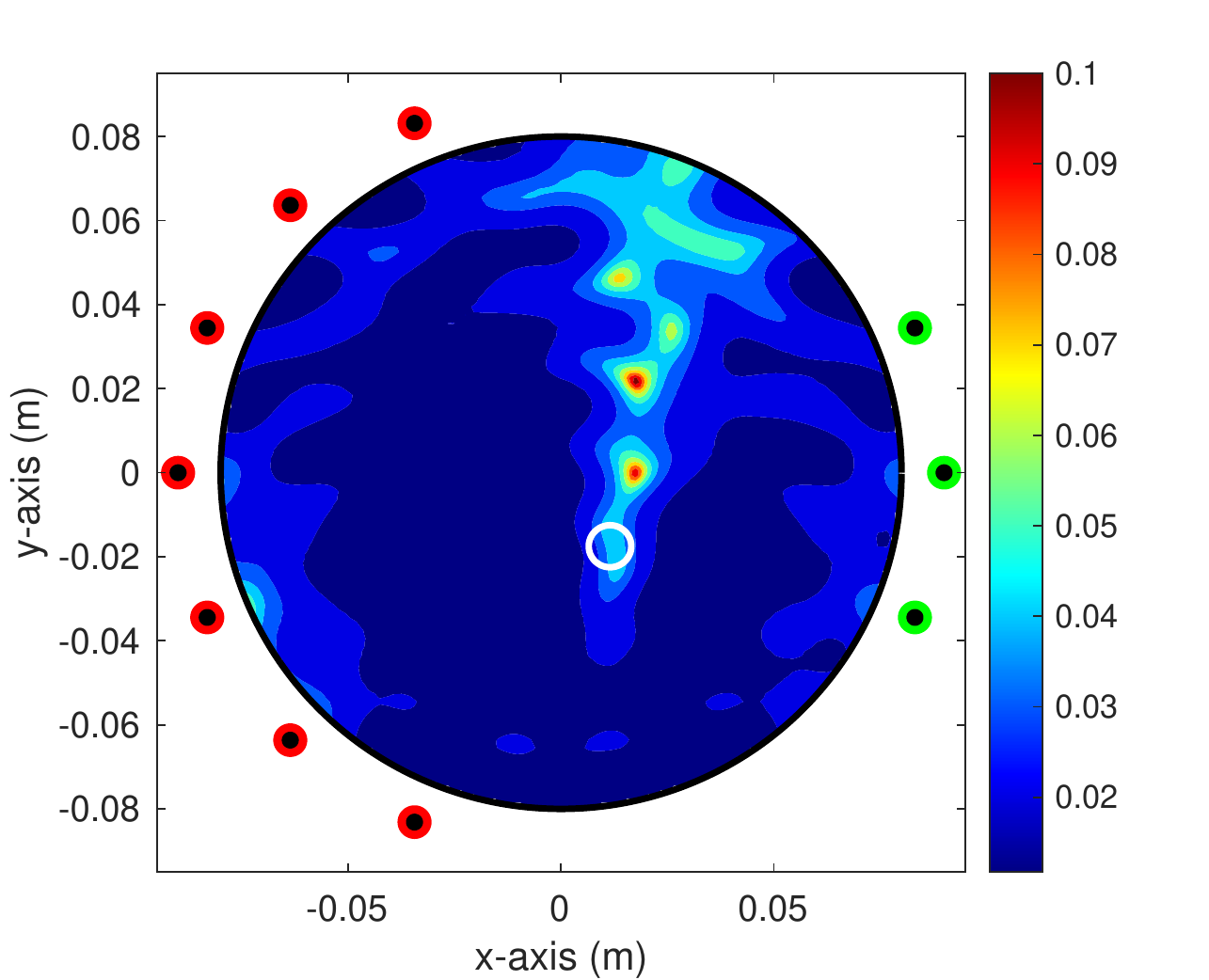}\hfill
  \includegraphics[width=0.25\textwidth]{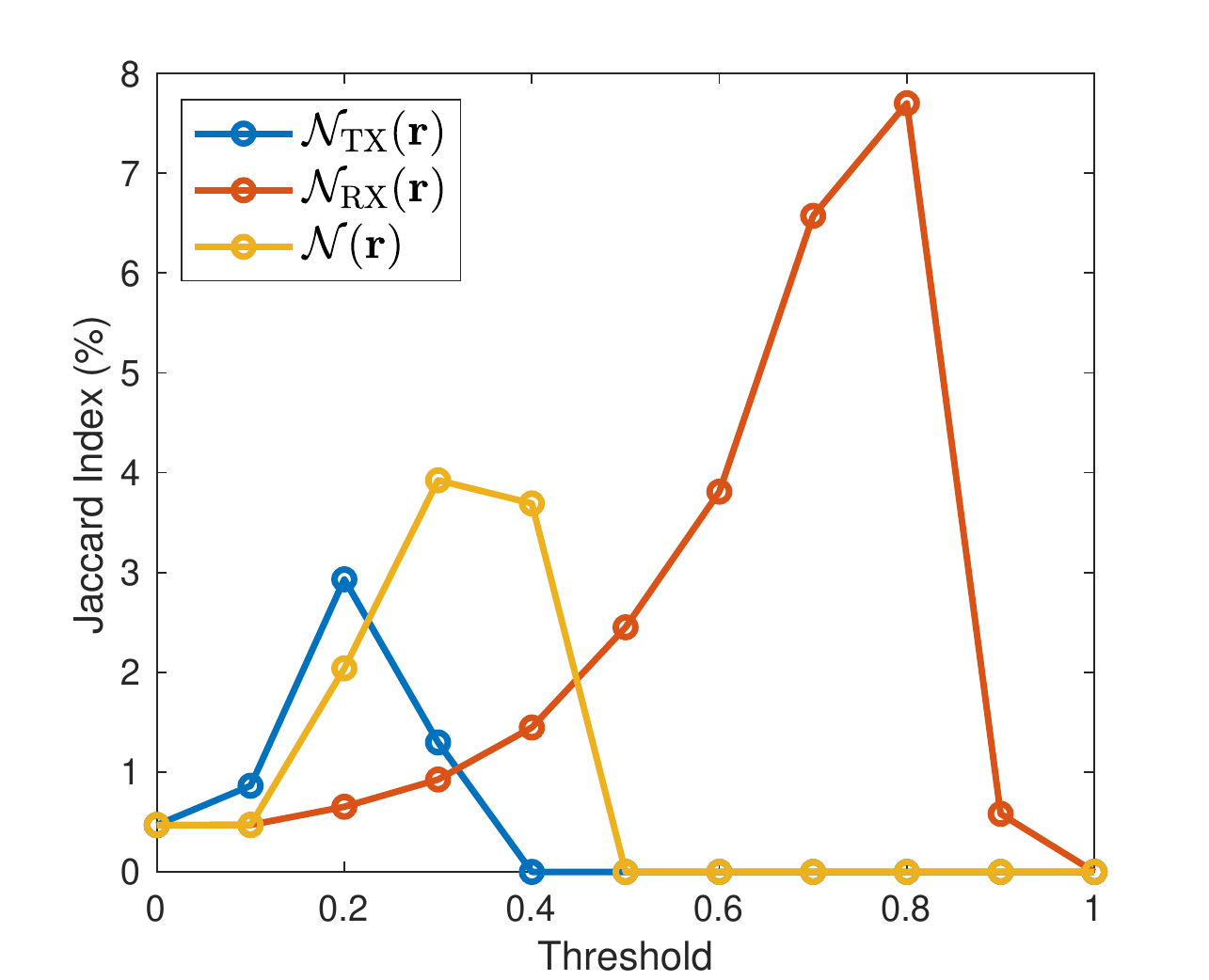}\\
  \includegraphics[width=0.25\textwidth]{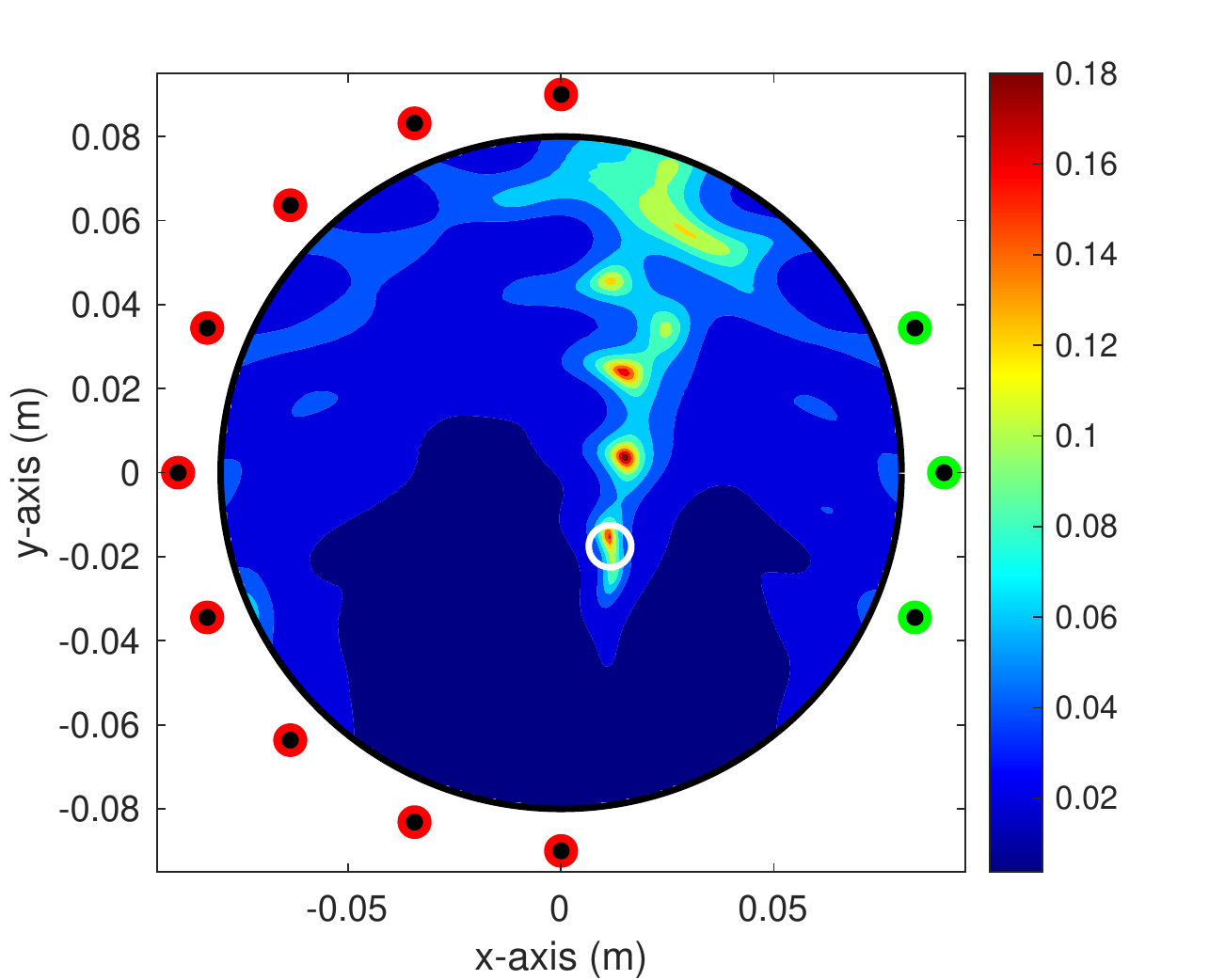}\hfill
  \includegraphics[width=0.25\textwidth]{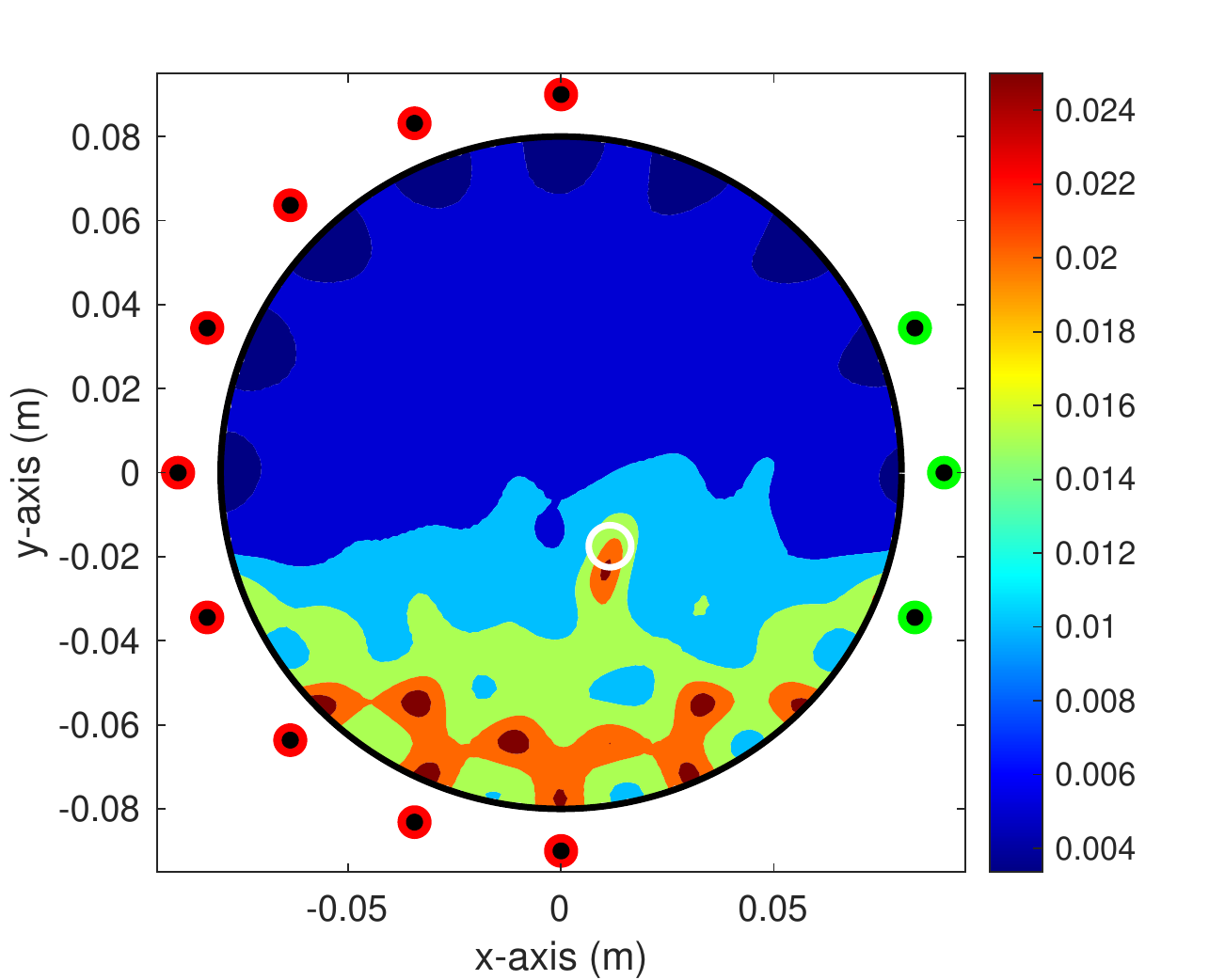}\hfill
  \includegraphics[width=0.25\textwidth]{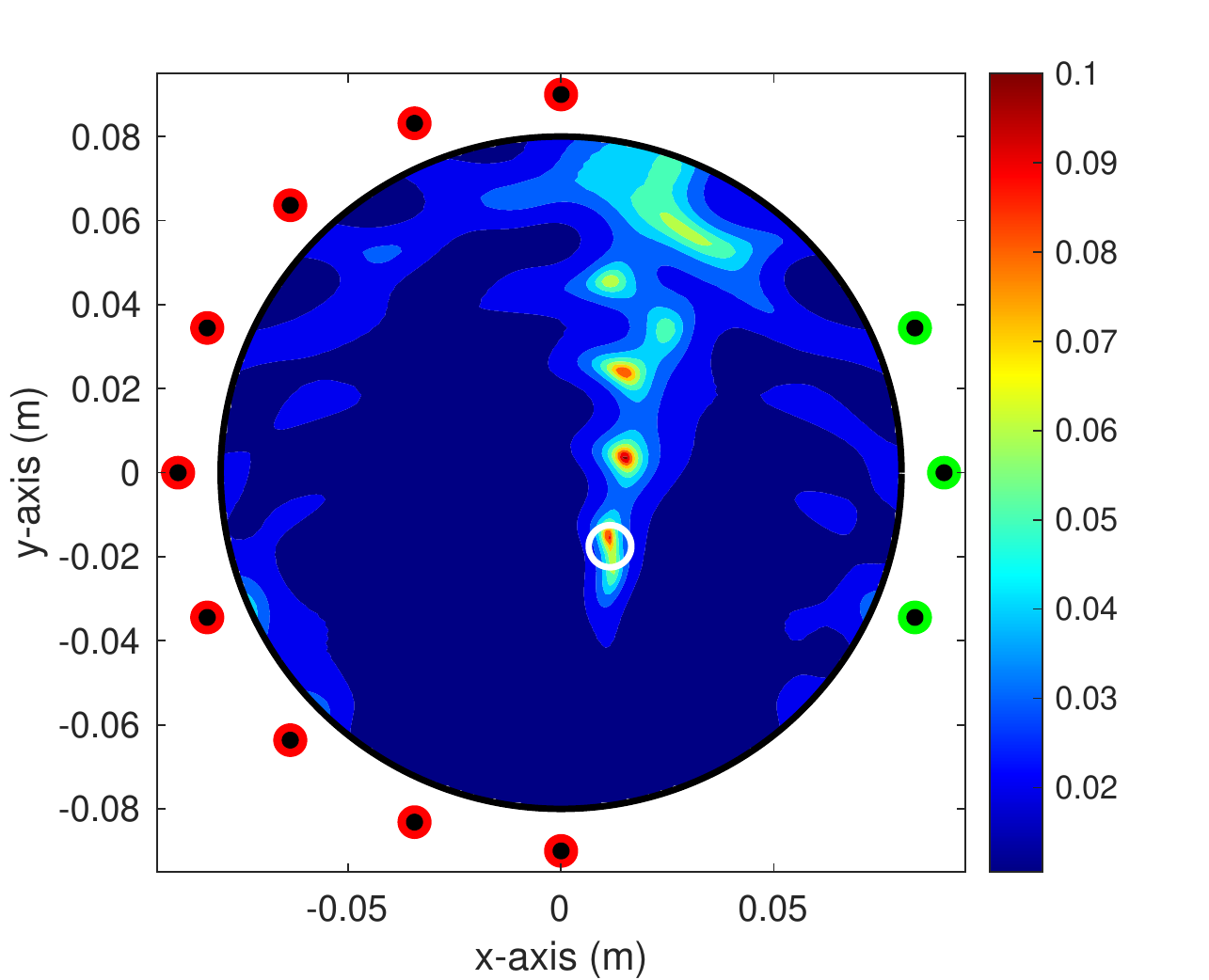}\hfill
  \includegraphics[width=0.25\textwidth]{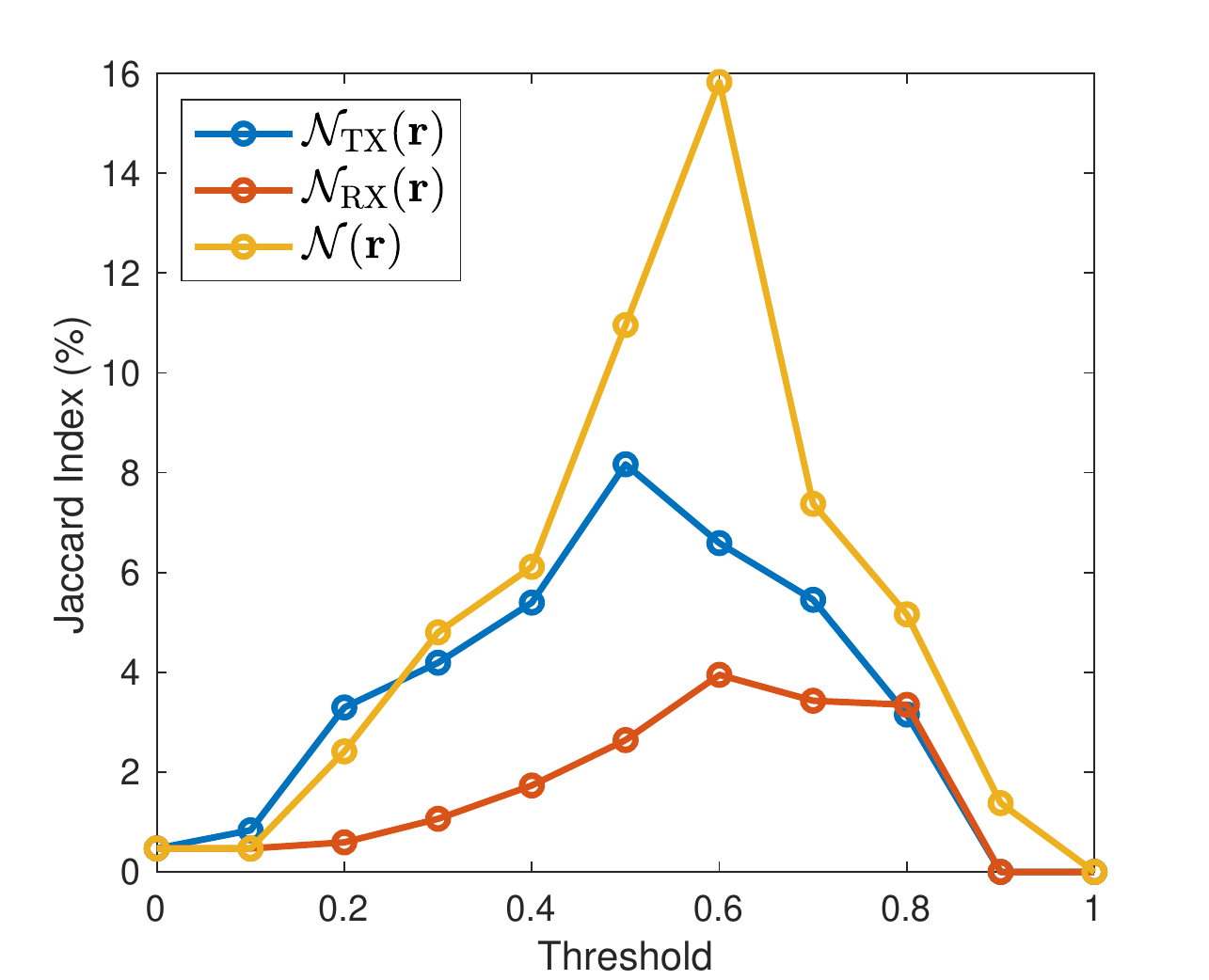}
  \caption{\label{ResultR1}(Example \ref{exR1}) Maps of $\mathfrak{F}_{\tx}(\mr)$ (first column), $\mathfrak{F}_{\rx}(\mr)$ (second column), $\mathfrak{F}(\mr)$ (third column), and Jaccard index (fourth column). Green and red colored circles describe the location of transmitters and receivers, respectively.}
\end{figure}

\begin{figure}[h]
  \centering
  \includegraphics[width=0.25\textwidth]{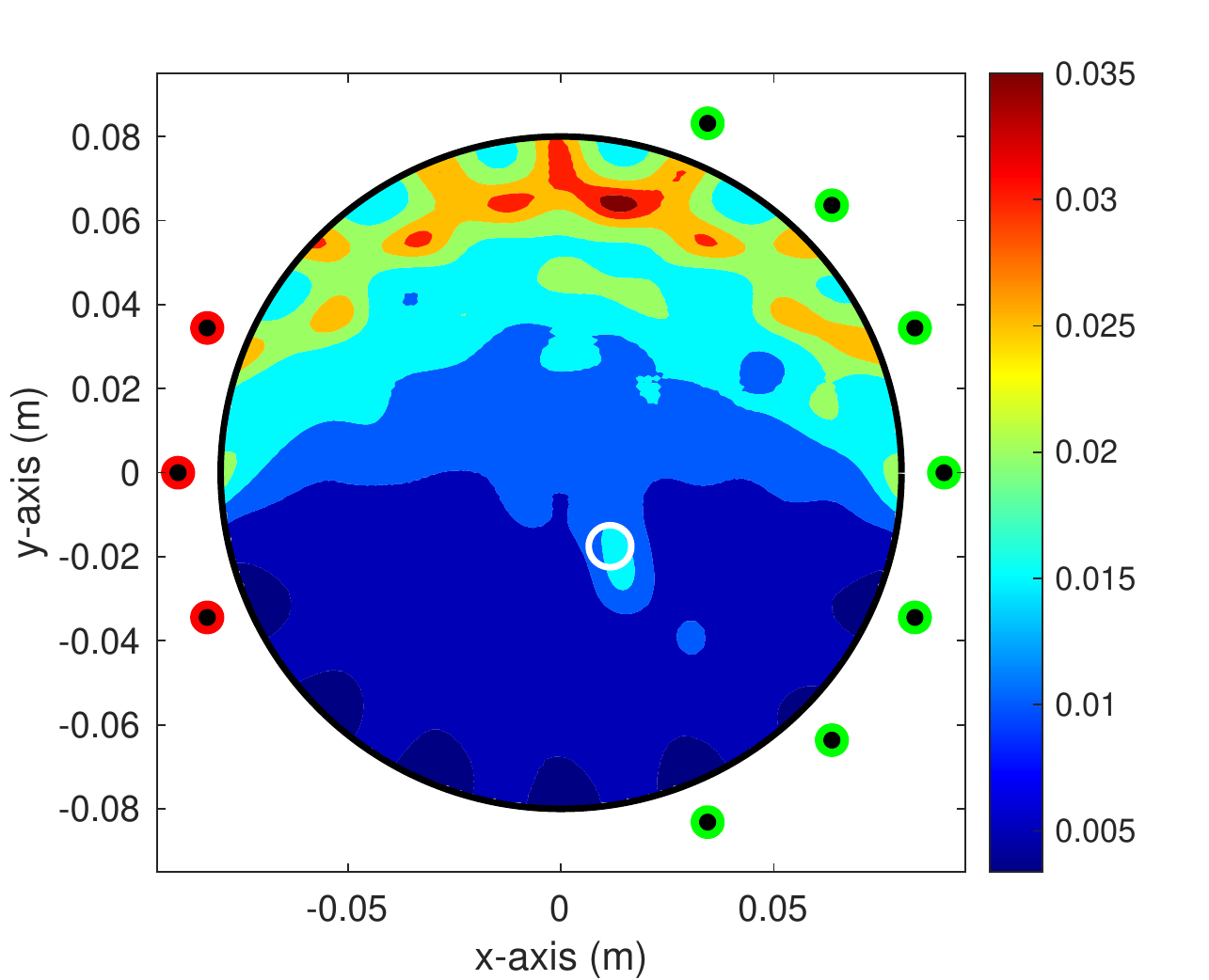}\hfill
  \includegraphics[width=0.25\textwidth]{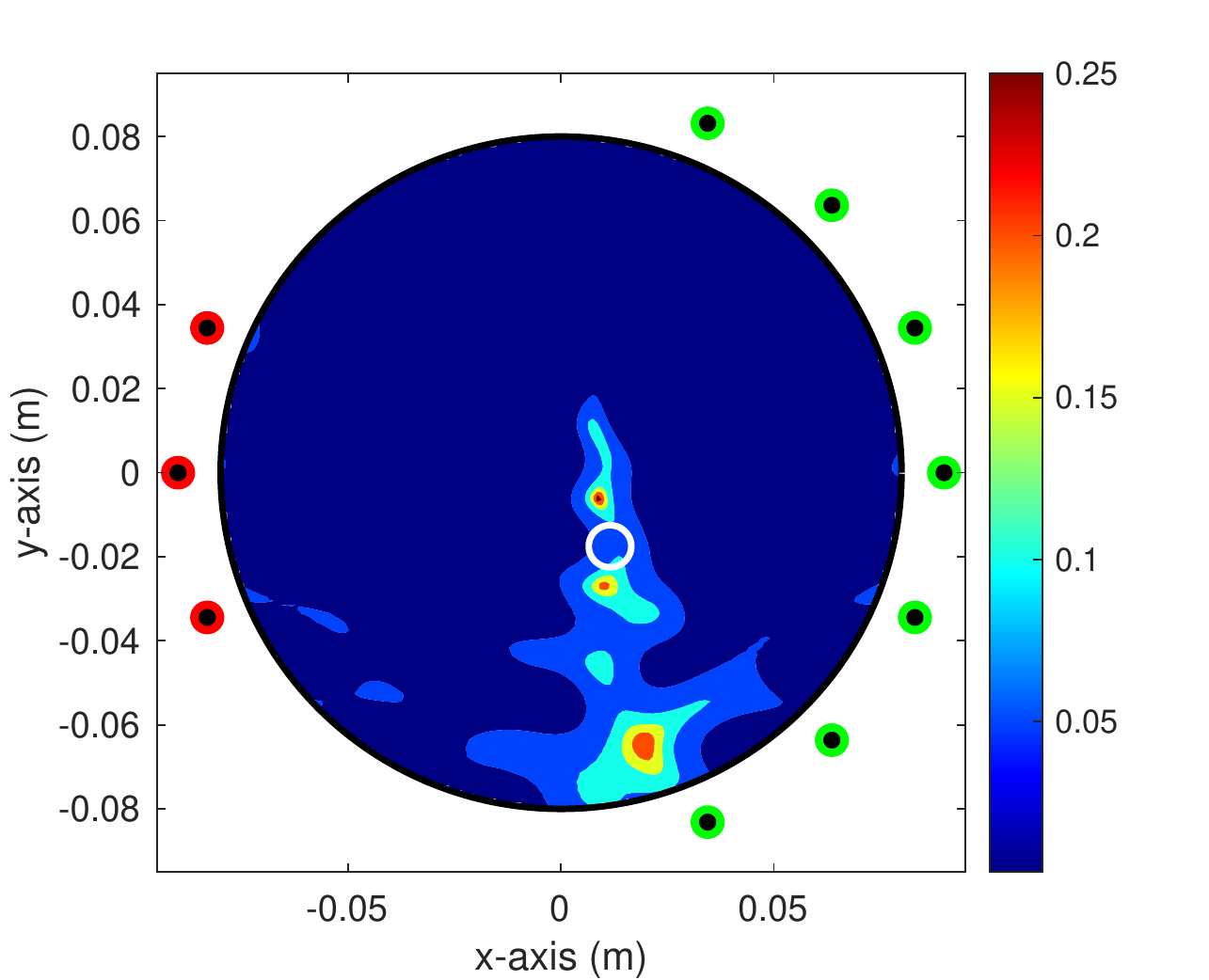}\hfill
  \includegraphics[width=0.25\textwidth]{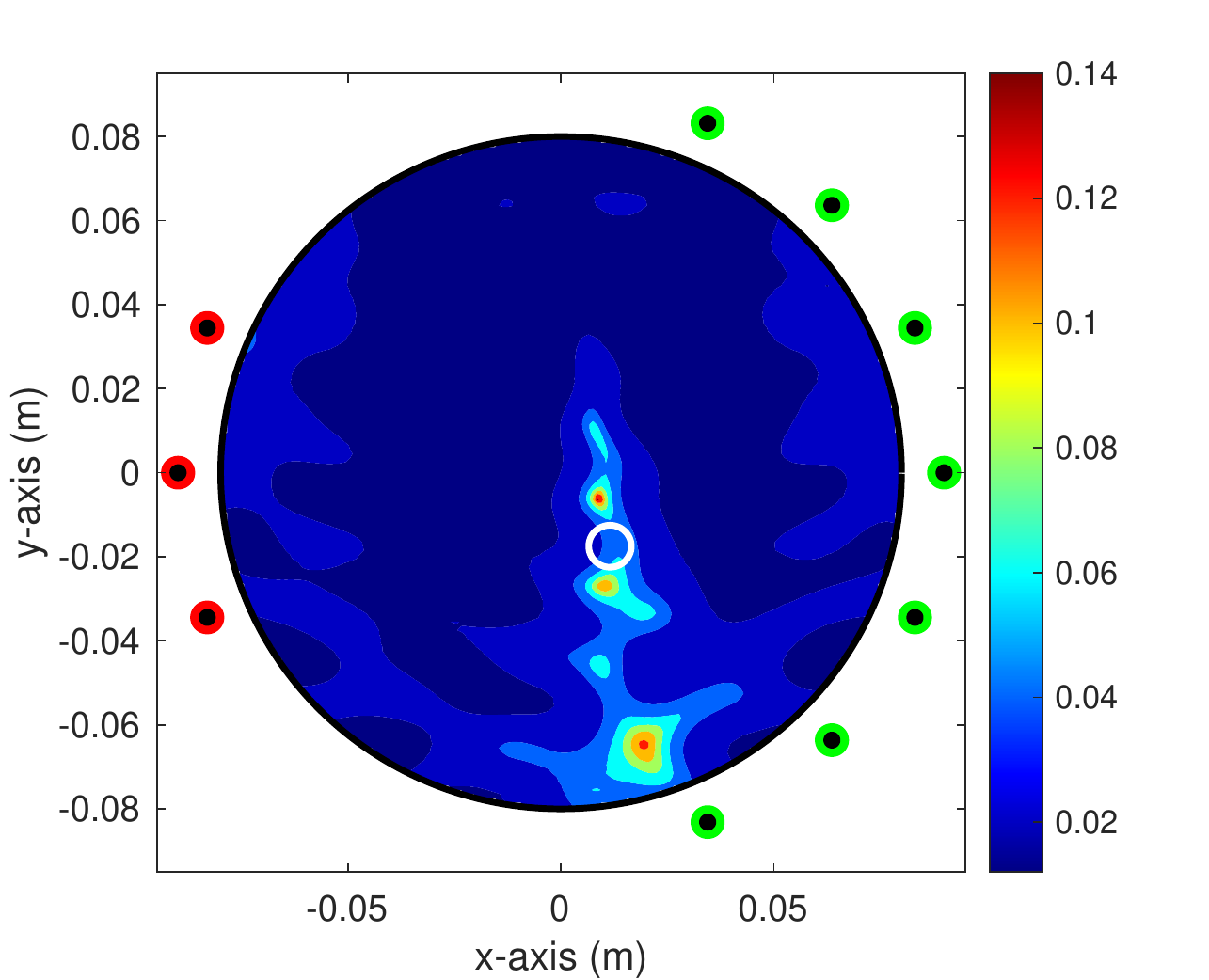}\hfill
  \includegraphics[width=0.25\textwidth]{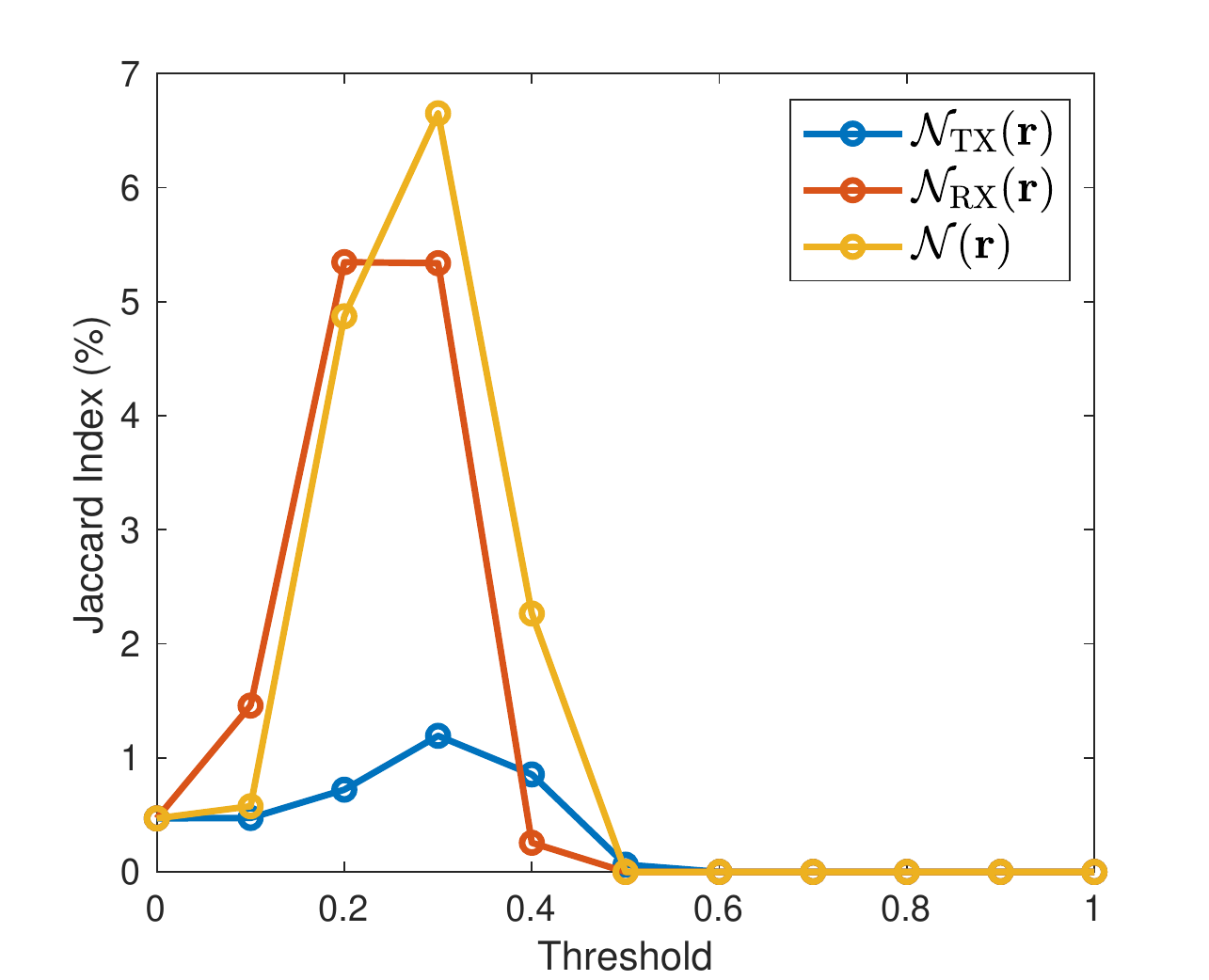}\\
  \includegraphics[width=0.25\textwidth]{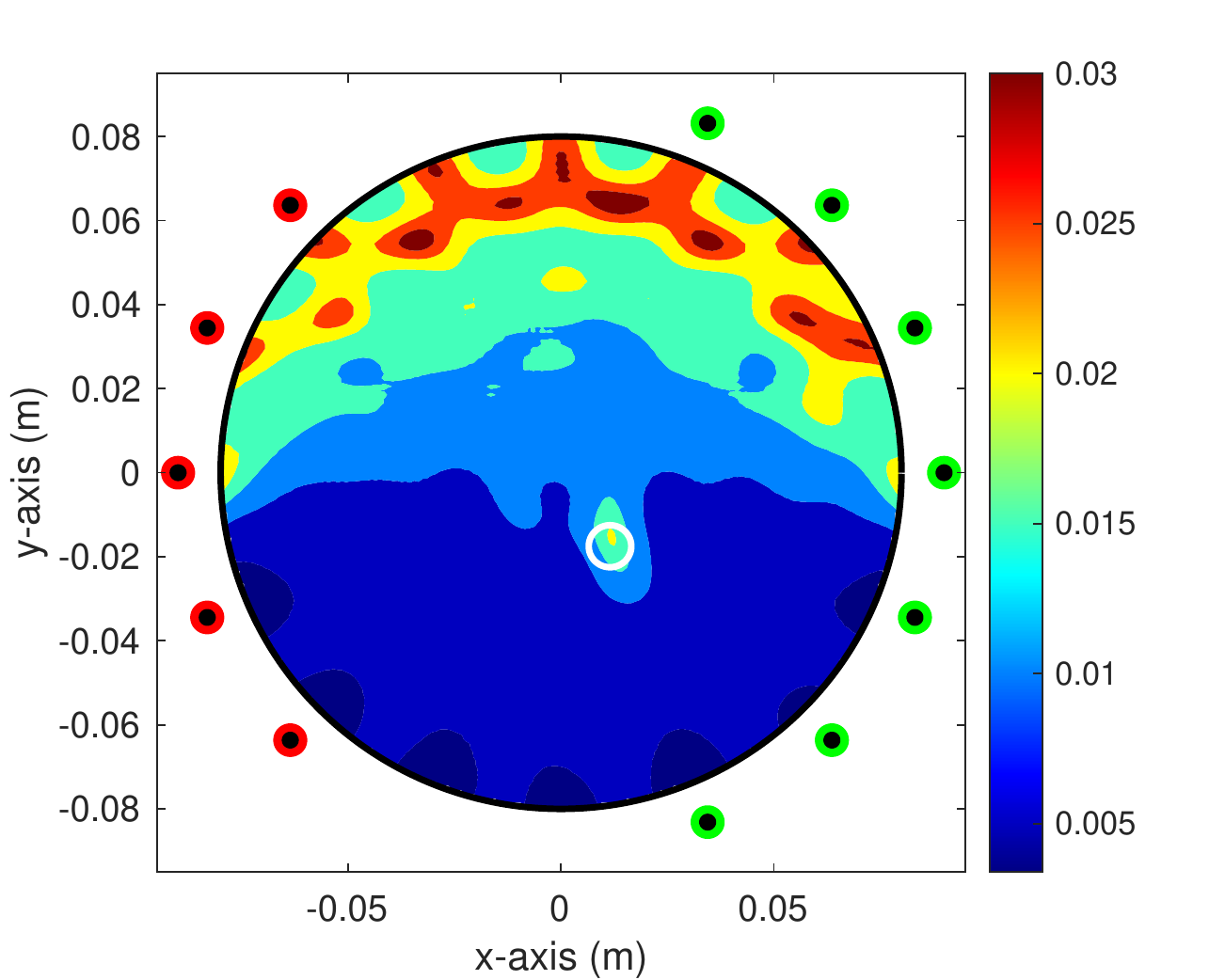}\hfill
  \includegraphics[width=0.25\textwidth]{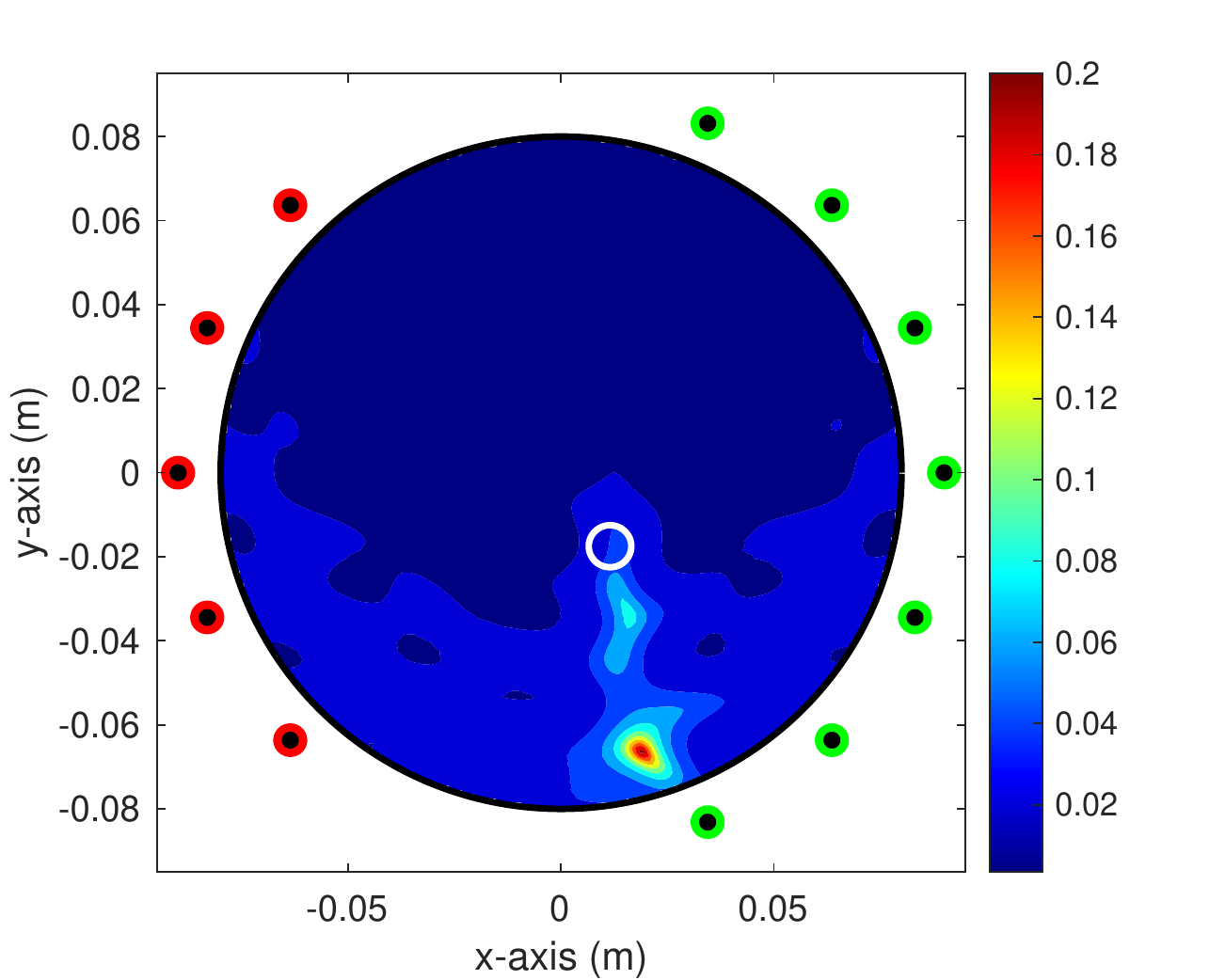}\hfill
  \includegraphics[width=0.25\textwidth]{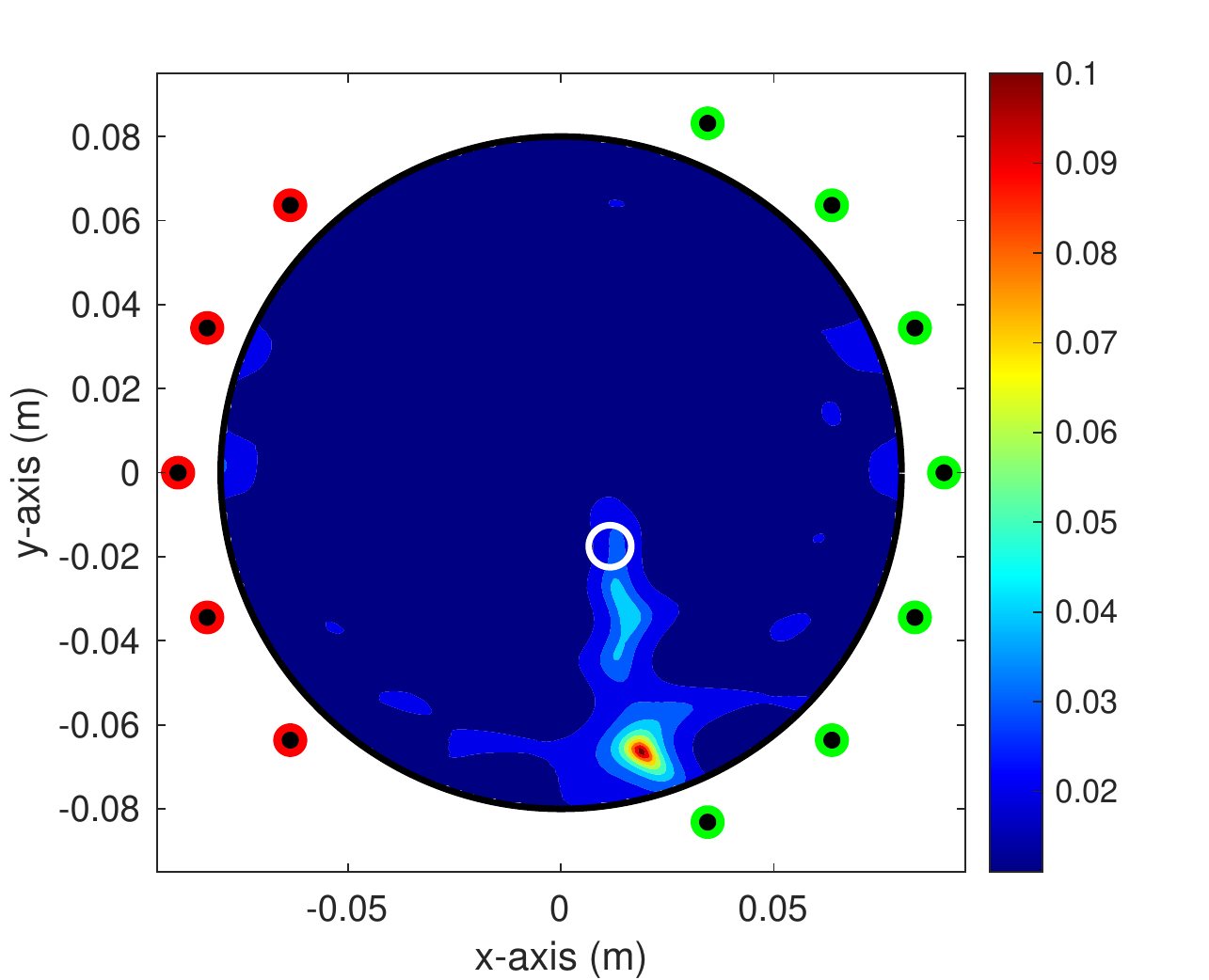}\hfill
  \includegraphics[width=0.25\textwidth]{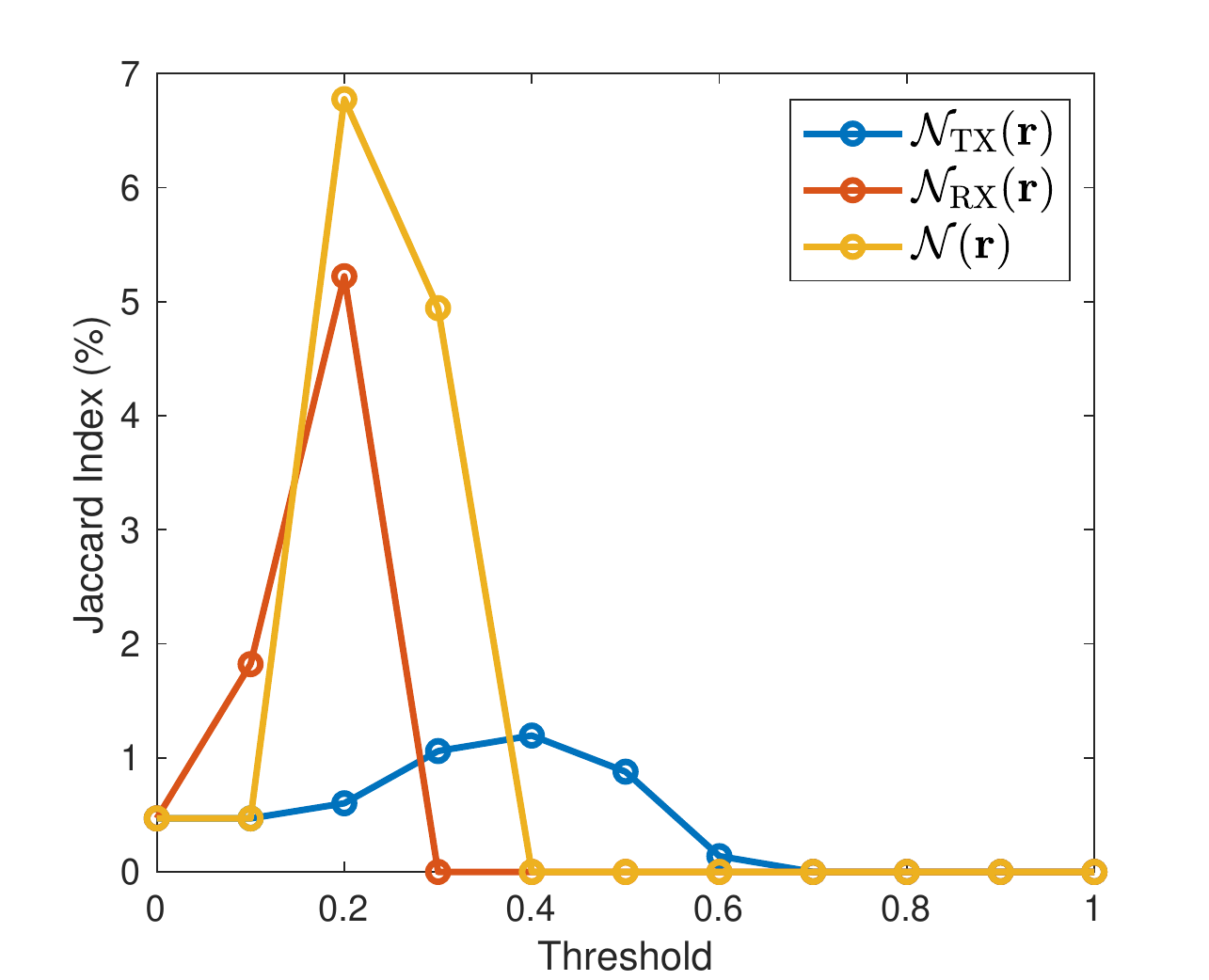}\\
  \includegraphics[width=0.25\textwidth]{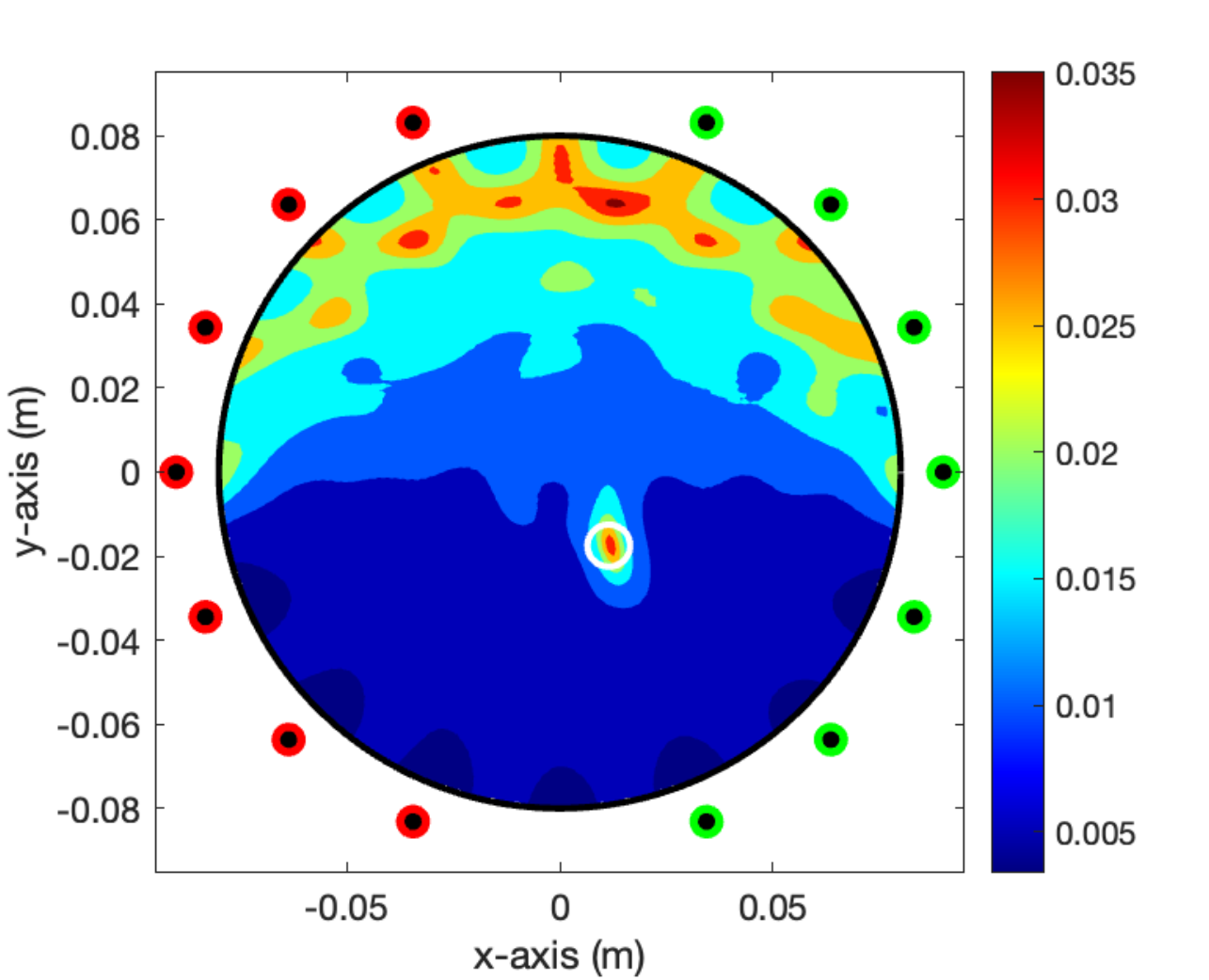}\hfill
  \includegraphics[width=0.25\textwidth]{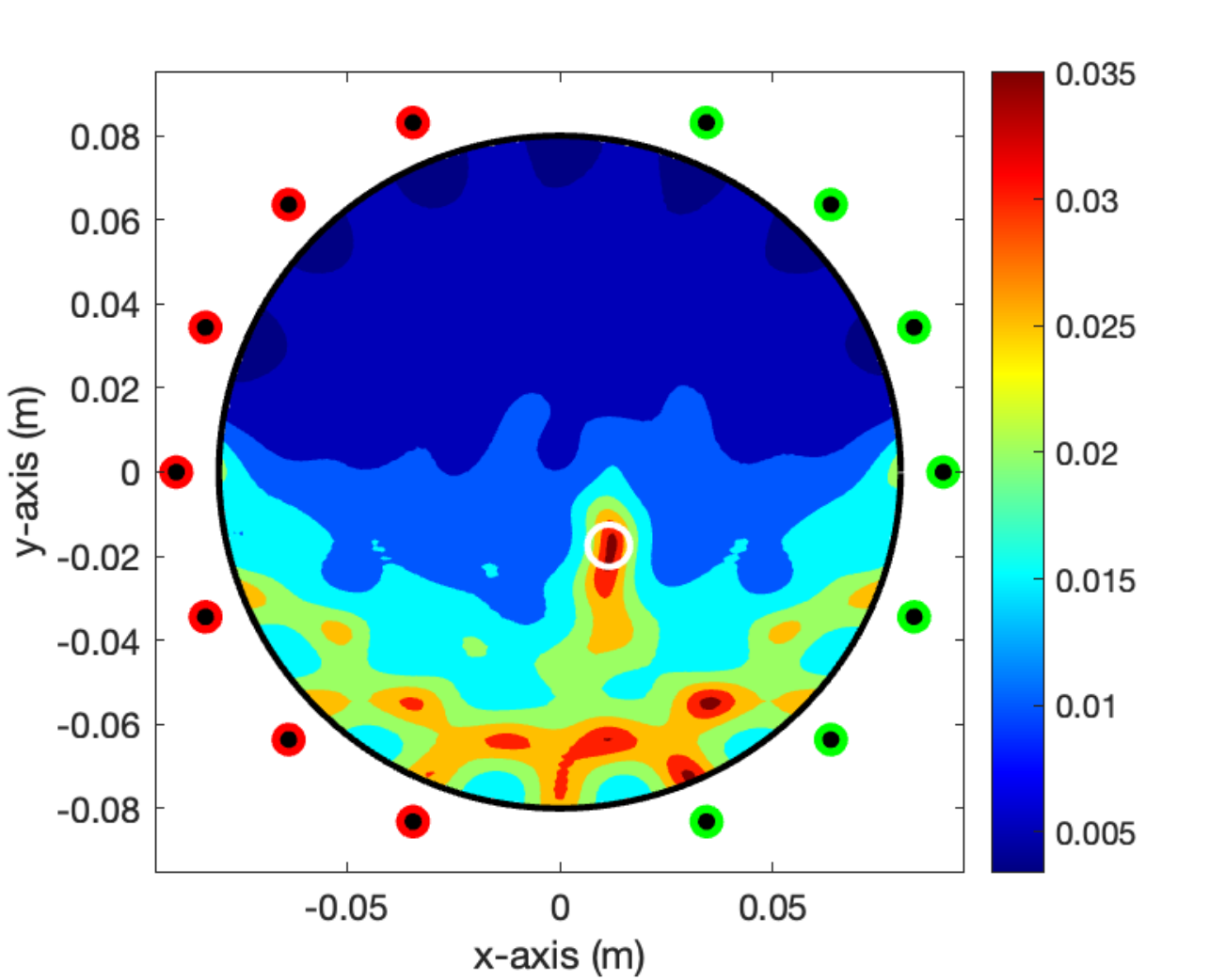}\hfill
  \includegraphics[width=0.25\textwidth]{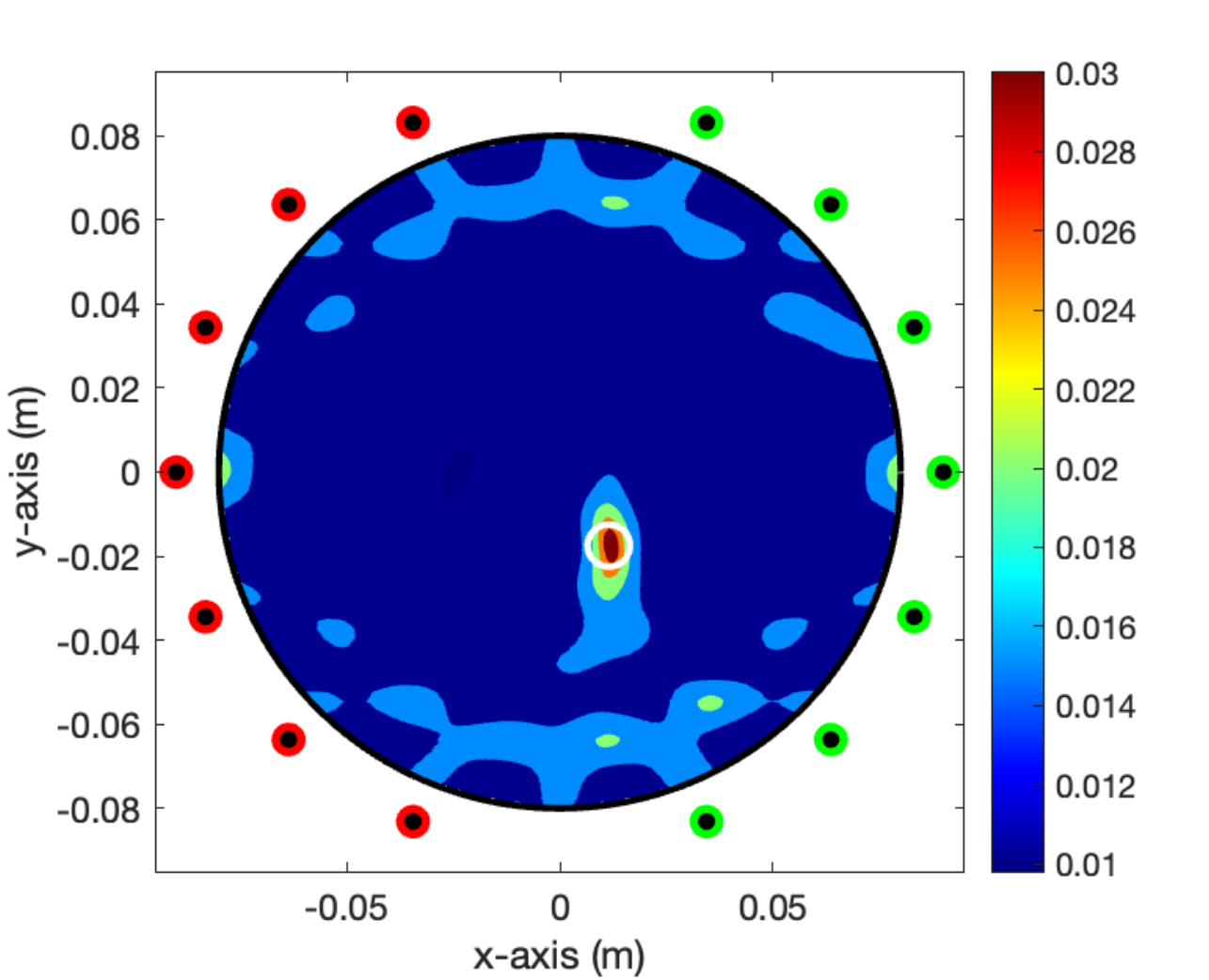}\hfill
  \includegraphics[width=0.25\textwidth]{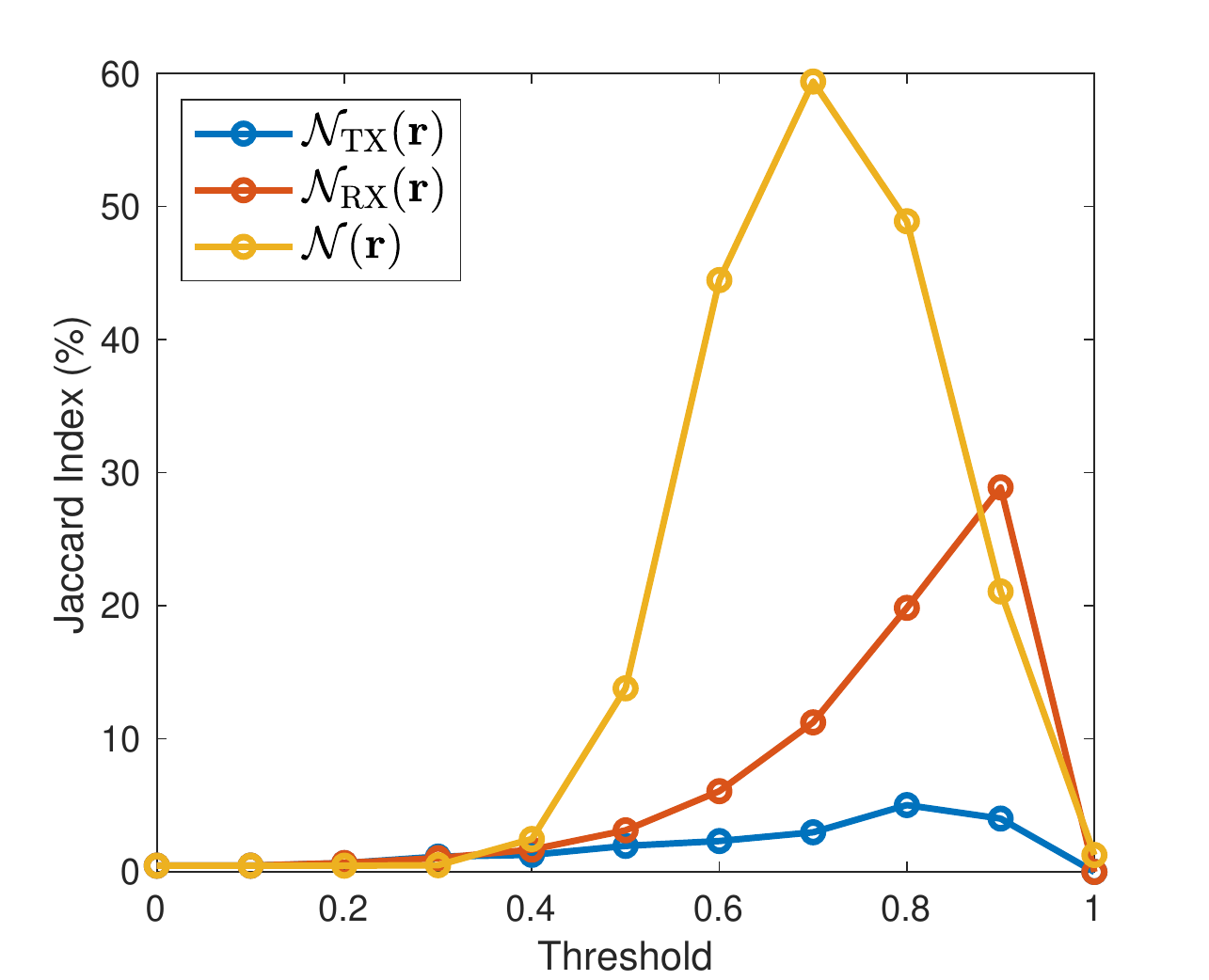}\\
  \includegraphics[width=0.25\textwidth]{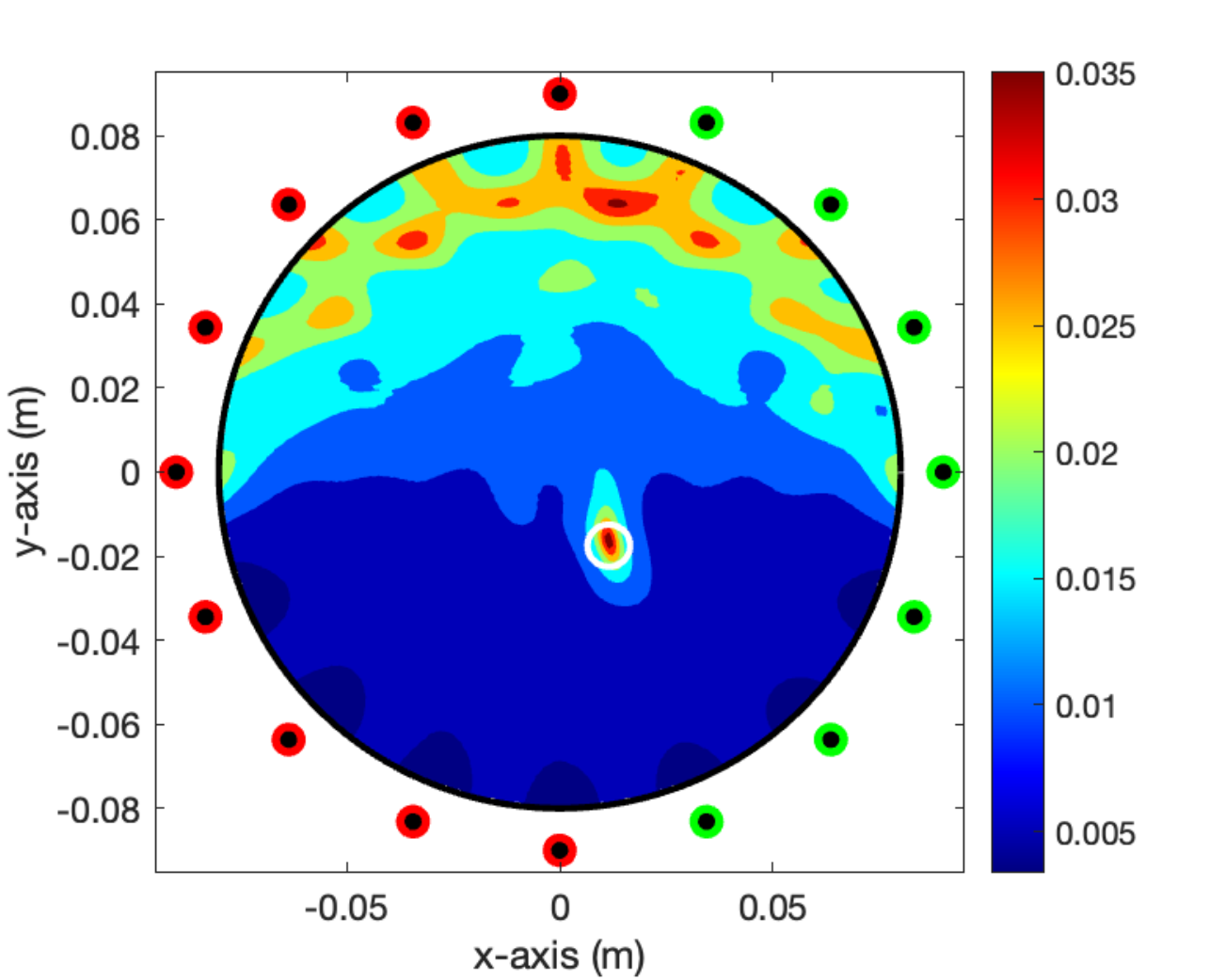}\hfill
  \includegraphics[width=0.25\textwidth]{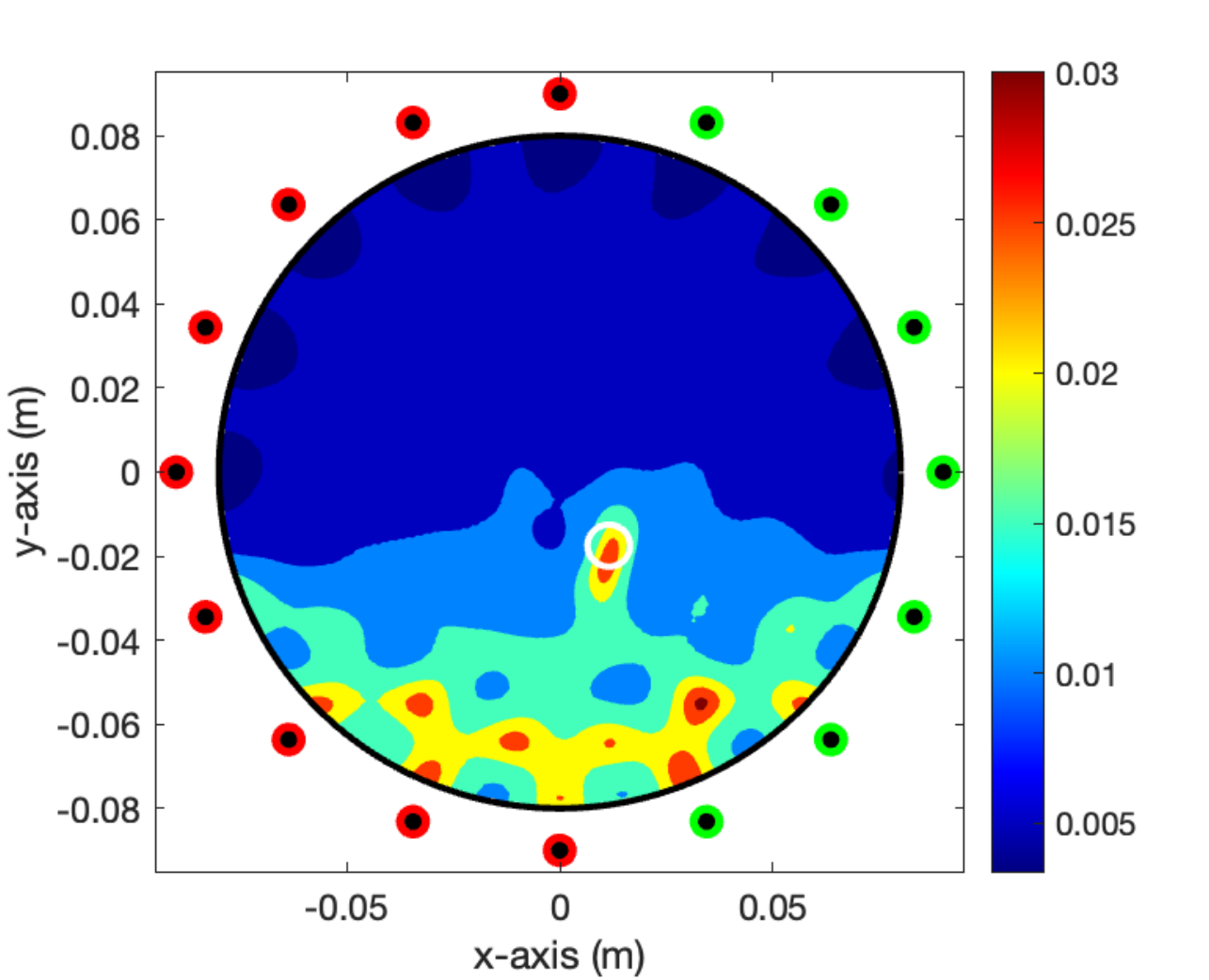}\hfill
  \includegraphics[width=0.25\textwidth]{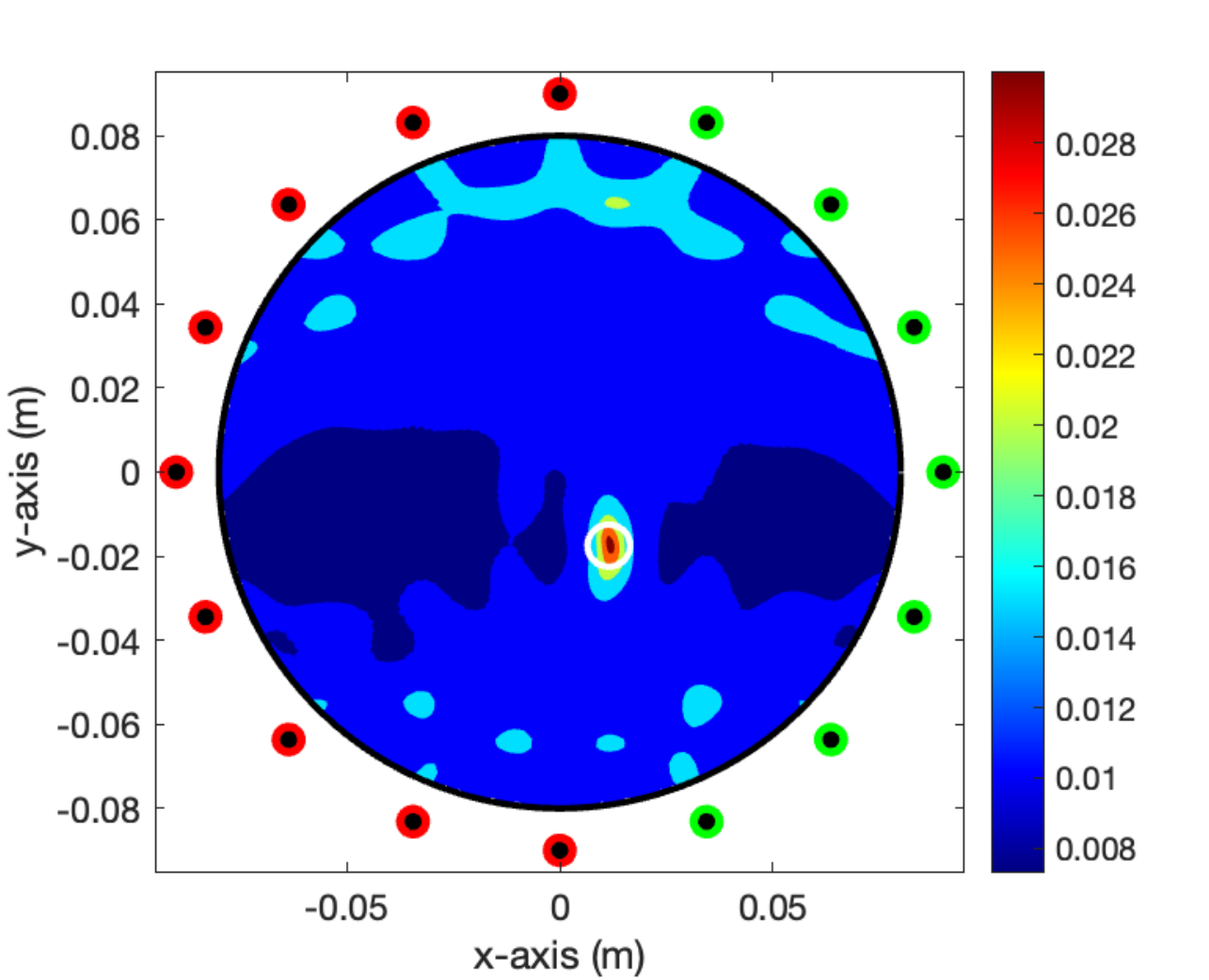}\hfill
  \includegraphics[width=0.25\textwidth]{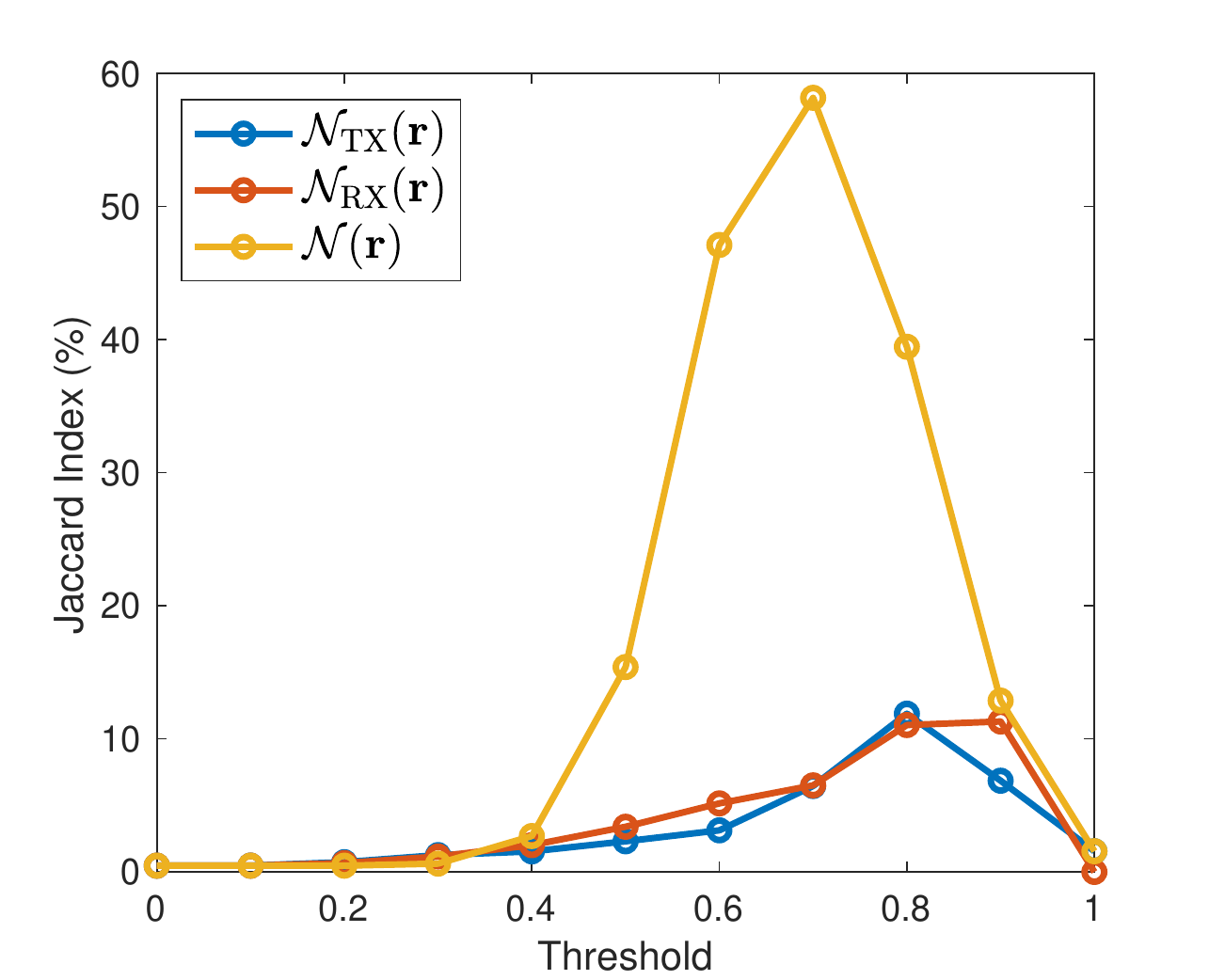}
  \caption{\label{ResultR2}(Example \ref{exR1}) Maps of $\mathfrak{F}_{\tx}(\mr)$ (first column), $\mathfrak{F}_{\rx}(\mr)$ (second column), $\mathfrak{F}(\mr)$ (third column), and Jaccard index (fourth column). Green and red colored circles describe the location of transmitters and receivers, respectively.}
\end{figure}

\begin{figure}[h]
  \centering
  \includegraphics[width=0.25\textwidth]{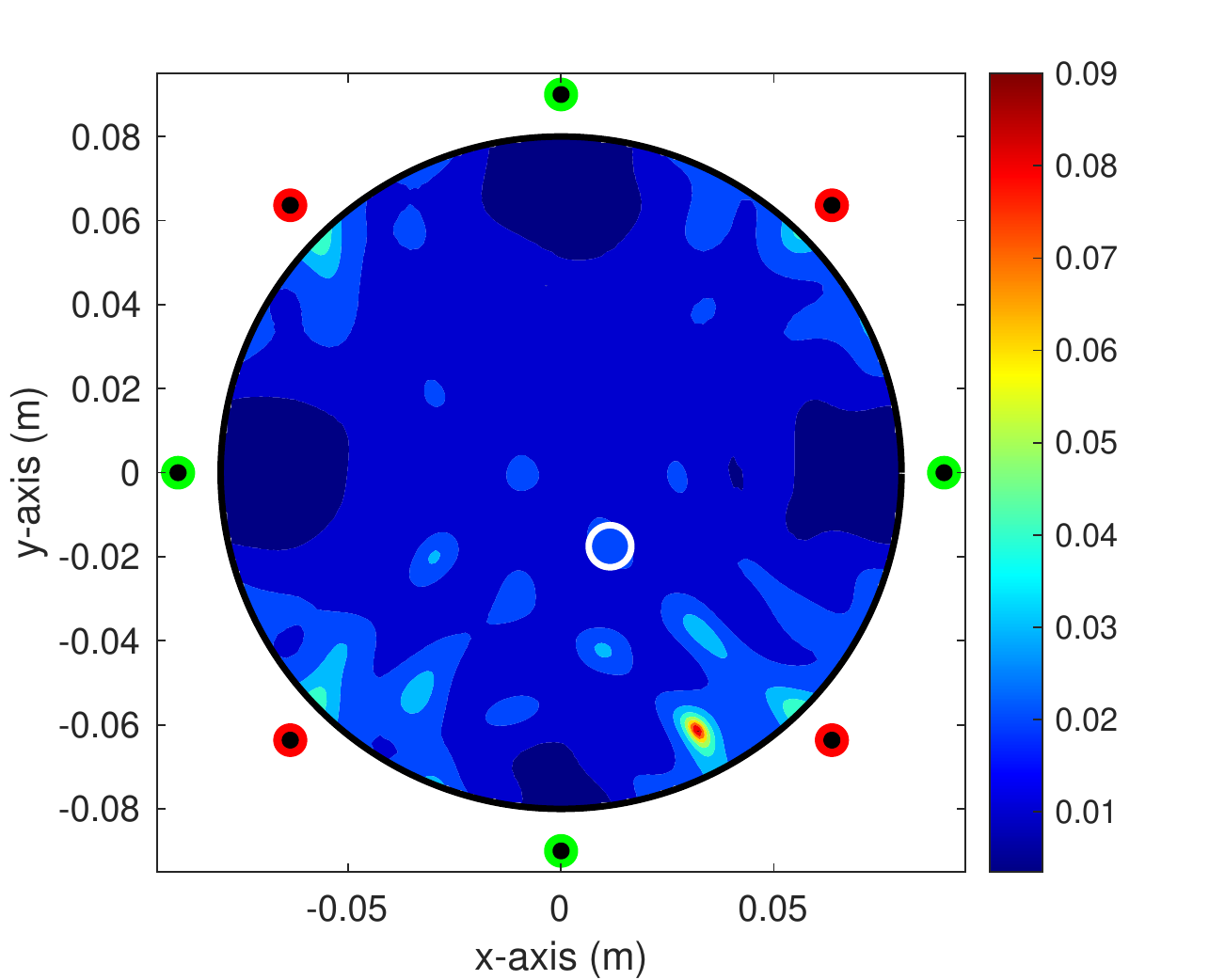}\hfill
  \includegraphics[width=0.25\textwidth]{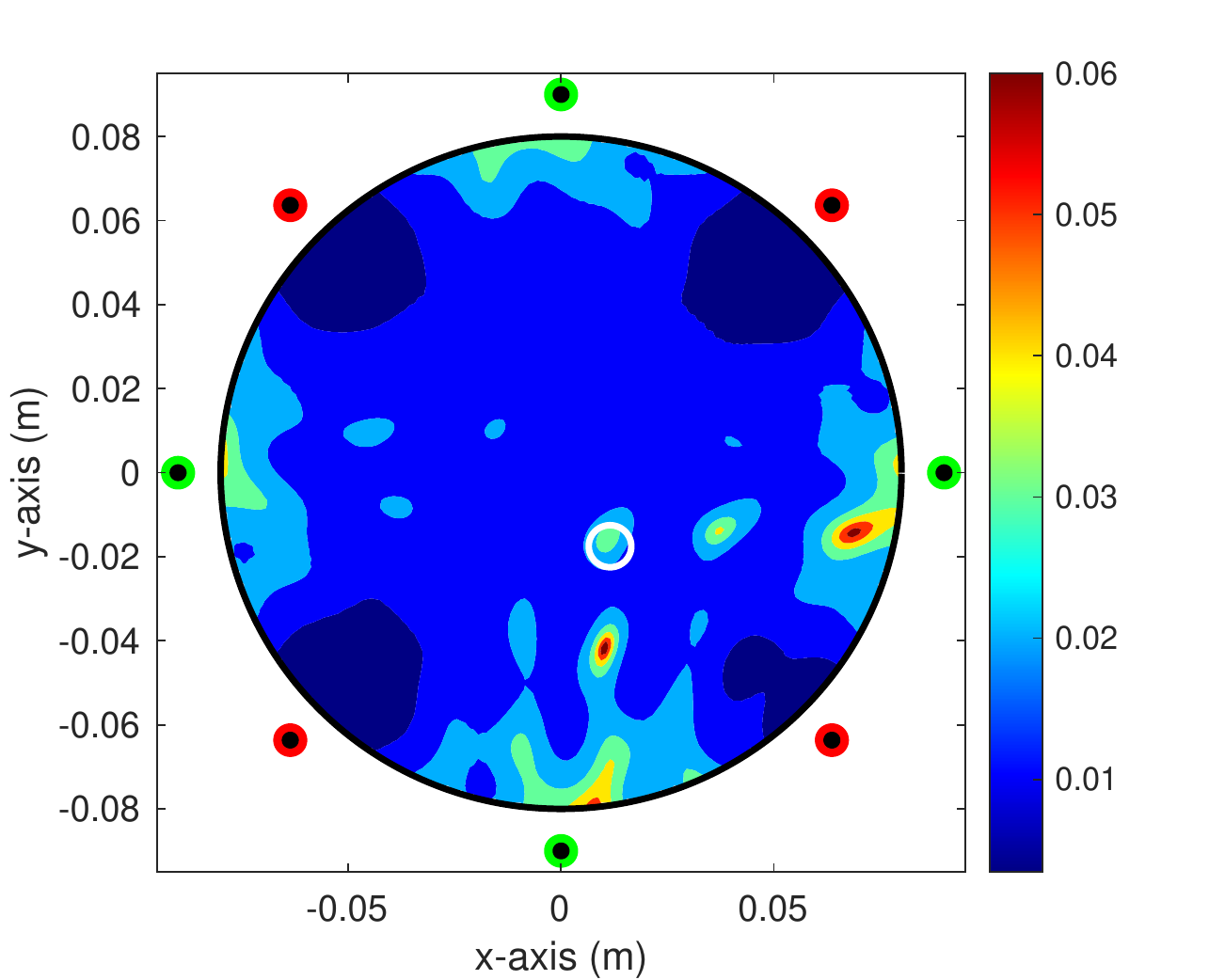}\hfill
  \includegraphics[width=0.25\textwidth]{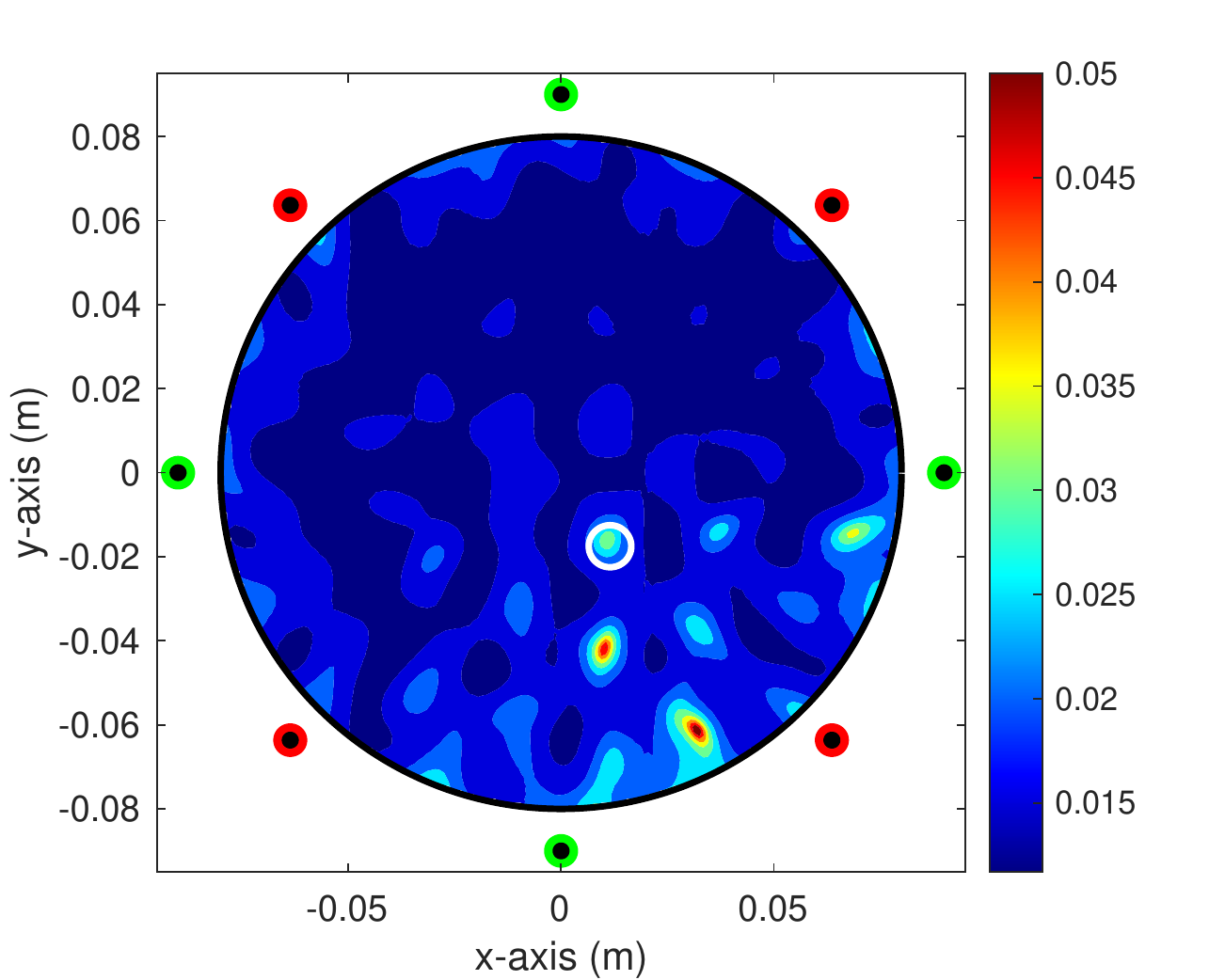}\hfill
  \includegraphics[width=0.25\textwidth]{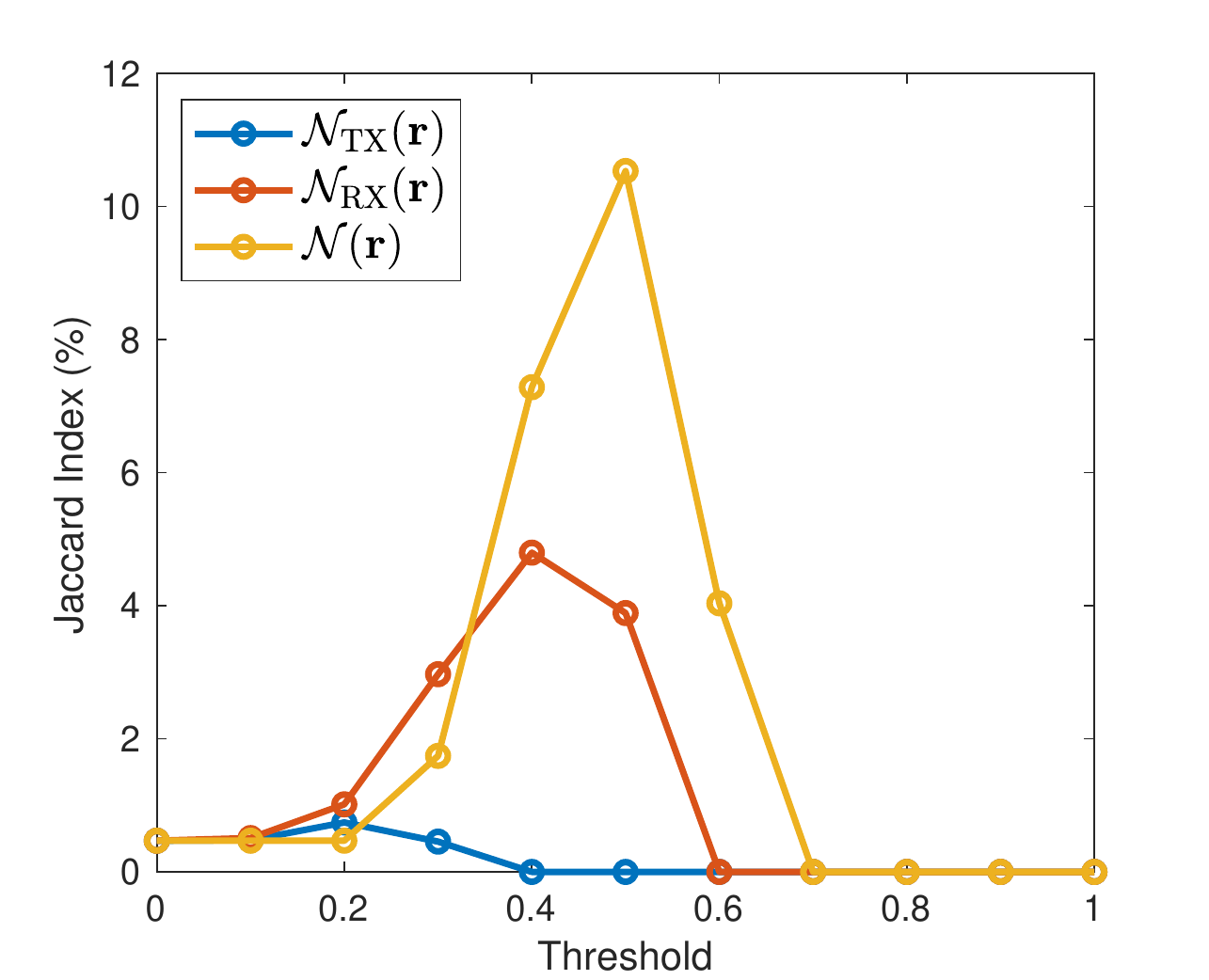}\\
  \includegraphics[width=0.25\textwidth]{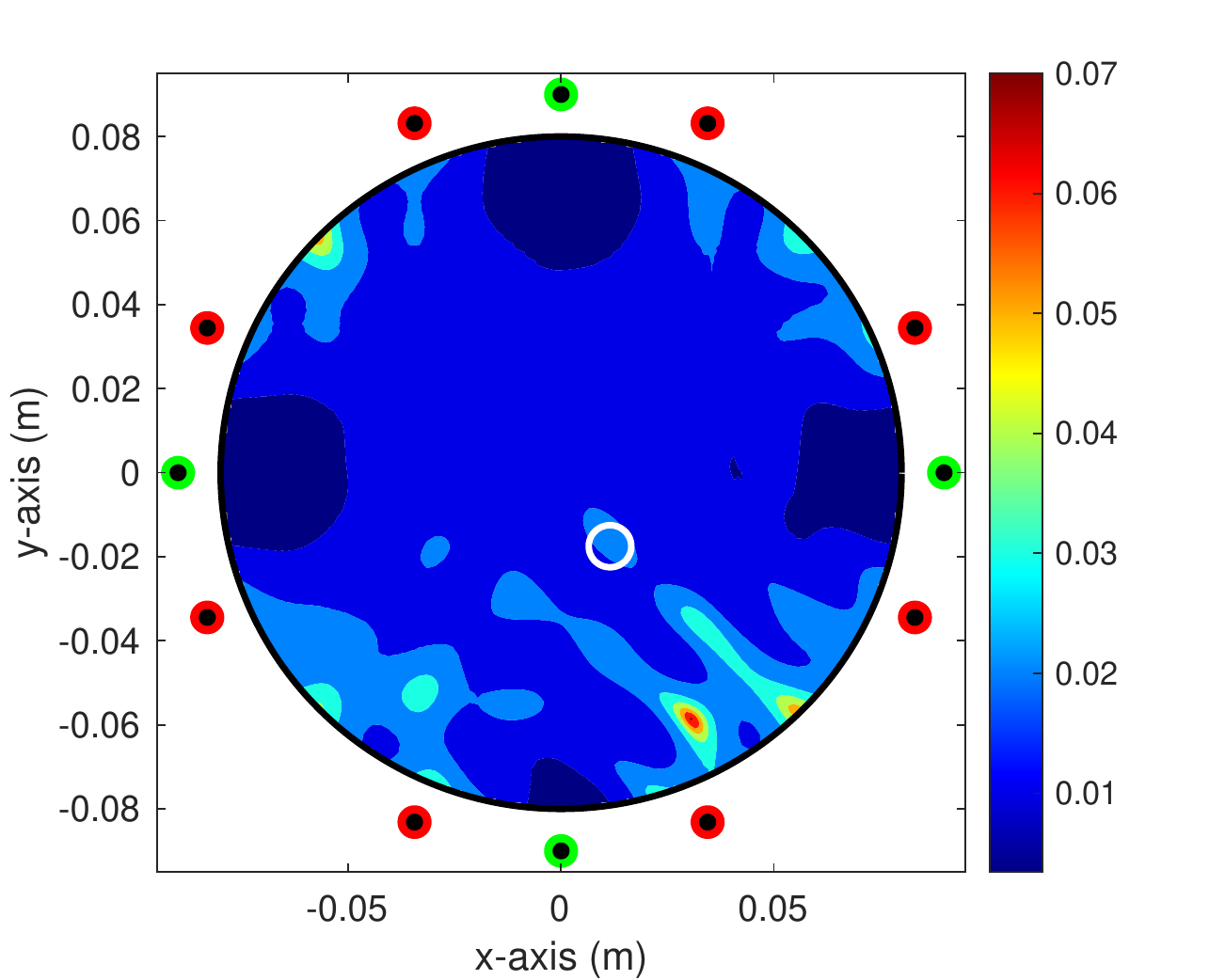}\hfill
  \includegraphics[width=0.25\textwidth]{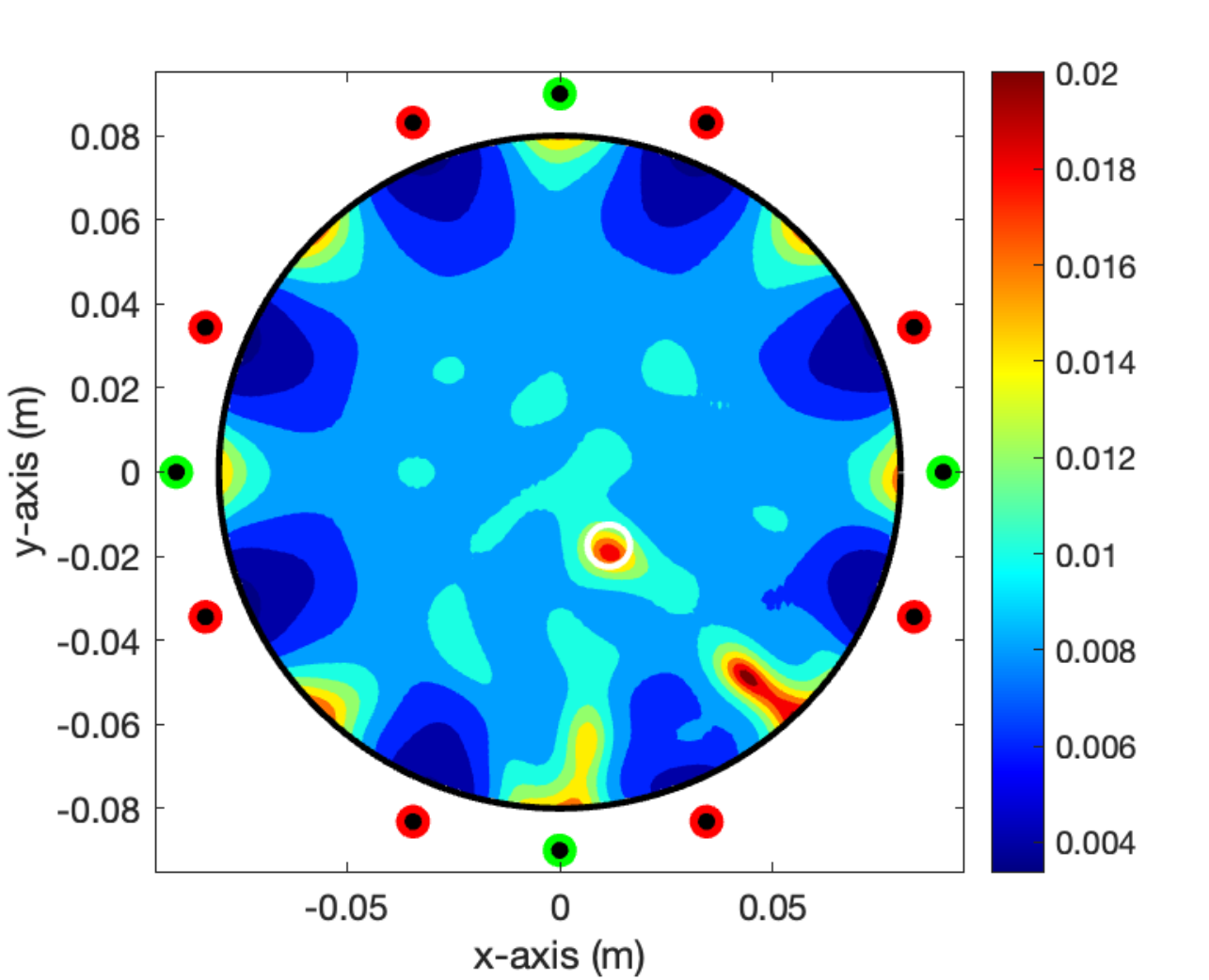}\hfill
  \includegraphics[width=0.25\textwidth]{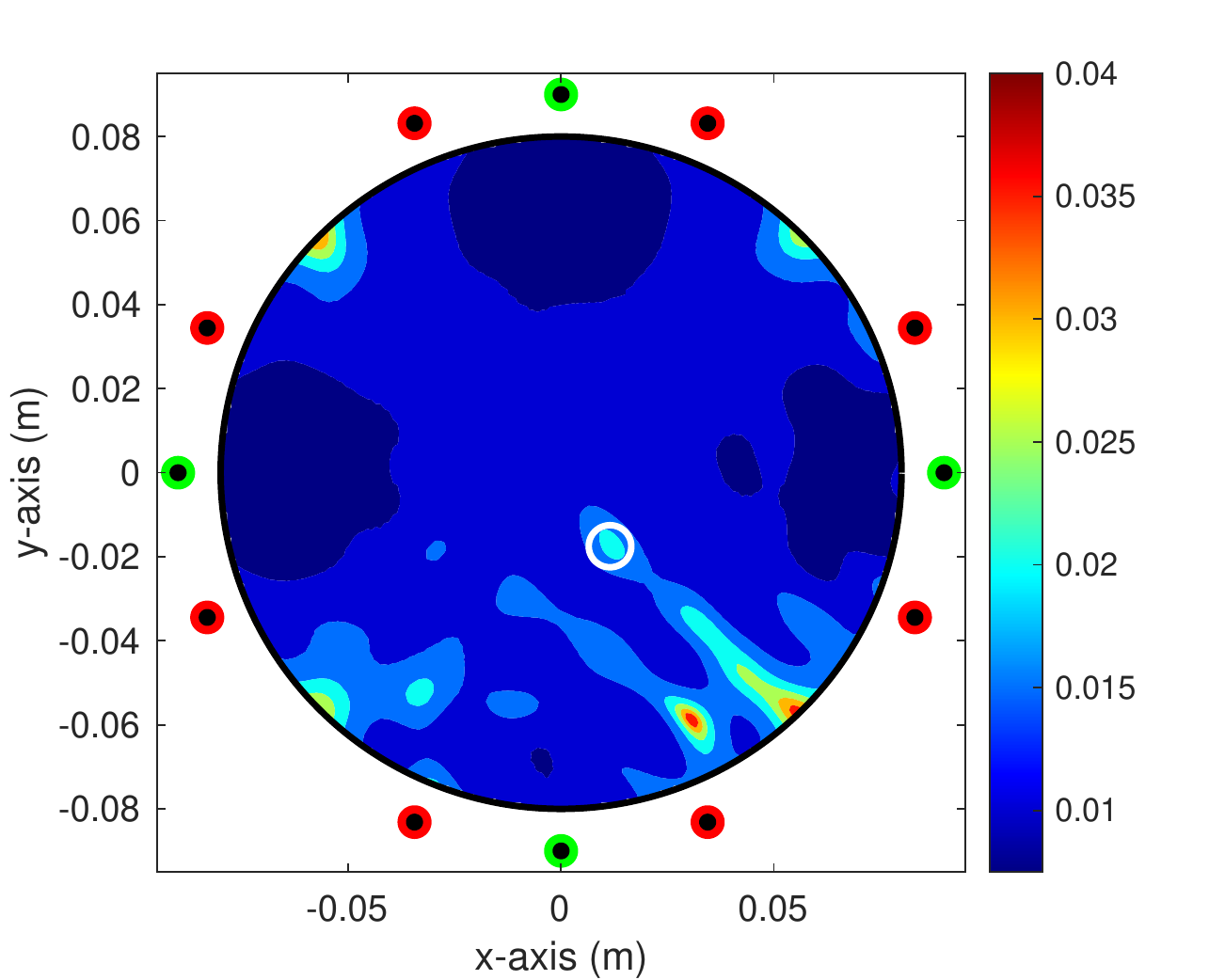}\hfill
  \includegraphics[width=0.25\textwidth]{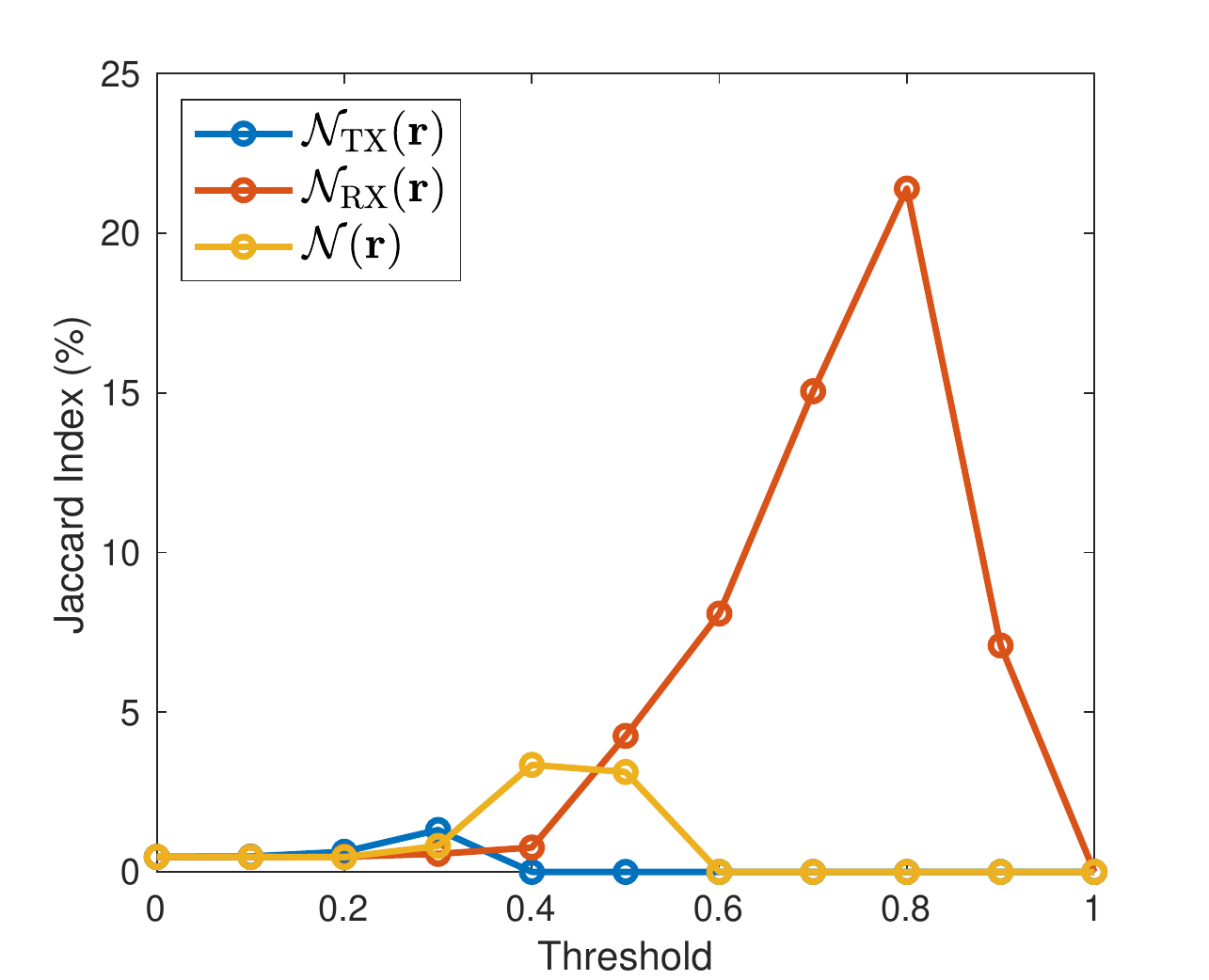}\\
  \includegraphics[width=0.25\textwidth]{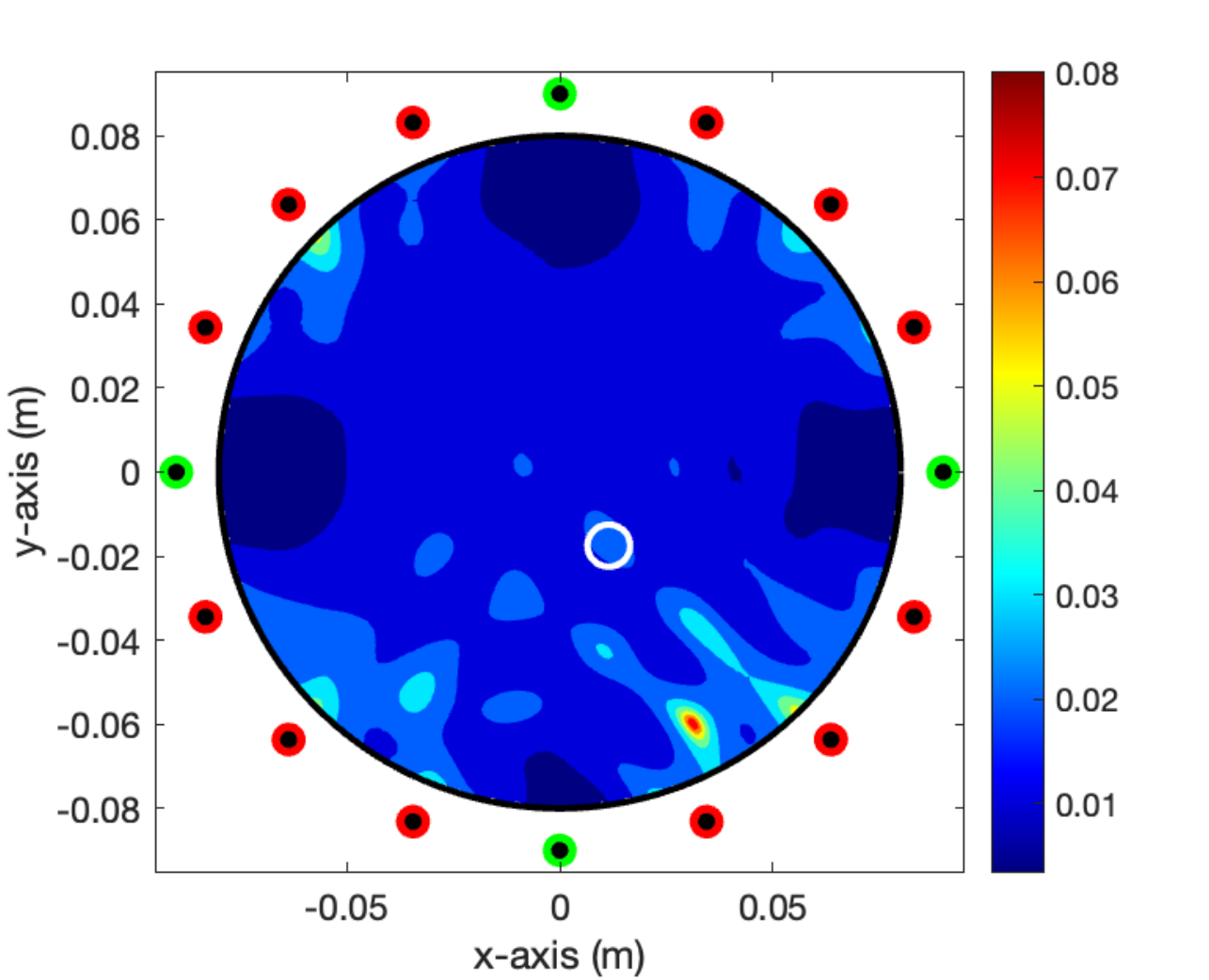}\hfill
  \includegraphics[width=0.25\textwidth]{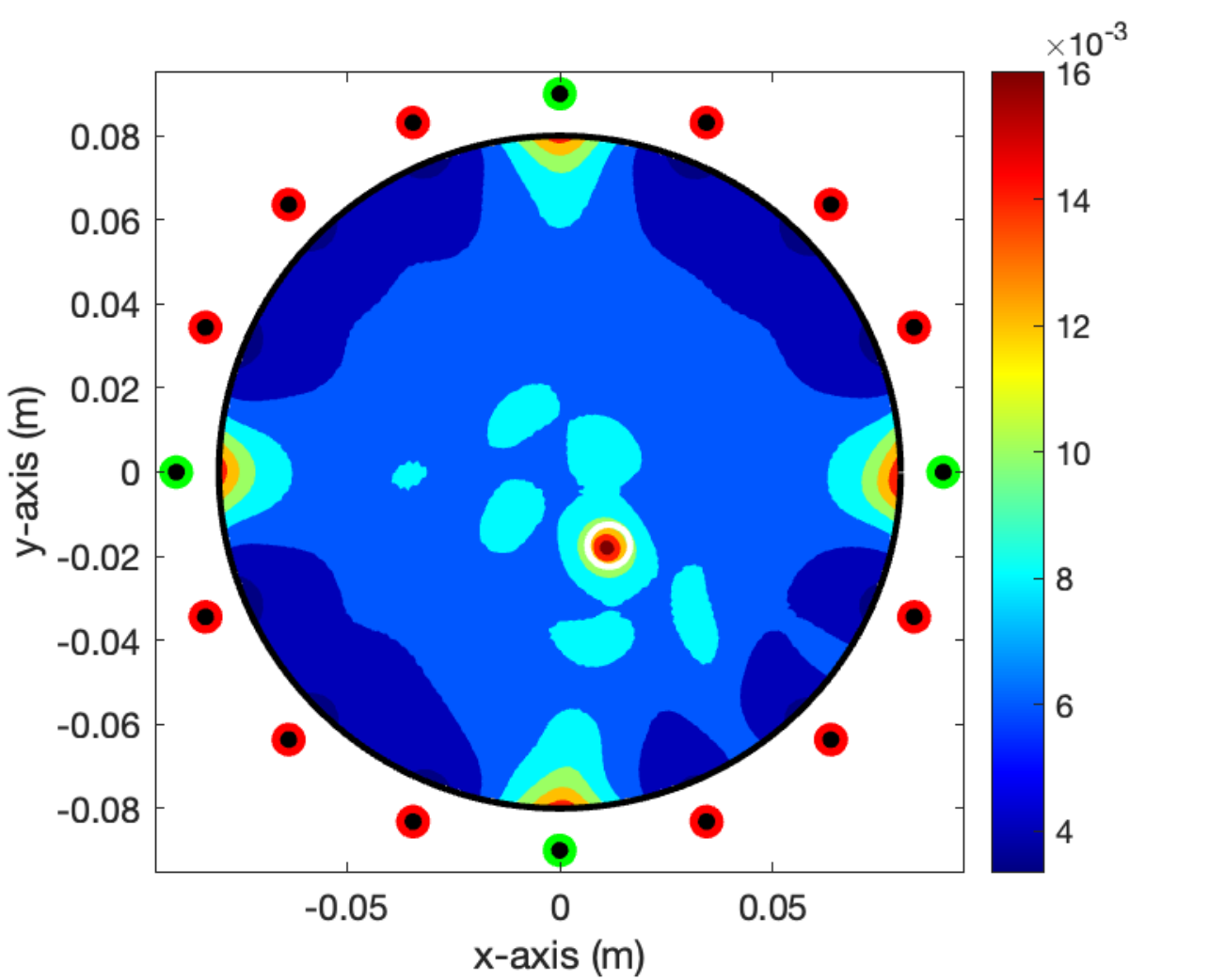}\hfill
  \includegraphics[width=0.25\textwidth]{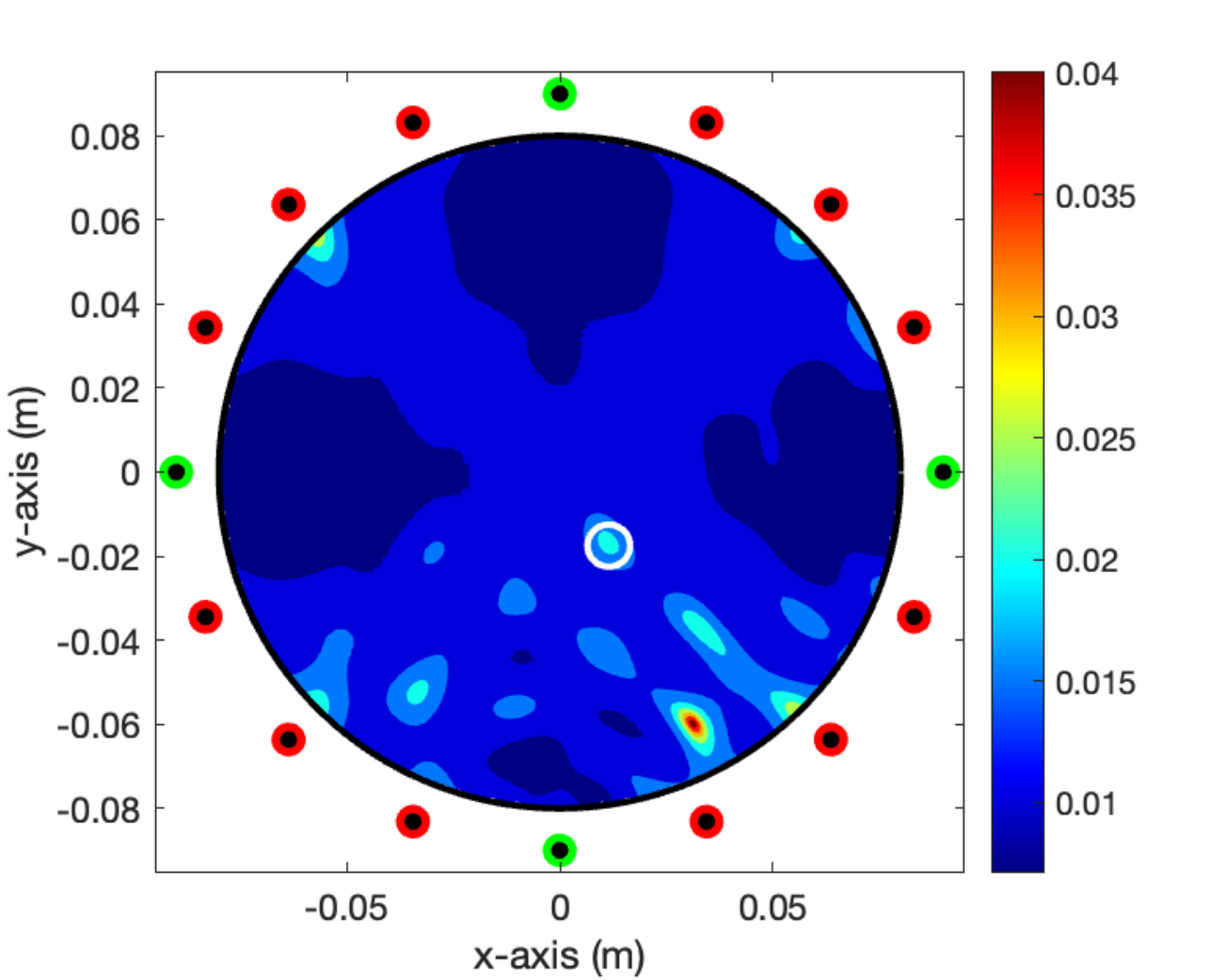}\hfill
  \includegraphics[width=0.25\textwidth]{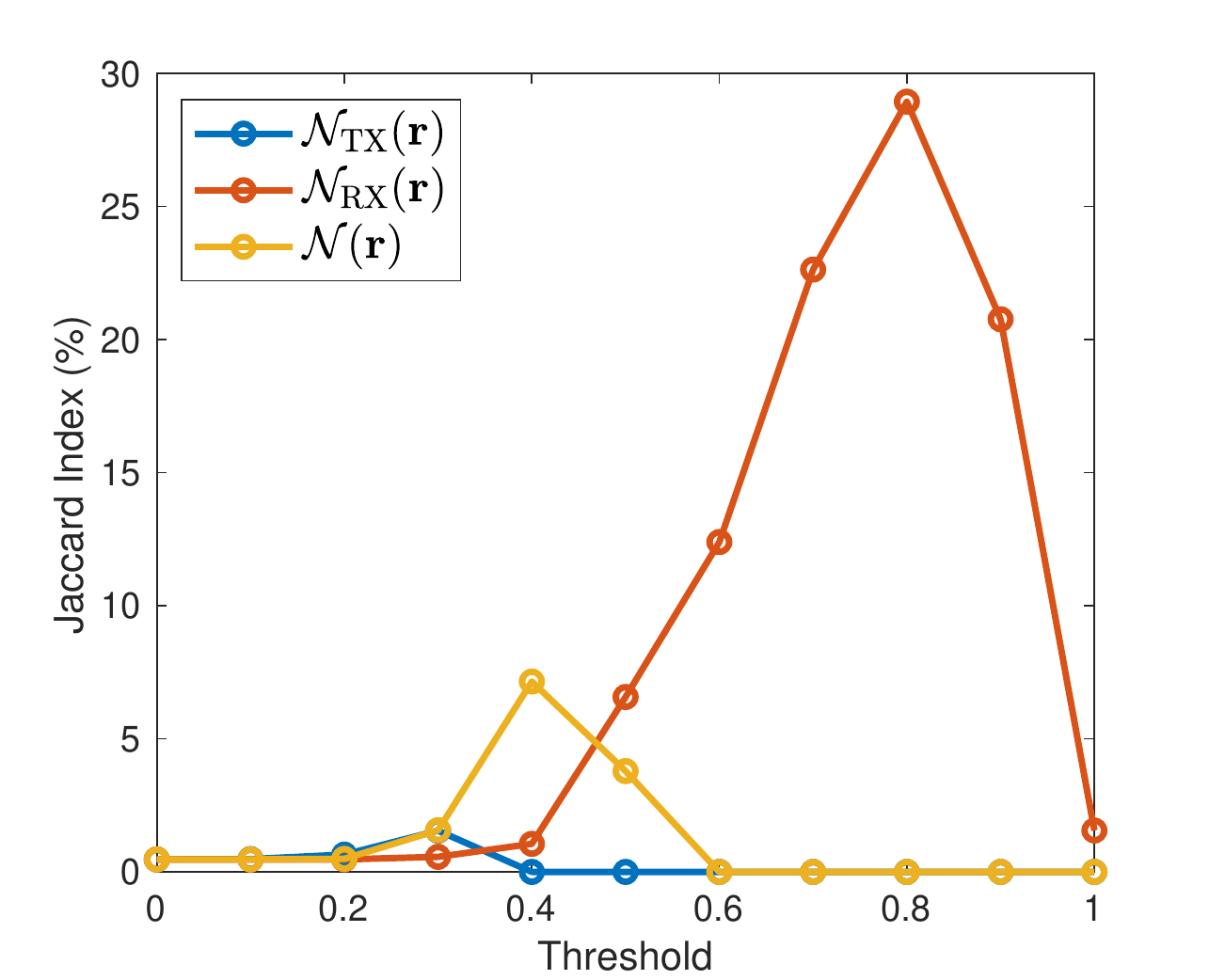}\\
  \includegraphics[width=0.25\textwidth]{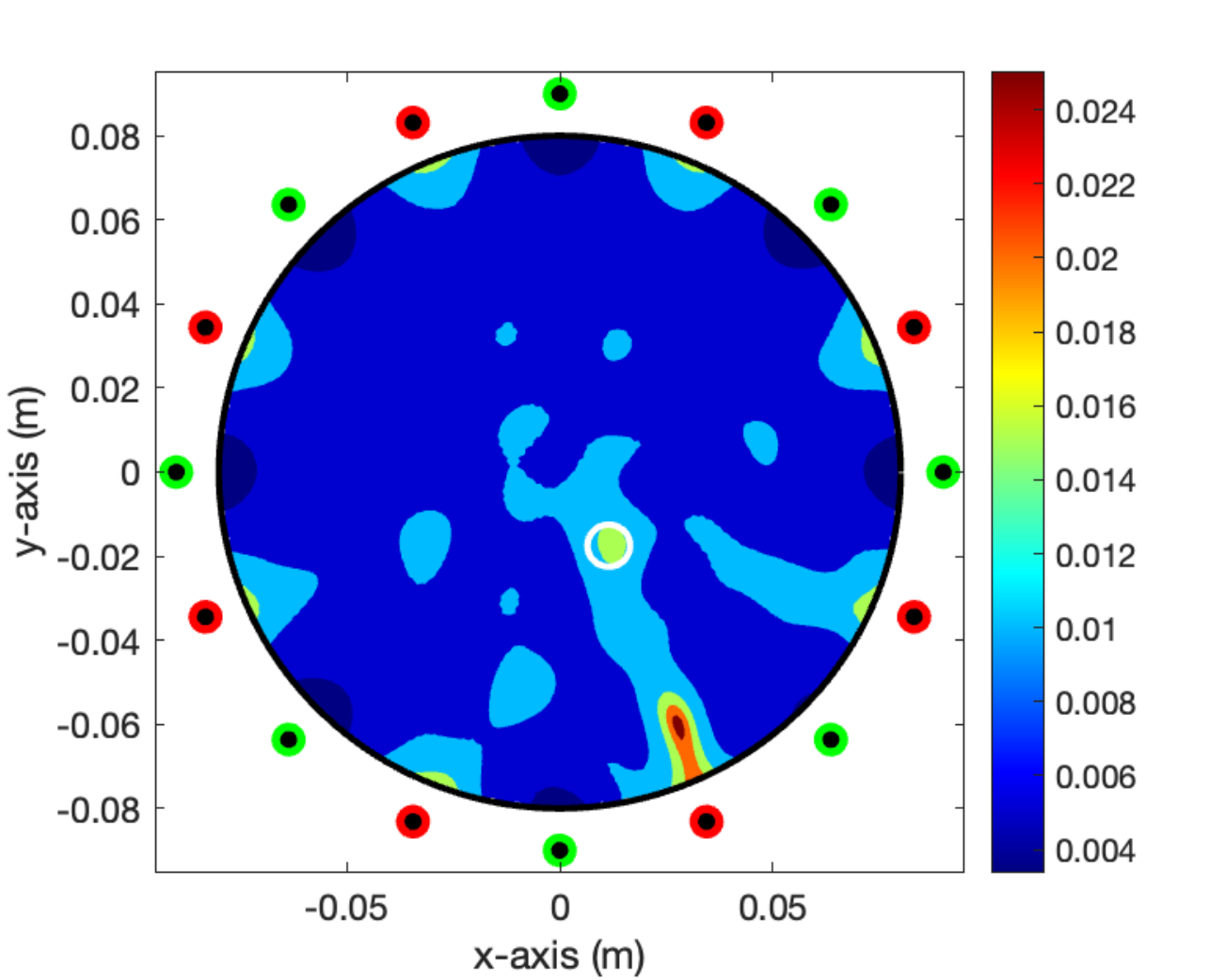}\hfill
  \includegraphics[width=0.25\textwidth]{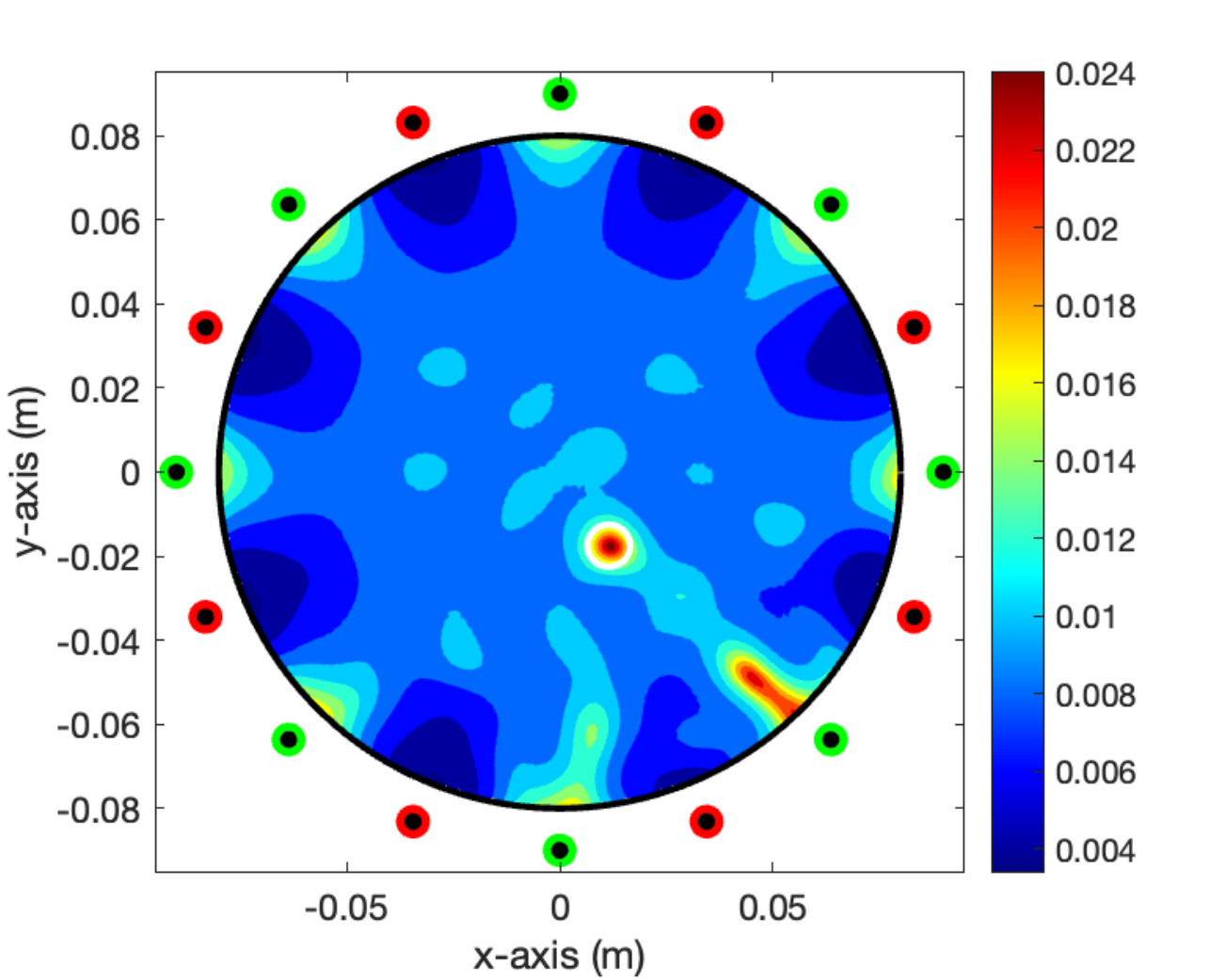}\hfill
  \includegraphics[width=0.25\textwidth]{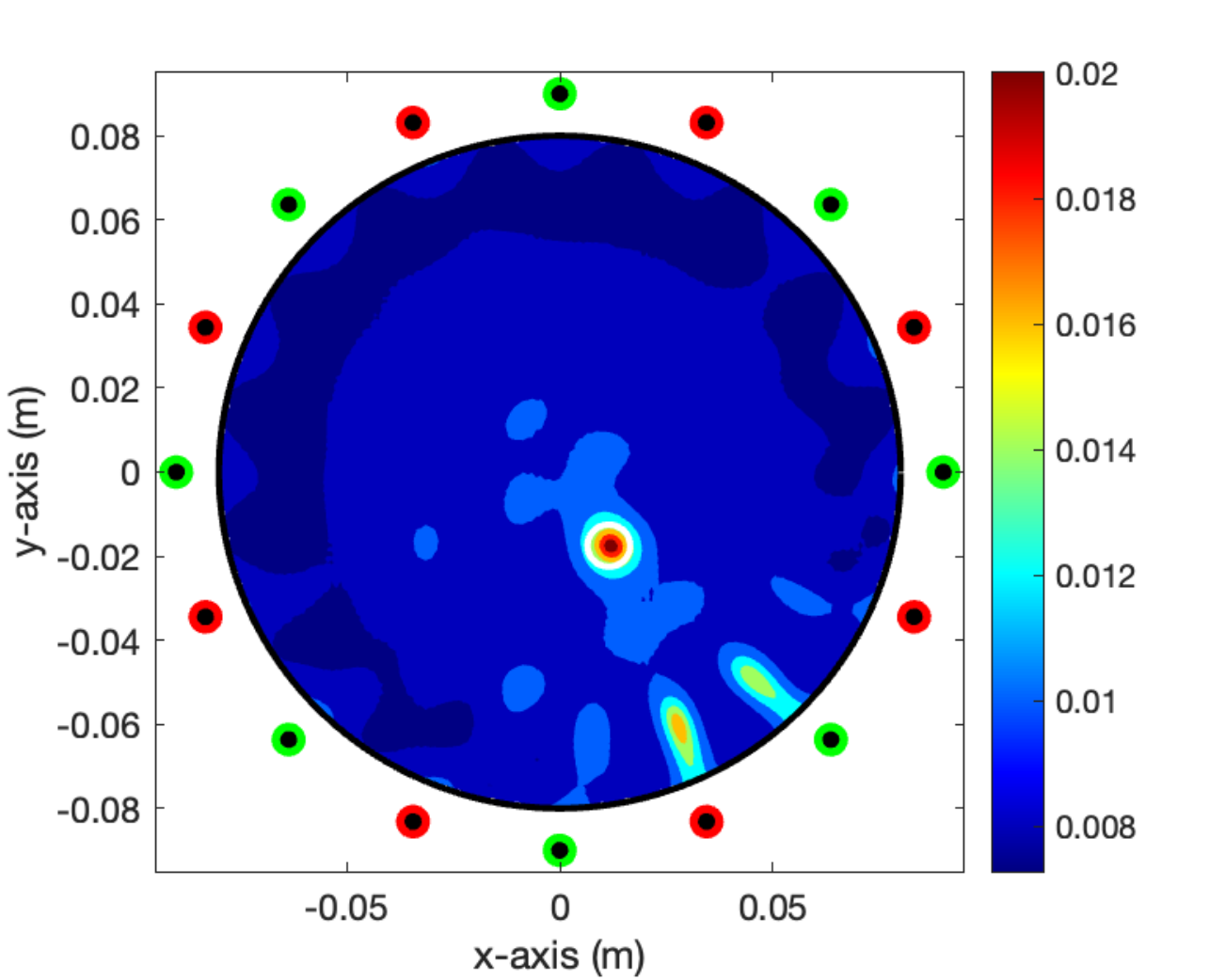}\hfill
  \includegraphics[width=0.25\textwidth]{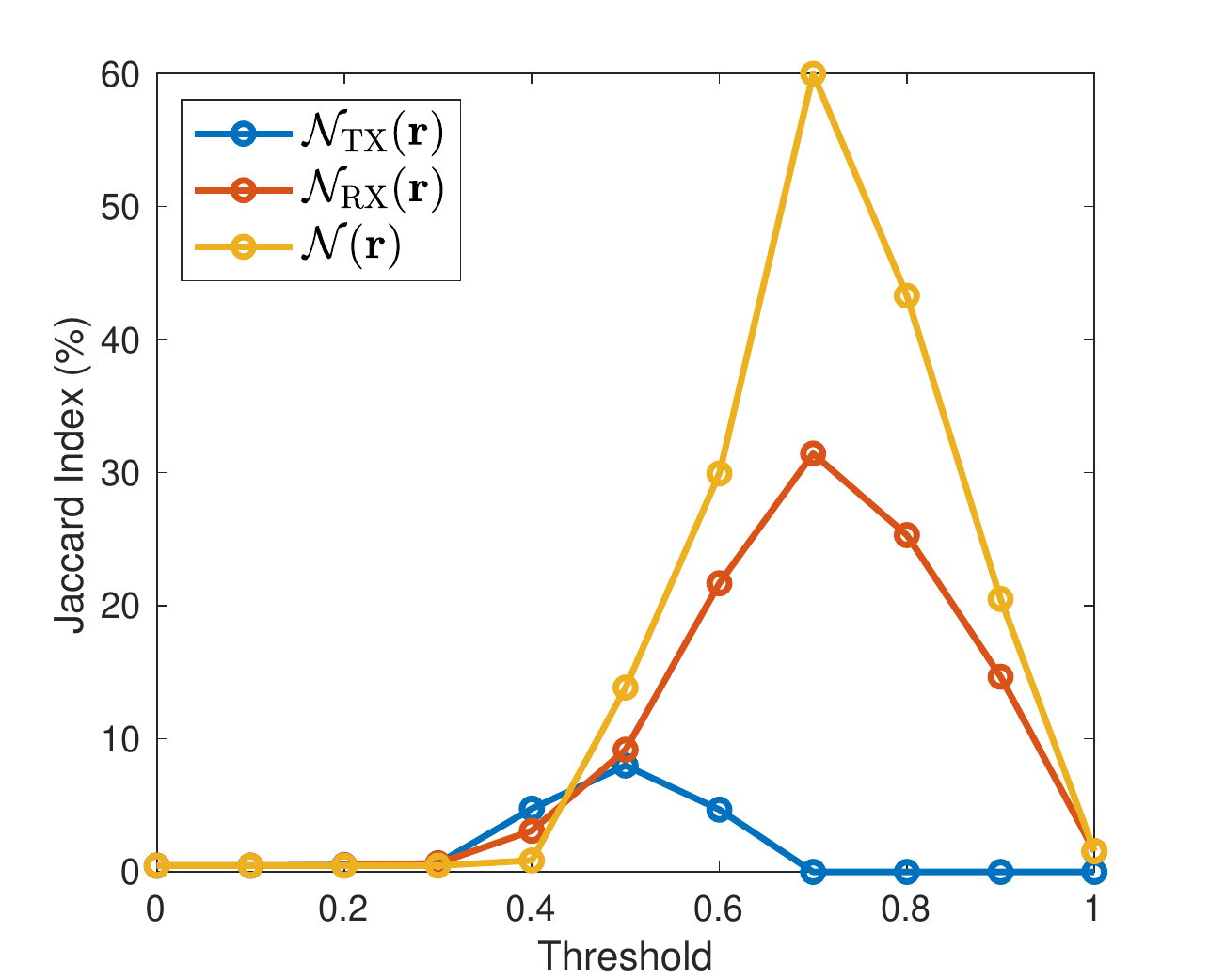}
  \caption{\label{ResultR3}(Example \ref{exR1}) Maps of $\mathfrak{F}_{\tx}(\mr)$ (first column), $\mathfrak{F}_{\rx}(\mr)$ (second column), $\mathfrak{F}(\mr)$ (third column), and Jaccard index (fourth column). Green and red colored circles describe the location of transmitters and receivers, respectively.}
\end{figure}

\begin{example}[Imaging of Three Screw Drivers]\label{exR2}
For the final result, we consider the imaging of the cross-section of three screw drivers. Figures \ref{ResultR4} and \ref{ResultR5} show the maps of $\mathfrak{F}_{\tx}(\mr)$, $\mathfrak{F}_{\rx}(\mr)$, $\mathfrak{F}(\mr)$, and Jaccard index with the same antenna settings in Examples \ref{ex1} and \ref{ex2}, respectively. Based on the result, we can observe that
\begin{enumerate}
\item[\textcircled{1}] it is impossible to recognize the existence of three drivers through the maps of $\mathfrak{F}_{\tx}(\mr)$ with settings $\mathbf{A}_j\cup\mathbf{B}_1$, $j=1,2,3,4$. It is interesting to examine that it is possible to identify the location of single driver, say $\Sigma_1$, but very difficult to recognize remaining two drivers due to the appearance of large artifacts in the neighborhood of drivers with selections $\mathbf{A}_j\cup\mathbf{B}_2$, $j=1,2$. However, by selecting arrangements $\mathbf{A}_j\cup\mathbf{B}_2$, another driver, say $\Sigma_2$, can be identified but remaining drivers cannot be recognized,
\item[\textcircled{2}] it is impossible to recognize the existence of three drivers through the maps of $\mathfrak{F}_{\rx}(\mr)$ with settings $\mathbf{A}_1\cup\mathbf{B}_l$, $l=1,2$ but peak of large magnitude is appear in the neighborhoof of $\Sigma_1$ with selections $\mathbf{A}_j\cup\mathbf{B}_1$, $j=2,3,4$ and $\mathbf{A}_2\cup\mathbf{B}_2$. Unfortunately, some artifacts with large magnitudes are also included in the neighborhood of $\Sigma_1$ so it is still difficult to distinguish $\Sigma_1$ from them. By selecting antenna arrangements $\mathbf{A}_j\cup\mathbf{B}_2$, $j=3,4$, it is possible to recognize the $\Sigma_1$ but it is still impossible to recognize remaining drivers.
\item[\textcircled{3}] based on the \textcircled{1} and \textcircled{2}, it is possible to recognize $\Sigma_1$ and $\Sigma_2$ through the map of $\mathfrak{F}(\mr)$ with selections $\mathbf{A}_4\cup\mathbf{B}_2$. However, with other selections of antenna arrangements, it is impossible to recognize the existence of three drivers or it is possible to identify only one driver.
\end{enumerate}
Figures \ref{ResultR6} shows maps of $\mathfrak{F}_{\tx}(\mr)$, $\mathfrak{F}_{\rx}(\mr)$, $\mathfrak{F}(\mr)$, and Jaccard index with the same antenna settings in Example \ref{ex3}. It is worth to mention that three drivers can be recognized through the maps of $\mathfrak{F}_{\tx}(\mr)$, $\mathfrak{F}_{\rx}(\mr)$, $\mathfrak{F}(\mr)$ with $\mathbf{A}_\star\cup\mathbf{B}_\star$. However, several artifacts are included in the maps of $\mathfrak{F}_{\tx}(\mr)$ and $\mathfrak{F}_{\rx}(\mr)$. Notice that by selecting antenna arrangements $\mathbf{A}_j\cup\mathbf{B}_5$, $j=5,6,7$, it is very difficult to identify three drivers through the maps of $\mathfrak{F}_{\tx}(\mr)$, $\mathfrak{F}_{\rx}(\mr)$, $\mathfrak{F}(\mr)$ but the existence of $\Sigma_1$ and three drivers can be recognized through the map of $\mathfrak{F}_{\rx}(\mr)$ with setting $\mathbf{A}_6\cup\mathbf{B}_5$ and $\mathbf{A}_7\cup\mathbf{B}_5$, respectively.
\end{example}

\begin{figure}[h]
  \centering
  \includegraphics[width=0.25\textwidth]{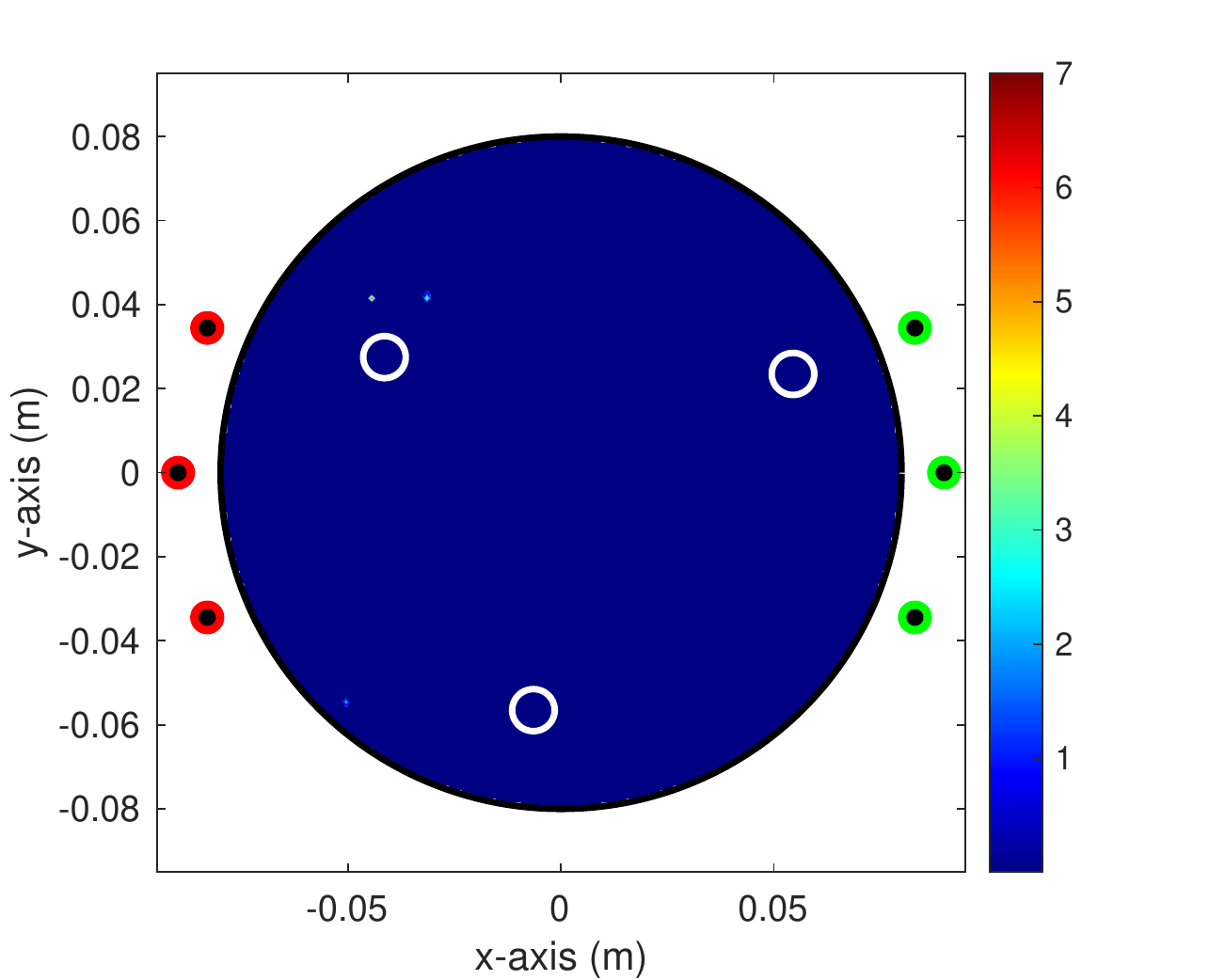}\hfill
  \includegraphics[width=0.25\textwidth]{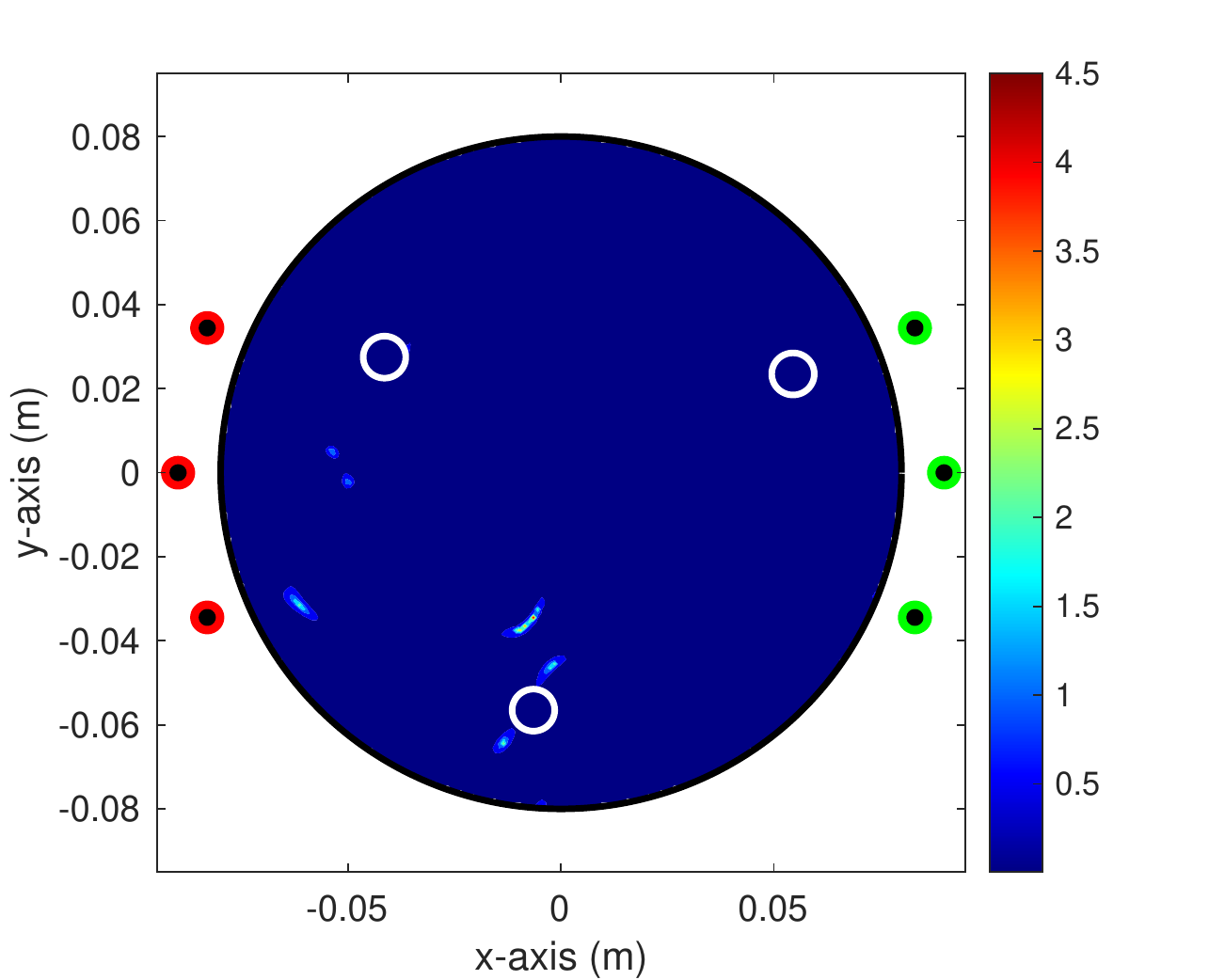}\hfill
  \includegraphics[width=0.25\textwidth]{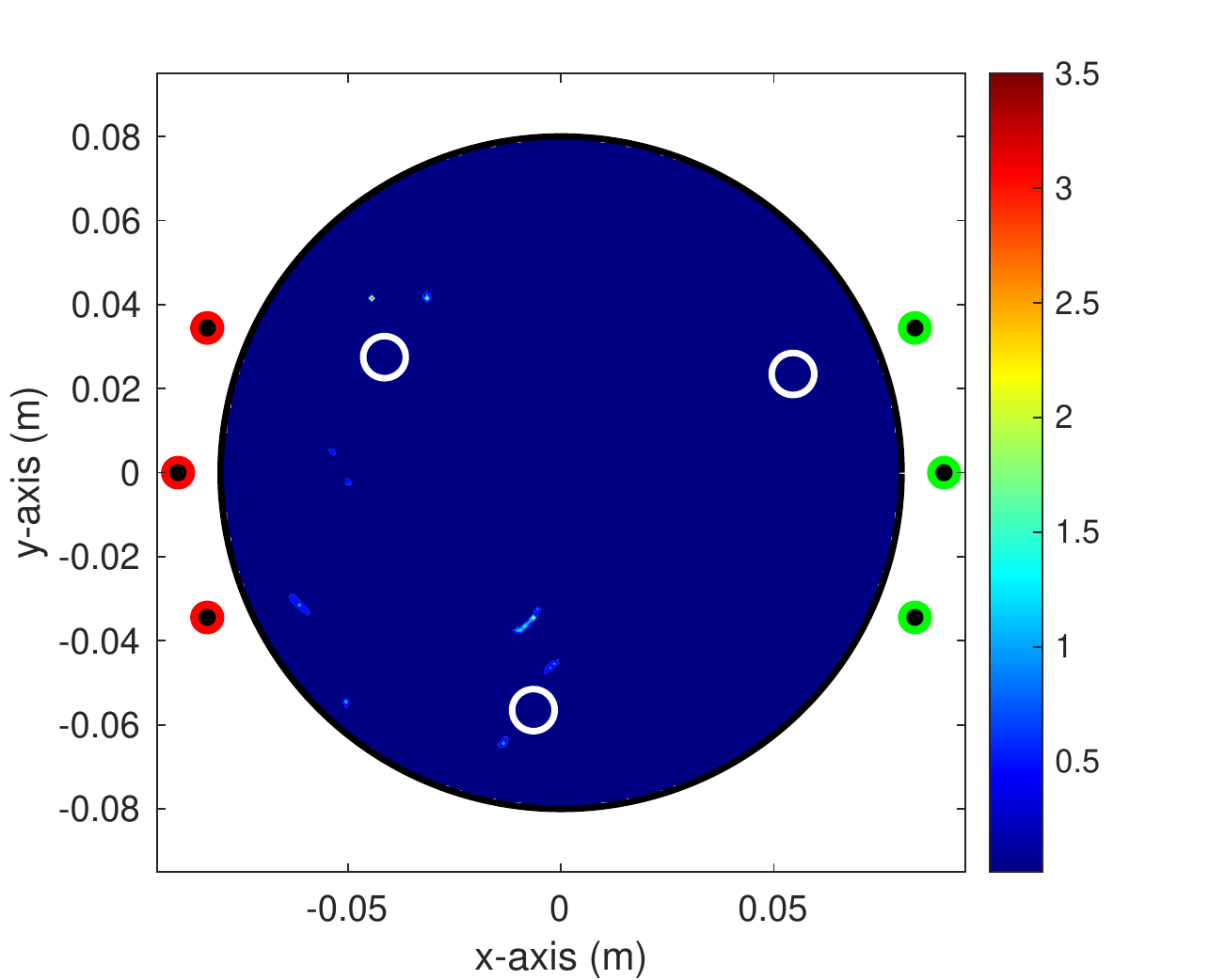}\hfill
  \includegraphics[width=0.25\textwidth]{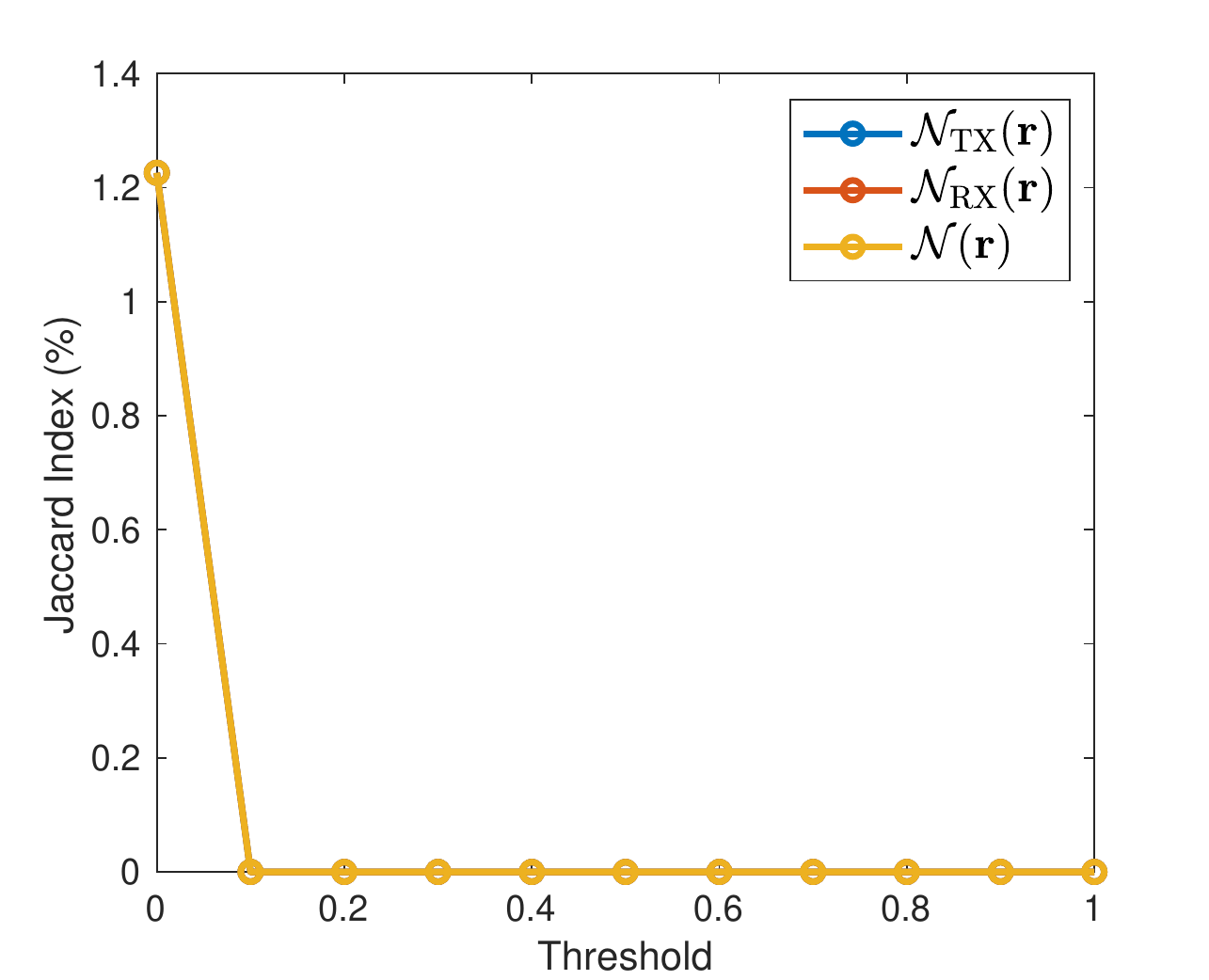}\\
  \includegraphics[width=0.25\textwidth]{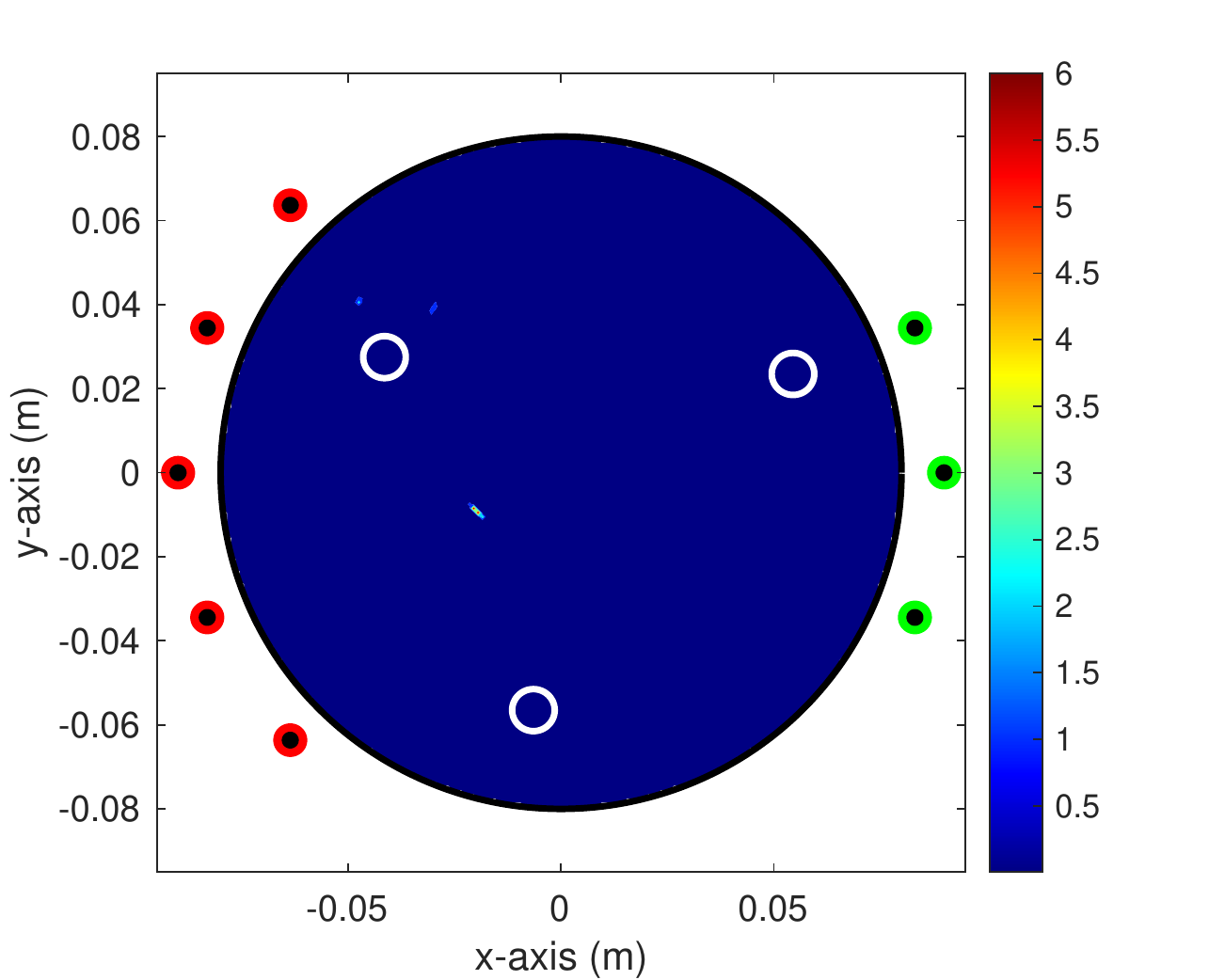}\hfill
  \includegraphics[width=0.25\textwidth]{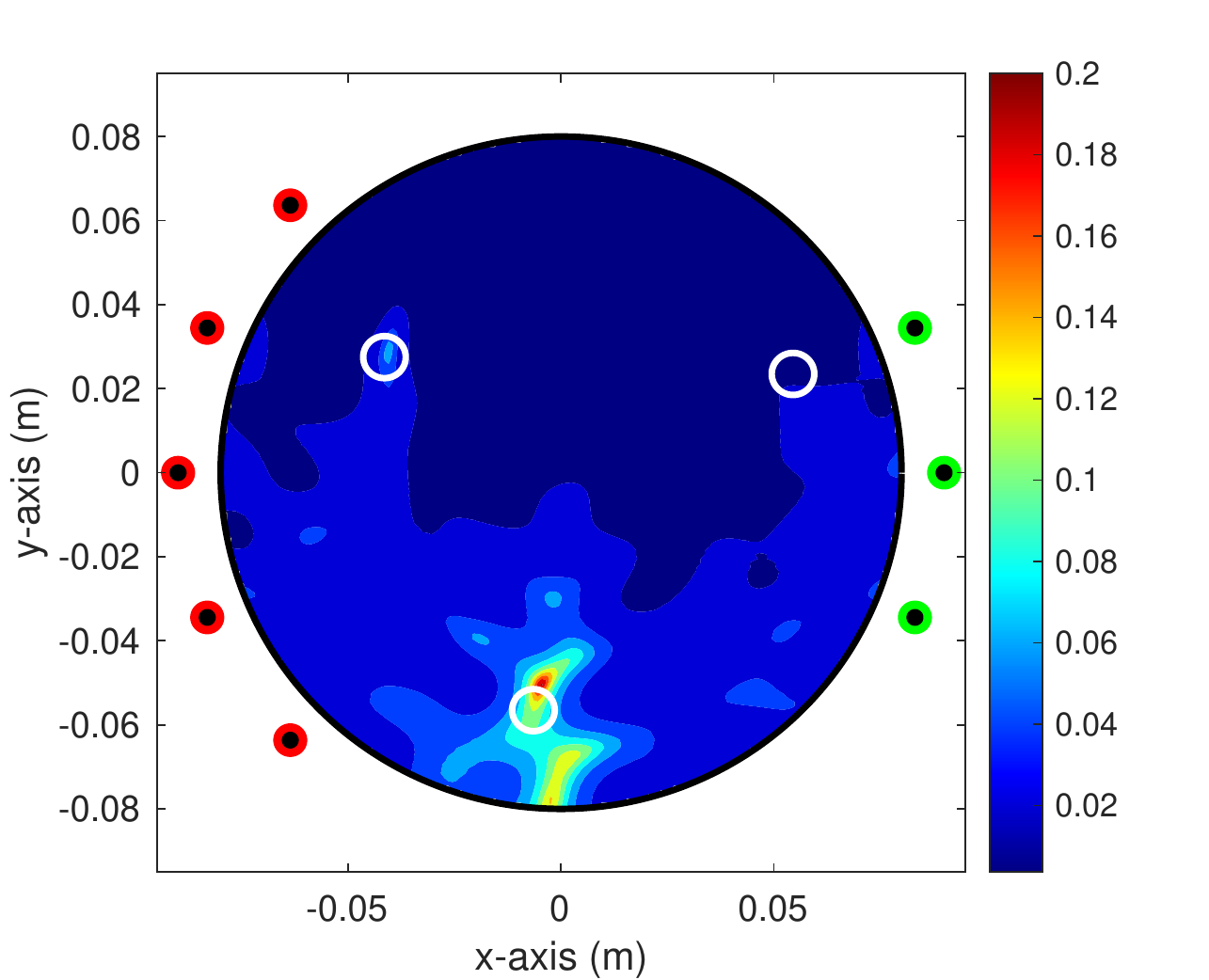}\hfill
  \includegraphics[width=0.25\textwidth]{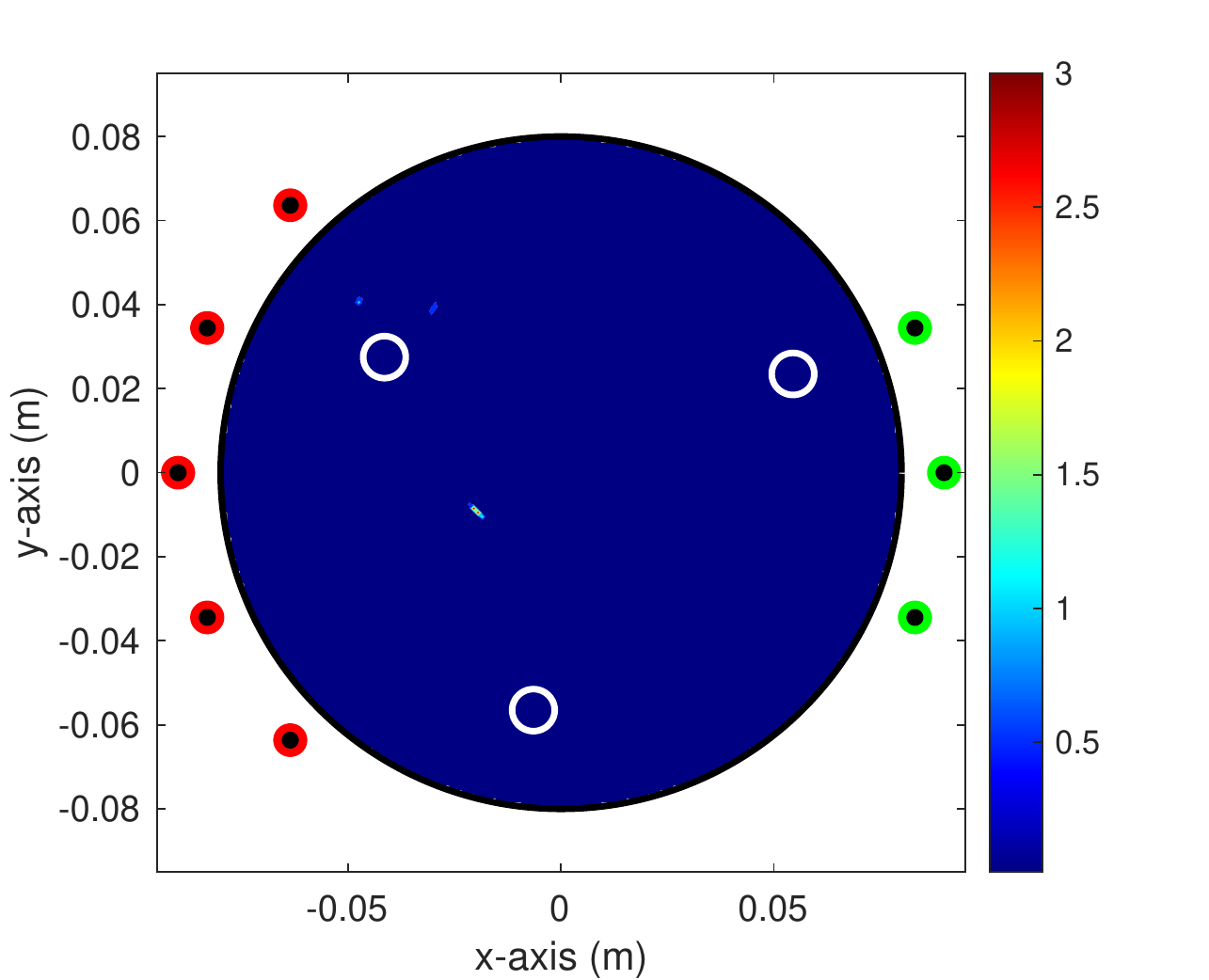}\hfill
  \includegraphics[width=0.25\textwidth]{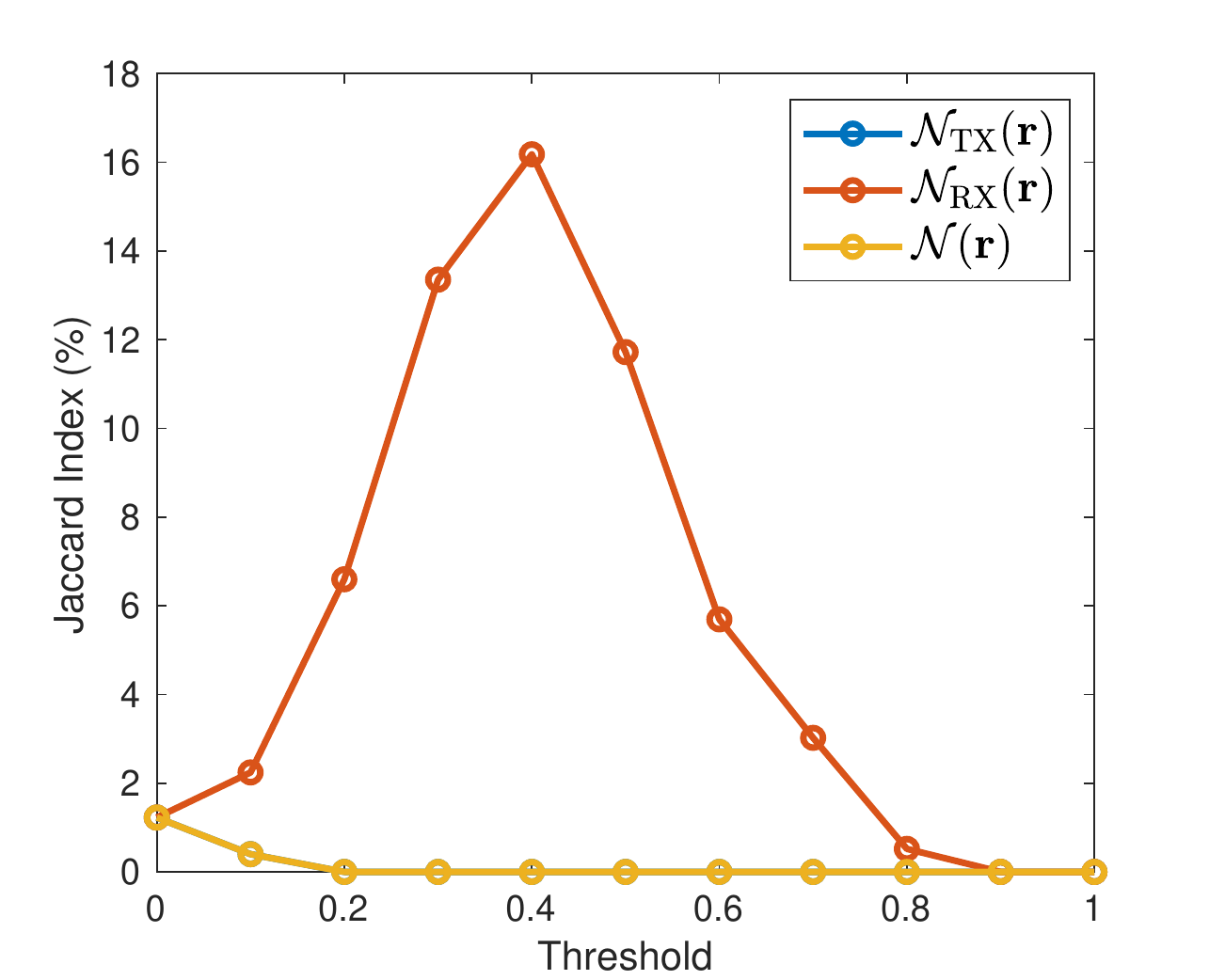}\\
  \includegraphics[width=0.25\textwidth]{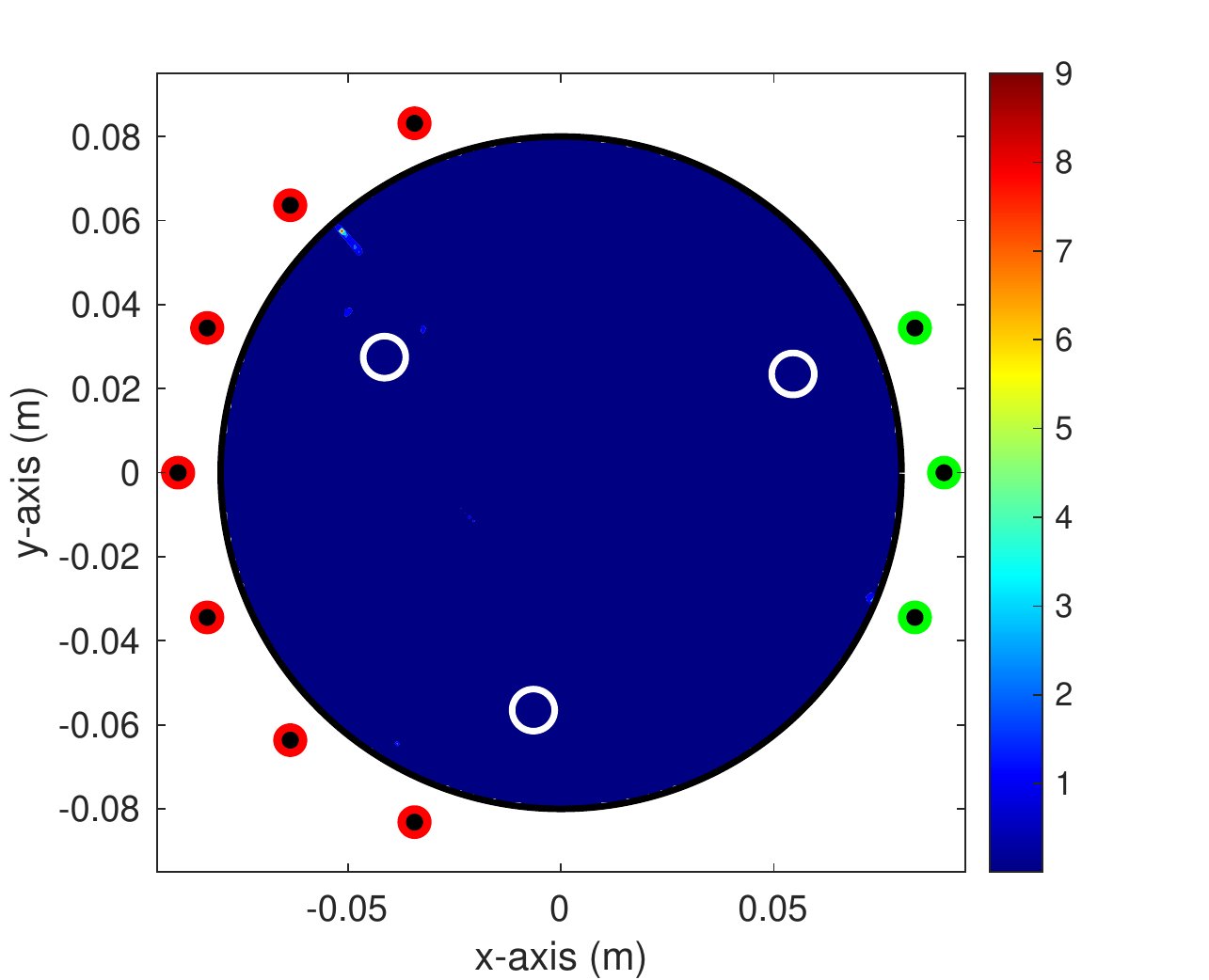}\hfill
  \includegraphics[width=0.25\textwidth]{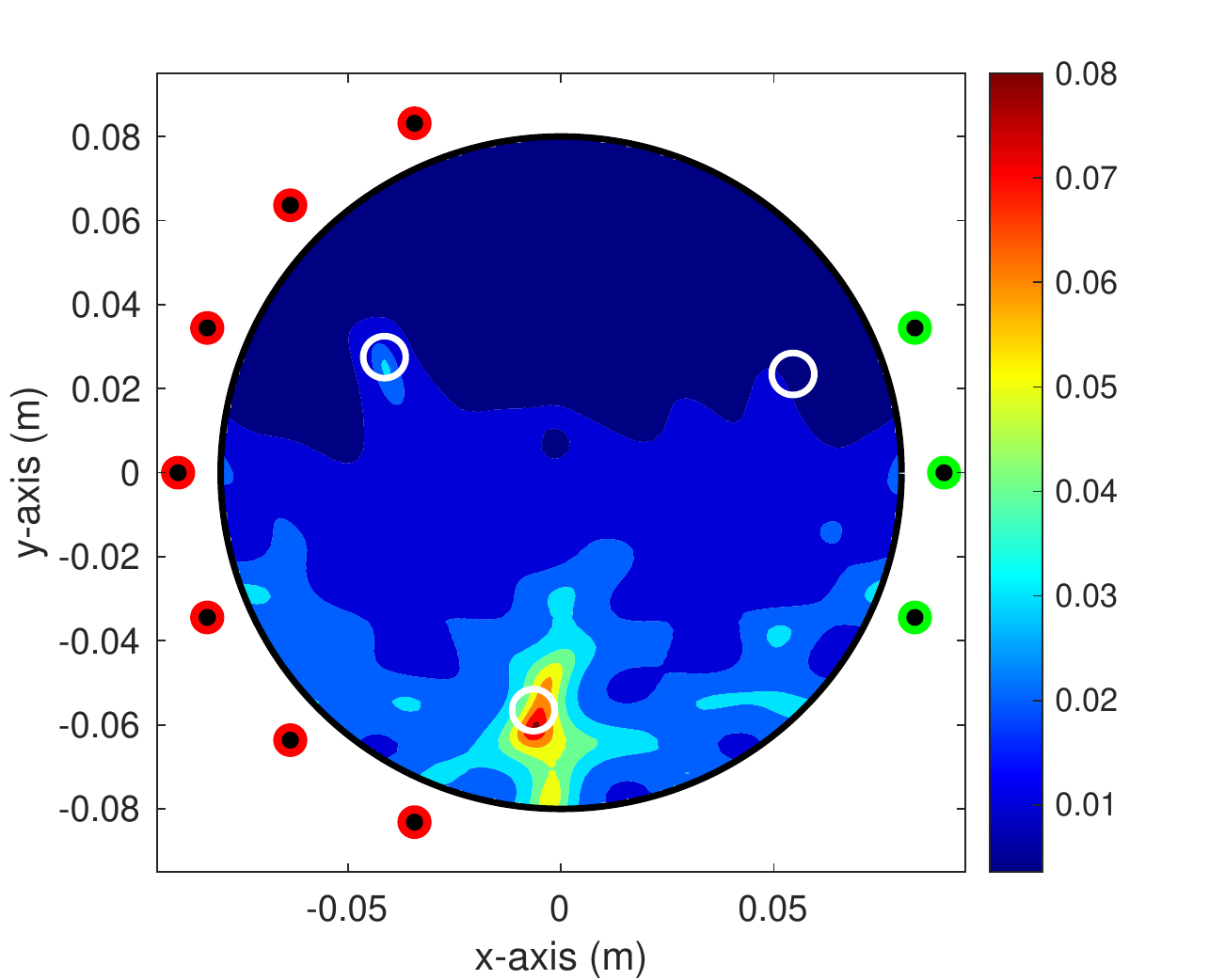}\hfill
  \includegraphics[width=0.25\textwidth]{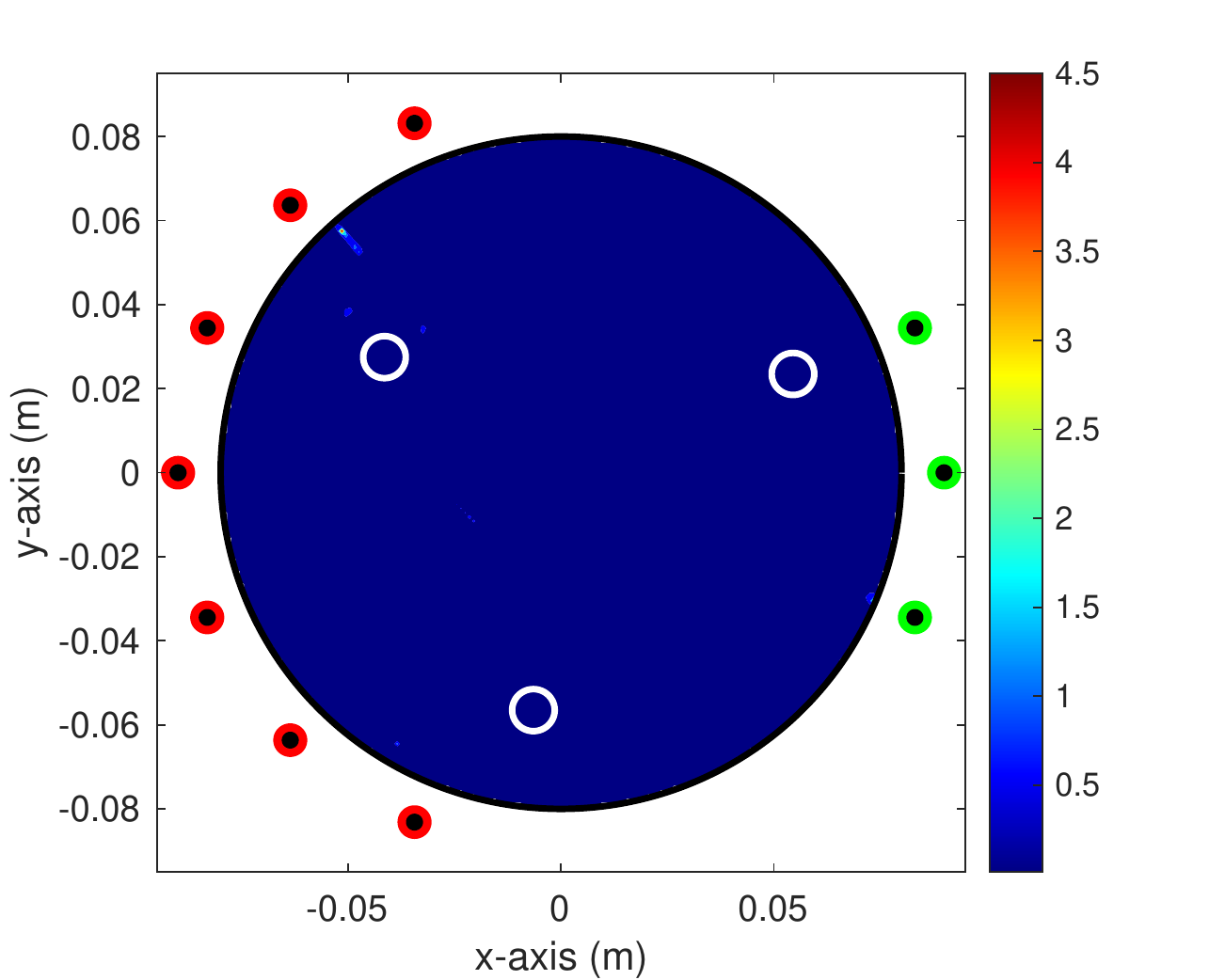}\hfill
  \includegraphics[width=0.25\textwidth]{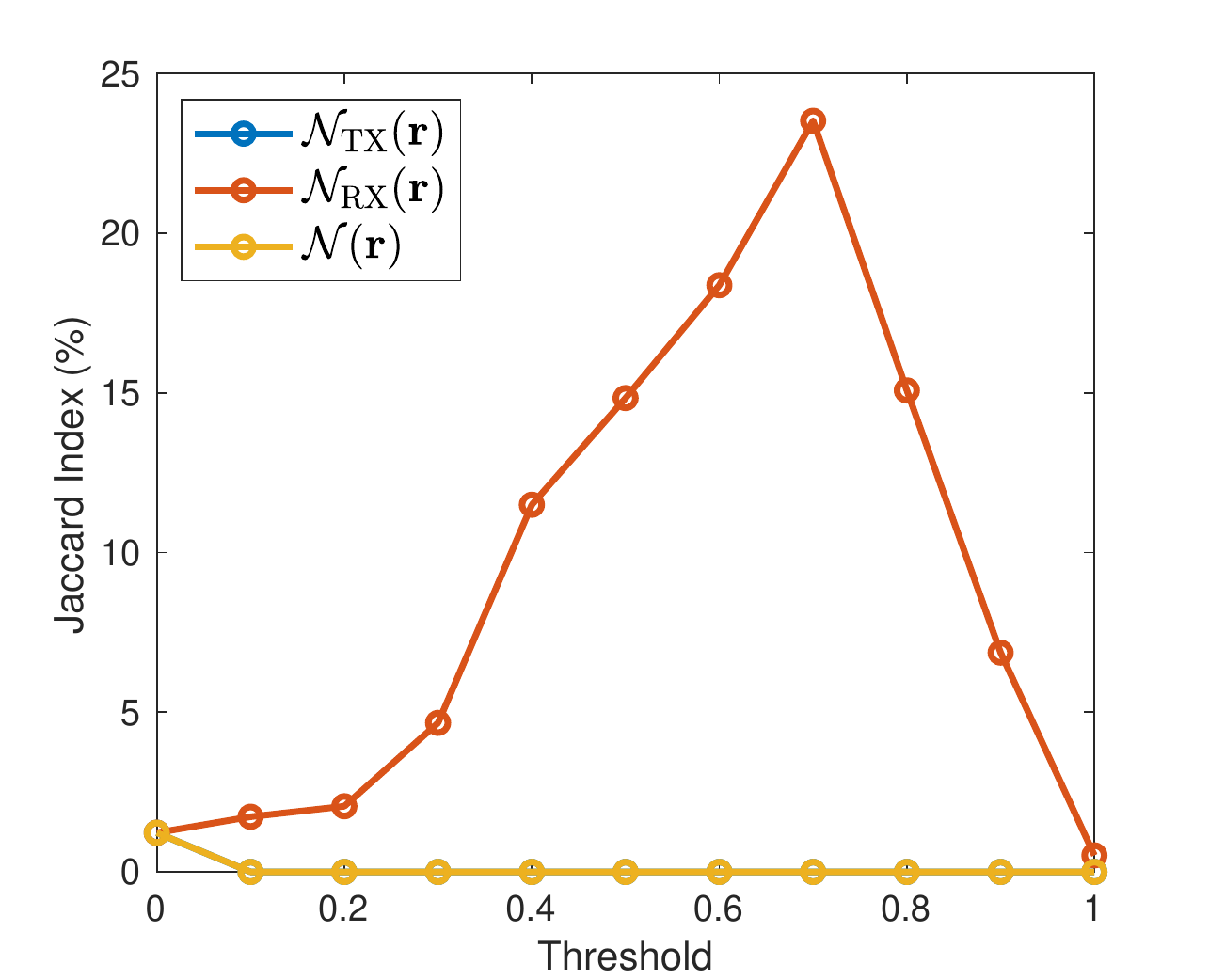}\\
  \includegraphics[width=0.25\textwidth]{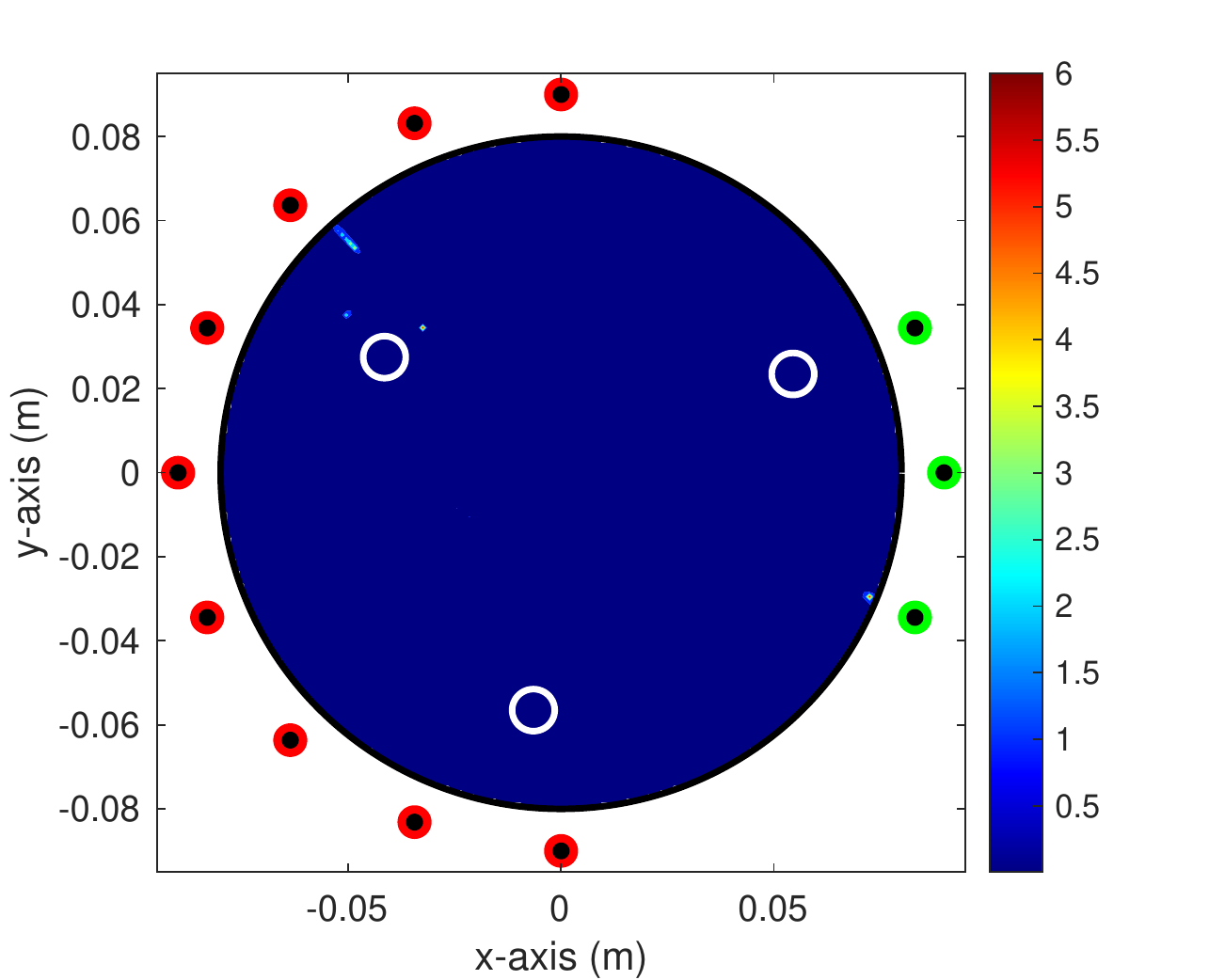}\hfill
  \includegraphics[width=0.25\textwidth]{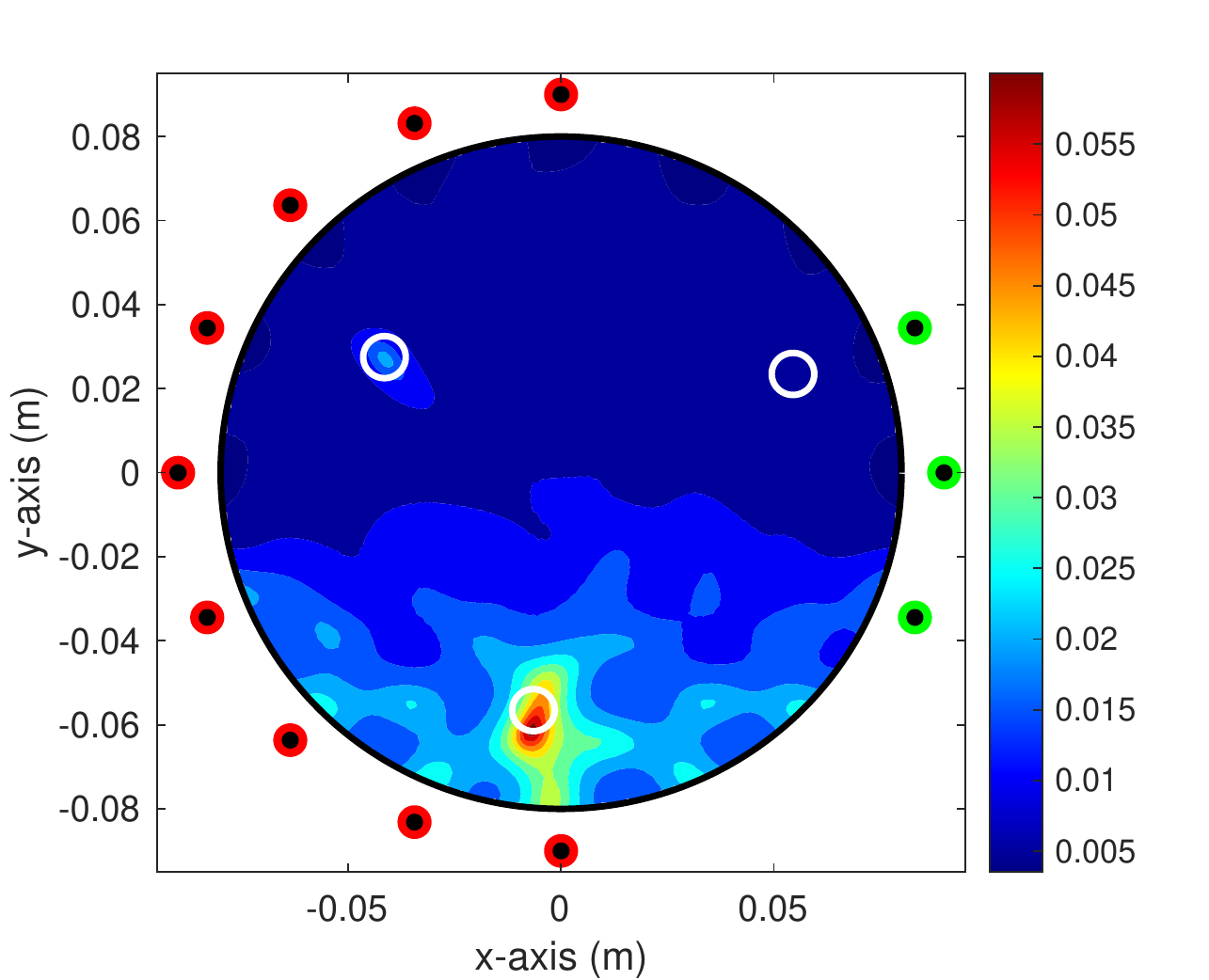}\hfill
  \includegraphics[width=0.25\textwidth]{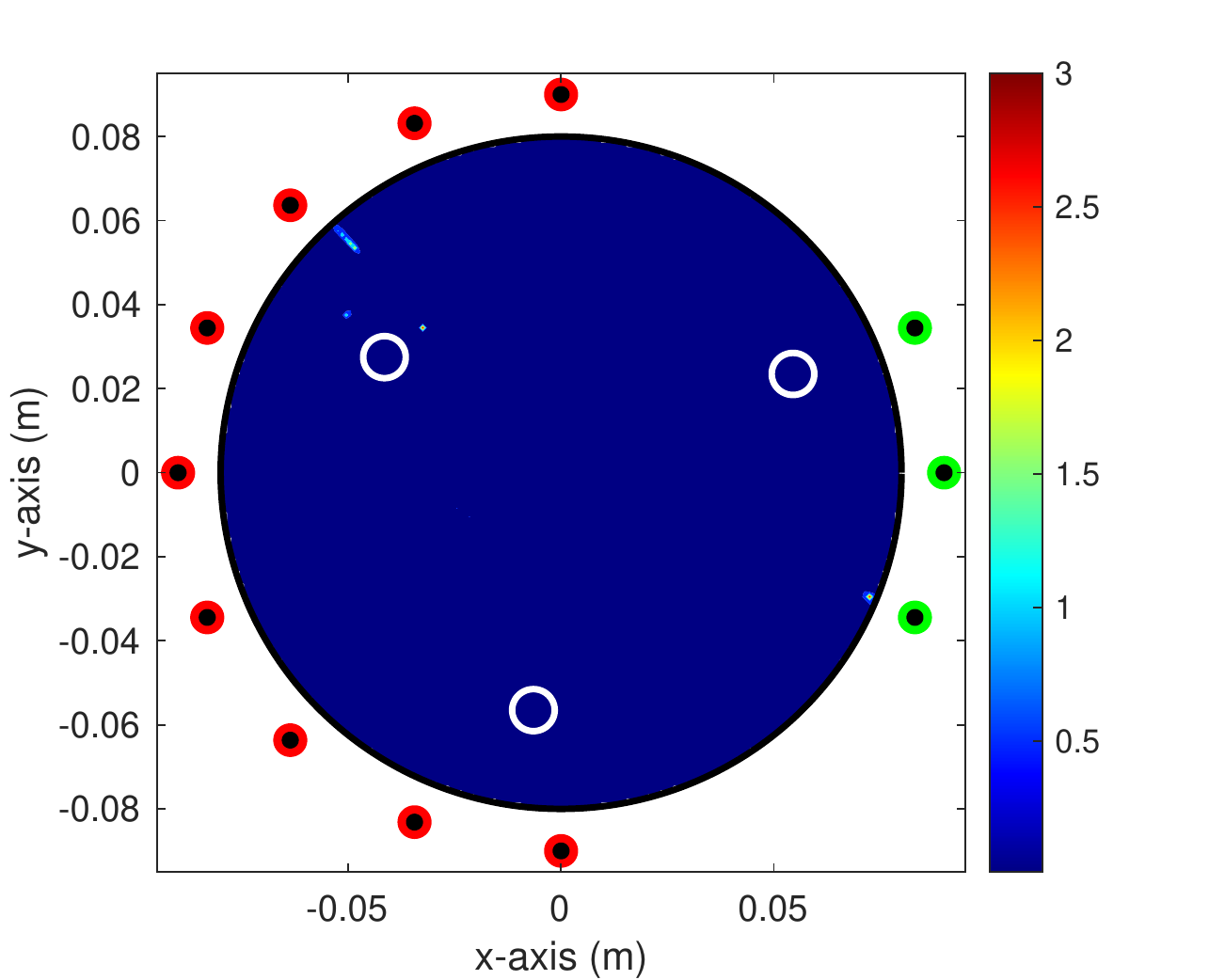}\hfill
  \includegraphics[width=0.25\textwidth]{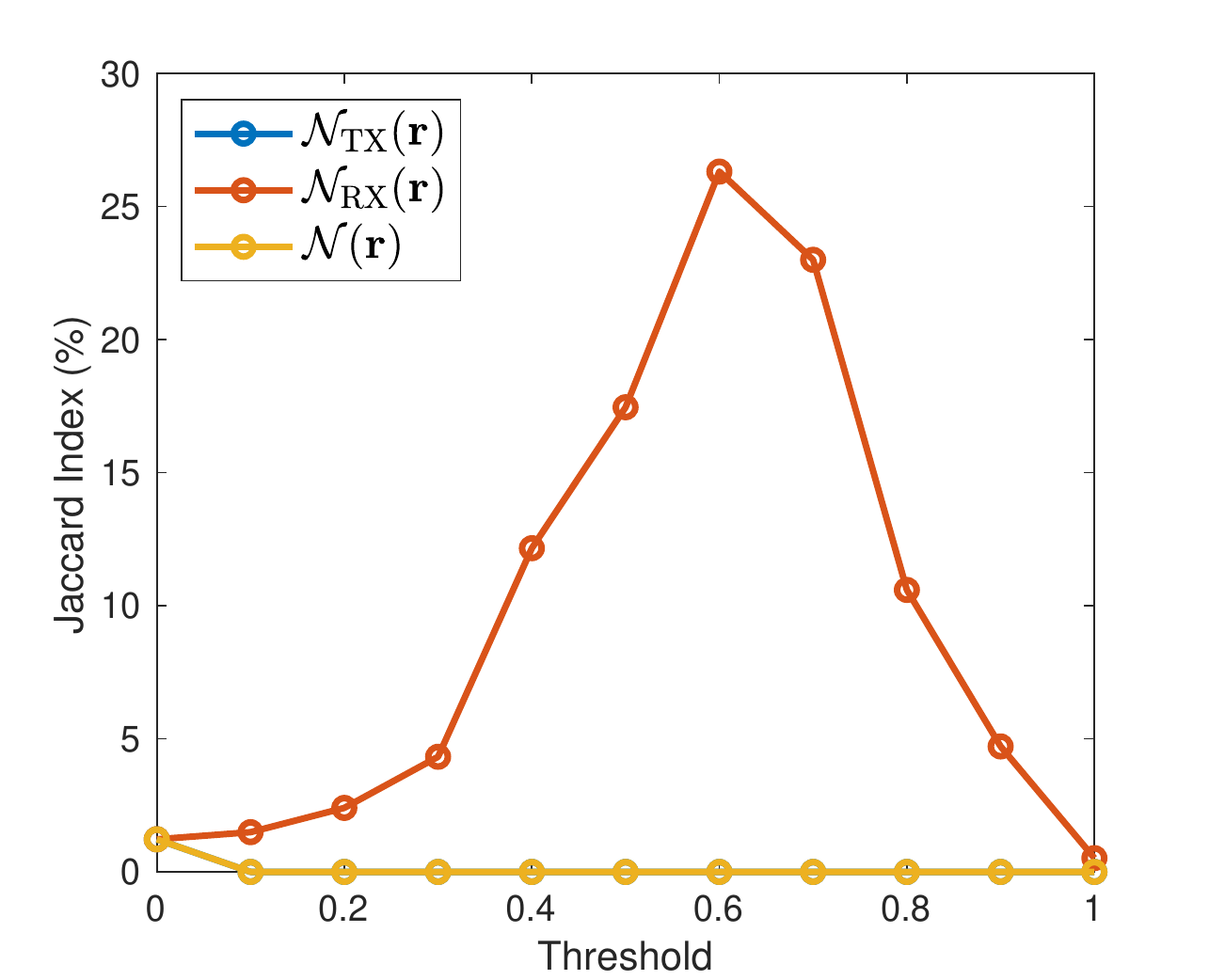}
  \caption{\label{ResultR4}(Example \ref{exR2}) Maps of $\mathfrak{F}_{\tx}(\mr)$ (first column), $\mathfrak{F}_{\rx}(\mr)$ (second column), $\mathfrak{F}(\mr)$ (third column), and Jaccard index (fourth column). Green and red colored circles describe the location of transmitters and receivers, respectively.}
\end{figure}

\begin{figure}[h]
  \centering
  \includegraphics[width=0.25\textwidth]{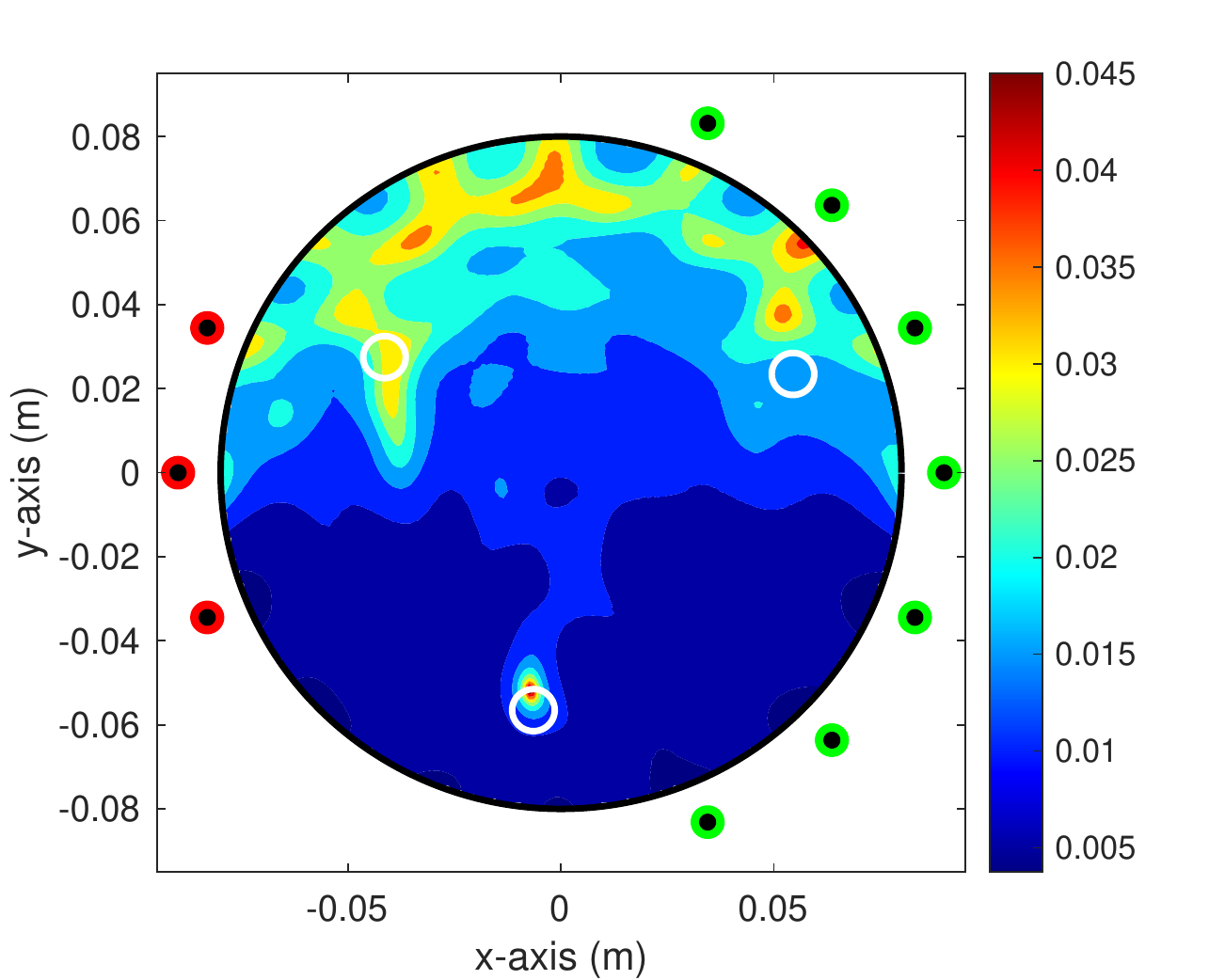}\hfill
  \includegraphics[width=0.25\textwidth]{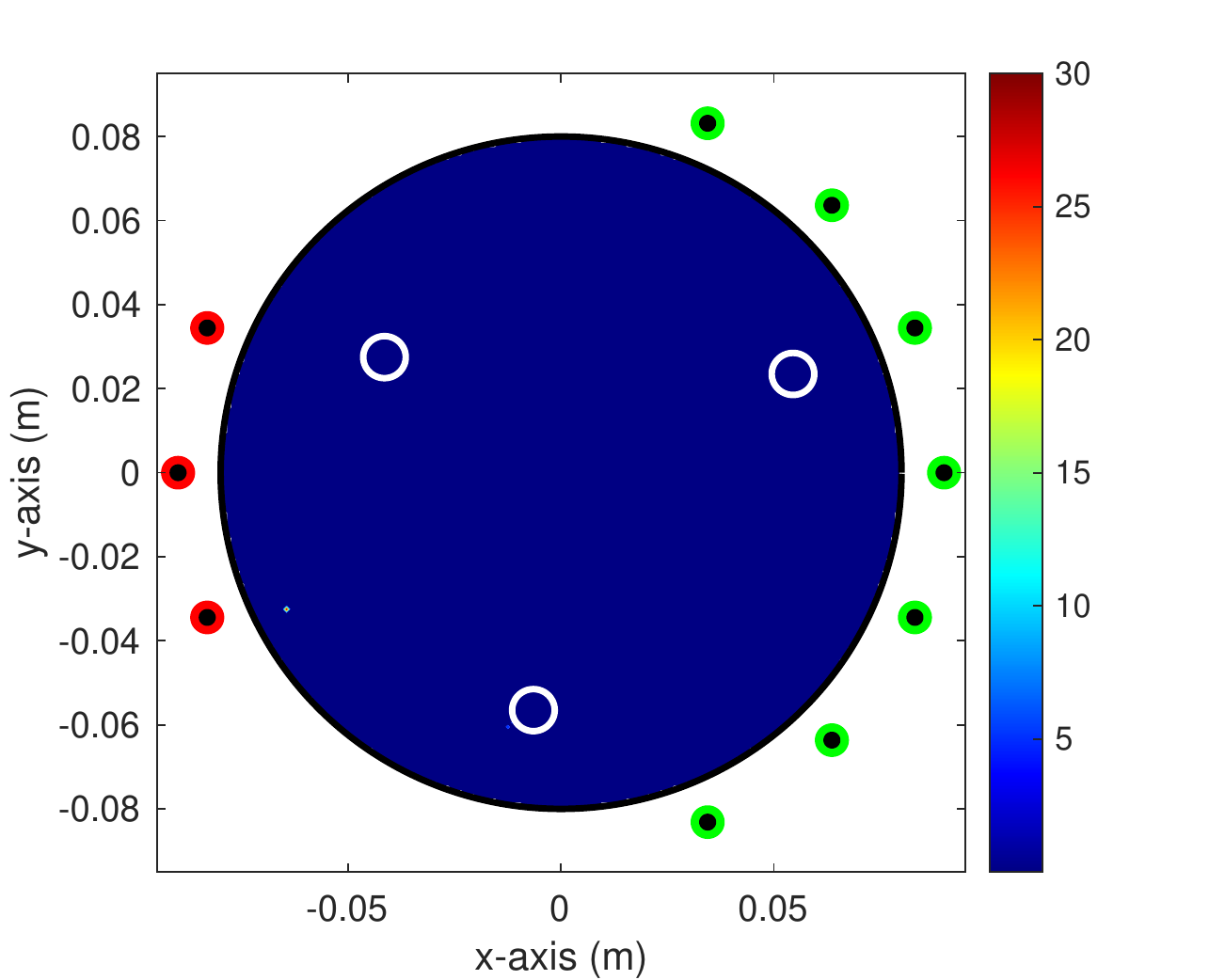}\hfill
  \includegraphics[width=0.25\textwidth]{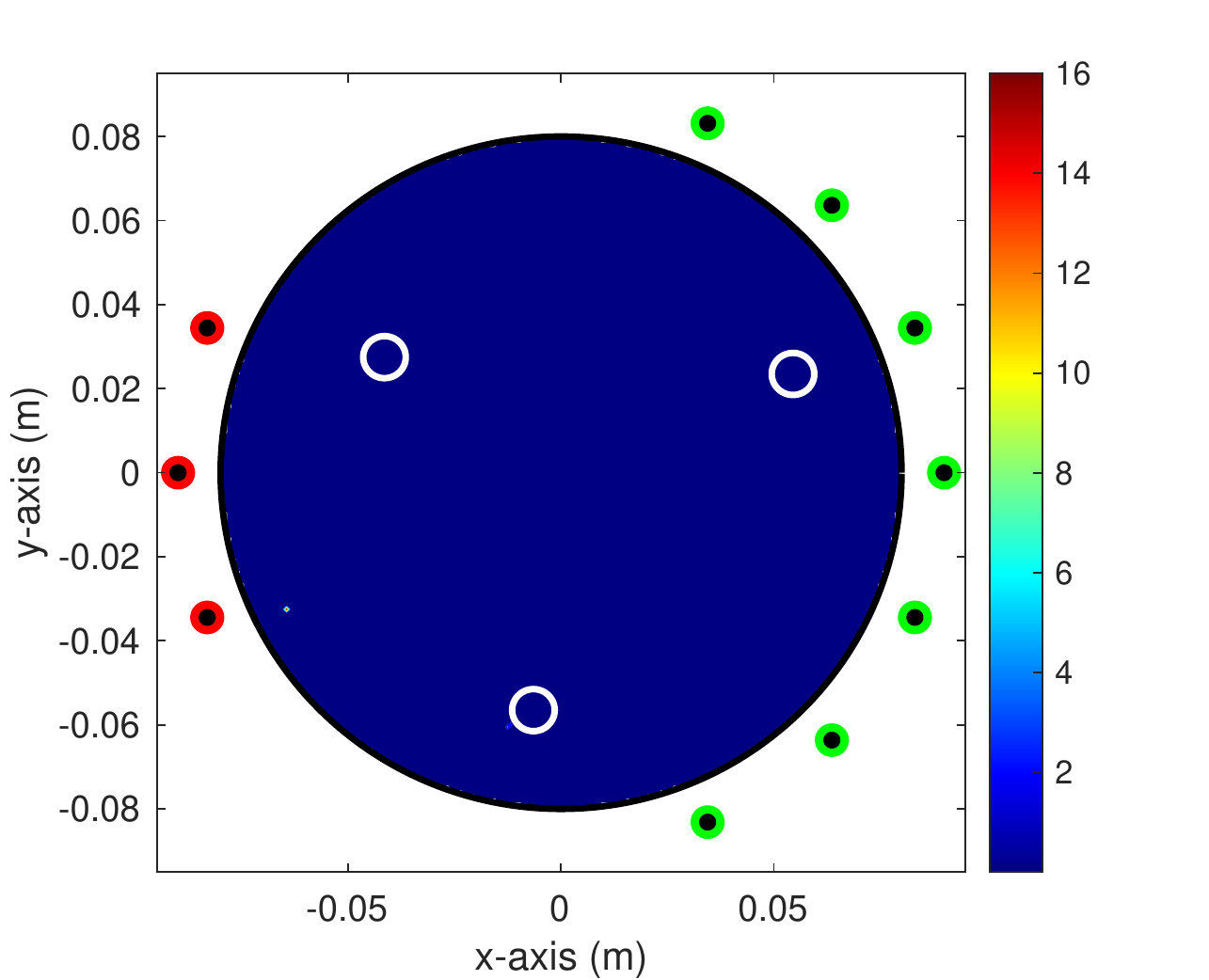}\hfill
  \includegraphics[width=0.25\textwidth]{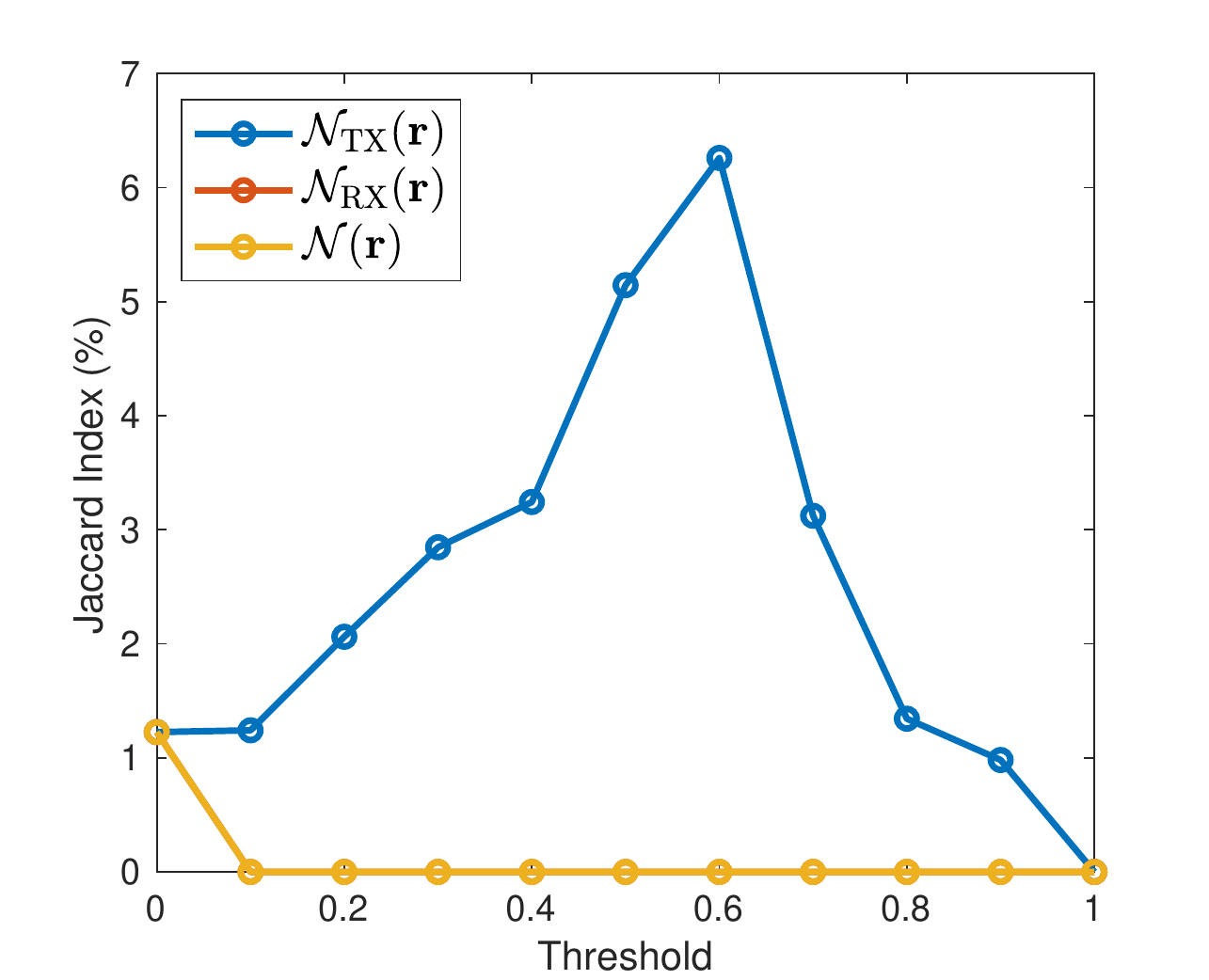}\\
  \includegraphics[width=0.25\textwidth]{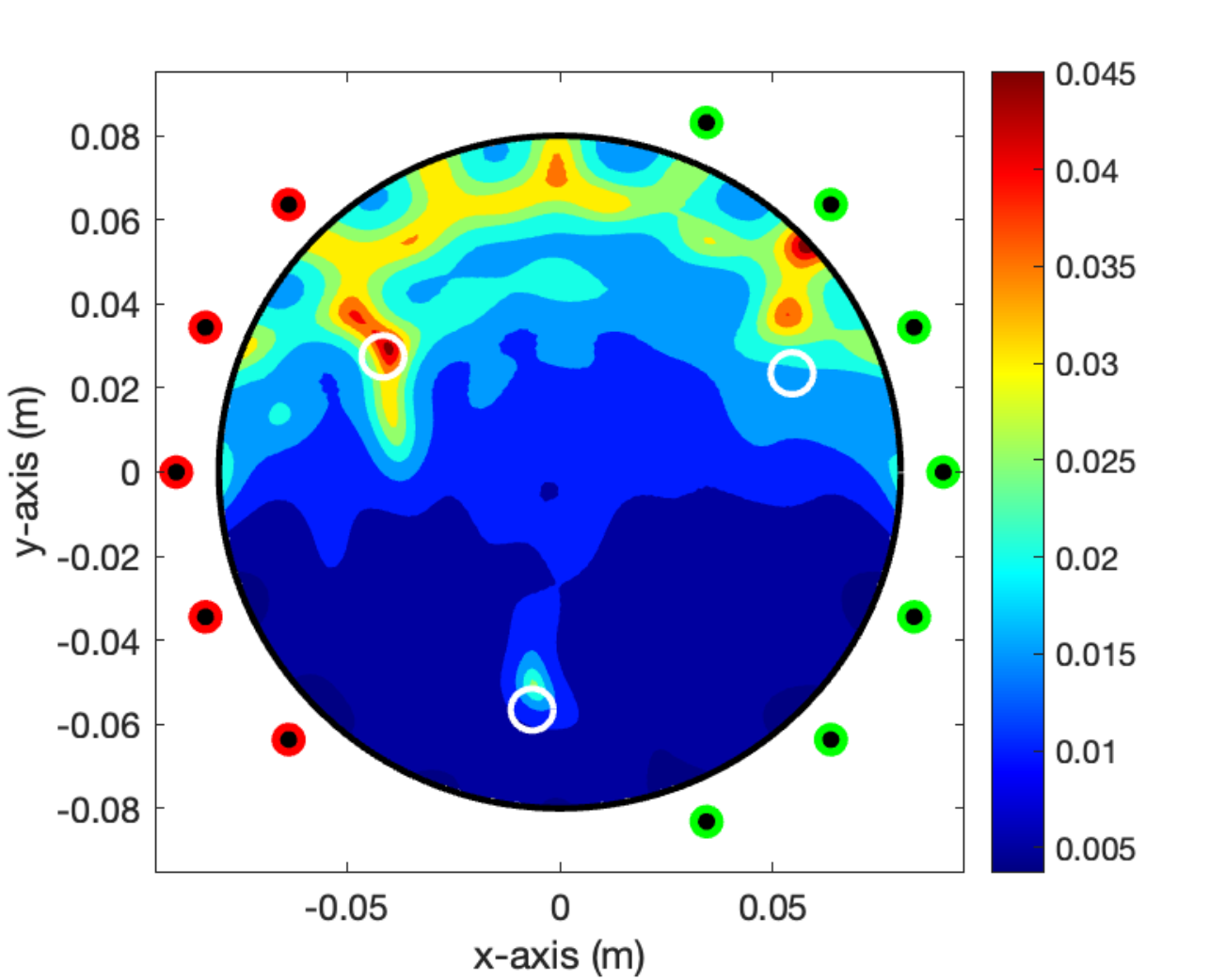}\hfill
  \includegraphics[width=0.25\textwidth]{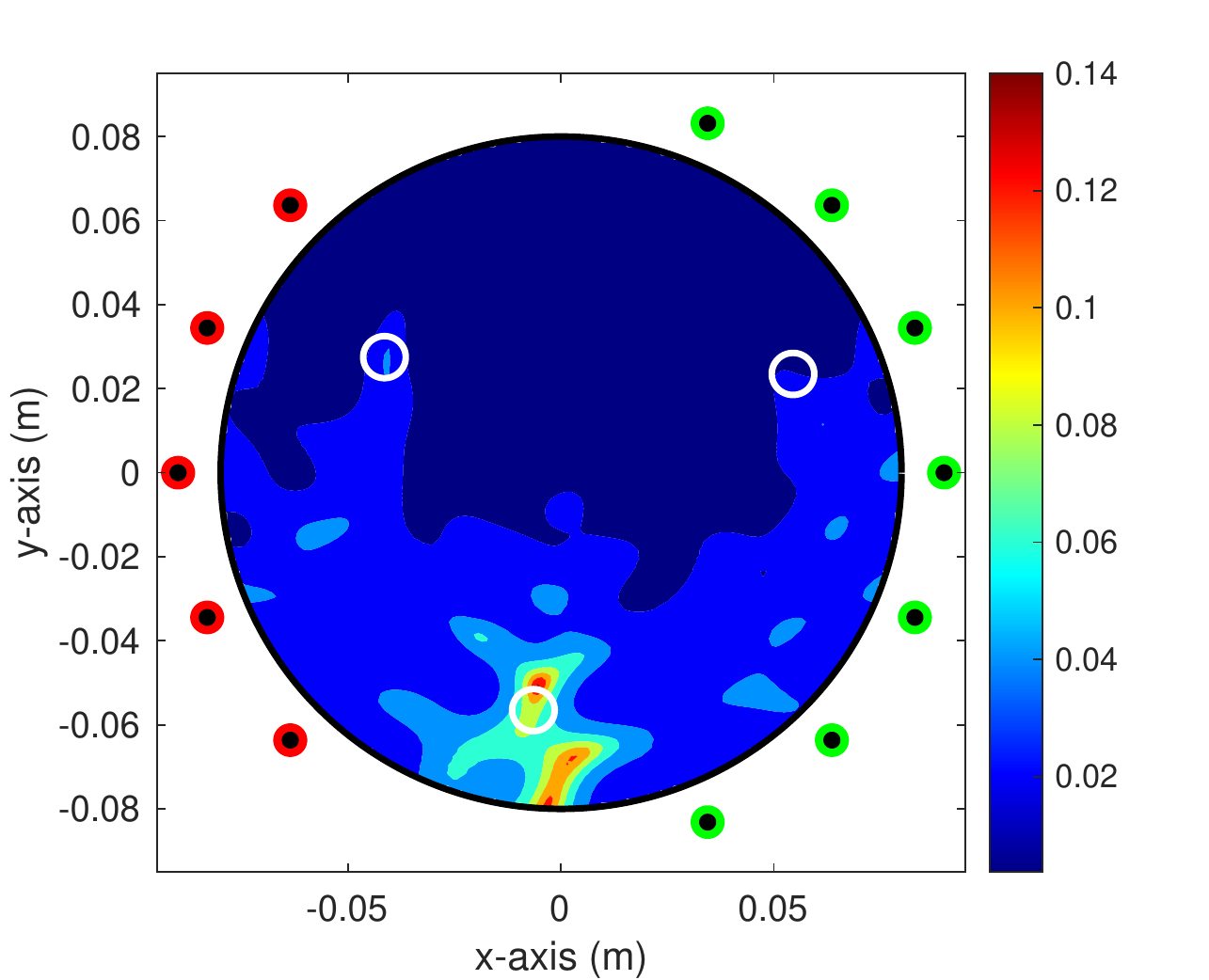}\hfill
  \includegraphics[width=0.25\textwidth]{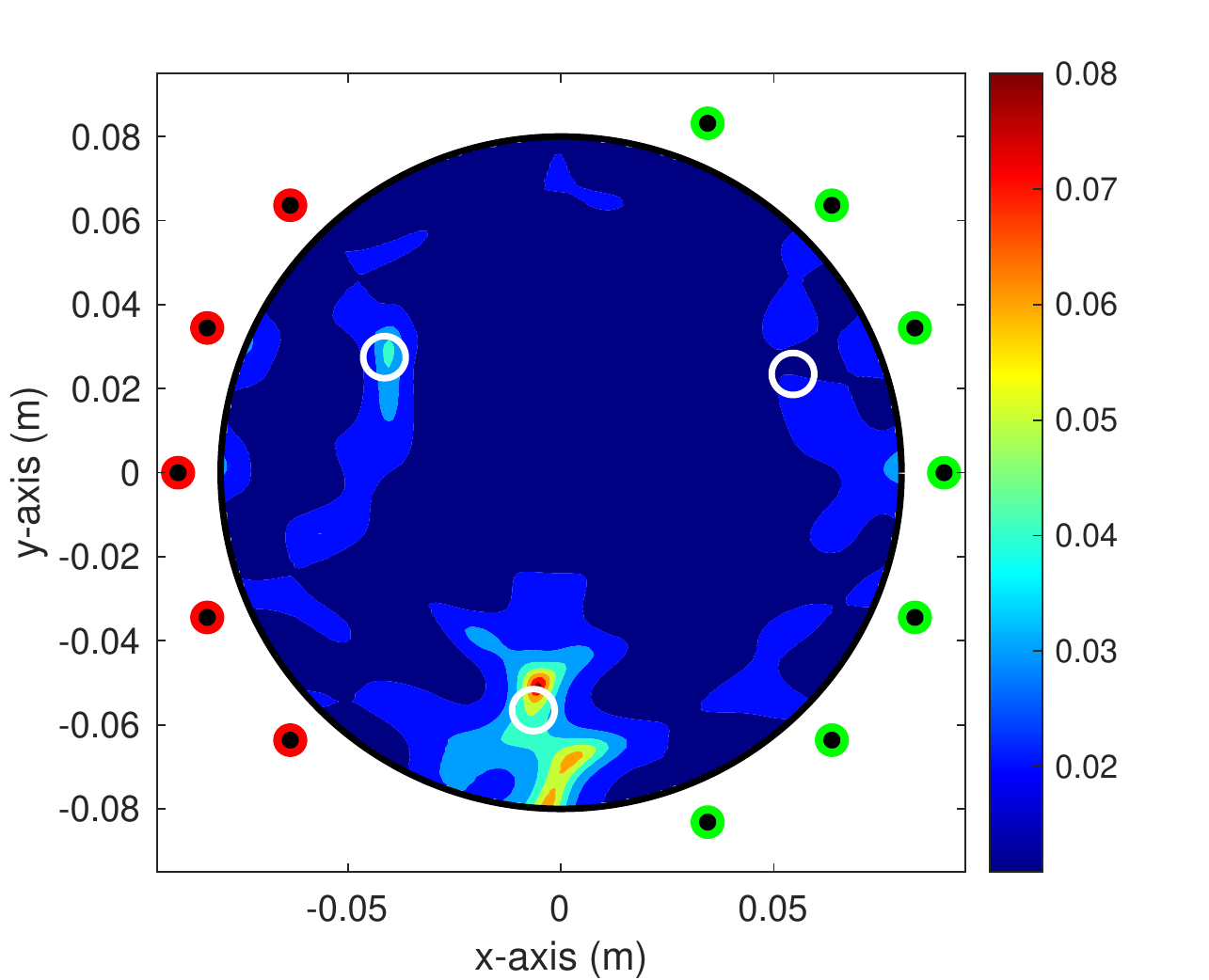}\hfill
  \includegraphics[width=0.25\textwidth]{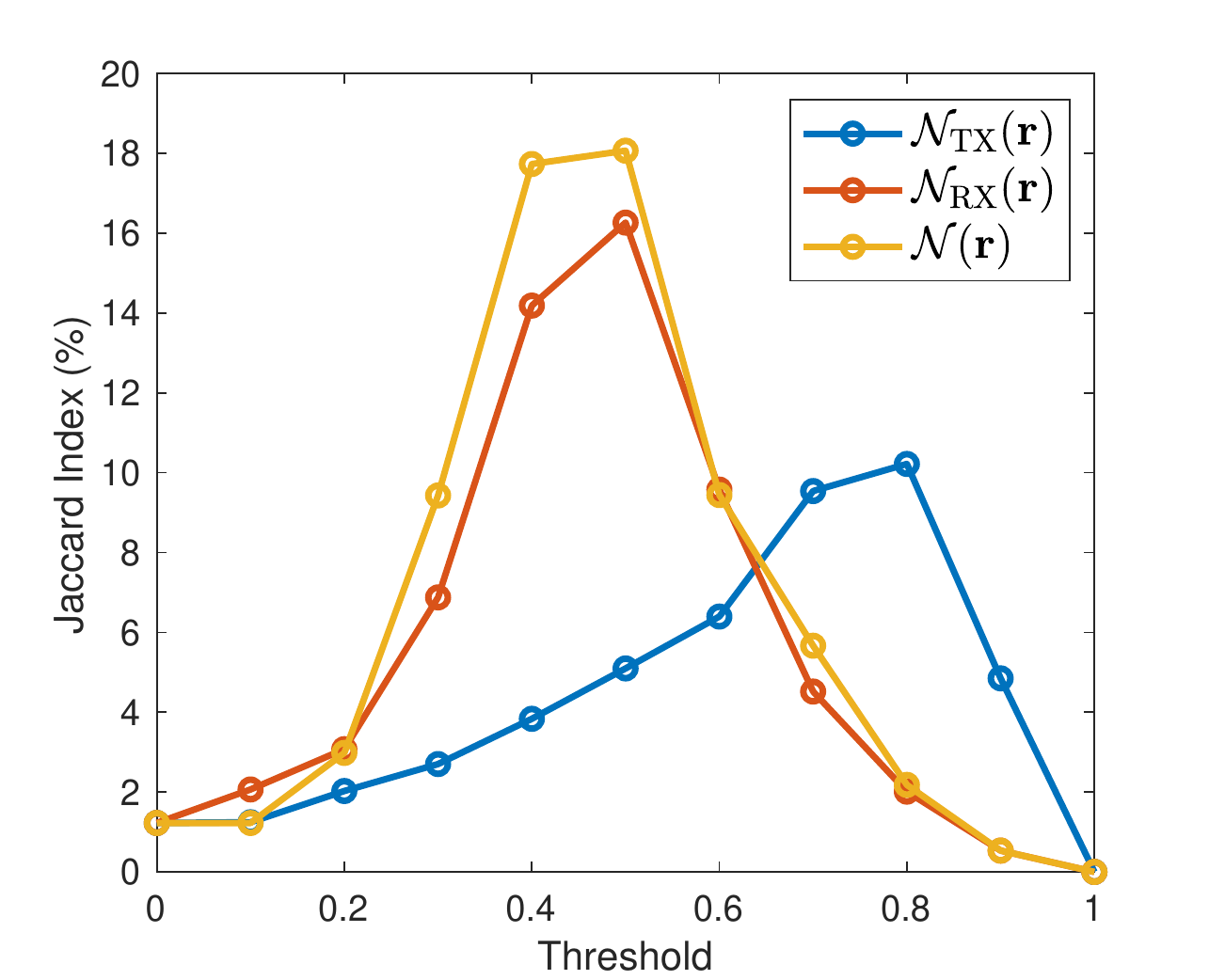}\\
  \includegraphics[width=0.25\textwidth]{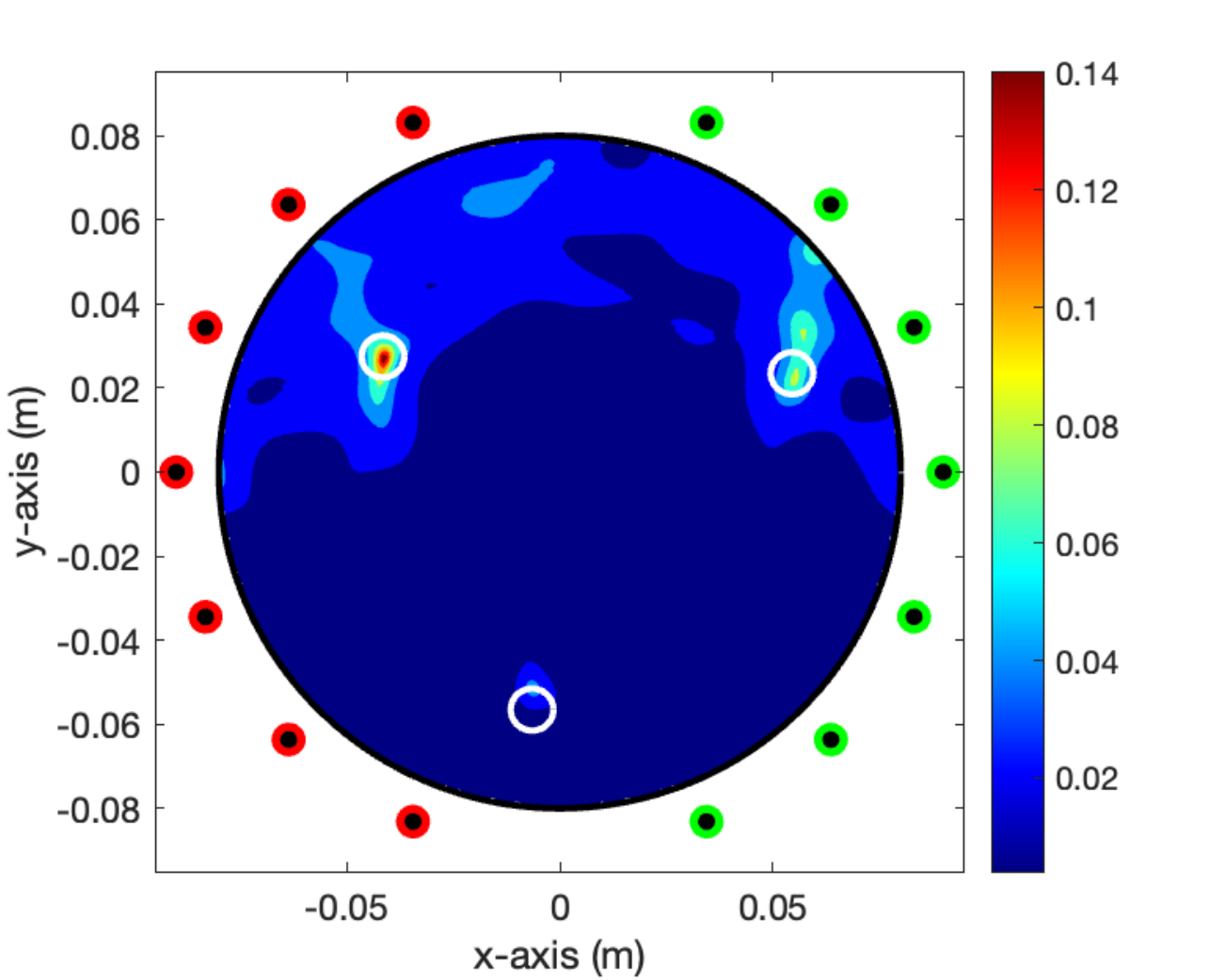}\hfill
  \includegraphics[width=0.25\textwidth]{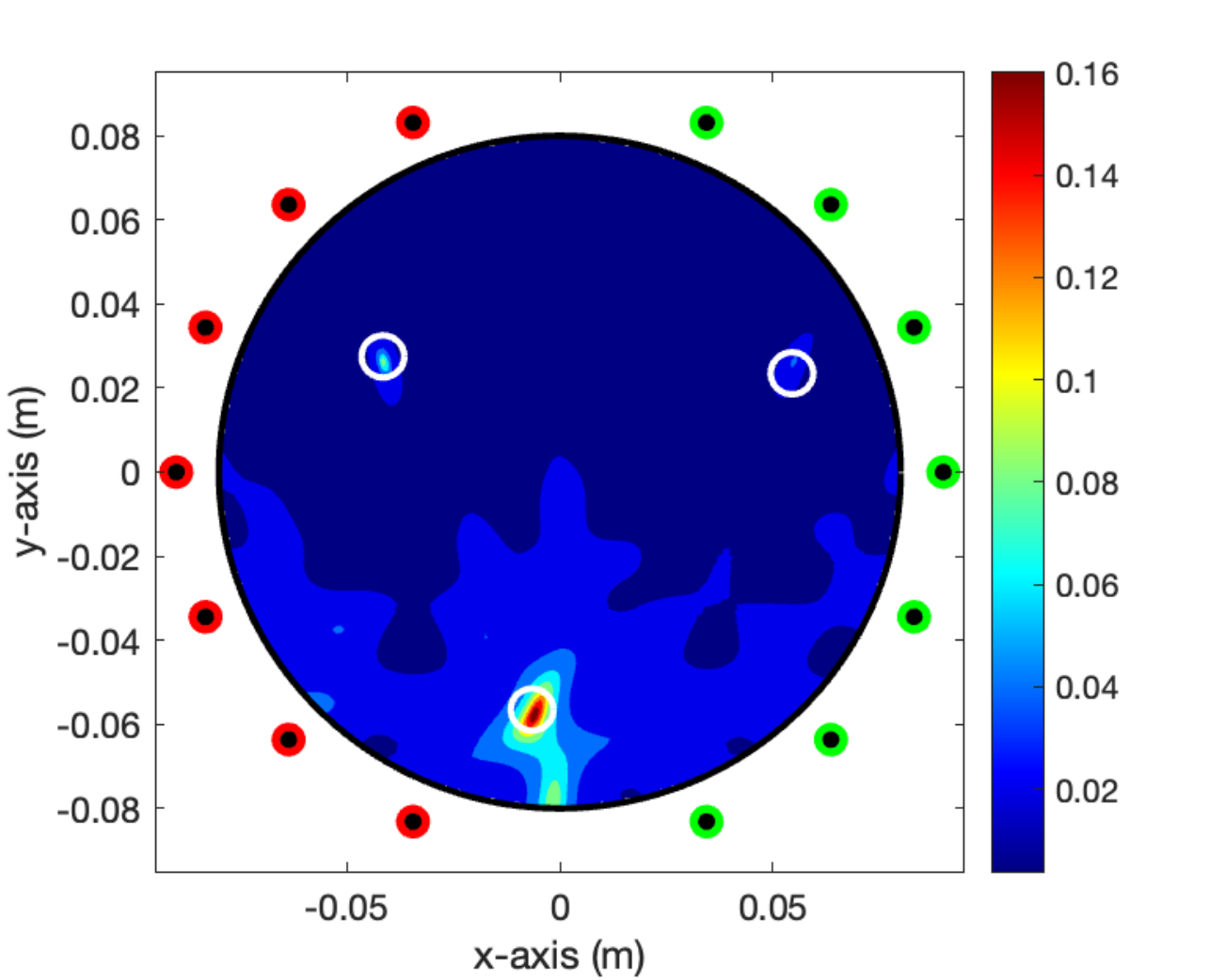}\hfill
  \includegraphics[width=0.25\textwidth]{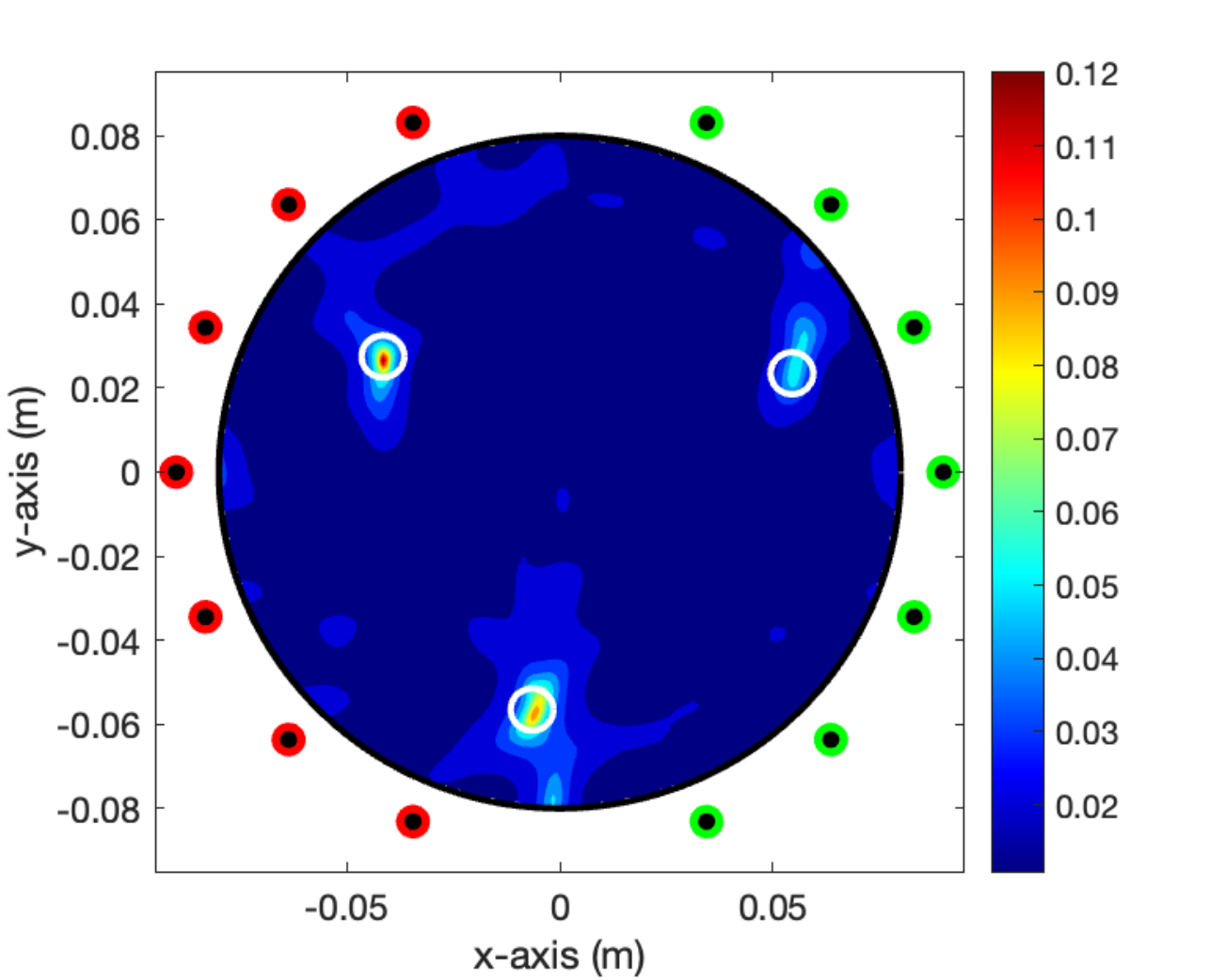}\hfill
  \includegraphics[width=0.25\textwidth]{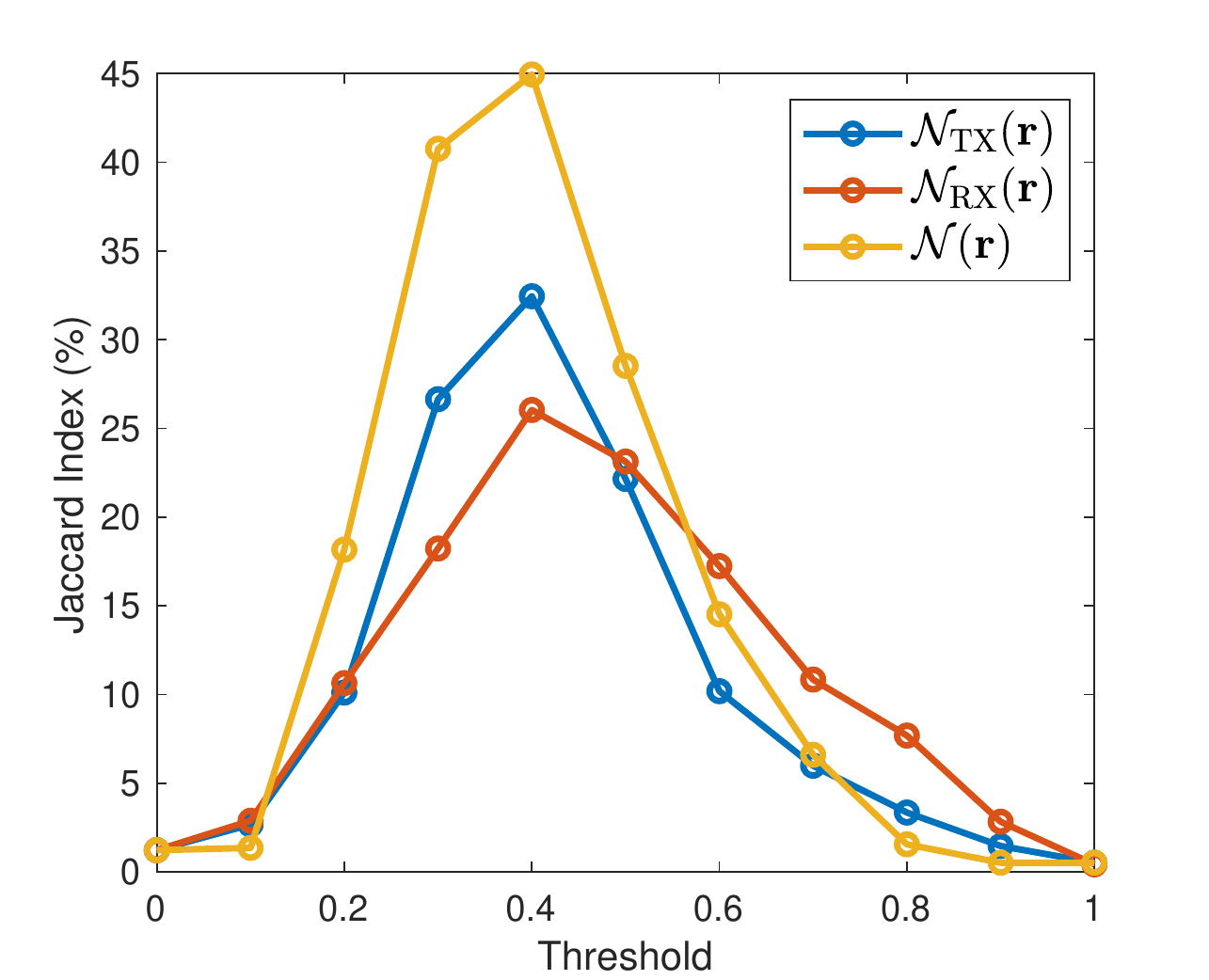}\\
  \includegraphics[width=0.25\textwidth]{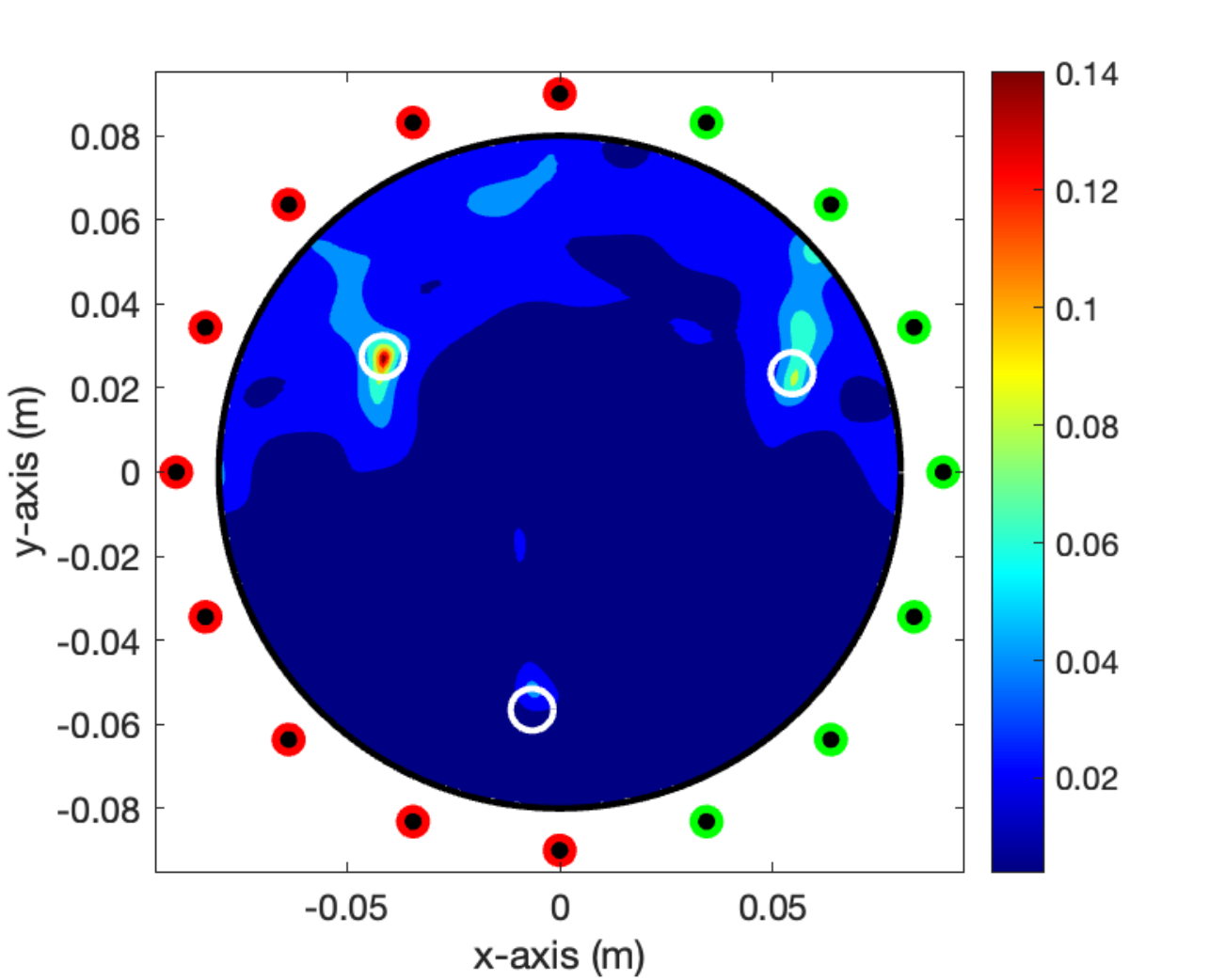}\hfill
  \includegraphics[width=0.25\textwidth]{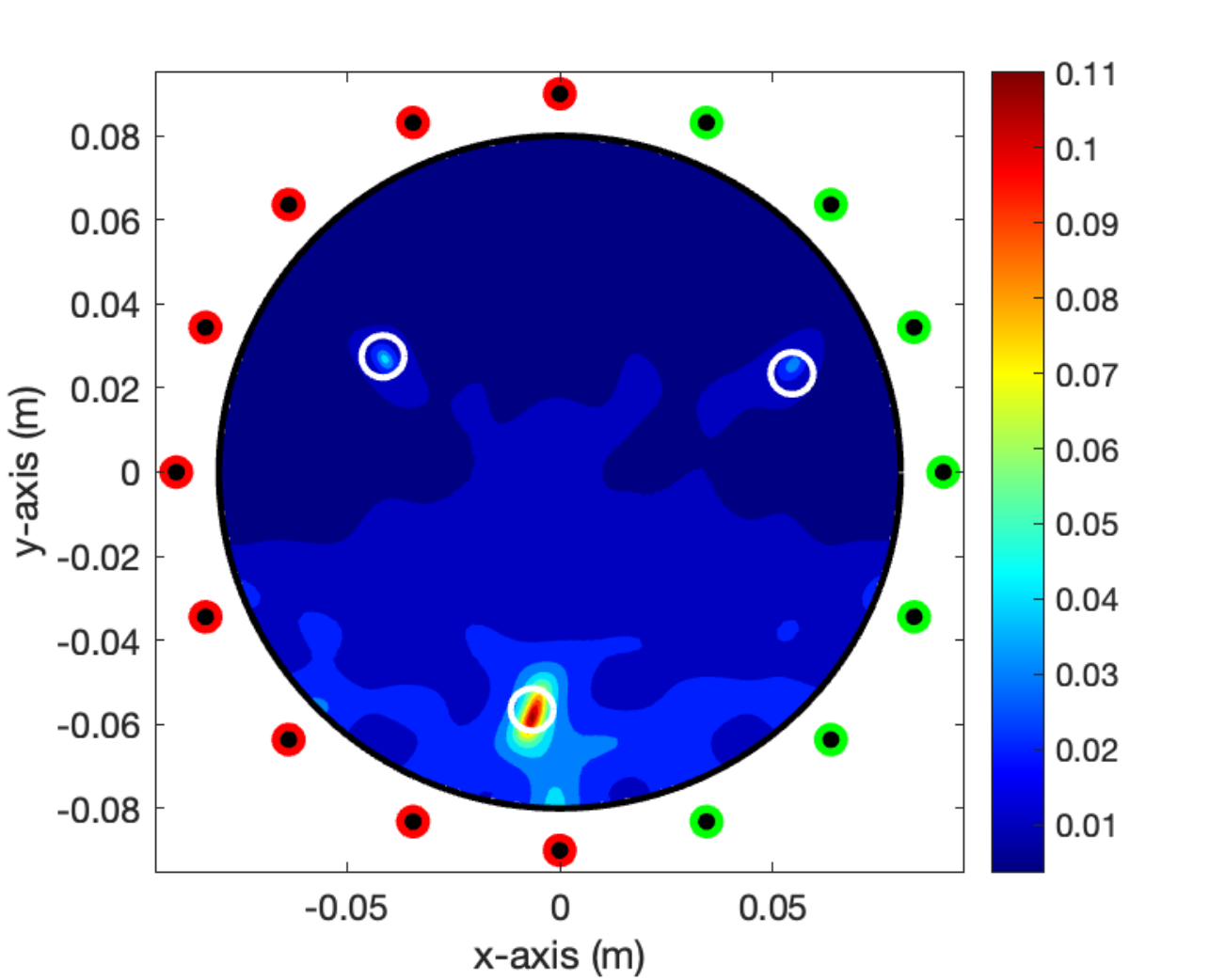}\hfill
  \includegraphics[width=0.25\textwidth]{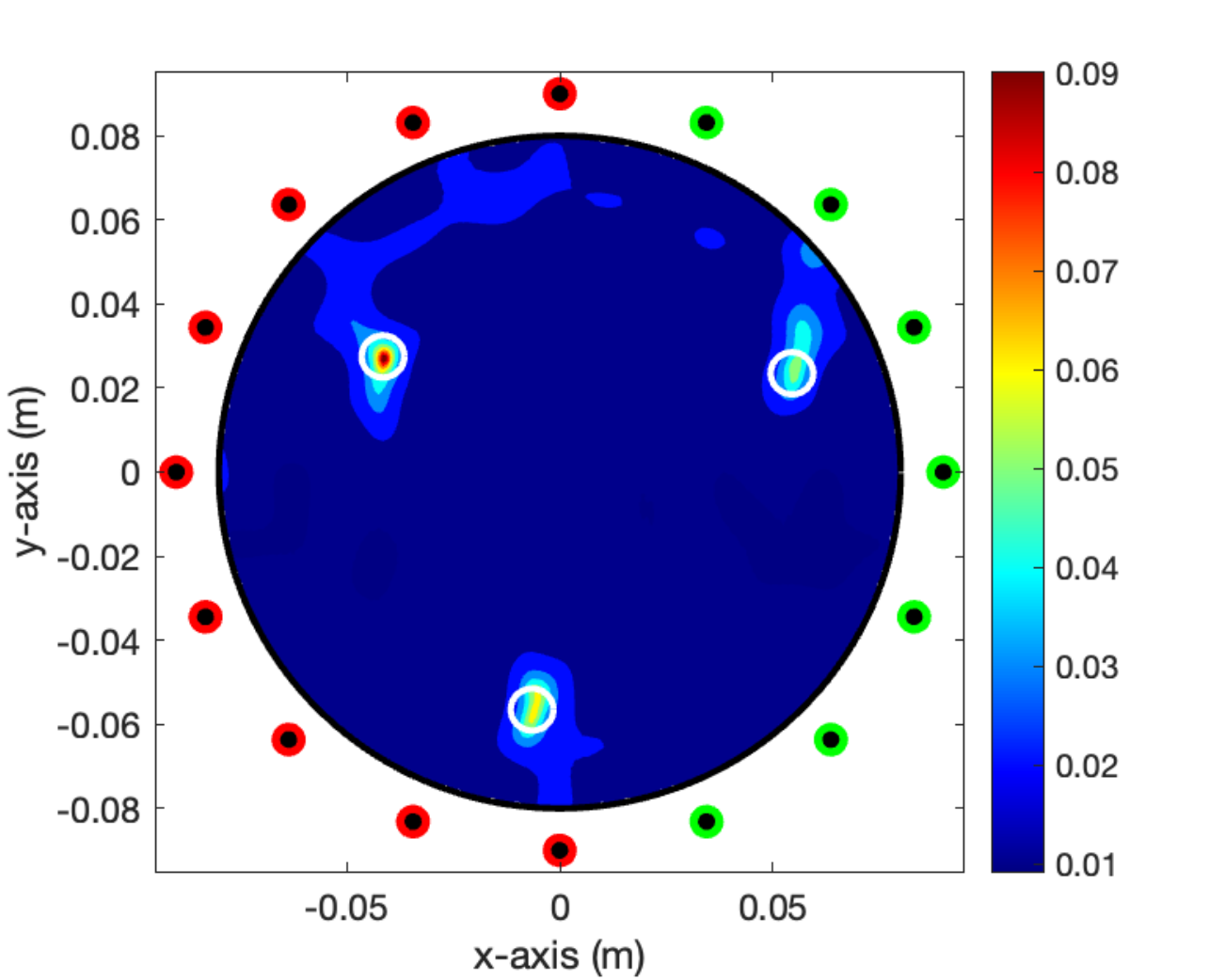}\hfill
  \includegraphics[width=0.25\textwidth]{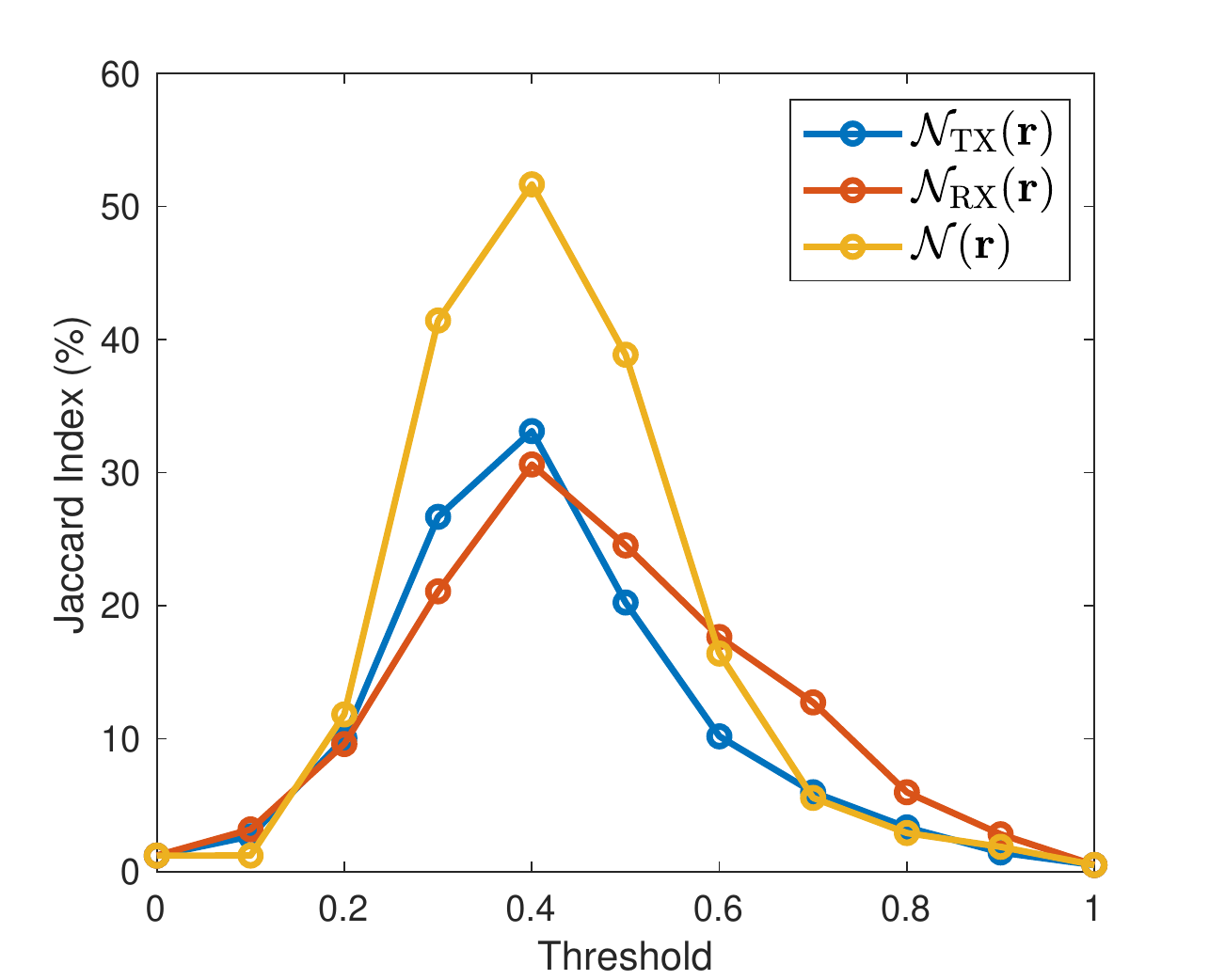}
  \caption{\label{ResultR5}(Example \ref{exR2}) Maps of $\mathfrak{F}_{\tx}(\mr)$ (first column), $\mathfrak{F}_{\rx}(\mr)$ (second column), $\mathfrak{F}(\mr)$ (third column), and Jaccard index (fourth column). Green and red colored circles describe the location of transmitters and receivers, respectively.}
\end{figure}

\begin{figure}[h]
  \centering
  \includegraphics[width=0.25\textwidth]{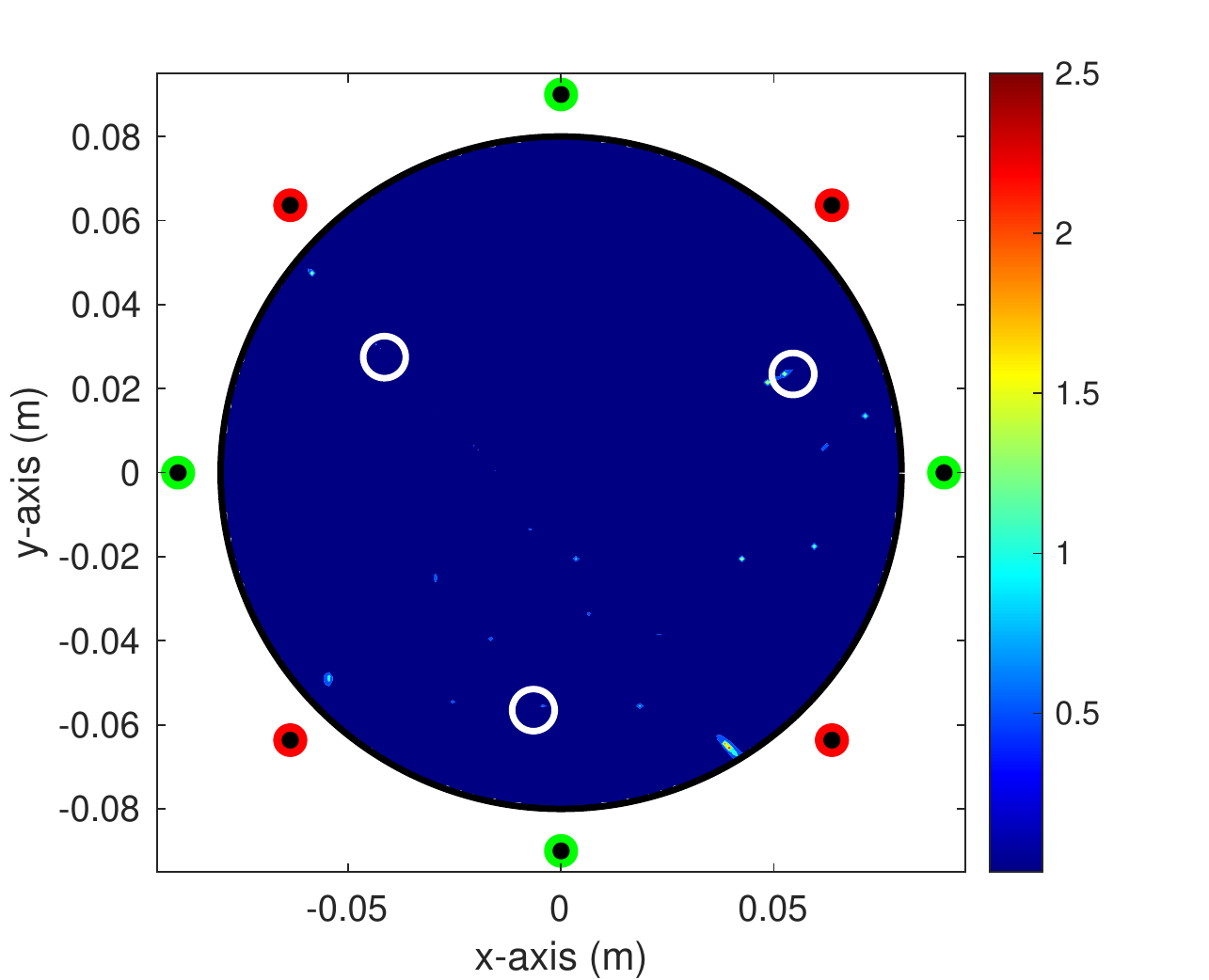}\hfill
  \includegraphics[width=0.25\textwidth]{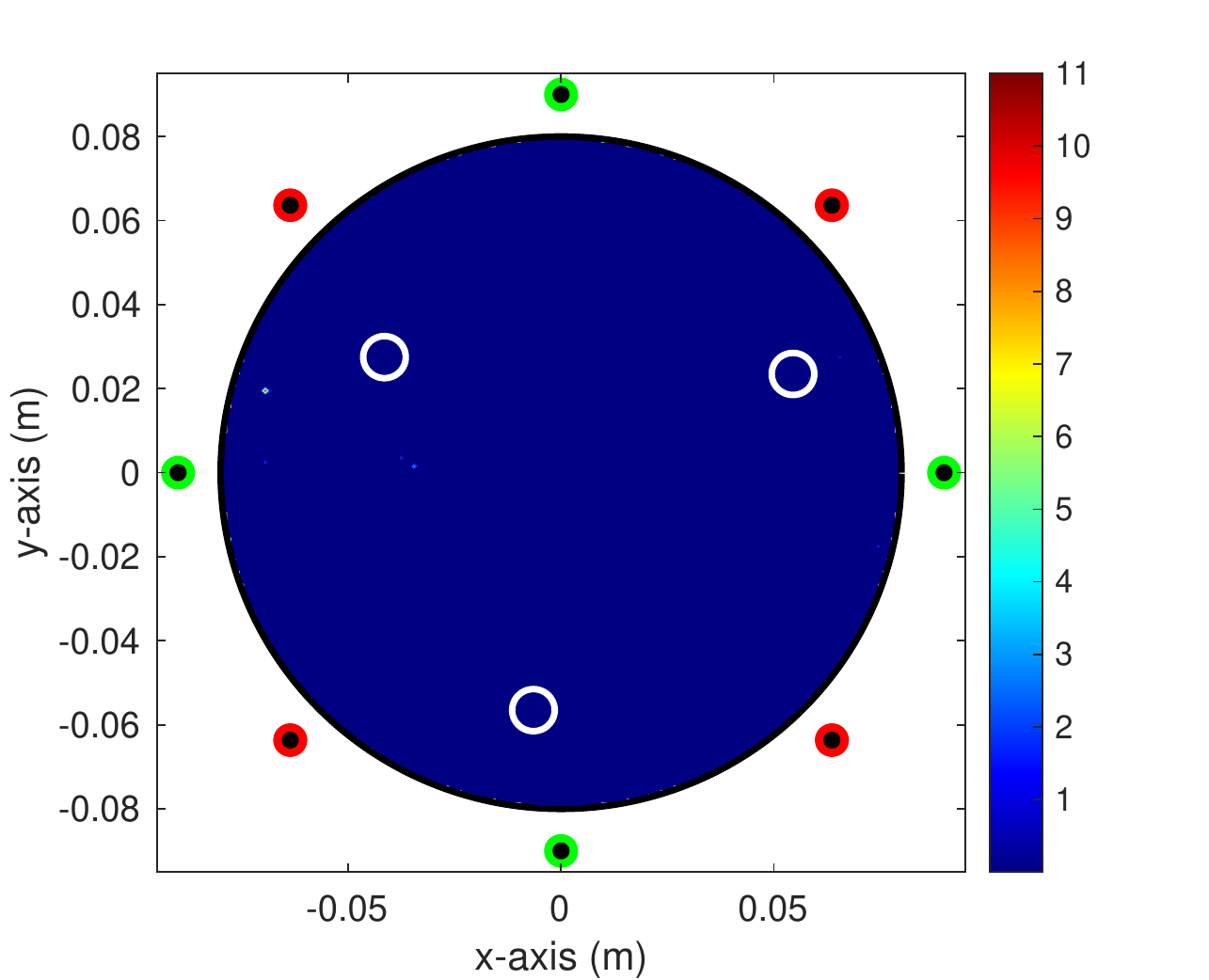}\hfill
  \includegraphics[width=0.25\textwidth]{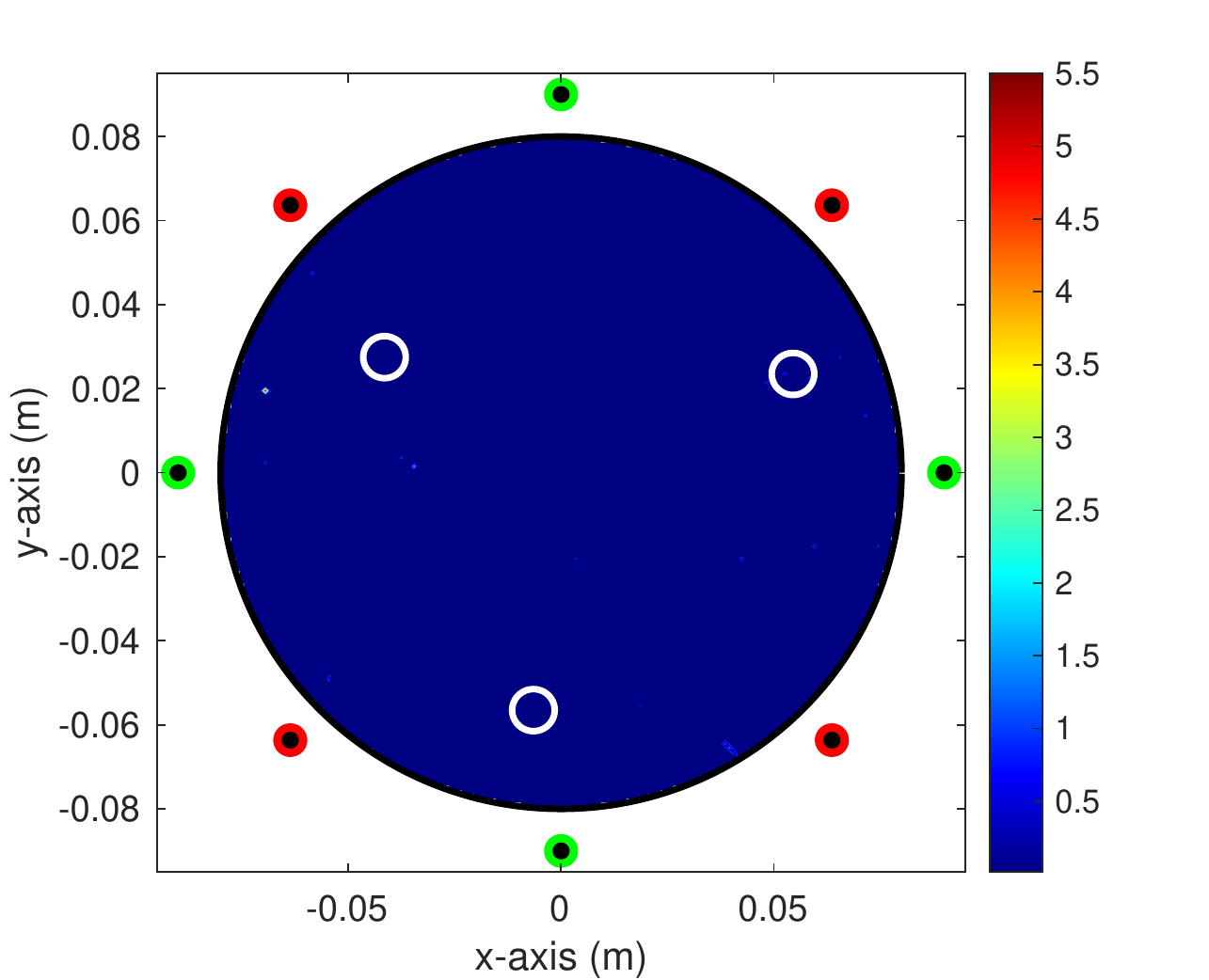}\hfill
  \includegraphics[width=0.25\textwidth]{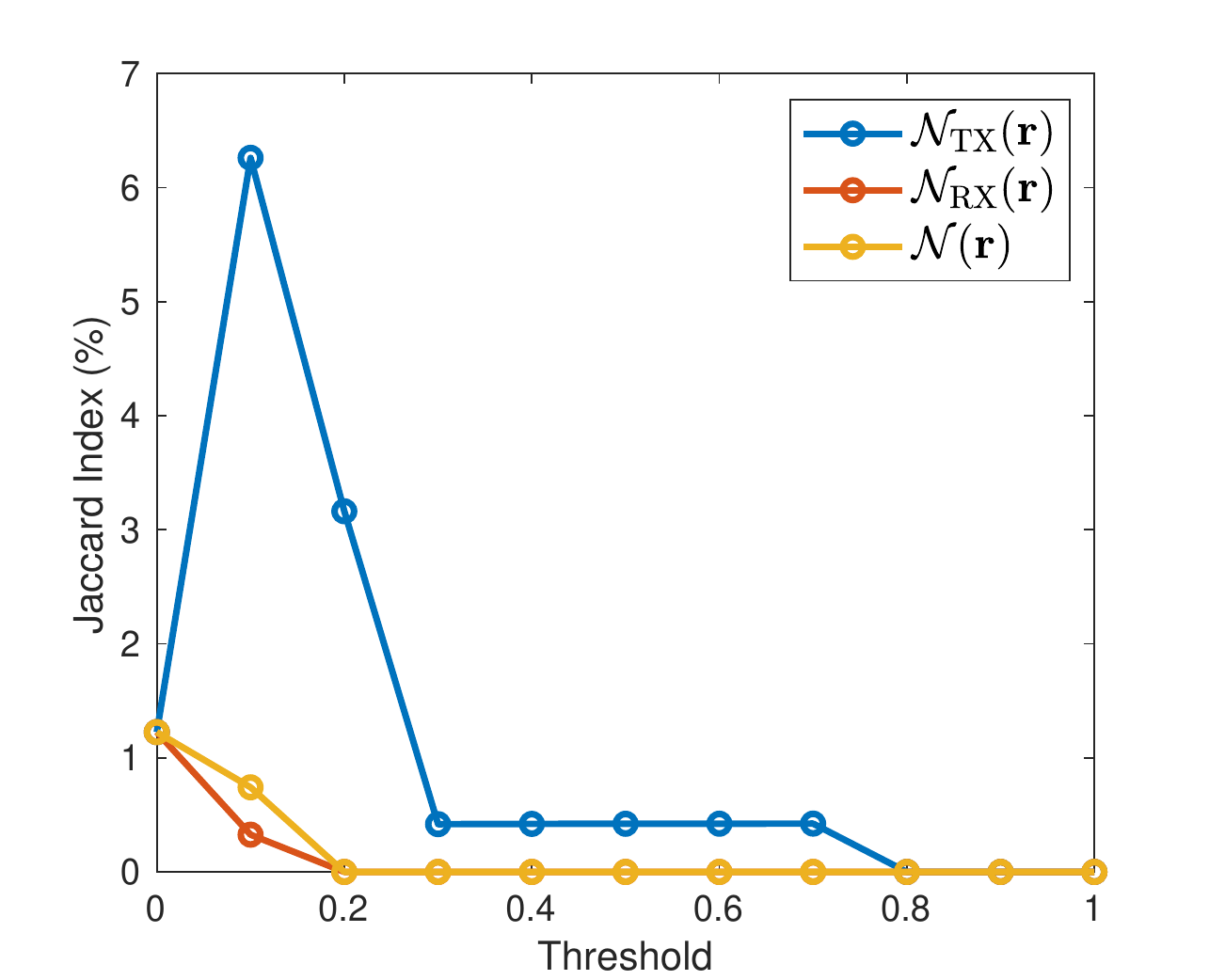}\\
  \includegraphics[width=0.25\textwidth]{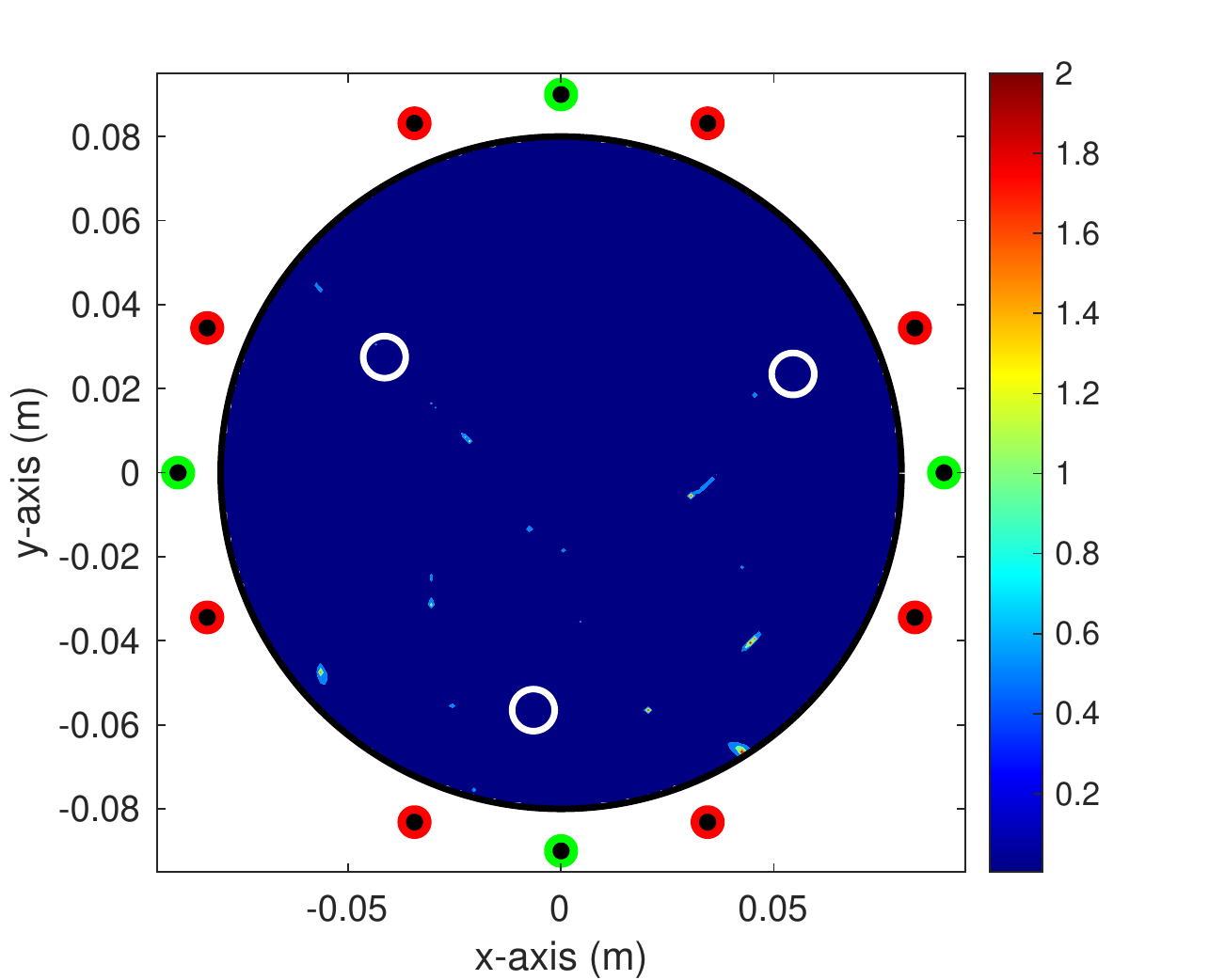}\hfill
  \includegraphics[width=0.25\textwidth]{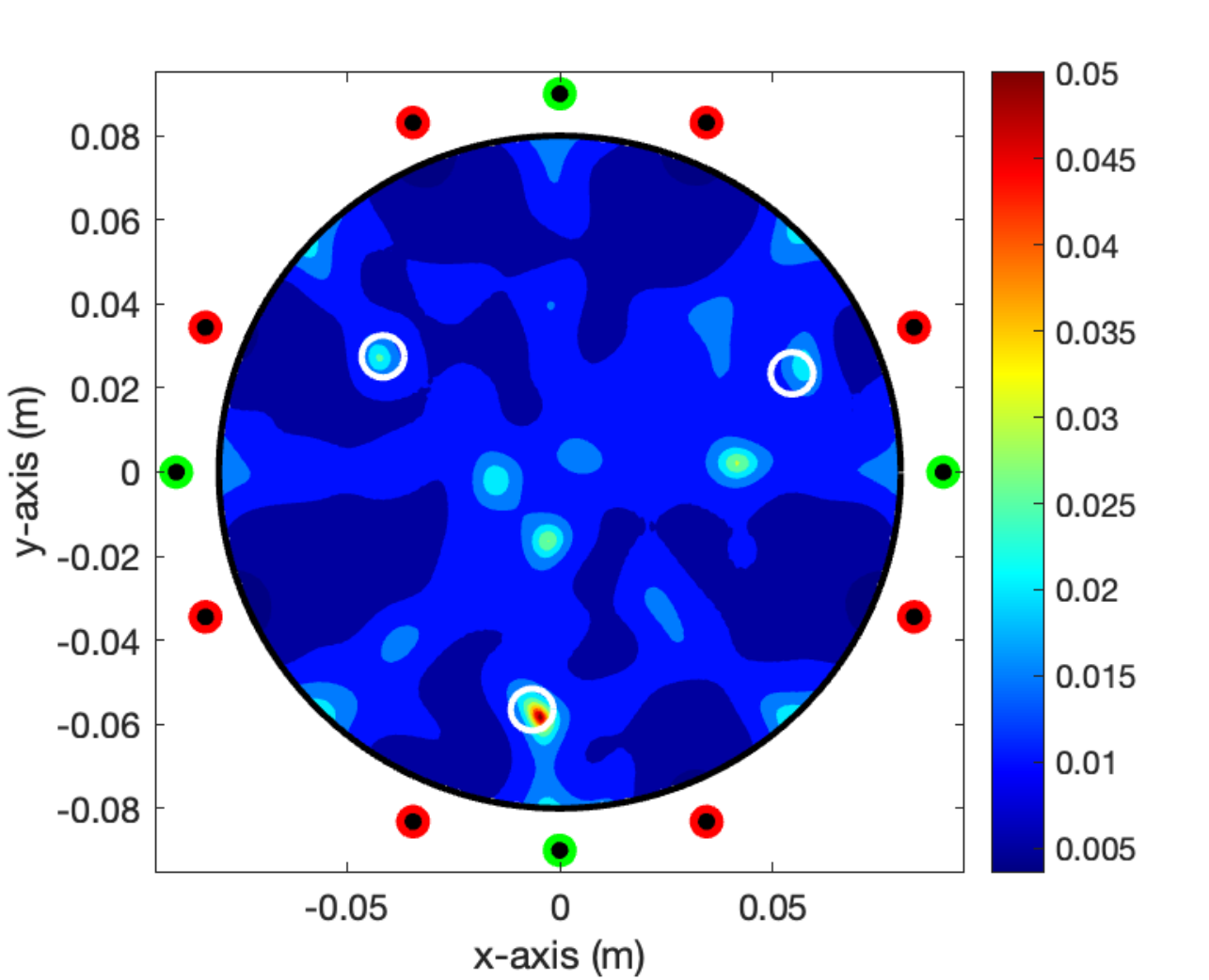}\hfill
  \includegraphics[width=0.25\textwidth]{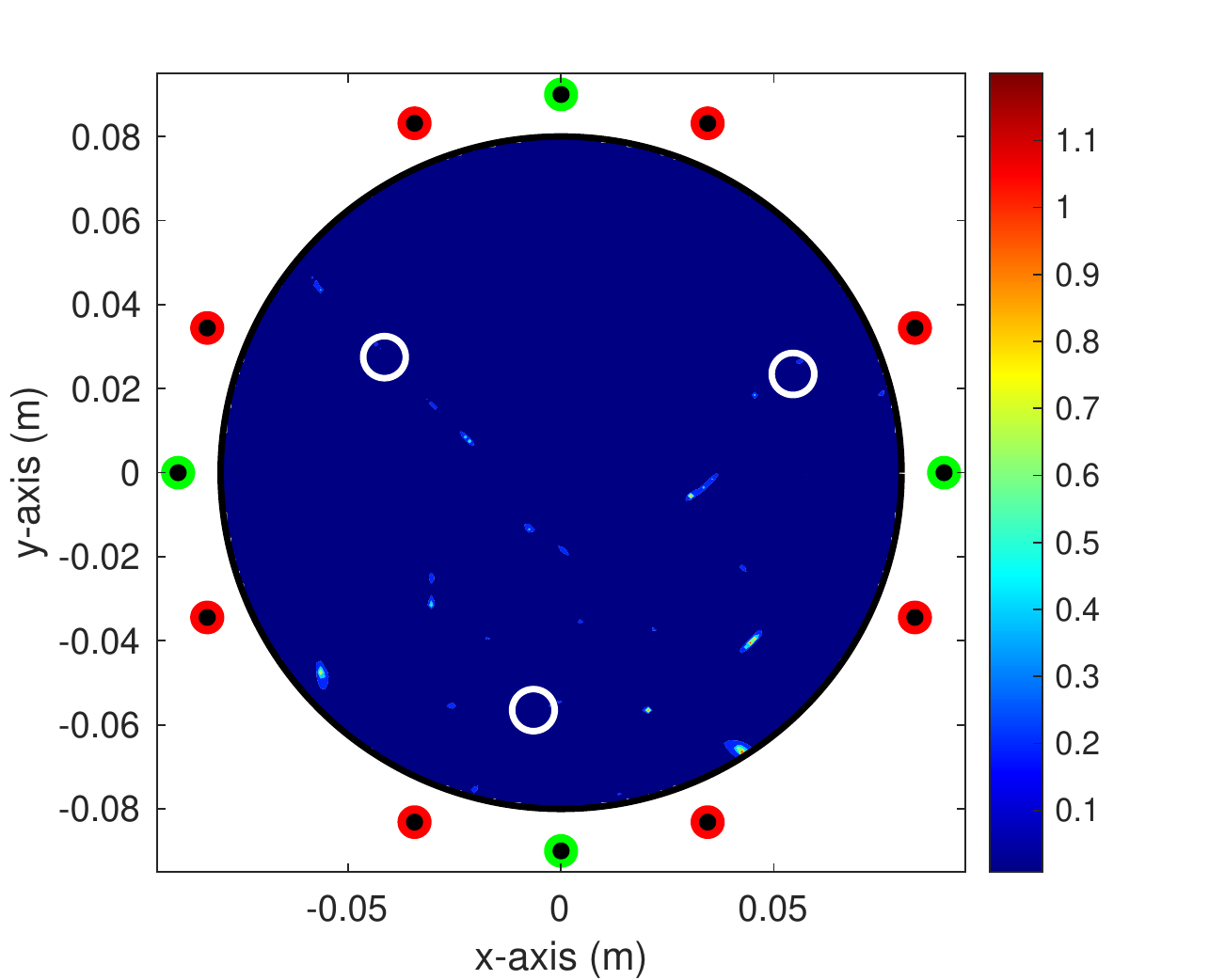}\hfill
  \includegraphics[width=0.25\textwidth]{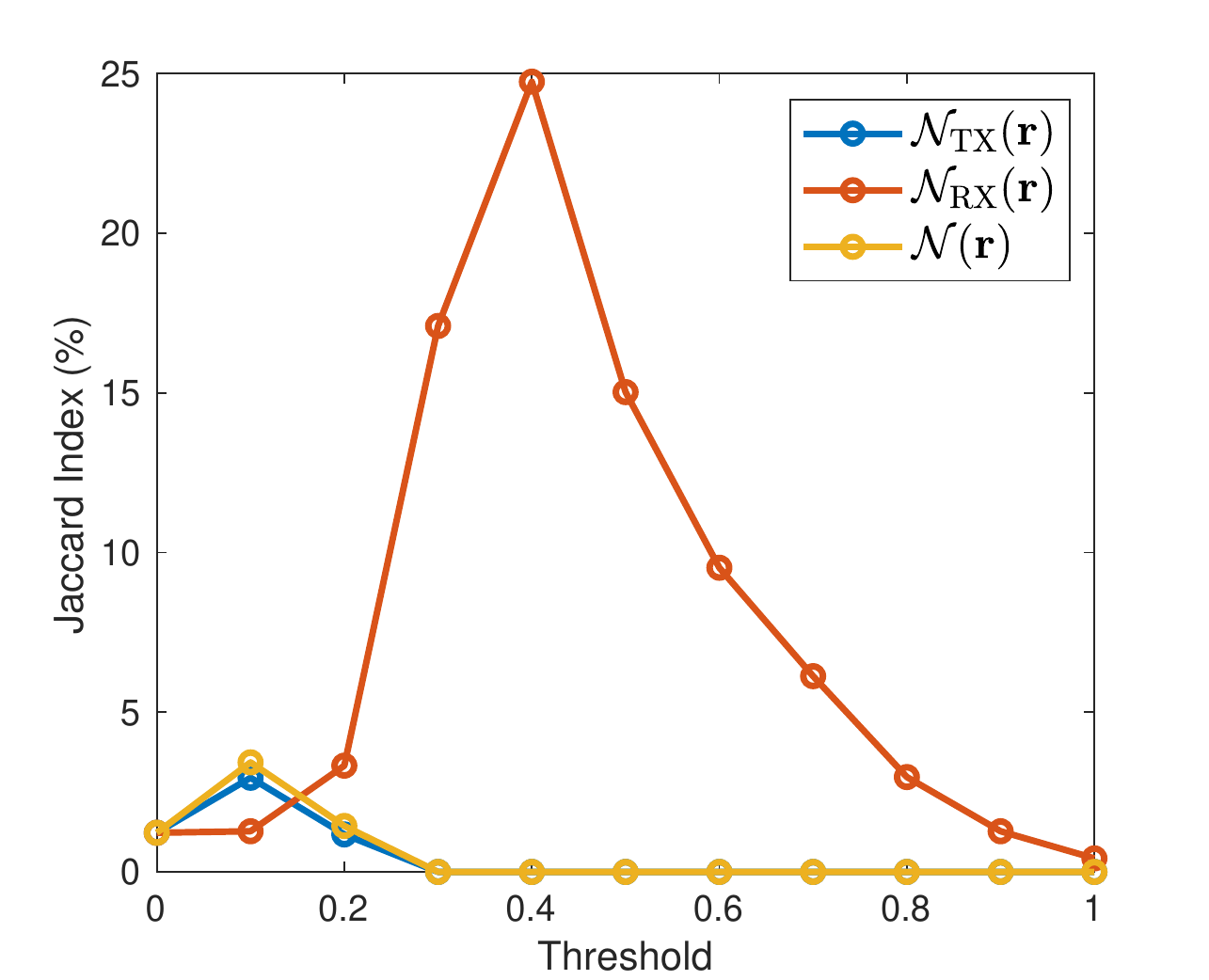}\\
  \includegraphics[width=0.25\textwidth]{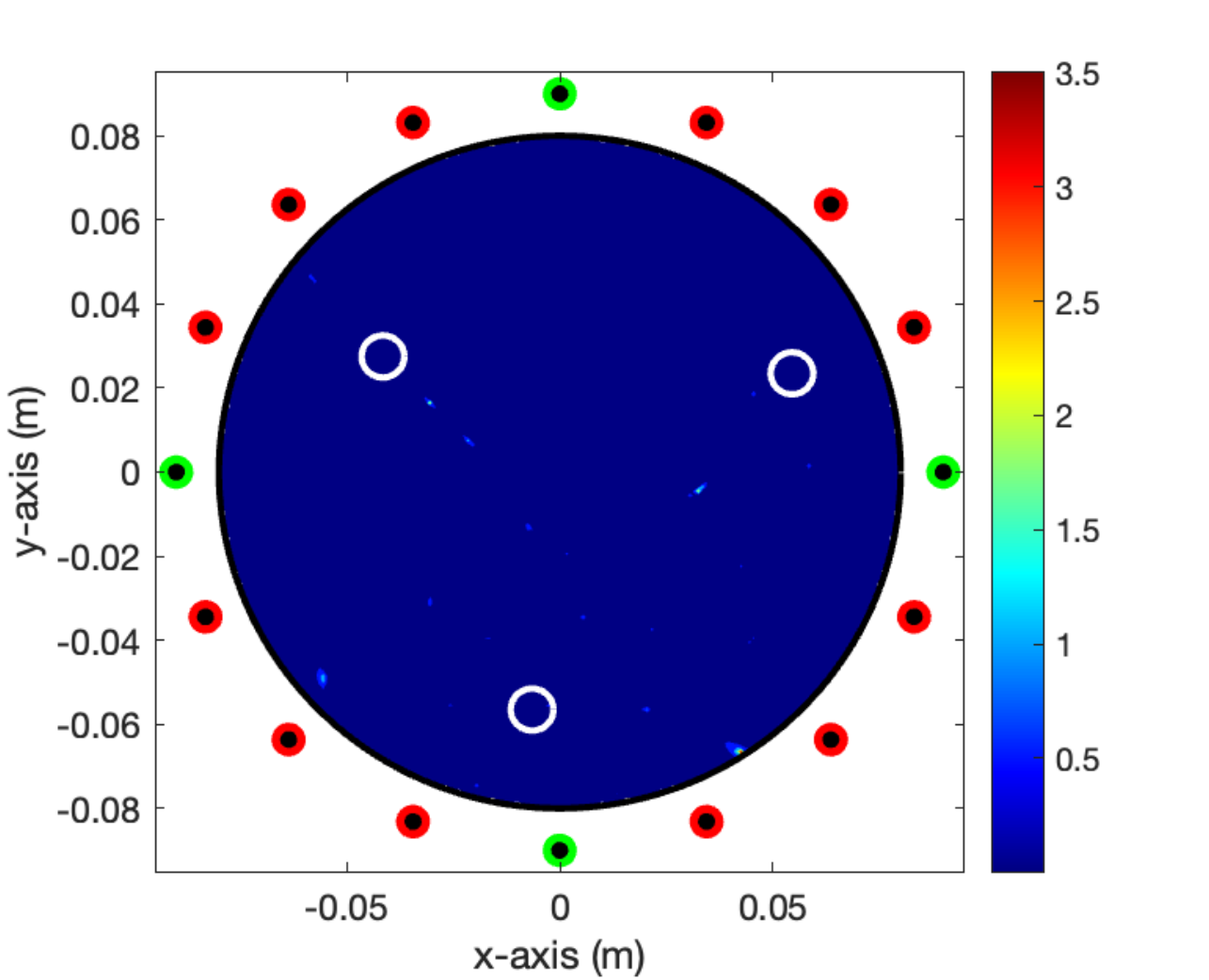}\hfill
  \includegraphics[width=0.25\textwidth]{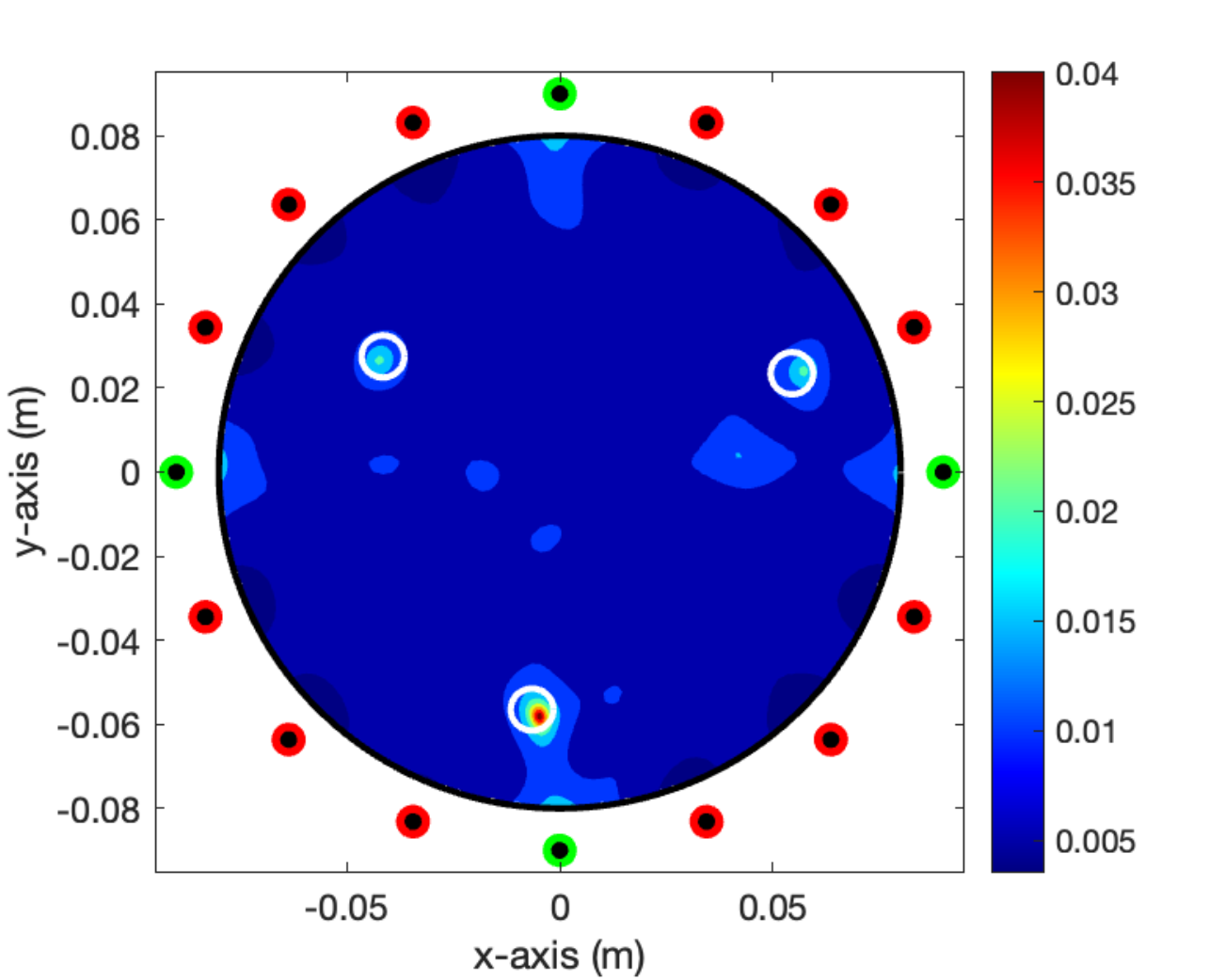}\hfill
  \includegraphics[width=0.25\textwidth]{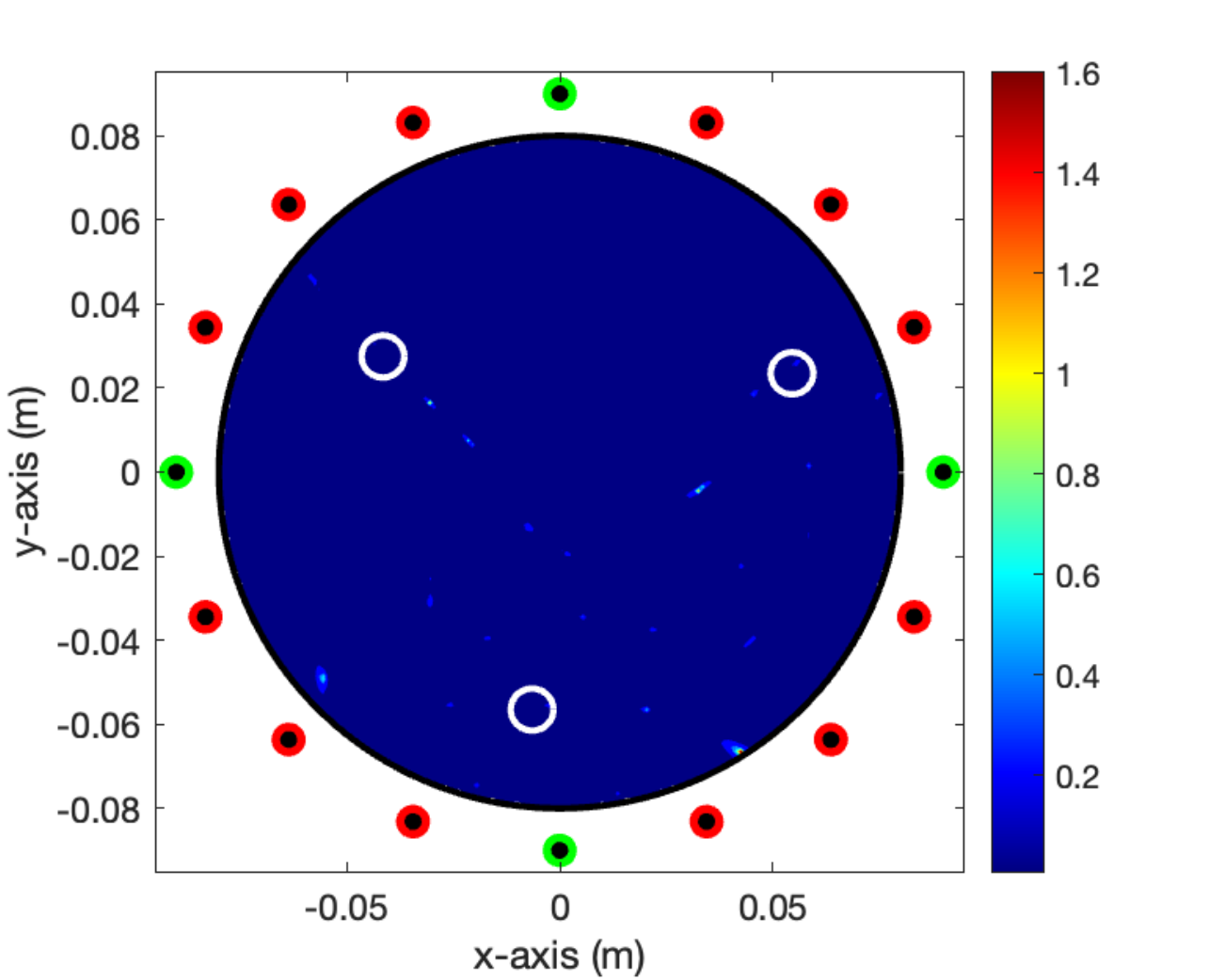}\hfill
  \includegraphics[width=0.25\textwidth]{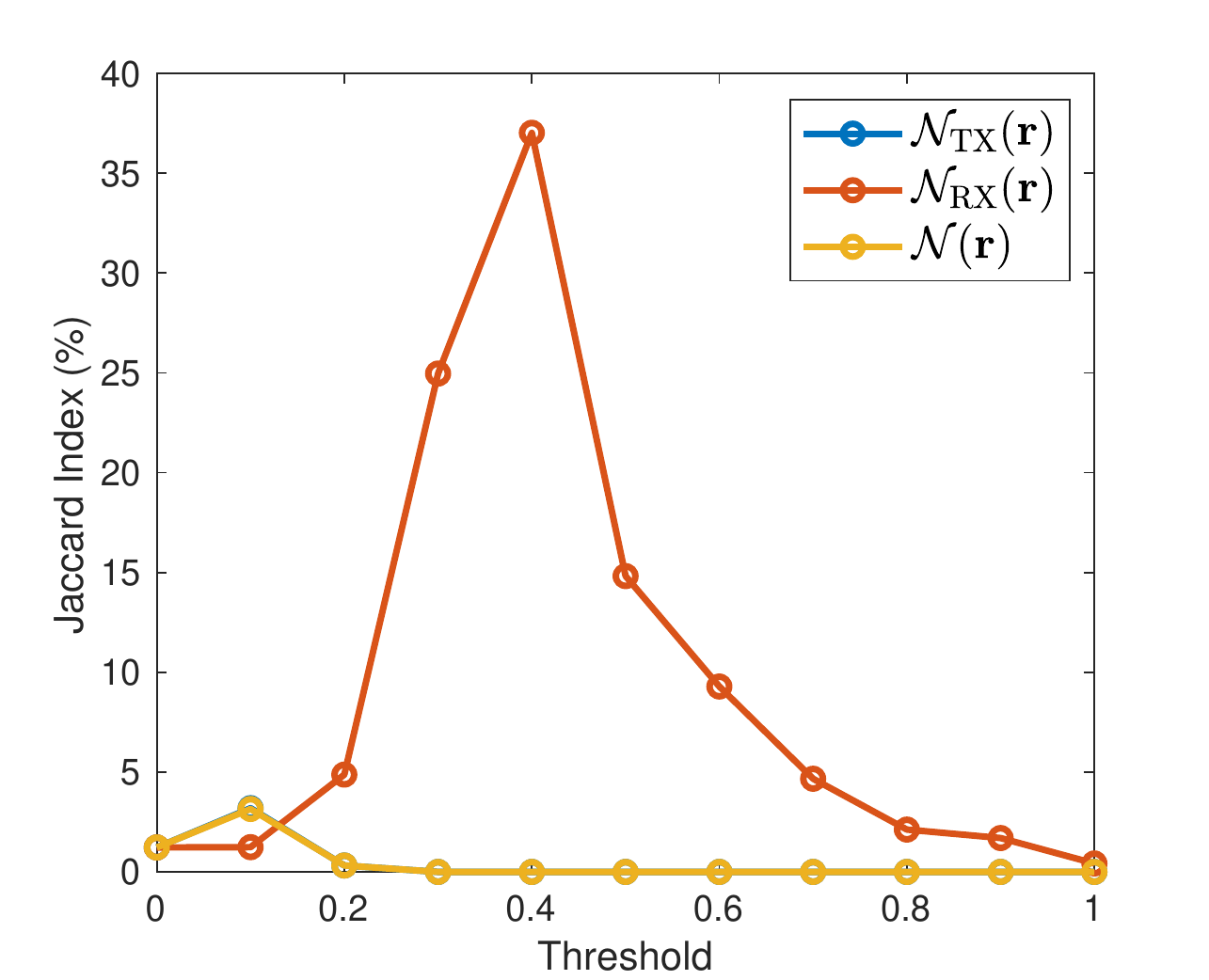}\\
  \includegraphics[width=0.25\textwidth]{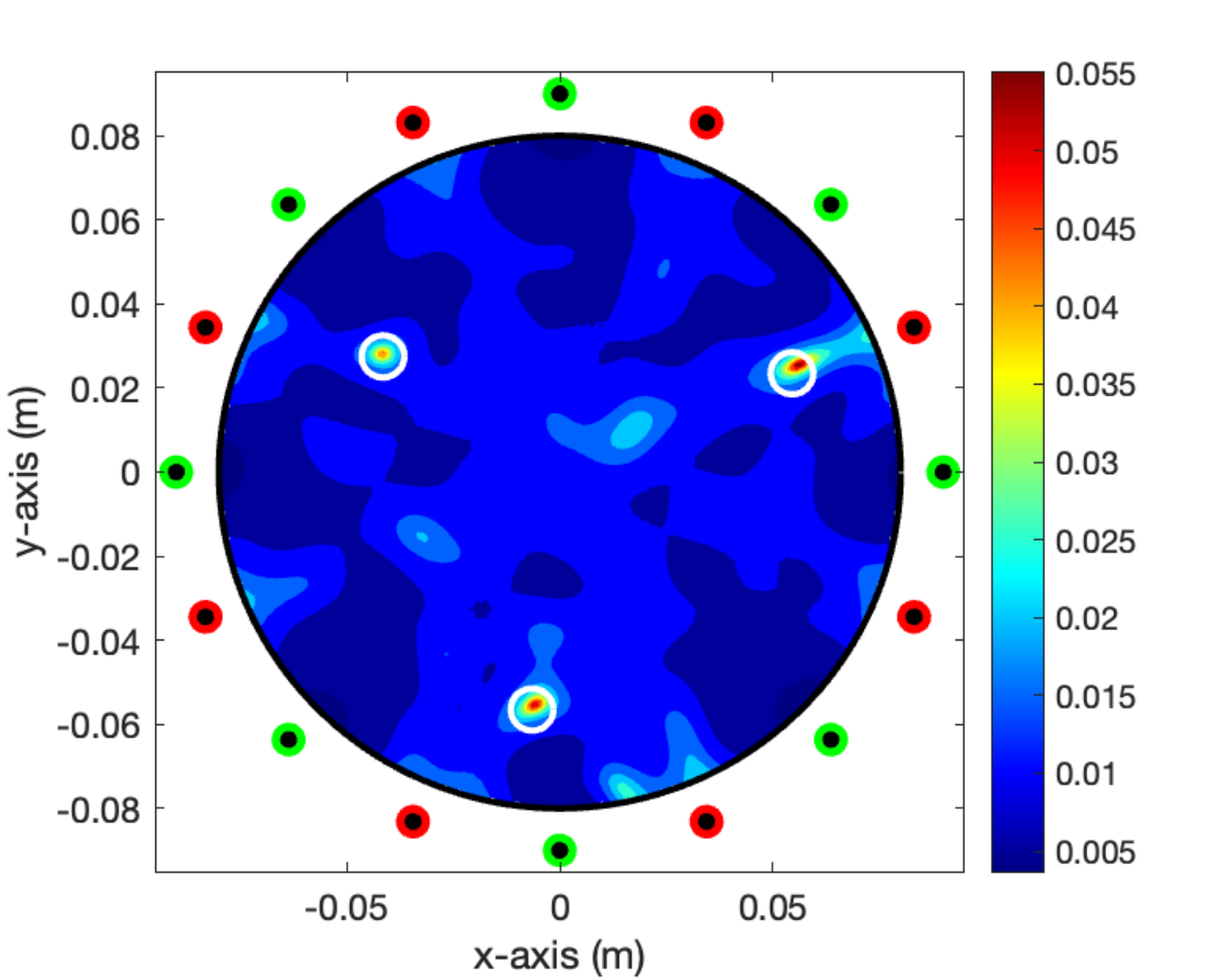}\hfill
  \includegraphics[width=0.25\textwidth]{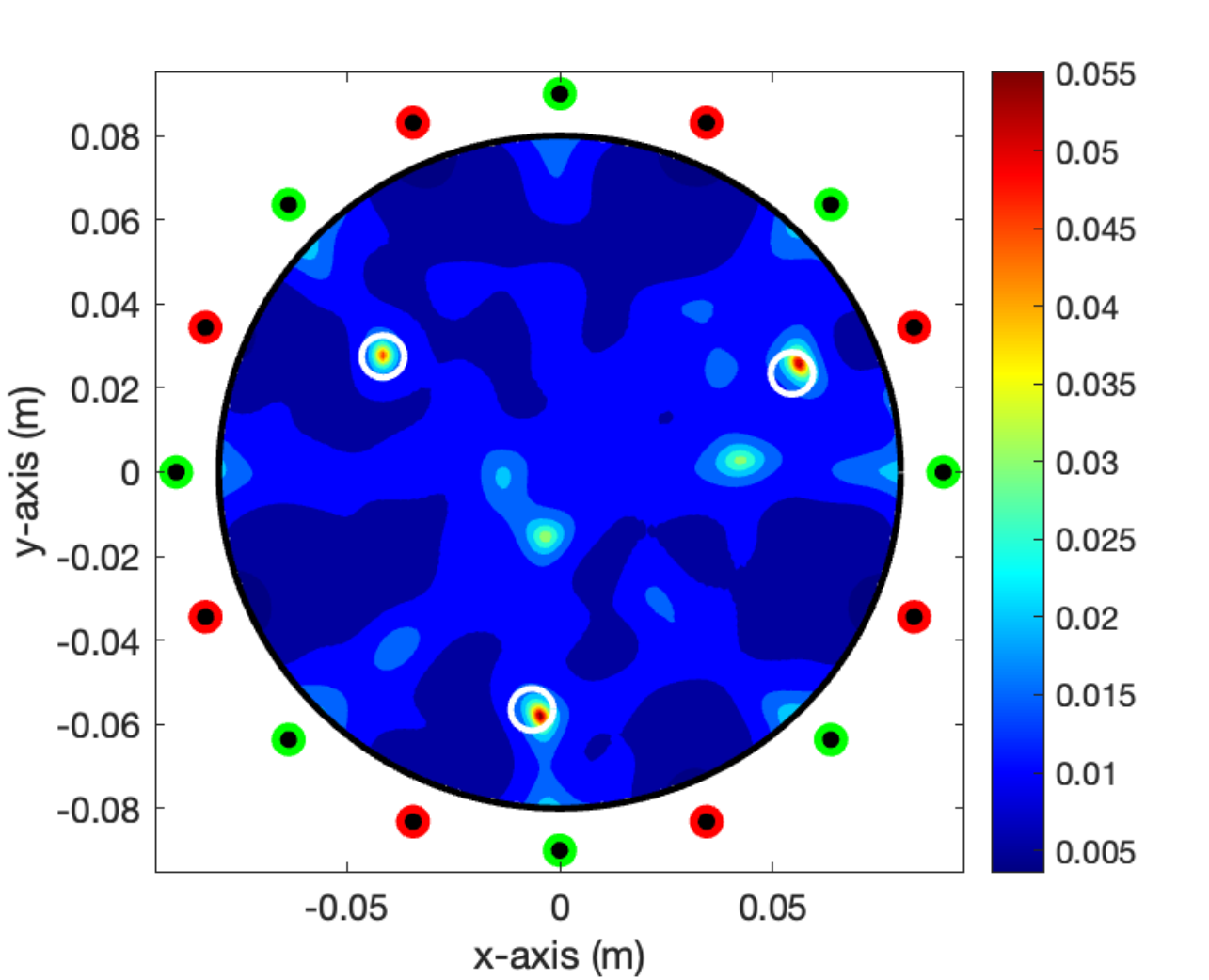}\hfill
  \includegraphics[width=0.25\textwidth]{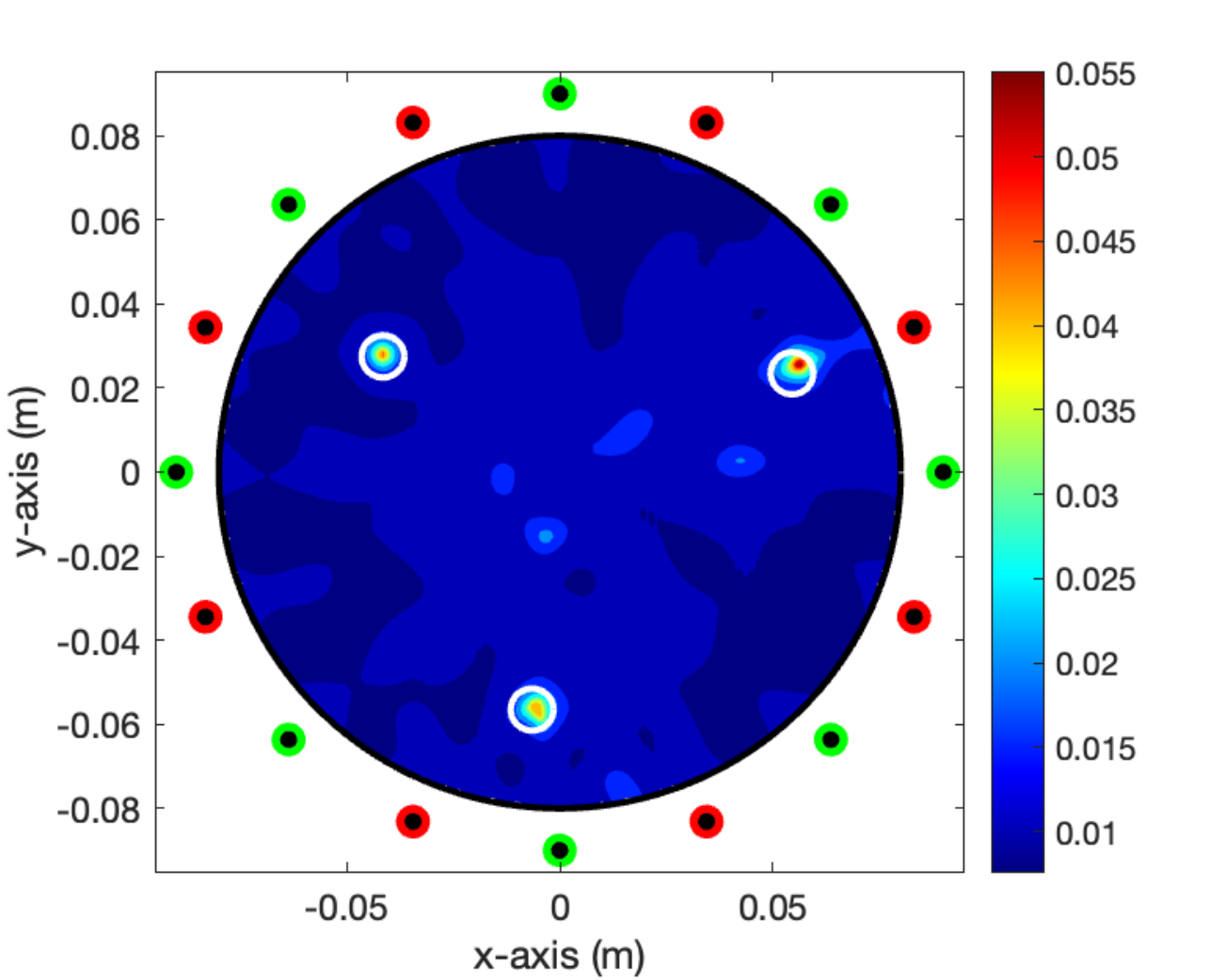}\hfill
  \includegraphics[width=0.25\textwidth]{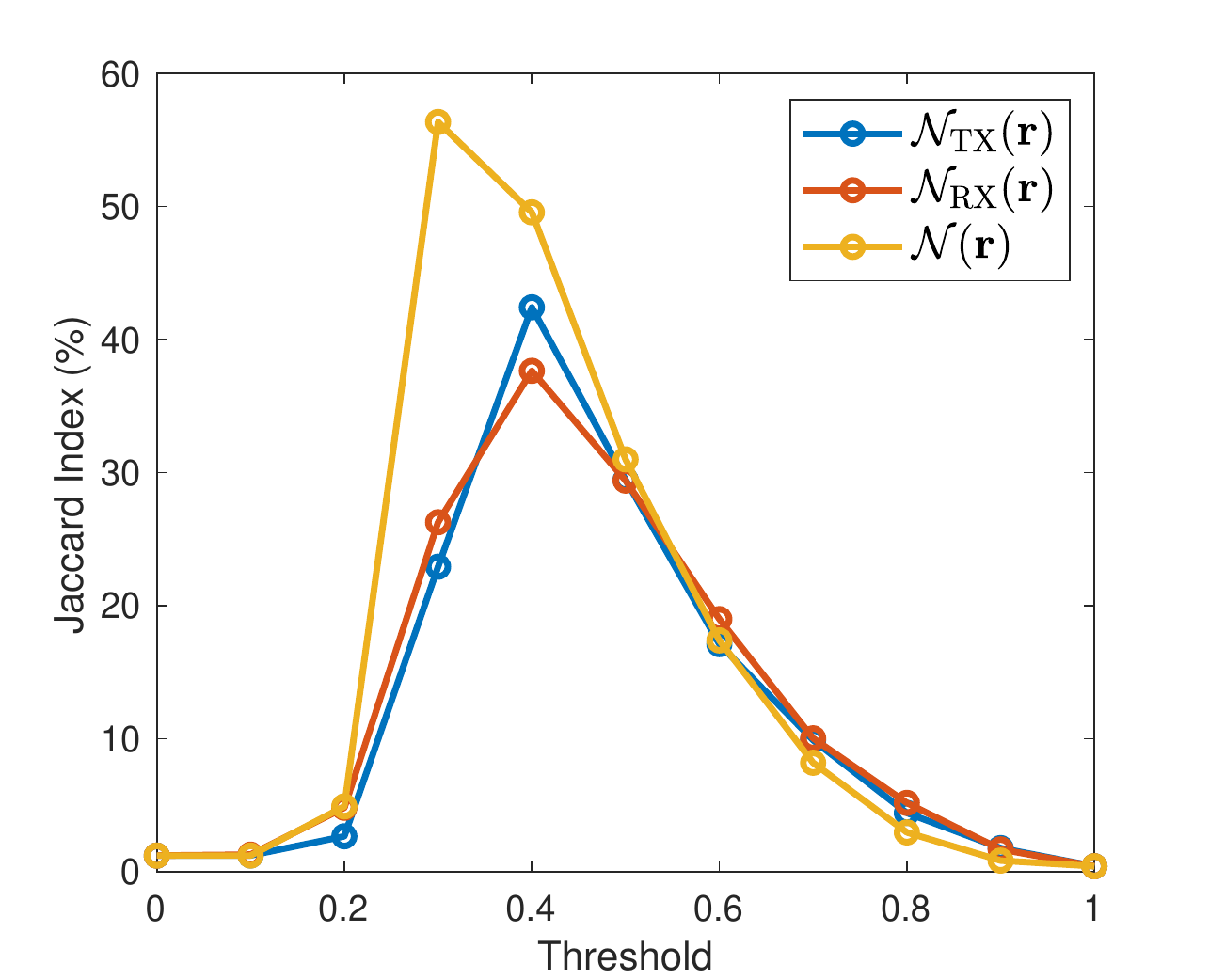}
  \caption{\label{ResultR6}(Example \ref{exR2}) Maps of $\mathfrak{F}_{\tx}(\mr)$ (first column), $\mathfrak{F}_{\rx}(\mr)$ (second column), $\mathfrak{F}(\mr)$ (third column), and Jaccard index (fourth column). Green and red colored circles describe the location of transmitters and receivers, respectively.}
\end{figure}

\section{Concluding Remarks}\label{sec:6}
In this paper, we designed an imaging function of MUSIC for identifying small anomalies from the scattering matrix without a switching device. Based on the analyzed structure of the designed imaging function, we confirmed that the imaging performance of the MUSIC significantly depends on the total number of transmitters and receivers. Additionally, the analyzed structure also allows us to find the optimal antenna arrangement for obtaining better imaging results via MUSIC. To support the derived theoretical result, we presented various numerical simulation results from synthetic and experimental data.

Here, we considered the identification of small targets. Extension to the identification of arbitrary shapes targets such as large anomaly, crack-like defects, etc., will be the forthcoming work. Moreover, to obtain better imaging results, development of proposed MUSIC algorithm will be an interesting topic for future research.

\section*{Acknowledgements}
The author would like to thank Kwang-Jae Lee and Seong-Ho Son for helping in generating scattering parameter data. This work was supported by the National Research Foundation of Korea (NRF) grant funded by the Korea government (MSIT) (NRF-2020R1A2C1A01005221).



\bibliographystyle{elsarticle-num-names}
\bibliography{../../../References}

\end{document}